\documentclass[11pt,reqno]{amsart}
\usepackage{fourier}
\usepackage{fullpage}
\usepackage{esint}
\usepackage[all]{xy}

\usepackage{mathrsfs,amssymb,graphicx,verbatim,amsmath,amsfonts}
\usepackage{paralist}
\usepackage[breaklinks,pdfstartview=FitH]{hyperref}

\usepackage[mathscr]{euscript}
\makeatletter
\def\resetMathstrut@{%
  \setbox\z@\hbox{%
    \mathchardef\@tempa\mathcode`\(\relax
    \def\@tempb##1"##2##3{\the\textfont"##3\char"}%
    \expandafter\@tempb\meaning\@tempa \relax
  }%
  \ht\Mathstrutbox@1.2\ht\z@ \dp\Mathstrutbox@1.2\dp\z@
}
\makeatother
\renewcommand{\le}{\leqslant}
\renewcommand{\ge}{\geqslant}

\renewcommand{\setminus}{\smallsetminus}
\newcommand{\ud}[0]{\,\mathrm{d}}

\usepackage{dutchcal}
\def\cprime{$'$}

\newcommand{\diam}{\mathrm{diam}}

\newcommand{\bd}{\boldsymbol{\delta}}
\renewcommand{\d}{\delta}
\newcommand{\cc}{\mathsf{c}}
\newcommand{\vr}{\mathrm{vr}}
\newcommand{\evr}{\mathrm{evr}}

\newcommand{\NN}{\mathcal{N}}

\newcommand{\dd}{\mathsf{d}}
\renewcommand{\det}{\mathrm{det}}

\newcommand{\proj}{\mathsf{Proj}}
\newcommand{\sign}{\mathrm{sign}}

\newcommand{\n}{\{1,\ldots,n\}}
\newcommand{\m}{\{1,\ldots,m\}}

\renewcommand{\k}{\{1,\ldots,k\}}

\newcommand{\Per}{\mathrm{Per}}
\newcommand{\X}{\mathbf X}
\newcommand{\Y}{\mathbf Y}
\newcommand{\f}{\varphi}

\renewcommand{\bd}{\boldsymbol{\delta}}

\newcommand{\e}{\varepsilon}
\newcommand{\R}{\mathbb R}

\newtheorem{theorem}{Theorem}
\newtheorem{lemma}[theorem]{Lemma}
\newtheorem{proposition}[theorem]{Proposition}

\newtheorem{corollary}[theorem]{Corollary}

\newtheorem{definition}[theorem]{Definition}

\theoremstyle{remark}
\newtheorem{remark}[theorem]{Remark}
\newcommand{\MM}{\mathcal{M}}

\newtheorem{example}[theorem]{Example}
\newtheorem{problem}[theorem]{Problem}

\newtheorem{conjecture}[theorem]{Conjecture}
\newtheorem{question}[theorem]{Question}

\newcommand{\EE}{\mathbf{E}}
\newcommand{\iq}{\mathrm{iq}}
\newcommand{\Id}{\mathsf{Id}}
\newcommand{\M}{\mathsf{M}}
\newcommand{\GL}{\mathsf{GL}}
\newcommand{\SL}{\mathsf{SL}}
\renewcommand{\O}{\mathsf{O}}
\renewcommand{\supset}{\supseteq}

\newcommand{\sfS}{\mathsf{S}}
\renewcommand{\subset}{\subseteq}
\newcommand{\C}{\mathbb C}

\DeclareMathOperator{\trace}{\bf Tr}
\newcommand{\E}{\mathbb{ E}}

\DeclareMathOperator{\supp}{supp}
\newcommand{\W}{\mathsf{W}}

\newcommand{\N}{\mathbb N}
\newcommand{\Z}{\mathbb Z}

\newcommand{\eqdef}{\stackrel{\mathrm{def}}{=}}

\newcommand{\Lip}{\mathrm{Lip}}
\newcommand{\BM}{\mathrm{BM}}

\renewcommand{\Pr}{\mathbf{Prob}}
\newcommand{\Ch}{\mathrm{Ch}}

\newcommand{\A}{\mathsf{A}}
\renewcommand{\emptyset}{\varnothing}
\newcommand{\ee}{\mathsf{e}}

\newcommand{\bfW}{\mathbf{W}}
\newcommand{\bfV}{\mathbf{V}}
\newcommand{\bfZ}{\mathbf{Z}}
\newcommand{\sub}{\mathscr{C}}
\newcommand{\vol}{\mathrm{vol}}
\renewcommand{\1}{\mathbf 1}
\newcommand{\cone}{\mathsf{Cone}}
\newcommand{\conv}{\mathrm{conv}}
\newcommand{\spn}{\mathrm{ span}}
\newcommand{\Part}{\mathscr{P}}
\newcommand{\ZZ}{\mathcal{Z}}
\newcommand{\cF}{\mathscr{F}}
\newcommand{\us}{\sigma}
\newcommand{\sep}{\mathsf{SEP}}
\newcommand{\gpu}{\mathsf{GPU}}
\newcommand{\pad}{\mathsf{PAD}}
\newcommand{\g}{\mathsf{g}}
\newcommand{\tp}{\hat{\otimes}}
\newcommand{\ti}{\check{\otimes}}
\newcommand{\XX}{\mathsf{X}}

\newcommand{\LT}{\mathsf{LT}}

\renewcommand{\H}{\mathsf{H}}
\newcommand{\V}{\mathsf{V}}

\newcommand{\RR}{\mathsf{R}}
\newcommand{\sfG}{\mathsf{G}}
\newcommand{\sfU}{\mathsf{U}}
\newcommand{\sfZ}{\mathsf{Z}}
\newcommand{\bfE}{\mathbf{E}}
\newcommand{\bfF}{\mathbf{F}}
\newcommand{\bfH}{\mathbf{H}}
\newcommand{\sfN}{\mathsf{N}}

\newcommand{\cG}{\mathscr{G}}

\newcommand{\p}{\mathfrak{p}}
\newcommand{\Q}{\mathbb{Q}}

\newcommand{\DD}{\mathscr{V}}

\newcommand{\yy}{\mathsf{y}}

\newcommand{\cx}{\mathcal{x}}

\makeatletter
\def\moverlay{\mathpalette\mov@rlay}
\def\mov@rlay#1#2{\leavevmode\vtop{%
   \baselineskip\z@skip \lineskiplimit-\maxdimen
   \ialign{\hfil$\m@th#1##$\hfil\cr#2\crcr}}}
\newcommand{\charfusion}[3][\mathord]{
    #1{\ifx#1\mathop\vphantom{#2}\fi
        \mathpalette\mov@rlay{#2\cr#3}
      }
    \ifx#1\mathop\expandafter\displaylimits\fi}
\makeatother

\newcommand{\bigcupdot}{\charfusion[\mathop]{\bigcup}{\cdot}}

\def\cprime{$'$}
\begin{document}

\title{Extension, separation and isomorphic reverse isoperimetry}

\dedicatory{Dedicated with awe to the memory of Jean Bourgain}

\author{Assaf Naor}
%\address{Mathematics Department\\ Princeton University}
\address{Mathematics Department, Princeton University, Fine Hall, Washington Road, Princeton, NJ 08544-1000, USA}
\email{naor@math.princeton.edu}
\thanks{Supported by NSF grant DMS-2054875, BSF grants  2010021 and 2018223, the Packard Foundation and the Simons Foundation. Part of this work  was conducted under the auspices of the Simons Algorithms and Geometry (A\&G) Think Tank.  An extended abstract~\cite{Nao17-SODA} titled ``Probabilistic clustering of high dimensional norms'' that announces discrete and algorithmic aspects of parts of this work   appeared in the proceedings of the 28th annual ACM--SIAM Symposium on Discrete Algorithms.}

%\date{Last modified March 17, 2022}
\maketitle

%\vspace{-0.05in}

\begin{abstract}
The {Lipschitz extension modulus} $\ee(\MM)$ of a metric space $\MM$ is the infimum over those $L\in [1,\infty]$ such that for {any} Banach space $\bfZ$ and {any} $\sub\subset \MM$, {any} $1$-Lipschitz function $f:\sub\to \bfZ$ can be extended to an $L$-Lipschitz function $F:\MM\to \bfZ$. Johnson, Lindenstrauss and Schechtman proved (1986) that  if $\X$ is an $n$-dimensional normed space, then $\ee(\X)\lesssim n$. In the reverse direction, we prove that {every} $n$-dimensional normed space $\X$ satisfies  $\ee(\X)\gtrsim n^c$, where $c>0$ is a universal constant.  Our core technical contribution is a geometric structural result on stochastic clustering  of finite dimensional normed spaces which implies upper bounds on their Lipschitz extension moduli using an extension method of Lee and the author (2005). The separation modulus  of a metric space $(\MM,d_\MM)$ is the infimum over those $\sigma\in (0,\infty]$ such that for any $\Delta>0$ there is a distribution over random partitions of $\MM$ into clusters of diameter at most $\Delta$ such that for every two points $x,y\in \MM$ the probability that they belong to different clusters is at most $\sigma d_\MM(x,y)/\Delta$. We obtain upper and lower bounds on the separation moduli of  finite dimensional normed spaces that relate them to well-studied volumetric invariants (volume ratios and projection bodies). Using these connections, we determine the asymptotic growth rate of the separation moduli of various normed spaces. If $\X$ is  an $n$-dimensional normed space with enough symmetries, then our bounds imply that its separation modulus is equal to $\vr(\X^*)\sqrt{n}$ up to factors of lower order, where $\vr(\X^*)$ is  the volume ratio of the unit ball of the dual of $\X$.  We formulate a conjecture on isomorphic reverse isoperimetric properties of symmetric convex bodies (akin to Ball's reverse isoperimetric theorem (1991), but permitting a non-isometric perturbation in addition to the choice of position) that can be used with our volumetric bounds on the separation modulus to obtain many more exact asymptotic evaluations of the separation moduli of normed spaces. Our estimates on the separation modulus imply asymptotically improved upper bounds on the Lipschitz extension moduli of various classical spaces. In particular, we deduce an  improved upper bound on $\ee(\ell_p^n)$ when $p>2$ that resolves a conjecture of Brudnyi and Brudnyi (2005), and we prove that  $\ee(\ell_\infty^n)\asymp{\sqrt{n}}$,
which is the first time that the  growth rate of $\ee(\X)$ has been evaluated (as $\dim(\X)\to \infty$)  for {\em any} finite dimensional normed space $\X$.
 \end{abstract}

\setcounter{tocdepth}{4}

\newpage

\tableofcontents

\section{Introduction}

Our core technical contribution is a geometric structural result (stochastic clustering) for subsets of finite dimensional normed spaces. It provides new links between nonlinear questions in metric geometry and volumetric issues in convex geometry. An unexpected aspect of our statement is that it contradicts an impossibility result of the well-known work~\cite{CCGGP98} by Charikar, Chekuri, Goel, Guha and Plotkin in the computer science literature, thus leading to  bounds that were previously thought to be impossible. This is reconciled in Section~\ref{sec:cluster}, where we explain the source of the error in~\cite{CCGGP98}.

The aforementioned link opens up a vista that allows one to apply the extensive literature on the linear theory to important and well-studied nonlinear questions. It also raises new fundamental issues within the linear theory that we will only begin to address here. So, in order to fully explain both the history and the ideas and their consequences, we will start with a quick overview of some of our main results that assumes familiarity with standard concepts in the respective areas. We  will then present a gradual and complete introduction to our work that specifies all of the necessary background.

\subsection{Brief highlights of main results}\label{sec:brief} Associate to every separable complete metric space $(\MM,d_\MM)$ two bi-Lipschitz invariants $\ee(\MM), \sep(\MM)\in (0,\infty]$ called, respectively,  the {\em Lipschitz extension modulus of $\MM$} and the {\em separation modulus of $\MM$}, that are defined as follows. The Lipschitz extension modulus of $\MM$ is the infimum over those $L\in (0,\infty]$ such that for {every} Banach space $\bfZ$ and {every} subset $\sub\subset \MM$, {every} $1$-Lipschitz function $f:\sub\to \bfZ$ can be extended to a $\bfZ$-valued $L$-Lipschitz function that is defined on all of $\MM$. The separation modulus of $\MM$ is the infimum over those $\sigma\in (0,\infty]$ such that for any $\Delta>0$ there is a distribution over random partitions\footnote{We are suppressing here measurability issues that are addressed  in Section~\ref{sec:cluster} and Section~\ref{sec:standard}.} of $\MM$ into clusters of diameter at most $\Delta$ such that for every two points $x,y\in \MM$ the probability that they belong to different clusters is at most $\sigma d_\MM(x,y)/\Delta$.

The question of estimating the Lipschitz extension modulus received great scrutiny over the past century; see Section~\ref{sec:ext def}  for an indication of (a small part of) the extensive knowledge on this topic. The separation modulus was introduced by Bartal in the mid-1990s and received a lot of attention in the computer science literature due to its algorithmic applications; see Section~\ref{sec:sep pad} for the history. Its connection to Lipschitz extension was found by Lee and the author~\cite{LN04,LN05}, who proved that $\ee(\MM)\lesssim \sep(\MM)$.

By a well-known theorem of Johnson, Lindenstrauss and Schechtman~\cite{JLS86}, every normed space $\X$ satisfies $\ee(\X)=O(\dim(\X))$. Here we obtain a power-type lower bound on $\ee(\X)$ in terms of $\dim(\X)$.

\begin{theorem}\label{thm:power lower} There is a universal constant $c>0$ such that $\ee(\X)\ge \dim(\X)^c$ for every normed space $\X$.
\end{theorem}
Theorem~\ref{thm:power lower} improves over the previously best-available bound $\ee(\X)\ge e^{c\sqrt{\log\dim(\X)}}$; see Remark~\ref{rem:history lower ext hilbert} for the history of this question. Despite substantial efforts, the asymptotic growth rate (as $\dim(\X)\to \infty$) of $\ee(\X)$ was not previously known (even up to lower order factors)  for {\em any} sequence of normed spaces.

\begin{theorem}\label{thm:ell infty in overview} For every $n\in \N$ we have\footnote{We use the following conventions for asymptotic notation, in addition to the usual $O(\cdot),o(\cdot),\Omega(\cdot)$ notation. Given $a,b>0$, by writing
$a\lesssim b$ or $b\gtrsim a$ we mean that $a\le Cb$ for some
universal constant $C>0$, and $a\asymp b$
stands for $(a\lesssim b) \wedge  (b\lesssim a)$. If we need to allow for dependence on parameters, we indicate it by subscripts. For example, in the presence of an auxiliary parameter $q$, the notation $a\lesssim_q b$ means that $a\le C(q)b$, where $C(q)>0$ may depend only on $q$, and similarly for  $a\gtrsim_q b$ and $a\asymp_q b$.} $\ee\big(\ell_\infty^n\big)\asymp \sqrt{n}$.
\end{theorem}

The previously best-known upper bound on $\ee(\ell_\infty^n)$ was nothing better than the aforementioned general $O(n)$ bound of~\cite{JLS86}. Theorem~\ref{thm:ell infty in overview} is just one instance of our asymptotically improved upper bounds on the Lipschitz extension moduli of many normed spaces of interest; we get e.g.~the best-known bound when $\X=\ell_p^n$ for any $p>2$. Nevertheless, currently $\ell_\infty^n$ is essentially\footnote{The proof of Theorem~\ref{thm:ell infty in overview} artificially gives more such spaces, e.g.~$\ell_\infty^n\oplus \ell_2^n$, or $\ell_\infty^n\oplus \X$ for any normed space $\X$ with $\dim(\X)\le \sqrt{n}$. } the only normed space whose Lipschitz extension modulus is known  up to lower order factors (by Theorem~\ref{thm:ell infty in overview}),  and the same question even for the Euclidean space $\ell_2^n$ remains a longstanding open problem; see Section~\ref{sec:ext def} for more on this.

All of the upper bounds on the Lipschitz extension modulus that we obtain herein use the upper bound on the separation modulus that appears in Theorem~\ref{thm:sep bounds in overview} below. This theorem also contains a new lower bound on the separation modulus, which we will see shows that in several cases of interest our results are a sharp evaluation of the asymptotic growth rate of the separation modulus.\footnote{Our approach also pertains to subsets of normed spaces, e.g.~we will prove that for any $p\in [1,\infty]$, $n\in \N$ and $r\in \n$, the separation modulus of the set of $n$-by-$n$ matrices of rank at most $r$, equipped with the Schatten--von Neumann-$p$ norm, is equal up to lower order factors to $\max\{\sqrt{r},r^{1/p}\}\sqrt{n}$, which is new even in the Euclidean (Hilbert--Schmidt) setting $p=2$. However, for the purpose of   this initial overview we will restrict attention to bounds for the entire space $\X$.}

\begin{theorem}\label{thm:sep bounds in overview}  Let $\X=(\R^n,\|\cdot\|_\X)$ and $\Y=(\R^n,\|\cdot\|_\Y)$ be normed spaces  whose unit balls satisfy $B_\Y\subset B_\X$. Then
\begin{equation}\label{eq:main thm for overview}
\vr(\X^*)\sqrt{n}\lesssim \sep(\X)\lesssim \frac{\diam_{\X^{\textbf{*}}}(\Pi B_\Y)}{\vol_n(B_{\Y})}.
\end{equation}
\end{theorem}
In the left hand side of~\eqref{eq:main thm for overview}, $\vr(\X^*)$ is the {\em volume ratio}~\cite{Sza78,ST-J80} of the dual $\X^*$, i.e., it is the $n$'th root of the ratio of the volume of $B_{\X^{\textbf{*}}}$ and maximal volume of an ellipsoid that is contained in $B_{\X^{\textbf{*}}}$. In the right hand side of~\eqref{eq:main thm for overview}, $\Pi B_\Y$ is the {\em projection body}~\cite{Pet67} of $B_\Y$, and $\diam_{\X^{\textbf{*}}}(\cdot)$ denotes diameter with respect to the metric on $\R^n$ that is induced by $\X^*$. We will recall the definition of a projection body later\footnote{By~\cite{Lud02,Lud05} the mapping that assigns a convex body $K\subset \R^n$ to its projection body $\Pi K$ is characterized axiomatically  as the unique (up to scaling) translation-invariant $\SL_n(\R)$-contravariant Minkowski valuation.} and it suffices to mention now that the  mapping $K\mapsto \Pi K$, which is of central importance in convex geometry (see~\cite{BL88,Lut93,Gar06,Sch14} for an indication of the extensive literature on this topic), associates to every convex body $K\subset \R^n$  a convex body $\Pi K\subset \R^n$ that encodes isoperimetric properties of $K$.

A key contribution of Theorem~\ref{thm:sep bounds in overview} is the role of the auxiliary normed space $\Y$, which appears despite the fact that we are interested in the separation modulus of $\X$. By substituting $\Y=\X$ into the right hand side of~\eqref{eq:main thm for overview} one {\em does} get a meaningful estimate, and in particular the resulting bound is $O(n)$, i.e., \eqref{eq:main thm for overview} implies the  bound of~\cite{JLS86}. However, we will see  that by introducing a suitable perturbation $\Y$ of $\X$, the second inequality in~\eqref{eq:main thm for overview} can sometimes be significantly  stronger than the special case $\Y=\X$. We will exploit this powerful degree of freedom heavily; its geometric significance is discussed in Section~\ref{sec:volumetric ext}.

The previously best-known upper and lower estimates on the separation moduli of normed spaces are due to~\cite{CCGGP98}, where it was proved that $\sep(\ell_1^n)\asymp n$ and $\sep(\ell_2^n)\asymp \sqrt{n}$. By bi-Lipschitz invariance, this implies that any $n$-dimensional normed space $\X$ satisfies
\begin{equation}\label{eq:use CCGGP bounds}
\frac{n}{d_{\BM}(\ell_1^n,\X)}\lesssim \sep(\X)\lesssim d_{\BM}\big(\ell_2^n,\X\big)\sqrt{n},
\end{equation}
where $d_{\BM}(\cdot,\cdot)$ denotes the Banach--Mazur distance. Both of the bounds in~\eqref{eq:use CCGGP bounds} can be inferior to those that follow from Theorem~\ref{thm:sep bounds in overview}. For example, suppose that $n=m^2$ for some $m\in \N$ and consider $\X=\ell_\infty^m(\ell_1^m)$. Then, $d_{\BM}(\X,\ell_1^n)\asymp d_{\BM}(\X,\ell_2^n)\asymp  \sqrt{n}$ by the work~\cite{KS89} of Kwapie\'n and Sch\"utt. Therefore in this case~\eqref{eq:use CCGGP bounds} becomes $\sqrt{n}\lesssim \sep(\X)\lesssim n$, while we will see that~\eqref{eq:main thm for overview} implies that $\sep(\X)\asymp n^{3/4}$.

The following corollary collects examples of applications of Theorem~\ref{thm:sep bounds in overview}  that we will deduce herein.
\begin{corollary}[examples of consequences of Theorem~\ref{thm:sep bounds in overview}]\label{coro:examples of apps} The following statements  hold for any $n\in \N$.
\begin{itemize}
\item For any $p\ge 1$, the separation modulus of $\ell_p^n$ satisfies
\begin{equation}\label{eq:ellp case overview}
\sep\big(\ell_p^n\big)\asymp  n^{\max\left\{\frac12,\frac{1}{p}\right\}}.
\end{equation}
More generally,  let $(\mathbf{E},\|\cdot\|_{\mathbf{E}})$ be  any $n$-dimensional normed space with a $1$-symmetric basis $e_1,\ldots,e_n$. Then,   $\sep(\bfE)$ is equal to the following quantity up to lower order factors:
\begin{equation*}\label{eq:symmetric in overview}
\|e_1+\ldots +e_n\|_\EE\bigg(\max_{k\in \n}\frac{\sqrt{k}}{\|e_1+\ldots+e_k\|_\EE}\bigg).
\end{equation*}
\item For any $p\ge 1$, the separation modulus of the Schatten von-Neumann trace class $\sfS_p^n$ on $\M_n(\R)$ is
\begin{equation}\label{eq:Sp case overview}
\sep\big(\sfS_p^n\big)= n^{\max\left\{1,\frac12+\frac{1}{p}\right\}+o(1)}=\dim\big(\sfS_p^n\big)^{\max\left\{\frac12,\frac14+\frac{1}{2p}\right\}+o(1)}.
\end{equation}
More generally, let $(\mathbf{E},\|\cdot\|_{\mathbf{E}})$ be  any $n$-dimensional normed space with a $1$-symmetric basis $e_1,\ldots,e_n$ and denote its unitary ideal by $\sfS_{\mathbf{E}}=(\M_n(\R),\|\cdot\|_{\sfS_\mathbf{E}})$. Then, $\sep(\sfS_\bfE)$ is equal to the following quantity up to lower order factors:
\begin{equation*}\label{eq:ideal in overview}
 \|e_1+\ldots +e_n\|_\EE\bigg(\max_{k\in \n}\frac{\sqrt{k}}{\|e_1+\ldots+e_k\|_\EE}\bigg)\sqrt{n}.
\end{equation*}
\item For any $p,q\ge 1$, the separation modulus of the $\ell_p^n(\ell_q^n)$ norm on $\M_n(\R)$ is
\begin{equation}\label{eq:ellpellq case overview}
\sep\big(\ell_p^n(\ell_q^n)\big)\asymp n^{\max\left\{1,\frac{1}{p}+\frac{1}{q},\frac12+\frac{1}{p},\frac12+\frac{1}{q}\right\}}
=\dim\big(\ell_p^n(\ell_q^n)\big)^{\max\left\{\frac12,\frac{1}{2p}+\frac{1}{2q},\frac14+\frac{1}{2p},\frac14+\frac{1}{2q}\right\}}.
\end{equation}
\item For any $p,q\ge 1$, the separation modulus of $\M_n(\R)$ equipped with the operator norm $\|\cdot\|_{\ell_p^n\to \ell_q^n}$ from $\ell_p^n$ to $\ell_q^n$ is equal to the following quantity up to lower order factors:
    $$
     \left\{ \begin{array}{ll}
n^{\frac32-\frac{1}{\min\{p,q\}}} &\mathrm{if}\quad  p,q\ge 2,\\
n^{\frac12+\frac{1}{\max\{p,q\}}} &\mathrm{if}\quad  p,q\le 2,\\
n& \mathrm{if}\quad  p\le 2\le q,\\
n^{\max\big\{1,\frac{1}{q}-\frac{1}{p}+\frac12\big\}} &\mathrm{if}\quad  q\le 2\le p.
  \end{array} \right.
    $$
    \item For any $p,q\ge 1$, the separation modulus of the projective tensor product $\ell_p^n\tp \ell_q^n$, i.e., the norm on $\M_n(\R)$ whose unit ball  is the convex hull of  $\{(x_iy_j)\in \M_n(\R);\ (x_1,\ldots, x_n)\in B_{\ell_p^n}\ \wedge \ (y_1,\ldots, y_n)\in B_{\ell_q^n}\}$, is equal to the following quantity up to lower order factors:
$$
\left\{ \begin{array}{ll} n^{\frac32} &\mathrm{if}\quad  \max\{p,q\}\ge 2,\\ n^{1+\frac{1}{\max\{p,q\}}} & \mathrm{if}\quad  \max\{p,q\}\le 2.
  \end{array} \right.
$$
\end{itemize}
\end{corollary}
All of the results in Corollary~\ref{coro:examples of apps} are new, except for the range $1\le p\le 2$ of~\eqref{eq:ellp case overview}, which is due to~\cite{CCGGP98}.  The range $p\in (2,\infty]$ of~\eqref{eq:ellp case overview} is $\sep(\ell_p^n)\asymp \sqrt{n}$, which is incompatible with the statement $\sep(\ell_p^n)\asymp n^{1-1/p}$ of~\cite{CCGGP98}. We will explain the reason why the latter assertion of~\cite{CCGGP98}  is erroneous in Remark~\ref{rem:explain indyk}.

The wealth of knowledge that is available on the volumetric quantities that appear in~\eqref{eq:main thm for overview} leads to new estimates that relate the separation modulus of an $n$-dimensional normed space $\X$ to classical invariants of $\X$. We will derive several such results herein, without attempting to be encyclopedic. As a noteworthy example, we will deduce from the first inequality in~\eqref{eq:main thm for overview} that if $B_\X$ is a polytope with $\rho n$ vertices, then
\begin{equation}\label{eq:sep lower vertices-intro}
\sep(\X)\gtrsim \frac{n}{\sqrt{\log \rho}}.
\end{equation}
We will also deduce that if $T_2(\X)$ denotes the type $2$ constant of $\X$ (see~\eqref{eq:def type cotype} or the survey~\cite{Mau03}), then
\begin{equation}\label{eq:max of sqrt{n} and type-intro}
\sep(\X)\gtrsim \max\Big\{\sqrt{\dim(\X)},T_2(\X)^2\Big\}.
\end{equation}
We will see that both~\eqref{eq:sep lower vertices-intro} and~\eqref{eq:max of sqrt{n} and type-intro} are sharp for the entire range of the relevant parameters (e.g.~in the two extremes, the case $\X=\ell_1^n$ corresponds to $\rho=O(1)$ and $T_2(\X)\asymp\sqrt{n}$ in~\eqref{eq:sep lower vertices-intro} and~\eqref{eq:max of sqrt{n} and type-intro}, respectively, and the case when $\X$ is $O(1)$-isomorphic to $\ell_2^n$ corresponds to $\log \rho\asymp n$ and $T_2(\X)=O(1)$ in~\eqref{eq:sep lower vertices-intro} and~\eqref{eq:max of sqrt{n} and type-intro}, respectively).

\subsubsection{A conjectural isomorphic reverse isoperimetric phenomenon}\label{sec:reverse iso in overview}
The lower bound on $\sep(\X)$ in Theorem~\ref{thm:sep bounds in overview}  is not always sharp. Indeed, consider  $\X=\ell_1^n\oplus \ell_2^n$ for which   $\sep(\X)\asymp n$ yet $\vr(\X^*)\sqrt{\dim(\X)}\asymp n^{3/4}$.  It could be, however, that the upper bound on $\sep(\X)$ in Theorem~\ref{thm:sep bounds in overview} is optimal for every $\X$.

\begin{question}\label{Q:shape optimization} Is it true that the separation modulus of any normed space $\X=(\R^n,\|\cdot\|_\X)$ is bounded above and below by universal constant multiples of the minimum of $\diam_{\X^{\textbf{*}}}(\Pi B_\Y)/\vol_n(B_{\Y})$ over all the  normed spaces $\Y=(\R^n,\|\cdot\|_\Y)$ that satisfy $B_\Y\subset B_\X$?
\end{question}
See Remark~\ref{rem:affine invariance} for an explanation why the minimum that is described in Question~\ref{Q:shape optimization}  is affine invariant, which is necessary for Question~\ref{Q:shape optimization} to make sense, since the separation modulus is a bi-Lipschitz invariant.

For sufficiently symmetric spaces, we expect that the lower bound on $\sep(\X)$ in Theorem~\ref{thm:sep bounds in overview} is sharp.

\begin{conjecture}\label{conj:symmetric volume ratio sep} Every finite dimensional normed space $\X$ with enough symmetries satisfies
\begin{equation}\label{eq:sep volume ratio conjecture}
\sep(\X)\asymp \vr(\X^*) \sqrt{\dim(\X)}.
\end{equation}
\end{conjecture}

The notion of having enough symmetries was introduced in~\cite{GG71}; its definition is recalled in Section~\ref{sec:positions}. We prefer to formulate Conjecture~\ref{conj:symmetric volume ratio sep} using this  notion at the present introductory juncture even though weaker requirements are needed for our purposes because it is a standard assumption in Banach space theory and it suffices for all of the most pressing applications that we have in mind.

The upper bound on $\sep(\X)$ in~\eqref{eq:sep volume ratio conjecture} implies by~\cite{LN05} that $\ee(\X)\lesssim  \vr(\X^*) \sqrt{\dim(\X)}$, which would be a valuable Lipschitz extension theorem due to the fact that estimating the volume ratio is typically tractable given the variety of tools and extensive knowledge that are available in the literature. For example, Milman and Pisier~\cite{MP86} proved (improving by lower-order factors over a  major theorem of Bourgain and Milman~\cite{BM85,BM87}; see also~\cite{Mil87}), that any finite dimensional normed space $\X$ satisfies
\begin{equation}\label{eq:quoteMP}
\vr(\X)\lesssim  C_2(\X)\big(1+\log C_2(\X)\big),
\end{equation}
where $C_2(\X)$ is the cotype $2$ constant of $\X$ (see~\eqref{eq:def type cotype} or the survey~\cite{Mau03}). Therefore, if~\eqref{eq:sep volume ratio conjecture} holds, then
\begin{equation}\label{eq:ee cotype}
\ee(\X)\lesssim C_2(\X)\big(1+\log C_2(\X)\big)\sqrt{\dim(\X)},
\end{equation}
which would be a remarkable generalization of the bound $\ee(\ell_2^n)\lesssim \sqrt{n}$ of~\cite{LN05}.

We expect that Theorem~\ref{thm:sep bounds in overview} already implies Conjecture~\ref{conj:symmetric volume ratio sep}, as expressed in the following conjecture which would yield a positive  answer to Question~\ref{Q:shape optimization}  for normed spaces with enough symmetries.

\begin{conjecture}\label{conj:affine invariant version} If  $\X=(\R^n,\|\cdot\|_\X)$ is a normed space with enough symmetries, then there is  a normed space $\Y=(\R^n,\|\cdot\|_\Y)$ that satisfies $B_\Y\subset B_\X$ and $\diam_{\X^{\textbf{*}}}(\Pi B_\Y)/\vol_n(B_{\Y})\lesssim \vr(\X^*)\sqrt{n}$.
%\begin{equation*}\label{eq:exhibitted already by the theorem}
%B_\Y\subset B_\X \qquad \mathrm{and}\qquad \frac{\diam_{\X^{\textbf{*}}}(\Pi B_\Y)}{\vol_n(B_{\Y})}\lesssim \vr(\X^*)\sqrt{n}.
%\end{equation*}
\end{conjecture}

As an illustrative example of Conjecture~\ref{conj:affine invariant version}, consider $\X=\ell_\infty^n$. Then $\vr((\ell_\infty^n)^*)=\vr(\ell_1^n)=O(1)$. One can compute  that $\Pi B_{\ell_\infty^n}=2^{n-1} B_{\ell_\infty^n}$. Hence, $\diam_{\ell_1^n}(\Pi B_{\ell_\infty^n})/\vol_n(B_{\ell_\infty^n})\asymp n$, so taking $\Y=\ell_\infty^n$ in Theorem~\ref{thm:sep bounds in overview} only gives the bound $\sep(\ell_\infty^n)\lesssim n$. However, we will later see that there exists a normed space $\Y=(\R^n,\|\cdot\|_\Y)$ with $B_\Y\subset B_{\ell_\infty^n}$ for which $\diam_{\ell_1^n}(\Pi B_\Y)/\vol_n(B_\Y)\lesssim \sqrt{n}$. More generally, we will prove that Conjecture~\ref{conj:affine invariant version}  (hence also Conjecture~\ref{conj:symmetric volume ratio sep}, by Theorem~\ref{thm:sep bounds in overview}) holds for any normed space for which the standard basis of $\R^n$ is $1$-symmetric, and we will also see that Conjecture~\ref{conj:affine invariant version} holds up to a logarithmic factor for its unitary ideal.

The minimization in Question~\ref{Q:shape optimization} can be viewed as a shape optimization problem~\cite{HP18} that could potentially be approached using calculus of variations. Given an origin-symmetric convex body $K\subset \R^n$, it asks for  the minimum of the affine invariant functional $L\mapsto \mathsf{outradius}_{K^\circ}(\Pi L)/\vol_n(L)$ over all origin-symmetric convex bodies $L\subset K$, where for any two origin-symmetric convex bodies $A,B\subset \R^n$ we denote the minimum radius of a dilate of $A$ that circumscribes $B$ by $\mathsf{outradius}_A(B)=\min\{r\ge 0:\ B\subset rA\}$, and $K^\circ=\{y\in \R^n:\ \sup_{x\in K} \langle x,y\rangle\le 1\}$ is the polar of $K$. Conjecture~\ref{conj:affine invariant version} asserts that if $K$ has enough symmetries, then this minimum is bounded above and below by universal constant multiples of $\vr(K^\circ)\sqrt{n}$.

The minimization problem in Question~\ref{Q:shape optimization} also has an  isoperimetric flavor. As such, its investigation led us to
formulate the following  conjectural  twist of Ball's reverse isoperimetric phenomenon~\cite{Bal91-reverse}, which we think is a fundamental geometric open question and it would be valuable to understand it even without its   consequences that we derive herein.

The {\em isoperimetric quotient} of a convex body $K\subset \R^n$ is defined (see~\cite[page~269]{Had57} or~\cite{Sch89}) to be
\begin{equation}\label{eq:def iq}
\iq(K)= \frac{\vol_{n-1}(\partial K)}{\vol_n(K)^{\frac{n-1}{n}}}.
\end{equation}
Using this notation, the classical Euclidean isoperimetric theorem states that
\begin{equation}\label{eq:quote ispoperimetric theorem}
\iq(K)\ge \iq\big(B_{\ell_2^n}\big)= \frac{n\sqrt{\pi}}{\Gamma\big(\frac{n}{2}+1\big)^{\frac{1}{n}}}\asymp \sqrt{n},
\end{equation}
The following theorem of Ball~\cite{Bal91-reverse} shows that a judicious  choice of the scalar product on $\R^n$  ensures that the isoperimetric quotient of a convex body can also be bounded from above.

\begin{theorem}[Ball's reverse isoperimetric theorem~\cite{Bal91-reverse}]\label{thm:ball reverse} For every $n\in \N$ and every origin-symmetric convex body $K\subset \R^n$ there exists a linear transformation $S\in \SL_n(\R)$ such that $\iq(SK)\le 2n=\iq([-1,1]^n)$.
\end{theorem}

 We expect that in the isomorphic regime (i.e.,  permitting non-isometric $O(1)$ perturbations), origin-symmetric convex bodies have asymptotically better reverse isoperimetric properties than what is guaranteed by Theorem~\ref{thm:ball reverse}. In fact, we conjecture that if in addition to passing from $K$ to $SK$ for some $S\in \SL_n(\R)$, a $O(1)$-perturbation of $SK$ is allowed, then the isoperimetric quotient can be decreased to be of the same order of magnitude as that of the Euclidean ball.

\begin{conjecture}[isomorphic reverse isoperimetry]\label{isomorphic reverse conj1} There exists a universal constant $c>0$ with the following property. For every $n\in \N$ and every origin-symmetric convex body $K\subset \R^n$, there exist a linear transformation $S\in \SL_n(\R)$ and an origin-symmetric convex body $L\subset \R^n$ with $c SK\subset L\subset SK$ and $\iq(L)\lesssim \sqrt{n}$.
\end{conjecture}

Conjecture~\ref{isomorphic reverse conj1} can be restated analytically as the assertion that any $n$-dimensional normed space is at Banach--Mazur distance $O(1)$ from a normed space whose unit ball has isoperimetric quotient $O(\sqrt{n})$. We will prove that Conjecture~\ref{isomorphic reverse conj1} holds when $K$ is the unit ball of $\ell_p^n$ for any $p\in [1,\infty]$ and $n\in \N$, and we will also see that Conjecture~\ref{isomorphic reverse conj1}  holds up to lower-order factors for any Schatten--von Neumann trace class.

The requirement $L\supset cSK$ of Conjecture~\ref{isomorphic reverse conj1} implies that $\sqrt[n]{\vol_n(L)}\ge c \sqrt[n]{\vol_n(K)}$. So, the following weaker conjecture is implied by Conjecture~\ref{isomorphic reverse conj1}; we will prove it for any $1$-unconditional body.

\begin{conjecture}[weak isomorphic reverse isoperimetry]\label{weak isomorphic reverse conj1} For every $n\in \N$ and every origin-symmetric convex body $K\subset \R^n$ there exist a linear transformation $S\in \SL_n(\R)$ and an origin-symmetric convex body $L\subset SK$ that satisfies $\sqrt[n]{\vol_n(L)}\gtrsim \sqrt[n]{\vol_n(K)}$ and $\iq(L)\lesssim\sqrt{n}$.
\end{conjecture}

In Section~\ref{sec:reverse iso} we will elucidate the relation between the task of  bounding from above the rightmost quantity in~\eqref{thm:sep bounds in overview} and isomorphic reverse isoperimetry. While Conjecture~\ref{isomorphic reverse conj1}  is the strongest version of the isomorphic reverse isoperimetric phenomenon that we expect holds in full generality, we will see that it would suffice to prove its weaker variant Conjecture~\ref{weak isomorphic reverse conj1} for the  purpose of using Theorem~\ref{thm:sep bounds in overview}. In particular, consider the following symmetric version of Conjecture~\ref{weak isomorphic reverse conj1}, which we will prove in Section~\ref{sec:reverse iso} implies Conjecture~\ref{conj:affine invariant version} (hence, using Theorem~\ref{thm:sep bounds in overview},  it also implies Conjecture~\ref{conj:symmetric volume ratio sep}).

\begin{conjecture}[symmetric version of Conjecture~\ref{weak isomorphic reverse conj1}]\label{weak isomorphic reverse conj1-symmetric} For every $n\in \N$, if $\X=(\R^n,\|\cdot\|_\X)$ is a normed space with enough symmetries whose isometry group is a subgroup of the orthogonal group $\mathsf{O}_n\subset \GL_n(\R)$, then there is a normed space $\Y=(\R^n,\|\cdot\|_\Y)$ with $B_\Y\subset B_\X$ and $\sqrt[n]{\vol_n(B_\Y)}\gtrsim \sqrt[n]{\vol_n(B_\X)}$ such that $\iq(B_\Y)\lesssim \sqrt{n}$.
\end{conjecture}

The only difference between Conjecture~\ref{weak isomorphic reverse conj1} and  Conjecture~\ref{weak isomorphic reverse conj1-symmetric} is that if we impose the further requirement that $K$ is the unit ball of a normed space with enough symmetries whose isometry group consists only of orthogonal matrices, then we are naturally conjecturing that $S$ can be taken to be the identity matrix, i.e., there is no need to change the standard Euclidean structure on $\R^n$.

We will prove Conjecture~\ref{weak isomorphic reverse conj1-symmetric}  for various spaces, including $\ell_p^n(\ell_q^n)$ for any $p,q\ge 1$ and $n\in \N$, and any finite dimensional space with a $1$-symmetric basis. Also, we will show that Conjecture~\ref{weak isomorphic reverse conj1-symmetric} holds up to a factor of $O(\sqrt{\log n})$ for any unitarily invariant norm on $\M_n(\R)$. In general, an argument that was shown to us by B.~Klartag and E.~Milman and is included in Section~\ref{sec:log weak} (see also Section~\ref{sec:intersection}) shows that Conjecture~\ref{weak isomorphic reverse conj1}  and Conjecture~\ref{weak isomorphic reverse conj1-symmetric} hold up to a factor of $O(\log n)$. We will see that these results lead to  Corollary~\ref{coro:examples of apps}, and in general we will deduce that Conjecture~\ref{conj:affine invariant version}, and hence, thanks to Theorem~\ref{thm:sep bounds in overview}, also Conjecture~\ref{conj:symmetric volume ratio sep}, hold up to lower order factors. Thus, we will obtain the following theorem.

\begin{theorem}\label{thm:symmetric volume ratio sep up to lower order} $\sep(\X)\asymp \vr(\X^*) \dim(\X)^{\frac12+o(1)}$ for any normed space $\X$ with enough symmetries.
\end{theorem}

\begin{comment}
\begin{conjecture}[symmetric version of Conjecture~\ref{isomorphic reverse conj1}]\label{isomorphic reverse conj1-symmetric} For every $n\in \N$, if $\X=(\R^n,\|\cdot\|_\X)$ is a normed space with enough symmetries, then there is a normed space $\Y=(\R^n,\|\cdot\|_\Y)$ with enough symmetries satisfying $B_\Y\subset B_\X$ and $\iq(B_\Y)\lesssim \sqrt{n}$.
\end{conjecture}

We will show that if a normed space  $\X$ with enough symmetries satisfies Conjecture~\ref{weak isomorphic reverse conj1-symmetric}, then it satisfies  Conjecture~\ref{conj:affine invariant version}, and hence by  Theorem~\ref{thm:sep bounds in overview} it  satisfies Conjecture~\ref{conj:symmetric volume ratio sep}.
\end{comment}

Assuming Conjecture~\ref{weak isomorphic reverse conj1-symmetric}, it is possible to compute the exact asymptotic growth rate of the  separation moduli of several important matrix spaces. For example,  if  Conjecture~\ref{weak isomorphic reverse conj1-symmetric} holds for $\sfS_\infty^n$, then  we will see that the $o(1)$ term in~\eqref{eq:Sp case overview} could be removed altogether, i.e.,
\begin{equation}\label{eq:sharp lp sep assuming cojecture}
\forall (p,n)\in [1,\infty]\times \N,\qquad \sep\big(\sfS_p^n\big)\asymp n^{\max\left\{1,\frac12+\frac{1}{p}\right\}}.
\end{equation}
Also, assuming Conjecture~\ref{weak isomorphic reverse conj1-symmetric} the lower order factors in the last two statements of Corollary~\ref{coro:examples of apps} could be removed, namely we will see that Conjecture~\ref{weak isomorphic reverse conj1-symmetric}  implies that the separation modulus of $\M_n(\R)$ equipped with the operator norm $\|\cdot\|_{\ell_p^n\to \ell_q^n}$ from $\ell_p^n$ to $\ell_q^n$ satisfies
%$\|T\|_{\ell_p^n\to \ell_q^n}=\sup_{x\in B_{\ell_p^n}\setminus\{0\}} \|Tx\|_{\ell_q^n}$ satisfies
\begin{equation}\label{eq:sep of operator from ellp to ellq}
\sep\big(\M_n(\R),\|\cdot\|_{\ell_p^n\to \ell_q^n} \big)\asymp  \left\{ \begin{array}{ll}
n^{\frac32-\frac{1}{\min\{p,q\}}} &\mathrm{if}\quad  p,q\ge 2,\\
n^{\frac12+\frac{1}{\max\{p,q\}}} &\mathrm{if}\quad  p,q\le 2,\\
n& \mathrm{if}\quad  p\le 2\le q,\\
n^{\max\big\{1,\frac{1}{q}-\frac{1}{p}+\frac12\big\}} &\mathrm{if}\quad  q\le 2\le p,
  \end{array} \right.
\end{equation}
and the separation modulus of  the projective tensor product $\ell_p^n\tp \ell_q^n$ satisfies
\begin{equation}\label{eq:projective in overview}
\sep\big(\ell_p^n\tp \ell_q^n \big)\asymp \left\{ \begin{array}{ll} n^{\frac32} &\mathrm{if}\quad  \max\{p,q\}\ge 2,\\ n^{1+\frac{1}{\max\{p,q\}}} & \mathrm{if}\quad  \max\{p,q\}\le 2.
  \end{array} \right.
\end{equation}
Remark~\ref{rem:cut and factor} describes  ramifications of these conjectural statements to norms of algorithmic importance.

\bigskip
\noindent{\bf Roadmap.} The rest of the Introduction  effectively restarts the description of the present work, with many more details/definitions/background/ideas of proofs, than what we have included above. We organized the introductory material in this way since this work pertains to multiple mathematical disciplines, including Banach spaces, convex geometry, nonlinear functional analysis, metric embeddings, extension of functions, and theoretical computer science. The backgrounds of potential readers are therefore  varied, so even though the above overview achieves the goal of presenting the main results quickly, it inevitably includes  terminology that is not familiar to some. The aforementioned organizational choice makes the ensuing discussion accessible.  Additional background can be found in the  monographs~\cite{LT77,MS86,Tom89} (Banach space theory), \cite{BL00} (nonlinear functional analysis), \cite{Mat02,Ost13} (metric embeddings), \cite{BB12} (extension of functions), as well as the references that are cited throughout.

While the ensuing extended introductory text is not short, it achieves more than merely a description of the results, history, concepts and methods: It also contains groundwork that is needed for the subsequent sections. Thus, reading the Introduction will lead to a thorough conceptual understanding of the contents, leaving to the remaining sections considerations that are for the most part more technical.

We will start by focusing on the classical Lipschitz extension problem because it is more well known than the stochastic clustering issues that lead to most of our new results on Lipschitz extension, and also because it requires less technicalities (e.g.~a suitable measurability setup) than our subsequent treatment of stochastic clustering. Throughout the Introduction (and beyond), we will formulate conjectures and questions that are valuable even without the links to Lipschitz extension and clustering that are derived herein. After the Introduction, the rest of this work will be organized thematically as follows. Section~\ref{sec:lower} is devoted to proofs of our various lower bounds, namely impossibility results that rule out the existence of extensions and clusterings with certain properties. Section~\ref{sec:prelim random part main} and Section~\ref{sec:upper} deal with positive results about random partitions. Specifically, Section~\ref{sec:prelim random part main} is of a more foundational nature as it describes the concepts, basic constructions, and proofs of measurability statements that are needed for later applications in the infinitary setting (of course, measurability can be ignored  for statements about finite sets). Section~\ref{sec:upper} analyses in the case of normed spaces a periodic version of a commonly used   randomized partitioning technique called {\em iterative ball partitioning}, and computes optimally (up to universal constant factors) the probabilities of its separation and padding events. Section~\ref{sec:ext} shows how to pass from random partitions to Lipschitz extension, by adjusting to the present setting the method that was developed in~\cite{LN05}.  Section~\ref{sec:ext}  also contains further foundational results on Lipschitz extension, as well questions and conjectures that are of independent interest. Section~\ref{sec:volumes and cone measure} contains a range of volume and surface area estimates that are needed in conjunction with the theorems of the preceding sections in order to deduce new Lipschitz extension and stochastic clustering results for various normed spaces and their subsets. Section~\ref{sec:log weak} proves that Conjecture~\ref{weak isomorphic reverse conj1}  and Conjecture~\ref{weak isomorphic reverse conj1-symmetric} hold up to a factor of $O(\log n)$, and also shows that the approach that leads to this result cannot fully resolve Conjecture~\ref{weak isomorphic reverse conj1-symmetric}.

\subsection{Basic notation}\label{sec:basic defs}

Given a metric space $(\MM,d_\MM)$, a point $x\in \MM$ and a radius $r\ge 0$, the corresponding {\em closed} ball is denoted $B_\MM(x,r)=\{y\in \MM:\ d_\MM(y,x)\le r\}$. If $(\X,\|\cdot\|_{\X})$ is a Banach space (in this work, all  vector spaces are over the real scalars unless stated otherwise), then  denote by $B_{\X}$ the unit ball centered at the origin. Under this notation we have  $B_{\X}=B_{\X}(0,1)$ and $B_{\X}(x,r)=x+rB_{\X}$ for every $x\in X$ and $r\ge 0$. %The unit sphere of $X$ is denoted $S_{\X}=\partial B_{\X}=\{x\in X:\ \|x\|_{\X}=1\}$.

If $(\MM,d_\MM), (\NN,d_\NN)$ are metric spaces and $\psi:\MM\to \NN$, then for $\sub\subset \MM$ the Lipschitz constant of $\psi$ on $\sub$ is denoted $\|\psi\|_{\Lip(\sub;\NN)}\in [0,\infty]$. Thus, if $\sub$ contains at least two points, then
$$
\|\psi\|_{\Lip(\sub;\NN)}\eqdef\sup_{\substack{x,y\in \sub\\ x\neq y}}\frac{d_\NN\big(\psi(x),\psi(y)\big)}{d_\MM(x,y)}.
$$
In the special case $\NN=\R$ we will use the simpler notation $\|\psi\|_{\Lip(\sub;\R)}=\|\psi\|_{\Lip(\sub)}$.

If $(\X,\|\cdot\|_{\X}), (\Y,\|\cdot\|_{\Y})$ are isomorphic Banach spaces, then their Banach--Mazur distance $d_{\BM}(\X,\Y)$ is the infimum of the products of the operator norms $\|T\|_{\X\to \Y}$ and $\|T^{-1}\|_{\Y\to \X}$ over all possible linear isomorphisms $T:\X\to \Y$. The (bi-Lipschitz) distortion of  a metric space $(\MM,d_\MM)$ into a metric space $(\NN,d_\NN)$, denoted $\cc_{(\NN,d_\NN)}(\MM,d_\MM)$ or  $\cc_{\NN}(\MM)$ if the underlying metrics are clear from the context, is the infimum over those $D\in [1,\infty]$ for which there exists a mapping $\phi:\MM\to \NN$ and (a scaling factor) $\lambda>0$ such that
\begin{equation}\label{eq:def distortion}
\forall  x,y\in \MM,\qquad \lambda d_\MM(x,y)\le d_\NN\big(\phi(x),\phi(y)\big)\le D\lambda d_\MM(x,y).
\end{equation}

Fix $n\in \N$. Throughout what follows, $\R^n$ will be always be endowed with its standard Euclidean structure, i.e., with the scalar product  $\langle x,y\rangle=x_1y_1+\ldots+x_ny_n$ for $x=(x_1,\ldots,x_n),y=(y_1,\ldots,y_n)\in \R^n$.   Given $z\in \R^n\setminus\{0\}$, the orthogonal projection onto its orthogonal hyperplane $z^\perp=\{x\in \R^n:\ \langle x,z\rangle= 0\}$ will be denoted $\proj_{z^\perp}: \R^n\to \R^n$. For $0<s\le n$, the $s$-dimensional Hausdorff measure of a closed subset $A\subset \R^n$ is denoted $\vol_s(A)$. Integration with respect to the $s$-dimensional Hausdorff measure is indicated by $\mathrm{d} x$. If $0<\vol_s(A)<\infty$ and $f:A\to \R$ is continuous, then write $
\fint_A f(x)\ud x=\vol_s(A)^{-1}\int_A f(x)\ud x.$

Given a normed space $(\X,\|\cdot\|_{\X})$ and $p\in [1,\infty]$,  $\ell_p^n(\X)$ is the vector space $\X^n$ equipped with the norm
$$
\forall x=(x_1,\ldots,x_n)\in \X^n,\qquad \|x\|_{\ell_p^n(\X)}=\big(\|x_1\|_\X+\ldots+\|x_n\|_\X\big)^{\frac{1}{p}},
$$
where for $p=\infty$ this is understood  to be $\|x\|_{\ell_\infty^n(\X)}=\max_{j\in \n}\|x_j\|_\X$. It is common to use the simpler notation $\ell_p^n=\ell_p^n(\R)$ and we write as usual $S^{n-1}=\partial B_{\ell_2^n}$.     The Schatten--von Neumann trace class $\sfS_p^{n}$ is the $(n^2$-dimensional) space of all $n$ by $n$ real matrices $\mathsf{M}_n(\R)$, equipped with the norm that is defined by
$$
\forall T\in \mathsf{M}_n(\R),\qquad \|T\|_{\sfS_p^{n}}=\Big(\trace\big((TT^*)^{\frac{p}{2}}\big)\Big)^{\frac{1}{p}}=\Big(\trace\big((T^*T)^{\frac{p}{2}}\big)\Big)^{\frac{1}{p}},
$$
where $\|T\|_{\sfS_\infty^{n}}=\|T\|_{\ell_2^n\to\ell_2^n}$ is the operator norm of $T$ when it is viewed as a linear operator from $\ell_2^n$ to $\ell_2^n$.

\subsection{Lipschitz extension}\label{sec:ext def} As we recalled  in Section~\ref{sec:brief}, one associates to every metric space $(\MM,d_\MM)$ a bi-Lipschitz invariant\footnote{ The assertion that $\ee(\MM)$ is a bi-Lipschitz invariant refers to the fact that the definition immediately implies that if $(\mathcal{N},d_\mathcal{N})$ is another metric space into which $(\MM,d_\MM)$ admits a bi-Lipschitz embedding, then $\ee(\MM)\le \cc_{\mathcal{N}}(\MM)\ee(\mathcal{N})$.}, called the { Lipschitz extension modulus} of $(\MM,d_\MM)$ and denoted $\ee(\MM,d_\MM)$ or $\ee(\MM)$ if the metric is clear from the context, by defining it to be the infimum over those $K\in [1,\infty]$ with the property that for {\em every} nonempty subset $\sub\subset \MM$, {\em every} Banach space $(\bfZ,\|\cdot\|_{\bfZ})$ and {\em every} Lipschitz function $f:\sub\to \bfZ$ there is a mapping $F:\MM\to \bfZ$ that extends $f$, i.e., $F(x)=f(x)$ whenever $x\in \sub$, and $\|F\|_{\Lip(\MM,\bfZ)}\le K\|f\|_{\Lip(\sub,\bfZ)}$; see Figure~\ref{fig:commuting}. All of the ensuing extension theorems hold for a larger class of target metric spaces that need not necessarily be Banach spaces, including Hadamard spaces and Busemann nonpositively curved spaces~\cite{BH99}, or more generally spaces that posses a conical geodesic bicombing (see e.g.~\cite{Des15}). This greater generality will be discussed in Section~\ref{sec:ext}, but we prefer at this introductory juncture to focus on the more classical and highly-studied setting of Banach space targets.

\begin{figure}[h]
\centering
\fbox{
\begin{minipage}{6.25in}
$$
\xymatrix @C=7.7pc {
\MM \ar@{-->}[rd]^{F} \\
\strut   \sub \ar@{^{(}->}[u]^{\mathsf{Id}_{\sub\to \MM}} \ar[r]^{f} & \bfZ
}
$$
\caption{\em \small Given $K\ge 1$, the assertion that the Lipschitz extension modulus of  a metric space $\MM$ satisfies $\ee(\MM)<K$ means that for {\bf \em all} subsets $\sub\subset \MM$, {\bf \em all} Banach spaces $\bfZ$ and {\bf \em all} $1$-Lipschitz mappings $f:\sub\to \bfZ$, there is a $K$-Lipschitz mapping $F:\MM\to \bfZ$ such that the above diagram commutes, where $\mathsf{Id}_{\sub\to \MM}:\sub\to \MM$ is the formal inclusion. }\label{fig:commuting}
\end{minipage}
}
\end{figure}

When $(\X,\|\cdot\|_{\X})$ is a finite dimensional normed space, the currently best-available general bounds on the quantity $\ee(\X)$ in terms of $\dim(\X)$ are contained the following theorem.

\begin{theorem}\label{thm:complemented subspace}
There is a universal constant $c>0$ such that for any finite dimensional normed space $\X$,
\begin{equation}\label{eq:upper lower eX}
\dim(\X)^c\lesssim \ee(\X)\lesssim \dim(\X).
\end{equation}
\end{theorem}

The bound $\ee(\X)\lesssim \dim(\X)$ in~\eqref{eq:upper lower eX} is a famous result of Johnson, Lindenstrauss and Schechtman~\cite{JLS86}, which they proved by cleverly refining the classical extension method of Whitney~\cite{Whi34}; different  proofs of this estimate were found by Lee and the author~\cite{LN05} as well as by Brudnyi and Brudnyi~\cite{BB06} (see also the discussion in the paragraph following equation~\eqref{eq:bad direction cone} below).   It remains a major longstanding open problem to determine whether or not the bound of~\cite{JLS86} could be improved to $\ee(\X)=o(\dim(\X))$.

The new content of Theorem~\ref{thm:complemented subspace} is the lower bound on $\ee(\X)$, which improves over the previously known bound $\ee(\X)\ge \exp(c\sqrt{\log \dim(\X)})$; see Remark~\ref{rem:history lower ext hilbert} for the history of this question. It is a very interesting open problem to determine the supremum over those  $c$ for which Theorem~\ref{thm:complemented subspace} holds.\footnote{Our proof of the lower bound on $\ee(\X)$ of Theorem~\ref{thm:complemented subspace} shows that this supremum is at least $\frac{1}{12}$; see equation~\eqref{eq:1/12}.} More generally, it is natural to aim to evaluate the precise power-type behavior of $\ee(\X)$ as $\dim(\X)\to \infty$ for specific (sequences of) finite dimensional normed spaces $\X$. However, prior to the present work and despite many efforts over the years, this was not achieved for {\em any} finite dimensional normed space whatsoever.
\begin{theorem}[restatement of Theorem~\ref{thm:ell infty in overview}]\label{thm:ell infty}
For every $n\in \N$ we have $\ee\big(\ell_\infty^n\big)\asymp\sqrt{n}$.
\end{theorem}

The bound $\ee(\ell_\infty^n)\gtrsim \sqrt{n}$ follows from  a  combination of~\cite[Theorem~4]{BB05} and~\cite[Theorem~1.2]{BB07-2}.  The new content of Theorem~\ref{thm:ell infty} is the the upper bound  $\ee(\ell_\infty^n)\lesssim \sqrt{n}$ (and, importantly, the extension procedure that leads to it; see below). The previously best-known upper bound on $\ee(\ell_\infty^n)$ was the aforementioned $O(n)$ estimate of~\cite{JLS86}.  The question of evaluating the asymptotic behavior of $\ee(\ell_p^n)$ as $n\to\infty$ for each $p\in [1,\infty]$ is natural and longstanding;  it was stated  in~\cite[Problem~2]{BB05} and reiterated in~\cite[Section~4]{BB07}, \cite[Problem~1.4]{BB07-2} and~\cite[Problem~8.14]{BB12}. Theorem~\ref{thm:ell infty} answers this question when $p=\infty$. The upper bound on $\ee(\ell_\infty^n)$ of Theorem~\ref{thm:ell infty} is a special case of a general extension criterion that provides the best-known Lipschitz extension results in other settings (including for $\ell_p^n$ when $p>2$), but we chose to state it separately because it yields the first (and currently essentially only) family of normed spaces for which the growth rate  of their Lipschitz extension moduli has been determined.

\begin{remark}\label{rem:holder} It is meaningful to study extension of $\theta$-H\"older functions for any $0<\theta\le 1$. Namely, one can analogously define the $\theta$-H\"older extension modulus of a metric space $(\MM,d_\MM)$, denoted $\ee^\theta(\MM)$. Alternatively, this notion falls into the above Lipschitz-extension framework because one can define
\begin{equation}\label{eq:define holder ext}
\ee^\theta(\MM)\eqdef \ee\big(\MM,d_\MM^\theta\big).
\end{equation}
The results that we obtain herein also yield improved estimates on $\theta$-H\"older extension moduli; see Corollary~\ref{cor:holder ext}. However, when $\theta<1$ we never get a matching lower bound (the reason why we can do better in the Lipschitz regime $\theta=1$ is essentially due to the fact that Lipschitz functions are differentiable almost everywhere).  For example, in the setting of Theorem~\ref{thm:ell infty} we get the upper bound
\begin{equation}\label{eq:theta holder l infty}
\forall \theta\in (0,1],\qquad \ee^\theta\big(\ell_\infty^n\big)\lesssim n^{\frac{\theta}{2}},
\end{equation}
but the best lower bound on $\ee^\theta(\ell_\infty^n)$ that we are at present able to prove is
\begin{equation}\label{eq:lower theta holder l infty}
\ee^\theta\big(\ell_\infty^n\big)\gtrsim n^{\max\left\{\frac{\theta}{4},\frac{\theta}{2}+\theta^2-1\right\}}=\left\{\begin{array}{ll}n^{\frac{\theta}{4}}&\mathrm{if}\quad 0\le \theta\le \frac{\sqrt{65}-1}{8},\\
n^{\frac{\theta}{2}+\theta^2-1}&\mathrm{if}\quad  \frac{\sqrt{65}-1}{8}\le \theta\le 1.\end{array}\right.
\end{equation}
We conjecture that $\ee^\theta(\ell_\infty^n)\asymp_\theta n^{\frac{\theta}{2}}$, but proving this for $\theta<1$  would likely require a genuinely new idea.
\end{remark}

\begin{question}
Despite its utility in many cases, the extension method that underlies Theorem~\ref{thm:ell infty} does not yield improved bounds for some important spaces, including notably $\ell_1^n$ and $\ell_2^n$. Thus, determining the asymptotic behavior of $\ee(\ell_1^n)$ and $\ee(\ell_2^n)$ as $n\to \infty$ remains a tantalizing open question. Specifically, the currently best-known bounds on $\ee(\ell_1^n)$ are
\begin{equation}\label{eq:ell1 what's known}
\sqrt{n}\lesssim \ee\big(\ell_1^n\big)\lesssim n,
\end{equation}
where the first inequality in~\eqref{eq:ell1 what's known} is due to Johnson and Lindenstrauss~\cite{JL84} and the second inequality in~\eqref{eq:ell1 what's known} is the aforementioned general upper bound of~\cite{JLS86} on the Lipschitz extension modulus of {\em any} $n$-dimensional normed space.  The currently best-known bounds in the Hilbertian setting  are
\begin{equation}\label{eq:ell2 what's known}
\sqrt[4]{n}\lesssim \ee\big(\ell_2^n\big)\lesssim \sqrt{n},
\end{equation}
where the first inequality in~\eqref{eq:ell2 what's known}  is due to Mendel and the author~\cite{MN13-bary} (a different proof of this lower bound on $\ee(\ell_2^n)$ follows from~\cite{Nao20-almost}), and the second inequality in~\eqref{eq:ell2 what's known}  is from~\cite{LN05}.
\end{question}

By the bi-Lipschitz invariance of the Lipschitz extension modulus,  the second inequality in~\eqref{eq:ell2 what's known} implies  the following bound from~\cite{LN05}, which holds for every finite dimensional normed space $\X$.
\begin{equation}\label{eq:dist to hilbert ext}
\ee(\X)\lesssim d_{\BM}\big(\X,\ell_2^{\dim(\X)}\big) \sqrt{\dim(\X)}.
\end{equation}
This refines the  upper bound on $\ee(\X)$ in~\eqref{eq:upper lower eX} because $d_{\BM}(\X,\ell_2^{\dim(\X)})\le \sqrt{\dim(\X)}$ by John's theorem~\cite{Joh48}.

\begin{remark} In the context of the aforementioned question if  the bound $\ee(\X)\lesssim \dim(\X)$ of~\cite{JLS86} is optimal, by~\eqref{eq:dist to hilbert ext} we see that $\ee(\X)=o(\dim(\X))$ unless the Banach--Mazur distance between $\X$ and Euclidean space is of order $\sqrt{\dim(\X)}$. Structural properties of such spaces of extremal distance to Euclidean space have been studied in~\cite{MW78,Pis78-extremal,Bou81-Pl,JS82-maximal,ATTJ05}; see also chapters~6 and~7 of~\cite{Tom89}. In particular, the Mil{\cprime}man--Wolfson theorem~\cite{MW78} asserts that this holds if and only if $\X$ has a subspace of dimension $k=k(\dim (\X))$ whose Banach--Mazur distance to $\ell_1^k$ is $O(1)$, where $\lim_{n\to \infty} k(n)=\infty$.
\end{remark}

As $d_{\BM}(\ell_p^n,\ell_2^n)\asymp n^{|p-2|/(2p)}$ for all $n\in \N$ and $p\in [1,\infty]$ (see~\cite[Section~8]{JL01}), it follows from~\eqref{eq:dist to hilbert ext} that
\begin{equation}\label{eq:previous lp ext}
\ee\big(\ell_p^n\big)\lesssim \left\{\begin{array}{ll}n^{\frac{1}{p}}&\mathrm{if}\ p\in [1,2],\\
n^{1-\frac{1}{p}}&\mathrm{if}\ p\in [2,\infty].\end{array}\right.
 \end{equation}
\eqref{eq:previous lp ext} was the previously best-known upper bound on $\ee(\ell_p^n)$, and here we improve it for every $p>2$.

 \begin{theorem}\label{thm:p>2} For every $n\in \N$ and every $p\in [1,\infty]$ we have
 $
 \ee\big(\ell_p^n\big)\lesssim n^{\max\left\{\frac12,\frac{1}{p}\right\}}.
 $
\end{theorem}
Theorem~\ref{thm:ell infty} is the case $p=\infty$ of Theorem~\ref{thm:p>2}. We do not know if Theorem~\ref{thm:p>2} is optimal (perhaps up to lower order factors) as $n\to \infty$ for fixed $p\in [2,\infty)$, but we conjecture that this is indeed the case, which would resolve~\cite[Problem~2]{BB05}. The currently best-known lower bound on $\ee(\ell_p^n)$ for every $p\in [1,\infty]$ is
\begin{equation}\label{eq:e lp lower best-known}
\ee\big(\ell_p^n\big)\gtrsim \left\{\begin{array}{ll} n^{\frac{1}{p}-\frac12}&\mathrm{if}\ 1\le p\le \frac43,\\
\sqrt[4]{n}&\mathrm{if}\ \frac43\le p\le 2,\\
n^{\frac{1}{2p}}&\mathrm{if}\ 2\le p\le 3,\\
n^{\frac12-\frac{1}{p}}&\mathrm{if}\ 3\le p\le \infty.\end{array}\right.
\end{equation}
 A lower bound on $\ee(\ell_p^n)$ that coincides with~\eqref{eq:e lp lower best-known} when $p\in [1,4/3]\cup[3,\infty]$ is stated in Corollary~8.12 of~\cite{BB12}, but~\cite[Corollary~8.12]{BB12} is weaker than~\eqref{eq:e lp lower best-known} when $4/3<p<3$. The reason for this is that the lower bound of~\cite{MN13-bary} on $\ee(\ell_2^n)$ that appears in~\eqref{eq:ell2 what's known} was not available when~\cite{BB12} was written, but~\eqref{eq:e lp lower best-known} for $4/3<p<3$ follows quickly by combining the first inequality in~\eqref{eq:ell2 what's known}  with~\cite{FLM77}; see Remark~\ref{rem:e ellp lower bounds with FLM}.

 \begin{remark}\label{rem:BB conjecture}
 Theorem~\ref{thm:p>2} resolves negatively a conjecture  that A.~Brudnyi and Y.~Brudnyi posed as Conjecture~5 in~\cite{BB05}. They conducted a comprehensive study of the {\em linear} extension problem for  real-valued Lipschitz functions, where one considers for a metric space $(\MM,d_\MM)$ a quantity $\lambda(\MM)$ which is defined the same as  $\ee(\MM)$, but with the further requirements that the function $f$ is real-valued and that the extended function $F$ depends linearly on $f$. Namely,  $\lambda(\MM)$ is the infimum over those $K\in [1,\infty]$ such that for every $\sub\subset \MM$ there is a linear operator $\mathsf{Ext}_\sub:\Lip(\sub)\to \Lip(\MM)$ that assigns to every Lipschitz function $f:\sub\to \R$ a function $\mathsf{Ext}_\sub f:\MM\to \R$  satisfying $\mathsf{Ext}_\sub f(s)=f(s)$ for every $s\in \sub$, and  $$\|\mathsf{Ext}_\sub f\|_{\Lip(\MM)}\le K\|f\|_{\Lip(\sub)}.$$ They  also considered a natural variant of this quantity when $\MM=\X$ is a Banach space, denoted $\lambda_\conv(\X)$, which is defined almost identically to $\lambda(\X)$ except that now the subset $\sub$ is only allowed to be any {\em convex} subset of $\X$ rather than a  subset of $\X$ without any additional restriction. Conjecture~5 in~\cite{BB05} states that
 \begin{equation}\label{eq:BB conj}
 \forall (p,n)\in [1,\infty]\times \N,\qquad  \lambda\big(\ell_p^n\big)\asymp_p \lambda_{\conv}\big(\ell_p^n\big)\sqrt{n}.
 \end{equation}
Theorem~\ref{thm:p>2}  implies that this conjecture is {\em false} for every $p\in (2,\infty]$. Indeed, the asymptotic behavior of $\lambda_{\conv}(\ell_p^n)$ was evaluated in~\cite[Theorem~2.19]{BB07}, where it was shown that
 $$
 \forall  p\in [1,\infty],\qquad \lambda_{\conv}\big(\ell_p^n\big)\asymp n^{\big|\frac12-\frac{1}{p}\big|}.
 $$
Consequently, $\lambda_{\conv}(\ell_p^n)\sqrt{n}\asymp n^{1-\frac{1}{p}}$ when $p>2$. Next, in~\cite{BB07-2} a quantity $\nu(\MM)$ was associated  to a metric space $(\MM,d_\MM)$ by defining it almost identically to the definition of $\ee(\MM)$, except that the target Banach space $\bfZ$ is allowed to be any {\em finite dimensional} Banach space rather than any Banach space whatsoever. By definition $\nu(\MM)\le \ee(\MM)$, but actually   $\lambda(\MM)=\nu(\MM)$ thanks to~\cite[Theorem~1.2]{BB07-2} (see the work~\cite{AP20} of Ambrosio and Puglisi for more on this ``linearization phenomenon''). Using these results in combination with Theorem~\ref{thm:p>2}, we see that for every $p\in (2,\infty]$,   as $n\to \infty$ we have
 $$\lambda\big(\ell_p^n\big)=\nu\big(\ell_p^n\big)\le\ee\big(\ell_p^n\big)\lesssim \sqrt{n}=o\Big(n^{1-\frac{1}{p}}
 \Big).$$
Thus,   $\lambda(\ell_p^n)=o\big(\lambda_{\conv}(\ell_p^n)\sqrt{n}\big)$ as $n\to \infty$ for any $p>2$, in contrast to the conjecture~\eqref{eq:BB conj} of~\cite{BB05}.
\end{remark}

Prior to passing to the general Lipschitz extension theorem that underlies the new results that were described above, we will further illustrate its utility by stating one more concrete  application. For each $p\in [1,\infty]$ and $n\in \N$, if  $k\in \n$, then let $(\ell_p^n)_{\le k}$ denote the subset of $\R^n$ consisting of those vectors with at most $k$ nonzero coordinates, equipped with the metric that is inherited from $\ell_p^n$.
\begin{theorem}\label{thm:sparse} For every $p\in [1,\infty]$, every $n\in \N$ and every $k\in \n$ we have
$
\ee\big((\ell_p^n)_{\le k}\big)\lesssim  k^{\max\left\{\frac{1}{p},\frac12\right\}}.
$\end{theorem}

Theorem~\ref{thm:p>2}  is the special case $k=n$ and $p\ge 2$ of Theorem~\ref{thm:sparse}.  If $1\le p\le 2$ and $k=n$, then Theorem~\ref{thm:sparse}  is the estimate~\eqref{eq:previous lp ext}, which is the best-known upper bound on $\ee(\ell_p^n)$ for $p$ in this range. However, for general $k\in \n$ Theorem~\ref{thm:sparse}  yields a refinement of~\eqref{eq:previous lp ext}  in the entire range $p\in [1,\infty]$ which does not seem to follow from previously known results.  In particular, the case $p=2$ of Theorem~\ref{thm:sparse} becomes
\begin{equation}\label{eq:ell 2 sparse}
\ee\big((\ell_2^n)_{\le k}\big)\lesssim  \sqrt{k}.
\end{equation}
 Even though~\eqref{eq:ell 2 sparse}  concerns a  Euclidean setting, its proof relies on a   construction that employs a multi-scale partitioning scheme using balls of an auxiliary metric on $\R^n$ that differs from the ambient Euclidean metric. The utility of such a non-Euclidean geometric reasoning despite the Euclidean nature of the question being studied is  discussed further in Section~\ref{sec:volumetric ext}.

\subsection{A volumetric upper bound on the Lipschitz extension modulus}\label{sec:volumetric ext} We will prove that Theorem~\ref{thm:sparse} (hence also its special cases Theorem~\ref{thm:ell infty} and Theorem~\ref{thm:p>2})
is a consequence of Theorem~\ref{thm:XY version ext} below,  which is a Lipschitz extension theorem for subsets of finite dimensional normed spaces in terms of volumes of hyperplane projections of their unit balls. Throughout what follows, for dealing with volumetric notions we will adhere to the following conventions. Given $n\in \N$, when we say that $\X=(\R^n,\|\cdot\|_{\X})$ is a normed space we mean that the underlying vector space is $\R^n$ and that $\|\cdot\|_{\X}:\R^n\to [0,\infty)$ is a norm on $\R^n$. This is, of course, always achievable by fixing any scalar product on an $n$-dimensional normed space. While the ensuing statements hold in this setting, i.e., for an {arbitrary} identification of $\X$ with $\R^n$, a judicious choice of such an identification  is  beneficial; the discussion of this important matter is postponed to Section~\ref{sec:positions} because it is not needed for the initial description of the main results. We will continue using the notation $B_{\X}=\{x\in \R^n:\ \|x\|_{\X}\le 1\}$ for the unit ball of $\X$. Also, given $\sub\subset \R^n$ we denote by $\sub_{\X}$ the metric space consisting of the set $\sub$ equipped with the metric that is inherited from $\|\cdot\|_{\X}$. This notation is important for us because we will crucially need to simultaneously consider more than one  norm on $\R^n$.

\begin{theorem}\label{thm:XY version ext} Suppose that $n\in \N$ and that $\X=(\R^n,\|\cdot\|_{\X})$ and $\Y=(\R^n,\|\cdot\|_{\Y})$ are two normed spaces.  Then, for every $\sub\subset \R^n$ we have the following upper bound on the Lipschitz extension modulus of $\sub_\X$.
\begin{equation}\label{eq:ext XY version}
\ee(\sub_{\X})\lesssim  \bigg( \sup_{\substack{x,y\in \sub\\x\neq y}}\frac{\|x-y\|_{\X}}{\|x-y\|_{\Y}}\bigg)\sup_{\substack{x,y\in \sub\\x\neq y}} \bigg( \frac{\vol_{n-1}\big(\proj_{(x-y)^\perp}B_{\Y}\big)}{\vol_n(B_{\Y})}\cdot \frac{\|x-y\|_{\ell_2^n}}{\|x-y\|_{\X}}\bigg).
\end{equation}
\end{theorem}

We will next discuss the geometric meaning of Theorem~\ref{thm:XY version ext} and derive some of its consequences, including Theorem~\ref{thm:sparse}. Firstly, by homogeneity the case $\sub=\R^n$ of~\eqref{eq:ext XY version} becomes
\begin{align}\label{eq:sub is Rn}
\ee(\X)\lesssim \big(\sup_{y\in \partial B_\Y} \|y\|_{\X}\big)\sup_{x\in \partial B_\X}\bigg(\frac{\vol_{n-1}\big(\proj_{x^\perp}B_{\Y}\big)}{\vol_n(B_{\Y})}\|x\|_{\ell_2^n}\bigg).
\end{align}
The quantity $\sup_{y\in \partial B_\Y} \|y\|_{\X}$ in~\eqref{eq:sub is Rn} is the norm $\|\Id_n\|_{\Y\to \X}$ of the identity matrix $\Id_n\in \M_n(\R)$ as an operator from $\Y$ to $\X$. Alternatively,  $\sup_{y\in \partial B_\Y} \|y\|_{\X}=\diam_\X(B_\Y)/2$, where for each $\sub\subset \R^n$ we denote its diameter with respect to the metric that $\X$ induces by $\diam_\X(\sub)=\sup_{x,y\in \sub} \|x-y\|_\X$.

Given a convex body $K\subset \R^n$, let $\Pi^*K\subset \R^n$ be the polar of the {\em projection body} of $K$, which is defined to be the unit ball of the norm $\|\cdot\|_{\Pi^{\textbf *}K}$ on $\R^n$ that is given by setting
\begin{equation}\label{eq:use cauchy}
\forall  x\in \R^n\setminus \{0\},\qquad \|x\|_{\Pi^{\textbf *}K}\eqdef \frac12 \int_{\partial K}\big|\langle x,N_K(y)\rangle\big|\ud y =\vol_{n-1}\big(\proj_{x^\perp}K\big)\|x\|_{\ell_2^n},
\end{equation}
where $N_K(y)\in S^{n-1}$ denotes the unit outer normal  to $\partial K$ at $y\in \partial K$ (which is uniquely defined almost everywhere with respect to the surface-area measure on $\partial K$), and the final equality in~\eqref{eq:use cauchy} is the  Cauchy projection formula (see e.g.~\cite[Appendix~A]{Gar06}). The projection body $\Pi K$ of $K$ is the polar of $\Pi^*K$. These important notions were introduced by Petty~\cite{Pet67}. When $\X=(\R^n,\|\cdot\|_\X)$ is a normed space let  $\Pi^*\X$ be the normed space whose unit ball is $\Pi^*B_\X$. Let $\Pi \X= (\Pi^*\X)^*$ be the normed space whose unit ball is $\Pi B_\X$.

 By substituting~\eqref{eq:use cauchy} into~\eqref{eq:sub is Rn} we get the following interpretation of our bound on $\ee(\X)$ in terms of analytic and geometric properties of projection bodies; it is worthwhile to state it as a separate corollary even though it is only a matter of notation because of its intrinsic interest and also because   these alternative viewpoints were useful for guiding some of the subsequent considerations.

\begin{corollary}\label{coro:analytic and geometric restatements} Any two normed spaces $\X=(\R^n,\|\cdot\|_{\X}), \Y=(\R^n,\|\cdot\|_{\Y})$ satisfy
\begin{align}\label{eq:adjoint}
\begin{split}
\ee(\X)&\lesssim  \frac{\diam_\X(B_\Y)\diam_{\Pi^{\textbf *} \Y}(B_\X)}{\vol_n(B_{\Y})}
\asymp  \frac{\|\Id_n\|_{\Y\to \X}\|\Id_n\|_{\X\to \Pi^{\textbf *} \Y}}{\vol_n(B_{\Y})}\\&=\frac{\|\Id_n\|_{\X\to \Y}\|\Id_n\|_{\Pi \Y\to \X^{\textbf *}}}{\vol_n(B_{\Y})}\asymp \frac{\diam_\X(B_\Y)\diam_{\X^{\textbf{*}}}(\Pi B_\Y)}{\vol(B_{\Y})}.
\end{split}
\end{align}
\end{corollary}
The penultimate step in~\eqref{eq:adjoint} is  duality (the norm of an operator equals the norm of its adjoint) and the final quantity in~\eqref{eq:adjoint} relates  Theorem~\ref{thm:XY version ext} to the second estimate in Theorem~\ref{thm:sep bounds in overview}.

\begin{remark}\label{rem:affine invariance} Corollary~\ref{coro:analytic and geometric restatements} has the right {\em affine invariance}. For  $S\in \SL_n(\R)$ let $S\X=(\R^n,\|\cdot\|_{S\X})$ be the normed space whose unit ball is $SB_\X$; equivalently, $\|x\|_{S\X}=\|S^{-1}x\|_\X$ for every $x\in \R^n$. Then $\X$ and $S\X$ are isometric as metric spaces, so $\ee(S\X)=\ee(\X)$.
%By definition, $\vr((S\X)^*)=\vr(\X^*)$, but the right hand side of~\eqref{eq:main thm for overview} requires a bit more scrutiny.
We have $(S\X)^*=(S^*)^{-1} \X^*$ (by definition), and  $\Pi (SB_\Y)=(S^*)^{-1}\Pi B_\Y$ by~\cite{Pet67}. From this we see that $\diam_{(S \X)^{\textbf{*}}}(\Pi B_{S\Y})=\diam_{\X^{\textbf{*}}}(\Pi B_\Y)$. Thus, the minimum of the right hand side of~\eqref{eq:adjoint} over all normed spaces $\Y=(\R^n,\|\cdot\|_\Y)$ is also invariant under the action of $\SL_n(\R)$.
\end{remark}

The special case of Theorem~\ref{thm:XY version ext} in which the normed space $\Y$  coincides with the given normed space $\X$  is in itself a nontrivial bound on the Lipschitz extension modulus. Examining this special case first will help elucidate how the idea arose to introduce an auxiliary space $\Y$ that may differ from $\X$, and why this can yield stronger estimates. If $\X=\Y$, then the bound~\eqref{eq:ext XY version} becomes
\begin{equation}\label{eq:ext volume ratio}
\ee(\sub_{\X})\lesssim \sup_{\substack{x,y\in \sub\\x\neq y}} \bigg( \frac{\vol_{n-1}\big(\proj_{(x-y)^\perp}B_{\X}\big)}{\vol_n(B_{\X})}\cdot \frac{\|x-y\|_{\ell_2^n}}{\|x-y\|_{\X}}\bigg).
\end{equation}
Correspondingly, the bound~\eqref{eq:sub is Rn} becomes
\begin{equation}\label{eq:ext volume ratio X}
\ee(\X)\lesssim \sup_{z\in \partial B_{\X}}\bigg( \frac{\vol_{n-1}\big(\proj_{z^\perp}B_{\X}\big)}{\vol_n(B_{\X})}\|z\|_{\ell_2^n}\bigg)=\frac{\diam_{\Pi^{\textbf *} \X}(B_\X)}{\vol_n(B_\X)}.
\end{equation}
Even these weaker estimates suffice to obtain new results, e.g.~we will see that this is so if $2\le p=O(1)$ and $\X=\ell_p^n$. However,  as we will soon explain, \eqref{eq:ext volume ratio X} does not imply an upper bound on $\ell_\infty^n$ that is better than the aforementioned general bound of~\cite{JLS86}. Despite this shortcoming of~\eqref{eq:ext volume ratio} and~\eqref{eq:ext volume ratio X} relative to~\eqref{eq:ext XY version}, it is worthwhile to state these special cases of Theorem~\ref{thm:XY version ext} separately because they are simpler than~\eqref{eq:ext XY version} and hence perhaps somewhat easier to remember. Moreover, a na\"ive way to enhance the applicability of~\eqref{eq:ext volume ratio}  is to leverage the fact that the Lipschitz extension modulus is a bi-Lipschitz invariant, so that
 $$\ee(\sub_{\X})\le \|\mathsf{Id}_n\|_{\Lip(\sub_{\Y},\sub_{\X})} \|\mathsf{Id}_n\|_{\Lip(\sub_{\X},\sub_{\Y})}\ee(\sub_{\Y}).$$ Consequently, by estimating $\ee(\sub_{\Y})$ through~\eqref{eq:ext volume ratio} we formally deduce from~\eqref{eq:ext volume ratio} that
\begin{equation}\label{eq:use bi lip in simpler}
\ee(\sub_{\X})\lesssim \bigg( \sup_{\substack{x,y\in \sub\\x\neq y}}\frac{\|x-y\|_{\X}}{\|x-y\|_{\Y}}\bigg)\bigg( \sup_{\substack{x,y\in \sub\\x\neq y}}\frac{\|x-y\|_{\Y}}{\|x-y\|_{\X}}\bigg)\cdot\sup_{\substack{x,y\in \sub\\x\neq y}} \bigg(\frac{\vol_{n-1}\big(\proj_{(x-y)^\perp}B_{\Y}\big)}{\vol_n(B_{\Y})}\cdot \frac{\|x-y\|_{\ell_2^n}}{\|x-y\|_{\Y}}\bigg).
\end{equation}
We do  not see how to deduce Theorem~\ref{thm:p>2} and Theorem~\ref{thm:sparse} from~\eqref{eq:use bi lip in simpler}. However, we will show that~\eqref{eq:use bi lip in simpler}  suffices for proving Theorem~\ref{thm:ell infty} (as well as some other results that will be presented later). In summary, even the case of Theorem~\ref{thm:XY version ext} in which the auxiliary space $\Y$ coincides with $\X$ is valuable, but Theorem~\ref{thm:XY version ext} does not follow from merely combining  its special case $\Y=\X$ with bi-Lipschitz invariance.

Given a normed space $\X=(\R^n,\|\cdot\|_\X)$ and  $z\in \R^n\setminus\{0\}$, the quantity
\begin{equation}\label{eq:volume of cone}
\frac{1}{n}\vol_{n-1}\big(\proj_{z^\perp}B_{\X}\big)\|z\|_{\ell_2^n}
\end{equation}
is equal to the volume of the cone
\begin{equation}\label{eq:def cone}
\cone_z(B_{\X})\eqdef\conv \big(\{z\}\cup \proj_{z^\perp}B_{\X}\big)\subset \R^n
\end{equation}
whose base is the $(n-1)$-dimensional convex set $\proj_{z^\perp}B_{\X}\subset z^\perp$ and whose apex is $z$. In~\eqref{eq:def cone} and throughout what follows,  $\conv(\cdot)$ denotes the convex hull. Thus, the estimate~\eqref{eq:ext volume ratio X} can be restated as follows.
\begin{equation}\label{eq:bad direction cone}
\ee(\X)\lesssim n \sup_{z\in \partial B_{\X}} \frac{\vol_{n}\big(\cone_z(B_{\X})\big)}{\vol_n(B_{\X})}.
\end{equation}

Through~\eqref{eq:bad direction cone} we see that the geometric interpretation of the ``bad spaces'' $\X$ for~\eqref{eq:ext volume ratio X} is that these are the spaces that have a ``pointy direction'' $z\in \partial B_{\X}$ for which the volume of the cone  $\cone_z(B_{\X})$ is a significant fraction of the volume of $B_{\X}$. Examples will be presented next, but note first that  a short geometric argument  (see the proof of~\cite[Lemma~5.1]{GNS12}) shows that $\vol_n(\cone_z(B_{\X}))\le \vol_n(B_{\X})/2$, so the right hand side of~\eqref{eq:bad direction cone} is at most $n/2$. Hence, \eqref{eq:ext volume ratio X} is a refinement of the classical bound $\ee(\X)\lesssim n$ of~\cite{JLS86}.

Nevertheless, a ``vanilla'' application of~\eqref{eq:ext volume ratio X} does not yield an asymptotically  better estimate than that of~\cite{JLS86}  even when $\X=\ell_\infty^n$. Indeed, $B_{\ell_\infty^n}=[-1,1]^n$ and a simple argument (see~\cite{CF86}) shows that
\begin{equation}\label{eq:ell infty proj formula}
\forall z\in \R^n\setminus\{0\},\qquad \frac{\vol_{n-1}\big(\proj_{z^\perp}[-1,1]^n\big)}{\vol_n([-1,1]^n)}=\frac{\|z\|_{\ell_1^n}}{2\|z\|_{\ell_2^n}}.
\end{equation} So, by considering the all $1$'s vector  $z=\1_{\n}\in \partial B_{\ell_\infty^n}$ we see that for $\X=\ell_\infty^n$ the right hand side of~\eqref{eq:ext volume ratio X} is at least $n/2$. The right hand side of~\eqref{eq:ext volume ratio X}  is  at least $n/2$ when $\X=\ell_1^n$, as seen by taking  $z=(1,0,\ldots,0)\in \partial B_{\ell_1^n}$. Such ``problematic" directions $z\in \partial B_{\X}$ can sometimes be the overwhelming majority of $\partial B_{\X}$. Consider Ball's counterexample~\cite{Bal91} to the Shepard Problem~\cite{She64}, which states that for any $n\in \N$ there is a normed space $\X=(\R^n,\|\cdot\|_{\X})$ such that $\vol_n(B_{\X})=1$ yet $\vol_{n-1}(\proj_{z^\perp}B_{\X})\gtrsim \sqrt{n}$ for {\em every} $z\in S^{n-1}$. Since $\vol_n(B_{\ell_2^n})\le (3/\sqrt{n})^n$ while $\vol_n(B_{\X})= 1$,  the proportion of those $z\in  \partial B_{\X}$ for which $\|z\|_{\ell_2^n}\ge\sqrt{n}/4$ tends to $1$ as $n\to\infty$ (exponentially fast). Any such $z$ satisfies $\|z\|_{\ell_2^n}\vol_{n-1}(\proj_{z^\perp}B_{\X})/\vol_n(B_{\X})\gtrsim n$.

These obstacles can sometimes be overcome by perturbing the given normed space $\X$ prior to invoking~\eqref{eq:ext volume ratio X}, i.e., by using of Theorem~\ref{thm:XY version ext}  with a suitably chosen auxiliary normed space $\Y=(\R^n,\|\cdot\|_\Y)$. In particular, since by H\"older's inequality $\|\cdot\|_{\ell_2^n}\le n^{1/2-1/p}\|\cdot\|_{\ell_p^n}$ when $p\ge 2$, Theorem~\ref{thm:p>2} follows from a substitution of the space $\Y_p^n$ of Theorem~\ref{prop:rounded cube} below into Theorem~\ref{thm:XY version ext} (with $\X=\ell_p^n$), or even into~\eqref{eq:use bi lip in simpler}.

\begin{theorem}\label{prop:rounded cube} For any $n\in \N$ and $p\in [1,\infty]$ there is a normed space $\Y_p^n=(\R^n,\|\cdot\|_{\Y_p^n})$ that satisfies
\begin{equation}\label{eq:round cube}
\forall x\in \R^n\setminus \{0\},\qquad  \|x\|_{\Y_p^n}\asymp \|x\|_{\ell_p^n}\qquad \mathrm{and}\qquad \frac{\vol_{n-1}\big(\proj_{x^\perp}B_{\Y_p^n}\big)}{\vol_n\big(B_{\Y_p^n}\big)}\lesssim n^{\frac{1}{p}}.
\end{equation}
\end{theorem}

The case $p=\infty$ of Theorem~\ref{prop:rounded cube} implies Theorem~\ref{thm:sparse} through an application of Theorem~\ref{thm:XY version ext}. Indeed, fix $p\ge 1$ and $n\in \N$. Suppose that $x,y\in (\ell_p^n)_{\le k}$ for some $k\in \n$. Then $x-y$ has at most $2k$ nonzero coordinates. Therefore, if $\Y_\infty^n$ is as in Theorem~\ref{prop:rounded cube}, then by H\"older's inequality we have
\begin{equation}\label{eq:holder on small supprt}
(2k)^{-\max\left\{\frac12-\frac{1}{p},0\right\}}\|x-y\|_{\ell_2^n}\le \|x-y\|_{\ell_p^n}\le (2k)^{\frac{1}{p}} \|x-y\|_{\ell_\infty^n}\asymp k^{\frac{1}{p}} \|x-y\|_{\Y_\infty^n}.
\end{equation}
Theorem~\ref{thm:sparse} follows by substituting these bounds and the case $p=\infty$ of~\eqref{eq:round cube} into~\eqref{eq:ext XY version}. Observe that we would have obtained the weaker bound $\ee((\ell_p^n)_{\le k})\lesssim k^{1/p+1/2}$ if we used~\eqref{eq:use bi lip in simpler} instead of~\eqref{eq:ext XY version}.

If $p=O(1)$, then one can take $\Y_p^n=\ell_p^n$ in Theorem~\ref{prop:rounded cube}.  In fact, the direction $z\in S^{n-1}$ at which
\begin{equation}\label{eq:proj max}
\max_{z\in S^{n-1}} \vol_{n-1}\big(\proj_{z^{\perp}}B_{\ell_p^n}\big)
\end{equation}
is attained was determined by Barthe and the author in~\cite{BN02}. This result implies that
\begin{equation}\label{eq:crude BN}
\forall p\ge 1,\qquad \max_{z\in S^{n-1}}\frac{\vol_{n-1}\big(\proj_{z^\perp}B_{\ell_p^n}\big)}{\vol_n(B_{\ell_p^n})}\asymp n^{\frac{1}{p}} \sqrt{\min\{p,n\}}.
\end{equation}
As~\cite{BN02} computes~\eqref{eq:proj max} exactly,  the implicit constant factors in~\eqref{eq:crude BN} can be evaluated, but in the present context such  precision is of secondary importance. While~\eqref{eq:crude BN} follows from~\cite{BN02} (see the deduction in~\cite{Nao17-SODA}), we will give a self-contained proof  of~\eqref{eq:crude BN} in Section~\ref{sec:volumes and cone measure} as a special case of a more general result that we will use for other purposes as well. In the range $q\in (2,\infty)$, a different approach to computing~\eqref{eq:proj max} was found in~\cite{KRZ04}. Earlier methods for estimating~\eqref{eq:proj max} with worse lower order  factors are due to~\cite{Sch89} and~\cite{Mul90}; the latter is an adaptation of an idea (used for related purposes) in~\cite{Bou87-maximal}.

For each $k\in \n$, by applying~\eqref{eq:ext XY version} with $\Y=\ell_q^n$ for some $q\ge p$, using~\eqref{eq:crude BN}  with $p$ replaced by $q$, and  optimizing the resulting bound over $q$, one obtains a result that matches Theorem~\ref{thm:sparse} up to unbounded lower order factors. More precisely,  the best that one can get with this approach (up to universal constant factors)  is when $q=\max\{2\log(n/k),p\}$ if $p\le \log (2k)$. If $p\ge \log(2k)$, then use~\eqref{eq:ext XY version} with $\Y=\ell_{\log(2k)}^n$.

Theorem~\ref{prop:rounded cube} provides an auxiliary space $\Y$ for which a use of~\eqref{eq:ext XY version} removes the above lower order factors, and yields a sharp result when $p=\infty$ (we conjecture that it is sharp for any $p\ge 2$). Regardless of whether we apply~\eqref{eq:ext XY version} with the space $\Y=\Y_\infty^n$ of Theorem~\ref{prop:rounded cube}  or with $\Y=\ell_q^n$ for a suitable choice of $q\ge p$, we have seen  that without using an auxiliary space $\Y\neq \ell_p^n$ in~\eqref{eq:ext XY version} we do not come close to such results.

Even though in Theorem~\ref{thm:XY version ext} we are interested in extending functions that are Lipschitz in the metric that is induced by the given norm $\|\cdot\|_\X$, the underlying reason for the bounds of Theorem~\ref{thm:XY version ext} is a partitioning scheme  (to be described below) that iteratively carves out  balls in the metric that is induced by the auxiliary norm $\|\cdot\|_{\Y}$.  So, the  perturbation of $\X$ into $\Y$  amounts to exhibiting  a Lipschitz extension operator through the use of a multi-scale construction that utilizes geometric shapes that differ from balls in the ambient metric.  This strategy is feasible because the quantity $\ee(\sub_{\X})$ in the left hand side of~\eqref{eq:ext volume ratio} is a bi-Lipschitz invariant, while the volumes that appear in the right hand side of~\eqref{eq:ext volume ratio} scale exponentially in $n$. Hence, by passing to an equivalent norm one could hope to reduce the right hand side of~\eqref{eq:ext volume ratio} significantly, while not changing the left hand side of~\eqref{eq:ext volume ratio} by too much.

This perturbative approach is decisively useful for $\X=\ell_\infty^n$. When one unravels the ensuing proofs,  the upper bound on $\ee(\ell_\infty^n)$ of Theorem~\ref{thm:ell infty} arises from a multi-scale construction of an extension operator (using a {\em gentle partition of unity}~\cite{LN05}) that utilizes a partition of space that is obtained by iteratively  removing sets of the form $x+rB_{\Y_\infty^n}$, where $\Y_\infty^n$ is as in Theorem~\ref{prop:rounded cube}. If one carries out the same  procedure while using  balls of the intrinsic metric of $\ell_\infty^n$ (namely, hypercubes $x+r[-1,1]^n$ in place of  $x+rB_{\Y_\infty^n}$, which look like hypercubes with ``rounded corners''), then only the weaker  bound  $\ee(\ell_\infty^n)\lesssim n$ is obtained. We already mentioned that such a phenomenon even occurs in the proof of the Euclidean estimate~\eqref{eq:ell 2 sparse}.

\smallskip
The following two examples describe further uses of Theorem~\ref{thm:XY version ext}; we will work out several more later.

%We will end this section by presenting a couple of additional examples of applications of Theorem~\ref{thm:XY version ext}.

\begin{example}\label{rem:from alpha paper} In the forthcoming work~\cite{NS18-extension}, the author and Schechtman prove (for an application to metric embedding theory) the following asymptotic evaluation of the maximal volumes of hyperplane projections of the unit balls of the Schatten--von Neumann trace classes.

\begin{equation}\label{eq:quote with gid}
\forall q\ge 1,\qquad \max_{A\in \M_n(\R)\setminus\{0\}}\frac{\vol_{n^2-1}\big(\proj_{A^\perp}B_{\sfS_q^n}\big)}{\vol_{n^2}\big(B_{\sfS_q^n}\big)}\asymp n^{\frac12+\frac{1}{q}}\sqrt{\min\{q,n\}}.
\end{equation}
Upon substitution into Theorem~\ref{thm:XY version ext}, this yields the following new estimates on the Lipschitz extension moduli of  Schatten--von Neumann trace classes, which holds for every $p\ge 1$ and every integer $n\ge 2$.
\begin{equation}\label{eq:all of schatten}
\ee\big(\sfS_p^n\big)\lesssim \left\{\begin{array}{ll}n^{\frac12+\frac{1}{p}}&\mathrm{if}\ p\in [1,2],\\
n\sqrt{\min\{p,\log n\}}& \mathrm{if}\ p\in [2,\infty].\end{array}\right.
\end{equation}
Indeed, by H\"older's inequality  $\|\cdot\|_{\sfS_2^n}\le n^{\max\left\{0,\frac12-\frac{1}{p}\right\}}\|\cdot\|_{\sfS_p^n}$, so~\eqref{eq:all of schatten} for $p\le \log n$ follows from a substitution of these point-wise bounds and~\eqref{eq:quote with gid}  when $q=p$   into the case $\X=\Y=\sfS_p^n$ of Theorem~\ref{thm:XY version ext}. The case $p\ge \log n$ of~\eqref{eq:all of schatten} follows from the same reasoning using~\eqref{eq:quote with gid} when $q=\log n$ and  Theorem~\ref{thm:XY version ext} for $\X=\sfS_p^n$ and $\Y=\sfS_{q}^n$, since in this case $d_{\BM}(\sfS_p^n,\sfS_q^n)\lesssim 1$. Note that, since $\dim(\sfS_p^n)=n^2$, for every $p\in [1,\infty]$ the bound on $\ee(\sfS_p^n)$ in~\eqref{eq:all of schatten} is $o(\dim(\sfS_p^n))$, i.e., it is asymptotically better than what follows from~\cite{JLS86}.

More generally, given $p\ge 1$, an integer $n\ge 2$ and $r\in \{3,\ldots,n\}$, let $(\sfS_p^n)_{\le r}$ be the set of $n$ by $n$ matrices of rank at most $r$, equipped with the metric  inherited from $\sfS_p^n$. Then, \eqref{eq:all of schatten} has the following strengthening.
\begin{equation}\label{eq:schatten extension}
\ee\big((\sfS_p^{n})_{\le r}\big)\lesssim r^{\max\left\{\frac{1}{p},\frac12\right\}}\sqrt{n} \cdot \left\{\begin{array}{ll}\sqrt{\max\left\{\log\left(\frac{ n}{r}\right),p\right\}}&\mathrm{if}\ p\le \log r,\\
\sqrt{\log n}&\mathrm{if}\ p\ge \log r.\end{array}\right.
\end{equation}
To justify~\eqref{eq:schatten extension}, apply Theorem~\ref{thm:XY version ext} with $\X=\sfS_p^n$ and $\Y=\sfS_q^n$ for some $q\ge p$ while using~\eqref{eq:quote with gid}, and optimize the resulting bound over $q$. Specifically, since for any $A,B\in (\sfS_p^n)_{\le r}$ the matrix $A-B$ has at most $2r$ nonzero singular values,  by H\"older's inequality we have
$$
\|A-B\|_{\sfS_2^n}\le (2r)^{\max\left\{0,\frac12-\frac{1}{p}\right\}}\|A-B\|_{\sfS_p^n}\qquad\mathrm{and}\qquad \|A-B\|_{\sfS_p^n} \le (2r)^{\frac{1}{p}-\frac{1}{q}}\|A-B\|_{\sfS_q^n}.
$$
 In combination with~\eqref{eq:quote with gid}, we therefore get the following bound from~\eqref{eq:ext XY version}.
\begin{equation}\label{eq:to optimize over q ehn p small}
\ee\big((\sfS_p^n)_{\le r}\big)\lesssim \bigg( \sup_{\substack{A,B\in (\sfS_p^n)_{\le r}\\A\neq B}}\frac{\|A-B\|_{\sfS_p^n}}{\|A-B\|_{\sfS_q^n}}\bigg)\sup_{\substack{A,B\in (\sfS_p^n)_{\le r}\\A\neq B}} \bigg( n^{\frac12+\frac{1}{q}}\sqrt{q}\frac{\|A-B\|_{\sfS_2^n}}{\|A-B\|_{\sfS_p^n}}\bigg)\lesssim r^{\frac{1}{p}-\frac{1}{q}}n^{\frac12+\frac{1}{q}}\sqrt{q}r^{\max\left\{\frac{1}{2}-\frac{1}{p},0\right\}}.
\end{equation}
The $q\ge p$ that minimizes the right hand side of~\eqref{eq:to optimize over q ehn p small} is $\max\{2\log(n/r),p\}$, yielding~\eqref{eq:schatten extension} when $p\le \log r$. If $p\ge \log r$, then $\|A-B\|_{\sfS_p^n}\asymp \|A-B\|_{\sfS_{\log r}^n}$ for every $A,B\in (\sfS_p^n)_{\le r}$, so~\eqref{eq:schatten extension} reduces to its special case $p=\log r$.

We conjecture that it is possible to replace the logarithmic factor in~\eqref{eq:schatten extension} by a universal constant, i.e.,
\begin{equation}\label{eq:S_p rank k without logs}
\ee\big((\sfS_p^{n})_{\le r}\big)\lesssim r^{\max\left\{\frac{1}{p},\frac12\right\}}\sqrt{n}.
\end{equation}
As we will see in Section~\ref{sec:reverse iso}, Conjecture~\ref{conj:S infty reverse iso} below is equivalent to the symmetric  isomorphic reverse isoperimetry conjecture (see Conjecture~\ref{conj:reverse iso when canonical})  for $\M_n(\R)$ equipped with the operator norm, which is an especially interesting special case of this much more general conjectural phenomenon; by reasoning as we did in the above deduction of Theorem~\ref{thm:sparse} from (the special case $p=\infty$ of) Theorem~\ref{prop:rounded cube} (recall the discussion immediately following~\eqref{eq:holder on small supprt}), a positive answer to Conjecture~\ref{conj:S infty reverse iso} would imply~\eqref{eq:S_p rank k without logs}.
\begin{conjecture}\label{conj:S infty reverse iso} For every  $n\in \N$ there exists a normed space $\Y=(\M_n(\R),\|\cdot\|_\Y)$ such that for every nonzero $n$ by $n$ matrix $A\in \M_n(\R)\setminus\{0\}$ we have  $\|A\|_\Y\asymp \|A\|_{\sfS_\infty^n}$ and $\vol_{n^2-1}(\proj_{A^\perp}B_{\Y})\lesssim \vol_{n^2}(B_{\Y})\sqrt{n}$.
\end{conjecture}
\end{example}

\begin{example}\label{ex:lplq} Since the  $\ell_\infty^n(\ell_\infty^n)$ norm on $\M_n(\R)$ is isometric to $\ell_\infty^{n^2}$, by Theorem~\ref{prop:rounded cube} there is a normed space $\Y=(\M_n(\R),\|\cdot\|_\Y)$ that satisfies $\|A\|_{\ell_\infty^n(\ell_\infty^n)}\le \|A\|_\Y\lesssim \|A\|_{\ell_\infty^n(\ell_\infty^n)}$ for every $A\in \M_n(\R)$, and
$$
\max_{A\in \M_n(\R)\setminus \{0\}}\frac{\vol_{n^2-1}\big(\proj_{A^\perp}B_{\Y}\big)}{\vol_{n^2}(B_{\Y})}=O(1).
$$
By H\"older's inequality, for every $p,q\in [1,\infty]$ and $A\in \M_n(\R)$ we  have
$$
\|A\|_{\ell_p^n(\ell_q^n)}\le n^{\frac{1}{p}+\frac{1}{q}}\|A\|_{\ell_\infty^n(\ell_\infty^n)}\le n^{\frac{1}{p}+\frac{1}{q}}\|A\|_{\Y}\qquad\mathrm{and}\qquad \|A\|_{\ell_2^n(\ell_2^n)}\le n^{\max\left\{\frac12-\frac{1}{p},0\right\}+\max\left\{\frac12-\frac{1}{q},0\right\}}\|A\|_{\ell_p^n(\ell_q^n)}.
$$
Therefore, Theorem~\ref{thm:XY version ext} gives the Lipschitz extension  bound
\begin{equation}\label{eq:e upper lpn lqn}
\ee\big(\ell_p^n(\ell_q^n)\big)\lesssim n^{\frac{1}{p}+\frac{1}{q}+\max\left\{\frac12-\frac{1}{p},0\right\}+\max\left\{\frac12-\frac{1}{q},0\right\}}=n^{\max\left\{1,\frac{1}{p}+\frac{1}{q},\frac12+\frac{1}{p},\frac12+\frac{1}{q}\right\}}.
\end{equation}
As in the case of $\ell_p^n$, we get~\eqref{eq:e upper lpn lqn} if $p,q=O(1)$ by using Theorem~\ref{thm:XY version ext}  with $\Y=\X=\ell_p^n(\ell_p^n)$, but otherwise we need to work with an auxiliary space $\Y\neq \X$ as above. Specifically, in Section~\ref{sec:volumes and cone measure} we will prove the following asymptotic evaluation of the maximal volume of hyperplane projections of the unit ball of $\ell_p^n(\ell_q^n)$:
\begin{equation}\label{eq:max prj lpn ellqn}
\max_{A\in \M_n(\R)\setminus \{0\}}\frac{\vol_{n^2-1}\big(\proj_{A^\perp}B_{\ell_p^n(\ell_q^n)}\big)}{\vol_{n^2}\big(B_{\ell_p^n(\ell_q^n)}\big)}\asymp
\left\{\begin{array}{ll}n& \mathrm{if}\ n\le \min\{\sqrt{p},q\},\\ \sqrt{q}n^{\frac12+\frac{1}{q}}  &\mathrm{if}\  q\le n\le \sqrt{p},\\ \sqrt{p}&\mathrm{if}\ \sqrt{p}\le n\le \min\{p,q\},\\\sqrt{pq}n^{\frac{1}{q}-\frac12}
& \mathrm{if}\ \max\{\sqrt{p},q\}\le n\le p,\\
n^{\frac12+\frac{1}{p}}&\mathrm{if}\  p\le n\le q,\\
\sqrt{q}n^{\frac{1}{p}+\frac{1}{q}}  &\mathrm{if}\  n\ge\max\{p,q\}.\end{array} \right.
\end{equation}
The  intricacy of~\eqref{eq:max prj lpn ellqn} is perhaps unexpected, though it is nonetheless sharp in all of the six ranges (depending on the relative locations of $p,q,n$ and, somewhat curiously, $\sqrt{p}$) that  appear in~\eqref{eq:max prj lpn ellqn}. By reasoning analogously to the discussion following~\eqref{eq:crude BN}, one can prove a bound on $\ee(\ell_p^n(\ell_q^n))$ that matches~\eqref{eq:e upper lpn lqn} up to lower order factors by applying Theorem~\ref{thm:XY version ext}  with $\Y=\ell_r^n(\ell_s^n)$ and then optimizing over $r,s\ge 1$.  For the sole purpose of this application, only the range $n\ge\max\{p,q\}$ of~\eqref{eq:max prj lpn ellqn} is needed. However, results such as~\eqref{eq:max prj lpn ellqn} have geometric interest in their own right for all of the possible values of the relevant parameters.  We will actually prove a version of~\eqref{eq:max prj lpn ellqn} for $\ell_p^n(\ell_q^m)$ even when $n\neq m$; the case of rectangular matrices is independently interesting, but  we will also use it elsewhere (see Remark~\ref{rem:nested lp} below).
\end{example}

\begin{problem} Determine the exact maximizers of volumes of hyperplane projections of the unit balls of $\sfS_p^n$ and $\ell_p^n(\ell_q^n)$, i.e., for which $A\in \M_n(\R)\setminus \{0\}$ are the maxima in~\eqref{eq:quote with gid} and~\eqref{eq:max prj lpn ellqn} attained.
\end{problem}

\subsection{A dimension-independent extension theorem}\label{sec:extension two metrics} In the preceding sections we stated all of the extension theorems using the traditional setup that aims to extend  a Lipschitz function to a function that is Lipschitz with respect to the given metric. However, all of our new (positive) extension theorems are a consequence of Theorem~\ref{thm:nonstandard ext} below, which is a nonstandard Lipschitz extension theorem.

Theorem~\ref{thm:nonstandard ext}  asserts that if $\X=(\R^n,\|\cdot\|_\X)$ is a normed space and $f$ is a $1$-Lipschitz function from a subset of $\R^n$ to a Banach space $\bfZ$, then $f$ can be extended to a $\bfZ$-valued function that is defined on all of $\R^n$ and is $O(1)$-Lipschitz with respect to the metric that is induced on $\R^n$ by the norm $|||\cdot|||=2\|\cdot\|_{\Pi^{\textbf *}\X}/\vol_n(B_\X)$, i.e., a suitable rescaling of the norm whose unit ball is the polar projection body of $B_\X$. This rescaling  ensures that $|||\cdot|||$ dominates $\|\cdot\|_\X$;  indeed, by an elementary geometric argument (see Remark~\ref{rem:volume of projection loose ends}),
\begin{equation}\label{eq:polar projection body bounds}
\forall x\in \R^n,\qquad \|x\|_\X\le \frac{2\|x\|_{\Pi^{\textbf *}\X}}{\vol_n(B_\X)}\le n\|x\|_\X.
\end{equation}
Thus, the conclusion of Theorem~\ref{thm:nonstandard ext} that the extended function is Lipschitz with respect to $|||\cdot|||$ is less stringent than the traditional requirement that it should be Lipschitz with respect to $\|\cdot\|_\X$, but Theorem~\ref{thm:nonstandard ext} has the feature that the upper bound on the Lipschitz constant is independent of the dimension.

\begin{theorem}\label{thm:nonstandard ext} Fix $n\in \N$, a normed space $\X=(\R^n,\|\cdot\|_\X)$ and a Banach space $(\bfZ,\|\cdot\|_{\bfZ})$. Suppose that $\sub\subset \R^n$  and $f:\sub\to \bfZ$ is $1$-Lipschitz with respect to the metric that is induced by $\|\cdot\|_\X$, i.e., $\|f(x)-f(y)\|_\bfZ\le \|x-y\|_\X$ for every $x,y\in \sub$.  Then, there exists $F:\R^n\to \bfZ$ that coincides with $f$ on $\sub$ and satisfies
$$
\forall x,y\in \R^n,\qquad \|F(x)-F(y)\|_\bfZ\lesssim \frac{\|x-y\|_{\Pi^{\textbf *}\X}}{\vol_n(B_\X)}.
$$
\end{theorem}

To see how Theorem~\ref{thm:nonstandard ext}  implies Theorem~\ref{thm:XY version ext},  denote (in the setting of the statement of Theorem~\ref{thm:XY version ext}):
\begin{equation}\label{eq:MM' notation}
M=\sup_{\substack{x,y\in \sub\\x\neq y}}\bigg(\frac{\|x-y\|_{\X}}{\|x-y\|_{\Y}}\bigg)\qquad \mathrm{and}\qquad M'=\sup_{\substack{x,y\in \sub\\x\neq y}} \bigg( \frac{\vol_{n-1}\big(\proj_{(x-y)^\perp}B_{\Y}\big)}{\vol_n(B_{\Y})}\cdot \frac{\|x-y\|_{\ell_2^n}}{\|x-y\|_{\X}}\bigg).
\end{equation}
Thus, every $x,y\in \sub$ satisfy $\|x-y\|_\X\le M\|x-y\|_\Y$ and, recalling~\eqref{eq:use cauchy}, also $\|x-y\|_{\Pi^{\textbf *}\Y}/\vol_n(B_\Y)\le M'\|x-y\|_\X$. Let $(\bfZ,\|\cdot\|_\bfZ)$ be any Banach space and consider an arbitrary subset $\sub'\subset \sub$. If $f:\sub'\to \bfZ$ is $1$-Lipschitz with respect to the metric that is induced by $\|\cdot\|_\X$, then $f/M$ is $1$-Lipschitz with respect to the metric that is induced by $\Y$. By Theorem~\ref{thm:nonstandard ext} (with $\X$ replaced by $\Y$, $\sub$ replaced by $\sub'$, $f$ replaced by $f/M$) we therefore see that there exists $F:\R^n\to \bfZ$ (for Theorem~\ref{thm:XY version ext} we only need $F$ to be defined on $\sub$) that extends $F$ and satisfies $\|F(x)-F(y)\|_\bfZ\lesssim M \|x-y\|_{\Pi^{\textbf *}\Y}/\vol_n(B_\Y)\le MM'\|x-y\|_\X$ for all $x,y\in \sub$. This coincides with~\eqref{eq:ext XY version}.

\begin{remark} Given $p\ge 1$, consider what happens when we apply Theorem~\ref{thm:nonstandard ext} to the space $\Y_p^n$ of Theorem~\ref{prop:rounded cube}. We get that for any $\sub \subset \R^n$ and any Banach space $\bfZ$, if $f:\sub \to \bfZ$ is $1$-Lipschitz with respect to the $\ell_p^n$ metric, then $f$ can be extended to  $F:\R^n\to \bfZ$ that is $O(n^{1/p})$-Lipschitz with respect to the  Euclidean metric. When $p<2$, the Lipschitz assumption on $f$ is less stringent than requiring it to be $O(1)$-Lipschitz with respect to the Euclidean metric, but we then get an extension $F$ that is $O(n^{1/p})$-Lipschitz with respect to the Euclidean metric; this  upper bound on the Lipschitz constant of $F$ is asymptotically larger than the $O(\sqrt{n})$  bound that we would get if $f$ were assumed to be $1$-Lipschitz with respect to the Euclidean metric and we applied the second inequality in~\eqref{eq:ell2 what's known}, but we get it while requiring less from $f$. In particular, when $p=1$ we see that any $\bfZ$-valued function on a subset of $\R^n$ that is $1$-Lipschitz with respect to the $\ell_1^n$ metric can be extended to a $\bfZ$-valued function defined on all of $\R^n$ whose Lipschitz constant with respect to the Euclidean metric is $O(n)$, while an application of~\cite{JLS86} will give an extension that is $O(n)$-Lipschitz with respect to the $\ell_1^n$ metric. On the other hand, if $p>2$, then the Lipschitz assumption on $f$ is more stringent than requiring it to be $O(1)$-Lipschitz with respect to the Euclidean metric, but we then get an extension $F$ that is $O(n^{1/p})$-Lipschitz with respect to the Euclidean metric, which is asymptotically better than the $O(\sqrt{n})$ bound from~\eqref{eq:ell2 what's known}. In particular, when $p=\infty$ we see that any $\bfZ$-valued function on a subset of $\R^n$ that is $1$-Lipschitz with respect to the $\ell_\infty^n$ metric can be extended to a $\bfZ$-valued function on all of $\R^n$ whose Lipschitz constant with respect to the Euclidean metric is $O(1)$.
\end{remark}

\subsection{Isomorphic reverse isoperimetry}\label{sec:reverse iso} All of the applications that we found for Theorem~\ref{thm:XY version ext} proceed by bounding  the volumes of hyperplane projections of $B_\Y$ that appear in right hand side of~\eqref{eq:ext XY version} by
\begin{equation}\label{eq:max projection display}
\mathrm{MaxProj}(B_\Y)\eqdef \max_{z\in S^{n-1}} \vol_{n-1}\big(\proj_{z^\perp}B_{\Y}\big).
\end{equation}
Thus, it follows from~\eqref{eq:sub is Rn} that for any two normed spaces $\X=(\R^n, \|\cdot\|_\X), \Y=(\R^n,\|\cdot\|_\Y)$ with $B_\Y\subset B_\X$,
\begin{equation}\label{eq:substitute MaxProj}
\ee(\X)\lesssim \frac{\mathrm{MaxProj}(B_\Y)}{\vol_n(B_\Y)}\diam_{\ell_2^n}(B_\X).
\end{equation}

While there could conceivably be an application of~\eqref{eq:sub is Rn} that is more refined than~\eqref{eq:substitute MaxProj}, in this section we will investigate the ramifications of bounding  $\mathrm{MaxProj}(B_\X)$ as a way to use Theorem~\ref{thm:XY version ext}. This will relate to the isomorphic reverse isoperimetric phenomena  that we conjectured in Section~\ref{sec:reverse iso in overview}.

Any origin-symmetric convex body $L\subset \R^n$ satisfies
\begin{equation}\label{eq:max proj lower}
\mathrm{MaxProj}(L)\gtrsim \frac{\vol_{n-1}(\partial L)}{\sqrt{n}}.
\end{equation}
Indeed, this follows immediately from  the following classical  {\em Cauchy surface area formula} (see e.g.~equation~5.73 in~\cite{Sch14}) by bounding the integrand by its maximum.
\begin{equation}\label{eq:from proj to per}
\vol_{n-1}(\partial L)=\frac{2\sqrt{\pi}\Gamma\big(\frac{n+1}{2}\big)}{\Gamma\big(\frac{n}{2}\big)}\fint_{S^{n-1}} \vol_{n-1}\big(\proj_{z^\perp}L\big)\ud z\asymp \sqrt{n} \fint_{S^{n-1}} \vol_{n-1}\big(\proj_{z^\perp}L\big)\ud z.
\end{equation}

\begin{remark}\label{rem:isomorphic reverse for lpn} Using~\eqref{eq:max proj lower}, Theorem~\ref{prop:rounded cube} implies that Conjecture~\ref{isomorphic reverse conj1} (isomorphic reverse isoperimetry) holds (with $S$ the identity mapping) when $K=B_{\ell_p^n}$ for any $p\ge 1$ and $n\in \N$. Indeed, let $\Y_p^n$ be the normed space from Theorem~\ref{prop:rounded cube}. By the first inequality in~\eqref{eq:holder on small supprt} we have
\begin{equation}\label{eq:vol rad of ypn}
\vol_n\big(B_{\Y_p^n}\big)^{\frac{1}{n}}\asymp \vol_n\big(B_{\ell_p^n}\big)^{\frac{1}{n}}\asymp n^{-\frac{1}{p}},
\end{equation}
where the last equivalence in~\eqref{eq:vol rad of ypn} is a standard computation (e.g.~\cite[page~11]{Pis89}).  By~\eqref{eq:max proj lower} and~\eqref{eq:vol rad of ypn}, the second  inequality in~\eqref{eq:holder on small supprt} implies that the isoperimetric quotient of $B_{\Y_p^n}$ is $O(\sqrt{n})$. So, Conjecture~\ref{isomorphic reverse conj1} holds for $K=B_{\ell_p^n}$ if we take $L$ to be a  rescaling by a universal constant factor of $B_{\Y_p^n}$ so that $L\subset K$.
\end{remark}

Thanks to~\eqref{eq:max proj lower}, if we set $K=B_\X$ and $L=B_\Y$ in~\eqref{eq:substitute MaxProj}, then the right hand side of~\eqref{eq:substitute MaxProj} satisfies
\begin{equation}\label{eq:to reverse}
\frac{\mathrm{MaxProj}(L)}{\vol_n(L)}\diam_{\ell_2^n}(K)\gtrsim \frac{\vol_{n-1}(\partial L)}{\sqrt{n}\vol_n(L)}\diam_{\ell_2^n}(K)= \frac{\iq(L)}{\sqrt{n}}\cdot \frac{\diam_{\ell_2^n}(K)}{\vol_n(L)^{\frac{1}{n}}}\gtrsim \frac{\diam_{\ell_2^n}(K)}{\vol_n(K)^{\frac{1}{n}}},
\end{equation}
where we recall notation~\eqref{eq:def iq}  for the isoperimetric quotient $\iq(\cdot)$ and the last step  uses the isoperimetric theorem~\eqref{eq:quote ispoperimetric theorem} and the assumption $L\subset K$. The following proposition explains what it would entail for one to be able to reverse~\eqref{eq:to reverse} after an application of a suitable linear transformation; in particular, it shows that one can find $S\in \SL_n(\R)$ and an origin-symmetric convex body $L\subset SK$ such that
$$
\frac{\mathrm{MaxProj}(L)}{\vol_n(L)}\diam_{\ell_2^n}(SK)\lesssim \frac{\diam_{\ell_2^n}(SK)}{\vol_n(K)^{\frac{1}{n}}}
$$
{\em if and only if} Conjecture~\ref{weak isomorphic reverse conj1} on weak isomorphic reverse isoperimetry holds for $K$.

\begin{proposition}\label{prop:equivalence of upper bound and weak iso}The following two statements are equivalent for every $n\in\N$, every origin-symmetric convex body $K\subset \R^n$ and every $\alpha>0$.
\begin{enumerate}
\item There exist  a linear transformation $S\in \SL_n(\R)$ and an origin-symmetric convex body $L\subset S K$ with
\begin{equation}\label{eq:1 in equivalence KL}
\frac{\mathrm{MaxProj}(L)}{\vol_n(L)}\vol_n(K)^{\frac{1}{n}}\lesssim \alpha.
\end{equation}
\item There exist a linear transformation $S\in \SL_n(\R)$ and an origin-symmetric convex body $L\subset SK$ that satisfies $\sqrt[n]{\vol_n(L)}\ge \beta\sqrt[n]{\vol_n(K)}$ and $\iq(L)\le \gamma\sqrt{n}$ for some $\beta\gtrsim 1/\alpha$ and $\gamma \lesssim \alpha$ with $\gamma/\beta\lesssim \alpha$.
\end{enumerate}
\end{proposition}

\begin{proof} For the implication {\em (1)}$\implies${\em (2)} denote $\beta=\sqrt[n]{\vol_n(L)}/\sqrt[n]{\vol_n(K)}$ and $\gamma=\iq(L)/\sqrt{n}$. Then
$$
\alpha\stackrel{\eqref{eq:1 in equivalence KL}}{\gtrsim} \frac{\mathrm{MaxProj}(L)}{\vol_n(L)}\vol_n(K)^{\frac{1}{n}}\stackrel{\eqref{eq:max proj lower}}{\gtrsim} \frac{\vol_{n-1}(\partial L)}{\vol_n(L)\sqrt{n}}\vol_n(K)^{\frac{1}{n}}=\frac{\gamma}{\beta}.
$$
Since by the isoperimetric theorem~\eqref{eq:quote ispoperimetric theorem} we have $\gamma\gtrsim 1$, it follows from this that $\beta\gtrsim 1/\alpha$, and since $L\subset S K$ and $S\in SL_n(\R)$, we have $\vol_n(L)\le \vol_n(K)$, so $\beta\le 1$ and it also follows from this that $\gamma\lesssim \alpha$.

For the implication {\em (2)}$\implies${\em (1)}, fix $T\in \SL_n(\R)$ with $\vol_{n-1}(\partial TL)=\min\{\vol_{n-1}(\partial T'L):\ T'\in \SL_n(\R)\}$,
 i.e., $TL$ is in its {\em minimum surface area position}~\cite{Pet61}. By definition, $\vol_{n-1}(\partial T L)\le \vol_{n-1}(\partial L)$   and  by Proposition~3.1 in the work~\cite{GP99} of Giannopoulos and Papadimitrakis combined with~\eqref{eq:max proj lower} we have $$\mathrm{MaxProj}(T L)\asymp \frac{\vol_{n-1}(\partial T L)}{\sqrt{n}}.$$ Consequently, if $L$ satisfies part  {\em (2)} of Proposition~\ref{prop:equivalence of upper bound and weak iso}, then
\begin{multline*}
\frac{\mathrm{MaxProj}(TL)}{\vol_n(TL)}\vol_n(K)^{\frac{1}{n}}\asymp \frac{\vol_{n-1}(\partial T L)}{\vol_n(TL)\sqrt{n}}\vol_n(K)^{\frac{1}{n}}\le  \frac{\vol_{n-1}(\partial L)}{\vol_n(TL)\sqrt{n}}\vol_n(K)^{\frac{1}{n}}=\frac{\iq(L)}{\sqrt{n}}\bigg(\frac{\vol_n(K)}{\vol_n(L)}\bigg)^{\frac{1}{n}}
\le\frac{\gamma}{\beta}\lesssim \alpha.
\end{multline*}
It follows that {\em (1)} holds with $S$ replaced by $TS\in \SL_n(\R)$ and $L$ replaced by $TL\subset TSK$.
\end{proof}

Since when $\alpha\lesssim 1$ in Proposition~\ref{prop:equivalence of upper bound and weak iso} the assertion of its part~{\em (2)}  coincides with Conjecture~\ref{weak isomorphic reverse conj1}, it follows that Conjecture~\ref{weak isomorphic reverse conj1}, and a fortiori Conjecture~\ref{isomorphic reverse conj1}, imply that for any normed space $\X=(\R^n,\|\cdot\|_\X)$ there is   $S\in \SL_n(\R)$ such that $\ee(\X)$ is at most a universal constant multiple of $\diam_{\ell_2^n}(S B_\X)/\sqrt[n]{\vol_n(B_\X)}$. Indeed, this follows by applying Theorem~\ref{thm:XY version ext}  to the normed spaces $\X'=(\R^n,\|\cdot\|_{\X'})$ and $\Y=(\R^n,\|\cdot\|_\Y)$ whose unit balls are $SB_\X$ and $L$, respectively, where $S$ and $L$ are as in part~{\em (1)} of Proposition~\ref{prop:equivalence of upper bound and weak iso} for $K=B_\X$, while noting that $\ee(\X')=\ee(\X)$ since $\X'$ is isometric to $\X$. We record this conclusion as the following corollary.

\begin{corollary}\label{lem:consequence of reverse isoperimetry} If Conjecture~\ref{weak isomorphic reverse conj1} holds for a normed space $\X=(\R^n,\|\cdot\|_\X)$, then  there is $S\in \SL_n(\R)$ such that
\begin{equation}\label{eq:desired Lambda}
\ee(\X)\lesssim \frac{\diam_{\ell_2^n}(S B_\X)}{\vol_n(B_\X)^{\frac{1}{n}}}.
\end{equation}
\end{corollary}

The upshot of Corollary~\ref{lem:consequence of reverse isoperimetry} is that the right hand side of~\eqref{eq:desired Lambda} involves only Euclidean diameters and  $n$'th roots of volumes, which are typically much easier to estimate than extremal volumes of hyperplane projections. This comes at the cost of having to find the auxiliary linear transformation $S\in \SL_n(\R)$, but we expect that in concrete settings it will be simple to determine $S$. Moreover, in all of the specific examples of spaces for which we are interested (at least initially) in estimating their Lipschitz extension modulus, $S$ should be the identity mapping. We will discuss this matter and its consequences in Section~\ref{sec:positions}.

\begin{remark} There is a degree of freedom that the above discussion does not exploit. Let $\X=(\R^n,\|\cdot\|_\X)$ be a normed space. By~\eqref{eq:adjoint}, we know that $\ee(\X)$ is bounded from above by a constant multiple of the minimum of $\diam_{\Pi^{\textbf *} \Y}(B_\X)/\vol_n(B_{\Y})$ over all the normed spaces  $\Y=(\R^n,\|\cdot\|_\Y)$ for which $B_\Y\subset B_\X$.  By~\eqref{eq:substitute MaxProj}, to control this minimum it suffices to estimate the minimum of $\mathrm{MaxProj(B_\Y)}/\vol_n(B_\Y)$ over all such $\Y$, which relates to isomorphic reverse isoperimetric phenomena. But, we could also take a normed space $\mathbf{W}=(\R^m,\|\cdot\|_\mathbf{W})$ for $m\ge n$ such that  $B_{\mathbf{W}}\cap \R^n=B_\X$ (we need that $\mathbf{W}$ contains an isometric copy of $\X$), estimate either of the two minima above  for the super-space  $\mathbf{W}$, and then use $\ee(\X)\le  \ee(\mathbf{W})$. Thus, it would suffice to embed $\X$ into a larger normed space that exhibits good isomorphic reverse isoperimetry. Our conjectures imply that such an embedding step is not needed, namely we expect that the desired isomorphic reverse isoperimetric property holds for  $\X$. Nevertheless, it could be that by finding a suitable super-space $\mathbf{W}$ one could bound $\ee(\X)$ while circumventing the difficulty of proving Conjecture~\ref{weak isomorphic reverse conj1} for $\X$. For example, if $\X$ is a subspace of $\ell_\infty^m$ for some $m=O(n)$, then by Theorem~\ref{thm:ell infty} we know that $\ee(\X)\lesssim \sqrt{n}$, but  this is because we know that $\ell_\infty^m$ has the desired isomorphic reverse isoperimetric property, and it is not clear how to prove it for $\X$ itself. It is also unclear how to construct for a given normed $\X$  a super-space $\mathbf{W}$ that could be used as above. We leave the exploration of this possibility for future research.
\end{remark}

\subsubsection{A spectral interpretation, reverse Faber--Krahn and the Cheeger space of a normed space}\label{sec:spectral} We will henceforth quantify the extent to which  Conjecture~\ref{weak isomorphic reverse conj1} holds through the following condition:
\begin{equation}\label{eq:our quantification}
\frac{\iq(L)}{\sqrt{n}}\bigg(\frac{\vol_n(K)}{\vol_n(L)}\bigg)^{\frac{1}{n}}=\frac{\vol_n(K)^{\frac{1}{n}}}{\sqrt{n}}\left(\frac{\vol_{n-1}(\partial L)}{\vol_n(L)}\right)\le \alpha.
\end{equation}
The factors $\iq(L)/\sqrt{n}$ and $(\vol_n(K)/\vol_n(L))^{1/n}$ that appear in the left hand side of~\eqref{eq:our quantification} are at least a positive universal constant (by, respectively, the isoperimetric theorem and the assumed inclusion $L\subset K$), so~\eqref{eq:our quantification} implies that $\sqrt[n]{\vol_n(L)}\gtrsim \alpha^{-1}\sqrt[n]{\vol_n(K)}$ and $\iq(L)\le \alpha \sqrt{n}$. Thus, if $\alpha=O(1)$, then~\eqref{eq:our quantification} is equivalent to the conclusion of  Conjecture~\ref{weak isomorphic reverse conj1}. However, even though Conjecture~\ref{weak isomorphic reverse conj1} expresses our expectation that~\eqref{eq:our quantification} is always achievable with $\alpha=O(1)$ upon a judicious choice of the Euclidean structure on $\R^n$, in lieu of Conjecture~\ref{weak isomorphic reverse conj1} it would still be valuable to obtain~\eqref{eq:our quantification} with $\alpha$ unbounded but slowly growing. In such a situation, the bi-parameter quantification that we used in part {\em (2)} of Proposition~\ref{prop:equivalence of upper bound and weak iso} contains more geometric information than~\eqref{eq:our quantification}, but below we will work with~\eqref{eq:our quantification} in order to simplify the ensuing discussion; this suffices for our purposes because~\eqref{eq:our quantification} is what shows up  in all of the applications herein (per the proof  Proposition~\ref{prop:equivalence of upper bound and weak iso}) since they all proceed by bounding the right hand side of~\eqref{eq:substitute MaxProj} from above.

Alter and Caselles proved~\cite{AC09} that for every convex body $K\subset \R^n$ there is a {\em unique} measurable set $A\subset K$, which we call  the {\em Cheeger body} of $K$ and denote $\Ch K$, satisfying $\Per(A)/\vol_n(A)\le \Per(B)/\vol_n(B)$ for every measurable $B\subset K$ with $\vol_n(B)>0$, where $\Per(\cdot)$ denotes perimeter in the sense of Caccioppoli and de~Giorgi; this notion is covered in~\cite{AFP00} but we do not need to recall its definition here since the perimeter of a convex body coincides with the $(n-1)$-dimensional Hausdorff measure of its boundary.
It was  proved in~\cite{AC09} that $\Ch K$ is convex and its boundary is $C^{1,1}$. Further information on this remarkable theorem can be found in~\cite{AC09}, where $\Ch K$ is characterized in terms of the mean curvature of its boundary through the work~\cite{ACC05} of Alter, Caselles and Chambolle (see also the precursor~\cite{CCN07} by Caselles, Chambolle and Novaga which obtained these statements under stronger assumptions on $K$).

Beyond the fact that it allows us to use the notation $\Ch K$ and call it {\em the} Cheeger body of $K$, the aforementioned uniqueness statement will be used substantially in the ensuing reasoning. It  implies in particular that if $K$ is origin-symmetric, then so is $\Ch K$. Consequently, if $\X=(\R^n,\|\cdot\|_\X)$ is a normed space, then $\Ch B_\X$ is the unit ball of a normed space that we denote by $\Ch \X$ and call  the {\em Cheeger space} of $\X$.

For a convex body $K\subset \R^n$, let $\lambda(K)$ be the smallest Dirichlet eigenvalue of the Laplacian on $K$, namely it is the smallest $\lambda>0$ for which there is a nonzero function $\f:K\to \R$ that is smooth on the interior of $K$, vanishes on the boundary of $K$, and satisfies $\Delta \f=-\lambda\f$ on the interior of $K$; see e.g.~\cite{PS51,CH53,Cha84} for background on this classical topic. If $\X=(\R^n,\|\cdot\|_\X)$ is a normed space, then we denote $\lambda(\X)=\lambda(B_\X)$.

The quantity $h(K)=\vol_{n-1}(\partial \Ch K)/\vol_n(\Ch K)$ is called the Cheeger constant of $K$; it relates  to $\lambda(K)$ by
\begin{equation}\label{eq:cheeger buser}
 \frac{2}{\pi}\sqrt{\lambda(K)}\le h(K)=\frac{\vol_{n-1}(\partial \Ch K)}{\vol_n(\Ch K)}\le 2\sqrt{\lambda(K)}.
\end{equation}
It is important for our purposes  that the constants appearing in~\eqref{eq:cheeger buser} are independent of the dimension $n$.
The second inequality in~\eqref{eq:cheeger buser} is the Cheeger inequality for the Dirichlet Laplacian on Euclidean domains. Cheeger's proof  of it for compact Riemannian manifolds without boundary appears in~\cite{Che70} and that proof works mutatis mutandis in the present setting; see its derivation in e.g.~the appendix of~\cite{LW97}. The first inequality in~\eqref{eq:cheeger buser} can be called the Buser inequality for  the Dirichlet Laplacian on convex Euclidean domains, since Buser proved~\cite{Bus82} its analogue for compact Riemannian manifolds without boundary that have a lower bound on their Ricci curvature. In our setting, this reverse Cheeger inequality is more recent, namely it was noted for planar convex sets by Parini~\cite{Par17} and in any dimension  by Brasco~\cite{Bra20}. It can be justified quickly  using  the convexity of $K$ and its Cheeger body $\Ch K$ as follows. By a classical theorem of P\'olya we have $\lambda(K)\le \pi^2 (\vol_{n-1}(\partial K)/\vol_n(K))^2/4$ (P\'olya  proved this for planar convex sets, but in~\cite{JP82}  Jo\'o and Stach\'o carried out P\'olya's approach for convex bodies in $\R^n$ for any $n\in \N$). Therefore,  $\lambda(K)\le \lambda(\Ch K)\le \pi^2 (\vol_{n-1}(\partial \Ch K)/\vol_n(\Ch K))^2/4=\pi^2 h(K)^2/4$, since $\Ch K$ is convex.

Let $j_{n/2-1,1}$ be the smallest positive zero of the Bessel function $J_{n/2-1}$; see Chapter~4 of~\cite{AAR99} for a treatment of Bessel functions and their zeros, though here we will only need to know that $j_{n/2-1,1}\asymp n$ (see~\cite{Tri49} for more precise asymptotics). By classical computations (see e.g. equation~1.29 in~\cite{Hen06}), $$\lambda\big(B_{\ell_2^n}\big)=j_{\frac{n}{2}-1,1}^2.$$
The Faber--Krahn inequality~\cite{Fab23,Kra26} (see also e.g.~\cite{PS51,Cha84}) asserts that $\lambda(K)$ is at least the first Dirichlet eigenvalue of a Euclidean ball whose volume is the same as the volume of $K$. Thus,
$$
\lambda(K)\vol_n(K)^{\frac{2}{n}}\ge \lambda\big(B_{\ell_2^n}\big)\vol_n\big(B_{\ell_2^n}\big)^{\frac{2}{n}}= j_{\frac{n}{2}-1,1}^2 \vol_n\big(B_{\ell_2^n}\big)^{\frac{2}{n}} \asymp n,
$$
where we used the straightforward fact that $\lambda(r K)=\lambda(K)/r^2$ for every $r>0$.

Observe that~\eqref{eq:cheeger buser} can be rewritten as follows for every convex body $K\subset \R^n$.
$$
\frac{2}{\pi} \left(\frac{\lambda(K)\vol_n(K)^{\frac{2}{n}}}{n}\right)^{\frac12}\le \frac{\iq(\Ch K)}{\sqrt{n}}\left(\frac{\vol_n(K)}{\vol_n(\Ch K)}\right)^{\frac{1}{n}} \le 2\left(\frac{\lambda(K)\vol_n(K)^{\frac{2}{n}}}{n}\right)^{\frac12}.
$$
Hence, for every $\alpha>0$ we have
\begin{equation}\label{eq:spoectral reverse iso equiv}
\frac{\iq(\Ch K)}{\sqrt{n}}\bigg(\frac{\vol_n( K)}{\vol_n(\Ch K)}\bigg)^{\frac{1}{n}}\lesssim \alpha\iff \lambda(K)\vol_n(K)^{\frac{2}{n}}\lesssim \alpha^2 n.
\end{equation}
Since $\Ch K$ is convex, the convex body $L\subset K$ that minimizes the left hand side of~\eqref{eq:our quantification} is equal to $\Ch K$. We therefore see that Conjecture~\ref{conj:reverse FK}  below is equivalent to Conjecture~\ref{weak isomorphic reverse conj1}. Furthermore, if one of these two conjectures hold for a matrix $S\in \SL_n(\R)$, then the same matrix would work for the other conjecture.

\begin{conjecture}[reverse Faber--Krahn]\label{conj:reverse FK} For any origin-symmetric convex body $K\subset \R^n$ there exists a volume-preserving linear transformation  $S\in \SL_n(\R)$ such that
$$
 \lambda(SK)\vol(K)^{\frac{2}{n}}\asymp n.
$$
\end{conjecture}

\begin{remark}
One can also wonder about exact maximizers in the context of Conjecture~\ref{conj:reverse FK}. Specifically, Bucur and Fragal\`a stated in~\cite[page~389]{BF16} that they expect that for any origin-symmetric convex body $K\subset \R^n$ with $\vol_n(K)=1$ there exists  $S\in \SL_n(\R)$ such that $ \lambda(SK)\le
 \lambda([0,1]^n)=\pi^2n$. If true, then this would be a beautiful statement even though it does not have substantial impact on Conjecture~\ref{weak isomorphic reverse conj1} and its implications herein (it would only influence the value of the implicit constant factors in our statements, which incur further losses that are most likely not sharp in other steps of their derivations). The only available evidence for the aforementioned (speculative) exact statement is the partial result of~\cite{BF16} in the planar case $n=2$, which proves that it indeed holds when $K\subset \R^2$ is a convex axisymmetric octagon that has four of its vertices lying on the axes at the same distance from the origin; see specifically Proposition~10 in~\cite{BF16}, whose proof involves delicate reasoning that incorporate computer-assisted steps. A complete result for $n=2$ has been subsequently obtained by the same authors in~\cite{BF18} for the analogous question in which one replaces the Dirichlet eigenvalue of the Laplacian by the Cheeger constant. Namely, Theorem~1.1 of~\cite{BF18} states that for {every} origin-symmetric convex body $K\subset \R^2$ with $\vol_2(K)=1$ there exists  $S\in \SL_2(\R)$ such that $ h(SK)\le
 h([0,1]^2)=2+\sqrt{\pi}$ (furthermore, in this case $S$ can be taken to be the matrix that puts $K$ in John position, i.e., the ellipse of maximal area that is contained in $S K$ is a circle).
\end{remark}

This spectral interpretation of Conjecture~\ref{weak isomorphic reverse conj1} is useful  for multiple purposes, including the following lemma whose proof appears in Section~\ref{sec:direct}. For its statement, as well as throughout the ensuing discussion, recall that a basis $x_1,\ldots,x_n$ of an $n$-dimensional normed space $(\X,\|\cdot\|_\X)$ is  a $1$-unconditional basis of $\X$ if $\|\e_1a_1x_1+\ldots +\e_n a_nx_n\|_\X= \|a_1x_1+\ldots +a_nx_n\|_\X$ for every choice of scalars $a_1,\ldots,a_n\in \R$ and signs $\e_1,\ldots,\e_n\in \{-1,1\}$.  When we say that  $\X=(\R^n,\|\cdot\|_\X)$ is an unconditional normed space, we mean that the standard (coordinate) basis $e_1,\ldots,e_n$ of $\R^n$ is a $1$-unconditional basis of $\X$.

\begin{lemma}[closure of Conjecture~\ref{weak isomorphic reverse conj1} under unconditional composition]\label{lem:unconditional composiiton} Fix $n\in \N$ and $m_1,\ldots,m_n\in \N$. Let $\X_1=(\R^{m_1},\|\cdot\|_{\X_1}),\ldots,\X_n=(\R^{m_n},\|\cdot\|_{\X_{n}})$ be normed spaces. Also, let $\bfE=(\R^n,\|\cdot\|_\bfE)$ be an unconditional normed space. Define a normed space $\X=(\R^{m_1}\times\ldots\times \R^{m_n},\|\cdot\|_\X)$ by
\begin{equation}\label{eq:def composed norm}
\forall x=(x_1,\ldots,x_n)\in \R^{m_1}\times\ldots\times \R^{m_n},\qquad \|x\|_\X\eqdef \big\|\big(\|x_1\|_{\X_1},\ldots,\|x_n\|_{\X_n}\big)\big\|_\bfE.
\end{equation}
Suppose that there exist $\alpha>0$, linear transformations  $S_1\in \SL_{m_1}(\R),\ldots,S_n\in \SL_{m_n}(\R)$, and normed spaces $\Y_1=(\R^{m_1},\|\cdot\|_{\Y_1}),\ldots,\Y_n=(\R^{m_n},\|\cdot\|_{\Y_{n}})$ such that
\begin{equation}\label{eq:Yk assumption less alpha-intro}
\forall k\in \n,\qquad B_{\Y_k}\subset S_kB_{\X_k}\qquad\mathrm{and}\qquad \frac{\iq\big(B_{\Y_k}\big)}{\sqrt{m_k}}\left(\frac{\vol_{m_k}\big(B_{\X_k}\big)}{\vol_{m_k}\big(B_{\Y_k}\big)}\right)^{\frac{1}{m_k}}\le \alpha.
\end{equation}
Then, there exist a normed space $\Y=(\R^{m_1}\times\ldots\times \R^{m_n},\|\cdot\|_\X)$ and $S\in \SL(\R^{m_1}\times\ldots\times \R^{m_n})$ such that
\begin{equation}\label{eq:Y conclusion composition-intro}
B_\Y\subset S B_\X\qquad \mathrm{and}\qquad \frac{\iq(B_{\Y})}{\sqrt{m_1+\ldots+m_n}}\left(\frac{\vol_{m_1+\ldots+m_n}(B_{\X})}{\vol_{m_1+\ldots+m_n}(B_{\Y})}\right)^{\frac{1}{m_1+\ldots+m_n}}\lesssim \alpha.
\end{equation}
\end{lemma}

Since~\eqref{eq:Yk assumption less alpha-intro} with $\alpha=O(1)$ is immediate  when $n_0=1$, Lemma~\ref{lem:unconditional composiiton}   establishes Conjecture~\ref{weak isomorphic reverse conj1} for when $K$ is the unit ball of an unconditional normed space $\X=(\R^n,\|\cdot\|_\X)$. This holds, in particular, for $\X=\ell_p^n$, though we will prove in Section~\ref{sec:direct} that the stronger conclusion of Conjecture~\ref{isomorphic reverse conj1} holds in this case (recall Remark~\ref{rem:isomorphic reverse for lpn}). Lemma~\ref{lem:unconditional composiiton}  also shows that Conjecture~\ref{weak isomorphic reverse conj1} holds for, say, $\X=\ell_p^n(\ell_q^m)$; we expect that the reasoning of Section~\ref{sec:direct} could be adapted to yield Conjecture~\ref{isomorphic reverse conj1} for these spaces as well, but we did not attempt to carry this out. Other  spaces that satisfy~\eqref{eq:Yk assumption less alpha-intro} with $\alpha$ slowly growing will be presented in Section~\ref{sec:positions}; upon their substitution into  Lemma~\ref{lem:unconditional composiiton}, more examples  for which Conjecture~\ref{weak isomorphic reverse conj1} holds up to lower-order factors are obtained (of course, we are conjecturing here that it holds for {\em any} space).

\begin{remark}\label{rem:cheeger position} Say that a normed space $\X=(\R^n,\|\cdot\|_\X)$ is in {\em Cheeger position} if
$$
\forall S\in \SL_n(\R),\qquad \frac{\vol_{n-1}(\partial \Ch B_\X)}{\vol_{n}(\Ch B_\X)}\le \frac{\vol_{n-1}(\partial \Ch S B_\X)}{\vol_{n}(\Ch S B_\X)}.
$$
Observe that if $\X$ is in Cheeger position, then its Cheeger space $\Ch\X$ is in minimum surface area position, namely, $\vol_{n-1}(\partial \Ch B_\X)\le \vol_{n-1}(\partial S \Ch B_\X)$ for every $S\in \SL_n(\R)$. Indeed, $S\Ch B_\X\subset SB_\X$, so by the definition of the Cheeger body of $SB_\X$ we have $\vol_{n-1}(\partial S \Ch B_\X)/\vol_n(\Ch B_\X)\ge \vol_{n-1}(\partial  \Ch S B_{\X})/\vol_n(\Ch S B_\X)$. At the same time, $\vol_{n-1}(\partial  \Ch S B_{\X})/\vol_n(\Ch S B_\X)\ge \vol_{n-1}(\partial \Ch B_\X)/\vol_{n}(\Ch B_\X)$ by the definition of the Cheeger position, so $\vol_{n-1}(\partial S \Ch B_\X)\ge \vol_{n-1}(\partial \Ch B_\X)$. This shows that in the proof of the implication {\em (2)}$\implies${\em (1)} of Proposition~\ref{prop:equivalence of upper bound and weak iso}, if we worked with $L=\Ch SK$, then there would be no need to introduce the additional linear transformation $T\in \SL_n(\R)$. It would be worthwhile to study the Cheeger position for its own sake even if it weren't for its connection to reverse isoperimetry. In particular, we do not know if the converse of the above deduction holds, namely whether it is true that if $\Ch \X$ is in minimum surface area position, then $\X$ is in Cheeger position. We also do not know if the Cheeger position is unique up to orthogonal transformation (as is the case for the minimum surface area position~\cite{GP99});  we did not investigate these matters since they are not needed for the present purposes, but we  expect that the characterisations of the Cheeger body in~\cite{AC09} would be relevant here. One could also define that a normed space $\X=(\R^n,\|\cdot\|_\X)$ is in {\em Dirichlet position} if $\lambda(\X)\le \lambda(S\X)$ for every $S\in \SL_n(\R)$. It is unclear how the Cheeger position relates to the Dirichlet position and it would be also worthwhile to study the Dirichlet position for its own sake. By~\eqref{eq:cheeger buser}, working with either the Cheeger position or the Dirichlet position would be equally valuable for  the reverse isoperimetric questions in which we are interested here.
\end{remark}

\subsubsection{Symmetries and positions}\label{sec:positions} Thus far we considered an arbitrary  scalar product on an $n$-dimensional normed space through which we identified its underlying vector space structure with $\R^n$.  However, the Lipschitz extension modulus is insufficiently understood for  ``very nice'' normed spaces (including even the Euclidean space $\ell_2^n$) that belong to a natural class of normed spaces that have a canonical identification with $\R^n$. It therefore makes sense to first focus on this class.

For a finite dimensional normed space $(\X,\|\cdot\|_\X)$, let $\mathsf{Isom}(\X)$ be the group of all of the isometric automorphism of $\X$, i.e., all the linear operators $U:\X\to \X$ that satisfy $\|Ux\|_{\X}=\|x\|_{\X}$ for every $x\in \X$. We will denote  the Haar probability measure  on the compact group $\mathsf{Isom}(\X)$ by $\mathcal{h}_\X$.

\begin{definition}\label{def:canonical position} We say that a finite dimensional  normed space $(\X,\|\cdot\|_\X)$ is {\bf \em canonically positioned} if any two $\mathsf{Isom}(\X)$-invariant scalar products on $\X$ are proportional to each other. In other words, if $\langle\cdot,\cdot\rangle:\X\times \X\to \R$ and $\langle\cdot,\cdot\rangle':\X\times \X\to \R$ are scalar products on $\X$ such that $\langle Ux,Uy\rangle =\langle x,y\rangle$ and $\langle Ux,Uy\rangle' =\langle x,y\rangle'$ for every $x,y\in \X$ and every $U\in \mathsf{Isom}(\X)$, then there necessarily exists $\lambda\in \R$ such that $\langle \cdot,\cdot \rangle'=\lambda \langle \cdot,\cdot\rangle$.
\end{definition}
Any finite dimensional normed space $\X$ has at least one scalar product $\langle \cdot,\cdot\rangle:\X\times \X\to \R$ that is  invariant under $\mathsf{Isom}(\X)$, as  seen e.g.~by averaging any given scalar product $\langle\cdot,\cdot\rangle_0$ on $\X$ with respect $\mathcal{h}_\X$, i.e., defining
$$
\forall x,y\in \X,\qquad \langle x,y\rangle \eqdef \int_{\mathsf{Isom}(\X) } \langle Sx,Sy\rangle_0 \ud\mathcal{h}_\X(S).
$$
Definition~\ref{def:canonical position} concerns those spaces $\X$ for which such an invariant scalar product is unique up to rescaling, so there is (essentially, i.e., up to rescaling) no arbitrariness when we identify $\X$ with $\R^{\dim(\X)}$.

\begin{example}\label{ex permutation and sign}
The class of $n$-dimensional canonically positioned normed spaces $(\X,\|\cdot\|_\X)$ includes   those with a basis $e_1,\ldots,e_n$ such that for any distinct $i,j\in \n$ there are a permutation $\pi\in S_n$ with $\pi(i)=j$ and a sign vector $\e=(\e_1,\ldots,\e_n)\in \{-1,1\}^n$ with $\e_i=-\e_j$ such that $T_\pi,S_\e\in \mathsf{Isom}(\X)$, where we denote $T_\pi x= \sum_{i=1}^na_{\pi(i)}e_i$ and $S_\e x=\sum_{i=1}^n \e_i a_i e_i$ for  $x=\sum_{i=1}^n a_ie_i\in \X$ with $a_1,\ldots,a_n\in \R$. Indeed, let $\langle\cdot,\cdot\rangle$ be a scalar product on $\X$ that is $\mathsf{Isom}(\X)$-invariant. For every distinct $i,j\in \n$, if $\pi\in S_n$ and $\e\in \{-1,1\}^n$ are as above, then $\langle e_i,e_i\rangle=\langle e_{\pi(i)},e_{\pi(i)}\rangle =\langle e_j,e_j\rangle$ while $\langle e_i,e_j\rangle =\langle \e_i e_i,\e_j e_j\rangle =-\langle e_i,e_j\rangle$, so $\langle e_i,e_j\rangle=0$.
\end{example}

Example~\ref{ex permutation and sign} covers all of the spaces for which we think that it is most pressing (given the current state of  knowledge) to understand their Lipschitz extension modulus, including normed spaces $(\mathbf{E},\|\cdot\|_{\mathbf{E}})$ that have a $1$-symmetric basis, i.e., a basis $e_1,\ldots,e_n\in \mathbf{E}$  such that $\|\sum_{i=1}^n \e_ia_{\pi(i)}e_i\|_\mathbf{E}=\|\sum_{i=1}^n a_ie_i\|_\mathbf{E}$ for every $(\e,\pi)\in \{-1,1\}^n\times S_n$. In particular, $\ell_p^n$, and more generally Orlicz and Lorentz spaces (see e.g.~\cite{LT77}), are canonically positioned. We will use below the common convention that a normed space $(\R^n,\|\cdot\|)$ is said to be symmetric if it is $1$-symmetric with respect to the standard (coordinate) basis $e_1,\ldots,e_n$ of $\R^n$.

Example~\ref{ex permutation and sign} also includes matrix norms $\X=(\M_n(\R),\|\cdot\|_\X)$ that remain unchanged if one transposes a pair of rows or columns, or changes the sign of an entire row or a column, such as  $\sfS_p^n$. More generally, if $\mathbf{E}=(\R^n,\|\cdot\|_{\mathbf{E}})$ is a symmetric normed space, then its  unitary  ideal $\sfS_{\mathbf{E}}=(\M_n(\R),\|\cdot\|_{\sfS_\mathbf{E}})$ is canonically positioned (see e.g.~\cite{Bha97}), where   for $T\in \M_n(\R)$ one denotes its singular values by $s_1(T)\ge\ldots \ge s_n(T)$ and defines  $\|T\|_{\sfS_\mathbf{E}}=\|(s_1(T),\ldots,s_n(T))\|_{\mathbf{E}}$. More examples of such matrix norms are projective and injective tensor products (see e.g.~\cite{Rya02}) of symmetric spaces, where if $\X=(\R^n,\|\cdot\|_\X)$ and $\Y=(\R^m,\|\cdot\|_\Y)$ are normed spaces, then their projective tensor product $\X\tp \Y$ is the norm on $M_{n\times m}(\R)=\R^n\otimes \R^m$ whose unit ball is the convex hull of $\{x\otimes y:\ (x,y)\in B_\X\times B_\Y\}$, and their injective tensor product $\X\ti \Y$ is the dual of $\X^*\!\tp \Y^*$ (equivalently, $\X\ti \Y$ is isometric to the operator norm from $\X^*$ to $\Y$; see e.g.~\cite[Section~1.1]{DFS08}).

Henceforth, when we will say that a normed space $\X=(\R^n,\|\cdot\|_\X)$ is canonically positioned it will always be tacitly assumed that the standard scalar product $\langle\cdot,\cdot\rangle$ on $\R^n$ is $\mathsf{Isom}(\X)$-invariant, i.e., $\mathsf{Isom}(\X)$ is a subgroup of the orthogonal group $\mathsf{O}_n\subset \M_n(\R)$. This is equivalent to the requirement that for every symmetric positive definite matrix $T\in \M_n(\R)$, if $TU=UT$ for every $U\in \mathsf{Isom}(\X)$, then there is $\lambda\in (0,\infty)$ such that $T=\lambda \Id_n$. Indeed, any scalar product $\langle\cdot,\cdot\rangle':\R^n\times \R^n\to \R$ is of the form $\langle x,y\rangle'=\langle Tx,y\rangle$ for some symmetric positive definite $T\in \M_n(\R)$ and all $x,y\in \R^n$, and using the $\mathsf{Isom}(\X)$-invariance of $\langle\cdot,\cdot\rangle$ we see that  $\langle\cdot,\cdot\rangle'$ is $\mathsf{Isom}(\X)$-invariant if and only if $T$ commutes with all of the elements of $\mathsf{Isom}(\X)$.

\begin{remark}\label{rem:octagon} A symmetry assumption that is common in  the literature is {\em enough symmetries}. A normed space $(\X,\|\cdot\|_\X)$ is said~\cite{GG71} to have enough symmetries  if any linear transformation $T:\X\to \X$ must be a scalar multiple of the identity if $T$ commutes with every element of $\mathsf{Isom}(\X)$. By the above discussion, if $\X$ has enough symmetries, then $\X$ is canonically positioned. The converse implication does not hold, i.e., there exist normed spaces that are  canonically positioned but do not have enough symmetries. For example, let $\mathsf{Rot}_{\pi/2}\in \mathsf{O}_2$ be the rotation by $90$ degrees and let $G$ be the subgroup of $\mathsf{O}_2$ that is generated by $\mathsf{Rot}_{\pi/2}$. Thus, $G$ is cyclic of order $4$. Let $\X=(\R^2,\|\cdot\|_\X)$ be a normed space with $\mathsf{Isom}(\X)=G$; the fact that there is such a normed space  follows from the general result~\cite[Theorem~3.1]{GL79} of Gordon and Loewy on existence of norms with a specified group of isometries, though in this particular case it is simple to construct such an example (e.g.~the unit ball of $\X$ can be taken to be a suitable non-regular octagon). Since $\mathsf{Isom}(\X)$ is Abelian, the matrix $\mathsf{Rot}_{\pi/2}$ commutes with all of the elements of $\mathsf{Isom}(\X)$  yet it is not a multiple of the identity matrix, so $\X$ does not have enough symmetries. Nevertheless, $\X$ is canonically positioned. Indeed, suppose that $T\in \M_2(\R)$ is a symmetric matrix  that commutes with $\mathsf{Rot}_{\pi/2}$. Then, $\mathsf{Rot}_{\pi/2}$ preserves any eigenspace of $T$, which means that any such eigenspace must be $\{0\}$ or $\R^2$. But  $T$ is diagonalizable over $\R$, so it follows that $T=\lambda \Id_2$ for some $\lambda\in \R$. If $n$ is even, then one obtains such an $n$-dimensional example by considering $\ell_\infty^{n/2}(\X)$. However, a representation-theoretic argument due to Emmanuel Breuillard (private communication; details omitted) shows that if $n$ is odd, then any $n$-dimensional normed space has enough symmetries if and only if it is canonically positioned.
\end{remark}

The following lemma is important for us even though it is an immediate consequence of the (major) theorem of~\cite{AC09} that  the Cheeger body of a given convex body in $\R^n$ is unique (recall Section~\ref{sec:spectral}).

\begin{lemma}\label{lem:cheeger uniqueness} Let $\X=(\R^n,\|\cdot\|_\X)$ be a normed space such that $\mathsf{Isom}(\X)\le \mathsf{O}_n$ is a subgroup of the orthogonal group. Then the isometry group of its Cheeger space $\Ch\X$ satisfies
$$
\mathsf{Isom}(\Ch\X)\supset \mathsf{Isom}(\X).
$$
Consequently, if $\X$ is canonically positioned, then also $\Ch\X$ is canonically positioned.
\end{lemma}

\begin{proof} For any $U\in \mathsf{Isom}(\X)$ we have  $\vol_{n-1}(\partial U\Ch B_\X)/\vol_n(U\Ch B_\X)= \vol_{n-1}(\partial \Ch B_\X)/\vol_n(\Ch B_\X)$, since $U\in \mathsf{O}_n$, and also $U\Ch B_\X\subset U B_\X= B_\X$. Hence (by definition), $U\Ch B_\X$ is a Cheeger body of $B_\X$, so by the uniqueness of the Cheeger body we have $U\Ch B_\X=\Ch B_\X$. Therefore, $U\in \mathsf{Isom}(\Ch \X)$.
\end{proof}

The following corollary is a quick consequence of Lemma~\ref{lem:cheeger uniqueness}.

\begin{corollary}\label{cor chi space} Let $\bfE=(\R^n,\|\cdot\|_\bfE)$ be a symmetric normed space. Then, its Cheeger space $\Ch \bfE$ is also symmetric and  there exists a (unique) symmetric normed space $\chi\bfE=(\R^n,\|\cdot\|_{\chi\bfE})$  such that the Cheeger space of the unitary ideal $\sfS_\bfE$ is the unitary ideal of $\chi\bfE$, i.e., $\Ch \sfS_\bfE=\sfS_{\chi \bfE}.$
\end{corollary}

\begin{proof} The assertion that $\Ch \bfE$ is symmetric coincides with the requirement that $\mathsf{Isom}(\Ch \bfE)$ contains the group $\{-1,1\}^n\rtimes S_n=\{T_\e S_\pi:\ (\e,\pi)\in \{-1,1\}^n\times S_n\}\le \mathsf{O}_n$, where we recall the notation of Example~\ref{ex permutation and sign}. Since we are assuming that $\mathsf{Isom}(\bfE)\supseteq \{-1,1\}^n\rtimes S_n$, this follows from Lemma~\ref{lem:cheeger uniqueness}. Next, for every $U,V\in \mathsf{O}_n$ define $R_{U,V}:\M_n(\R)\to \M_n(\R)$ by $(A\in M_n(\R))\mapsto UAV$. Since $\mathsf{Isom}(\sfS_\bfE)$ contains $\{R_{U,V}:\ U,V\in \mathsf{O}_n\}$, by Lemma~\ref{lem:cheeger uniqueness} so does $\mathsf{Isom}(\Ch S_\bfE)$. A normed space $(\M_n(\R),\|\cdot\|)$ that is invariant under $R_{U,V}$ for all $U,V\in \mathsf{O}_n$  is the unitary ideal of a symmetric normed space $\bfF=(\R^n,\|\cdot\|_{\bfF})$; see e.g.~\cite[Theorem~IV.2.1]{Bha97}. This $\bfF$ is unique (consider the values of $\|\cdot\|_{S_\bfF}$ on diagonal matrices), so we can introduce the notation $\bfF=\chi \bfE$.
\end{proof}

The same reasoning as in the  proof of Corollary~\ref{cor chi space} shows that if $\bfE=(\R^n,\|\cdot\|_\bfE)$ is an unconditional normed space, then so is $\Ch \bfE$. Thus, the space $\Y$ in Lemma~\ref{lem:unconditional composiiton} when $\X_1=\ldots=\X_n=\R$ that satisfies~\eqref{eq:Y conclusion composition-intro} can be taken to unconditional, as seen by an inspection of the proof of Lemma~\ref{lem:unconditional composiiton} (specifically, the operator $S$ in~\eqref{eq:Y conclusion composition-intro} that arises in this case is diagonal, so $S\bfE$ is also unconditional and we can take $\Y=\Ch S \bfE$).

\begin{problem}\label{prob: weyl} We associated above to every symmetric normed space $\bfE=(\R^n,\|\cdot\|_\bfE)$ two symmetric normed spaces $\Ch \bfE=(\R^n,\|\cdot\|_{\Ch\bfE})$ and $\chi\bfE=(\R^n,\|\cdot\|_{\chi\bfE})$. It would be valuable to understand these auxiliary norms on $\R^n$, and in particular how they relate to each other. By the definition of the Cheeger body, its convexity and uniqueness, $\Ch\bfE$ is the unique minimizer of the functional
\begin{equation}\label{eq:CHE ratio}
\bfF\mapsto \frac{\vol_{n-1}\big(\partial B_\bfF\big)}{\vol_n\big(B_\bfF\big)}=\frac{\int_{\partial B_{\bfF}}1\ud x}{\int_{B_\bfF}1\ud x}
\end{equation}
over  all symmetric normed spaces $\bfF=(\R^n,\|\cdot\|_\bfF)$ with $B_\bfF\subset B_\bfE$; denote the set of all such $\bfF$ by $\mathcal{Sym}(\subset B_\bfE)$. In contrast to~\eqref{eq:CHE ratio}, $\chi \bfE$ is the unique minimizer of the functional
\begin{equation}\label{eq:chi E ratio}
\bfF\mapsto \frac{\int_{\partial B_\bfF} \prod_{1\le i<j\le n} |x_i^2-x_j^2|\ud x}{\int_{B_\bfF} \prod_{1\le i<j\le n} |x_i^2-x_j^2|\ud x}
\end{equation}
over the same domain $\mathcal{Sym}(\subset B_\bfE)$. To justify~\eqref{eq:chi E ratio}, observe first that by Corollary~\ref{cor chi space}  we know that $\chi\bfE$ is the unique minimizer of the following functional over $\mathcal{Sym}(\subset B_\bfE)$:
\begin{equation}\label{eq:limiting version noncommutative boundary}
\bfF\mapsto \frac{\vol_{n^2-1}\big(\partial B_{\sfS_\bfF}\big)}{\vol_{n^2}\big(B_{S_\bfF}\big)}=\lim_{\e\to 0^+}\frac{\int_{\big(B_{\sfS_\bfF}+\e B_{\sfS_2^n}\big)\setminus B_{\sfS_\bfF}}1\ud x}{\e\int_{B_{\bfF}}1\ud x}.
\end{equation}
We claim that for every $\bfF\in \mathcal{Sym}(\subset B_\bfE)$ and $\e>0$,
\begin{equation}\label{eq:noncommutative boundary}
\big(B_{\sfS_\bfF}+\e B_{\sfS_2^n}\big)\setminus B_{\sfS_\bfF}=\Big\{A\in M_n(\R):\ s(A)\eqdef \big(s_1(A),\ldots,s_n(A)\big)\in \big(B_\bfF+\e B_{\ell_2^n}\big)\setminus B_\bfF\Big\},
\end{equation}
where  we denote the singular values of $A\in \M_n(\R)$ by $s_1(A)\ge\ldots\ge s_n(A)$. Indeed, if $A$ belongs to the right hand side of~\eqref{eq:noncommutative boundary}, then $\|s(A)\|_\bfF>1$ and $s(A)=x+y$  for $x,y\in \R^n$ that satisfy $\|x\|_\bfF\le 1$ and $\|y\|_{\ell_2^n}\le \e$. Write $A=UDV$, where  $D\in \M_n(\R)$ is the diagonal matrix whose diagonal is the vector $s(A)\in \R^n$, and $U,V\in \mathsf{O}_n$. Let $D(x),D(y)\in \M_n(\R)$ be the diagonal matrices whose diagonals equal $x,y$, respectively. By noting that $\|A\|_{\sfS_\bfF}=\|s(A)\|_{\bfF}>1$ and $A=UD_xV+UD_yV$, where  $\|UD(x)V\|_{\sfS_\bfF}\le 1$ and $\|UD(y)V\|_{\sfS_2^n}\le \e$, we conclude that $A$ belongs to the left hand side of~\eqref{eq:noncommutative boundary}. The reverse inclusion is less straightforward. If $A$ belongs to the left hand side of~\eqref{eq:noncommutative boundary}, then $\|A\|_{\sfS_\bfF}>1$ and $A=B+C$, where $B,C\in \M_n(\R)$ satisfy $\|B\|_{\sfS_{\bfF}}=\|s(B)\|_{\bfF}\le 1$ and $\|C\|_{\sfS_2^n}\le \e$. By an inequality of Mirsky~\cite{Mir60} we have
$\|s(A)-s(B)\|_{\ell_2^n}\le \|A-B\|_{\sfS_2^n}=\|C\|_{\sfS_2^n}\le \e.
$
 Hence $s(A)=s(B)+(s(A)-s(B))\in  (B_\bfF+\e B_{\ell_2^n})\setminus B_\bfF$, i.e., $A$ belongs to the right hand side of~\eqref{eq:noncommutative boundary}. With~\eqref{eq:noncommutative boundary} established, since membership of a matrix $A$ in either $B_\bfF$ or $(B_\bfF+\e B_{\ell_2^n})\setminus B_\bfF$ depends only on $s(A)$, by  the Weyl integration formula~\cite{Wey39} (see~\cite[Proposition~4.1.3]{AGZ10} for the formulation that we are using),
$$
\frac{\int_{\big(B_{\sfS_\bfF}+\e B_{\sfS_2^n}\big)\setminus B_{\sfS_\bfF}}1\ud x}{\int_{B_{\bfF}}1\ud x}=\frac{\int_{\big(B_\bfF+\e B_{\ell_2^n}\big)\setminus B_\bfF} \prod_{1\le i<j\le n} |x_i^2-x_j^2|\ud x}{\int_{B_\bfF} \prod_{1\le i<j\le n} |x_i^2-x_j^2|\ud x}.
$$
Thus~\eqref{eq:chi E ratio} follows from~\eqref{eq:limiting version noncommutative boundary}. Analysing the functional in~\eqref{eq:chi E ratio} seems nontrivial but likely tractable using ideas from random matrix theory. It would be especially interesting to treat the case $\bfE =\ell_\infty^n$. While we have a reasonably good understanding of the (isomorphic) geometry space $\Ch\ell_\infty^n$, its noncommutative counterpart $\chi \ell_\infty^n$ is still mysterious and understanding its geometry is closely related to Conjecture~\ref{weak isomorphic reverse conj1} (and likely also Conjecture~\ref{isomorphic reverse conj1}) in the important special case of the operator norm $\sfS_\infty^n$; see also Remark~\ref{rem:schatten infty implies SE}.
\end{problem}

If $\X=(\R^n,\|\cdot\|_\X)$ is canonically positioned and $\mu$ is a Borel measure on $\R^n$ that is $\mathsf{Isom}(\X)$-invariant, i.e., $\mu(UA)=\mu(A)$ for every $U\in \mathsf{Isom}(\X)$ and every Borel subset $A\subset \R^n$, then consider the scalar product
$$
\forall x,y\in \R^n,\qquad \langle x,y\rangle'\eqdef \int_{\R^n} \langle x,z\rangle \langle y,z\rangle\ud\mu(z).
$$
For every $U\in \mathsf{Isom}(\X)$ and $x,y\in \R^n$ we have
$$
\langle Ux,Uy\rangle'=\int_{\R^n} \langle Ux,z\rangle \langle Uy,z\rangle\ud\mu(z)=\int_{\R^n} \langle x,U^{-1}z\rangle \langle y,U^{-1}z\rangle\ud\mu(z)=\int_{\R^n} \langle x,z\rangle \langle y,z\rangle\ud\mu(z)=\langle x,y\rangle',
$$
where the second step uses the   $\mathsf{Isom}(\X)$-invariance of $\langle\cdot,\cdot\rangle$, and the third step uses the $\mathsf{Isom}(\X)$-invariance of $\mu$. Hence $\langle x,y\rangle'=\lambda\langle x,y\rangle$ for some $\lambda\in \R$ and every $x,y\in \R^n$. By considering the case $x=y$ of this identity and integrating over $x\in S^{n-1}$ one sees that necessarily $n\lambda=\int_{\R^n} \|z\|_{\ell_2^n}^2\ud\mu(z)$. Hence,
\begin{equation}\label{eq:isotropic}
\forall x,y\in \R^n, \qquad \int_{\R^n} \langle x,z\rangle \langle y,z\rangle\ud\mu(z)=\frac{\int_{\R^n} \|z\|_{\ell_2^n}^2\ud\mu(z)}{n}\langle x,y\rangle.
\end{equation}

By establishing~\eqref{eq:isotropic} we have shown that if $\X=(\R^n,\|\cdot\|_\X)$ is a canonically positioned normed space, then any $\mathsf{Isom}(\X)$-invariant Borel measure on $\R^n$ is {\em isotropic}~\cite{GM00,BGVV14} (the converse also holds, i.e., $\X$  is canonically positioned if and only if every $\mathsf{Isom}(\X)$-invariant Borel measure on $\R^n$ is isotropic). In particular, let $\sigma_\X$ be the measure on $S^{n-1}$ that is given by $\sigma_\X(A)=\vol_{n-1}(\{x\in \partial B_\X:\ N_\X(x)\in A\})$ for every measurable $A\subset S^{n-1}$, where for $x\in \partial B_\X$ the vector $N_\X(x)\in S^{n-1}$ is the (almost-everywhere uniquely defined) unit outer normal to $\partial B_\X$ at $x$, i.e., recalling~\eqref{eq:use cauchy}, we use the simpler notation $N_{B_\X}=N_\X$. In other words, $\sigma_\X$ is the image under the Gauss map of the $(n-1)$-dimensional Hausdorff measure on $\partial B_\X$. Then, $\sigma_\X$ is $\mathsf{Isom}(\X)$-invariant because every $U\in \mathsf{Isom}(\X)$ is an orthogonal transformation and $N_\X\circ U=U\circ N_\X$ almost everywhere on $\partial B_\X$. By~\cite{Pet61}, this implies that $\X$ is in its minimum surface area position (recall the proof of Proposition~\ref{prop:equivalence of upper bound and weak iso}), so $\mathrm{MaxProj}(B_\X)\asymp \vol_{n-1}(\partial B_\X)/\sqrt{n}$ by~\cite[Proposition~3.1]{GP99}.

The following corollary follows by substituting the above conclusion into Theorem~\ref{thm:XY version ext}.

\begin{corollary}\label{cor:canonically positioned ext} Suppose that $n\in \N$ and that $\X=(\R^n,\|\cdot\|_{\X})$ and $\Y=(\R^n,\|\cdot\|_{\Y})$ are two $n$-dimensional normed spaces.  Suppose also that $\Y$ is canonically positioned and $B_\Y\subset B_\X$. Then,
$$
\ee(\X)\lesssim \frac{\vol_{n-1}(\partial B_\Y)\diam_{\ell_2^n}(B_\X)}{\vol_n(B_\Y)\sqrt{n}}.
$$
\end{corollary}
The assumption in Corollary~\ref{cor:canonically positioned ext} that $\Y$ is canonically positioned can be replaced by the requirement $\mathrm{MaxProj}(B_\Y)\lesssim \vol_{n-1}(\partial B_\Y)/\sqrt{n}$, which is much less stringent. In particular, by~\cite[Proposition~3.1]{GP99} it is enough to assume here that $B_\Y$ is in its minimum surface area position; see also Section~\ref{sec:negative correlation}.

We will denote the John and L\"owner ellipsoids of a normed space $\X=(\R^n,\|\cdot\|_\X)$ by $\mathscr{J}_\X$ and $\mathscr{L}_\X$, respectively; see~\cite{Hen12}. Thus,  $\mathscr{J}_\X\subset \R^n$ is the ellipsoid of maximum volume that is contained in $B_\X$ and $\mathscr{L}_\X\subset \R^n$ is the ellipsoid of minimum volume that contains $B_\X$. Both of these ellipsoids are unique~\cite{Joh48}. The {\em volume ratio} $\vr(\X)$ of $\X$ and {\em external volume ratio} $\evr(\X)$ of $\X$ are defined by
\begin{equation}\label{eq:def vr evr}
\vr(\X)\eqdef \bigg(\frac{\vol_n(B_\X)}{\vol_n(\mathscr{J}_\X)}\bigg)^{\frac{1}{n}}\qquad\mathrm{and}\qquad  \evr(\X)\eqdef \bigg(\frac{\vol_n(\mathscr{L}_\X)}{\vol_n(B_\X)}\bigg)^{\frac{1}{n}}.
\end{equation}
By the Blaschke--Santal\'o inequality~\cite{Bla17,San49} and the Bourgain--Milman inequality~\cite{BM87},
\begin{equation}\label{eq:state bourgain milman}
\evr(\X)\asymp \vr(\X^*).
\end{equation}

By the above discussion, we can quickly deduce the following  theorem that relates the Lipschitz extension modulus of a canonically positioned space to volumetric and spectral properties of its unit ball.

\begin{theorem}\label{thm: e spec} Suppose that $n\in \N$ and that $\X=(\R^n,\|\cdot\|_{\X})$ is a canonically positioned normed space. Then,
\begin{equation}\label{eq:ext spect}
\ee(\X)\lesssim \frac{\diam_{\ell_2^n}(B_\X)}{\sqrt{n}}\sqrt{\lambda(\X)}\asymp\evr(\X)\sqrt{\lambda(\X)\vol_n(B_\X)^{\frac{2}{n}}}\asymp \vr(\X^*)\sqrt{\lambda(\X)\vol_n(B_\X)^{\frac{2}{n}}}.
\end{equation}
In fact, the minimum of the right hand side of~\eqref{eq:substitute MaxProj} over all the normed spaces $\Y=(\R^n,\|\cdot\|_\Y)$ with $B_\Y\subset B_\X$ is bounded above and below by universal constant multiples of $\diam_{\ell_2^n}(B_\X)\sqrt{\lambda(\X)/n}$.
\end{theorem}

\begin{proof} By Lemma~\ref{lem:cheeger uniqueness} the Cheeger space $\Ch\X$ is canonically positioned. So, by Corollary~\ref{cor:canonically positioned ext} with $\Y=\Ch\X$,
$$
\ee(\X)\lesssim \frac{\vol_{n-1}(\partial \Ch B_\Y)\diam_{\ell_2^n}(B_\X)}{\vol_n(\Ch B_\Y)\sqrt{n}}\stackrel{\eqref{eq:cheeger buser}}{\lesssim}\frac{\diam_{\ell_2^n}(B_\X)}{\sqrt{n}}\sqrt{\lambda(\X)}.
$$
This proves the first inequality in~\eqref{eq:ext spect}. The final equivalence in~\eqref{eq:ext spect} is~\eqref{eq:state bourgain milman}. To prove the rest of~\eqref{eq:ext spect}, let  $r_{\min}=\min\{r>0:\ r B_{\ell_2^n}\supseteq B_\X\}$ denote the radius of the circumscribing Euclidean ball of $B_\X$.  We claim that $r_{\min}B_{\ell_2^n}=\mathscr{L}_\X$. Indeed, for every $U\in \mathsf{Isom}(\X)\subset \mathsf{O}_n$ the ellipsoid  $U\mathscr{L}_\X$ contains $B_\X$ and has the same volume as $\mathscr{L}_\X$, so because the minimum volume ellipsoid that contains $B_\X$ is unique~\cite{Joh48}, it follows  that $U\mathscr{L}_\X=\mathscr{L}_\X$. Hence, the scalar product that corresponds to $\mathscr{L}_\X$ is $\mathsf{Isom}(\X)$-invariant and since $\X$ is canonically positioned, this means that $\mathscr{L}_\X$ is a multiple of $B_{\ell_2^n}$. Now,
\begin{equation*}
\vol_n(B_\X)^{\frac{1}{n}}\evr(\X)\stackrel{\eqref{eq:def vr evr}}{=}\vol_n\big(r_{\min} B_{\ell_2^n}\big)^{\frac{1}{n}}\asymp \frac{r_{\min}}{\sqrt{n}} = \frac{\diam_{\ell_2^n}(B_\X)}{2\sqrt{n}}.
\end{equation*}

The above reasoning shows that the minimum of the right hand side of~\eqref{eq:substitute MaxProj} over all the normed spaces $\Y=(\R^n,\|\cdot\|_\Y)$ with $B_\Y\subset B_\X$ is at most a universal constant multiple of $\diam_{\ell_2^n}(B_\X)\sqrt{\lambda(\X)/n}$ (take $\Y=\Ch\X$). In the reverse direction, for any such $\Y$ by~\eqref{eq:max proj lower} with $L=B_\Y$ we have
$$
\frac{\mathrm{MaxProj}(B_\Y)}{\vol_n(B_\Y)}\gtrsim \frac{\vol_{n-1}(\partial B_\Y)}{\vol_n(B_\Y)\sqrt{n}}\ge \frac{\vol_{n-1}(\partial \Ch B_\X)}{\vol_n(\Ch B_\X)\sqrt{n}}\stackrel{\eqref{eq:cheeger buser}}{\ge} \frac{2\sqrt{\lambda(\X)}}{\pi\sqrt{n}},
$$
where the penultimate step follows from the definition of the Cheeger body $\Ch B_\X$.
\end{proof}

It is natural to expect that if $\X=(\R^n,\|\cdot\|_\X)$ is a canonically positioned normed space, then in Conjecture~\ref{isomorphic reverse conj1} for $K=B_\X$ holds with $S$ the identity matrix and with $L$ being the unit ball of a canonically positioned normed space. We formulate this refined special case of Conjecture~\ref{isomorphic reverse conj1} as the following conjecture.

\begin{conjecture}\label{conj:reverse iso when canonical} Fix $n\in \N$ and a canonically positioned normed space $\X=(\R^n,\|\cdot\|_\X)$. Then, there exists a canonically positioned normed space $\Y=(\R^n,\|\cdot\|_\Y)$ with $\|\cdot\|_\Y\asymp\|\cdot\|_\X$ and $\iq(B_\Y)\lesssim \sqrt{n}$.
\end{conjecture}

Theorem~\ref{thm:strong conjecture for most lp} below shows that Conjecture~\ref{conj:reverse iso when canonical} holds if $\X=\ell_p^n$ for any $p\ge 1$  and infinitely many dimensions $n\in \N$; specifically, it holds if $n$ satisfies the mild arithmetic (divisibility) requirement~\eqref{eq:divisor} below. An obvious question that this leaves  is to prove Conjecture~\ref{conj:reverse iso when canonical} for $\X=\ell_p^n$ and arbitrary $(p,n)\in [1,\infty]\times \N$. We expect that this question is tractable by (likely nontrivially) adapting the approach herein, but we did not make a major effort to do so since obtaining  Conjecture~\ref{conj:reverse iso when canonical} for such a dense set of dimensions $n$ suffices for our purposes (the bi-Lipschitz invariants that we consider  can be estimated from above for any $n\in \N$ since the requirement~\eqref{eq:divisor} holds for some $N\in \N\cap[n,O(n)]$ and $\ell_p^n$ embeds isometrically into $\ell_p^N$). In Section~\ref{sec:volumes and cone measure} we will prove Theorem~\ref{thm:strong conjecture for most lp}, and deduce Theorem~\ref{prop:rounded cube}  from it. Recall Remark~\ref{rem:isomorphic reverse for lpn}, which explains that Conjecture~\ref{isomorphic reverse conj1} when $K$ is the unit ball of $\ell_p^n$ follows (with $S$ the identity matrix) from Theorem~\ref{prop:rounded cube}. Thus, we {\em do} know that a body $L$ as in Conjecture~\ref{isomorphic reverse conj1} exists for all the possible choices of $p\ge 1$ and $n\in \N$, and~\eqref{eq:divisor} is only relevant to ensure that $L$ is the unit ball of a canonically positioned normed space.

\begin{theorem}\label{thm:strong conjecture for most lp}
Fix $n\in \N$ and $p\ge 1$. Conjecture~\ref{conj:reverse iso when canonical} holds for $\X=\ell_p^n$ if the following condition is satisfied.
\begin{equation}\label{eq:divisor}
\exists m\in \N,\qquad m\mid n\quad\mathrm{and}\quad \max\{p,2\}\le m\le e^p.
\end{equation}
\end{theorem}

The following conjecture is a variant of Conjecture~\ref{weak isomorphic reverse conj1-symmetric}.
% (the symmetric version of the weaker variant Conjecture~\ref{weak isomorphic reverse conj1} of Conjecture~\ref{isomorphic reverse conj1}).
\begin{conjecture}\label{conj:weak reverse iso when canonical} Fix $n\in \N$ and a canonically positioned normed space $\X=(\R^n,\|\cdot\|_\X)$. There exists a normed space $\Y=(\R^n,\|\cdot\|_\Y)$ with $B_\Y \subset B_\X$ yet $\sqrt[n]{\vol_n(B_\Y)}\gtrsim \sqrt[n]{\vol_n(B_\X)}$ such that $\iq(B_\Y)\lesssim \sqrt{n}$.
\end{conjecture}
Conjecture~\ref{conj:reverse iso when canonical} requires  $\Y$ to be canonically positioned while   Conjecture~\ref{conj:weak reverse iso when canonical} does not. The reason for this is that if any normed space $\Y$ satisfies the conclusion of Conjecture~\ref{conj:weak reverse iso when canonical}, then also the Cheeger space $\Ch \X$ of $\X$ satisfies it (this is so because the convex body $L$ that minimizes  the second quantity in~\eqref{eq:our quantification} is, by definition, the Cheeger body of $K=B_\X$), and by Lemma~\ref{lem:cheeger uniqueness} the Cheeger space of $\X$ inherits from $\X$ the property of being canonically positioned. This use of the uniqueness of the Cheeger body will be important below. By~\eqref{eq:spoectral reverse iso equiv}, Conjecture~\ref{conj:weak reverse iso when canonical} is equivalent to the following symmetric version of Conjecture~\ref{conj:reverse FK}.
\begin{conjecture}\label{conj:sym reverse FK} If $\X=(\R^n,\|\cdot\|_\X)$ is  a canonically positioned normed space, then
$
 \lambda(\X)\vol(B_\X)^{\frac{2}{n}}\asymp n.
$
\end{conjecture}

The following corollary is a substitution of Conjecture~\ref{conj:sym reverse FK}  into Theorem~\ref{thm: e spec}.

\begin{corollary}\label{cor:if weak reverse holds} If Conjecture~\ref{conj:weak reverse iso when canonical}  (equivalently, Conjecture~\ref{conj:sym reverse FK}) holds for a canonically positioned normed space $\X=(\R^n,\|\cdot\|_\X)$, then the right hand side of~\eqref{eq:substitute MaxProj} when $\Y=\Ch \X$ is  $O(\evr(\X)\sqrt{n})$. Consequently, \begin{equation}\label{eq:ext volume ratio of dual}
\ee(\X)\lesssim \evr(\X)\sqrt{n}\asymp \vr(\X^*)\sqrt{n}.\end{equation}
\end{corollary}
It is worthwhile to note that by~\cite{Bal89}, the rightmost quantity in~\eqref{eq:ext volume ratio of dual} is maximized (over all possible $n$-dimensional normed spaces) when  $\X=\ell_1^n$, in which case we have $\evr(\ell_1^n)\sqrt{n}\asymp n$.

\begin{remark}\label{rem:stabilization} We currently do not have any example of a normed space $\X=(\R^n,\|\cdot\|_\X)$ for which~\eqref{eq:ext volume ratio of dual}  provably does not hold. If~\eqref{eq:ext volume ratio of dual}  were true in general, or even if it were true for a restricted class of normed spaces that is affine invariant and closed under direct sums, such as  spaces that embed into $\ell_1$ with distortion $O(1)$, then it would be an excellent result. When one leaves the realm of canonically positioned spaces, ~\eqref{eq:ext volume ratio of dual}   acquires  a self-improving property\footnote{We recommend checking that the analogous stabilization argument does not lead to a similar self-improvement phenomenon in Conjecture~\ref{isomorphic reverse conj1},  Conjecture~\ref{weak isomorphic reverse conj1} and Corollary~\ref{lem:consequence of reverse isoperimetry}; the computations in Section~4 of~\cite{MPS12} are relevant for this purpose.}  as follows. Suppose that $\X$ is in L\"owner position, i.e., $\mathscr{L}_\X=B_{\ell_2^n}$. Fix $m\in \N$ and consider the $(n+m)$-dimensional space $\X'=\X\oplus_\infty \ell_2^m$. If~\eqref{eq:ext volume ratio of dual} holds for $\X'$, then
\begin{align}\label{eq:tensor trick}
\begin{split}
\ee(\X)\le \ee(\X')&\lesssim  \evr(\X')\sqrt{\dim(\X')}\lesssim \Bigg(\frac{\vol_{n+m}\big(B_{\ell_2^{n+m}}\big)}{\vol_n(B_\X)\vol_m\big(B_{\ell_2^m}\big)}\Bigg)^{\frac{1}{n+m}} \sqrt{n+m}
\\ &=\bigg(\frac{\vol_n(\mathscr{L}_\X)}{\vol_n(B_\X)}\bigg)^{\frac{1}{n+m}}
 \Bigg(\frac{\vol_{n+m}\big(B_{\ell_2^{n+m}}\big)}{\vol_n\big(\ell_2^n\big)\vol_m\big(B_{\ell_2^m}\big)}\Bigg)^{\frac{1}{n+m}}\sqrt{n+m}\asymp \evr(\X)^{\frac{n}{n+m}} n^{\frac{n}{2(n+m)}}m^{\frac{m}{2(n+m)}}.
\end{split}
 \end{align}
 The value of $m$ that minimizes the right hand side of~\eqref{eq:tensor trick} is $m\asymp n\log(\evr(\X)+1)$, for which~\eqref{eq:tensor trick} becomes
 \begin{equation}\label{eq:wow bound}
 \ee(\X)\lesssim \sqrt{n\log\big(\evr(\X)+1\big)}.
 \end{equation}
As $\evr(\X)\le \sqrt{n}$ by John's theorem, \eqref{eq:wow bound} gives $\ee(\X)\lesssim \sqrt{n\log n}$, which would be an improvement of~\cite{JLS86}.  Also, by~\eqref{eq:quoteMP} the bound~\eqref{eq:wow bound} gives $\ee(\X)\lesssim \sqrt{n\log(C_2(\X)+1)}$, which is better than the conjectural bound~\eqref{eq:ee cotype}. Here and throughout, for $1\le p\le 2\le q$ the (Gaussian) type-$p$ and cotype-$q$ constants~\cite{MP76} of a Banach space $(\X,\|\cdot\|_{\X})$, denoted $T_p(\X)$ and $C_q(\X)$, respectively, are the infimum over those $T\in [1,\infty]$ and $C\in [1,\infty]$, respectively, for which the following inequalities hold for every $m\in \N$ and every  $x_1,\ldots,x_m\in \X$, where the expectation is with respect to i.i.d.~standard Gaussian random variables $\g_1,\ldots,\g_m$.
\begin{equation}\label{eq:def type cotype}
\frac{1}{C}\bigg(\sum_{j=1}^m\|x_j\|^q_{\X}\bigg)^{\frac{1}{q}}\le \left(\E\bigg[\Big\|\sum_{j=1}^m \g_j x_j\Big\|_{\X}^2\bigg]\right)^{\frac12} \le T\bigg(\sum_{j=1}^m\|x_j\|^p_{\X}\bigg)^{\frac{1}{p}}.
\end{equation}
 This observation indicates that it might be too optimistic to expect that~\eqref{eq:ext volume ratio of dual} holds in full generality, but it would be very interesting to understand the extent to which it does. Obvious potential counterexamples  are  $\ell_1^n\oplus \ell_2^m$; if~\eqref{eq:ext volume ratio of dual}  holds for these spaces, then $\ee(\ell_1^n)\lesssim \sqrt{n\log n}$ by the above reasoning (with $m\asymp n\log n$), which would be a big achievement because  the best-known bound  remains $\ee(\ell_1^n)\lesssim n$ from~\cite{JLS86}.
\end{remark}

Lemma~\ref{lem:weak iso for enouhg permutations and unconditional} below, whose proof appears in Section~\ref{sec:direct}, shows that Conjecture~\ref{conj:weak reverse iso when canonical} holds for a class of normed space that includes any normed spaces with a $1$-symmetric basis, as well as, say, $\ell_p^n(\ell_q^m)$ for any $n,m\in \N$ and $p,q\ge 1$. Other (related) examples of such spaces arise from Lemma~\ref{lem:unconditional composiiton-later}  below.

\begin{lemma}\label{lem:weak iso for enouhg permutations and unconditional}
Let $\X=(\R^n,\|\cdot\|_\X)$ be an unconditional normed space. Suppose that for any $j,k\in \n$ there is a permutation $\pi\in S_n$ with $\pi(j)=k$ such that $\|\sum_{i=1}^n a_{\pi(i)}e_i\|_\X= \|\sum_{i=1}^n a_{i}e_i\|_\X$ for every $a_1,\ldots,a_n\in \R$. Then, Conjecture~\ref{conj:weak reverse iso when canonical} holds for $\X$. Therefore, we have $\lambda(\X)\vol_n(B_\X)^{2/n}\asymp n$ and  $\ee(\X)\lesssim \evr(\X)\sqrt{n}$.
\end{lemma}

By~\cite[Theorem~2.1]{ST-J80}, any unconditional normed space $\X=(\R^n,\|\cdot\|_\X)$ satisfies $\vr(\X)\lesssim C_2(\X)\sqrt{n}$, where $C_2(\X)$ is the cotype-$2$ constant of $\X$ (this is an earlier special case of~\eqref{eq:quoteMP} in which the logarithmic term is known to be redundant). Hence, if $\X$ satisfies the assumptions of Lemma~\ref{lem:weak iso for enouhg permutations and unconditional}, then we know that
\begin{equation}\label{extension contpye unc}
\ee(\X)\lesssim C_2(\X^*)\sqrt{n}.
\end{equation}
By combining~\cite[Theorem~6]{Bal91-reverse} and~\eqref{eq:state bourgain milman}, for any $p\in [1,\infty]$, if a normed space $\X=(\R^n,\|\cdot\|_\X)$ is isometric to  a quotient of $L_p$ (equivalently, the dual of $\X$ is isometric to a subspace of $L_{p/(p-1)}$), then $$\evr(\X)\lesssim \evr\Big(\ell_{\!\!\frac{p}{p-1}}^n\Big)\asymp \min\left\{n^{\frac{1}{p}-\frac12},1\right\}.$$
Consequently, if $\X=(\R^n,\|\cdot\|_\X)$ satisfies the assumptions of Lemma~\ref{lem:weak iso for enouhg permutations and unconditional} and is also a quotient of $L_p$, then
\begin{equation}\label{eq:quoient of L_p unc}
\ee(\X)\lesssim n^{\max\left\{\frac12,\frac{1}{p}\right\}}.
\end{equation}
Both~\eqref{extension contpye unc} and~\eqref{eq:quoient of L_p unc}  are generalizations of Theorem~\ref{thm:p>2}.

Lemma~\ref{lem:reverse iso for sym} below, whose proof appears  in Section~\ref{sec:volume ratios}, shows that the unitary ideal of any $n$-dimensional normed space with a $1$-symmetric basis (in particular, any Schatten--von Neumann trace class), satisfies  Conjecture~\ref{conj:weak reverse iso when canonical} up to a factor of $O(\sqrt{\log n})$. Upon its substitution into  Lemma~\ref{lem:unconditional composiiton-later} below, more such examples are obtained.

\begin{lemma}\label{lem:reverse iso for sym} Let $\bfE=(\R^n,\|\cdot\|_{\bfE})$ be a symmetric normed space. Conjecture~\ref{conj:weak reverse iso when canonical} holds up to lower order factors for its unitary ideal $\sfS_{\bfE}$. More precisely, there is a normed space $\Y=(\M_n(\R),\|\cdot\|_\Y)$ such that $B_\Y\subset B_{\sfS_\bfE}$ and
\begin{equation}\label{what we know about weak ideal}
\vol_{n^2}(B_\Y)^{\frac{1}{n^2}}\asymp \vol_{n^2}\big(B_{\sfS_\bfE}\big)^{\frac{1}{n^2}}\qquad \mathrm{and}\qquad n\lesssim \iq(B_\Y)\lesssim n\sqrt{\log n}.
\end{equation}
Therefore, we have
$$
n^2\lesssim \lambda\big(\sfS_\bfE\big)\vol_{n^2}\big(B_{\sfS_\bfE}\big)^{\frac{2}{n^2}}\lesssim n^2\log n
\qquad\mathrm{and}\qquad  \ee(\sfS_\bfE)\lesssim \evr(\sfS_\bfE)n\asymp \evr(\bfE)n.
$$
\end{lemma}
For the final assertion of Lemma~\ref{lem:reverse iso for sym}, the fact that $\evr(\sfS_\bfE)\asymp \evr(\bfE)$ follows by combining Proposition~2.2 in~\cite{Sch82}, which states that $\vr(\sfS_\bfE)\asymp \vr(\bfE)$, with~\eqref{eq:state bourgain milman} and the  duality $\sfS_\bfE^*=\sfS_{\bfE^{\textbf *}}$ (e.g.~\cite[Theorem~1.17]{Sim79}).

The proof of Lemma~\ref{lem:reverse iso for sym} also shows (see Remark~\ref{rem:schatten infty implies SE} below) that if we could prove Conjecture~\ref{conj:weak reverse iso when canonical} for $\sfS_\infty^n$, then it would follow that $\sfS_\bfE$ satisfies Conjecture~\ref{conj:weak reverse iso when canonical} for any symmetric normed space $\bfE=(\R^n,\|\cdot\|_\bfE)$, i.e.,  the logarithmic factor in~\eqref{what we know about weak ideal} could be replaced by a universal constant.

By substituting Lemma~\ref{lem:reverse iso for sym} into Corollary~\ref{cor:if weak reverse holds} and using volume ratio computations of Sch\"utt~\cite{Sch82}, we will derive in Section~\ref{sec:volume ratios} the following proposition.

\begin{proposition}\label{prop:weak rev iso sym}  If $\bfE=(\R^n,\|\cdot\|_{\bfE})$ is a symmetric normed space, then
$$
\ee(\bfE)\lesssim \diam_{\ell_2^n}\big(B_\bfE\big)\|e_1+\ldots+e_n\|_{\bfE}\qquad\mathrm{and}\qquad \ee\big(\sfS_\bfE\big)\lesssim \diam_{\ell_2^n}\big(B_\bfE\big)\|e_1+\ldots+e_n\|_{\bfE}\sqrt{n\log n}.
$$
\end{proposition}

The following remark sketches an alternative approach towards  Conjecture~\ref{isomorphic reverse conj1}  when $K$ is the hypercube  $[-1,1]^n$ that differs from how we will prove Theorem~\ref{prop:rounded cube}. It  yields the desired result up to a lower order factor that grows extremely slowly; specifically, it constructs an origin-symmetric convex body $L\subset [-1,1]^n$ for which $[-1,1]^n\subset \exp(O(\log^*\!\! n)) L$ and $\iq(L)= \exp(O(\log^*\!\! n))$. Here, for each  $x\ge 1$ the quantity $\log^*\!\! x$ is defined to be the  $k\in \N$ such that $\mathsf{tower}(k-1)\le x< \mathsf{tower}(k)$ for the  sequence $\{\mathsf{tower}(i)\}_{i=0}^\infty$ that is defined  by $\mathsf{tower}(0)=1$ and $\mathsf{tower}(i+1)=\exp(\mathsf{tower}(i))$. We think that this approach is worthwhile to describe despite the fact that it falls slightly short of fully establishing  Conjecture~\ref{isomorphic reverse conj1} for $[-1,1]^n$ due to its flexibility that could be used for other purposes, as well as due to its intrinsic interest.

\begin{remark}\label{rem:nested lp}
Fix $n\in \N$ and $q\ge 1$. Since the $n$'th root of the volume of the unit ball of $\ell_q^n$ is of order $n^{-1/q}$ and $\ell_q^n$ is in minimum surface area position, we can restate~\eqref{eq:crude BN} as
\begin{equation}\label{eq:iq lq}
\iq\big(B_{\ell_q^n}\big)\asymp \min\left\{\sqrt{qn},n\right\}.
\end{equation}
In particular, for $\Y=\ell_q^n$ with $q=\log n$, we have  $\|\cdot\|_{\Y}\asymp \|\cdot\|_{\ell_\infty^n}$ and  $\iq(\Y)\lesssim \sqrt{n\log n}$, which already comes  close to the conclusion of  Conjecture~\ref{isomorphic reverse conj1}. We can do better using the following evaluation of the isoperimetric quotient of the unit ball of $\ell_p^n(\ell_q^m)$, which holds for every $n,m\in \N$ and $p,q\ge 1$.
\begin{equation}\label{eq:iq of lpn lqm}
\iq\big(B_{\ell_p^n(\ell_q^m)}\big)\asymp
\left\{\begin{array}{ll}nm& m\le \min\left\{\frac{p}{n},q\right\},\\ n\sqrt{qm}  & q\le m\le \frac{p}{n},\\ \sqrt{pnm}&\frac{p}{n}\le m\le \min\{p,q\},\\\sqrt{pqn}
& \max\left\{\frac{p}{n},q\right\}\le m\le p,\\
m\sqrt{n}& p\le m\le q,\\
\sqrt{qnm}  & m\ge\max\{p,q\}.\end{array} \right.
\end{equation}
We will prove~\eqref{eq:iq of lpn lqm} in Section~\ref{sec:volumes and cone measure}. Note that when $m=1$ this yields~\eqref{eq:iq lq}. The case  $n=m$ of~\eqref{eq:iq of lpn lqm} is equivalent to~\eqref{eq:max prj lpn ellqn} since $\ell_p^n(\ell_q^m)$ is canonically positioned (it belongs to the class of spaces in Example~\ref{ex permutation and sign}) and using a simple evaluation of the volume of its unit ball (see~\eqref{eq:volume of ellp ellq} below). The range of~\eqref{eq:iq of lpn lqm} that is most pertinent for the present context is  $m\ge \max\{p,q\}$, which has the feature that the factor that multiplies the quantity
$$\sqrt{nm}=\sqrt{\dim\big(\ell_p^n(\ell_q^m)\big)}$$
is $O(\sqrt{q})$ and there is no dependence on $p$. This can be used as follows. Suppose that $n=ab$ for $a,b\in \N$ satisfying  $a\asymp n/\log n$ and $b\asymp \log n$. Identify  $\ell_\infty^n$ with $\ell_\infty^{a}(\ell_\infty^{b})$. If we  set $\Y=\ell_p^a(\ell_q^b)$ for $p=\log a\asymp \log n$ and $q=\log b\asymp \log\log n$, then $\|\cdot\|_\Y\asymp \|\cdot\|_{\ell_\infty^n}$, while $\iq(B_\Y)\asymp \sqrt{n\log\log n}$ by~\eqref{eq:iq of lpn lqm}. By iterating  we get that  for infinitely many $n\in \N$ there is a normed space $\Y=(\R^n,\|\cdot\|_\Y)$ for which $\|\cdot\|_\Y\le \|\cdot\|_{\ell_\infty^n}\le \exp(O(\log^*\!\! n))\|\cdot\|_\Y$ and $\iq(B_\Y)=\exp(O(\log^*\!\! n))$. Even though the set of $n\in \N$ for which this works is not all of $\N$, it is quite dense in $\N$ per Lemma~\ref{lem log* approx} below.  This will allow us to deduce that a space $\Y$ with the above properties exists for every $n\in \N$; see Section~\ref{sec:direct} for the details.
\end{remark}

\begin{comment}
\begin{question} Can one mimic the strategy in Remark~\ref{rem:nested lp} for $\sfS_\infty^n$, thus proving Conjecture~\ref{isomorphic reverse conj1} for the unit ball of $S_\infty^n$ up to  $\exp(O(\log^*\!\! n))$ factors,  and correspondingly reducing the lower order factors in~\eqref{eq:all of schatten} and~\eqref{eq:schatten extension}? This seems to be technically challenging, but accessible.  The best bound that we currently know uses
$$
\forall q\ge 1,\qquad \iq\big(B_{\sfS_q^n}\big)\asymp n\sqrt{\min \{q,n\}},
$$
which is a restatement of~\eqref{eq:quote with gid} because $\vol_{n^2-1}(\partial B_{\sfS_q^n})\asymp \dim(S_q^n)^{1/2}\mathrm{MaxProj}(B_{\sfS_q^n})=n\mathrm{MaxProj}(B_{\sfS_q^n})$, as $\sfS_q^n$ is canonically positioned, and by~\cite[equation~(2.2)]{Sch82} we have
\begin{equation}\label{eq:volume of schatten p}
\vol_{n^2}\big(B_{\sfS_q^n}\big)^{\frac{1}{n^2}}\asymp \frac{1}{n^{\frac12+\frac{1}{q}}}.
\end{equation}
Consequently, if we consider $\Y=\sfS_q^n$ for $q=\log n$, then $\|\cdot\|_\Y\asymp\|\cdot\|_{\sfS_\infty^n}$ and $\iq(\Y)\asymp n\sqrt{\log n}$. If one would like to apply the idea in Remark~\ref{rem:nested lp} in order to improve this bound, then it would be desirable to obtain an analogue of~\eqref{eq:iq of lpn lqm} for the unitary ideal $\sfS_{\mathbf{E}}$ of a suitable normed space $\mathbf{E}=(\M_{n\times m}(\R),\|\cdot\|_{\mathbf{E}})$ for which the standard basis $\{e_i\otimes e_j:\ (i,j)\in \n\times \{1,\ldots,m\}\}$ is $1$-symmetric, rather than the $\ell_p^n(\ell_q^m)$ norm.
\end{question}
\end{comment}

\begin{remark}\label{rem:msevr} Recalling Remark~\ref{rem:cheeger position}, Conjecture~\ref{weak isomorphic reverse conj1} is equivalent to the assertion  that if a normed space $\X=(\R^n,\|\cdot\|_\X)$ is in Cheeger position, then $\vol_n(\Ch B_\X)^{1/n} \gtrsim \vol_n(B_\X)^{1/n}$ and $\iq(\Ch B_\X)\lesssim \sqrt{n}$. Since $\Ch \X$ is in minimum surface area position when $\X$ is in Cheeger position (as explained in Remark~\ref{rem:cheeger position}), the proof of Proposition~\ref{prop:equivalence of upper bound and weak iso} shows that Conjecture~\ref{weak isomorphic reverse conj1} implies that if $\X$ is in Cheeger position, then
\begin{equation}\label{eq:conj most general}
\ee(\X)\lesssim \frac{\diam_{\ell_2^n}(B_\X)}{\vol_n(B_\X)^{\frac{1}{n}}}
\end{equation}
In fact, the right hand side of~\eqref{eq:substitute MaxProj} is at most the right hand side of~\eqref{eq:conj most general} for a suitable choice of normed space $\Y=(\R^n,\|\cdot\|_\Y)$, specifically for $\Y=\Ch \X$. The discussion in Section~\ref{sec:positions}  was  about establishing~\eqref{eq:conj most general} when $\X$ is canonically positioned (conceivably that assumption implies that $\X$ is in Cheeger position or close to it, which would be a worthwhile to prove, if true). Even though, as we explained earlier,  given the current state of knowledge, understanding the Lipschitz extension problem for canonically positioned spaces is the most pressing issue for future research, it would be very interesting to study if~\eqref{eq:conj most general} holds in other situations. For examples, we pose the following two natural questions.

\begin{question}\label{Q:min surface position ratio} Does~\eqref{eq:conj most general} hold if the normed space $\X=(\R^n,\|\cdot\|_\X)$ is in  minimum surface area position?
\end{question}

The extent to which $\Pi\X$ is close to being in minimum surface area position when $\X$ is in minimum surface area position seems to be unknown. Therefore, the connection between Question~\ref{Q:question zonoid l2 diam} below and Question~\ref{Q:min surface position ratio}  is unclear, but even if there is no formal link between these two questions, both are natural next steps beyond the setting of canonically positioned normed spaces.

\begin{question}\label{Q:question zonoid l2 diam} Let $\bfZ=(\R^n,\|\cdot\|_\bfZ)$ be a normed space in minimum surface area position. Does~\eqref{eq:conj most general} hold for the normed space $\X=\Pi\bfZ$ whose unit ball is the projection body of $B_\X$?
\end{question}
If $\bfZ=(\R^n,\|\cdot\|_\bfZ)$ is a normed space in minimum surface area position, then
\begin{equation}\label{eq:msevr for projection bodies}
\frac{\diam_{\ell_2^n}(\Pi B_\bfZ)}{\vol_n(\Pi B_\bfZ)^{\frac{1}{n}}}\asymp \sqrt{n}.
\end{equation}
Indeed, because $\bfZ$ is in minimum surface area position, we have $\vol_n(\Pi B_\bfZ)^{1/n}\asymp \vol_{n-1}(\partial B_\bfZ)/n$ by~\cite[Corollary~3.4]{GP99}, and also $\mathrm{MaxProj}(B_\bfZ)\asymp \vol_{n-1}(\partial B_\bfZ)/\sqrt{n}$  by combining~\cite[Proposition~3.1]{GP99} and~\eqref{eq:max proj lower}.  We can therefore justify~\eqref{eq:msevr for projection bodies} using these results from~\cite{GP99} and duality as follows.
$$
\frac{\diam_{\ell_2^n}(\Pi B_\bfZ)}{\vol_n(\Pi B_\bfZ)^{\frac{1}{n}}}\asymp \frac{n\|\Id_n\|_{\Pi\bfZ\to \ell_2^n}}{\vol_{n-1}(\partial B_\bfZ)}=
\frac{n\|\Id_n\|_{ \ell_2^n\to \Pi^{\textbf *}\bfZ}}{\vol_{n-1}(\partial B_\bfZ)}=\frac{n\max_{z\in S^{n-1}} \|z\|_{\Pi^{\textbf *}\bfZ}}{\vol_{n-1}(\partial B_\bfZ)}\stackrel{\eqref{eq:use cauchy}}{=} \frac{n\mathrm{MaxProj}(B_\bfZ)}{\vol_{n-1}(\partial B_\bfZ)}\asymp\sqrt{n}.
$$
By this observation,  a positive answer to Question~\ref{Q:question zonoid l2 diam} would show that $\ee(\Pi\bfZ)\lesssim \sqrt{n}$ for any normed space $\bfZ=(\R^n,\|\cdot\|_\bfZ)$. Indeed, if we take  $S\in \SL_n(\R)$ such that $S\bfZ$ is in minimum surface area position, then by~\cite{Pet67} we know that $\Pi \bfZ$ and $\Pi S\bfZ$ are isometric, so   $\ee(\Pi \bfZ)=\ee(\Pi S\bfZ)$. As the class of projection bodies coincides with the class of zonoids~\cite{Bol69,Sch83}, which coincides with the class of convex bodies whose polar is the unit ball of a subspace of $L_1$, we have thus shown that a positive answer to Question~\ref{Q:question zonoid l2 diam} would imply  the following conjecture (which would simultaneously improve~\eqref{eq:dist to hilbert ext} and generalize Theorem~\ref{thm:p>2}).
\begin{conjecture}\label{conj:zonoids} For any normed space $\X=(\R^n,\|\cdot\|_\X)$ we have $\ee(\X)\lesssim \cc_{L_1}(\X^*)\sqrt{n}.$
\end{conjecture}
Note that Conjecture~\ref{conj:zonoids} is consistent with the  estimate $\ee(\X)\lesssim \evr(\X)\sqrt{n}$ that has been arising thus far. Indeed, if  $\X^*$ is isometric to a subspace of $L_1$ (it suffices to consider only this case in Conjecture~\ref{conj:zonoids} by a well-known differentiation argument; see e.g.~\cite[Corollary~7.10]{BL00}), then we have the bound $\evr(\X)\lesssim 1$ which can be seen to hold by combining~\eqref{eq:state bourgain milman} with~\eqref{eq:quoteMP}, since $C_2(\X^*)\le C_2(L_1)\lesssim 1$.\footnote{Alternatively,  $\evr(\X)\lesssim 1$ can be justified by writing $\X=\Pi\bfZ$ for some normed space $\bfZ=(\R^n,\|\cdot\|_\bfZ)$ (using~~\cite{Bol69,Sch83}), and then applying the bound~\eqref{eq:msevr for projection bodies} that we derived above (this even demonstrates that the external volume ratio of $\Pi\bfZ$ is $O(1)$ when $\bfZ$ is in minimum surface area position rather when $\bfZ$ is in L\"owner position). Actually,  the sharp bound $\evr(\X)\le \evr(\ell_\infty^n)$ holds, as seen by combining~\cite[Theorem~6]{Bal91-reverse} with Reisner's theorem~\cite{Rei86} that the Mahler conjecture~\cite{Mah39} holds for zonoids.}
\end{remark}

Relating $\ee(\X)$ to $\evr(\X)$ is valuable since the Lipschitz extension modulus is for the most part shrouded in mystery, while  the literature contains extensive knowledge on volume ratios (we have already seen several examples of such consequences above, and we will derive more later). Section~\ref{sec:volume ratios} contains  examples of volume ratio evaluations for various canonically positioned normed spaces. Through their substitution into Corollary~\ref{cor:if weak reverse holds}, they illustrate how our work yields a range of new Lipschitz extension results, some of which are currently conjectural because they hold assuming Conjecture~\ref{conj:weak reverse iso when canonical}  for the respective spaces; specifically, consider the Lipschitz extension bounds that correspond  to using~\eqref{eq:sep of operator from ellp to ellq} and~\eqref{eq:projective in overview} with~\cite{LN05}.

\subsubsection{Intersection with a Euclidean ball}\label{sec:intersection} Fix an integer $n\ge 2$ and a canonically positioned normed space $\X=(\R^n,\|\cdot\|_\X)$.  A natural first attempt to prove Conjecture~\ref{conj:weak reverse iso when canonical}  for $\X$ is to consider the normed space $\Y=(\R^n,\|\cdot\|_\Y)$ such that  $B_\Y=B_\X\cap rB_{\ell_2^n}$ for a suitably chosen $r>0$ (equivalently, $\|x\|_\Y=\max\{\|x\|_\X,\|x\|_{\ell_2^n}/r\}$ for every $x\in \R^n$). However, we checked with G.~Schechtman that this fails even when $\X=\ell_\infty^n$. Specifically, if the $n$'th root of the volume of $B_{\ell_\infty^n}\cap (rB_{\ell_2^n})$ is at least a universal constant, then necessarily $r\gtrsim \sqrt{n}$, but
\begin{equation}\label{eq:with gid s version}
\forall s>0,\qquad \iq\big(B_{\ell_\infty^n}\cap (s\sqrt{n}B_{\ell_2^n})\big)\gtrsim_s n.
\end{equation}
A justification of~\eqref{eq:with gid s version} appears in Section~\ref{sec:log weak} below. In terms of the quantification~\eqref{eq:our quantification} of Conjecture~\ref{conj:weak reverse iso when canonical} that is pertinent to the applications that we study herein, we will also show in Section~\ref{sec:log weak} that
\begin{equation}\label{eq:best radius cube}
\min_{r>0}\frac{\iq\big(B_{\ell_\infty^n}\cap (rB_{\ell_2^n})\big)}{\sqrt{n}}\left(\frac{\vol_n(B_{\ell_\infty^n})}{\vol_n\big(B_{\ell_\infty^n}\cap (rB_{\ell_2^n})\big)}\right)^{\frac{1}{n}}\asymp \sqrt{\log n},
\end{equation}
where the minimum in the right hand side of~\eqref{eq:best radius cube} is attained at some $r>0$ that satisfies $r\asymp \sqrt{n/\log n}$.

Even though the above bounds demonstrate that it is impossible to  resolve  Conjecture~\ref{conj:weak reverse iso when canonical}  by intersecting with a Euclidean ball,  this approach cannot fail by more than a lower-order factor; the reasoning that proves this assertion was shown to us by B.~Klartag and E.~Milman in unpublished private communication that is explained with their permission in Section~\ref{sec:log weak}. Specifically, we have the following proposition.

\begin{proposition}\label{prop:K convexity}For any normed space $\X=(\R^n,\|\cdot\|_\X)$ there exist a matrix  $S\in \SL_n(\R)$ and a radius $r>0$ such that for $L=(SB_\X)\cap (r B_{\ell_2^n})\subset SB_\X$ we have $\iq(L)\lesssim \sqrt{n}$ and $\sqrt[n]{\vol_n(L)}\gtrsim \sqrt[n]{\vol_n(B_\X)}/K(\X)$, where $K(\X)$ is the $K$-convexity constant of $\X$. If $\X$ is canonically positioned, then this holds when  $S$ is the identity matrix.
\end{proposition}

For Proposition~\ref{prop:K convexity}, the $K$-convexity constant of $\X$ is an isomorphic invariant that was introduced by Maurey and Pisier~\cite{MP76}; we defer recalling its definition to Section~\ref{sec:log weak} since for the discussion here it suffices to state the following bounds that relate $K(\X)$ to quantities that we already encountered. Firstly,
\begin{equation}\label{eq:pisier bound K convex}
K(\X)\lesssim \log \big(d_{\BM}(\ell_2^n,\X)+1\big)\lesssim \log n,
\end{equation}
The first inequality in~\eqref{eq:pisier bound K convex} is a useful theorem of Pisier~\cite{Pis80-K,Pis80}. The second inequality in~\eqref{eq:pisier bound K convex} follows from John's theorem~\cite{Joh48}, though  for this purpose it suffices to use the older Auberbach lemma (see~\cite[page~209]{Ban32} and~\cite{day47,Tay47}). By~\cite{Pis80} (see also e.g.~\cite[Lemma~17]{JS01}) the rightmost quantity in~\eqref{eq:pisier bound K convex} can be reduced if $\X$ is a subspace of $L_1$, namely we have
\begin{equation}\label{eq:sqrt log L1}
K(\X)\lesssim \cc_{L_1}(\X)\sqrt{\log n}.
\end{equation}
Secondly, $K(\X)$ relates to the notion of type that we recalled  in~\eqref{eq:def type cotype} through the following bounds:
\begin{equation}\label{eq:K convexity type}
T_{1+\frac{c}{K(\X)^2}}(\X)^{\frac12}\lesssim K(\X)\le \min_{p\in (1,2]}e^{( C T_p(\X))^{\frac{p}{p-1}}},
\end{equation}
The qualitative meaning of~\eqref{eq:K convexity type} is that the $K$-convexity constant of a Banach space is finite if and only if it has type $p$ for some $p>1$; this is a landmark theorem of Pisier (the `if' direction is due to~\cite{Pis82} and the `only if' direction is due to~\cite{Pis73}).  Since in our setting $\X$ is finite dimensional ($\dim(\X)=n\ge 2$), such a qualitative statement is vacuous without its quantitative counterpart~\eqref{eq:K convexity type}. The first inequality in~\eqref{eq:K convexity type}  can be deduced from~\cite{Pis83} (together with the computation of the implicit dependence on $p$ in~\cite{Pis83} that was carried out in~\cite[Lemma~32]{HLN16}). The second inequality in~\eqref{eq:K convexity type} follows from an examination of the proof in~\cite{Pis82}. We omit the details of both deductions as they would result in a (quite lengthy and tedious) digression. It would be very interesting to determine the best bounds in the context of~\eqref{eq:K convexity type}.

Proposition~\ref{prop:K convexity} combined with~\eqref{eq:pisier bound K convex} implies that Conjecture~\ref{weak isomorphic reverse conj1} holds up to a logarithmic factor in the sense that for every integer $n\ge 2$, any origin-symmetric convex body $K\subset \R^n$ admits a matrix $S\in \SL_n(\R)$ and an origin-symmetric convex body $L\subset SK$ such that
\begin{equation}\label{eq:our quantification with log}
\frac{\iq(L)}{\sqrt{n}}\bigg(\frac{\vol_n(K)}{\vol_n(L)}\bigg)^{\frac{1}{n}}\lesssim \log n.
\end{equation}
Furthermore, by~\eqref{eq:sqrt log L1} the $\log n$ in~\eqref{eq:our quantification with log}  can be replaced by $\sqrt{\log n}$ if $K$ is the unit ball of a subspace of $L_1$ (equivalently, the polar of $K$ is a zonoid), and by  the second inequality in~\eqref{eq:K convexity type} if $p>1$, then the $\log n$ in~\eqref{eq:our quantification with log} can be replaced by a dimension-independent quantity that depends only on $p$ and the type-$p$ constant of the norm whose unit ball is $K$. Also, Corollary~\ref{lem:consequence of reverse isoperimetry} holds with the right hand side of~\eqref{eq:desired Lambda} multiplied by $\log n$, and the reverse Faber--Krahn inequality of Conjecture~\ref{conj:reverse FK}  holds up to a factor of $(\log n)^2$, i.e., for any origin-symmetric convex body $K\subset \R^n$ there is  $S\in \SL_n(\R)$ such that
$
 \lambda(SK)\vol(K)^{2/n}\lesssim n(\log n)^2$.
If $\X=(\R^n,\|\cdot\|_\X)$ is a canonically positioned normed space, then it follows that for a suitable choice of normed space $\Y=(\R^n,\|\cdot\|_\Y)$ the right hand side of~\eqref{eq:ext XY version}, and hence also $\ee(\X)$ by Theorem~\ref{thm:XY version ext},  is at most a universal constant multiple of $\evr(\X)\sqrt{n}\log n$, and also $n\lesssim \lambda(\X)\vol_n(B_\X)^{2/n}\lesssim n(\log n)^2$.

\subsection{Randomized clustering}\label{sec:cluster} All of the new upper bounds on  Lipschitz extension moduli that we stated above rely on a  geometric structural result for  finite dimensional normed spaces (and subsets thereof). Beyond the application to Lipschitz extension, this result is of value in its own right because it yields an improvement of a basic randomized clustering method from the computer science literature.

The link between random partitions of metric spaces and Lipschitz extension was found in~\cite{LN05}. We will adapt the methodology of~\cite{LN05} to deduce the aforementioned  Lipschitz extension theorems from our new bound on randomized partitions of normed spaces. In order to formulate the corresponding definitions and results, one must first set  some groundwork for a  notion of a random partition of a metric space, whose subsequent applications  necessitate certain measurability requirements.

A framework for reasoning about random partitions of metric spaces was developed in~\cite{LN05}, but  we will formulate a  different approach. The reason for this is that the definitions of~\cite{LN05} are in essence  the minimal requirements that allow one to use at once several different types of random partitions for  Lipschitz extension, which leads to definitions that are more cumbersome than the approach that we take below. Greater simplicity is not the only reason why we chose  to formulate a foundation that differs from~\cite{LN05}. The approach that we take  is easier to implement, and, importantly, it yields a bi-Lipschitz invariant, while we do not know if the corresponding notions in~\cite{LN05} are bi-Lipschitz invariants (we suspect that they are {\em not}, but we did not attempt to construct examples that demonstrate this). The Lipschitz extension theorem of~\cite{LN05} is adapted accordingly in Section~\ref{sec:ext}, thus making the present article self-contained, and also yielding simplification  and further applications. Nevertheless, the key geometric ideas that underly this use of random partitions are the same as in~\cite{LN05}.

Obviously, there are no measurability issues when one considers finite metric spaces (in our setting, finite subsets of normed spaces). The ensuing measurability discussions can therefore be ignored in the finitary setting. In particular, the computer science literature on random partitions  focuses exclusively on finite objects. So, for the purpose of algorithmic clustering, one does not need the more general treatment below, but it is needed for the purpose of Lipschitz extension.

\vspace{-0.1in}
\subsubsection{Basic definitions related to random partitions}\label{sec:random partition intro} Let $(\MM,d_\MM)$ be a metric space. Suppose that $\Part\subset 2^\MM$ is a partition of $\MM$. For $x\in \MM$, denote by $\mathscr{P}(x)\subset \MM$ the unique element of $\mathscr{P}$ to which $x$ belongs. The sets $\{\Part(x)\}_{x\in \MM}$ are often called the {\em clusters} of $\Part$. Given $\Delta>0$, one says that $\mathscr{P}$ is  $\Delta$-bounded if $\diam_\MM(\Part(x))\le \Delta$ for every $x\in \MM$, where $\diam_\MM(S)=\sup\{d_\MM(x,y):\ x,y\in S\}$ denotes the diameter of $\emptyset \neq S\subset \MM$.

Suppose that $(\ZZ,\cF)$ is a measurable space, i.e., $\ZZ$ is a set and $\cF\subset 2^\ZZ$ is a $\sigma$-algebra of subsets of $\ZZ$. Recall (see~\cite{Jac68} or the convenient survey~\cite{Wag77}) that if $(\MM,d_\MM)$ is a metric space, then a set-valued mapping $\Gamma:\ZZ\to 2^\MM$ is said to be strongly measurable if for every closed subset $E\subset \MM$ we have
\begin{equation}\label{eq:def measurability}
\Gamma^{-}(E)\eqdef \big\{z\in \ZZ:\ E\cap \Gamma(z)\neq \emptyset\big\}\in \cF.
\end{equation}
%This implies in particular that $\{z\in \ZZ:\ x\in \Gamma(z)\}=\{z\in \ZZ:\ \{x\}\cap \Gamma(z)\neq\emptyset\}\in \cF$ for all $x\in \MM$.

Throughout what follows, when we say that $\Part$ is a  {\em random partition}  of a metric space $(\MM,d_\MM)$, we mean the following (formally, the objects that  we will be considering  are random {\em ordered} partitions into countably many clusters). There is  a probability space $(\Omega,\Pr)$  and a sequence of set-valued mappings $$\left\{\Gamma^k:\Omega\to 2^\MM\right\}_{k=1}^\infty.$$ We write $\Part^\omega=\{\Gamma^k(\omega)\}_{k=1}^\infty$ for each $\omega\in \Omega$ and require that  the mapping $\omega\mapsto \Part^\omega$ takes values in partitions of $\MM$. We also require that for every fixed $k\in \N$, the set-valued mapping $\Gamma^k:\Omega\to 2^\MM$ is strongly measurable, where the  $\sigma$-algebra on $\Omega$ is the $\Pr$-measurable sets. Given $\Delta>0$, we say that $\Part$ is a $\Delta$-bounded random partition of $(\MM,d_\MM)$ if $\Part^\omega$ is a $\Delta$-bounded partition of $(\MM,d_\MM)$ for every $\omega\in \Omega$.

\begin{remark}\label{rem:boundedness convention}Recall that when we say that $\X=(\R^n,\|\cdot\|_{\X})$ is a normed space we mean that the underlying vector space is $\R^n$, equipped with a norm $\|\cdot\|_{\X}:\R^n\to [0,\infty)$. By doing so, we introduce a second metric on $\X$, i.e., $\R^n$  is  also endowed with the standard Euclidean structure that corresponds to the norm $\|\cdot\|_{\ell_2^n}$. This leads to ambiguity when we discuss $\Delta$-bounded partitions of $\X$ for some $\Delta>0$, as there are two possible metrics with respect to which one could bound the diameters of the clusters. In fact, a key aspect of our work is that  it can be beneficial to consider another auxiliary norm $\|\cdot\|_{\Y}$ on $\R^n$, as in e.g.~Theorem~\ref{thm:XY version ext}, thus leading to three possible interpretations of $\Delta$-boundedness of a partition of $\R^n$. To avoid any confusion, we will adhere throughout to the convention that when we say that a partition $\Part$ of $\X$ is $\Delta$-bounded we mean exclusively that all the clusters of $\Part$ have diameter at most $\Delta$ with respect to the norm $\|\cdot\|_{\X}$.
\end{remark}

\subsubsection{Iterative ball partitioning}\label{sec:iterated balls} Fix $\Delta\in (0,\infty)$. {Iterative ball partitioning} is a common procedure to construct a $\Delta$-bounded random partition of a metric probability space.  We will next describe it to clarify at the outset the nature of the objects that we investigate, and because our new positive partitioning results  are solely about this  type of partition.  Thus,  our contribution to the theory of random partitions  is a sharp understanding of the performance of iterative ball partitioning of normed spaces, and, importantly, the demonstration of the utility of its implementation using balls that are induced by a suitably chosen auxiliary norm rather than the given norm  that we aim to study. On the other hand, our impossibility results rule out the existence of any random partition whatsoever with certain desirable properties.

The iterative ball partitioning method is a ubiquitous tool in metric geometry and algorithm design.  To the best of our knowledge, it was first used by Karger, Motwani and Sudan~\cite{KMS98} and the aforementioned work~\cite{CCGGP98} in the context of normed spaces, and it has become very influential in the context of general metric spaces due to its use in that setting (with the important twist of randomizing the radii) by Calinescu, Karloff and Rabani~\cite{CKR04}. To describe it, suppose that $(\MM,d_\MM)$ is a  metric space and that $\mu$ is a Borel probability measure on $\MM$. Let $\{\XX_k\}_{k=1}^\infty$ be a sequence of i.i.d. points sampled $\mu$. Define inductively a sequence $\{\Gamma^k\}_{k=1}^\infty$ of random subsets of $\MM$ by setting $\Gamma^1=B_\MM(\XX_1,\Delta/2)$ and
$$
\forall k\in \{2,3,\ldots,\},\qquad \Gamma^k\eqdef B_\MM\Big(\XX_k,\frac{\Delta}{2}\Big)\setminus \bigcup_{j=1}^{k-1} B_\MM\Big(\XX_j,\frac{\Delta}{2}\Big).
$$

By design, $\diam_\MM(\Gamma^k)\le \Delta$. Under mild assumptions on $\MM$ and $\mu$ that are simple to check, $\Gamma^k$ will have the measurability properties that we require below and  $\Part=\{\Gamma^k\}_{k=1}^\infty$ will be a partition of $\MM$ almost-surely. While initially the clusters of  $\Part$ are quite ``tame,'' e.g.~they start out as balls in $\MM$, as the iteration proceeds and we discard the balls that were used thus far, the resulting sets become increasingly ``jagged.'' In particular, even when the underlying metric space $(\MM,d_\MM)$ is very ``nice,'' the clusters of $\Part$ need not be connected; see Figure~\ref{fig:hexagons}. Nevertheless, we will see that such a simple procedure results in a random partition with probabilistically small boundaries in  sense that will be described rigorously below.

\begin{figure}[h]
\centering
\fbox{
\begin{minipage}{6.25in}
\centering
\smallskip
\includegraphics[scale=0.5]{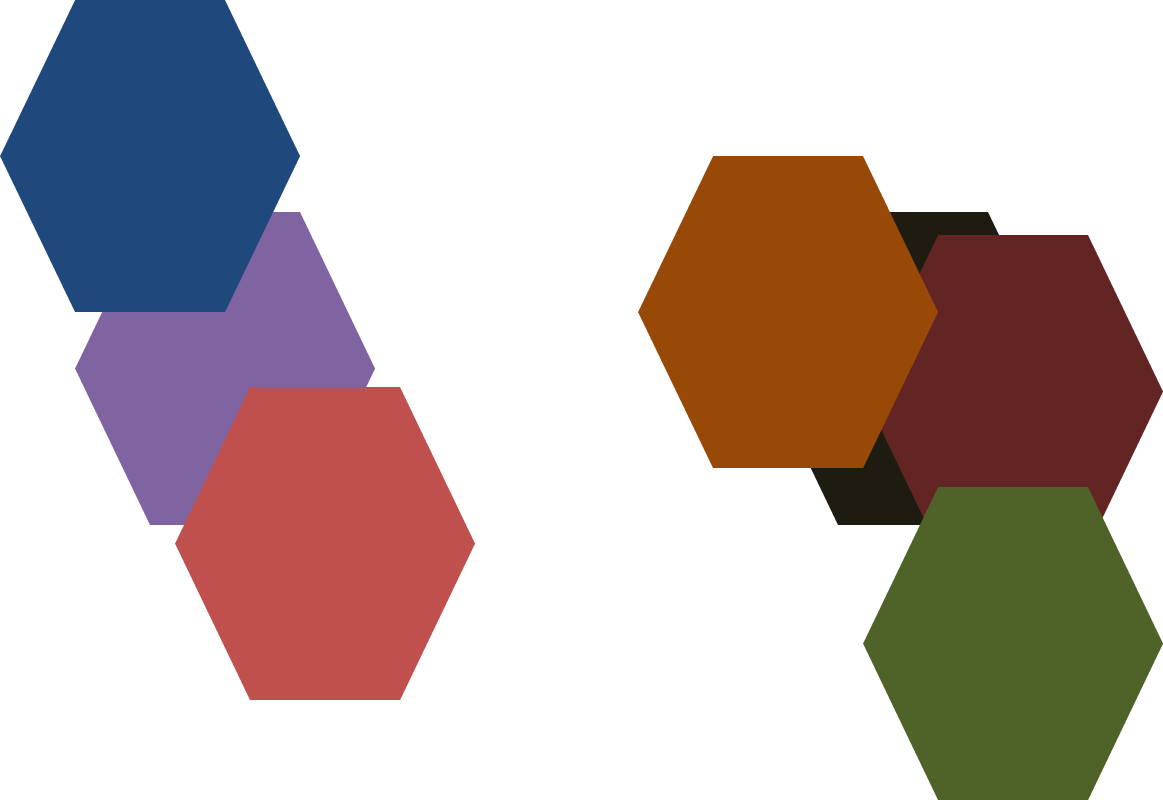}
\smallskip
\caption{\small \em A schematic depiction of (randomized) iterative ball partitioning of a bounded subset of $\R^2$, where $\R^2$ is equipped with a norm whose unit ball is a regular hexagon. The centers of the above hexagons are chosen independently and uniformly at random  from a large region that contains the given  subset of $\R^2$. At each step of the iteration, a new hexagon appears, and it carves out a new cluster which consists of the part of the hexagon that does not intersect any of the clusters that have been formed in the previous stages of the iteration. The first few clusters that are formed by this procedure are typically hexagons, but at later stages the clusters become more complicated and less ``round.'' In particular, they can eventually become disconnected, as exhibited by the region that is shaded black above.
}
\label{fig:hexagons}
\end{minipage}
}
\end{figure}

In the present setting, the metric space that we wish to partition is a normed space $\X=(\R^n,\|\cdot\|_\X)$, so it is natural to want to use the Lebesgue measure on $\R^n$ in the above construction. Since this measure is not a probability measure, we cannot use the above framework directly. For this reason, we will in fact use a periodic variant of iterative ball partitioning of $\X$ by adapting a  construction  that was used in~\cite{LN05}.

\subsubsection{Separation and padding}\label{sec:sep pad} Fix  $\Delta>0$. Let $\Part$ be a  $\Delta$-bounded random partition of a metric space~$\MM$. As a random ``clustering'' of $\MM$ into pieces of small diameter, $\Part$ yields a certain  ``simplification'' of  $\MM$. For such a simplification to be useful, one must add a requirement that it ``mimics'' the geometry  of $\MM$ in a meaningful way. The literature contains multiple definitions that achieve this goal, leading to  applications in both algorithms and pure mathematics. We will not attempt to survey  the literature on this topic, quoting only the  definitions of {\em separating} and {\em padded} random partitions, which are the simplest and most popular notions of random partitions of metric spaces among those that have been introduced.

\begin{definition}[separating random partition and separation modulus]\label{def:separating finite} Let $(\MM,d_\MM)$ be a  metric space. For $\us,\Delta>0$, a $\Delta$-bounded random  partition $\Part$ of $(\MM,d_\MM)$ is {\bf \em $\us$-separating} if
\begin{equation}\label{eq:separating condition}
\forall  x,y\in \MM,\qquad \Pr\big[\Part(x)\neq \Part(y)\big]\le \frac{\us}{\Delta}d_\MM(x,y).
\end{equation}
The {\bf \em separation modulus}\footnote{In~\cite{Nao17-SODA} we called the same quantity the ``modulus of separated decomposability.''} of $(\MM,d_\MM)$, denoted $\sep(\MM,d_\MM)$ or simply $\sep(\MM)$ if the metric is clear from the context, is the infimum over those $\sigma >0$ such that for every $\Delta>0$ there exists  a $\sigma$-separating  $\Delta$-bounded  random partition of $(\MM,d_\MM)$. If no such $\sigma$ exists, then write $\sep(\MM,d_\MM)=\infty$. Similarly, for $n\in \N$, the {\bf \em size-$n$ separation modulus} of $(\MM,d_\MM)$, denoted $\sep^n(\MM,d_\MM)$ or simply $\sep^n(\MM)$ if the metric is clear from the context, is the infimum over those $\us>0$ such that for every $S\subset \MM$ with $|S|\le n$ and every $\Delta>0$ there exists a $\us$-separating $\Delta$-bounded  random partition of $(S,d_\MM)$. In other words,
$$
\sep^n(\MM,d_\MM)\eqdef \sup_{\substack{S\subset \MM\\ |S|\le n}}\sep(S,d_\MM).
$$
\end{definition}

While the notions that we presented in Definition~\ref{def:separating finite} are standard (see below for the history), it will be beneficial for us (e.g.~for proving Theorem~\ref{thm:nonstandard ext}) to introduce the following terminology.

\begin{definition}[separation profile]\label{def:separation profile} Let $(\MM,d_\MM)$ be a  metric space. We say that a metric $\mathfrak{d}:\MM\times \MM\to [0,\infty)$ on $\MM$ is   a {\bf \em separation profile} of $(\MM,d_\MM)$ if for every $\Delta>0$ there exists a $\Delta$-bounded random partition $\Part_\Delta$ of $(\MM,d_\MM)$ that is defined on some probability space $(\Omega_\Delta,\Pr_\Delta)$ such that
\begin{equation}\label{eq:def separation profile}
\forall x,y\in \MM,\qquad \mathfrak{d}(x,y)\ge \sup_{\Delta\in (0,\infty)} \Delta\Pr_\Delta\big[\Part_\Delta(x)\neq \Part_\Delta(y)\big].
\end{equation}
\end{definition}

So, the separation modulus of $(\MM,d_\MM)$ is the infimum over those $\sigma>0$ for which $\sigma d_\MM$ is a separation profile of $(\MM,d_\MM)$. Definition~\ref{def:separation profile} would make sense for functions $\mathfrak{d}:\MM\times \MM\to [0,\infty)$  that need not be metrics on $\MM$, but we prefer to deal only with separation profiles of $(\MM,d_\MM)$  that are metrics on $\MM$ so as to be able to discuss the Lipschitz condition with respect to them; observe that the right hand side of~\eqref{eq:def separation profile} is a metric on $\MM$, so any such function is always at least (point-wise) a metric that is a separation profile of $(\MM,d_\MM)$. If $\mathfrak{d}:\MM\times \MM\to [0,\infty)$ is a separation profile of $(\MM,d_\MM)$, then $\mathfrak{d}(x,y)\ge d_\MM(x,y)$ for all $x,y\in \MM$ because $\diam_\MM(\Part_{d_\MM(x,y)-\e}(x))\le d_\MM(x,y)-\e<d_\MM(x,y)$ for any $0<\e<d_\MM(x,y)$, so we necessarily have  $y\notin \Part_{d_\MM(x,y)-\e}(x)$ (deterministically) and therefore \begin{equation}\label{eq:frak d lower}
\mathfrak{d}(x,y)\ge (d_\MM(x,y)-\e)\Pr_{d_\MM(x,y)-\e}\big[\Part_{d_\MM(x,y)-\e}(x)\neq \Part_{d_\MM(x,y)-\e}(y)\big]=d_\MM(x,y)-\e.
\end{equation}

\begin{definition}[padded random partition and padding modulus]\label{def:padded finite} Let $(\MM,d_\MM)$ be a  metric space. For $\d,{\mathfrak{p}},\Delta>0$, a $\Delta$-bounded random  partition $\Part$ of $(\MM,d_\MM)$ is {\bf \em $({\mathfrak{p}},\d)$-padded} if
\begin{equation}\label{eq:padded condition}
\forall  x\in \MM,\qquad \Pr\Big[B_\MM\Big(x,\frac{\Delta}{{\mathfrak{p}}}\Big)\subset \Part(x)\Big]\ge \d.
\end{equation}
Denote by $\pad_\d(\MM,d_\MM)$, or simply $\pad_\d(\MM)$ if the metric is clear from the context, the infimum over those $\mathfrak{p}>0$ such that for every $\Delta>0$ there exists  a $({\mathfrak{p}},\d)$-padded $\Delta$-bounded random partition $\Part$ of $(\MM,d_\MM)$. If no such $\mathfrak{p}$ exists, then write $\pad_\d(\MM,d_\MM)=\infty$. For every $n\in \N$, denote
$$
\pad^n_\d(\MM,d_\MM)\eqdef \sup_{\substack{S\subset \MM\\ |S|\le n}}\pad_\d(S,d_\MM).
$$
\end{definition}

See Section~\ref{sec:prelim random part main} for a quick justification why the above definition of random partition implies that the events that appear in~\eqref{eq:separating condition} and~\eqref{eq:padded condition} are indeed $\Pr$-measurable.

Qualitatively, condition~\eqref{eq:separating condition} says that despite the fact that $\Part$ decomposes $\MM$ into clusters of small diameter, any two nearby points are likely to belong to the same cluster. Condition~\eqref{eq:padded condition} says that  every point in $\MM$ is likely to be ``well within'' its cluster (its distance to the complement of its cluster is at least a definite proportion of the assumed upper bound on the diameter of that cluster). Both of these requirements express the (often nonintuitive) property that  the ``boundaries'' that the random partition induces are ``thin'' in a certain distributional sense, despite the fact that each  realization of the partition consists only of small diameter clusters that can sometimes be very jagged. Neither of the above two definitions implies the other, but it follows from~\cite{LN03} that if $\Part$ is a $({\mathfrak{p}},\d)$-padded  $\Delta$-bounded random partition of $(\MM,d_\MM)$, then there exits  a random partition $\Part'$ of    $(\MM,d_\MM)$ that is $(2\Delta)$-bounded and $(4{\mathfrak{p}}/\d)$-separating.

Separating and padded random partitions  were introduced in the articles~\cite{Bar96,Bar98}  of Bartal, which contained decisive algorithmic applications and influenced a flurry of subsequent works that obtained many more applications in several directions. Other works considered such partitions implicitly, with a variety of applications; see the works of Leighton--Rao~\cite{LR88}, Awerbuch--Peleg~\cite{AP90}, Linial--Saks~\cite{LS91}, Alon--Karp--Peleg--West~\cite{AKPW91}, Klein--Plotkin--Rao~\cite{KPR93} and Rao~\cite{Rao99}. The nomenclature of Definition~\ref{def:separating finite}  and Definition~\ref{def:padded finite} comes from~\cite{GKL03,LN03,LN04,LN05,KLMN05}.

By~\cite{Bar96}, for every metric space $(\MM,d_\MM)$ and every integer $n\ge 2$, we have $\sep^n(\MM)\lesssim \log n$. It was observed by Gupta,  Krauthgamer and Lee~\cite{GKL03} that~\cite{Bar96} also implicitly yields the padding bound $\pad_{0.5}^n(\MM)\lesssim \log n$. It was proved in~\cite{Bar96} that both of these estimates are sharp.

Random partitions of normed spaces were first studied  by Peleg and Reshef~\cite{PR98} for applications to network routing and distributed
computing. The aforementioned work~\cite{CCGGP98} improved and generalized the bounds of~\cite{PR98}, and  influenced  later works; see
e.g.~\cite{LN05}, and the work~\cite{AI06} of Andoni and Indyk. Similar partitioning schemes appeared implicitly in earlier work~\cite{KMS98}
on  algorithms for graph colorings based on semidefinite programming.

\subsubsection{From separation to Lipschitz extension} As we already explained, the connection between random partitions and  Lipschitz extension  was found in~\cite{LN05}. Here  we will use the following theorem to deduce Theorem~\ref{thm:nonstandard ext}. It implies in particular the bound
\begin{equation}\label{eq:quote LN}
\ee(\MM)\lesssim \sep(\MM)
\end{equation}
of~\cite{LN05} and its proof is an adaptation of the ideas of~\cite{LN05} to both the present setup (extension to a function that is Lipschitz with respect to a different metric) and our different measurability requirements from the random partitions; we stress, however, that even though we cannot apply~\cite{LN05} directly as a ``black box,'' the geometric ideas that underly the proof of Theorem~\ref{thm:local compact ext} are the same as those of~\cite{LN05}.

\begin{theorem}\label{thm:local compact ext} Suppose that $\mathfrak{d}$ is a separation profile of a locally compact metric space $(\MM,d_\MM)$. For every Banach space $(\bfZ,\|\cdot\|_\bfZ)$ and every subset $\sub\subset \MM$, if $f:\sub\to \bfZ$ is $1$-Lipschitz with respect to the metric $d_\MM$, i.e., $\|f(x)-f(y)\|_\bfZ\le d_\MM(x,y)$ for every $x,y\in \MM$, then there is $F:\MM\to \bfZ$ that extends $f$ and is $O(1)$-Lipschitz with respect to the metric $\mathfrak{d}$, i.e., $\|F(x)-F(y)\|_\bfZ\lesssim \mathfrak{d} (x,y)$ for every $x,y\in \MM$.
\end{theorem}

\subsubsection{Bounds on the separation and padding moduli of normed spaces}\label{sec:pad sep bounds} To facilitate the ensuing discussion of upper and lower bounds on the separation and padding moduli of (subsets of) normed spaces, we will first record two of their rudimentary properties. Firstly, the following lemma formally expresses the aforementioned advantage of the definitions in Section~\ref{sec:sep pad} over those of~\cite{LN05}, namely that the moduli $\sep(\cdot)$ and $\pad_\d(\cdot)$ are bi-Lipschitz invariants; its straightforward proof appears in Section~\ref{sec:prelim random part main}.

\begin{lemma}[bi-Lipschitz invariance of separation and padding moduli]\label{lem:invariance} Let $(\MM,d_\MM)$ be a complete metric space that admits a bi-Lipschitz embedding into a metric space $(\NN,d_\NN)$. Then
\begin{equation}\label{eq:sep invariance}
\sep(\MM,d_\MM)\le \cc_{(\NN,d_\NN)}(\MM,d_\MM)\sep(\NN,d_\NN),
\end{equation}
and
\begin{equation}\label{eq:pad invariance}
\forall  \d\in (0,1),\qquad \pad_\d(\MM,d_\MM)\le \cc_{(\NN,d_\NN)}(\MM,d_\MM)\pad_\d(\NN,d_\NN).
\end{equation}
\end{lemma}

Secondly,  we have the following  tensorization property  whose simple proof appears in Section~\ref{sec:prelim random part main}. For $s\in [1,\infty]$ and  metric spaces $(\MM_1,d_{\MM_1}),(\MM,d_{\MM_2})$, the metric $d_{\MM_1\oplus_s\MM_2}:\MM_1\times \MM_2\to [0,\infty)$ on the Cartesian product $\MM_1\times \MM_2$ is defined by setting for every $(x_1,x_2),(y_1,y_2)\in \MM_1\times \MM_2$,
\begin{equation}\label{eq:product metric}
d_{\MM_1\oplus_s\MM_2}\big((x_1,x_2),(y_1,y_2)\big)\eqdef \big(d_\MM(x_1,y_1)^s+d_\NN(x_2,y_2)^s\big)^{\frac{1}{s}}.
\end{equation}
With the usual convention that when $s=\infty$ the right hand side of~\eqref{eq:product metric} is equal to the maximum of $d_\MM(x_1,y_1)$ and $d_\NN(x_2,y_2)$. The metric space $(\MM_1\times \MM_2,d_{\MM_1\oplus_s\MM_2})$ is will be denoted $\MM_1\oplus_s\MM_2$.
\begin{lemma}[tensorization of separation and padding moduli]\label{lem:tensorization} For any $s\in [1,\infty]$ and $\d_1,\d_2\in (0,1)$, any two metric spaces $(\MM_1,d_{\MM_1})$ and $(\MM_2,d_{\MM_2})$ satisfy
\begin{equation}\label{eq:sep subsadditive}
\sep(\MM_1\oplus_s \MM_2)\le  \sep(\MM_1)+\sep(\MM_2),
\end{equation}
and
\begin{equation}\label{eq:pad subadditive}
\pad_{\d_1\d_2}(\MM_1\oplus_s \MM_2)\le  \big(\pad_{\d_1}(\MM_1)^s+\pad_{\d_2}(\MM_2)^s\big)^{\frac{1}{s}}.
\end{equation}
\end{lemma}

The following theorem shows that the bi-Lipschitz invariant $\pad_\d(\cdot)$ is not sufficiently sensitive to distinguish substantially between normed spaces, as its value is essentially independent of the norm.
\begin{theorem}\label{thm:pad sharp} For every $n\in \N$, every normed space $\X=(\R^n,\|\cdot\|_{\X})$ satisfies
\begin{equation}\label{eq:padding nth root}
\forall  \d\in (0,1),\qquad \frac{1}{1-\sqrt[n]{\d}}\le  \frac12 \pad_\d(\X)\le \frac{1+\sqrt[n]{\d}}{1-\sqrt[n]{\d}}.
\end{equation}
Therefore, $\pad_\d(\X)\asymp \max\left\{1,\frac{\dim(\X)}{\log\left(1/\d\right)}\right\}$ for every finite dimensional normed space $\X$ and $\d\in (0,1)$.
\end{theorem}
As we explained above, in the setting of Theorem~\ref{thm:pad sharp} the fact that $\pad_{0.5}(\X)=O(n)$ is well-known. We will prove the upper bound on $\pad_\d(\X)$ that appears in~\eqref{eq:padding nth root}, i.e., with sharp dependence on both $n$ and $\d$, in Section~\ref{sec:the partition}. The fact that $\pad_{0.5}(\X)$  is at least a universal constant multiple of $n$ was proved in the manuscript~\cite{LN03}. Because~\cite{LN03} is not intended for publication, we will prove the lower bound on  $\pad_\d(\X)$  that appears in~\eqref{eq:padding nth root} in Section~\ref{sec:pad lower}, by following the reasoning of~\cite{LN03} while taking more care than we did in~\cite{LN03} in order to obtain sharp dependence on $\d$ in addition to sharp dependence on $n$.

 In contrast to Theorem~\ref{thm:pad sharp}, the separation modulus of a finite dimensional normed space  can have different asymptotic dependencies  on its dimension. Indeed,    $\sep(\ell_2^n)\asymp\sqrt{n}$  and $\sep(\ell_1^n)\asymp n$ by~\cite{CCGGP98}, so using Lemma~\ref{lem:invariance} we see that every normed space $\X=(\R^n,\|\cdot\|_\X)$ satisfies the a priori bounds
\begin{equation}\label{eq:a priori with BM}
\frac{n}{d_{\BM}(\ell_1^n,\X)}\lesssim \sep(\X)\lesssim d_{\BM}(\ell_2^n,\X)\sqrt{n},
\end{equation}
 which we already quoted in the above overview as~\eqref{eq:use CCGGP bounds}.

 Giannopoulos proved~\cite{Gia95} that every $n$-dimensional normed space $\X$ satisfies $d_{\BM}(\ell_1^n,\X)\lesssim n^{5/6}$,  so the first inequality in~\eqref{eq:a priori with BM} implies that $\sep(\X)\gtrsim \sqrt[6]{n}$. Alternatively, the fact that $\sep(\X)\ge n^{c}$ for some universal constant $c>0$  follows from by combining Theorem~\ref{thm:power lower}  with~\eqref{eq:quote LN}. Actually, we always have
 \begin{equation}\label{eq:sep root n}
 \sep(\X)\gtrsim \sqrt{n},
  \end{equation}
  which coincides with the first half of~\eqref{eq:max of sqrt{n} and type-intro}. Observe that~\eqref{eq:sep root n} cannot follow from a ``vanilla'' application of the first inequality in~\eqref{eq:a priori with BM} by Szarek's work~\cite{Sza90}. In fact, the first inequality of~\eqref{eq:a priori with BM} must sometimes yield a worse power type dependence on $n$ than in~\eqref{eq:sep root n}, because Tikhomirov proved in~\cite{Tik19} that there is a normed space $\X=(\R^n,\|\cdot\|_\X)$ that satisfies $d_{\BM}(\ell_1^n,\X)\ge n^a$ for some universal constant  $a>1/2$.

  Nevertheless, we can prove~\eqref{eq:sep root n} by the following a ``hereditary'' application of~\eqref{eq:a priori with BM}. Bourgain and Szarek~\cite{BS88} and independently Ball (see~\cite[Remark~7]{BS88}, \cite[Remark~7]{Sza91}, \cite[page~138]{Tom89}) proved (relying on the Bourgain--Tzafriri restricted invertibility principle~\cite{BT87}) that there is $m\in \n$ with $m\asymp n$ such that $\cc_{\X}(\ell_1^m)\lesssim \sqrt{n}$ (in fact, by~\cite{BS88} any $2n$-dimensional normed space has Banach--Mazur distance $O(\sqrt{n})$ from $\ell_1^n\oplus \ell_2^n$). Therefore, by~\eqref{eq:sep invariance} we have $\sep(\X)\gtrsim \sep(\ell_1^m)/\cc_\X(\ell_1^m)\asymp m/\cc_\X(\ell_1^m)\gtrsim \sqrt{n}$.

  The second half of~\eqref{eq:max of sqrt{n} and type-intro} is the following lower bound on $\sep(\X)$ in terms of the type $2$ constant of $\X$.
  \begin{equation}\label{eq:sep square type}
  \sep(\X)\gtrsim T_2(\X)^2.
  \end{equation}
We will prove~\eqref{eq:sep square type} in Section~\ref{sec:consequences of distance}  using Talagrand's refinement~\cite{Tal-elton92} of Elton's theorem~\cite{Elt83}, by the same hereditary use of~\eqref{eq:a priori with BM}, namely showing that there is $m\in \n$ for which $m/\cc_\X(\ell_1^m)\gtrsim T_2(\X)^2$.

\begin{remark} It is impossible to improve~\eqref{eq:max of sqrt{n} and type-intro} for all the values of the relevant parameters, as seen  by considering $\X=\ell_{2}^{n-m}\oplus_2\ell_1^m$ for each $m\in\n$. Indeed, since in this case $T_2(\X)\asymp \sqrt{m}$,
$$
\sep(\X)\stackrel{\eqref{eq:sep subsadditive}}{\le} \sep\big(\ell_2^{n-m}\big)+\sep\big(\ell_1^m\big)\asymp \sqrt{n-m}+m\asymp \sqrt{n}+k\asymp \max\Big\{\sqrt{\dim(\X)},T_2(\X)^2\Big\}.
$$
\end{remark}

Thanks to~\eqref{eq:state bourgain milman}, the following theorem is a restatement of the lower bound on $\sep(\X)$ in Theorem~\ref{thm:sep bounds in overview}.

\begin{theorem}\label{thm:sep lower evr} For every $n\in \N$, any normed space $\X=(\R^n,\|\cdot\|_\X)$ satisfies $\sep(\X)\gtrsim \evr(\X)\sqrt{n}.
$
\end{theorem}
As $\evr(\X)\ge 1$ (by definition), Theorem~\ref{thm:sep lower evr}  implies~\eqref{eq:sep root n}, via a  proof that differs from the above reasoning. Also, Theorem~\ref{thm:sep lower evr} is stronger than the first inequality in~\eqref{eq:a priori with BM} because $\evr(\ell_1^n)\asymp \sqrt{n}$, and hence
$$
\evr(\X)\sqrt{n}\ge \frac{\evr(\ell_1^n)}{d_{\BM}(\ell_1^n,\X)}\sqrt{n}\asymp \frac{n}{d_{\BM}(\ell_1^n,\X)}.
$$
We will prove Theorem~\ref{thm:sep lower evr} in Section~\ref{sec:proof of vr lower} by adapting to the setting of general normed spaces the strategy that was used in~\cite{CCGGP98} to treat $\ell_1^n$. The volumetric lower bound on $\sep(\X)$ of Theorem~\ref{thm:sep lower evr} is typically quite easy to use and it often leads to  estimates that are better than the first inequality in~\eqref{eq:a priori with BM}.

For example, by~\cite[Proposition~2.2]{Sch82} the  Schatten--von Neumann trace class $\sfS_p^n$ satisfies
\begin{equation}\label{eq:evr schatten intro}
\forall p\ge 1,\qquad \evr\big(\sfS_p^n\big)\asymp n^{\max\left\{\frac{1}{p}-\frac12,0\right\}}.
\end{equation}
By substituting~\eqref{eq:evr schatten intro} into Theorem~\ref{thm:sep lower evr} we  get that
\begin{equation}\label{eq:sep spn lower}
\forall 1\le p\le 2,\qquad \sep(\sfS_p^n)\gtrsim  n^{\frac{1}{p}-\frac12}\sqrt{\dim\big(\sfS_p^n\big)}\asymp  n^{\frac1{p}+\frac12}.
\end{equation}
An upper bound that  matches~\eqref{eq:sep spn lower}  is a consequence of the second inequality in~\eqref{eq:a priori with BM} as follows
$$
\sep\big(\sfS_p^n\big)\lesssim d_{\BM}\big(\sfS_p^n,\ell_2^{n^2}\big)\sqrt{\dim\big(\sfS_p^n\big)} = d_{\BM}\big(\sfS_p^n,\sfS_2^n\big)n=n^{\frac{1}{p}+\frac12}.
$$
We therefore have
$$
\forall 1\le p\le 2,\qquad  \sep\big(\sfS_p^n\big)\asymp n^{\frac{1}{p}+\frac12}.
$$

At the same time, the first inequality in~\eqref{eq:a priori with BM} does not imply~\eqref{eq:sep spn lower} since  by a theorem of Davis (which was published only in the monograph~\cite{Tom89}; see Theorem~41.10 there), for every $1\le p\le 2$ we have
\begin{equation}\label{eq:davis}
d_{\BM}\big(\ell_1^{n^2},\sfS_p^n\big)\asymp n.
\end{equation}
So, the first inequality in~\eqref{eq:a priori with BM} only implies the weaker lower bound $\sep(\sfS_p^n)\gtrsim n$. Of course, this  rules out  a ``vanilla'' use of~\eqref{eq:a priori with BM} and  a hereditary application of~\eqref{eq:a priori with BM} as we did above could conceivably lead to~\eqref{eq:sep spn lower}, i.e., there could be $m\in \n$ such that $m/\cc_{\sfS_p^n}(\ell_1^m)$ is at least the right hand side of~\eqref{eq:sep spn lower}. However, this possibility seems to be unlikely, as it would mean that the following conjecture has a negative answer, which would entail finding a remarkable (and likely valuable elsewhere) subspace of $\sfS_p^n$.

\begin{conjecture}\label{con:proportional subspace of schatten} Fix $1\le p\le 2$ and $0<\d\le 1$. If $n,m\in \N$ satisfy $m\ge \d n^2$, then
$$
d_{\BM}(\ell_1^m,\X)\gtrsim_{p,\d} n
$$
for every $m$-dimensional subspace $\X$ of $\sfS_p^n$.
\end{conjecture}
Thus, \eqref{eq:davis} is the case $\d=1$ of Conjecture~\ref{con:proportional subspace of schatten}, which asserts that the same asymptotic lower bound persists if we consider subspaces of $\sfS_p^n$ of proportional dimension rather than $\sfS_p^n$  itself. Conjecture~\ref{con:proportional subspace of schatten} is attractive in its own right, but it also implies that~\eqref{eq:sep spn lower} does not follow from a hereditary application of  the first inequality in~\eqref{eq:a priori with BM}. To see this,  suppose for contradiction  that there were $m\in \n$ such that
\begin{equation}\label{eq:contrapositive m}
\frac{m}{\cc_{\sfS_p^n}(\ell_1^m)}\gtrsim_p n^{\frac{1}{p}+\frac12}.
\end{equation}
By Rademacher's differentiation theorem~\cite{Rad19} there is an $m$-dimensional subspace $\X$ of $\sfS_p^n$ satisfying
\begin{equation}\label{eq:rademacher intro}
\cc_{\sfS_p^n}(\ell_1^m)=d_{\BM}(\ell_1^m,\X)\gtrsim \frac{d_{\BM}(\ell_1^m,\ell_2^m)}{d_{\BM}(\sfS_p^n,\sfS_2^n)}=\frac{\sqrt{m}}{n^{\frac{1}{p}-\frac{1}{2}}}.
\end{equation}
By contrasting~\eqref{eq:rademacher intro} with~\eqref{eq:contrapositive m} we deduce that necessarily $m\gtrsim_p n^2$, so an application of Conjecture~\ref{con:proportional subspace of schatten} gives $m/\cc_{\sfS_p^n}(\ell_1^m)\lesssim_p n$, which contradicts~\eqref{eq:contrapositive m}  since $p<2$.

\begin{remark} The L\"owner ellipsoid of $\ell_\infty^n(\ell_1^n)$ is $\sqrt{n}B_{\ell_2^n(\ell_2^n)}$, and $B_{\ell_\infty^n(\ell_1^n)}=(B_{\ell_1^n})^n$. Consequently,
$$
\evr\big(\ell_\infty^n(\ell_1^n)\big)n=
 n\Bigg(\frac{(\pi n)^{\frac{n^2}{2}}/\Gamma\big(\frac{n^2}{2}+1\big)}{2^{n^2}/(n!)^n}\Bigg)^{\frac{1}{n^2}}\asymp n^{\frac{3}{2}}.
$$
Therefore,  Theorem~\ref{thm:sep lower evr} gives
\begin{equation}\label{eq: infty one sep}
\sep\big(\ell_\infty^n(\ell_1^n)\big)\gtrsim n^{\frac32}.
\end{equation}
We will soon see that~\eqref{eq: infty one sep} is optimal, though unlike the above discussion for $\sfS_p^n$ when $1\le p\le 2$, this does not follow from the second inequality in~\eqref{eq:a priori with BM} because by~\cite{KS89},
\begin{equation}\label{eq:quite KS}
d_{\BM}\big(\ell_2^{n^2},\ell_\infty^n(\ell_1^n)\big)\asymp d_{\BM}\big(\ell_1^{n^2},\ell_\infty^n(\ell_1^n)\big)\asymp n.
\end{equation}
\eqref{eq:quite KS} also shows that~\eqref{eq: infty one sep} does not follow from  the first inequality in~\eqref{eq:a priori with BM}. It seems that the method used in~\cite{KS89} to prove~\eqref{eq:quite KS} is insufficient for proving that~\eqref{eq: infty one sep}    does not follow from a hereditary application of the first inequality in~\eqref{eq:a priori with BM}. Analogously to Conjecture~\ref{con:proportional subspace of schatten}, we conjecture that this is impossible, which is a classical-sounding question about Banach--Mazur distances of independent interest.
\end{remark}

Before passing to a description of our upper bounds on the separation modulus, we formulate the following corollary of Theorem~\ref{thm:sep lower evr} on the separation modulus of norms whose unit ball is a polytope; it restates the lower bound~\eqref{eq:sep lower vertices-intro} and establishes its optimality.

\begin{theorem}\label{cor:sep lower vertices} Fix $n\in \N$ and a normed space $\X=(\R^n,\|\cdot\|_{\X})$.
Suppose that $B_{\X}$ is a polytope that has exactly $\rho n$ vertices (note that necessarily $\rho\ge 2$, since $B_{\X}$ is origin-symmetric). Then
\begin{equation}\label{eq:sep lower vertices}
\sep(\X)\gtrsim \frac{n}{\sqrt{\log \rho}}.
\end{equation}
Moreover, this bound cannot be improved in general.
\end{theorem}

As an example of a consequence of Theorem~\ref{cor:sep lower vertices}, let $\mathbf{G}=(\R^n,\|\cdot\|_{\mathbf{G}})$ be a  Gluskin space~\cite{Glu81}, i.e.~it is a certain {\em random} norm on $\R^n$ whose unit ball has $O(n)$ vertices; see the survey~\cite{MTJ03} for extensive information about this important construction and its variants. The expected Banach--Mazur distance between two independent copies of $\mathbf{G}$ is at least $cn$ for some universal constant $c>0$, so the expected Banach--Mazur distance between $\mathbf{G}$ and $\ell_1^n$ is at least $\sqrt{cn}$. Thus, the first inequality in~\eqref{eq:a priori with BM} only shows that $\sep(\mathbf{G})\gtrsim \sqrt{n}$ in expectation, while Theorem~\ref{cor:sep lower vertices} shows that in fact $\sep(\mathbf{G})\gtrsim n/\sqrt{\log n}$. It would be interesting to determine the growth rate of $\E[\sep(\mathbf{G})]$. In particular, can it be that $\E[\sep(\mathbf{G})]\gtrsim n$?

\begin{proof}[Proof of Theorem~\ref{cor:sep lower vertices}] By applying a linear isometry of  $\X$ we may assume that $B_{\ell_2^n}$ is the L\"owner ellipsoid of $B_{\X}$. Since $B_\X$ is a polytope with $\rho n$ vertices that is contained in $B_{\ell_2^n}$, we have $$\sqrt[n]{\vol_n(B_{\X})}\lesssim \frac{\sqrt{\log \rho}}{n}$$
by a  result  of Maurey~\cite{Pis81} (see also~\cite{Car85,BF87,CP88,Glu88,BLM89,BP90,Kyr00} and  the expository treatments  in~\cite{Bal01,BGVV14}). Hence, $\evr(\X)\gtrsim \sqrt{n/\log \rho}$, so~\eqref{eq:sep lower vertices} follows from Theorem~\ref{thm:sep lower evr}.

Consider the following (dual of an) example of Figiel and Johnson~\cite{FJ80}. Fix $m\in \N$. Let $\bfZ=(\R^m,\|\cdot\|_\bfZ)$ be a normed space with $d_{\BM}(\ell_2^m,\bfZ)\lesssim 1$ such that $B_\bfZ$ is a polytope of $e^{O(m)}$ vertices; e.g.~$B_{\bfZ}$ can be taken to be the convex hull of a net of $S^{m-1}$. For $k\in \N$, let $\X=\ell_1^k(\bfZ)$. So, $\dim(\X)=km$ and $B_{\X}$ is a polytope of $2ke^{O(m)}$ vertices.  Thus~\eqref{eq:sep lower vertices} becomes $\sep(\X)\gtrsim k\sqrt{m}$. At the same time, since $d_{\BM}(\ell_2^m,\bfZ)\lesssim 1$ we have $d_{\BM}(\ell_2^{km},\X)\lesssim \sqrt{k}$, so by~\eqref{eq:a priori with BM} in fact $\sep(\X)\lesssim \sqrt{k} \cdot\sqrt{km} =k\sqrt{m}$, i.e., \eqref{eq:sep lower vertices}  is sharp in this case.
\end{proof}

Theorem~\ref{thm:nonstandard ext}  follows from Theorem~\ref{thm:local compact ext} thanks to the following randomized partitioning theorem.

\begin{theorem}\label{thm:mixed extension}
For every $n\in \N$ and every normed space $\X=(\R^n,\|\cdot\|_\X)$, the metric $\mathfrak{d}$ that is defined by
$$
\forall x,y\in \R^n,\qquad \mathfrak{d}(x,y)=\frac{4\|x-y\|_{\Pi^{\textbf *}\X}}{\vol_n(B_\X)}.
$$
is a separation profile for $\X$.
\end{theorem}

To illustrate  Theorem~\ref{thm:mixed extension}, fix $1\le p\le\infty$ and apply it when $\X$ is  the space $\Y_p^n$ of Theorem~\ref{prop:rounded cube}. By using Theorem~\ref{thm:mixed extension}  we see that for every $\Delta>0$ there is a random partition $\Part$ of $\R^n$ with the following properties.
\begin{enumerate}%[label=(\roman*)]
\item For every $x\in \R^n$ we have $\diam_{\ell_p^n}\big(\mathscr{P}(x)\big)\le \Delta$.
\item For every $x,y\in \R^n$ we have
\begin{equation}\label{eq:bi crteria}
\Pr \big[\mathscr{P}(x)\neq \mathscr{P}(y)\big]\lesssim \frac{\|x-y\|_{\Pi^{\textbf *}\Y_p^n}}{\vol_n(B_{\Y_p^n})}\stackrel{\eqref{eq:use cauchy} \wedge\eqref{eq:round cube}}{\lesssim}  \frac{n^{\frac{1}{p}}}{\Delta} \|x-y\|_{\ell_2^n}.
\end{equation}
\end{enumerate}
In comparison to the $O(\sqrt{n})$-separating partition of $\ell_2^n$ from~\cite{CCGGP98}, when $p<2$ the above random partition  has  smaller clusters in the sense that their diameter in the $\ell_p^n$ metric is at most $\Delta$, which is more stringent than the requirement that their Euclidean diameter is at most $\Delta$. This improved control on the size of the clusters comes at the cost that in the probabilistic separation requirement~\eqref{eq:bi crteria} the quantity that multiplies the Euclidean distance increases from $O(\sqrt{n})$ to $O(n^{1/p})$. When $p>2$ this tradeoff is reversed, i.e., we get an asymptotic improvement in the separation guarantee~\eqref{eq:bi crteria}  at the cost of requiring less from the cluster size, namely the diameter of each cluster is now guaranteed to be small in the $\ell_p^n$ metric rather than the more stringent requirement that it is small in the Euclidean metric.

Theorem~\ref{thm:sep volume ratio-intro} below  follows from   Theorem~\ref{thm:mixed extension} the same way we deduced Theorem~\ref{thm:XY version ext} from Theorem~\ref{thm:nonstandard ext}.

\begin{theorem}\label{thm:sep volume ratio-intro}Fix $n\in \N$ and two normed spaces $\X=(\R^n,\|\cdot\|_{\X}), \Y=(\R^n,\|\cdot\|_{\Y})$.  Every closed  $\sub\subset \R^n$ satisfies
\begin{equation}\label{eq:ext XY version'}
\sep(\sub_{\X})\le 4\bigg( \sup_{\substack{x,y\in \sub\\x\neq y}}\frac{\|x-y\|_{\X}}{\|x-y\|_{\Y}}\bigg)\sup_{\substack{x,y\in \sub\\x\neq y}} \bigg(\frac{\vol_{n-1}\big(\proj_{(x-y)^\perp}(B_{\Y})\big)}{\vol_n(B_{\Y})}\cdot \frac{\|x-y\|_{\ell_2^n}}{\|x-y\|_{\X}}\bigg).
\end{equation}
\end{theorem}

\begin{proof}[Proof of Theorem~\ref{thm:sep volume ratio-intro}  assuming Theorem~\ref{thm:mixed extension}] Let $M,M'$ be as in~\eqref{eq:MM' notation}. By Theorem~\ref{thm:mixed extension} applied to $\Y$, for every $\Delta>0$ there is a random partition $\Part$ of $\R^n$ that is $(\Delta/M)$-bounded with respect to $\Y$, i.e.,
$$
\frac{\diam_\X\big(\Part(x)\big)}{M}\stackrel{\eqref{eq:MM' notation}}{\le} \diam_\Y\big(\Part(x)\big)\le \frac{\Delta}{M}
$$
for every $x\in \R^n$, and also, recalling Definition~\ref{def:separation profile}, for every distinct  $x,y\in \R^n$ we have
\begin{equation*}
\frac{\Delta}{M}\Pr\big[\Part(x)\neq\Part(y)\big]\le  \frac{4\|x-y\|_{\Pi^{\textbf *}\Y}}{\vol_n(B_\Y)}\stackrel{\eqref{eq:use cauchy}}{=}\frac{4\vol_{n-1}\big(\proj_{(x-y)^\perp}(B_{\Y})\big)\|x-y\|_{\ell_2^n}}{\vol_n(B_{\Y})}\stackrel{\eqref{eq:MM' notation}}{\le} 4M'\|x-y\|_\X.\tag*{\qedhere}
\end{equation*}
\end{proof}

The special case $\sub=\R^n$ of Theorem~\ref{thm:sep volume ratio-intro}  coincides (with an explicitly stated constant factor) with the upper bound on $\sep(\X)$ in Theorem~\ref{thm:sep bounds in overview}, since under the normalization  $B_\Y\subset B_\X$ we have
$$
\sep(\X)\stackrel{\eqref{eq:use cauchy}\wedge \eqref{eq:ext XY version'}}{\le }4\frac{\sup_{z\in \partial B_{\X}}\|z\|_{\Pi^{\textbf *}\Y}}{\vol_n(B_\Y)}=4\frac{\|\Id_n\|_{\X\to\Pi^{\textbf *}\Y}}{\vol_n(B_\Y)}=4\frac{\|\Id_n\|_{\Pi\Y\to \X^{\textbf *}}}{\vol_n(B_\Y)}=2\frac{\diam_{\X^{\textbf *}}(\Pi B_\Y)}{\vol_n(B_\Y)}.
$$
Also, Theorem~\ref{thm:sep volume ratio-intro} is stronger than the second inequality in~\eqref{eq:a priori with BM} because by applying a linear isometry of  $\X$ we may assume without loss of generality that $\|x\|_{\X}\le \|x\|_{\ell_2^n}\le d_{\BM}(\ell_2^n,\X)\|x\|_{\X}$ for all $x\in \R^n$, in which case the special case $\sub=\R^n$ and $\Y=\ell_2^n$ of~\eqref{eq:ext XY version'} implies that
$$
\sep(\X)\le \frac{4\vol_{n-1}\big(B_{\ell_2^{n-1}}\big)}{\vol_n\big(B_{\ell_2^n}\big)}d_{\BM}(\ell_2^n,\X)= \frac{4\pi^{\frac{n-1}{2}}\Gamma\left(\frac{n}{2}+1\right)}{\pi^{\frac{n}{2}}\Gamma\left(\frac{n-1}{2}+1\right)} d_{\BM}(\ell_2^n,\X)=\frac{2^\frac32+o(1)}{\sqrt{\pi}} d_{\BM}(\ell_2^n,\X)\sqrt{n}.
$$

The right hand side of~\eqref{eq:ext XY version'} coincides (up to a universal constant factor) with the right hand side of~\eqref{eq:ext XY version}, so all of the upper bounds  for the Lipschitz extension modulus that we derived in the previous sections from Theorem~\ref{thm:XY version ext} hold for the separation modulus, by  Theorem~\ref{thm:sep volume ratio-intro}. For the separation modulus, we get several lower bounds from Theorem~\ref{thm:sep lower evr} that either provably match our upper bounds up to lower order factors, or  match them assuming our conjectural isomorphic reverse isoperimetry. We will next spell out some of those  consequences on randomized clustering of high dimensional norms.

\begin{theorem}\label{thm:sparse and low rank} For every $p\ge 1$, $n\in \N$ and $k,r\in \n$ we have
\begin{equation}\label{eq:sep k-sparse}
\sep\big((\ell_p^n)_{\le k}\big)\asymp k^{\max\left\{\frac{1}{p},\frac12\right\}},
\end{equation}
and
\begin{equation}\label{eq:sep schatten rank k}
r^{\max\left\{\frac{1}{p},\frac12\right\}}\sqrt{n}\lesssim \sep\big((\sfS_p^{n})_{\le r}\big)\lesssim r^{\max\left\{\frac{1}{p},\frac12\right\}}\sqrt{n} \cdot \left\{\begin{array}{ll}\sqrt{\max\left\{\log\left(\frac{ n}{r}\right),p\right\}}&\mathrm{if}\ p\le \log r,\\
\sqrt{\log n}&\mathrm{if}\ p\ge \log r.\end{array}\right.
\end{equation}
Moreover, if Conjecture~\ref{conj:weak reverse iso when canonical}  holds for  $\X=\sfS_p^n$, then in fact
\begin{equation}\label{eq:sharp rank less r intro}
\sep\big((\sfS_p^{n})_{\le r}\big)\asymp r^{\max\left\{\frac{1}{p},\frac12\right\}}\sqrt{n}.
\end{equation}
\end{theorem}

\begin{proof} The deduction of the upper bounds on the separation modulus in~\eqref{eq:sep k-sparse} and~\eqref{eq:sep schatten rank k} from Theorem~\ref{thm:sep volume ratio-intro} are identical, respectively, to the ways we deduced Theorem~\ref{thm:sparse}  and~\eqref{eq:schatten extension} from   Theorem~\ref{thm:XY version ext}.

For the first inequality in~\eqref{eq:sep k-sparse},  since $(\ell_p^n)_{\le k}$ contains an isometric copy of $\ell_p^k$, we have
$$
\sep\big((\ell_p^n)_{\le k}\big)\ge \sep\big(\ell_p^k\big)\gtrsim \frac{k}{d_{\BM}\big(\ell_p^k,\ell_1^k\big)}\stackrel{\eqref{eq:a priori with BM}}{\asymp} \frac{k}{k^{\max\left\{1-\frac{1}{p},\frac12\right\}}}=k^{\min\left\{\frac{1}{p},\frac12\right\}},
$$
where the asymptotic evaluation of $d_{\BM}(\ell_p^k,\ell_q^k)$ for all $p,q\ge 1$ is due  Gurari{\u\i}, Kadec{\cprime} and Macaev~\cite{GKM66}.

For the first inequality in~\eqref{eq:sep schatten rank k}, use the fact that $(\sfS_p^{n})_{\le r}$ contains an isometric copy of $\sfS_p^{r\times n}$, which is the Schatten--von Neumann trace class on the $r$-by-$n$ real matrices $\M_{r\times n}(\R)$, whose norm is given by
\begin{equation}\label{eq:def schatten rectangular}
\forall A\in M_{r\times n}(\R),\qquad \|A\|_{\sfS_p^{r\times n}}= \Big(\trace \big((AA^*)^{\frac{p}{2}}\big)\Big)^{\frac{1}{p}}.
\end{equation}
We then have the following rectangular version of~\eqref{eq:evr schatten intro} whose derivation is explained in Remark~\ref{rem:rectangular schatten}.
\begin{equation}\label{eq:evr rectangular k}
\evr\big(\sfS^{r\times n}_p\big)\asymp r^{\max\left\{\frac{1}{p}-\frac12,0\right\}}.
\end{equation}
The desired lower bound on $\sep((\sfS_p^{n})_{\le r})$ is now an application of Theorem~\ref{thm:sep lower evr}.
\end{proof}

\begin{remark}\label{rem:explain indyk}
Theorem~3.3  in~\cite{CCGGP98} asserts that $\sep(\ell_p^n)\asymp n^{\max\{1/p,1-1/p\}}$ for every $p\ge 1$. Therefore, when $p> 2$ it was previously thought that $\sep(\ell_p^n)\asymp n^{1-1/p}$, which   contradicts the case $k=n$ of~\eqref{eq:sep k-sparse}. While~\cite{CCGGP98}  provides a complete and correct proof that $\sep(\ell_p^n)\asymp n^{1/p}$ when $1\le p\le 2$, in the range $p>2$ the assertion $\sep(\ell_p^n)\asymp n^{1-1/p}$ in~\cite{CCGGP98}  is justified through the use of a result from reference~[14] in~\cite{CCGGP98}, which is cited there as a  ``personal communication'' with P.~Indyk (dated April 1998). This reference  was never published. After discovering Theorem~\ref{thm:sparse and low rank},  we confirmed with Indyk that his aforementioned personal communication  with the authors of~\cite{CCGGP98}  contained a gap.
\end{remark}

\begin{corollary}\label{cor:sep sharp if conj} Conjecture~\ref{conj:weak reverse iso when canonical}  implies  Conjecture~\ref{conj:symmetric volume ratio sep}. Namely, if   Conjecture~\ref{conj:weak reverse iso when canonical} holds for a canonically positioned normed space $\X=(\R^n,\|\cdot\|_\X)$, then \begin{equation}\label{eq:sep volume ratio of dual}
\sep(\X)\asymp \evr(\X)\sqrt{n}\asymp \vr(\X^*)\sqrt{n}.
\end{equation}
In particular, if $\X$ satisfies the assumptions of Lemma~\ref{lem:weak iso for enouhg permutations and unconditional} (e.g.~if $\X$ is symmetric), then~\eqref{eq:sep volume ratio of dual} holds.  Furthermore, if $\bfE=(\R^n,\|\cdot\|_\bfE)$ is a symmetric normed space, then $\sep(\sfS_\bfE)=\evr(\bfE)n^{1+o(1)}$. More precisely,
\begin{equation}\label{eq:sep of ideal intro}
\evr(\bfE)n\lesssim \sep(\sfS_\bfE)\lesssim \evr(\bfE)n\sqrt{\log n}.
\end{equation}
\end{corollary}

\begin{proof} The lower bound on $\sep(\X)$ in~\eqref{eq:sep volume ratio of dual} is Theorem~\ref{thm:sep lower evr}  (thus, it requires neither Conjecture~\ref{conj:weak reverse iso when canonical} nor $\X$ being canonically positioned). The matching upper bound on $\sep(\X)$ in~\eqref{eq:sep volume ratio of dual} follows from Corollary~\ref{cor:if weak reverse holds} and the fact that by Theorem~\ref{thm:sep volume ratio-intro} the separation modulus of any (not necessarily canonically positioned) normed space $\X=(\R^n,\|\cdot\|_\X)$ is bounded from above by the right hand side of~\eqref{eq:substitute MaxProj}. The rest of the assertions of Corollary~\ref{cor:sep sharp if conj} follow from Lemma~\ref{lem:weak iso for enouhg permutations and unconditional} and Lemma~\ref{lem:reverse iso for sym}.
\end{proof}

By incorporating Proposition~\ref{prop:K convexity} into the same reasoning as in the  justification of Corollary~\ref{cor:sep sharp if conj}, we also deduce the following stronger version of Theorem~\ref{thm:symmetric volume ratio sep up to lower order}.

\begin{theorem}\label{thm:sharp up to log n canonically} If $\X=(\R^n,\|\cdot\|_\X)$ is a canonically positioned normed space, then $$\evr(\X)\sqrt{n} \lesssim \sep(\X)\lesssim K(\X)\evr(\X)\sqrt{n}\stackrel{\eqref{eq:pisier bound K convex}}{\lesssim} \evr(\X)\sqrt{n}\log n.$$
\end{theorem}

Section~\ref{sec:volume ratios} contains volume ratio computations that show how Corollary~\ref{cor:sep sharp if conj} and Theorem~\ref{thm:sharp up to log n canonically} imply Corollary~\ref{coro:examples of apps}, as well as the conjectural (i.e., conditional on the validity of Conjecture~\ref{conj:weak reverse iso when canonical}  for the respective spaces) asymptotic evaluations~\eqref{eq:sep of operator from ellp to ellq} and~\eqref{eq:projective in overview}, and several further results of this type. Most of the volume ratio computations in Section~\ref{sec:volume ratios} rely on the available literature (notably Sch\"utt's work~\cite{Sch82}), with a few new twists that are perhaps of independent geometric/probabilisitic interest (e.g.~Lemma~\ref{lem:bernoulli chevet}).

\subsubsection{Dimension reduction} Fix $n\in \N$ and a metric space $(\MM,d_\MM)$. Recall that in Definition~\ref{def:separating finite} we denoted by $\sep^n(\MM,d_\MM)$ the supremum over all the separation moduli of subsets of $\MM$ of size at most $n$. In~\cite{CCGGP98} it was shown that $\sep^n(\ell_2)\lesssim \sqrt{\log n}$. Indeed, this follows from  the Johnson--Lindenstrauss dimension reduction lemma~\cite{JL84}, which asserts that any $n$-point subset of $\ell_2$ can be embedded with $O(1)$ distortion into $\ell_2^m$ with $m\lesssim \log n$, combined with the proof in~\cite{CCGGP98} that $\sep(\ell_2^m)\lesssim \sqrt{m}$.

One might expect that the optimal bounds that we know for $\sep(\ell_p^n)$ in the entire range $p\in (1,\infty)$ also translate to improved bounds on $\sep^n(\ell_p)$. The term ``improved'' is used here to mean any upper  bound of the form $o_p(\log n)$ as $n\to \infty$, since the benchmark general result is the aforementioned  upper bound  $\sep^n(\MM,d_\MM)\lesssim \log n$ from~\cite{Bar96}, which holds for any $n$-point metric space $(\MM,d_\MM)$. This bound is  sharp in general~\cite{Bar96}, so (because every $n$-point metric space embeds isometrically into $\ell_\infty^n$) we cannot hope to get a better bound on $\sep^n(\ell_\infty)$ despite the fact that  we obtained here an improved upper bound on $\sep(\ell_\infty^n)$.

The obstacle is that when $p\in [1,\infty]\setminus \{2\}$ no bi-Lipschitz dimension reduction result is known for finite subsets of $\ell_p$, and poly-logarithmic bi-Lipschitz  dimension reduction  is impossible if $p\in \{1,\infty\}$; the case $p=\infty$ is due to Matou\v{s}ek~\cite{Mat96} (see also~\cite{Nao17,Nao21}) and the case $p=1$ is due to Brinkman and Charikar~\cite{BC05} (see also~\cite{LN04-diamond,Reg13,NPS20,NY21}). When $p\in [1,\infty]\setminus \{1,2,\infty\}$ remarkably nothing is known, i.e., neither positive results nor impossibility results are available for bi-Lipschitz  dimension reduction, and it is a major open problem to make any progress in this setting; see~\cite{Nao18} for more on this area. Despite this obstacle, we have the following theorem that treats  the range $p\in [1,2]$.
\begin{theorem}\label{eq:finite ellp}
For every $p\in (1,2]$ and $n\in \N$ we have $$(\log n)^{\frac{1}{p}}\lesssim\sep^n(\ell_p)\lesssim \frac{(\log n)^{\frac{1}{p}}}{p-1}.$$
\end{theorem}
The lower bound on $\sep^n(\ell_p)$ of Theorem~\ref{eq:finite ellp} can be deduced from~\cite{CCGGP98}; see Section~\ref{sec:finite Lp} for the details.  An upper bound of $\sep^n(\ell_p)\lesssim_p (\log n)^{1/p}$ was obtained when $p\in (1,2]$  in the manuscript~\cite{LN03}. As~\cite{LN03} is not intended for publication, a proof of the  upper bound on $\sep^n(\ell_p)$ that is stated in Theorem~\ref{eq:finite ellp} is included in Section~\ref{sec:finite Lp}, where we perform the argument with more care than the way we initially did it in~\cite{LN03}, so as to obtain the best dependence on $p$ that is achievable by this approach. Nevertheless, we conjecture that  the dependence on $p$ in Theorem~\ref{eq:finite ellp} could be removed altogether, though this would likely require a substantially new idea.
\begin{conjecture}\label{conj:endpoint}
The dependence on $p$ in  Theorem~\ref{eq:finite ellp} can be improved to $\sep^n(\ell_p)\lesssim (\log n)^{\frac{1}{p}}$.
\end{conjecture}
So, if $p\le 1+ c(\log\log \log n)/\log\log n$ for some universal constant $c>0$, then Theorem~\ref{eq:finite ellp} does not improve asymptotically over $\sep^n(\ell_p)\lesssim \log n$, while Conjecture~\ref{conj:endpoint} would imply that $\sep^n(\ell_p)=o(\log n)$ if and only if $\lim_{n\to\infty}(p-1)\log\log n=\infty$.

For fixed $p\in (2,\infty)$, at present we do not see how to obtain an upper bound on $\sep^n(\ell_p)$ of the form $o_p(\log n)$ as $n\to \infty$. We state this separately as  an interesting and  challenging open question.
\begin{question}\label{Q:dim reductoion p>2}
Is it true that for every $n\in \N$ and $p\in (2,\infty)$ we have $\lim_{n\to \infty}\sep^n(\ell_p)/\log n=0$? More ambitiously, is it true that $\sep^n(\ell_p)\lesssim_p\sqrt{\log n}$?
\end{question}
Note that $\sep^n(\X)\gtrsim \sqrt{\log n}$ for any infinite-dimensional normed space $\X$, because by Dvoretzky's theorem~\cite{Dvo60} we have $\cc_\X(\ell_2^m)=1$  for every $m\in \N$, and therefore $\sep^n(\X)\ge \sep^n(\ell_2)\asymp\sqrt{\log n}$.

\subsection{Consequences in the linear theory}\label{sec:stability intro} Even though the purpose of the present article was to investigate the nonlinear invariants $\ee(\cdot)$ and $\sep(\cdot)$, by relating them to volumetric quantities and other linear invariants  of Banach spaces (such as type and cotype), we arrive at consequences that have nothing to do with nonlinear issues.  In this section, we will give a flavor of such consequences, though we will not be exhaustive since it would be more natural to pursue them separately for their own right in future work.

Denote the Minkowski functional of an origin-symmetric convex body $K\subset \R^n$ by $\|\cdot\|_K$, i.e., it is the norm on $\R^n$ whose unit ball is equal to $K$. The following theorem coincides with the second inequality in~\eqref{eq:main thm for overview} upon a straightforward application of duality as we did in~\eqref{eq:adjoint}; this formulation  is intended to highlight how we are bounding a convex-geometric quantity by a bi-Lipschitz invariant.

\begin{theorem}[nonsandwiching between a convex body and its polar projection body]\label{thm:non-sandwiching} Fix $n\in \N$ and $\alpha,\beta\in (0,\infty)$. Let $K,L\subset \R^n$ be symmetric convex bodies with $\vol_n(L)=1$. Suppose that
\begin{equation}\label{eq:sandwiching formulation}
\alpha L\subset K\subset \beta \Pi^*\!L.
\end{equation}
Then,
\begin{equation}\label{eq:beta over alpha conclusion}
\frac{\beta}{\alpha}\gtrsim \sep\big(\R^n,\|\cdot\|_K\big).
\end{equation}
\end{theorem}
%The homogeneous version of the requirement~\eqref{eq:sandwiching formulation} of Theorem~\ref{thm:non-sandwiching}, namely without the %normalization $\vol_n(L)=1$, is $\alpha L\subset K\subset \beta\vol_n(L) \Pi^{*}\!L$, which implies that same conclusion~\eqref{eq:beta over %alpha conclusion}.

Since the separation modulus of a metric space is at least the separation modulus of any of its subsets, by combining~\eqref{eq:beta over alpha conclusion} with the first inequality in~\eqref{eq:main thm for overview} we see that the sandwiching hypothesis~\eqref{eq:sandwiching formulation} implies the following purely volumetric consequence for every linear subspace $\mathbf{V}\subset \R^n$.
\begin{equation}\label{eq:beta over alpha volume ratio}
\frac{\beta}{\alpha}\gtrsim  \evr\big(K\cap\mathbf{V}\big)\sqrt{n}\asymp \vr\big(\proj_\mathbf{V} K^\circ\big)\sqrt{n}.
\end{equation}
In particular, using $\evr(\ell_1^n)\asymp \sqrt{n}$, we record separately  the following special case of~\eqref{eq:beta over alpha volume ratio}.
\begin{corollary}[nonsandwiching of the cross-polytope] Fix $n\in \N$ and $\alpha,\beta\in (0,\infty)$. If $L\subset \R^n$ is a convex body of volume $1$ that satisfies $\alpha L\subset B_{\ell_1^n}\subset \beta \Pi^{*}\!L$, then necessarily $\beta/\alpha\gtrsim n$.
\end{corollary}

The geometric meaning of Theorem~\ref{thm:non-sandwiching} when $L=K$ is spelled out  in the following corollary.

\begin{corollary}[every origin-symmetric convex body admits a large cone]\label{cor:big cone sep} For every $n\in \N$, every origin-symmetric convex body $K\subset \R^n$ has a boundary point $z\in \partial K$ that satisfies
\begin{equation}\label{eq:large cone sep}
\frac{\vol_n\big(\cone_z(K)\big)}{\vol_n(K)}\gtrsim \frac{1}{n}\sep\big(\R^n,\|\cdot\|_K\big).
\end{equation}
\end{corollary}
To see that Corollary~\ref{cor:big cone sep}  coincides with the case $L=K$ of Theorem~\ref{thm:non-sandwiching},  simply recall the definition of the polar projection body $\Pi^* K$ in~\eqref{eq:use cauchy}, while also recalling that for $z\in \R^n\setminus \{0\}$ we denote  the cone whose base is $\proj_{z^{\perp}}(K)\subset z^\perp$ and whose apex is $z$ by $\cone_z(K)$, and the volume of $\cone_z(K)$ is given in~\eqref{eq:volume of cone}.

\begin{comment}
By tracking our proofs, it is mechanical to write a derivation of~\eqref{eq:beta over alpha volume ratio} that fits strictly within the classical Brunn--Minkowski theory (without any reference to the separation modulus), but it is noteworthy that we arrived at it through the current nonlinear context. Since $\evr(K)\ge 1$,  we get $\beta/\alpha\gtrsim \sqrt{n}$ from~\eqref{eq:beta over alpha volume ratio}.
\end{comment}

A substitution of~\eqref{eq:sep root n} into Corollary~\ref{cor:big cone sep}  shows that  any origin-symmetric convex body $K\subset \R^n$ has a boundary point $z\in \partial K$ that satisfies
\begin{equation}\label{eq:1 over sqrt n bound}
\frac{\vol_n\big(\cone_z(K)\big)}{\vol_n(K)}\gtrsim \frac{1}{\sqrt{n}}.
\end{equation}
It seems (based on inquiring with  experts in convex geometry) that the classical-looking geometric statement~\eqref{eq:1 over sqrt n bound} did not previously appear in the literature. However, in response to our inquiry Lutwak found a different proof of~\eqref{eq:1 over sqrt n bound} which in addition shows that the best possible constant in~\eqref{eq:1 over sqrt n bound} is $1/\sqrt{2\pi}$. More precisely, we have the following proposition, whose proof (which relies on classical Brunn--Minkowski theory, unlike the indirect way by which we found~\eqref{eq:1 over sqrt n bound}), is included  in Section~\ref{sec:erwin} (this proof is a restructuring of the proof that Lutwak found; we thank him for allowing us to include it here).
\begin{proposition}[Lutwak]\label{prop:large cone} For every $n\in \N$, any origin symmetric convex body $K\subset \R^n$ satisfies
\begin{equation}\label{eq:sharp cone}
\max_{z\in \partial K} \frac{\vol_n\big(\cone_z(K)\big)}{\vol_n(K)}\ge  \frac{\Gamma\left(\frac{n}{2}\right)}{2\sqrt{\pi}\Gamma\left(\frac{n+1}{2}\right)}\ge \frac{1+\frac{1}{4n}}{\sqrt{2\pi n}}.
\end{equation}
Moreover, the first inequality in~\eqref{eq:sharp cone} holds as equality if and only if $K$ is an ellipsoid.
\end{proposition}

A substitution of~\eqref{eq:sep square type} into Corollary~\ref{cor:big cone sep}  yields  the following geometric inequality.
\begin{corollary}\label{prop:elton} Fix $n\in \N$ and suppose that $K\subset \R^n$ is an origin-symmetric convex body. There is a boundary point $z\in \partial K$ such that the following inequality holds for every $x_1,\ldots,x_n\in K$.
\begin{equation}\label{eq:kahane version}
\frac{\vol_n\big(\cone_z(K)\big)}{\vol_n(K)} \gtrsim \frac{1}{n} \fint_{S^{n-1}} \Big\|\sum_{i=1}^n \theta_i x_i\Big\|_K^2\ud \theta.
\end{equation}
\end{corollary}
By combining~\cite{Tom79} with Lemma~\ref{lem:equal norm} below, the maximum of the right hand side of~\eqref{eq:kahane version} over all possible $x_1,\ldots,x_n\in K$ is bounded above and below by universal constant multiples of $T_2(\R^n,\|\cdot\|_K)^2/n$ (recall the definition~\eqref{eq:def type cotype} of the type-$2$ constant), so Corollary~\ref{prop:elton} is indeed a substitution of~\eqref{eq:sep square type} into~\eqref{eq:large cone sep}.

Returning to Corollary~\ref{cor:big cone sep}, recall that both the cross-polytope $B_{\ell_1^n}$ and the hypercube $[-1,1]^n$ are examples of extremal symmetric convex bodies $K\subset \R^n$ that have a boundary point $z\in \partial K$ for which the volume of  $\cone_z(K)$ is a  universal constant proportion of the volume of $K$ (the Euclidean ball is an example of a convex body that is not extremal in this regard). But,  there is a difference between  the cross-polytope and the hypercube in terms of the stability of this property. Specifically,  there is an origin-symmetric convex body $K\subset [-1,1]^n\subset O(1)K$ such that for every $z\in \partial K$ the left hand side of~\eqref{eq:large cone sep} is at most a universal constant multiple of $1/\sqrt{n}$.  In contrast, the following proposition shows that the extremality of $\max_{z\in \partial B_{\ell_1^n}} \vol_n(\cone_z(B_{\ell_1^n}))/\vol_n(B_{\ell_1^n})$ (up to constant factors)  persists under $O(1)$ perturbations.

\begin{proposition}\label{prop:stability cross polytope} Fix $n\in \N$ and $\alpha,\beta\in (0,\infty)$. Suppose that $K\subset \R^n$ is an origin-symmetric convex body that satisfies $\alpha K\subset B_{\ell_1^n}\subset \beta K$. Then there exists a boundary point $z\in \partial K$ such that
\begin{equation}\label{eq:alpha over beta}
\frac{\vol_n\big(\cone_z(K)\big)}{\vol_n(K)}\gtrsim \frac{\alpha}{\beta}.
\end{equation}
\end{proposition}
 Proposition~\ref{prop:stability cross polytope} is a direct consequence of Corollary~\ref{cor:big cone sep}, the bi-Lipschitz invariance of the modulus of separated decomposability, and the lower bound $\sep(\ell_1^n)\gtrsim n$  of~\cite{CCGGP98}.

\smallskip
The following proposition is an application in a different direction of the results that we described in the preceding sections.
\begin{proposition}\label{prop:monotonicity of evr} If $(\bfE,\|\cdot\|_\bfE)$ is a finite dimensional normed space with a $1$-symmetric basis, then every  subspace $\X$ of $\bfE$ satisfies
\begin{equation}\label{eq:evr monotonicity}
\evr(\X)\sqrt{\dim(\X)}\lesssim \evr(\bfE)\sqrt{\dim(\bfE)}.
\end{equation}
\end{proposition}
Proposition~\ref{prop:monotonicity of evr} holds because $\sep(\bfE)\lesssim \evr(\bfE)\sqrt{\dim(\bfE)}$ by Corollary~\ref{cor:sep sharp if conj}, while $\sep(\X)\gtrsim \evr(\X)\sqrt{\dim(\X)}$ by Theorem~\ref{thm:sep lower evr}, so~\eqref{eq:evr monotonicity} follows from $\sep(\X)\le \sep(\bfE)$. This justification shows that Proposition~\ref{prop:monotonicity of evr} holds for a class of spaces that is larger than those that have a $1$-symmetric basis, and Conjecture~\ref{conj:symmetric volume ratio sep} would imply that Proposition~\ref{prop:monotonicity of evr} holds when $\bfE$ is any canonically positioned normed space.

Nevertheless, Proposition~\ref{prop:monotonicity of evr} fails to  hold without any further assumption on the normed space $\bfE$. For example, the computation in  Remark~\ref{rem:stabilization} shows that for any $n,m\in \N$ with $n\ge 2$ and $m\asymp n\log n$, the space $\bfE=\ell_1^n\oplus \ell_2^m$ satisfies $\evr(\bfE)\sqrt{\dim(\bfE)}\lesssim \sqrt{n\log n}$ while its subspace $\X=\ell_1^n$ satisfies $\evr(\X)\sqrt{\dim(\X)}\asymp n$.

Proposition~\ref{prop:monotonicity of evr}  shows that if $\bfE$ has a $1$-symmetric basis, then among the linear subspaces $\X$ of $\bfE$ the invariant $\evr(\X)\sqrt{\dim(\X)}$ is maximized up to universal constant factors  at $\X=\bfE$. The fact we are multiplying here the external volume ratio of $\X$ by the square root of its dimension is an artifact of our proof and it would be interesting to understand  what  correction factors allow for such a result to hold:

\begin{question}\label{Q:An} Characterize (up to universal constant factors) those  $A:[1,\infty)\to [1,\infty)$ with the property that for any $n\ge 1$ we have $\evr(\X)A(k)\le \evr(\bfE)A(n)$ for every  normed space $(\bfE,\|\cdot\|_\bfE)$ of dimension at most $n$ that has a $1$-symmetric basis, every $k\in \n$, and every $k$-dimensional subspace $\X$ of $\bfE$.
\end{question}

Proposition~\ref{prop:monotonicity of evr}  shows that if  $A(n)\asymp\sqrt{n}$, then $A:[1,\infty)\to [1,\infty)$ has the properties that are described in Question~\ref{Q:An}. At the same time,  no $A:[1,\infty)\to [1,\infty)$  with $A(n)=O(1)$ can be as in  Question~\ref{Q:An}. Indeed, for any such  $A$ consider the symmetric normed space $\bfE=\ell_\infty^n$. There is a universal constant $\eta>0$ such that any normed space $\X$ with $\dim(\X)\le \eta\log n$ is at Banach--Mazur distance at most $2$ from a subspace of $\ell_\infty^n$.\footnote{This assertion is  standard, here is a quick sketch. Take a $\d$-net $\NN$ of the unit sphere of $\X^*$   for a sufficiently small universal constant $\d>0$ and consider the embedding $x\mapsto (x^*(x))_{x^*\in \NN}$ from $\X$ to $\ell_\infty(\NN)$. Since $\log |\NN|\asymp \dim(\X)$, this gives a distortion $2$-embedding (say, for $d=1/10$) of $\X$ into $\ell_\infty^n$ provided $\log n$ is at least a sufficiently large universal constant multiple of $\dim(\X)$.} In particular, this holds for $\X=\ell_1^m$ when $m\in \N$ satisfies $m\le \eta \log n$, so we get that
\begin{equation}\label{eq:for l1 in l infty}
A(\eta\log n)\sqrt{\log n}\asymp \evr\big(\ell_1^m\big)A(\eta\log n)\le 2\evr\big(\ell_\infty^n\big)A(n)\asymp A(n).
\end{equation}
 So, $A(n)\gtrsim \sqrt{\log n}$ and by iterating~\eqref{eq:for l1 in l infty} one gets the slightly better lower bound $A(n)\gtrsim\sqrt{(\log n)\log\log n}$, as well as $A(n)\gtrsim\sqrt{(\log n)(\log\log n)\log\log \log n}$ and so forth, yielding in the end the estimate
\begin{equation}\label{eq:log star product}
 A(n)\ge \frac{\left(\prod_{k=1}^{\log^*\!\!n}\log^{[k]}\!\!n\right)^{\frac12}}{e^{O(\log^*\!\! n)}},
\end{equation}
where for  $k\in \N\cup\{0\}$ we denote the $k$'th iterant of the logarithm  by $\log^{[k]}$, i.e.,   $\log^{[0]}\!x=x$ for $x>0$, and
\begin{equation}\label{eq:iterated logarithm def}
\log^{[k]}\!\!x>0\implies  \log^{[k+1]}\!\!x=\log \big(\log^{[k]}\!\!x\big).
\end{equation}
There is no reason to expect that the lower bound~\eqref{eq:log star product} is close to being optimal, but in combination with Proposition~\ref{prop:monotonicity of evr}  it does show that the answer to Question~\ref{Q:An} is likely nontrivial.

These considerations lead to the following open-ended question. The literature contains results showing that $\ell_p^n$ maximizes certain geometric invariants (e.g.~Banach--Mazur distance to $\ell_2^n$~\cite{Lew78}, or volume ratio~\cite{Bal91-reverse}) among all  the $n$-dimensional subspaces or quotients of $L_p$. Is there an  analogous theory in the spirit of~\eqref{eq:evr monotonicity} in the much more general setting of spaces that have a $1$-symmetric basis? This could be viewed as a symmetric space variant of the classical work of Lewis~\cite{Lew78,Lew79}. An interesting step in this direction can be found in~\cite{Tom80}; specifically, see~\cite[Theorem~1.2]{Tom80}, which could be relevant to Question~\ref{Q:An} through the approach of~\cite[Section~2]{Bal91-reverse}.

\medskip{\noindent \bf Acknowledgements.} I am grateful to Moses Charikar and Piotr Indyk for a helpful conversation on the erroneous optimality assertion of~\cite{CCGGP98} in the range $p\in (2,\infty]$, and to Piotr Indyk for  subsequent conversations and pointers to the literature.  I thank Gideon Schechtman for multiple discussions that led in particular to Lemma~\ref{lem:bernoulli chevet} and~\eqref{eq:with gid s version}.   I thank Emmanuel Breuillard for discussions regarding Remark~\ref{rem:octagon} and for sharing with me his  proofs of statements on the comparison between the notions of being canonically positioned and having enough symmetries. I thank Erwin Lutwak for showing me his proof of Proposition~\ref{prop:large cone}. I thank Bo'az Klartag and Emanuel Milman for showing me the proof of Proposition~\ref{prop:K convexity}.
 I am grateful to Keith Ball, Franck Barthe, K\'aroly B\"or\"oczky, Ronen Eldan, Charles Fefferman, Apostolos Giannopoulos, David Jerison, Grigoris Paouris, Gilles Pisier, Oded Regev, Carsten Sch\"utt, Ramon van Handel and Elisabeth Werner for helpful discussions and suggestions. Finally, I am grateful to the anonymous referee for helpful corrections and recommendations that improved the presentation.

\section{Lower bounds}\label{sec:lower}

In this section we will prove the  impossibility results that were stated in the Introduction. Throughout what follows, all Banach spaces will be tacitly assumed to be separable. Given a  Banach space $\X$, its Banach--Mazur distance to a Hilbert space will be denoted $\dd_{\X}\in [1,\infty]$, i.e., $\dd_{\X}=d_{\BM}(\X,\mathbf{H})$ where $\mathbf{H}$ is a Hilbert space with either $\dim(\mathbf{H})=\dim(\X)$ when $\dim(\X)<\infty$, or $\mathbf{H}=\ell_2$ when $\X$ is infinite dimensional. By a classical result of Enflo~\cite[Theorem~6.3.3]{Enf70} (see also~\cite[Corollary~7.10]{BL00}) we have $\dd_{\X}=\cc_2(\X)$.

\subsection{Proof of Theorem~\ref{thm:complemented subspace}}
Recall that the (Gaussian) type $2$ and cotype $2$ constants of a Banach space $(\X,\|\cdot\|_{\X})$, denoted $T_2(\X)$ and $C_2(\X)$, respectively, are the infimum over those $T\in [1,\infty]$ and $C\in [1,\infty]$, respectively, for which the following inequalities hold for every $m\in \N$ and every $x_1,\ldots,x_m\in \X$:
\begin{equation}\label{eq:def type cotype-older}
\frac{1}{C^2}\sum_{j=1}^m\|x_j\|^2_{\X}\le \E\bigg[\Big\|\sum_{j=1}^m \g_j x_j\Big\|_{\X}^2\bigg] \le T^2\sum_{j=1}^m\|x_j\|^2_{\X},
\end{equation}
where  henceforth $\g_1,\g_2,\ldots$ will always denote  i.i.d.~standard Gaussian random variables. The following theorem of Kwapie\'n~\cite{Kwa72} is fundamental (see also~\cite[Theorem~3.3]{Pis-factorization} or~\cite[Theorem~13.15]{Tom89}).

\begin{theorem}\label{thm:Kwapien}
Every Banach space $(\X,\|\cdot\|_{\X})$ satisfies $\dd_{\X}\le T_2(\X)C_2(\X)$.
\end{theorem}

We will use Theorem~\ref{thm:Kwapien} to estimate the following quantity, which in turn will be used to get the best bound that we currently have on the constant $c$ that appears in the lower bound on $\ee(\X)$ of Theorem~\ref{thm:complemented subspace}.

\begin{definition}[Lindenstrauss--Tzafriri constant]\label{def:LT} Suppose that $(\X,\|\cdot\|_{\X})$ is a Banach space. Define $\LT(\X)$ to be the infimum over those $K\in [1,\infty]$ such that for every closed linear subspace $\mathbf{V}\subset \X$ there exists a projection $\proj:\X\twoheadrightarrow \mathbf{V}$ from $\X$ onto $\mathbf{V}$ whose operator norm satisfies $\|\proj\|_{\X\to\X}\le K$.
\end{definition}

So, the Lindenstrauss--Tzafriri constant of a Hilbert space equals $1$, and Sobczyk proved~\cite{Sob41} that
\begin{equation}\label{eq:sobczyk}
\forall n\in \N,\qquad \LT\big(\ell_1^n\big)\asymp \LT\big(\ell_\infty^n\big)\asymp \sqrt{n}.
\end{equation}
We chose the nomenclature of Definition~\ref{def:LT} in reference to the  famous solution~\cite{LT71} by Lindenstrauss and Tzafriri of the {\em complemented subspace problem}, which asserts that if   $(\X,\|\cdot\|_{\X})$ is a Banach space for which $\LT(\X)<\infty$, then $\X$ is isomorphic to a Hilbert space, i.e., $\dd_{\X}<\infty$. Moreover, if $\X$ is infinite dimensional, then it was shown in~\cite{LT71} that $\dd_{\X}\lesssim \LT(\X)^4$. This dependence was improved in~\cite{KM73} by  Kadec and Mitjagin, who established the following theorem, which is the currently best-known bound in the Lindenstrauss--Tzafriri theorem  (see also~\cite{Fig77,Pis88,Pis96,AK06,Kal08}  for subsequent improvements of the implicit universal constant factor and further generalizations).

\begin{theorem}\label{thm:LT infinite dim} Every infinite dimensional Banach space $(\X,\|\cdot\|_{\X})$ satisfies $\dd_{\X}\lesssim \LT(\X)^2$.
\end{theorem}
When $\dim(\X)<\infty$ the question of bounding $\dd_{\X}$ by a function of $\LT(\X)$ was left open in~\cite{LT71}. This question, which was eventually solved by Figiel, Lindenstrauss and Milman~\cite[Theorem~6.7]{FLM77}, turned out to be significantly more subtle than its infinite dimensional counterpart. The currently best-known estimate is due to Tomczak-Jaegermann~\cite[Theorem~29.4]{Tom89}, who proved the following theorem.
\begin{theorem}\label{thm:TJ}
Every finite dimensional Banach space $(\X,\|\cdot\|_{\X})$ satisfies $ \dd_{\X}\lesssim \LT(\X)^5$.
\end{theorem}
The proof of Theorem~\ref{thm:TJ} is achieved in~\cite{Tom89} through an interesting combination of the {\em proof of} the Lindenstrauss--Tzafriri theorem~\cite{LT71} with the finite dimensional machinery of~\cite{FLM77} and Milman's Quotient of Subspace Theorem~\cite{Mil85}.

The following theorem is a link between the Lindenstrauss--Tzafriri constant and Lipschitz extension.

\begin{theorem}\label{thm:lin projections}
Every Banach space $(\X,\|\cdot\|_{\X})$ satisfies $\ee(\X)\ge \LT(\X)$.
\end{theorem}

\begin{proof} By Remark~\ref{rem:history lower ext hilbert}, if $\dim(\X)=\infty$, then $\ee(\X)=\infty$, so we may assume that $\dim(\X)<\infty$. Fix $L>\ee(\X)$ and let $\bfV\subset \X$ be a linear subspace of $\X$. Then, the identity mapping from $\bfV$ to $\bfV$ can be extended to an $L$-Lipschitz mapping $\rho:\X\to \bfV$. In other words, $\rho$ is an $L$-Lipschitz retraction from $\X$ onto $\bfV$. By  a classical theorem of Lindenstrauss~\cite{Lin64} (see also its  elegant alternative proof by Pe\l czy\'nsky in~\cite[page~61]{Pel68}), there is a projection of norm at most  $L$ from $\X$ onto $\bfV$. This proves  that $\LT(\X)\le L$.
\end{proof}

The following theorem is the lower bound $\ee(\ell_2^n)\gtrsim \sqrt[4]{n}$ of~\cite{MN13-bary} that we already quoted in~\eqref{eq:ell2 what's known}, in combination with the bi-Lipschitz invariance of the Lipschitz extension modulus.

\begin{theorem}\label{thm:MN} For every $n\in \N$, any normed space $\X=(\R^n,\|\cdot\|_{\X})$ satisfies
 $\ee(\X)\gtrsim\frac{\sqrt[4]{n}}{\dd_{\X}}$.
\end{theorem}

\begin{remark}\label{rem:history lower ext hilbert} The question whether $\ee(\ell_2)$ is finite or infinite was open for quite some time: It was first stated in print  in~\cite[page~137]{JLS86}, and it was also posed by Ball in~\cite[page~170]{Bal92} (Ball conjectured  that $\ee(\ell_2)=\infty$). We answered it  in~\cite{Nao01} by proving  that $\lim_{n\to \infty} \ee(\ell_2^n)=\infty$. Due to Dvoretzky's theorem~\cite{Dvo60} this  implies that $\ee(\X)$ is at least an unbounded function of $\dim(\X)$ for {\em any} normed space $\X$, and in particular $\ee(\X)=\infty$ if $\dim(\X)=\infty$. A rate at which $\ee(\ell_2^n)$ tends to $\infty$ was not specified in~\cite{Nao01}, but the reasoning of~\cite{Nao01} was inspected quantitatively in~\cite[Remark~5.3]{LN05}, yielding an explicit lower bound  that depends on an auxiliary parameter, and it was noted in~\cite{BB07-2} that an optimization over this  parameter yields the estimate $\ee(\ell_2^n)\gtrsim \sqrt[8]{n}$. A further improvement from~\cite{MN13-bary} (whose proof refines ideas of Kalton~\cite{Kal04,Kal12}) was  the aforementioned estimate $\ee(\ell_2^n)\gtrsim \sqrt[4]{n}$ (a different proof of this bound  follows from~\cite{Nao20-almost}), which is the currently best-known lower bound on $\ee(\ell_2^n)$. By Milman's sharpening~\cite{Mil71}  of Dvoretzky's theorem~\cite{Dvo60}, it follows  that every normed space $\X$ satisfies $\ee(\X)\gtrsim \sqrt[4]{\log n}$. As we explained in Section~\ref{sec:ext def}, the bound $\ee(\ell_\infty^n)\gtrsim \sqrt{n}$ is classical (specifically, by substituting~\eqref{eq:sobczyk} into Theorem~\ref{thm:lin projections}). In combination with the Alon--Milman theorem~\cite{AM83} (see also~\cite{Tal95}), the fact that both $\ee(\ell_2^n)=n^{\Omega(1)}$ and $ \ee(\ell_\infty^n)=n^{\Omega(1)}$ formally implies that  $$\ee(\X)\ge e^{\eta \sqrt{\log n}}$$
for some universal constant $\eta>0$ and every $n$-dimensional normed space $\X$, which was the best-known general lower bound on the Lipschitz extension modulus  prior to Theorem~\ref{thm:power lower}.
\end{remark}

The above results imply as follows the lower bound on $\ee(\X)$ of Theorem~\ref{thm:complemented subspace}. By combining Theorem~\ref{thm:TJ} and Theorem~\ref{thm:lin projections}, we have $\ee(\X)\gtrsim \sqrt[5]{\dd_{\X}}$. In combination with Theorem~\ref{thm:MN}, it therefore follows that
\begin{equation}\label{eq:1/24}
\ee(\X)\gtrsim \max \left\{\frac{\sqrt[4]{n}}{\dd_{\X}},\sqrt[5]{\dd_{\X}}\right\}\ge \sqrt[24]{n},
\end{equation}
where the last step follows from elementary calculus and holds as equality when $\dd_\X=n^{5/24}$.

We will derive a better lower bound on $\ee(\X)$ than~\eqref{eq:1/24} through the following theorem which improves over the power of $\LT(\X)$ in Theorem~\ref{thm:TJ}, showing that in the finite dimensional setting one can come close (up to logarithmic factors) to the infinite dimensional bound of Theorem~\ref{thm:LT infinite dim}; see also Remark~\ref{rem:Elton} below.

\begin{theorem}\label{thm:LT finite with log} For every integer $n\ge 2$,  any $n$-dimensional Banach space $(\X,\|\cdot\|_{\X})$ satisfies
\begin{equation}\label{eq:cubed}
\dd_{\X}\lesssim \LT(\X)^2(\log n)^3.
\end{equation}
\end{theorem}
Assuming Theorem~\ref{eq:1/24}, reason analogously to~\eqref{eq:1/24} while using~\eqref{eq:cubed} in place of Theorem~\ref{thm:TJ} to get
\begin{equation}\label{eq:1/12}
\ee(\X)\gtrsim \max \left\{\frac{\sqrt[4]{n}}{\dd_{\X}},\frac{\sqrt{\dd_{\X}}}{(\log n)^3}\right\}\ge \frac{n^{\frac{1}{12}}}{(\log n)^{2}},
\end{equation}
where equality holds in the final step of~\eqref{eq:1/12} if and only if  $\dd_\X=\sqrt[6]{n}(\log n)^{2}$.

Prior to proving Theorem~\ref{thm:LT finite with log}, we will record the following two standard lemmas that will be used in its proof; both will be established in correct generality that also treats infinite dimensional Banach spaces even though here we will need them only in the finite dimensional setting (the infinite dimensional formulations are relevant to the discussion in Remark~\ref{rem:Elton}).

\begin{lemma}\label{lem:LT+1} For every Banach space $(\X,\|\cdot\|_{\X})$ we have $\LT(\X^*)\le \LT(\X)+1$.
\end{lemma}

\begin{proof} We may assume that $\LT(\X)<\infty$. Then $\X$ is reflexive (even isomorphic to Hilbert space), by~\cite{LT71}. Fix a closed linear subspace $\bfW$ of $\X^*$ and denote its pre-annihilator by $${}^\perp \bfW\eqdef  \bigcap_{x^{*}\in \bfW}\big\{x\in \X:\ x^*(x)=0\big\}\subset \X.$$ Suppose that $K>\LT(\X)$. By the definition of $\LT(\X)$ there exists $\proj:\X\to \X$ that is a projection from $\X$ onto ${}^\perp \bfW$ whose operator norm satisfies $\|\proj\|_{\X\to \X}\le K$. Observe that for every $x^*\in \X^*$ and $x\in {}^\perp \bfW$, $$\big(x^*-\proj^*x^*\big)(x)=x^*(x)-x^*(\proj x)=0,$$
since $\proj x=x$. This shows that
$$\big(\mathsf{Id}_{\X^*}-\proj^*\big)(\X^*)\subset ({}^\perp \bfW)^\perp=\big\{x^*\in X^*:\ x^*({}^\perp \bfW)=\{0\}\big\}=\bfW,$$ where the last step follows from the double annihilator theorem since $\X$ is reflexive and hence $\bfW$ is $\mathrm{weak}^*$ closed in $\X^*$. If $x^*\in \bfW$, then for any $x\in \X$  we have $\proj^*x^*(x)=x^*(\proj x)=0$, as $\proj x\in {}^\perp \bfW$. Hence $\proj^*x^*=0$, and so $\mathsf{Id}_{\X^*}-\proj^*$ acts as the identity when it is restricted to $\bfW$, i.e.,  $\mathsf{Id}_{\X^*}-\proj^*:\X^*\to \X^*$ is a projection from $\X^*$ onto $\bfW$. It remains to note that
\begin{equation*}
\big\|\mathsf{Id}_{\X^*}-\proj^*\big\|_{\X^{\textbf{*}}\to \X^{\textbf{*}}}\le 1+\big\|\proj^*\big\|_{\X^{\textbf{*}}\to \X^{\textbf{*}}}=1+\big\|\proj\big\|_{\X\to \X}\le K+1.\tag*{\qedhere}\end{equation*}
\end{proof}

The following simple lemma shows that the Lindenstrauss--Tzafriri constant is a bi-Lipschitz invariant.

\begin{lemma} Any two Banach spaces $(\bfW,\|\cdot\|_\bfW)$  and $(\X,\|\cdot\|_{\X})$ satisfy
\begin{equation}\label{eq:LT invaraiant}
\LT(\bfW)\le \cc_{\X}(\bfW)\LT(\X).
\end{equation}
\end{lemma}

\begin{proof} We may assume that $\cc_\X(\bfW)<\infty$ and $\LT(\X)<\infty$. By~\cite{LT71}, the latter assumption implies that $\X$ is isomorphic to a Hilbert space, and hence it is reflexive. We may therefore apply a differentiation argument (see e.g.~\cite[Corollary~7.10]{BL00}) to deduce that there is a closed subspace $\Y$ of $\X$ such that $d_\BM(\bfW,\Y)= \cc_\X(\bfW)$. In other words, for every $D>\cc_\X(\bfW)$ there is a linear isomorphism $T:\bfW\to \Y$ satisfying $\|T\|_{\bfW\to \Y}\|T^{-1}\|_{\Y\to \bfW}<D$. If $\bfV$ is a closed subspace of $\bfW$ and $K>\LT(\X)$, then there is a projection $\proj$ from $\X$ onto $T\bfV$ with $\|\proj\|_{\X\to T\bfV}<K$. Now, $T^{-1}\proj T$ is a projection from $\bfW$ onto $\bfV$ of norm less than $DK$.
\end{proof}

The type-$2$ constant of a normed space $(\X,\|\cdot\|_\X)$ is equal to its ``equal norm type-$2$ constant,'' namely to the infimum over those $T>0$ for which the second inequality in~\eqref{eq:def type cotype-older} holds for every $m\in \N$ and every choice of vectors $x_1,\ldots,x_m\in \X$ that satisfy the additional requirement $\|x_1\|_\X=\ldots=\|x_m\|_\X$; this is a well-known result of  Pisier, though it first appeared in James' important work~\cite{Jam78}, where it had a vital role. We will likewise  need to use this result, with the  twist that we require a small number of unit vectors for which the type-$2$ constant of $\X$ is almost attained. The classical proof of the aforementioned equivalence between  type-$2$ and ``equal norm type-$2$''  (page~2 of~\cite{Jam78}) increases the number of vectors potentially uncontrollably, so we will preform the analysis more carefully in the following lemma, which shows that one need not increase the number of vectors when passing from general vectors to unit vectors.

\begin{lemma}[equal norm type $2$ without increasing the number of vectors]\label{lem:equal norm} Fix $n\in \N$ and $0<\beta\le 1$. Let $(\X,\|\cdot\|_\X)$ be a normed space and suppose that there exist vectors $x_1,\ldots,x_n\in \X\setminus\{0\}$ that satisfy
\begin{equation}\label{eq:saturation type assumption}
\Bigg(\E\bigg[\Big\|\sum_{i=1}^n \g_i x_i\Big\|_\X^2\bigg]\Bigg)^{\frac12}\ge \beta T_2(\X) \bigg(\sum_{i=1}^n \|x_i\|_\X^2\bigg)^{\frac12}.
\end{equation}
Then, there also exist unit vectors  $y_1,\ldots,y_n\in \{x_i/\|x_i\|_\X\}_{i=1}^n\subset \partial B_\X$ that satisfy
\begin{equation}\label{eq:equal norm type two conclusion in lemma}
\Bigg(\E\bigg[\Big\|\sum_{i=1}^n \g_i y_i\Big\|_\X^2\bigg]\Bigg)^{\frac12}\gtrsim \beta^2T_2(\X)\sqrt{n}.
\end{equation}
\end{lemma}

\begin{proof} We may assume without loss of generality the following normalized version of assumption~\eqref{eq:saturation type assumption}.
\begin{equation}\label{eq:TJmaximizers of type}
\sum_{i=1}^n \|x_i\|_\X^2=1\qquad\mathrm{and}\qquad \E\bigg[\Big\|\sum_{i=1}^n \g_i x_i\Big\|_\X^2\bigg]\ge \beta^2 T_2(\X)^2.
\end{equation}
For every $k\in \N$ define a subset $I_k$ of $\n$ by
\begin{equation}\label{eq:def Ik}
I_k\eqdef \left\{i\in \{1\,\dots,n\}:\ \frac{1}{2^k}<\|x_i\|_{\X}\le \frac{1}{2^{k-1}}\right\}.
\end{equation}
So, $\{I_k\}_{k\in \N}$ is a partition of $\{1,\ldots,n\}$ as $0<\|x_i\|_{\X}\le 1$ for all $i\in \n$ by the first equation in~\eqref{eq:TJmaximizers of type}. Write
\begin{equation}\label{eq:def mU}
m\eqdef \left\lceil\log_2\left(\frac{3\sqrt{n}}{\beta}\right)\right\rceil\qquad \mathrm{and}\qquad
U\eqdef \bigcup_{k=1}^m I_k\times \Big\{1,\ldots, 2^{2(m-k)}\Big\}.
\end{equation}
With this notation,  Lemma~\ref{lem:equal norm} will be proven if we  show that there exists  $S\subset U$ with $|S|=n$ such that
\begin{equation}\label{eq:goal equal norm}
\Bigg(\E\bigg[\Big\|\sum_{(i,j)\in S} \frac{\g_{ij}}{\|x_i\|_\X}x_i\Big\|_\X^2\bigg]\Bigg)^{\frac12}\gtrsim \beta^2T_2(\X)\sqrt{n},
\end{equation}
where  $\{\g_{ij}\}_{i,j=1}^\infty$  are i.i.d.~standard Gaussian random variables.

To prove~\eqref{eq:goal equal norm},  observe first that by the contraction principle (see e.g.~\cite[Section~4.2]{LT91}) we have
\begin{equation}\label{eq:use contraction principle}
\Bigg(\E\bigg[\Big\|\sum_{(i,j)\in S} \frac{\g_{ij}}{\|x_i\|_\X}x_i\Big\|_\X^2\bigg]\Bigg)^{\frac12}\ge  \Bigg(\E\bigg[\Big\|\sum_{k=1}^m 2^{k-1} \sum_{i\in I_k} \sum_{j=1}^{2^{2(m-k)}}\1_{\{(i,j)\in S\}} \g_{ij}x_i\Big\|_\X^2\bigg]\Bigg)^{\frac12},
\end{equation}
where we used the fact that $1/\|x_i\|_\X\ge 2^{k-1}$  for every $k\in \N$ and $i\in I_k$ (by the definition~\eqref{eq:def Ik} of $I_k$). Also,
$$
1\stackrel{\eqref{eq:TJmaximizers of type}}{=}\sum_{i=1}^n \|x_i\|_\X^2=\sum_{k=1}^\infty\sum_{i\in I_k} \|x_i\|_\X^2 \stackrel{\eqref{eq:def Ik}}{\le} \sum_{k=1}^\infty \frac{|I_k|}{2^{2k-2}}\le \frac{4\sum_{k=1}^m 2^{2(m-k)}|I_k|+\sum_{k=m+1}^\infty |I_k|}{2^{2m}}\stackrel{\eqref{eq:def mU}}{\le}\frac{\beta^2(4|U|+n)}{9n}.
$$
This simplifies to give that $|U|\ge 2n/\beta^2>n$. We can therefore average the right hand side of~\eqref{eq:use contraction principle} over all the $n$-point subsets of $U$ a to get the following estimate.
\begin{align}\label{eq:reduce to size of U}
\begin{split}
\frac{1}{\binom{|U|}{n}}\sum_{\substack{S\subset U\\ |S|=n}}\Bigg(\E\bigg[\Big\|\sum_{k=1}^m 2^{k-1} \sum_{i\in I_k} \sum_{j=1}^{2^{2(m-k)}}\1_{\{(i,j)\in S\}} \g_{ij}x_i\Big\|_\X^2\bigg]\Bigg)^{\frac12}&\ge \Bigg(\E\bigg[\Big\|\sum_{k=1}^m 2^{k-1} \sum_{i\in I_k} \sum_{j=1}^{2^{2(m-k)}}\frac{\binom{|U|-1}{n-1}}{\binom{|U|}{n}} \g_{ij}x_i\Big\|_\X^2\bigg]\Bigg)^{\frac12}\\
&=\frac{n}{2|U|}\Bigg(\E\bigg[\Big\|\sum_{k=1}^m 2^{k} \sum_{i\in I_k} \sum_{j=1}^{2^{2(m-k)}} \g_{ij}x_i\Big\|_\X^2\bigg]\Bigg)^{\frac12}\\
&=\frac{2^{m-1}n}{|U|} \Bigg(\E\bigg[\Big\|\sum_{k=1}^m\sum_{i\in I_k}\g_{i}  y_i\Big\|_\X^2\bigg]\Bigg)^{\frac12}\\&\asymp\frac{n^{\frac32}}{\beta|U|}\Bigg(\E\bigg[\Big\|\sum_{k=1}^m\sum_{i\in I_k}\g_{i}  y_i\Big\|_\X^2\bigg]\Bigg)^{\frac12},
\end{split}
\end{align}
where the first step of~\eqref{eq:reduce to size of U} uses convexity, the penultimate step of~\eqref{eq:reduce to size of U} uses the fact that
$$
\Bigg(\bigg(\sum_{j=1}^{2^{2(m-k)}} \g_{ij}\bigg)_{i\in I_k}\Bigg)_{k=1}^m \qquad \mathrm{and}\qquad \Big(\big(2^{m-k}\g_i\big)_{i\in I_k}\Big)_{k=1}^m
$$
have the same distribution, and for the final step of~\eqref{eq:reduce to size of U} recall the definition~\eqref{eq:def mU} of $m$.

It follows from~\eqref{eq:use contraction principle} and~\eqref{eq:reduce to size of U} that there must exist $S\subset U$ with $|S|=n$ such that
\begin{equation}\label{eq:reduce to size of U'}
\Bigg(\E\bigg[\Big\|\sum_{(i,j)\in S} \frac{\g_{ij}}{\|x_i\|_\X}x_i\Big\|_\X^2\bigg]\Bigg)^{\frac12}\gtrsim \frac{n^{\frac32}}{\beta|U|}\Bigg(\E\bigg[\Big\|\sum_{k=1}^m\sum_{i\in I_k}\g_{i}  x_i\Big\|_\X^2\bigg]\Bigg)^{\frac12}.
\end{equation}
To use~\eqref{eq:reduce to size of U'}, we claim that  $|U|\lesssim n/\beta^2$. Indeed,
$$
1\stackrel{\eqref{eq:TJmaximizers of type}}{=}\sum_{i=1}^n \|x_i\|_\X^2=\sum_{k=1}^\infty\sum_{i\in I_k} \|x_i\|_\X^2 \stackrel{\eqref{eq:def Ik}}{>}   \sum_{k=1}^m \frac{|I_k|}{2^{2k}}\stackrel{\eqref{eq:def mU}}{=}\frac{|U|}{2^{2m}}\stackrel{\eqref{eq:def mU}}{\ge} \frac{\beta^2|U|}{81n}.
$$
By combining the aforementioned upper bound on the size of $U$ with~\eqref{eq:use contraction principle} and~\eqref{eq:reduce to size of U'}, we see that
$$
\Bigg(\E\bigg[\Big\|\sum_{(i,j)\in S} \frac{\g_{ij}}{\|x_i\|_\X}x_i\Big\|_\X^2\bigg]\Bigg)^{\frac12}\gtrsim \beta\sqrt{n}\Bigg(\E\bigg[\Big\|\sum_{k=1}^m\sum_{i\in I_k}\g_{i}  y_i\Big\|_\X^2\bigg]\Bigg)^{\frac12}.
$$
From this, we deduce the desired estimate~\eqref{eq:goal equal norm} by combining as follows the second inequality in our assumption~\eqref{eq:TJmaximizers of type} with the triangle inequality and the definition~\eqref{eq:def type cotype-older} of the type-$2$ constant $T_2(\X)$.
\begin{multline*}
\Bigg(\E\bigg[\Big\|\sum_{k=1}^m\sum_{i\in I_k}\g_{i}  x_i\Big\|_\X^2\bigg]\Bigg)^{\frac12}\ge \Bigg(\E\bigg[\Big\|\sum_{i=1}^\infty\g_{i}  x_i\Big\|_\X^2\bigg]\Bigg)^{\frac12}-\Bigg(\E\bigg[\Big\|\sum_{k=m+1}^\infty\sum_{i\in I_k}\g_{i}  x_i\Big\|_\X^2\bigg]\Bigg)^{\frac12}\\\stackrel{\eqref{eq:def type cotype-older}}{\ge}
\beta T_2(\X)-T_2(\X)\Bigg(\sum_{k=m+1}^\infty\sum_{i\in I_k}\|x_i\|_\X^2\Bigg)^{\frac12}
\stackrel{\eqref{eq:def Ik}}{\ge} \beta T_2(\X)-\frac{T_2(\X)\sqrt{n}}{2^m}\stackrel{\eqref{eq:def mU}}{\asymp} \beta T_2(\X). \tag*{\qedhere}
\end{multline*}
\end{proof}

\begin{proof}[Proof of Theorem~\ref{thm:LT finite with log}] We will prove that the type $2$ constant of $\X$ satisfies
\begin{equation}\label{eq:type 2 upper}
T_2(\X)\lesssim \LT(\X)(\log n)^{\frac32}.
\end{equation}
After~\eqref{eq:type 2 upper} will be proven, we deduce Theorem~\ref{thm:LT finite with log} as follows. We first claim that the estimate~\eqref{eq:type 2 upper} implies the same upper bound on the cotype $2$ constant of $\X$. Namely, we also have
\begin{equation}\label{eq:cotype 2 upper}
C_2(\X)\lesssim \LT(\X)(\log n)^{\frac32}.
\end{equation}
Indeed,
\begin{equation}\label{eq:justify cotype}
C_2(\X)\le T_2(\X^*)\lesssim \LT(\X^*)(\log n)^{\frac32}\lesssim \LT(\X)(\log n)^{\frac32},
\end{equation}
where the first step of~\eqref{eq:justify cotype} follows from a standard duality argument~\cite{MP76} (see also e.g.~\cite[Section~9.10]{MS86}, \cite[Section~4.9]{PW98} or~\cite[Proposition~6.2.12]{AK06}), the second step of~\eqref{eq:justify cotype}  is an application of~\eqref{eq:type 2 upper} to $\X^*$, and the third step of~\eqref{eq:justify cotype}  is application of Lemma~\ref{lem:LT+1}. The desired estimate~\eqref{eq:cubed} now follows by a substitution of~\eqref{eq:type 2 upper}  and~\eqref{eq:cotype 2 upper} into Theorem~\ref{thm:Kwapien} (Kwapie\'n's theorem).

 By~\cite[Lemma~6.1]{FLM77} (see also the exposition of this fact in~\cite[page~546]{JN10}) there exists an integer\footnote{By~\cite{Tom79}, if one does not mind losing a universal constant factor in~\eqref{eq:maximizers of type}, then one could take $m=n$ here, but  for the purpose of the ensuing reasoning it suffices to use the much simpler result~\cite[Lemma~6.1]{FLM77}.}
 \begin{equation}\label{eq:m quadratic}
 1\le m\le \frac{n(n+1)}{2}\end{equation} and $x_1,\ldots,x_m\in \X\setminus \{0\}$ such that
\begin{equation}\label{eq:maximizers of type}
\Bigg(\E\bigg[\Big\|\sum_{i=1}^m \g_i x_i\Big\|_{\X}^2\bigg]\Bigg)^{\frac12}= T_2(\X) \bigg(\sum_{i=1}^m \|x_i\|_{\X}^2\bigg)^{\frac12}
\end{equation}
By Lemma~\ref{lem:equal norm}, it follows that there  exist $y_1,\ldots, y_m\in \partial B_\X$ and a universal constant $0<\gamma<1$ such that
\begin{equation}\label{eq:to substitute best beta}
\E\bigg[\Big\|\sum_{i=1}^m \g_i y_i\Big\|_{\X}\bigg]\ge \sqrt{\frac{2}{\pi}}\bigg(\E\bigg[\Big\|\sum_{i=1}^m \g_i y_i\Big\|_{\X}^2\bigg]\bigg)^{\frac12}\ge \gamma T_2(\X)\sqrt{m},
\end{equation}
where the first step in~\eqref{eq:to substitute best beta} holds by (the Gaussian version of) Kahane's inequality~\cite{Kah64} (see e.g.~\cite[Corollary~3.2]{LT91} and specifically~\cite[Corollary~3]{LO99} for the (optimal) constant that we are quoting here even though its value is of secondary importance in the present context). If we denote
\begin{equation}\label{eq:def delta RV}
\delta\eqdef \frac{\gamma T_2(\X)}{\sqrt{m}},
\end{equation}
then a different way to write~\eqref{eq:to substitute best beta} is
\begin{equation}\label{eq:conclusion for RV}
\E\bigg[\Big\|\sum_{i=1}^m \g_i y_i\Big\|_{\X}\bigg]\ge \d m.
\end{equation}

Because we ensured that $y_1,\ldots,y_m$ are unit vectors in $\X$, we may use a theorem of Rudelson and Vershynin~\cite[Theorem~7.4]{RV06} (an improved Talagrand-style two-parameter version of Elton's theorem; see Remark~\ref{rem:Elton}), to deduce  from~\eqref{eq:conclusion for RV} that there are two numbers $0<s\le 1$ and  $\d\lesssim t\le 1$ that satisfy
\begin{equation}\label{eq:st RV}
t\sqrt{s}\gtrsim \frac{\d}{\left(\log\left(\frac{2}{\d}\right)\right)^{\frac32}},
\end{equation}
such that there exists a subset $J$ of $\m$ whose cardinality satisfies
\begin{equation}\label{eq:RV subset lower}
|J|\ge sm,
\end{equation}
and moreover we have
\begin{equation}\label{eq:RV conclusion}
\forall (a_j)_{j\in J}\in \R^J,\qquad t\sum_{j\in J} |a_j|\lesssim \Big\|\sum_{j\in J} a_j y_j\Big\|_{\X}\le \sum_{j\in J} |a_j|.
\end{equation}
\eqref{eq:RV conclusion} means that the Banach--Mazur distance between $\spn(\{y_j\}_{j\in J})$ and $\ell_1^{|J|}$ is $O(1/t)$. Hence,
\begin{equation}\label{eq:1/t distortion}
\cc_{\X}\big(\ell_1^{|J|}\big)\lesssim \frac{1}{t}.
\end{equation}
Now, the justification of~\eqref{eq:type 2 upper}, and hence also the proof of Theorem~\ref{thm:LT finite with log},  can be completed as follows.
\begin{align}\label{eq:punchline LT}
\LT(\X)\stackrel{\eqref{eq:LT invaraiant}}{\ge}\frac{\LT\big(\ell_1^{|J|}\big)}{\cc_{\X}\big(\ell_1^{|J|}\big)}\stackrel{\eqref{eq:sobczyk} \wedge \eqref{eq:1/t distortion}}{\gtrsim} t\sqrt{|J|}\stackrel{\eqref{eq:RV subset lower}}{\ge} t\sqrt{sm}\stackrel{\eqref{eq:st RV}}{\gtrsim} \frac{\d\sqrt{m}}{\left(\log\left(\frac{2}{\d}\right)\right)^{\frac32}}
\stackrel{\eqref{eq:def delta RV}}{=} \frac{\gamma T_2(\X)}{ \Big(\log\Big(\frac{2\sqrt{m}}{\gamma T_2(\X)}\Big)\Big)^{\frac32}}\gtrsim \frac{T_2(\X)}{(\log n)^\frac32},
\end{align}
where the final step of~\eqref{eq:punchline LT} holds because $T_2(\X)\ge 1$ and $ \log m\lesssim \log n$ by~\eqref{eq:m quadratic}.
\end{proof}

\begin{remark}\label{rem:Elton} In the  proof of Theorem~\ref{thm:LT finite with log} we relied on~\cite[Theorem~7.4]{RV06}, which improves (in terms  of the power of the logarithm in~\eqref{eq:st RV}) Talagrand's refinement~\cite{Tal-elton92} of Elton's theorem~\cite{Elt83} (which is itself a major quantitative strengthening of an important theorem from~\cite{Pis73}). Continuing with the notation of Theorem~\ref{thm:LT finite with log}, Elton's theorem is a similar statement, except that the size of the subset $J$ is a definite proportion of $m$ that depends only on the parameter  $\d$ for which~\eqref{eq:conclusion for RV} holds, and also the parameter $t$ for which~\eqref{eq:RV conclusion} holds depends only on $\d$. The asymptotic dependence on $\d$ in Elton's theorem~\cite{Elt83} was improved by Pajor~\cite{Paj83}, a further improvement was obtained in~\cite{Tal-elton92}, and the optimal dependence on $\d$ was found by Mendelson and Vershynin in~\cite{MV03}. However, plugging this sharp dependence into our proof of Theorem~\ref{thm:LT finite with log} shows that the classical formulation of Elton's theorem is insufficient for our purposes. The two-parameter formulation of Elton's theorem that was introduced in~\cite{Tal-elton92} allows for the subset $J$ to have any size through the parameter $s$ in~\eqref{eq:RV subset lower}, but imposes a relation between $s$ and $t$ such as~\eqref{eq:st RV}, thus making it possible for us to obtain Theorem~\ref{thm:LT finite with log}.

The only reason why the logarithmic factor in~\eqref{thm:LT finite with log} occurs is our use of a Talagrand-style  two-parameter version of Elton's theorem, for which the currently best-known bound~\cite{RV06} is~\eqref{eq:st RV}. Thus, if~\eqref{eq:st RV} could be improved to $t\sqrt{s}\gtrsim \d$, i.e., if Question~\ref{eq:sharp 2 parameter elton} below has a positive answer, then  the conclusion~\eqref{thm:LT finite with log} of Theorem~\ref{thm:LT finite with log}  would become $\dd_{\X}\lesssim \LT(\X)^2$. This  would improve Theorem~\ref{thm:TJ} to match the  bound of Theorem~\ref{thm:LT infinite dim} which is currently known only for infinite dimensional Banach spaces. Moreover, since the resulting bound is independent of the dimension of $\X$, this would yield a new proof of the Lindenstrauss--Tzafriri solution of the complemented subspace problem; the infinite dimensional statement follows formally from its finite dimensional counterpart (e.g.~\cite[Theorem~12.1.6]{AK06}), though all of the steps that led to Theorem~\ref{thm:LT finite with log} work for any reflexive Banach space. Question~\ref{eq:sharp 2 parameter elton} is interesting in its own right regardless of the above application to the complemented subspace problem. In particular, a positive answer to Question~\ref{eq:sharp 2 parameter elton} would resolve the question that Talagrand posed in the remark right after Corollary~1.2 in~\cite{Tal-elton92}, though we warn that he characterises this in~\cite{Tal-elton92} as ``certainly a rather formidable question.''

\begin{question}\label{eq:sharp 2 parameter elton} Fix $0<\d<1$ and $n\in \N$. Let $(\X,\|\cdot\|_\X)$ be a Banach space and suppose that $x_1,\ldots,x_n\in \partial B_\X$ satisfy $
\E[\|\sum_{i=1}^m \g_i x_i\|_{\X}]\ge \d n$.
Does this imply that there are two numbers $0<s,t\le 1$ satisfying $t\sqrt{s}\gtrsim \d$ and a subset $J\subset \n$ with $|J|\ge sn$ such that $\|\sum_{j\in J} a_j x_j\|_\X\ge t\sum_{j\in J} |a_j|$ for every $a_1,\ldots,a_n\in \R$?
\end{question}
\end{remark}

\subsection{Proof of~\eqref{eq:sep square type}}\label{sec:consequences of distance} Because by~\cite{CCGGP98} we know that $\sep(\ell_1^n)\asymp n$ for every $n\in \N$, using bi-Lipschitz invariance we see that in order to prove~\eqref{eq:sep square type} it suffices to show that for any normed space $\X=(\R^n,\|\cdot\|_\X)$,
\begin{equation}\label{eq:m over dist}
\exists m\in \n,\qquad \frac{m}{\cc_\X(\ell_1^m)}\ge T_2(\X)^2.
\end{equation}

We will prove~\eqref{eq:m over dist} using Talagrand's two-parameter refinement of Elton's theorem~\cite{Tal-elton92} that we discussed in Remark~\ref{rem:Elton} (the aforementioned improvements over~\cite{Tal-elton92} in~\cite{MV03,RV06} do not yield a better bound in the ensuing reasoning. Also, the classical formulation of Elton's theorem is insufficient for our purposes, even if one incorporates the asymptotically sharp dependence on $\d$ from~\cite{MV03}).  Suppose that $k\in \N$ and $x_1,\ldots,x_k\in B_\X$. Let $\g_1,\ldots,\g_k$ be i.i.d.~standard Gaussian random variables. Denote
\begin{equation*}
E\eqdef \E\bigg[\Big\|\sum_{j=1}^k\g_j x_j\Big\|_{\X}\bigg].
\end{equation*}
By~\cite[Corollary~1.2]{Tal-elton92}, there is a universal constant $C\in [1,\infty)$ and a subset $S\subset \{1,\ldots,k\}$ satisfying
$$
m\eqdef |S|\ge \frac{E^2}{Ck},
$$ and such that
\begin{equation}\label{eq:elton tal}
\forall(a_j)_{j\in S}\in \R^S,\qquad \frac{E}{\sqrt{Ckm}\Big(\log\big(\frac{eCkm}{E^2}\big)\Big)^C} \sum_{j\in S} |a_j|\le \Big\|\sum_{j\in S} a_j x_j\Big\|_\X\le \sum_{j\in S} |a_j|.
\end{equation}
Consequently,
\begin{equation*}
\cc_{\X}\big(\ell_1^m\big)\le \frac{\sqrt{Ckm}}{E}\bigg(\log\Big(\frac{eCkm}{E^2}\Big)\bigg)^C.
\end{equation*}
Therefore,
$$
\frac{m}{\cc_\X(\ell_1^m)}\ge \frac{E\sqrt{m}}{\sqrt{Ck}\Big(\log\big(\frac{eCkm}{E^2}\big)\Big)^C}\ge \frac{e^{C-\frac12}}{2^CC^{C+1}}\cdot \frac{E^2}{k}\asymp \frac{E^2}{k},
$$
where the last step uses the fact that the minimum of the function $u\mapsto \sqrt{u}/(\log(eCku/E^2))^{C}$ on the ray $[E^2/(Ck),\infty)$ is attained at $u=e^{2C-1}E^2/(Ck)$. It remains to choose $x_1,\ldots,x_k$ so that $E^2/k\asymp T_2(\X)^2$. This is possible because the equal norm type $2$ constant of $\X$ equals $T_2(\X)$, so there are $x_1,\ldots,x_k\in \partial B_\X$ for which
$$
T_2(\X)\sqrt{k}\asymp  \bigg(\E\bigg[\Big\|\sum_{j=1}^k\g_j x_j\Big\|_{\X}^2\bigg]\bigg)^{\frac12}\asymp E,
$$
where the last step uses Kahane's inequality.\qed

\subsection{H\"older extension}\label{sec:holder ext} In this section we will prove the lower bound on $\ee^\theta(\ell_\infty^n)$ in~\eqref{eq:lower theta holder l infty} for every $n\in \N$ and $0<\theta\le 1$. It consists of two estimates, the first of which is
\begin{equation}\label{eq:holder reduce to linear}
\ee^\theta(\ell_\infty^n)\gtrsim n^{\frac{\theta}{2}+\theta^2-1},
\end{equation}
and the second of which is
\begin{equation}\label{eq:holder from cotype}
\ee^\theta(\ell_\infty^n)\gtrsim n^{\frac{\theta}{4}}.
\end{equation}
We will justify~\eqref{eq:holder reduce to linear} and~\eqref{eq:holder from cotype} separately.

Note that~\eqref{eq:holder reduce to linear} is vacuous if $\theta/2+\theta^2-1\le 0$, i.e., if $0<\theta\le (\sqrt{17}-1)/2$. The reason for this is that~\eqref{eq:holder reduce to linear} is based on a reduction to the linear theory from~\cite{NR17} (extending the approach of~\cite{JL84} to the H\"older regime), that breaks down for functions which are too far from being Lipschitz. Specifically, for a Banach space $\X$ and a closed subspace $\bfE$ of $\X$, let $\lambda(\bfE;\X)$ be the  projection constant~\cite{Gru60} of $\bfE$ relative to $\X$, i.e., it is the infimum over those $\lambda\in [1,\infty]$ for which there is a projection $\proj$ from $\X$ onto $\bfE$ whose operator norm satisfies $\|\proj\|_{\X\to \bfE}\le \lambda$. Also, let $\ee^\theta(\X;\bfE)$ be the infimum over those $L\in [1,\infty]$ such that for every $\sub\subset \X$ and every $f:\sub\to \bfE$ that is $\theta$-H\"older with constant $1$, there is $F:\X\to \bfE$ that extend $f$ and is $\theta$-H\"older with constant $L$. With this notation,  it was proved in~\cite{NR17} (see equation~(106) there) that
\begin{equation}\label{eq:quote NR}
\ee^\theta(\X;\bfE)\gtrsim \frac{\lambda(\bfE;\X)^\theta}{\dim(\bfE)^{\frac{\1-\theta}{2}}\dim(\X)^{\theta(1-\theta)}\cc_2(\bfE)^{1-\theta}}.
\end{equation}
Using the bounds $\dim(\bfE)\le \dim(\X)$ and  $\cc_2(\bfE)\le \sqrt{\dim(\bfE)}$ (John's theorem) in~\eqref{eq:quote NR}, we get that
\begin{equation}\label{eq:crude susstitutions into NR}
\ee^\theta(\X;\bfE)\gtrsim \frac{\lambda(\bfE;\X)^\theta}{\dim(\X)^{1-\theta^2}}.
\end{equation}
By~\cite{Sob41} there is a linear subspace $\bfE$ of $\ell_\infty^n$ with $\lambda(\bfE;\ell_\infty^n)\asymp \sqrt{n}$, using which~\eqref{eq:crude susstitutions into NR} implies~\eqref{eq:holder reduce to linear}.

\begin{remark} In~\cite{NR17} it was deduced from~\eqref{eq:quote NR} that
\begin{equation}\label{eq:NR ell1}
\ee^\theta(\ell_1^n)\gtrsim n^{\theta^2-\frac12}.
\end{equation}
Specifically, by~\cite{Kas77} there is a linear subspace $\bfE$ of $\ell_1^n$ with $\cc_2(\bfE)\lesssim 1$ and $\dim(\bfE)=\lfloor n/2\rfloor$;  call such $\bfE$ a Ka\v{s}in subspace of $\ell_1^n$. By~\cite{Rut65} we have $\lambda(\bfE;\ell_1^n)\asymp \sqrt{n}$, so~\eqref{eq:NR ell1} follows by substituting these parameters into~\eqref{eq:quote NR}. For $\X=\ell_\infty^n$, the  poorly-complemented subspace  that we used above can be taken to be the orthogonal complement of any Ka\v{s}in subspace of $\ell_1^n$. Such a subspace of $\ell_\infty^n$ has pathological properties~\cite{FJ80}; in particular its Banach--Mazur distance to a Euclidean space is of order $\sqrt{n}$. So, a ``vanilla'' use of~\eqref{eq:quote NR} leads at best to~\eqref{eq:holder reduce to linear}. However, we  expect that it should be possible to improve~\eqref{eq:holder reduce to linear} to
\begin{equation}\label{eq:conjecture holder infty}
\ee^\theta(\ell_\infty^n)\gtrsim n^{\theta^2-\frac12}.
\end{equation}
If~\eqref{eq:conjecture holder infty} holds,  then~\eqref{eq:lower theta holder l infty} improves to
\begin{equation}\label{eq:lower theta holder l infty'}
\ee^\theta(\ell_\infty^n)\gtrsim n^{\max\left\{\frac{\theta}{4},\theta^2-\frac12\right\}}=\left\{\begin{array}{ll}n^{\frac{\theta}{4}}&\mathrm{if}\quad 0\le \theta\le \frac{1+\sqrt{33}}{8},\\
n^{\theta^2-\frac12}&\mathrm{if}\quad  \frac{1+\sqrt{33}}{8}\le \theta\le 1.\end{array}\right.
\end{equation}
For~\eqref{eq:conjecture holder infty}, it would suffice to prove the following variant of Conjecture~\ref{conj:affine invariant version}   for random subspaces of $\ell_\infty^n$ . Let $\bfE$ be a  subspace of $\R^n$ of dimension $m=\lfloor n/2\rfloor$ that is chosen from the Haar measure on the Grassmannian. We conjecture that there is a universal constant $D\ge 1$ such that with high probability there is an origin-symmetric convex body $L\subset B_\bfE$ that satisfies $\mathrm{MaxProj}(L)/\vol_m(L)\lesssim 1$. If this indeed holds, then by using it in the {\em proof of}~\eqref{eq:quote NR} in~\cite{NR17} we can deduce~\eqref{eq:conjecture holder infty} (specifically, replace in Lemma~20 of~\cite{NR17} the averaging over $B_{\ell_2^m}$ by averaging over $L$; we omit the details of this adaptation of~\cite{NR17}).
\end{remark}

\begin{proof}[Proof of~\eqref{eq:holder from cotype}] Fix $k,m\in \N$ satisfying $k\le 2m\le n/2$ whose value will be specified later so as to optimize the ensuing reasoning (see~\eqref{eq:thetha constraint} below). Denote $\ell= \lfloor (4m/k)\rfloor$ and define $\sub=\sub(k,m,n)\subset \ell_\infty^n(\C)$ by
$$
\sub\eqdef \big\{ E_m(ks):\ s\in \{1,\ldots \ell\}^n\big\},
$$
where for every $s=(s_1,\ldots, s_n)\in \R^n$ we define $E_m(s)\in \C^n$ by
$$
E_m(s)\eqdef \sum_{j=1}^n e^{\frac{\pi i}{2m}s_j}e_j.
$$

Denote  the standard  basis (delta masses) of $\R^\sub$ by $\{\bd_s\}_{s\in \sub}$. Let $\R^\sub_0$ be the hyperplane of $\R^\sub$ consisting of  those $(a_s)_{s\in \sub}=\sum_{s\in \sub} a_s\bd_s$ with $\sum_{s\in \sub} a_s=0$. Suppose that $\X_\theta=(\R^\sub_0,\|\cdot\|_{\X_\theta})$ is a normed space that satisfies
\begin{equation}\label{eq:conditions from X theta1}
\forall x,y\in \sub,\qquad \|\bd_x-\bd_y\|_{\X_\theta}= \|x-y\|_{\ell_\infty^n(\C)}^\theta,
\end{equation}
and,
\begin{equation}\label{eq:conditions from X theta2}
\forall\mu\in \R_0^\sub,\qquad \Big(\frac{k}{m}\Big)^\theta\|\mu\|_{\ell_1(\sub)}\lesssim \|\mu\|_{\X_\theta}\lesssim \|\mu\|_{\ell_1(\sub)}.
\end{equation}
For this, $\X_\theta$ can be taken to be the normed space whose unit ball is
\begin{equation}\label{eq:ball of X theta}
B_{\X_\theta}=\conv \Bigg\{\frac{1}{\|x-y\|_{\ell_\infty^n(\C)}^\theta}(\bd_x-\bd_y):\ x,y\in \sub,\ x\neq y\Bigg\}\subset \R^\sub_0,
\end{equation}
which is the maximal norm on $\R_0^\sub$ satisfying~\eqref{eq:conditions from X theta1}. To check that~\eqref{eq:conditions from X theta2} holds for the choice~\eqref{eq:ball of X theta}, note that, as $1\le k\le 2m$, distinct $x,y\in \sub$ satisfy $k/m \lesssim \|x-y\|_{\ell^n_\infty(\C)}\lesssim 1$. It is simple to deduce~\eqref{eq:conditions from X theta2} from this, as done in~\cite[Lemma~7]{NR17}. The choice~\eqref{eq:ball of X theta} makes $\X_\theta$ be the Wasserstein-1 space over $(\sub, d_\theta)$, where $d_\theta$ is the $\theta$-snowflake of the $\ell_\infty^n(\C)$  metric, i.e., $d_\theta(x,y)=\|x-y\|^\theta_{\ell_\infty^n(\C)}$ for $x,y\in \ell_\infty^n(\C)$; see Section~\ref{sec:transport notation}.

By virtue of~\eqref{eq:conditions from X theta1}, if we define $f:\sub \to \X_\theta$ by setting
$$
\forall x\in \sub,\qquad f(x)\eqdef \bd_x-\frac{1}{|\sub|}\sum_{y\in \sub}\bd_y,
$$ then $f$ is $\theta$-H\"older with constant $1$. We claim that if $m\ge \pi\sqrt{n}$, then by~\eqref{eq:conditions from X theta1} every $F:\ell_\infty^n(\C)\to \X_\theta$ satisfies
\begin{align}\label{eq:cotype 2 in Xtheta}
\begin{split}
\frac{1}{(4m)^n}\sum_{j=1}^n \sum_{s\in \{1,\ldots,4m\}^n} &\big\|F \big(E_m(s+2me_j)\big)-F\big(E_m(s)\big)\big\|_{\X_\theta}\\ &\lesssim \frac{m^{2+\theta}}{k^\theta (12m)^n}\sum_{\e\in \{-1,0,1\}^n} \sum_{s\in \{1,\ldots,4m\}^n}\big\|F\big(E_m(s+\e)\big)-F\big(E_m(s)\big)\big\|_{\X_\theta}.
\end{split}
\end{align}
\begin{comment}
\begin{align}\label{eq:cotype 2 in Xtheta'}
\begin{split}
\frac{1}{(4m)^n}\sum_{j=1}^n \sum_{s\in \{1,\ldots,4m\}^n} &\bigg\|F \Big(-e^{\frac{\pi i}{2m}s_j}e_j+\sum_{r\in \{1,\ldots,4m\}\setminus \{j\}}e^{\frac{\pi i}{2m}s_r}e_r\Big)-F\Big(\sum_{r=1}^{4m} e^{\frac{\pi i}{2m}s_r}e_r\Big)\bigg\|_{\X_\theta}\\ &\lesssim \frac{nm^\theta}{k^\theta (12m)^n}\sum_{\e\in \{-1,0,1\}^n} \sum_{s\in \{1,\ldots,4m\}^n}\bigg\|F\Big(\sum_{r=1}^{4m} e^{\frac{\pi i}{2m}(s_r+\e_r)}e_r\Big)-F\Big(\sum_{r=1}^{4m} e^{\frac{\pi i}{2m}s_r}e_r\Big)\bigg\|_{\X_\theta}.
\end{split}
\end{align}
\end{comment}
Indeed, \eqref{eq:cotype 2 in Xtheta} follows from a substitution of~\eqref{eq:conditions from X theta1} into the following inequality from~\cite[Remark~7.5]{MN08}.
\begin{align*}
\frac{1}{(4m)^n}\sum_{j=1}^n \sum_{s\in \{1,\ldots,4m\}^n} &\big\|F \big(E_m(s+2me_j)\big)-F\big(E_m(s)\big)\big\|_{\ell_1(\sub)}\\ &\lesssim \frac{m^2}{(12m)^n}\sum_{\e\in \{-1,0,1\}^n} \sum_{s\in \{1,\ldots,4m\}^n} \big\|F\big(E_m(s+\e)\big)-F\big(E_m(s)\big)\big\|_{\ell_1(\sub)}.
\end{align*}

Suppose that $F:\{1,\ldots,4m\}^n\to \X_\theta$ is $\theta$-H\"older with constant $L\ge 1$ on $(\{1,\ldots,4m\}^n,\|\cdot\|_{\ell_\infty^n(\C)})$, i.e.,
$$
x,y\in \{1,\ldots,4m\}^n,\qquad \|F(x)-F(y)\|_{\X_\theta}\le L\|x-y\|_{\ell_\infty^n(\C)}^\theta.
$$
Then, each of the summands that appear in the right hand side of~\eqref{eq:cotype 2 in Xtheta} is at most $2L/m^\theta$. Consequently,
\begin{equation}\label{eq:sup cotype}
\frac{1}{n(4m)^n}\sum_{j=1}^n \sum_{s\in \{1,\ldots,4m\}^n} \big\|F \big(E_m(s+2me_j)\big)-F\big(E_m(s)\big)\big\|_{\X_\theta}\lesssim \frac{Lm^{2}}{k^\theta n}.
\end{equation}
If $F$ also extends $f$, then $F(E_m(s))=f(E_m(s'))$ for every $s\in \N^n$, where $s'=(s_1',\ldots,s_n')$ and for each $u\in \N$ we let $u'$ be an element $\alpha$ of $\{k,2k,\ldots, \ell k\}$ for which $|\alpha-u\mod (4m)|$ is minimized, so that  $s'\in \sub$ and   \begin{equation}\label{eq:sub grid approximates}
\forall s\in \N^n,\qquad \|E_m(s)-E_m(s')\|_{\ell_\infty^n(\C)}\lesssim \frac{k}{m}.
\end{equation} Hence, for any $j\in \n$ and $s\in \{1,\ldots,4m\}^n$ we have
\begin{align}
\nonumber 2^\theta &=\big\|-2e^{\frac{\pi i}{2m}s_j}e_j\big\|_{\ell_\infty^n(\C)}^\theta\\&=\|E_m(s+2me_j)-E_m(s)\|_{\ell_\infty^n(\C)}^\theta \label{eq:recall Em}\\&\le \|E_m((s+2me_j)')-E_m(s')\|_{\ell_\infty^n(\C)}^\theta
+\|E_m((s+2me_j)')-E_m(s+2me_j)\|_{\ell_\infty^n(\C)}^\theta+\|E_m(s')-E_m(s)\|_{\ell_\infty^n(\C)}^\theta\nonumber \\
&\le \|E_m((s+2me_j)')-E_m(s')\|_{\ell_\infty^n(\C)}^\theta+\frac{2k^\theta}{m^\theta}\label{eq:use subgrid approx1}\\
&= \big\|\bd_{E_m((s+2me_j)')}-\bd_{E_m(s')}\big\|_{\X_\theta}+\frac{2k^\theta}{m^\theta}\label{eq:theta holder isometry}\\
&= \big\|f\big(E_m((s+2me_j)')\big)-f\big(E_m(s')\big)\big\|_{\X_\theta}+\frac{2k^\theta}{m^\theta}\label{eq:recall def f}\\
&= \big\|F\big(E_m((s+2me_j)')\big)-F\big(E_m(s')\big)\big\|_{\X_\theta}+\frac{2k^\theta}{m^\theta}\label{eq:used F extends f} \\ \nonumber
&\le \big\|F\big(E_m(s+2me_j)\big)-F\big(E_m(s)\big)\big\|_{\X_\theta}+\big\|F\big(E_m((s+2me_j)')\big)-F\big(E_m(s+2me_j)\big)\big\|_{\X_\theta}
\\&\qquad \qquad +\big\|F\big(E_m(s')\big)-F\big(E_m(s)\big)\big\|_{\X_\theta}+\frac{2k^\theta}{m^\theta}\nonumber\\
\begin{split} \label{eq:use theta holder F}
&\le \big\|F\big(E_m(s+2me_j)\big)-F\big(E_m(s)\big)\big\|_{\X_\theta}+L\|E_m((s+2me_j)')-E_m(s+2me_j)\|_{\ell_\infty^n(\C)}^\theta
\\&\qquad \qquad+L\|E_m(s')-E_m(s)\|_{\ell_\infty^n(\C)}^\theta+\frac{2k^\theta}{m^\theta}\end{split}\\
& \le \big\|F\big(E_m(s+2me_j)\big)-F\big(E_m(s)\big)\big\|_{\X_\theta}+\frac{2(L+1)k^\theta}{m^\theta}\label{eq:use subgrid approx2},
\end{align}
where for~\eqref{eq:recall Em} recall the definition of $E_m$, in~\eqref{eq:use subgrid approx1} and~\eqref{eq:use subgrid approx2} we used~\eqref{eq:sub grid approximates}, in~\eqref{eq:theta holder isometry} we used~\eqref{eq:conditions from X theta1}, for~\eqref{eq:recall def f} recall the definition of $f$, in~\eqref{eq:used F extends f} we used the fact that $F$ extends $f$ and  $\{(s+2me_j)',s'\}\subset \sub$, and in~\eqref{eq:use theta holder F} we used the fact that $F$ is $\theta$-H\"older with constant $L$. By averaging this inequality over $(j,s)$ chosen uniformly at random from  $\n \times \{1,\ldots,4m\}^n$ and applying~\eqref{eq:sup cotype}, we conclude that
\begin{equation}\label{eq:thetha constraint}
1\lesssim \bigg(\frac{m^{2}}{k^\theta n}+\frac{k^\theta}{m^\theta}\bigg)L.
\end{equation}
This holds whenever $k,m\in \N$ satisfy $k\le 2m\le n/2$ and $m\ge \pi\sqrt{n}$, so choose $m\asymp \sqrt{n}$ and $k\asymp \sqrt[4]{n}$  to minimize (up to constants) the right hand side of~\eqref{eq:thetha constraint} and  deduce the desired lower bound $L\gtrsim n^{\theta/4}$.\end{proof}

By~\cite[Lemma~6.5]{MN13-bary}, for every $\theta\in (0,1]$ and $n\in \N$ we have
\begin{equation}\label{eq:holder euclidean lower}
\ee^\theta(\ell_2^n)\gtrsim n^{\frac{\theta}{4}}.
\end{equation}
In combination with~\eqref{eq:holder from cotype} and~\cite{AM83}, this implies that there is a universal constant $c>0$ such that
\begin{equation}\label{eq:sqrt log theta}
\ee^\theta(\X)\ge e^{c\theta \sqrt{\log n}}
\end{equation}
for every $n$-dimensional normed space $\X$ and every $\theta\in (0,1]$.
\begin{conjecture}\label{conj:c theta} For any $\theta\in (0,1]$ there is $c(\theta)>0$ such that $\ee^\theta(\X)\ge \dim(\X)^{c(\theta)}$ for every normed space $\X$.
\end{conjecture}

Conjecture~\ref{conj:c theta} has a positive answer when the H\"older exponent is close enough to $1$. Specifically, if
\begin{equation}\label{eq:theta range}
0.9307777...=\frac{\sqrt{193}+1}{16}< \theta \le 1,
\end{equation}
then
\begin{equation}\label{eq:strange theta exponent}
\ee^\theta(\X)\gtrsim \frac{ n^{\frac{\theta(8\theta^2-\theta-6)}{20\theta-8}}}{(\log n)^{\frac{3\theta^2}{5\theta-2}}}.
\end{equation}
Indeed, by bi-Lipschitz invariance, \eqref{eq:holder euclidean lower} implies the following generalization of Theorem~\ref{thm:MN}.
$$
\ee^\theta(\X)\gtrsim \frac{n^{\frac{\theta}{4}}}{\dd_\X^\theta}.
$$
Also,
$$
\ee^\theta(\X)\stackrel{\eqref{eq:quote NR}}{\gtrsim} \frac{\LT(\X)^\theta}{n^{(1-\theta)\left(\theta+\frac12\right)}\dd_\X^{1-\theta}}\stackrel{\eqref{eq:cubed}}{\gtrsim}  \frac{\dd_\X^{\frac{\theta}{2}}/(\log n)^{\frac{3\theta}{2}}}{n^{(1-\theta)\left(\theta+\frac12\right)}\dd_\X^{1-\theta}}=\frac{\dd_\X^{\frac{3\theta}{2}-1}}
{n^{(1-\theta)\left(\theta+\frac12\right)}(\log n)^{\frac{3\theta}{2}}}.
$$
Therefore, in analogy to~\eqref{eq:1/12} we see that
\begin{equation}\label{eq:max theta}
\ee^\theta(\X)\gtrsim \max\left\{\frac{n^{\frac{\theta}{4}}}{\dd_\X^\theta},\frac{\dd_\X^{\frac{3\theta}{2}-1}}
{n^{(1-\theta)\left(\theta+\frac12\right)}(\log n)^{\frac{3\theta}{2}}}\right\}.
\end{equation}
Elementary calculus shows that~\eqref{eq:max theta} implies~\eqref{eq:strange theta exponent} in the range~\eqref{eq:theta range}. If $\theta$ does not satisfy~\eqref{eq:theta range}, then~\eqref{eq:max theta} does not imply a lower bound $\ee^\theta(\X)$ that depends only on $n$ and grows to $\infty$ with $n$; for such $\theta$ the best lower bound that we know  is~\eqref{eq:sqrt log theta}. The application of~\eqref{eq:cotype 2 in Xtheta} in the above proof of~\eqref{eq:holder from cotype} can be mimicked  using other bi-Lipschitz invariants to prove~\eqref{conj:c theta} for various normed spaces, such as $\ell_2^n(\ell_1^n)$ or $\sfS_1^n$, using~\cite{NS16} and~\cite{NS21}, respectively. We do not know if Conjecture~\ref{conj:c theta} holds even when, say, $\X=\ell_1^n$.

\subsection{Justification of~\eqref{eq:e lp lower best-known}}\label{rem:e ellp lower bounds with FLM}
%In \eqref{eq:e lp lower best-known} we stated the  best-known lower bounds on $\ee(\ell_p^n)$ for $(p,n)\in [1,\infty]\times \N$.
In the range $p\in [1,4/3]\cup\{2\}\cup[3,\infty]$ the bound in~\eqref{eq:e lp lower best-known} is a combination of~\cite[Corollary~8.12]{BB12} and~\cite[Theorem~1.17]{MN13-bary}. We need to justify~\eqref{eq:e lp lower best-known} in the range $p\in (4/3,3)\setminus \{2\}$ because it was not previously stated in the literature. Suppose first that $p\in (4/3,2)$. By~\cite{FLM77}, there is $k\in \{1,\ldots, n\}$ with $k\asymp n$ such that $\cc_{\ell_p^n}(\ell_2^k)\asymp 1$. Hence, $$\ee\big(\ell_p^n\big)\gtrsim \ee\big(\ell_2^k\big)\gtrsim \sqrt[4]{k}\asymp \sqrt[4]{n},$$ where the penultimate inequality follows from~\cite[Theorem~1.17]{MN13-bary}. Analogously, if $q\in (2,3)$, then by~\cite{FLM77} there is $m\in \{1,\ldots, n\}$ with $m\asymp n^{2/q}$ such that $\cc_{\ell_q^n}(\ell_2^m)\asymp 1$. We therefore have  $$\ee\big(\ell_q^n\big)\gtrsim \ee\big(\ell_2^m\big)\gtrsim \sqrt[4]{m}\asymp n^{\frac{1}{2q}}.$$
%This completes the justification of~\eqref{eq:e lp lower best-known}.

\subsection{Proof of the lower bound on $\sep(\X)$ in Theorem~\ref{thm:sep bounds in overview}}\label{sec:proof of vr lower} Thanks to~\eqref{eq:state bourgain milman}, the first part of Theorem~\ref{thm:sep lower with cardinality size} below  coincides with the lower bound on $\sep(\X)$ in Theorem~\ref{thm:sep bounds in overview}, except that in~\eqref{eq:sep lower evr} below we also specify the constant factor that our proof provides (there is no reason to expect that this constant is optimal; due to the fundamental nature of this randomized clustering problem it would be interesting to find the optimal constant here). The second part of Theorem~\ref{thm:sep lower with cardinality size}  relates to dimension reduction by controlling the cardinality of a finite subset $\sub$ of $\X$ on which the lower bound is attained. We conjecture that the first part of~\eqref{eq:sep lower dim reduction} below could be improved to $|\sub|^{1/n}=O(1)$; an inspection of the ensuing proof suggests that a possible route towards this improved bound is to incorporate a proportional Dvoretzky--Rogers factorization~\cite{BS88,ST89,Gia96} in place of our  use of the ``vanilla'' Dvoretzky--Rogers lemma~\cite{DR50}.

\begin{theorem}\label{thm:sep lower with cardinality size} For every $n\in \N$, any  $n$-dimensional normed space $(\X,\|\cdot\|_\X)$ satisfies
\begin{equation}\label{eq:sep lower evr}
\sep(\X)\ge \evr(\X) \frac{2(n!)^{\frac{1}{2n}}\Gamma\left(1+\frac{n}{2}\right)^{\frac{1}{n}}}{\sqrt{\pi n }}=\frac{\sqrt{2}+o(1)}{e\sqrt{\pi}} \evr(\X)\sqrt{n}.
\end{equation}
Furthermore, there exists a finite subset  $\sub$ of $\X$ satisfying
\begin{equation}\label{eq:sep lower dim reduction}
|\sub|^{\frac{1}{n}}\lesssim \frac{\sqrt{n}}{\evr(\X)} \qquad\mathrm{and}\qquad \sep(\sub_\X)\gtrsim \evr(\X)\sqrt{n}.
\end{equation}
\end{theorem}

Our proof of Theorem~\ref{thm:sep lower with cardinality size} builds upon the strategy that was used in~\cite{CCGGP98} to treat $\ell_1^n$. A  combinatorial fact on which it relies is Lemma~\ref{lem:random partition combinatoirial} below,  which is implicit in the {proof of}~\cite[Lemma~3.1]{CCGGP98}.  After proving Theorem~\ref{thm:sep lower with cardinality size} while using Lemma~\ref{lem:random partition combinatoirial}, we will present a proof of Lemma~\ref{lem:random partition combinatoirial} which is a quick application of the Loomis--Whitney inequality~\cite{LW49};  the proof in~\cite{CCGGP98} uses a result of~\cite{AKPW91} which is proved in~\cite{AKPW91} via information-theoretic reasoning through the use of Shearer's inequality~\cite{CGFS86}; the relation between the Loomis--Whitney inequality and Shearer's inequality is well-known (see e.g.~\cite{BB12}), so our proof of  Lemma~\ref{lem:random partition combinatoirial} is in essence a repackaging of the classical ideas.

\begin{lemma}\label{lem:random partition combinatoirial} Fix $n,M\in \N$ and a nonempty finite subset $\Omega$ of $\Z^n$. Suppose that $\Part$ is a random partition of $\Omega$ that is supported on partitions into subsets of cardinality at most $M$, i.e.,
$$\Pr\Big[\max_{\Gamma\in \Part}|\Gamma|\le M\Big]=1.$$
Then, there exists $i\in \n$ and $x\in \Omega\cap (\Omega-e_i)$ for which
\begin{equation}\label{eq:combinatorial separation}
\Pr \big[\Part(x)\neq \Part(x+e_i)\big]\ge \frac{1}{\sqrt[n]{M}}-\frac{1}{n}\sum_{i=1}^n\frac{|\Omega\setminus(\Omega-e_i)|}{|\Omega|}.
\end{equation}
\end{lemma}

\begin{proof}[Proof of Theorem~\ref{thm:sep lower with cardinality size} assuming Lemma~\ref{lem:random partition combinatoirial}] By suitably choosing the identification of $\X$ with $\R^n$, we may assume without loss of generality that $\X=(\R^n,\|\cdot\|_\X)$  and  $B_{\ell_2^n}$ is the L\"owner ellipsoid of $B_\X$. Then,
\begin{equation}\label{eq:write evr}
\evr(\X)=\left(\frac{\vol_n(B_{\ell_2^n})}{\vol_n(B_\X)}\right)^{\frac{1}{n}}=\frac{\sqrt{\pi}}{ \Gamma\left(1+\frac{n}{2}\right)^{\frac{1}{n}}\vol_n(B_\X)^{\frac{1}{n}}}.
\end{equation}
By the Dvoretzky--Rogers lemma~\cite{DR50}, there exist contact points $x_1,\ldots, x_n\in S^{n-1}\cap \partial B_{\X}$ that satisfy
\begin{equation}\label{eq:quote DR}
\forall k\in \{1,\ldots, n\},\qquad \big\|\proj_{\spn(x_1,\ldots,x_{k-1})^\perp}(x_k)\big\|_{\ell_2^n}\ge \sqrt{\frac{n-k+1}{n}}.
\end{equation}

Let $\Lambda=\Lambda(x_1,\ldots,x_n)\subset \R^n$ denote the lattice that is generated by $x_1,\ldots,x_n$, namely
$$
\Lambda=\sum_{i=1}^n \Z x_i=\Big\{\sum_{i=1}^nk_i x_i:\ k_1,\ldots,k_n\in \Z\Big\}.
$$
By~\eqref{eq:quote DR}, $\Lambda$ is full-rank. Denote the fundamental parallelepiped of $\Lambda$ by $Q=Q(x_1,\ldots,x_n)$, i.e.,
$$
Q= \sum_{i=1}^n [0,1)x_i=\Big\{\sum_{i=1}^n s_ix_i:\ 0\le s_1,\ldots,s_n<1\Big\}.
$$
Since $x_1,\ldots,x_n\in B_\X$, we have $Q-Q\subset nB_\X$ and by~\eqref{eq:quote DR} the volume of $Q$ (the determinant of $\Lambda$) satisfies
\begin{equation}\label{eq:volume of Q}
\det(\Lambda)=\vol_n(Q)=\prod_{k=1}^n\big\|\proj_{\spn\{x_1,\ldots,x_{k-1}\}^\perp}(x_k)\big\|_{\ell_2^n} \stackrel{\eqref{eq:quote DR} }{\ge} \prod_{k=1}^n \sqrt{\frac{n-k+1}{n}}=\frac{\sqrt{n!}}{n^{\frac{n}{2}}}.
\end{equation}

Fix $m\in \N$ and $\sigma,\Delta>0$. Denote $\sub_m=\sub_m(x_1,\ldots,x_n)=\Lambda\cap (mQ)=\{\sum_{i=1}^n k_i x_i:\ k_1,\ldots,k_n\in \{0,\ldots,m-1\}\}$ and suppose that $\Part$ is $\sigma$-separating $\Delta$-bounded random partition of $\sub_m$. The $\Delta$-boundedness of $\Part$ means that $\Gamma-\Gamma\subset \Delta B_\X$ for every $\Gamma\subset \sub_m$ with $\Pr[\Gamma\in \Part]>0$. Recalling that $Q-Q\subset nB_\X$, this implies that
\begin{equation}\label{eq:BX contains difference body}
B_\X\supseteq \frac{1}{\Delta+n}\big( (\Gamma+Q)-(\Gamma+Q)\big).
\end{equation}
Now,
\begin{equation}\label{eq:Gamma cardinality bound}
\frac{\sqrt{\pi}}{ \Gamma\left(1+\frac{n}{2}\right)^{\frac{1}{n}}\evr(\X)} =\vol_n(B_\X)^{\frac{1}{n}}\ge \frac{2}{\Delta+n}\vol_n(\Gamma+Q)^{\frac{1}{n}}=\frac{2}{\Delta+n}\big(|\Gamma|\vol_n(Q)\big)^{\frac{1}{n}}\ge \frac{2(n!)^{\frac{1}{2n}}}{(\Delta+n)\sqrt{n}}|\Gamma|^{\frac{1}{n}},
\end{equation}
where the first step of~\eqref{eq:Gamma cardinality bound} is~\eqref{eq:write evr}, the second step of~\eqref{eq:Gamma cardinality bound} uses~\eqref{eq:BX contains difference body} and the Brunn--Minkowski inequality, the third step of~\eqref{eq:Gamma cardinality bound} holds because the parallelepipeds $\{\gamma+Q:\ \gamma\in \Gamma\}$ are disjoint, and the final step of~\eqref{eq:Gamma cardinality bound} is~\eqref{eq:volume of Q}. If $T\in \GL_n(\R)$ is given by $Te_i=x_i$, then it follows from~\eqref{eq:Gamma cardinality bound} that the random partition $T^{-1}\Part=\{T^{-1}\Gamma:\ \Gamma\in \Part\}$ of $T^{-1}\sub_m=\{0,\ldots,m-1\}^n$ satisfies the assumptions of Lemma~\ref{lem:random partition combinatoirial} with
$$
M=\frac{(\pi n)^{\frac{n}{2}}(\Delta+n)^n}{2^n\Gamma\left(1+\frac{n}{2}\right)\sqrt{n!}}\cdot\frac{1}{\evr(\X)^n}.
$$
If we choose $\Omega=\{0,\ldots,m-1\}^n=T^{-1}\sub_m$ in Lemma~\ref{lem:random partition combinatoirial}, then  $|\Omega|=m^n$ and $|\Omega\setminus (\Omega-e_i)|=m^{n-1}$ for every $i\in \n$, so it follows from Lemma~\ref{lem:random partition combinatoirial} that there exist $i\in \n$ and $x\in \sub_m$ such that
\begin{equation}\label{eq:for sigma lower}
\Pr \big[\Part(x)\neq \Part(x+e_i)\big]\ge \evr(\X)\frac{2(n!)^{\frac{1}{2n}}\Gamma\left(1+\frac{n}{2}\right)^{\frac{1}{n}}}{(\Delta+n)\sqrt{\pi n}} -\frac{1}{m}.
\end{equation}
At the same time, the left hand side of~\eqref{eq:for sigma lower} is at most $\sigma/\Delta$,  since $\Part$ is $\sigma$-separating and $\|x_i\|_\X\le 1$. Thus,
\begin{equation}\label{eq:before m limit}
\sigma\ge \evr(\X)\frac{2\Delta(n!)^{\frac{1}{2n}}\Gamma\left(1+\frac{n}{2}\right)^{\frac{1}{n}}}{(\Delta+n)\sqrt{\pi n}} -\frac{\Delta}{m}.
\end{equation}

By letting $m\to \infty$ in~\eqref{eq:before m limit} and then letting $\Delta\to \infty$ in the resulting estimate, we get~\eqref{eq:sep lower evr}. Also, if we set $\Delta=n$ in~\eqref{eq:before m limit}, then for sufficiently large  $m\asymp \sqrt{n}/\evr(\X)$ we have $\sep(\sub_m)\gtrsim   \evr(\X)\sqrt{n}$, giving~\eqref{eq:sep lower dim reduction}.
\end{proof}

We will next provide a proof of Lemma~\ref{lem:random partition combinatoirial} whose  main ingredient is the following lemma.

\begin{lemma}[application of Loomis--Whitney]\label{lem:Use LW} Fix an integer $n\ge 2$ and a finite subset $\Gamma$ of $\Z^n$. For  $x\in \Z^n$ and $i\in \n$, let $d_i(x;\Gamma)\in \N\cup\{0\}$ be the number of times that the oriented discrete axis-parallel line $x+\Z e_i$ transitions from $\Gamma$ to  $\Z^n\setminus \Gamma$, and let $\g(x;\Gamma)$ be the geometric mean of $d_1(x;\Gamma),\ldots,d_n(x;\Gamma)$. Thus
$$
\forall i\in \n,\qquad d_i(x;\Gamma)\eqdef \big|\{k\in \Z:\ x+ke_i\in \Gamma\ \wedge\ x+(k+1)e_i\notin \Gamma\}\big|,
$$
and
$$
g(x;\Gamma)\eqdef \sqrt[n]{d_1(x;\Gamma)\cdots d_n(x;\Gamma)}.
$$
Then,
\begin{equation}\label{eq:loomis whitney to use}
\frac{1}{n}\sum_{i=1}^n|\Gamma\setminus(\Gamma-e_i)|\ge \bigg(\sum_{x\in \Z^n} g(x;\Gamma)^{\frac{n}{n-1}}\bigg)^{\frac{n-1}{n}}\ge |\Gamma|^{\frac{n-1}{n}}.
\end{equation}
\end{lemma}

\begin{proof} The second inequality in~\eqref{eq:loomis whitney to use} holds because  $d_1(x;\Gamma),\ldots,d_n(x;\Gamma)\ge 1$ for every $x\in \Gamma$ (as $|\Gamma|<\infty$), and hence $g(\cdot;\Gamma)\ge \1_\Gamma(\cdot)$ point-wise. For the first inequality in~\eqref{eq:loomis whitney to use}, observe that for each $i\in \n$,
$$
|\Gamma\setminus(\Gamma-e_i)|=\sum_{x\in \Z^n} \1_\Gamma(x)\1_{\Z^n\setminus \Gamma}(x+e_i)=\sum_{y\in  \proj_{e_i^\perp}\Gamma}\bigg(\sum_{k\in \Z} \1_\Gamma(y+ke_i) \1_{\Z^n\setminus \Gamma}\big(y+(k+1)e_i\big)\bigg)=\!\!\!\sum_{y\in  \proj_{e_i^\perp}\Z^n} d_i(y;\Gamma).
$$
Consequently,
$$
\frac{1}{n}\sum_{i=1}^n|\Gamma\setminus(\Gamma-e_i)|=\frac{1}{n}\sum_{i=1}^n \Big\|d_i(\cdot;\Gamma)^{\frac{1}{n-1}}\Big\|_{\ell_{n-1}(\proj_{e_i^\perp}\Z^n)}^{n-1}\ge  \prod_{i=1}^n \Big\|d_i(\cdot;\Gamma)^{\frac{1}{n-1}}\Big\|_{\ell_{n-1}(\proj_{e_i^\perp}\Z^n)}^{\frac{n-1}{n}}\ge \sum_{x\in \Z^n} \prod_{i=1}^n d_i(\proj_{e_i^\perp} x)^{\frac{1}{n-1}},
$$
where the second step is an application of the arithmetic-mean/geometric-mean inequality and the final step is an application of the Loomis--Whitney inequality~\cite{LW49} (see~\cite[Theorem~3]{Sil73} for the functional version of the Loomis--Whitney inequality that the are using here); we note that even though this inequality is commonly stated for functions on $\R^n$ rather than for functions on $\Z^n$,  its proof for functions on $\Z^n$ is identical (in fact, \cite{LW49} proves the continuous inequality by first proving its discrete counterpart).
\end{proof}

Note that when  $n=1$ Lemma~\ref{lem:Use LW} holds trivially if we interpret~\eqref{eq:loomis whitney to use} as $|\Gamma\setminus (\Gamma-1)|\ge \max_{x\in \Z} g(x;\Gamma)\ge 1$, since in this case $g(x;\Gamma)=|\Gamma\setminus (\Gamma-1)|$ for every $x\in \Z$.

The following corollary of Lemma~\ref{lem:Use LW} is a deterministic counterpart of Lemma~\ref{lem:random partition combinatoirial}.

\begin{corollary}\label{cor:deterministic partition} Fix $n,M\in \N$ and a nonempty finite subset $\Omega$ of $\Z^n$. Suppose that $\Part$ is a partition of $\Omega$ with
\begin{equation}\label{eq:cardinality assumption M}
\max_{\Gamma\in \Part}|\Gamma|\le M.
\end{equation}
Then,
\begin{equation}\label{eq:small cardinality partition inequality}
\frac{1}{n}\sum_{i=1}^n|\{x\in \Omega\cap (\Omega-e_i):\  \Part(x)\neq \Part(x+e_i)\}|\ge \frac{|\Omega|}{\sqrt[n]{M}}-\frac{1}{n}\sum_{i=1}^n|\Omega\setminus(\Omega-e_i)|
\end{equation}
\end{corollary}

\begin{proof} Observe that for each fixed $i\in \n$ we have
\begin{align}\label{eq:add i boundary back}
\begin{split}
|\Omega\setminus(\Omega-e_i)|+\sum_{x\in \Omega\cap (\Omega-e_i)} \1_{\Part(x)\neq \Part(x+e_j)}&=|\Omega\setminus(\Omega-e_i)|+\sum_{x\in \Omega\cap (\Omega-e_i)} \bigg(\sum_{\Gamma\in \Part} \1_\Gamma(x)\1_{\Z^n\setminus \Gamma}(x+e_i)\bigg)\\&= \sum_{x\in \Z^n} \sum_{\Gamma\in \Part} \1_\Gamma(x)\1_{\Z^n\setminus \Gamma}(x+e_i)\\&=\sum_{\Gamma\in \Part}|\Gamma\setminus(\Gamma-e_i)|,
\end{split}
\end{align}
where the first step of~\eqref{eq:add i boundary back} holds because $\Part$ is a partition of $\Omega$ and  the second step of~\eqref{eq:add i boundary back} holds because $\1_\Gamma(x)\1_{\Z^n\setminus \Gamma}(x+e_i)=0$ for every $\Gamma\subset \Omega$ if $x\in \Z^n\setminus \Omega$, and if $x\in \Omega\setminus(\Omega-e_i)$, then  $\1_\Gamma(x)\1_{\Z^n\setminus \Gamma}(x+e_i)=1$ for exactly one $\Gamma\in \Part$ (specifically, this holds for $\Gamma=\Part(x)$ because   $x+e_i\in \Z^n\setminus \Omega\subset \Z^n\setminus \Part(x)$). Now,
\begin{multline*}
\frac{1}{n}\sum_{i=1}^n|\{x\in \Omega\cap (\Omega-e_i):\  \Part(x)\neq \Part(x+e_i)\}|+\frac{1}{n}\sum_{i=1}^n|\Omega\setminus(\Omega-e_i)|\\\stackrel{\eqref{eq:add i boundary back}}{=}\sum_{\Gamma\in \Part}\frac{1}{n}\sum_{i=1}^n|\Gamma\setminus(\Gamma-e_i)|\stackrel{\eqref{eq:loomis whitney to use}}{\ge} \sum_{\Gamma\in \Part}|\Gamma|^{\frac{n-1}{n}}\stackrel{\eqref{eq:cardinality assumption M}}{\ge} \frac{1}{\sqrt[n]{M}}\sum_{\Gamma\in \Part}|\Gamma|=\frac{|\Omega|}{\sqrt[n]{M}},
\end{multline*}
where the last step holds because $\Part$ is a partition of $\Omega$.
\end{proof}

\begin{proof}[Proof of Lemma~\ref{lem:random partition combinatoirial}] Denoting  $p=\max_{i\in \n}\max_{x\in \Omega\cap (\Omega-e_i)} \Pr[\Part(x)\neq \Part(x+e_i)]$, the goal is to show that $p$ is at least the right hand side of~\eqref{eq:combinatorial separation}. This follows from Corollary~\ref{cor:deterministic partition} because
\begin{multline*}
p|\Omega|\ge \frac{p}{n}\sum_{i=1}^n|\Omega\cap (\Omega-e_i)|\ge  \frac{1}{n}\sum_{i=1}^n\sum_{x\in \Omega\cap (\Omega-e_i)}\Pr\big[\Part(x)\neq \Part(x+e_i)\big]=\frac{1}{n}\sum_{i=1}^n \sum_{x\in \Omega\cap (\Omega-e_i)} \E\big[\1_{\Part(x)\neq \Part(x+e_i)}\big]\\=\E\bigg[\frac{1}{n}\sum_{i=1}^n|\{x\in \Omega\cap (\Omega-e_i):\  \Part(x)\neq \Part(x+e_i)\}|\bigg]\stackrel{\eqref{eq:small cardinality partition inequality}}{\ge} \frac{|\Omega|}{\sqrt[n]{M}}-\frac{1}{n}\sum_{i=1}^n|\Omega\setminus(\Omega-e_i)|.\tag*{\qedhere}
\end{multline*}
\end{proof}

\subsection{Proof of the lower bound on $\pad_\d(\X)$ in Theorem~\ref{thm:pad sharp}}\label{sec:pad lower}
Fixing $n\in \N$, a normed space $\X=(\R^n,\|\cdot\|_{\X})$, and $\d\in (0,1)$, recalling the notation in Definition~\ref{def:padded finite} we will prove here  that
\begin{equation}\label{eq:pad lower finatary}
\pad_\d(\X)\ge \sup_{m\in \N} \pad^m_\d(\X)\ge \frac{2}{1-\sqrt[n]{\d}},
\end{equation}
which gives the first inequality in~\eqref{eq:padding nth root}.

\begin{proof}[Proof of~\eqref{eq:pad lower finatary}] Suppose that $0<\e<1$ and $r>2$.  Let $\mathcal{N}_\e$ be any $\e$-net of $rB_{\X}$. Then, $\log |\mathcal{N}_\e|\asymp n\log (r/\e)$ (see e.g.~\cite[Lemma~9.18]{Ost13}). Fix a (disjoint) Voronoi tessellation $\{V_x\}_{x\in \mathcal{N}_\e}$  of $rB_{\X}$ that is induced by $\mathcal{N}_\e$. Thus, $\{V_x\}_{x\in \mathcal{N}_\e}$ is a partition of $rB_{\X}$ into Borel subsets such that $x\in V_x\subset x+\e B_{\X}$ for every $x\in  \mathcal{N}_\e$. So, for every $w\in rB_{\X}$ there is a unique net point $\cx(w)\in \mathcal{N}_\e$ such that $w\in V_{\cx(w)}$.

Fix ${\mathfrak{p}}>\sup_{m\in \N} \pad^m_\d(\X)\ge\pad_\d(\mathcal{N}_\e)$, and assume from now on that $0<\e<1/(2{\mathfrak{p}})$ and $r>1/{\mathfrak{p}}-2\e$ (eventually we will consider the limits $\e\to 0$ and $r\to \infty$). By the definition of $\pad_\d(\mathcal{N}_\e)$, there exists a probability distribution $\Part$ over $1$-bounded partitions of $\mathcal{N}_\e$ such that
\begin{equation}\label{eq:padding assumption}
\forall  y\in \mathcal{N}_\e,\qquad \Pr\Big[ \Big(y+\frac{1}{{\mathfrak{p}}}B_{\X}\Big)\cap \mathcal{N}_\e \subset\mathscr{P}(y)\Big]\ge \d.
\end{equation}
For every $y\in \mathcal{N}_\e$ define
$$
\Part^*(y)\eqdef \bigcup_{z\in \Part(y)} V_z=\big\{w\in rB_{\X}:\ \cx(w)\in \Part(y)\big\}.
$$
Then $\{\Part^*(y)\}_{y\in \mathcal{N}_\e}$ is a (finitely supported) random partition of $rB_{\X}$ into Borel subsets.

 We claim that for every $y\in \mathcal{N}_\e$ the following inclusion of events holds.
\begin{equation}\label{eq:inclusion for BM}
\left\{w\in \R^n:\ w+\frac{1-2\e{\mathfrak{p}}}{{\mathfrak{p}}}B_{\X}\subset \Part^*(y)\right\}+\frac{1-2\e{\mathfrak{p}}}{(1+2\e){\mathfrak{p}}}\big(\Part^*(y)-\Part^*(y)\big)\subset \Part^*(y).
\end{equation}
Indeed, take any $w\in \R^n$ such that $$w+\frac{1-2\e{\mathfrak{p}}}{{\mathfrak{p}}}B_{\X}\subset \Part^*(y),$$ and also take any $u,v\in \Part^*(y)$. By the definition of $\Part^*$ we have $\cx(u),\cx(v)\in \Part(y)$. As $\Part$ is $1$-bounded, we have  $\|\cx(u)-\cx(v)\|_{\X}\le 1$. Therefore $\|u-v\|_{\X}\le \|u-\cx(u)\|_{\X}+\|\cx(u)-\cx(v)\|_{\X}+\|v-\cx(v)\|_{\X}\le 1+2\e$. Hence, $$\frac{1-2\e{\mathfrak{p}}}{(1+2\e){\mathfrak{p}}}(u-v)\in \frac{1-2\e{\mathfrak{p}}}{{\mathfrak{p}}}B_{\X},$$ so the assumption on $w$ implies that $$w+\frac{1-2\e{\mathfrak{p}}}{(1+2\e){\mathfrak{p}}}(u-v)\in \Part^*(y).$$ This is precisely the assertion in~\eqref{eq:inclusion for BM}. By the Brunn--Minkowski inequality, \eqref{eq:inclusion for BM} gives
%$$
%\sqrt[n]{\vol_n\big(\Part^*(y)\big)}\ge 2\frac{1-2\e{\mathfrak{p}}}{(1+2\e){\mathfrak{p}}}\sqrt[n]{\vol_n\big(\Part^*(y)\big)}+\sqrt[n]{\vol_n\Big(\Big\{w\in \R^n:\ %w+\frac{1-2\e{\mathfrak{p}}}{{\mathfrak{p}}}B_{\X}\subset \Part^*(y)\Big\}\Big)}.
%$$
$$
\vol_n\big(\Part^*(y)\big)^{\frac{1}{n}}\ge 2\frac{1-2\e{\mathfrak{p}}}{(1+2\e){\mathfrak{p}}}\vol_n\big(\Part^*(y)\big)^{\frac{1}{n}}+\vol_n\Big(\Big\{w\in \R^n:\ w+\frac{1-2\e{\mathfrak{p}}}{{\mathfrak{p}}}B_{\X}\subset \Part^*(y)\Big\}\Big)^{\frac{1}{n}}.
$$
This simplifies to give the following estimate.
\begin{equation}\label{eq:squeezed volume}
\vol_n\Big(\Big\{w\in \R^n:\ w+\frac{1-2\e{\mathfrak{p}}}{{\mathfrak{p}}}B_{\X}\subset \Part^*(y)\Big\}\Big)\le \left(1-2\frac{1-2\e{\mathfrak{p}}}{(1+2\e){\mathfrak{p}}}\right)^n\vol_n\big(\Part^*(y)\big).
\end{equation}

Now,
\begin{align}
\vol_n&\Big(\Big\{w\in rB_{\X}:\ w+\frac{1-2\e{\mathfrak{p}}}{{\mathfrak{p}}}B_{\X}\subset \Part^*\big(\cx(w)\big)\Big\}\Big)\nonumber \\&=\sum_{y\in \mathcal{N}_\e}\vol_n\Big(\Big\{w\in \Part^*(y):\ w+\frac{1-2\e{\mathfrak{p}}}{{\mathfrak{p}}}B_{\X}\subset \Part^*\big(\cx(w)\big)\Big\}\Big)\label{eq:use union P star}\\
&=\sum_{y\in \mathcal{N}_\e}\vol_n\Big(\Big\{w\in \Part^*(y):\ w+\frac{1-2\e{\mathfrak{p}}}{{\mathfrak{p}}}B_{\X}\subset \Part^*(y)\Big\}\Big)\label{eq:y is x w}\\
&\le \left(1-2\frac{1-2\e{\mathfrak{p}}}{(1+2\e){\mathfrak{p}}}\right)^n\sum_{y\in \mathcal{N}_\e}\vol_n\big(\Part^*(y)\big)\label{eq:use squeez}\\
&=\left(1-2\frac{1-2\e{\mathfrak{p}}}{(1+2\e){\mathfrak{p}}}\right)^nr^n\vol_n(B_{\X}).\label{eq: use p star partition}
\end{align}
Here~\eqref{eq:use union P star} holds because  $\{\Part^*(y)\}_{y\in \mathcal{N}_\e}$ is a partition of $rB_{\X}$. The identity~\eqref{eq:y is x w} holds because, since by the definition of $\Part^*$ we have $w\in \Part^*(\cx(w))$ for every $w\in rB_{\X}$ and the sets $\{\Part^*(y)\}_{y\in \mathcal{N}_\e}$ are pairwise disjoint, if $w\in \Part^*(y)$ for some $y\in \mathcal{N}_\e$ then necessarily $\Part^*(\cx(w))=\Part^*(y)$. The estimate~\eqref{eq:use squeez} uses~\eqref{eq:squeezed volume}. The identity~\eqref{eq: use p star partition}  uses once more that  $\{\Part^*(y)\}_{y\in \mathcal{N}_\e}$ is a partition of $rB_{\X}$.

We next claim that for every $w\in (r+2\e-1/\p)B_{\X}$ the following inclusion of events holds.
\begin{equation}\label{eq:triangle inclusion}
 \bigg\{\Big(\cx(w)+\frac{1}{{\mathfrak{p}}}B_{\X}\Big)\cap \mathcal{N}_\e \subset\mathscr{P}\big(\cx(w)\big)\bigg\}\subset \bigg\{w+\frac{1-2\e{\mathfrak{p}}}{{\mathfrak{p}}}B_{\X}\subset \Part^*\big(\cx(w)\big)\bigg\}.
\end{equation}
Indeed, suppose that $w\in \X$ satisfies $\|w\|_{\X}\le r+2\e-1/{\mathfrak{p}}$ and $(\cx(w)+(1/\p)B_{\X})\cap \mathcal{N}_\e \subset\mathscr{P}(\cx(w))$. Fix any $z \in \X$ such that $\|w-z \|_{\X}\le (1-2\e\p)/\p$. Then we have $\|z \|_{\X}\le \|w\|_{\X}+\|w-z \|_{\X}\le r$, so $z \in rB_{\X}$ and therefore $\cx(z )\in \mathcal{N}_\e$ is well-defined. Now, $$\|\cx(w)-\cx(z )\|_{\X}\le \|\cx(w)-w\|_{\X}+\|w-z \|_{\X}+\|z -\cx(z )\|_{\X}\le \e+\frac{1-2\e{\mathfrak{p}}}{\p}+\e=\frac{1}{\p}.$$
Hence, our assumption on $w$ implies that $\cx(z )\in \Part(\cx(w))$. By the definition of $\Part^*(\cx(w))$, this means that $z \in \Part^*(\cx(w))$, thus completing the verification of~\eqref{eq:triangle inclusion}.  Due to~\eqref{eq:padding assumption} and~\eqref{eq:triangle inclusion} we conclude that
\begin{equation}\label{eq:padding for voronoi}
\forall  w\in \Big(r+2\e-\frac{1}{\mathfrak{p}}\Big)B_{\X},\qquad \Pr\Big[w+\frac{1-2\e{\mathfrak{p}}}{{\mathfrak{p}}}B_{\X}\subset \Part^*\big(\cx(w)\big)\Big]\ge \d.
\end{equation}

Finally,
\begin{eqnarray*}
\d\Big(r+2\e-\frac{1}{\mathfrak{p}}\Big)^n\vol_n(B_{\X})&\stackrel{\eqref{eq:padding for voronoi}}{\le} & \int_{\left(r+2\e-\frac{1}{\mathfrak{p}}\right)B_{\X}} \Pr\Big[w+\frac{1-2\e{\mathfrak{p}}}{{\mathfrak{p}}}B_{\X}\subset \Part^*\big(\cx(w)\big)\Big]\ud w\\
& =&\E\Big[\vol_n\Big(\Big\{w\in \Big(r+2\e-\frac{1}{\mathfrak{p}}\Big)B_{\X}:\  w+\frac{1-2\e{\mathfrak{p}}}{{\mathfrak{p}}}B_{\X}\subset \Part^*\big(\cx(w)\big)\Big\}\Big)\Big]\\&\stackrel{\eqref{eq: use p star partition}}{\le} & \left(1-2\frac{1-2\e{\mathfrak{p}}}{(1+2\e){\mathfrak{p}}}\right)^nr^n\vol_n(B_{\X}).
\end{eqnarray*}
This simplifies to give the estimate
$$
\sqrt[n]{\d}\left(1-\frac{1}{\p r}+\frac{2\e}{r}\right)\le 1-2\frac{1-2\e{\mathfrak{p}}}{(1+2\e){\mathfrak{p}}}.
$$
By letting $r\to \infty$, then $\e\to 0$, and then ${\mathfrak{p}}\to\sup_{m\in \N} \pad^m_\d(\X)$, the desired bound~\eqref{eq:pad lower finatary} follows.\end{proof}

\subsection{Proof of Proposition~\ref{prop:large cone}}\label{sec:erwin}

The final lower bound from the Introduction that remains to be proven is Proposition~\ref{prop:large cone}. The ensuing reasoning is a restructuring of a proof that was shown to us by Lutwak.

\begin{lemma}\label{lem:cone inequality} Every origin-symmetric convex body $K\subset \R^n$ satisfies
\begin{equation}\label{eq:cleverly averaged cone}
\int_{S^{n-1}} \frac{\vol_{n-1}\big(\proj_{u^\perp}(K)\big)}{\|u\|_K^{n+1}}\ud u\ge \frac{n^2\Gamma\left(\frac{n}{2}\right)}{2\sqrt{\pi}\Gamma\left(\frac{n+1}{2}\right)}\vol_n(K)^2.
\end{equation}
Equality in~\eqref{eq:cleverly averaged cone} holds if and only if $K$ is an ellipsoid.
\end{lemma}
Before proving Lemma~\ref{lem:cone inequality}, we will explain how it implies  Proposition~\ref{prop:large cone}.

\begin{proof}[Proof of Proposition~\ref{prop:large cone} assuming Lemma~\ref{lem:cone inequality}] The following standard identity follows from integration in polar coordinates (its quick derivation can be found, for example, on page~91 of~\cite{Pis89}).
\begin{equation}\label{eq:volume identity}
\vol_n(K)=\frac{1}{n}\int_{S^{n-1}}\frac{\ud u}{\|u\|_K^n}.
\end{equation}
Hence,
\begin{align}\label{eq:pass to max cone}
\begin{split}
\int_{S^{n-1}} &\frac{\vol_{n-1}\big(\proj_{u^\perp}(K)\big)}{\|u\|_K^{n+1}}\ud u\le \bigg(\int_{S^{n-1}} \frac{\ud u}{\|u\|_K^n}\bigg)\max_{u\in S^{n-1}} \frac{\vol_{n-1}\big(\proj_{u^\perp}(K)\big)}{\|u\|_K}\\&\stackrel{\eqref{eq:volume identity}}{=} n\vol_n(K)\max_{z\in \partial K} \Big(\|z\|_{\ell_2^n} \vol_{n-1}\big(\proj_{z^\perp}(K)\big)\Big)=n^2 \vol_n(K)\max_{z\in \partial K} \vol_n\big(\cone_z(K)\big).
\end{split}
\end{align}
The desired inequality~\eqref{eq:sharp cone} follows by contrasting~\eqref{eq:pass to max cone} with~\eqref{eq:cleverly averaged cone}. Consequently, if there is equality in~\eqref{eq:sharp cone}, then~\eqref{eq:cleverly averaged cone} must hold as equality as well, so the characterization of the equality case in Proposition~\ref{prop:large cone} follows from the characterization of the quality case in Lemma~\ref{lem:cone inequality}.
\end{proof}

The important {\em Petty projection inequality}~\cite{Pet71} (see also~\cite{Sch95,MM96} for different proofs, as well as the survey~\cite{Lut93}) states that for every convex body $K\subset \R^n$, the affine invariant quantity
\begin{equation}\label{eq:petty product}
\vol_n(K)^{n-1}\vol_n(\Pi^{*}K)
\end{equation}
 is maximized when $K$ is an ellipsoid, and  ellipsoids are the only maximizers of~\eqref{eq:petty product}. Recall that the polar projection body $\Pi^*K$ is given by~\eqref{eq:use cauchy}, which shows in particular that  $\vol_{n-1}(B_{\ell_2^{n-1}})\Pi^{ *}\!B_{\ell_2^n}=B_{\ell_2^n}$. Hence,
$$
\vol_n(K)^{n-1}\vol_n(\Pi^{*}K)\le \vol_n(B_{\ell_2^n})^{n-1}\vol_n(\Pi^{*}B_{\ell_2^n})=\left(\frac{\vol_n\big(B_{\ell_2^n}\big)}{\vol_{n-1}\big(B_{\ell_2^{n-1}}\big)}\right)^{n}
=\left( \frac{2\sqrt{\pi}\Gamma\left(\frac{n+1}{2}\right)}{n\Gamma\left(\frac{n}{2}\right)}\right)^n.
$$
At the same time, by combining~\eqref{eq:use cauchy} and~\eqref{eq:volume identity} we have
$$
\vol_n(\Pi^*K)=\frac{1}{n}\int_{S^{n-1}}\frac{\ud u}{\vol_{n-1}\big(\proj_{u^\perp}(K)\big)^n}.
$$
Consequently, Petty's projection inequality can be restated as the following estimate,
\begin{equation}\label{eq:petty in lemma}
\int_{S^{n-1}}\frac{\ud u}{\vol_{n-1}\big(\proj_{u^\perp}(K)\big)^n}\le \left( \frac{2\sqrt{\pi}\Gamma\left(\frac{n+1}{2}\right)}{n\Gamma\left(\frac{n}{2}\right)}\right)^n\frac{n}{\vol_n(K)^{n-1}},
\end{equation}
together with the assertion that~\eqref{eq:petty in lemma} holds as an equality if and only if $K$ is an ellipsoid.

\begin{proof}[Proof of Lemma~\ref{lem:cone inequality}] Observe that
\begin{align}
\vol_n(K)&=\frac{1}{n}\int_{S^{n-1}}  \Bigg(\frac{1}{\vol_{n-1}\big(\proj_{u^\perp}(K)\big)^{\frac{n}{n+1}}}\Bigg) \Bigg(\frac{\vol_{n-1}\big(\proj_{u^\perp}(K)\big)^{\frac{n}{n+1}}}{\|u\|_K^{n}} \Bigg)\ud u \label{eq:use -n volume formula} \\ &\le \frac{1}{n}\bigg(\int_{S^{n-1}}\frac{\ud u}{\vol_{n-1}\big(\proj_{u^\perp}(K)\big)^n} \bigg)^{\frac{1}{n+1}}\bigg(\int_{S^{n-1}} \frac{\vol_{n-1}\big(\proj_{u^\perp}(K)\big)}{\|u\|_K^{n+1}}\ud u\bigg)^{\frac{n}{n+1}} \label{eq:holder n+1}\\&\le
\frac{1}{n}\left( \frac{2\sqrt{\pi}\Gamma\left(\frac{n+1}{2}\right)}{n\Gamma\left(\frac{n}{2}\right)}\right)^{\frac{n}{n+1}}
\frac{n^{\frac{1}{n+1}}}{\vol_n(K)^{\frac{n-1}{n+1}}}\bigg(\int_{S^{n-1}} \frac{\vol_{n-1}\big(\proj_{u^\perp}(K)\big)}{\|u\|_K^{n+1}}\ud u\bigg)^{\frac{n}{n+1}} \label{eq:use petty},
\end{align}
where~\eqref{eq:use -n volume formula}  is~\eqref{eq:volume identity}, in~\eqref{eq:holder n+1} we used H\"older's inequality with the conjugate exponents $1+\frac{1}{n}$ and $n+1$, and~\eqref{eq:use petty} is an application of~\eqref{eq:petty in lemma}. This simplifies to give the desired inequality~\eqref{eq:cleverly averaged cone}.
\end{proof}

\begin{remark}\label{rem:volume of projection loose ends} Fix $n\in \N$, a normed space $\X=(\R^n,\|\cdot\|_\X)$ and $x\in S^{n-1}$. Both of the bounds in~\eqref{eq:polar projection body bounds} follow from elementary geometric reasoning (convexity and Fubini's theorem). Recalling~\eqref{eq:use cauchy}, the second inequality in~\eqref{eq:polar projection body bounds} is  $\vol_{n-1}(\proj_{x^{\perp}}B_\X)\le n\|x\|_\X\vol_n(B_\X)/2$; its justification can be found in the proof of Lemma~5.1 in~\cite{GNS12} (this was  not included in the version of~\cite{GNS12} that appeared in the journal, but it appears in the arxiv version of~\cite{GNS12}). The rest of~\eqref{eq:polar projection body bounds} is $\vol_n(B_\X)\|x\|_\X\le 2\vol_{n-1}(\proj_{x^{\perp}}B_\X)$; since we did not find a reference for the  derivation of this simple lower bound on hyperplane projections, we will now quickly justify it. For every $u\in \proj_{x^{\perp}}B_\X$ let $s(u)=\inf\{s\in \R:\ u+sx\in  B_\X\}$ and $t(u)=\sup\{t\in \R:\ u+tx\in  B_\X\}$. For every $u\in \proj_{x^{\perp}}B_\X$ we have $u+t(u)x\in B_\X$, and by symmetry also $-u-s(u)x\in B_\X$. Hence, by convexity $$
\frac12 \big(u+t(u)x\big)+\frac12 \big(-u-s(u)x\big)= \frac{t(u)-s(u)}{2}x\in B_\X.
$$ By the definition of $t(0)$, this means that $(t(u)-s(u))/2\le t(0)=1/\|x\|_\X$. Consequently, using Fubini's theorem (recall that $x\in S^{n-1}$) we conclude that
\begin{equation*}
\vol_n(B_\X)=\int_{\proj_{x^{\perp}}B_\X} \big(t(u)-s(u)\big)\ud u\le \int_{\proj_{x^{\perp}}B_\X}  \frac{2}{\|x\|_\X}\ud u= \frac{2}{\|x\|_\X}\vol_{n-1}\big(\proj_{x^{\perp}}B_\X\big).
\end{equation*}
\end{remark}

\section{Preliminaries on random partitions}\label{sec:prelim random part main}

This section treats basic properties of random partitions, including measurability issues that we need for subsequent applications. As such, it is of a technical/foundational nature and it can be skipped on first reading if one is willing to accept the  measurability requirements that are used in the proofs that appear in Section~\ref{sec:upper} and Section~\ref{sec:ext}.

Recall that  a random partition $\Part$ of a metric space $(\MM,d_\MM)$ was  defined in the Introduction as follows. One is given a probability space $(\Omega,\Pr)$ and a sequence of set-valued mappings $\{\Gamma^k:\Omega\to 2^\MM\}_{k=1}^\infty$ such that for each fixed $k\in \N$ the mapping $\Gamma^k:\Omega\to 2^\MM$ is strongly measurable relative to the $\sigma$-algebra of $\Pr$-measurable subsets of $\Omega$, i.e., the set $(\Gamma^k)^-(E)=\{\omega\in \Omega:\ E\cap \Gamma^k(\omega)\neq \emptyset \}$  is $\Pr$-measurable  for every closed $E\subset \MM$. We require that $\Part^\omega=\{\Gamma^k(\omega)\}_{k=1}^\infty$ is a partition of $\MM$ for every $\omega\in \Omega$.

Definition~\ref{def:separating finite} and Definition~\ref{def:padded finite}  (of separating and padded random partitions, respectively) assumed implicitly that the quantities that appear in the left hand sides of equations~\eqref{eq:separating condition} and~\eqref{eq:padded condition} are well-defined, i.e., that the events $\{\Part(x)\neq \Part(y)\}$ and $\{B_\MM(x,r)\subset \Part(x)\}$ are $\Pr$-measurable for every $x,y\in \MM$ and $r>0$. This follows from the  above definition, because  for  every  closed subset $E\subset \MM$ we have
\begin{equation*}
\big\{\omega\in \Omega:\ \Part^\omega(x)\neq \Part^\omega(y)\big\}=\bigcup_{\substack{k,\ell\in \N\\k\neq\ell}} \Big( \big\{\omega\in \Omega:\ \{x\}\cap \Gamma^k(\omega)\neq\emptyset \big\}\cap \big\{\omega\in \Omega:\ \{y\}\cap \Gamma^\ell(\omega)\neq\emptyset\big\}\Big),
\end{equation*}
and
$$
\big\{\omega\in \Omega:\ E\not\subset \Part^\omega(x)\big\}=\bigcup_{\substack{k,\ell\in \N\\k\neq\ell}} \Big( \big\{\omega\in \Omega:\ \{x\}\cap \Gamma^k(\omega)\neq\emptyset\big\}\cap \big\{\omega\in \Omega:\ E\cap \Gamma^\ell(\omega)\neq \emptyset\big\}\Big).
$$

Another ``leftover'' from the Introduction is the proof of Lemma~\ref{lem:invariance}, which asserts that the moduli of  Definition~\ref{def:separating finite} and Definition~\ref{def:padded finite} are bi-Lipschitz invariants. The proof of this simple but needed statement is the following  direct use of the definition of a $\Delta$-bounded random partition.

\begin{proof}[Proof of Lemma~\ref{lem:invariance}]
Fix $D>\cc_{(\NN,d_\NN)}(\MM,d_\MM)$. There is an embedding $\phi:\MM\to \NN$ and a scaling factor $\lambda>0$ such that~\eqref{eq:def distortion} holds.  Fix $\Delta>0$ and let $\Part$ be a $\lambda\Delta$-bounded random partition of $\NN$. Suppose that $\Part$ is induced by the probability space $(\Omega,\Pr)$, i.e., there are strongly measurable  mappings $\{\Gamma^k:\Omega\to 2^\NN\}_{k=1}^\infty$ such that $\Part^\omega=\{\Gamma^k(\omega)\}_{k=1}^\infty$ for every $\omega\in \Omega$. For every $k\in \N$ the mapping $\omega\mapsto \phi^{-1}(\Gamma^k(\omega))\in 2^\MM$ is strongly measurable. Indeed, if $E\subset \MM$ is closed then, because $\MM$ is complete and $\phi$ is a homeomorphism, also $\phi(E)\subset \NN$ is closed. So, $\{\omega\in \Omega:\ \phi(E)\cap \Gamma^k(\omega)\neq\emptyset\}=\{\omega\in \Omega:\ E\cap \phi^{-1}(\Gamma^k(\omega))\neq\emptyset\}$ is $\Pr$-measurable, as required. Therefore, if we define $\mathscr{Q}^\omega =\{\phi^{-1}(\Gamma^k(\omega))\}_{k=1}^\infty$ for $\omega\in \Omega$, then $\mathscr{Q}$ is a random partition of $\MM$.

$\mathscr{Q}$ is $\Delta$-bounded because for every  $x\in \MM$ and $u,v\in \mathscr{Q}(x)$ we have $\phi(u),\phi(v)\in \Part(\phi(x))$, and therefore $d_\MM(u,v)\le d_\NN(\phi(u),\phi(v))/\lambda\le \diam_\NN(\Part(\phi(x)))/\lambda\le \Delta$, using~\eqref{eq:def distortion}  and  that $\Part$ is $\lambda\Delta$-bounded. For every $x,y\in \MM$ the events $\{\mathscr{Q}(x)\neq \mathscr{Q}(y)\}$ and $\{\Part(\phi(x))\neq \Part(\phi(y))\}$ coincide. So, if  $\Part$ is $\sigma$-separating for some $\sigma>0$,
 $$
 \Pr\big[\mathscr{Q}(x)\neq \mathscr{Q}(y)\big]=\Pr\big[\Part\big(\phi(x)\big)\neq \Part\big(\phi(y)\big)\big]\le \frac{\sigma}{\lambda\Delta}d_\NN\big(\phi(x),\phi(y)\big)\stackrel{\eqref{eq:def distortion}}{\le} \frac{D\sigma}{\Delta}d_\MM(x,y).
 $$
 This shows that $\mathscr{Q}$ is $(D\sigma)$-separating, thus establishing the first assertion~\eqref{eq:sep invariance} of Lemma~\ref{lem:invariance}.

Suppose that $\Part$ is $(\mathfrak{p},\d)$-padded for some  $\mathfrak{p}>0$ and $0<\d<1$. Fix $x\in \MM$. Assuming that the event $\{B_\NN(\phi(x),\lambda\Delta/\mathfrak{p})\subset \Part(\phi(x))\}$ occurs, if $z\in B_\MM(x,\Delta/(D\mathfrak{p}))$, then $d_\NN(\phi(z),\phi(x))\le \lambda D d_\MM(z,x)\le \lambda\Delta/\mathfrak{p}$ by~\eqref{eq:def distortion}. Thus, $\phi(z)\in  B_\NN(\phi(x),\lambda\Delta/\mathfrak{p})$ and therefore  $\phi(z)\in \Part(\phi(x))$, i.e., $z\in \mathscr{Q}(x)$. This shows the inclusion of events  $\{B_\NN(\phi(x),\lambda\Delta/\mathfrak{p})\subset \Part(\phi(x))\}\subset \{B_\MM(x,\Delta/(D\mathfrak{p}))\subset \mathscr{Q}(x)\}$. Since $\Part$ is $(\mathfrak{p},\d)$-padded, it follows from this that also $\mathscr{Q}$ is $(D\mathfrak{p},\d)$-padded, thus establishing the second assertion~\eqref{eq:pad invariance} of Lemma~\ref{lem:invariance}.
\end{proof}

The final basic ``leftover'' from the Introduction is the following simple proof of Lemma~\ref{lem:tensorization}.

\begin{proof}[Proof of Lemma~\ref{lem:tensorization}] Fix $\Delta>0$ and suppose that  $\sigma_1>\sep(\MM_1)$ and $\sigma_2>\sep(\MM_2)$. Define
\begin{equation}\label{eq:def D1D2}
\Delta_1=\Delta\Big(\frac{\sigma_1}{\sigma_1+\sigma_2}\Big)^{\frac{1}{s}}\qquad\mathrm{and}\qquad  \Delta_2=\Delta\Big(\frac{\sigma_2}{\sigma_1+\sigma_2}\Big)^{\frac{1}{s}}.
\end{equation}
Let $\Part_{\Delta_1}$ be a $\sigma_1$-separating $\Delta_1$-bounded random partition of $\MM_1$. Similarly,  let $\Part_{\Delta_2}$ be a $\sigma_2$-separating $\Delta_2$-bounded random partition of $\MM_2$. Assume that $\Part_{\Delta_1}$ and $\Part_{\Delta_2}$ are independent random variables. Let $\Part_{\Delta}$ be the corresponding product random partition of $\MM_1\times\MM_2$, i.e., its clusters are give by
\begin{equation}\label{eq:product partition}
\forall(x_1,x_2)\in \MM_1\times \MM_2,\qquad \Part_\Delta(x_1,x_2)= \Part_{\Delta_1}(x_1)\times \Part_{\Delta_2}(x_2).
\end{equation}
By~\eqref{eq:def D1D2} we have $\Delta_1^s+\Delta_2^s=\Delta^s$, so $\Part_\Delta$ is a $\Delta$-bounded random partition of $\MM_1\oplus_s\MM_2$ (the required  measurability is immediate). It remains to note that every $(x_1,x_2),(y_1,y_2)\in \MM_1\times \MM_2$ satisfy
\begin{align}
\Pr\big[\Part_\Delta(x_1,x_2)\neq \Part_\Delta(y_1,y_2)\big]&=1-\Pr \big[\Part_{\Delta_1}(x_1)=\Part_{\Delta_1}(y_1)\big]\Pr \big[ \Part_{\Delta_2}(x_2)=\Part_{\Delta_2}(y_2)\big]\label{eq:use independence} \\
&\le1- \bigg(1-\frac{\sigma_1d_{\MM_1}(x_1,y_1)}{\Delta_1}\bigg)\bigg(1-\frac{\sigma_2d_{\MM_2}(x_2,y_2)}{\Delta_2}\bigg)\label{eq:use s1s1 sep}\\
&= \frac{\sigma_1d_{\MM_1}(x_1,y_1)}{\Delta_1}+\frac{\sigma_2d_{\MM_2}(x_2,y_2)}{\Delta_2}-
\frac{\sigma_1\sigma_2d_{\MM_1}(x_1,y_1)d_{\MM_2}(x_2,y_2)}{\Delta_1\Delta_2}\label{eq:drop quadratic}\\&\le
 \bigg(\Big(\frac{\sigma_1}{\Delta_1}\Big)^{\frac{s}{s-1}}+ \Big(\frac{\sigma_2}{\Delta_2}\Big)^{\frac{s}{s-1}}\bigg)^{\frac{s-1}{s}}\big(d_{\MM_1}(x_1,y_1)^s+d_{\MM_2}(x_2,y_2)^s\big)^{\frac{1}{s}}\label{eq:s holder}\\
 &=\frac{\sigma_1+\sigma_2}{\Delta}d_{\MM_1\oplus_s\MM_2}\big((x_1,x_2),(y_1,y_2)\big)\label{eq:done tensor},
\end{align}
where~\eqref{eq:use independence} uses~\eqref{eq:product partition} and the independence of $\Part_{\Delta_1}$ and $\Part_{\Delta_2}$, the bound~\eqref{eq:use s1s1 sep} is an application of the assumption that $\Part_{\Delta_1}$ is $\sigma_1$-separating and  $\Part_{\Delta_2}$ is $\sigma_2$-separating, \eqref{eq:s holder} is an application of H\"older's inequality, and~\eqref{eq:done tensor} follows from~\eqref{eq:product metric} and~\eqref{eq:def D1D2}. This proves~\eqref{eq:sep subsadditive}. Note that even though we dropped the quadratic additive improvement in~\eqref{eq:drop quadratic}, this does not change the final bound in~\eqref{eq:sep subsadditive} due to the need to work with all possible scales $\Delta>0$ and all possible values of $d_{\MM_1}(x_1,y_1)$ and $d_{\MM_2}(x_2,y_2)$.

To prove~\eqref{eq:pad subadditive}, fix $\p_1>\pad_{\d_1}(\MM_1)$ and $\p_2>\pad_{\d_2}(\MM_2)$ and replace~\eqref{eq:def D1D2} by
\begin{equation*}\label{eq:p1p2 tensor}
\Delta_1=\frac{\Delta\p_1}{\big(\p_1^s+\p_2^s\big)^{\frac{1}{s}}}   \qquad\mathrm{and}\qquad  \Delta_2=\frac{\Delta\p_2}{\big(\p_1^s+\p_2^s\big)^{\frac{1}{s}}}.
\end{equation*}
This time, we choose $\Part_{\Delta_1}$ to be a $(\p_1,\d_1)$-padded $\Delta_1$-bounded random partition of $\MM_1$. Similarly, let  $\Part_{\Delta_2}$    be a
$(\p_2,\d_2)$-padded $\Delta_2$-bounded random  partition of $\MM_2$, with $\Part_{\Delta_1}$ and $\Part_{\Delta_2}$ independent, and
we again combine them as in~\eqref{eq:product partition} to give  the product partition $\Part_\Delta$  of $\MM_1\times \MM_2$.
The analogous reasoning shows that $\Part_\Delta$ is a $((\p_1^s+\p_2^s)^{1/s},\d_1\d_2)$-padded $\Delta$-bounded random partition of
$\MM_1\oplus_s\MM_2$.
\end{proof}

\subsection{Standard set-valued mappings}\label{sec:standard} Recall that a metric space $(\MM,d_\MM)$ is said to be {Polish} if it is separable and complete. Polish metric spaces are the appropriate setting for Lipschitz extension theorems that are based on the assumption that for every $\Delta>0$ there is a probability distribution over $\Delta$-bounded partitions of $\MM$ with certain properties. Indeed,  a Banach space-valued Lipschitz function can always be extended to the completion of $\MM$ while preserving the Lipschitz constant, and the mere existence of countably many sets of diameter at most $\Delta$ that cover $\MM$ for every $\Delta>0$ implies that $\MM$ is separable.

Theorem~\ref{thm:local compact ext} assumes local compactness. Even though this assumption is more restrictive than being Polish, it  suffices for the applications that we obtain herein because they deal with finite dimensional normed spaces. It is, however, possible to treat general Polish metric spaces by working with a notion of measurability of set-valued mappings that differs from the strong measurability that was assumed in Section~\ref{sec:cluster}. We call this notion {\em standard set-valued mappings}; see Definition~\ref{def:standard def}.

The requirements for a set-valued mapping to be standard are quite innocuous and easy to check. In particular, the clusters of the specific random partitions that we will study  are easily seen to be standard set-valued mappings. It is also simple to verify that the clusters of the random partitions that we construct are strongly measurable. So, we have two approaches, which are both easy to work with. We chose to work in the Introduction with the requirement that the clusters are strongly measurable because this directly makes the quantity $\sep(\cdot)$ be bi-Lipschitz invariant, and it is also slightly simpler to describe. Nevertheless, in practice it is straightforward  to check that the clusters are standard, and even though we do not know that this leads to a bi-Lipschitz invariant (we suspect that it {\em does not}), it does lead to an easily implementable Lipschitz extension criterion that holds  in the maximal generality of Polish spaces.

\begin{definition}[standard set-valued mapping]\label{def:standard def} Suppose that $(\ZZ,d_\ZZ)$ is a Polish metric space and that $\Omega\subset \ZZ$ is a Borel subset of $\ZZ$. Given a metric space $(\MM, d_\MM)$, a set-valued mapping $\Gamma:\Omega\to 2^\MM$ is said to be {\bf \em standard} if the following three conditions hold.
\begin{itemize}
\item For every $x\in \MM$ the set $\{\omega\in \Omega:\ x\in \Gamma(\omega)\}$ is Borel.
\item The set $\cG_\Gamma= \Gamma^-(\MM)=\{\omega\in \Omega:\ \Gamma(\omega)\neq\emptyset\}$ is Borel.
\item For every $x\in \MM$ the mapping $(\omega\in \cG_\Gamma)\mapsto d_\MM(x,\Gamma(\omega))$ is Borel measurable on $\cG_\Gamma$.
\end{itemize}
\end{definition}

The following extension criterion is a counterpart to Theorem~\ref{thm:local compact ext} that works in the maximal generality of Polish metric spaces; its proof, which is an adaptation of ideas of~\cite{LN05},  appears in Section~\ref{sec:ext}.

\begin{theorem}\label{thm:polish extension} Let $(\MM,d_\MM)$ be a Polish metric space and fix another metric $\mathfrak{d}$ on $\MM$. Suppose that for every $\Delta>0$ there is a Polish metric space $\ZZ_\Delta$, a Borel subset $\Omega_\Delta\subset \ZZ_\Delta$, a Borel probability measure $\Pr_\Delta$ on $\Omega_\Delta$ and a sequence of standard set-valued mappings $\{\Gamma^k_\Delta:\Omega_\Delta\to 2^\MM\}_{k=1}^\infty$ such that   $\Part_\Delta^\omega=\{\Gamma^k_\Delta(\omega)\}_{k=1}^\infty$ is a partition of $\MM$ for every $\omega \in \Omega_\Delta$, for every $x\in \MM$ and $\omega\in \Omega_\Delta$ we have $\diam_\MM(\Part_\Delta^\omega(x))\le \Delta$, and
\begin{equation}
\forall x,y\in \MM,\qquad \Delta \Pr_\Delta\big[\omega\in \Omega_\Delta:\ \Part_\Delta^\omega(x)\neq\Part_\Delta^\omega(y)\big]\le \mathfrak{d}(x,y).
\end{equation}
Then, for every Banach space $(\bfZ,\|\cdot\|_\bfZ)$, every subset $\sub\subset \MM$ and every $1$-Lipschitz mapping $f:\sub\to \bfZ$, there exists  a mapping $F:\MM\to \bfZ$ that  extends $f$ and satisfies $\|F(x)-F(y)\|_\bfZ\lesssim \mathfrak{d}(x,y)$ for every $x,y\in \MM$ (namely, $F$ is Lipschitz on $\MM$ with respect to the metric $\mathfrak{d}$). Moreover, $F$ depends linearly on $f$.
\end{theorem}

\subsection{Proximal selectors}\label{sec:proximal} For later applications we need to know that  set-valued mappings that are either strongly measurable or standard  admit certain auxiliary measurable mappings that are (perhaps approximately) the closest point to a given (but arbitrary) nonempty closed subset of the metric space in question. We will justify this now using classical descriptive set theory.

\begin{lemma}\label{lem:nearest point selection criterion}
 Fix a measurable space $(\Omega,\mathscr{F})$. Suppose that $(\MM,d_\MM)$ is a metric space and that $S\subset \MM$ is nonempty and locally compact. Let $\Gamma:\Omega\to 2^\MM$ be a strongly measurable set-valued mapping such that  $\Gamma(\omega)$ is a bounded subset of $\MM$ for every $\omega\in \Omega$. Then there exists an $\mathscr{F}$-to-Borel measurable mapping $\gamma:\Omega\to S$ that satisfies $d_\MM(\gamma(\omega),\Gamma(\omega))=d_\MM(S,\Gamma(\omega))$ for every $\omega\in \Omega$ for which $\Gamma(\omega)\neq\emptyset$.
\end{lemma}

\begin{proof} For every $\omega\in \Omega$ define a subset $\Phi(\omega)\subset S$ as follows.
\begin{equation*}\label{eq:def Psi to select}
\Phi(\omega)\eqdef \left\{\begin{array}{cl}
\Big\{s\in S:\ d_\MM\big(s,\Gamma(\omega)\big)=d_\MM\big(S,\Gamma(\omega)\big)\Big\}&\mathrm{if}\ \Gamma(\omega)\neq\emptyset,\\
S&\mathrm{if}\ \Gamma(\omega)=\emptyset.
\end{array}\right.
\end{equation*}
The goal of Lemma~\ref{lem:nearest point selection criterion} is to demonstrate  the existence of an $\mathscr{F}$-to-Borel measurable mapping $\gamma:\Omega\to S$ that satisfies $\gamma(\omega)\in \Phi(\omega)$ for every $\omega\in \Omega$.  Since $(S,d_\MM)$ is locally compact, it is in particular Polish, so by the measurable selection theorem of Kuratowski and Ryll-Nardzewski~\cite{KR65} (see also~\cite{Wag77} or~\cite[Chapter~5.2]{Sri98}) it suffices to check that $\Phi(\omega)$ is nonempty and closed for every $\omega\in \Omega$, and  that $\{\omega\in \Omega:\ E\cap \Phi(\omega)= \emptyset\}\in \mathscr{F}$ for every closed  $E\subset S$. Since $S$ is locally compact, every closed subset of $S$ is a countable union of compact subsets, so it suffices to check the latter requirement for compact subsets of $S$, i.e., to show that $\{\omega\in \Omega:\ K\cap \Phi(\omega)= \emptyset\}\in \mathscr{F}$ for every compact $K\subset S$.

Fix $\omega\in \Omega$. If $\Gamma(\omega)=\emptyset$ then $\Phi(\omega)=S$ is closed (since $S$ is locally compact) and nonempty by assumption. If $\Gamma(\omega)\neq \emptyset $ then the continuity of the mapping $s\mapsto d_\MM(s,\Gamma(\omega))$ on $S$ implies that $\Phi(\omega)$ is closed. Moreover, in this case since $\Gamma(\omega)$ is bounded and $S$ is locally compact, the  continuous  mapping $s\mapsto d_\MM(s,\Gamma(\omega))$ attains its minimum on $S$, so that $\Phi(\omega)\neq \emptyset$.

It therefore remains to check that $\{\omega\in \Omega:\ K\cap \Phi(\omega)= \emptyset\}\in \mathscr{F}$ for every nonempty compact $K\subsetneq S$. Fixing such a $K$, since $S$  is locally compact and hence separable, there exist $\{\kappa_i\}_{i=1}^\infty\subset K$ and $\{\sigma_j\}_{j=1}^\infty\subset S$ that are dense in $K$ and $S$, respectively. Denote $\cG_\Gamma=\{\omega\in \Omega:\ \Gamma(\omega)\neq \emptyset\}$. Then $\cG_\Gamma\in \mathscr{F}$, because $\Gamma$ is strongly measurable. Observe that the following identity holds:
\begin{align}\label{eq:intersection idenity on K}
\begin{split}
 \big\{\omega\in \Omega:\ K\cap \Phi(\omega)= \emptyset\big\}&=\Big\{\omega\in \cG_\Gamma:\ \forall  \kappa\in K,\ \  d_\MM\big(\kappa,\Gamma(\omega)\big)>d_\MM\big(S,\Gamma(\omega)\big)\Big\}\\
&=\bigcup_{m=1}^\infty \bigcap_{i=1}^\infty\bigcup_{j=1}^\infty  \left\{\omega\in \cG_\Gamma:\ d_\MM\big(\kappa_i,\Gamma(\omega)\big)>d_\MM\big(\sigma_j,\Gamma(\omega)\big)+\frac{1}{m}\right\}.
\end{split}
\end{align}
The  verification of~\eqref{eq:intersection idenity on K} proceeds as follows. Since $\Phi(\omega)\neq \emptyset$ for every $\omega\in \Omega$ and $K\neq \emptyset$, if $K\cap \Phi(\omega)= \emptyset$ then  $\omega\in \cG_\Gamma$ (otherwise $\Phi(\omega)=S$).   This explains the first equality~\eqref{eq:intersection idenity on K}. For the second equality in~\eqref{eq:intersection idenity on K}, note that since $\Gamma(\omega)$ is bounded and $K$ is compact, $\inf_{\kappa\in K} d_\MM(\kappa,\Gamma(\omega))$ is attained. Therefore the second set in~\eqref{eq:intersection idenity on K} is equal to $A=\{\omega\in \cG_\Gamma:\ d_\MM(K,\Gamma(\omega))>d_\MM(S,\Gamma(\omega))\}$. If $\omega\in A$, then there is $m\in   \N$ such that $d_\MM(K,\Gamma(\omega))>d_\MM(S,\Gamma(\omega))+2/m$, implying in particular that $d_\MM(\kappa_i,\Gamma(\omega))>d_\MM(S,\Gamma(\omega))+2/m$ for every $i\in \N$. As $\{\sigma_j\}_{j=1}^\infty$ is dense in $S$, for every $i\in \N$ there is $j\in \N$ such that $d_\MM(\kappa_i,\Gamma(\omega))>d_\MM(\sigma_j,\Gamma(\omega))+1/m$. Hence, the second set in~\eqref{eq:intersection idenity on K} is contained in the third set in~\eqref{eq:intersection idenity on K}. For the reverse inclusion, if $\omega$ is in third set in~\eqref{eq:intersection idenity on K} then  $d_\MM(K,\Gamma(\omega))=\inf_{i\in \N}  d_\MM(\kappa_i,\Gamma(\omega))> \inf_{j\in \N} d_\MM(\sigma_j,\Gamma(\omega))= d_\MM(S,\Gamma(\omega))$.

  By~\eqref{eq:intersection idenity on K}, it suffices to show that $\{\omega\in \cG_\Gamma:\ d_\MM(x,\Gamma(\omega))>d_\MM(y,\Gamma(\omega))+r\}\in \mathscr{F}$ for every fixed $x,y\in S$ and $r>0$. For this, it suffices to show that for every $z\in \MM$ the mapping $\omega\mapsto d_\MM(z,\Gamma(\omega))$ is $\cF$-to-Borel measurable on $\cG_\Gamma$. Since $\cG_\Gamma\in \mathscr{F}$, this is a consequence of  the strong measurability of $\Gamma$, because for every $t\ge 0$ we have $\{\omega\in \cG_\Gamma:\ d_\MM(z,\Gamma(\omega))>t\}=\bigcup_{k=1}^\infty \cG_\Gamma\cap \{\omega\in \Omega:\ B_\MM(z,t+1/k)\cap \Gamma(\omega)=\emptyset\}$.
\end{proof}

Lemma~\ref{lem:nearest point selection criterion} is a satisfactory treatment of measurable nearest point selectors for strongly measurable set-valued mappings, though under an assumption of local compactness. We did not investigate the minimal assumptions that are required for the conclusion of Lemma~\ref{lem:nearest point selection criterion}  to hold. We will next treat the setting of standard set-valued mappings without assuming local compactness.

Let $(\ZZ,d_\ZZ)$ be a Polish metric space. Recall that a subset $A$ of $\ZZ$ is said to be {\em universally measurable} if it is measurable with respect to {\em every} complete $\sigma$-finite Borel measure $\mu$ on $\ZZ$ (see e.g.~\cite[page~155]{Kec95}). If $(\MM,d_\MM)$ is another metric space and $\Omega\subset \ZZ$ is Borel, then  a mapping $\psi:\Omega\to \MM$ is said to be universally measurable if $\psi^{-1}(E)$ is a universally measurable subset of $\Omega$ for every Borel subset $E$ of $\MM$. Finally, recall that $A\subset \MM$ is said to be {\em analytic}  if it is an image under a continuous mapping of a Borel subset of a Polish metric space (see e.g.~\cite[Chapter~14]{Kec95} or~\cite[Chapter~11]{Jec03}). By  Lusin's theorem~\cite{Luz17,Lus72} (see also e.g.~\cite[Theorem~21.10]{Kec95}), analytic subsets of Polish metric spaces are universally measurable.

\begin{lemma}\label{lem:proximal criterion} Let $(\MM,d_\MM)$ and $(\ZZ,d_\ZZ)$ be Polish metric spaces and fix a Borel subset $\Omega\subset \ZZ$. Fix also $\Delta>0$ such that $\diam(\MM)\ge \Delta$. Suppose that $\Gamma:\Omega\to 2^\MM$  satisfies the following two properties.
\begin{enumerate}%[label=(\alph*)]
\item For every $\omega\in \Omega$ such that $\Gamma(\omega)\neq\emptyset$ we have $\diam_\MM(\Gamma(\omega))< \Delta$. \label{assumtion:diameter}
\item For every $x\in \MM$ and $t\in \R$ the set $\{\omega\in \Omega:\ \Gamma(\omega)\neq\emptyset \ \wedge\ d_\MM(x,\Gamma(\omega))>t\}$ is analytic. \label{assumtion:analytic}
\end{enumerate}
Then, for every closed $\emptyset\neq S\subset \MM$ there is a universally measurable mapping $\gamma:\Omega\to S$ such that
$$
\forall (\omega,x)\in \Omega\times \MM,\qquad x\in \Gamma(\omega)\implies d_\MM\big(x,\gamma(\omega)\big)\le d_\MM(x,S)+\Delta.
$$
\end{lemma}

\begin{proof} For every $\omega\in \Omega$, define a subset $\Psi(\omega)\subset S$ as follows.
\begin{equation}\label{eq:def Psi}
\Psi(\omega)\eqdef \left\{\begin{array}{cl}
\bigcap_{x\in \MM}\big\{s\in S:\ d_\MM(x,s)\le 2d_\MM\big(x,\Gamma(\omega)\big)+d_\MM(x,S)+\Delta\big\}&\mathrm{if}\ \Gamma(\omega)\neq\emptyset,\\
S&\mathrm{if}\ \Gamma(\omega)=\emptyset.
\end{array}\right.
\end{equation}
We will show  that there exists a universally measurable mapping $\gamma:\Omega\to S$ such that $\gamma(\omega)\in \Psi(\omega)$ for every $\omega\in \Omega$. Since $S$ is a closed subset of $\MM$, it is Polish. Hence,  by the Kuratowski--Ryll-Nardzewski measurable selection theorem~\cite{KR65}, it suffices to prove that $\Psi(\omega)$ is nonempty and closed for every $\omega\in \Omega$, and  that $\Psi^-(E)=\{\omega\in \Omega:\ E\cap \Psi(\omega)\neq \emptyset\}$ is  universally measurable for every closed  $E\subset S$.

By design, $\Psi(\omega)=S$ is nonempty and closed if $\Gamma(\omega)= \emptyset$. So, fix $\omega\in \Omega$ such that $\Gamma(\omega)\neq \emptyset$. Then $\Psi(\omega)$ is closed because if $\{s_k\}_{k=1}^\infty\subset \Psi(\omega)$ and $s\in \MM$ satisfy $\lim_{k\to \infty} d_\MM(s_k,s)=0$, then for every $k\in \N$ and $x\in \MM$, since $s_k\in \Psi(\omega)$ we have  $d_\MM(s_k,x)\le 2d_\MM(x,\Gamma(\omega))+d_\MM(x,S)+\Delta$. Hence, by continuity also $d_\MM(s,x)\le 2d_\MM(x,\Gamma(\omega))+d_\MM(x,S)+\Delta$ for every $x\in \MM$, i.e., $s\in \Psi(\omega)$.

We will next check that $\Psi(\omega)\neq \emptyset$ for every $\omega\in \Omega$ such that $\Gamma(\omega)\neq \emptyset$. Denote $\e_\omega=\Delta-\diam_\MM(\Gamma(\omega))$. By assumption {\em (\ref{assumtion:diameter})} of Lemma~\ref{lem:proximal criterion} we have $\e_\omega>0$, so we may choose $s_\omega\in S$ and $y_\omega\in \Gamma(\omega)$  that satisfy $d_\MM(y_\omega,s_\omega)\le d_\MM(\Gamma(\omega),S)+\e_\omega$. We claim that $s_\omega\in \Psi(\omega)$. Indeed, for every $x\in \MM$ and $z\in \Gamma(\omega)$ we have
\begin{align}\label{eq:triangle s omega}
\begin{split}
d_\MM(x,s_\omega)\le d_\MM(x,z)+d_\MM(z,y_\omega)+ d_\MM(y_\omega,s_\omega)
\le d_\MM(x,z)+\diam_\MM\!\big(\Gamma(\omega)\big)+d_\MM(\Gamma(\omega),S)+\e_\omega\\\le d_\MM(x,z)+d_\MM(z,S)+\Delta\le d_\MM(x,z)+d_\MM(x,S)+d_\MM(x,z)+\Delta,
\end{split}
\end{align}
where in the penultimate step of~\eqref{eq:triangle s omega} we used the fact that $d_\MM(\Gamma(\omega),S)\le d_\MM(z,S)$, since $z\in \Gamma(\omega)$, and in the final step of~\eqref{eq:triangle s omega} we used the fact that the mapping $p\mapsto d_\MM(p,S)$ is $1$-Lipschitz on $\MM$. Since~\eqref{eq:triangle s omega} holds for every $z\in \Gamma(\omega)$, it follows that $d_\MM(x,s_\omega)\le 2d_\MM(x,\Gamma(\omega))+d_\MM(x,S)+\Delta$. Because this holds for every $x\in \MM$, it follows that $s_\omega\in \Psi(\omega)$.

Having checked that $\Psi$ takes values in closed and nonempty subsets of $S$, it remains to show that $\Psi^-(E)$ is universally measurable for every closed  $E\subset S$. To this end, since $\MM$ is separable, we may fix from now on a sequence $\{x_j\}_{j=1}^\infty$ that is dense in $\MM$. Note that by the case $t=0$ of assumption {\em (\ref{assumtion:analytic})} of Lemma~\ref{lem:proximal criterion}, for every $j\in \N$ the following set is analytic.
$$
\left\{\omega\in \Omega:\ \Gamma(\omega)\neq\emptyset\  \wedge\  d_\MM\big(x_j,\Gamma(\omega)\big)>0\right\}=\left\{\omega\in \Omega:\ \Gamma(\omega)\neq\emptyset\  \wedge\  x_j\notin \overline{\Gamma(\omega)}\right\}.
$$
Countable unions and intersections of analytic sets are analytic (see e.g.~\cite[Proposition~14.4]{Kec95}), so we deduce that the following set is analytic.
\begin{align}\label{eq:not dense}
\begin{split}
 \bigcup_{j=1}^\infty \left\{\omega\in \Omega:\ \Gamma(\omega)\neq\emptyset\  \wedge\  x_j\notin \overline{\Gamma(\omega)}\right\}&=\left\{\omega\in \Omega:\ \Gamma(\omega)\neq\emptyset\  \wedge\  \{x_j\}_{j=1}^\infty \not\subseteq \overline{\Gamma(\omega)}\right\}\\
&= \{\omega\in \Omega:\ \Gamma(\omega)\neq\emptyset\},
\end{split}
\end{align}
where for the final step of~\eqref{eq:not dense} observe that, since $\{x_j\}_{j=1}^\infty$ is dense in $\MM$, if $\{x_j\}_{j=1}^\infty$ were a subset of $\overline{\Gamma(\omega)}$ then it would follow that $\Gamma(\omega)$ is dense in $\MM$. This would imply that $\diam_\MM(\Gamma(\omega))=\diam(\MM)\ge \Delta$, in contradiction to assumption {\em (\ref{assumtion:diameter})} of Lemma~\ref{lem:proximal criterion}. We have thus checked that the set $\cG_\Gamma=\{\omega\in \Omega:\ \Gamma(\omega)\neq\emptyset\}$ is analytic, and hence by Lusin's theorem~\cite{Luz17,Lus72} it is universally measurable. Now,
\begin{align*}
\Psi^-(E) \stackrel{\eqref{eq:def Psi}}{=}(\Omega\setminus \cG_\Gamma) \cup\left\{\omega\in \cG_\Gamma:\ \exists\, s\in E\ \forall  x\in \MM,\ d_\MM(x,s)\le 2d_\MM\big(x,\Gamma(\omega)\big)+d_\MM(x,S)+\Delta\right\}.
\end{align*}
Hence, it remains to prove that the following set is universally measurable.
\begin{align}\label{eq:put x_j in}
\begin{split}
\big\{\omega\in \cG_\Gamma:\ \exists\, s\in E\ &\forall  x\in \MM,\ d_\MM(x,s)\le 2d_\MM\big(x,\Gamma(\omega)\big)+d_\MM(x,S)+\Delta\big\}\\
&=\left\{\omega\in \cG_\Gamma:\ \exists\, s\in E\ \forall  j\in \N,\ d_\MM(x_j,s)\le 2d_\MM\big(x_j,\Gamma(\omega)\big)+d_\MM(x_j,S)+\Delta\right\},
\end{split}
\end{align}
where we used the fact that $\{x_j\}_{j=1}^\infty$ is dense in $\MM$.

Consider the following subset $\mathcal{C}$ of  $ \Omega\times E$. %(note that $ \Omega\times E$ is a Borel subset of the Polish space $\ZZ\times E$).
$$
\mathcal{C}\eqdef  \big\{(\omega,s)\in  \cG_\Gamma\times E:\ \forall  j\in \N,\ d_\MM(x_j,s)\le 2d_\MM\big(x_j,\Gamma(\omega)\big)+d_\MM(x_j,S)+\Delta\big\}.
$$
The set in~\eqref{eq:put x_j in} is $\pi_1(\mathcal{C})$, where  $\pi_1:\Omega\times E\to \Omega$ is the projection to the first coordinate, i.e., $\pi_1(\omega,s)=\omega$ for every $(\omega,s)\in \Omega\times E$. Since continuous images and preimages of  analytic sets are analytic (see e.g.~\cite[Proposition~14.4]{Kec95}), by another application of Lusin's theorem it suffices to show that $\mathcal{C}$ is analytic. We already proved that $\cG_\Gamma\subset \Omega$ is analytic, so there is a Borel subset $L$ of  a Polish space $\mathcal{Y}$ and a continuous mapping $\phi:L\to \Omega$ such that $\phi(L)=\cG_\Gamma$. Denoting the identity mapping on $E$ by $\mathsf{Id}_E:E\to  E$, since $\phi$ maps $L$ onto $\cG_\Gamma$, the set $\mathcal{C}$ is the image under the continuous mapping $\phi\times \mathsf{Id}_E$ of the following subset of $\mathcal{Y}\times E$.
\begin{multline*}
 \big\{(y,s)\in  L\times E:\ \forall  j\in \N,\ d_\MM(x_j,s)\le 2d_\MM\big(x_j,\Gamma(\phi(y))\big)+d_\MM(x_j,S)+\Delta\big\}\\=
 \bigcap_{j=1}^\infty \big\{(y,s)\in  L\times E:\ d_\MM(x_j,s)\le 2d_\MM\big(x_j,\Gamma(\phi(y))\big)+d_\MM(x_j,S)+\Delta\big\}.
\end{multline*}
Hence, since continuous images and countable intersections of analytic sets are analytic,  by yet another application of Lusin's theorem we see that it suffices to show that for every fixed $x\in \MM$ the following set is analytic, where for every $q\in \Q$ we denote $A_q=\{(y,s)\in  L\times E:\ q< d_\MM(x,s)\}=L\times \{s\in   E:\ q< d_\MM(x,s)\}$.
\begin{align*}
&\big\{(y,s)\in  L\times E:\ d_\MM(x,s)\le 2d_\MM\big(x,\Gamma(\phi(y))\big)+d_\MM(x,S)+\Delta\big\}\\
&= \bigcap_{q\in \Q} \bigg(\big((L\times E)\setminus A_q\big)\cup \Big(A_q\cap \big\{(y,s)\in  L\times E:\ 2d_\MM\big(x,\Gamma(\phi(y))\big)>q-d_\MM(x,S)-\Delta\big\}\Big)\bigg),
\end{align*}
 Since $A_q$ is Borel for all $q\in \Q$, it  suffices to show that the following set is analytic for every $t\in \R$:
$$
 \big\{(y,s)\in  L\times E:\ d_\MM\big(x,\Gamma(\phi(y))\big)>t\big\}= \phi^{-1}\Big(\big\{\omega\in \cG_\Gamma:\ d_\MM\big(x,\Gamma(\omega)\big)>t\big\}\Big)\times E.
$$
Since a preimage under a  continuous  mapping of an analytic set is analytic, the above set is indeed analytic due to assumption {\em (\ref{assumtion:analytic})} of Lemma~\ref{lem:proximal criterion}  and the fact that $E$ is closed.
\end{proof}

\begin{remark}\label{rem:no diameter lower} The proof of Lemma~\ref{lem:proximal criterion}  used the assumption $\diam(\MM)\ge \Delta$ only to deduce that the set $\cG_\Gamma=\{\omega\in \Omega:\ \Gamma(\omega)\neq \emptyset\}$ is analytic from (the case $t=0$ of) assumption {\em(\ref{assumtion:analytic})} of Lemma~\ref{lem:proximal criterion}. Hence, if we add the assumption that $\cG_\Gamma$ is analytic to Lemma~\ref{lem:proximal criterion}, then we can drop the restriction  $\diam(\MM)\ge \Delta$  altogether. Alternatively, recalling equation~\eqref{eq:not dense} and the paragraph immediately after it, for the above proof of Lemma~\ref{lem:proximal criterion} to go through it suffices to assume that $\Gamma(\omega)$ is not dense in $\MM$ for any $\omega\in \Omega$.
\end{remark}

Recalling Definition~\ref{def:standard def}, Lemma~\ref{lem:proximal criterion} and Remark~\ref{rem:no diameter lower} imply the following corollary. Indeed, by Remark~\ref{rem:no diameter lower} we know that we can drop the assumption $\diam(\MM)\ge \Delta$ of Lemma~\ref{lem:proximal criterion}, and when $\Gamma$ is  a standard set-valued mapping the sets that appears in assumption {\em (\ref{assumtion:analytic})} of Lemma~\ref{lem:proximal criterion} are Borel.

\begin{corollary}\label{cor:proximal criterion standard} Fix $\Delta>0$. Let $(\MM,d_\MM)$ and $(\ZZ,d_\ZZ)$ be Polish metric spaces and fix a Borel subset $\Omega\subset \ZZ$. Suppose that  $\Gamma:\Omega\to 2^\MM$ is  a standard set-valued mapping such that $\diam_\MM(\Gamma(\omega))<\Delta$ for every $\omega\in \cG_\Gamma$.  Then for every closed $\emptyset\neq S\subset \MM$ there exists a universally measurable mapping $\gamma:\Omega\to S$ that satisfies
$$
\forall (\omega,x)\in \Omega\times \MM,\qquad x\in \Gamma(\omega)\implies d_\MM\big(x,\gamma(\omega)\big)\le d_\MM(x,S)+\Delta.
$$
\end{corollary}

\subsection{Measurability of iterative ball partitioning}\label{sec:iterative meas} The following set-valued mapping is a building block of much of the literature on random partitions, including the present investigation. Fix a metric space $(\MM,d_\MM)$ and $k\in \N$. Define a set-valued mapping $\Gamma:{\MM}^k\times [0,\infty)^k\to 2^{\MM}$ by
\begin{equation}\label{eq:def gamma lusin}
\forall \big(\vec{x},\vec{r}\big)=(x_1,\ldots,x_k,r_1,\ldots,r_k)\in {\MM}^k\times [0,\infty)^k,\quad  \Gamma\big(\vec{x},\vec{r}\big)\eqdef B_{\MM}(x_k,r_k)\setminus \bigcup_{j=1}^{k-1} B_{\MM}(x_j,r_j).
\end{equation}
We can think of $\Gamma$ as a random subset of $\MM$ if we are given a probability measure $\Pr$ on ${\MM}^k\times [0,\infty)^k$. The measure $\Pr$ can encode  the geometry of $(\MM,d_\MM)$; for example, if $(\MM,d_\MM)$ is a complete doubling metric space, then in~\cite{LN05} this measure arises from a doubling measure on $\MM$ (see~\cite{VK87,LS98}). The measure $\Pr$ can also have a ``smoothing effect'' through the randomness of the radii (see e.g.~\cite{Bar98,CKR04,FRT04,LN05,MN07,NT10,ABN11,NT12}; choosing a suitable distribution over the random radii is sometimes an important and quite delicate matter, but this intricacy will not arise in the present work. For  finite dimensional normed spaces, a random subset as in~\eqref{eq:def gamma lusin} was used in~\cite{CCGGP98,KMS98}. Note that given $\Delta>0$, if the measure $\Pr$ is supported on the set of those $(\vec{x},\vec{r})\in {\MM}^k\times [0,\infty)^k$ for which $r_k\le \Delta/2$, then the mapping $\Gamma$  takes values in subsets of $\MM$ of diameter at most $\Delta$.

While the definition~\eqref{eq:def gamma lusin} is very simple and natural, in order to use it in the ensuing reasoning we need to know that it satisfies certain  measurability requirements. Note first that the set-valued mapping $\Gamma$ in~\eqref{eq:def gamma lusin} automatically has the following basic measurability property: For every fixed $y\in \MM$ the set $\{(\vec{x},\vec{r})\in \MM^k\times [0,\infty)^k:\ y\in \Gamma(\vec{x},\vec{r})\}$ is Borel. Indeed, by definition we have
\begin{multline*}
\left\{\big(\vec{x},\vec{r}\big)\in \MM^k\times [0,\infty)^k:\ y\in \Gamma\big(\vec{x},\vec{r}\big)\right\}\\=
\bigcap_{j=1}^{k-1} \left\{\big(\vec{x},\vec{r}\big)\in \MM^k\times [0,\infty)^k:\ d_\MM(y,x_j)>r_j\right\}\cap \left\{\big(\vec{x},\vec{r}\big)\in \MM^k\times [0,\infty)^k:\ d_\MM(y,x_k)\le r_k\right\}.
\end{multline*}
In other words, the indicator mapping $(\vec{x},\vec{r})\mapsto \1_{\Gamma(\vec{x},\vec{r})}(y)$ is Borel measurable for every fixed $y\in \MM$.

\begin{lemma}\label{lem:analytic} Fix $k\in \N$.  Suppose that $(\MM,d_{\MM})$ is a Polish metric space. Let $\Gamma:{\MM}^k\times [0,\infty)^k\to 2^{\MM}$ be given in~\eqref{eq:def gamma lusin}. Then $\Gamma^-(S)=\{(\vec{x},\vec{r})\in {\MM}^k\times [0,\infty)^k:\ S\cap \Gamma(\vec{x},\vec{r})\neq \emptyset \}$ is analytic for every analytic subset $S\subset \MM$. Consequently, for every complete $\sigma$-finite Borel measure $\mu$ on  ${\MM}^k\times [0,\infty)^k$, if $\mathscr{F}_\mu$ denotes the  $\sigma$-algebra of $\mu$-measurable subsets of ${\MM}^k\times [0,\infty)^k$, then $\Gamma$ is a strongly measurable set-valued mapping from the measurable space $({\MM}^k\times [0,\infty)^k,\mathscr{F}_\mu)$ to $2^\MM$.
\end{lemma}

\begin{proof} Since $S$ is analytic, there exists a Borel subset $T$ of a  Polish metric space $\ZZ$ and a continuous mapping $\psi:T\to \MM$ such that $\psi(T)=S$.  Consider the following Borel subset $\mathcal{B}$ of the Polish  space ${\MM}^k\times [0,\infty)^k\times \ZZ$ ($\mathcal{B}$ is Borel because it is defined using finitely many continuous inequalities).
$$
\mathcal{B}\eqdef\Big\{(\vec{x},\vec{r},t)\in {\MM}^k\times [0,\infty)^k\times T:\ d_{\MM}(\psi(t),x_k)\le r_k\ \wedge\ \forall  j\in \{1,\ldots,k-1\},\ d_{\MM}(\psi(t),x_j)>r_j\Big\}.
$$
Then $\Gamma^-(S)=\pi(\mathcal{B})$, where $\pi:{\MM}^k\times [0,\infty)^k\times \ZZ\to {\MM}^k\times [0,\infty)^k$ is the projection onto the first two coordinates, i.e., $\pi(\vec{x},\vec{r},z)=(\vec{x},\vec{r})$ for every $(\vec{x},\vec{r},z)\in {\MM}^k\times [0,\infty)^k\times \ZZ$. Since  $\pi$ is continuous, it follows that $\Gamma^-(S)$ is analytic. By Lusin's theorem~\cite{Luz17,Lus72}, it follows that $\Gamma^-(S)$ is  universally measurable. In particular, if $\mu$ is a complete $\sigma$-finite Borel measure  on ${\MM}^k\times [0,\infty)^k$ and $\cF_\mu$ is the $\sigma$-algebra of $\mu$-measurable subsets of ${\MM}^k\times [0,\infty)^k$, then  $\Gamma^-(E)\in \cF_\mu$ for every closed subset $E\subset \MM$. Recalling~\eqref{eq:def measurability}, this means that $\Gamma$ is a strongly measurable set-valued mapping from the measurable space $({\MM}^k\times [0,\infty)^k,\cF_\mu)$ to $2^\MM$.
\end{proof}

Lemma~\ref{lem:open dense} below contains additional Borel measurability assertions that will be used later. Its assumptions are satisfied, for example, when $\MM$ is a separable normed space, which is the case of interest here. We did not investigate the maximal generality under which the conclusion of Lemma~\ref{lem:open dense} holds.

In what follows, given a metric space $(\MM, d_\MM)$, for every $x\in \MM$ and $r>0$ the open ball of radius $r$ centered at $x$ is denoted $B^{\mathsf{o}}_\MM(x,r)=\{y\in \MM:\ d_\MM(x,y)<r\}$.

\begin{lemma}\label{lem:open dense} Suppose that $(\MM,d_\MM)$ is a separable metric space such that
\begin{equation}\label{eq:ball assumption}
\forall (x,r)\in \MM\times (0,\infty),\qquad B_\MM(x,r)=\overline{B_\MM^{\mathsf{o}}(x,r)}.
\end{equation} Fix $k\in \N$ and let $\Gamma:{\MM}^k\times (0,\infty)^k\to 2^{\MM}$ be given in~\eqref{eq:def gamma lusin}. Then the following  set is Borel measurable.
$$\cG_\Gamma=\big\{(\vec{x},\vec{r})\in \MM^k\times (0,\infty)^k:\ \Gamma(\vec{x},\vec{r})\neq\emptyset\big\}.$$
Also, for each $y\in \MM$ the mapping from $\cG_\Gamma$ to $\R$ that is given by $(\vec{x},\vec{r})\mapsto d_\MM(y,\Gamma(x,r))$ is  Borel measurable.
\end{lemma}

\begin{proof} Let $\mathfrak{D}\subset \MM$ be a countable dense subset of $\MM$. The assumption~\eqref{eq:ball assumption} implies that $\mathfrak{D}\cap \Gamma(\vec{x},\vec{r})$ is dense in $\Gamma(\vec{x},\vec{r})$ for every $(\vec{x},\vec{r})\in \MM^k\times (0,\infty)^k$. This is straightforward to check as follows. Fix $y\in \Gamma(\vec{x},\vec{r})$ and $\d>0$. We need to find $q\in \mathfrak{D}\cap \Gamma(\vec{x},\vec{r})$ with $d_\MM(q,y)<\d$. Recalling~\eqref{eq:def gamma lusin}, since $y\in \Gamma(\vec{x},\vec{r})$ we know that $d_\MM(y,x_k)\le r_k$, and also $d_\MM(y,x_j)>r_j$ for every $j\in \{1,\ldots,k-1\}$, i.e., $\eta>0$ where
$$
\eta\eqdef\min\big\{\d,d_\MM(y,x_1)-r_1,\ldots,d_\MM(y,x_{k-1})-r_{k-1}\big\}.
$$
By~\eqref{eq:ball assumption} there is $z\in B_\MM^{\mathsf{o}}(x_k,r_k)$ with $d_\MM(z,y)<\eta/2$. Denote
$$\rho\eqdef \min\Big\{r_k-d_\MM(z,x_k),\frac12\eta \Big\}.$$
Then $\rho>0$, so the density of $\mathfrak{D}$ in $\MM$ implies that there is $q\in \mathfrak{D}$ with $d_\MM(q,z)<\rho$. Consequently,  $$d_\MM(q,y)\le d_\MM(q,z)+d_\MM(z,y)<\rho+\frac{\eta}{2}\le \d.$$ It remains to observe that $q\in \Gamma(\vec{x},\vec{r})$, because $d_\MM(q,x_k)\le d_\MM(q,z)+d_\MM(z,x_k)< \rho+d_\MM(z,x_k)\le r_k$ and also for every $j\in \{1,\ldots,k-1\}$ we have $$d_\MM(q,x_j)\ge d_\MM(y,x_j)-d_\MM(y,z)-d_\MM(z,q)>d_\MM(y,x_j)-\frac{\eta}{2}-\rho\ge d_\MM(y,x_j)-\eta\ge r_j.$$

For every $(\vec{x},\vec{r})\in \MM^k\times (0,\infty)^k$, we have $\Gamma(\vec{x},\vec{r})\neq\emptyset$  if and only if $\mathfrak{D}\cap \Gamma(\vec{x},\vec{r})\neq \emptyset$. Consequently,
$$
\cG_\Gamma=\left\{\big(\vec{x},\vec{r}\big)\in \MM^k\times (0,\infty)^k:\ \Gamma\big(\vec{x},\vec{r}\big)\neq\emptyset\right\}=\bigcup_{q\in\mathfrak{D}}\left\{\big(\vec{x},\vec{r}\big)\in \MM^k\times (0,\infty)^k:\ q\in \Gamma\big(\vec{x},\vec{r}\big)\right\}.
$$
Since $\mathfrak{D}$ is countable and we already checked in the paragraph immediately preceding Lemma~\ref{lem:analytic}  that $\{(\vec{x},\vec{r})\in \MM^k\times (0,\infty)^k:\ y\in \Gamma(\vec{x},\vec{r})\}$ is Borel measurable  for every $y\in \MM$, we get that $\cG_\Gamma$ is Borel measurable.

Next, $d_\MM(y,\Gamma(\vec{x},\vec{r}))=d_\MM(y,\mathfrak{D}\cap \Gamma(\vec{x},\vec{r}))$ for every $(\vec{x},\vec{r})\in \cG_\Gamma$ and $y\in \MM$. So, for every $t>0$ we have
$$
\left\{\big(\vec{x},\vec{r}\big)\in \cG_\Gamma:\ d_\MM\big(y,\Gamma\big(\vec{x},\vec{r}\big)\big)<t\right\}=\bigcup_{q\in\mathfrak{D}\cap B_\MM^{\mathsf{o}}(y,t)}\left\{\big(\vec{x},\vec{r}\big)\in \MM^k\times (0,\infty)^k:\ q\in \Gamma\big(\vec{x},\vec{r}\big)\right\}.
$$
It follows that $\{(\vec{x},\vec{r})\in \cG_\Gamma:\ d_\MM(y,\Gamma(\vec{x},\vec{r}))<t\}$ is Borel measurable for every $t\in\R$.
\end{proof}

Corollary~\ref{coro:standard} below follows directly from the definition of  a standard set-valued mapping due to Lemma~\ref{lem:open dense} and the discussion in the paragraph immediately preceding Lemma~\ref{lem:analytic}.

\begin{corollary}\label{coro:standard} Let $(\MM,d_\MM)$ be   a Polish metric space satisfying~\eqref{eq:ball assumption}. Then, for every $k\in \N$ the set-valued mapping $\Gamma:{\MM}^k\times (0,\infty)^k\to 2^{\MM}$ in~\eqref{eq:def gamma lusin} is  standard.
\end{corollary}

\section{Upper bounds on random partitions}\label{sec:upper}

In this section, we will prove the existence of random partitions with the separation and padding properties that were stated in the Introduction.

\subsection{Proof of Theorem~\ref{thm:mixed extension} and the upper bound on $\pad_\d(\X)$ in Theorem~\ref{thm:pad sharp}}\label{sec:the partition} Theorem~\ref{thm:sep and pad at once} below asserts that every normed space $\X=(\R^n,\|\cdot\|_{\X})$ admits a random partition that simultaneously has desirable padding and separation properties. In the literature, such properties are obtained for different random partitions: Separating partitions of normed spaces use iterative ball partitioning with deterministic radii, while padded partitions also rely on randomizing the radii. At present, we do not have in mind an application in which good padding and separation properties are needed simultaneously for the same random partition, so it is worthwhile to note this feature for potential future use but in what follows we will use Theorem~\ref{thm:sep and pad at once} to obtain two standalone conclusions that yield upper bounds on the moduli of  padded and separated decomposability (in fact, the separation profile of Theorem~\ref{thm:mixed extension}).

\begin{theorem}\label{thm:sep and pad at once} Fix $n\in \N$ and a normed space $\X=(\R^n,\|\cdot\|_{\X})$. For every $\Delta\in (0,\infty)$ there exists a  $\Delta$-bounded random partition $\Part_{\!\Delta}$ of $\X$ such that for every $x,y\in \R^n$ and every $\d\in (0,1)$ we have
\begin{equation}\label{eq:part Delta in theorem}
\Pr\big[\Part_{\!\Delta}(x)\neq \Part_{\!\Delta}(y)\big]\asymp \min\bigg\{1, \frac{\vol_{n-1}\big(\proj_{(x-y)^\perp}(B_{\X})\big)}{\Delta\vol_n(B_{\X})}\|x-y\|_{\ell_2^n}\bigg\},
\end{equation}
and,
\begin{equation*}
 \Pr\bigg[\Part_{\!\Delta}(x)\supset \frac{1-\sqrt[n]{\d}}{1+\sqrt[n]{\d}}\cdot\frac{\Delta}{2} B_{\X}\bigg]= \d.
\end{equation*}
\end{theorem}

 By the conventions of Remark~\ref{rem:boundedness convention}, the $\Delta$-boundedness of Theorem~\ref{thm:sep and pad at once} is with respect to the norm $\|\cdot\|_{\X}$, i.e., the clusters of the random partition $\Part_{\!\Delta}$ have $\X$-diameter at most $\Delta$. By the definitions in Section~\ref{sec:random partition intro}, the notion of random partition implies that each of the clusters of $\Part_{\!\Delta}$  is strongly measurable, but we will see that they are also standard (recall Definition~\ref{def:standard def}).

 %Note that Theorem~\ref{thm:sep and pad at once} implies Theorem~\ref{thm:volumetric sep intro} and the upper bound on $\pad_\d(\X)$ in %Theorem~\ref{thm:pad sharp}.

\begin{remark} For every $M>0$, consider the metric space $L_1^{\le M}=(L_1,d_M)$ that is given by
$$
\forall f,\in L_1,\qquad d_M(f,g)\eqdef \min\big\{M,\|f-g\|_{L_1}\big\}.
$$
A useful property~\cite[Lemma~5.4]{MN15} of this truncated $L_1$ metric is $c_{L_1}(L_1^{\le M})\lesssim 1$, i.e., $L_1^{\le M}$  embeds back into $L_1$ with bi-Lipschitz distortion $O(1)$. Theorem~\ref{thm:sep and pad at once}  gives a different proof of this since if $\X=\ell_\infty^n$, then by~\eqref{eq:ell infty proj formula} the right hand side  of~\eqref{eq:part Delta in theorem} is equal to $\min\{2\Delta,\|x-y\|_1\}/(2\Delta)$. At the same time, if $\Part_{\Delta}^\omega=\{\Gamma_\Delta^k(\omega)\}_{k=1}^\infty$, then the  left hand side of~\eqref{eq:part Delta in theorem} embeds isometrically into an $L_1(\mu)$ space via the embedding
$$
(f\in L_1)\mapsto\Big(\omega\mapsto \big(\1_{\Gamma^k(\omega)}(f)\big)_{k=1}^\infty\Big)\in L_1(\Pr;\ell_1).
$$
By~\eqref{eq:use cauchy}, the right hand side of~\eqref{eq:part Delta in theorem} equals  $\min\{\Delta, \|x-y\|_{\Pi^{\textbf *}\X}\}/\Delta$. But, by~\cite{Bol69} the class of finite dimensional normed spaces whose unit ball is a polar projection body coincides with those finite dimensional normed spaces that embed isometrically into $L_1$, so this does not give a new embedding result.
\end{remark}

We will first describe the construction that leads to the random partition whose existence is asserted in Theorem~\ref{thm:sep and pad at once}. This construction is a generalization of the construction that appears in the proof Lemma~3.16 of~\cite{LN05}, which itself combines a coloring argument with a generalization of the iterated ball partitioning technique that was used in the Euclidean setting in~\cite{CCGGP98,KMS98}.

In the rest of this section we will work under the assumptions and notation of Theorem~\ref{thm:sep and pad at once}. Let $\Lambda\subset \R^n$ be a lattice such that $\{z+B_{\X}\}_{z\in \Lambda}$ have pairwise disjoint interiors (equivalently, $\|z-z'\|_{\X}\ge 2$ for every distinct $z,z'\in \Lambda$) and $\bigcup_{z\in \Lambda} (z+3B_{\X})=\R^n$ (i.e., for every $x\in \R^n$ there is $z\in \Lambda$ such that $\|x-z\|_{\X}\le 3$). The existence of such a lattice follows from the work of Rogers~\cite{Rog50} (see~\cite[Remark~6]{Zon02}). The constant $3$ here is not the best-known (see~\cite{But72,Zon02}); we prefer to work with an explicit constant only for notational convenience despite the fact that its value is not important  in the present context.

Denote the $\X$-Voronoi cell of $\Lambda$, i.e., the set of points in $\R^n$ whose closest lattice point is the origin, by $$
\DD\eqdef \Big\{x\in \R^n:\ \|x\|_{\X}= \min_{z\in \Lambda}\|x-z\|_{\X}\Big\}.
$$
Then $\DD\subset 3B_{\X}$ and the translates $\{z+\DD\}_{z\in \Lambda}$ cover $\R^n$ and have pairwise disjoint interiors.

\begin{remark}
Our choice of the above lattice is natural since it is adapted to the intrinsic geometry of $\X=(\R^n,\|\cdot\|_{\X})$ and it leads to a simpler probability space in the construction below. Nevertheless, for the present purposes this choice is not crucial, and one could also work with any other lattice, including $\Z^n$. In that case, one could carry out the ensuing reasoning while adapting it to geometric characteristics of the lattice in question (its packing radius, covering radius and the diameter of its Voronoi cell, all of which are measured with respect to the metric induced by $\|\cdot\|_{\X}$). This requires several changes in the ensuing discussion, resulting in slightly more cumbersome computations that incorporate these geometric characteristics of the lattice. All of these quantities are universal constants for our choice of $\Lambda$.
\end{remark}

Define graph $\mathsf{G}=(\Lambda,\mathsf{E}_{\mathsf{G}})$ whose vertex set is the lattice $\Lambda$ and whose edge set $\mathsf{E}_{\mathsf{G}}$ is given by
\begin{equation*}\label{eq:def lattice edges}
\forall  w,z\in \Lambda,\qquad \{w,z\}\in \mathsf{E}_{\mathsf{G}}\iff w\neq z\ \wedge\ \inf_{\substack{a\in w+\DD\\ b\in z+\DD}} \|a-b\|_{\X}\le 10.
\end{equation*}
So, if $\{w,z\}\in \mathsf{E}_{\mathsf{G}}$ and $x\in B_{\X}$ then there exist $u,v\in \DD$ such that $\|(w+u)-(z+v)\|_{\X}\le 10$ and therefore, since $\DD\subset 3B_{\X}$, we have $\|w-(z+x)\|_{\X}\le  \|(w+u)-(z+v)\|_{\X}+\|u\|_{\X}+\|v\|_{\X}+\|x\|_{\X}\le 17$. Hence $z+B_{\X}\subset w+17B_{\X}$. It follows that if $w\in \Lambda$ and $z_1,\ldots,z_m\in \Lambda$ are the distinct neighbors of $w$ in the graph $\mathsf{G}$ then the balls $\{z_i+B_{\X}\}_{i=1}^m$ have disjoint interiors (since distinct elements of the lattice $\Lambda$ are at $\X$-distance at least $2$), yet they are all contained in the ball $w+17B_{\X}$. By comparing volumes, this implies that $m\le 17^n$. In other words, the degree of the graph $\mathsf{G}$ is at most $17^n$, and therefore (by applying the greedy algorithm, see e.g.~\cite{Bro41}) its chromatic number is at most $17^n+1\le 5^{2n}$, i.e., there is   $\chi:\Lambda\to \{1,\ldots,5^{2n}\}$ such that
\begin{equation}\label{coloring conclusion}
\forall  w,z\in \Lambda,\qquad w\neq z\ \wedge\  \inf_{\substack{a\in w+\DD\\ b\in z+\DD}} \|a-b\|_{\X}\le 10\implies \chi(w)\neq \chi(z).
\end{equation}

Consider the Polish space
$
\ZZ\eqdef\DD^\N\times \{1,\ldots,5^{2n}\}^\N.
$
In what follows, every $\omega\in \ZZ$ will be written as $\omega=(\vec{x},\vec{\gamma})$, where $\vec{x}=(x_1,x_2,\ldots)\in \DD^\N$ and $\vec{\gamma}=(\gamma_1,\gamma_2,\ldots)\in \{1,\ldots,5^{2n}\}^\N$. Denote by $\mu$ the normalized Lebesgue measure on $\DD$ and by $\nu$ the normalized counting measure on $\{1,\ldots,5^{2n}\}$, i.e.,  for every Lebesgue measurable $A\subset \R^n$ and every $F\subset \{1,\ldots,5^{2n}\}^\N$ we have
$$
\mu(A)\eqdef \frac{\vol_n(A\cap \DD)}{\vol_n(\DD)}\qquad \mathrm{and}\qquad  \nu(F)\eqdef \frac{|F|}{5^{2n}}.
 $$
Henceforth, the product probability measure $\mu^\N\times \nu^\N$ on $\ZZ$ will be denoted by $\Pr$.

For every $k\in \N$, $z\in \Lambda$ and $(\vec{x},\vec{\gamma})\in \ZZ$ define a subset $\Gamma^{k,z}(\vec{x},\vec{\gamma})\subset \R^n$ by
\begin{align}\label{def of our parition}
\begin{split}
\chi(z)=\gamma_k&\implies \Gamma^{k,z}\big(\vec{x},\vec{\gamma}\big)\eqdef (z+x_k+B_{\X}) \setminus \bigcup_{j=1}^{k-1}\bigcup_{\substack{w\in \Lambda\\ \chi(w)=\gamma_j}}(w+x_j+B_{\X}),\\
\chi(z)\neq \gamma_k&\implies \Gamma^{k,z}\big(\vec{x},\vec{\gamma}\big)\eqdef \emptyset.
\end{split}
\end{align}

\begin{lemma}\label{lem:measurability gamma} For every $k\in \N$ and $z\in \Lambda$ the set-valued mapping $\Gamma^{k,z}:\ZZ\to 2^{\R^n}$ is both strongly measurable and standard (where the underlying $\sigma$-algebra on $\ZZ$ is the $\Pr$-measurable sets).
\end{lemma}

\begin{proof} For every $\chi_1,\ldots,\chi_k\in \{1,\ldots,5^{2n}\}$ consider the cylinder set
$$
\mathcal{C}(\chi_1,\ldots,\chi_k)\eqdef \Big\{\big(\vec{x},\vec{\gamma}\big)\in \ZZ:\ (\gamma_1,\ldots,\gamma_k)=(\chi_1,\ldots,\chi_k)\Big\}.
$$
Since $\{\mathcal{C}(\chi_1,\ldots,\chi_k):\ (\chi_1,\ldots,\chi_k)\in \{1,\ldots,5^{2n}\}^k\}$ is a partition of $\ZZ$ into finitely many measurable sets, it suffices to fix from now on a $k$-tuple of colors $\vec{\chi}=(\chi_1,\ldots,\chi_k)\in \{1,\ldots,5^{2n}\}^k$ and to show that the restriction of $\Gamma^{k,z}$ to $\mathcal{C}(\chi_1,\ldots,\chi_k)$ is both strongly  measurable and standard.

Observe that for each fixed $z\in \Lambda$ and $\gamma\in \{1,\ldots,5^{2n}\}$ there is at most one $w\in \Lambda$ that satisfies $\chi(w)=\gamma$ and $(z+\DD+B_{\X})\cap (w+\DD+B_{\X})\neq \emptyset$. Indeed, if both $w\in \Lambda$ and $w'\in \Lambda$ satisfied these two requirements then we would have $\chi(w)=\gamma=\chi(w')$ and there would exist $a,a',b,b'\in \DD$ and $u,u',v,v'\in B_{\X}$ such that $w+a+u=z+b+v$ and $w'+a'+u'=z+b'+v'$. Hence,
\begin{multline*}
\inf_{\substack{\alpha\in w+\DD\\ \beta \in w'+\DD}} \|\alpha-\beta\|_{\X} \le \|(w+a)-(w'+a')\|_{\X}=\|(z+b+v-u)-(z+b'+v'-u')\|_{\X}\\\le \|b\|_{\X}+\|b'\|_{\X}+\|v\|_{\X}+\|v'\|_{\X}+\|u\|_{\X}+\|u'\|_{\X}\le 3+3+1+1+1+1=10,
\end{multline*}
where we used the fact that $b,b'\in \DD\subset 3B_{\X}$. By~\eqref{coloring conclusion} this contradicts the fact that $\chi(w)=\chi(w')$.

Having checked that the above $w$ is unique, denote it by $w(\gamma,z)\in \Lambda$. If there is no $w\in \Lambda$ that satisfies $\chi(w)=\gamma$ and $(z+\DD+B_{\X})\cap (w+\DD+B_{\X})\neq \emptyset$ then let $w(\gamma,z)\in\Lambda$ be an arbitrary (but fixed) lattice point such that $(z+\DD+B_{\X})\cap (w(\gamma,z)+\DD+B_{\X})= \emptyset$. Observe that $w(\chi(z),z)=z$. Under  this notation, for every  $x_1,\ldots,x_k\in \DD$ and $\gamma_1,\ldots,\gamma_{k-1}\in  \{1,\ldots,5^{2n}\}$ we have
\begin{equation*}
(z+x_k+B_{\X}) \setminus \bigcup_{j=1}^{k-1}\bigcup_{\substack{w\in \Lambda\\ \chi(w)=\gamma_j}}(w+x_j+B_{\X})=
(w(\chi(z),z)+x_k+B_{\X}) \setminus \bigcup_{j=1}^{k-1}(w(\gamma_j,z)+x_j+B_{\X}).
\end{equation*}
Equivalently, if we denote for every $\vec{\yy}=(\yy_1,\ldots,\yy_k)\in (\R^n)^k$,
$$
\Theta^k\big(\vec{\yy}\big)\eqdef (\yy_k+B_{\X}) \setminus \bigcup_{j=1}^{k-1}(\yy_j+B_{\X}),
$$
then the definition~\eqref{def of our parition} can be rewritten as the assertion that the restriction of $\Gamma^{k,z}$ to $\mathcal{C}(\vec{\chi})$ is the constant function $\emptyset$ if $\chi(z)\neq \chi_k$, while if $\chi(z)=\chi_k$ then    $\Gamma^{k,z}(\vec{x},\vec{\gamma})= \Theta^k(w(\vec{\chi},z)+\vec{x})$  for every $(\vec{x},\vec{\gamma})\in \mathcal{C}(\vec{\chi})$, where we use the notation $w(\vec{\chi},z)=(w(\chi_1,z),\ldots, w(\chi_k,z))\in (\R^n)^k$. The desired measurability of the restriction of $\Gamma^{k,z}$ to  $\mathcal{C}(\vec{\chi})$ now follows from Lemma~\ref{lem:analytic} and Corollary~\ref{coro:standard}.
\end{proof}

Since the sets $\{z+\DD\}_{z\in \Lambda}$ cover $\R^n$, for every rational point $q\in \Q^n$ we can fix from now on a lattice point $z_{q}\in \Lambda$ such that $q\in z_q+\DD$. Define a  subset $\Omega \subset \ZZ=\DD^\N\times \{1,\ldots,5^{2n}\}^\N$ by
\begin{equation}\label{eq:def Omega}
\Omega\eqdef \bigcap_{m=1}^\infty \bigcap_{q\in \Q^n} \bigcup_{k=1}^\infty \left\{\big(\vec{x},\vec{\gamma}\big)\in
\ZZ:\ \chi(z_q)= \gamma_k\ \wedge\ \|(z_q+x_k)-q\|_{\X}\le \frac{1}{m}\right\}.
\end{equation}
We record for ease of later use the following simple properties of $\Omega$.

\begin{lemma}\label{lem:Omega properties}
$\Omega$ is a Borel subset of $\ZZ$ that satisfies $\Pr[\Omega]=1$. Furthermore, for every $(\vec{x},\vec{\gamma})\in \Omega$ the set $\{z+x_k:\ (k,z)\in \N\times \Lambda\ \wedge\  \chi(z)=\gamma_k\}$ is dense in $\R^n$.
\end{lemma}

\begin{proof}
The fact that $\Omega$ is Borel is evident  from its definition~\eqref{eq:def Omega}. Also, if $(\vec{x},\vec{\gamma})\in \Omega$, $u\in \R^n$ and $\e\in (0,1)$, then choose $q\in \Q^n$ such that $\|u-q\|_{\X}<\e/2$. Setting $m=\lceil 2/\e\rceil\in \N$, it follows from~\eqref{eq:def Omega} that there exists $k\in \N$ satisfying $\chi(z_q)=\gamma_k$ and $\|(z_q+x_k)-q\|_{\X}\le 1/m\le
\e/2$. By our choice of $q$, it follows that $\|(z_q+x_k)-u\|_{\X}<\e$. Since this holds for every $\e\in (0,1)$, the set $\{z+x_k:\ (k,z)\in \N\times \Lambda\ \wedge\  \chi(z)=\gamma_k\}$  is dense in $\R^n$. It remains to show that $\Pr[\Omega]=1$. Indeed,
\begin{align}\label{eq:first union bound}
\begin{split}
\Pr[\ZZ\setminus \Omega]&\stackrel{\eqref{eq:def Omega}}{\le} \sum_{m=1}^\infty \sum_{q\in \Q^n}\Pr\bigg[\bigcap_{k=1}^\infty \ZZ\setminus \Big\{\big(\vec{x},\vec{\gamma}\big)\in
\ZZ:\ \chi(z_q)= \gamma_k\ \wedge\ \|(z_q+x_k)-q\|_{\X}\le \frac{1}{m}\Big\}\bigg]
\\&\ \ =\sum_{m=1}^\infty \sum_{q\in \Q^n} \lim_{\ell\to \infty} \bigg(1-\frac{\vol_n\left( \left(q-z_q+\frac{1}{m}B_{\X}\right)\cap\DD\right)}{5^{2n}\vol_n(\DD)}\bigg)^{\!\ell}=0,
\end{split}
\end{align}
where for the penultimate step of~\eqref{eq:first union bound} recall that $\Pr=\mu^\N\times \nu^\N$. For the final step of~\eqref{eq:first union bound}  one needs to check that $\vol_n((q-z_q+rB_{\X})\cap \DD)=\vol_n((q+rB_{\X})\cap (z_q+\DD))>0$ for every fixed $q\in \Q^n$ and $r\in (0,\infty)$. This is so because $z_q\in \Lambda$ was chosen so that $q\in z_q+\DD$ (and $\DD$ is a convex body).
\end{proof}

The following lemma introduces the random partition that will be used to prove Theorem~\ref{thm:sep and pad at once}.

\begin{lemma}\label{lem:is a bounded partition} $\Part\eqdef \{\Gamma^{k,z}|_\Omega:\Omega\to 2^{\R^n}\}_{(k,z)\in \N\times \Lambda}$ is a $2$-bounded random partition of $\X=(\R^n,\|\cdot\|_{\X})$, each of whose clusters are both strongly measurable and standard set-valued mappings.
\end{lemma}

\begin{proof} Since $\Omega$ is a Borel subset of $\ZZ$, for each $(k,z)\in \N\times \Lambda$ the measurability requirements for the restriction of $\Gamma^{k,z}$ to $\Omega$ follow from Lemma~\ref{lem:measurability gamma}. Fix $(\vec{x},\vec{\gamma})\in \ZZ$. Recalling~\eqref{def of our parition}, if  $\Gamma^{k,z}(\vec{x},\vec{\gamma})\neq\emptyset$, then  $\diam_{\X}(\Gamma^{k,z}(\vec{x},\vec{\gamma}))\le \diam_{\X}(z+x_k+B_{\X})\le 2$.  Note also that by~\eqref{def of our parition} if $\Gamma^{k,z}(\vec{x},\vec{\gamma})\neq\emptyset$, then
$$
\Gamma^{k,z}\big(\vec{x},\vec{\gamma}\big)=(z+x_k+B_{\X}) \setminus \bigcup_{j=1}^{k-1}\bigcup_{w\in \Lambda}\Gamma^{j,w}\big(\vec{x},\vec{\gamma}\big).
$$
Hence $\Gamma^{k,z}(\vec{x},\vec{\gamma})\cap \Gamma^{j,w}(\vec{x},\vec{\gamma})=\emptyset$ for every distinct $j,k\in \N$ and for every $w,z\in \Lambda$. We claim that also $$\Gamma^{k,z}(\vec{x},\vec{\gamma})\cap \Gamma^{k,w}(\vec{x},\vec{\gamma})=\emptyset$$ for every $k\in \N$ and every distinct $w,z\in \Lambda$. Indeed, it suffices to check this under the assumption that $\chi(w)=\chi(z)=\gamma_k$, since otherwise $\emptyset\in \{\Gamma^{k,z}(\vec{x},\vec{\gamma}), \Gamma^{k,w}(\vec{x},\vec{\gamma})\}$. So, suppose that $$\chi(w)=\chi(z)=\gamma_k\qquad \mathrm{yet}\qquad  \Gamma^{k,z}(\vec{x},\vec{\gamma})\cap \Gamma^{k,w}(\vec{x},\vec{\gamma})\neq\emptyset.$$ By~\eqref{def of our parition}, this implies that  there are $u,v\in B_{\X}$ such that $w+x_k+u=z+x_k+v$. Hence, for every $\alpha,\beta \in \DD$,
$$
\|(w+\alpha)-(z+\beta)\|_{\X}=\|\alpha-\beta+v-u\|_{\X}\le \|\alpha\|_{\X}+\|\beta\|_{\X}+\|u\|_{\X}+\|v\|_{\X}\le 3+3+1+1<10,
$$
where we used the fact that $\DD\subset 3B_{\X}$. Since $w$ and $z$ are distinct and $\chi(w)=\chi(z)$, this is in contradiction to~\eqref{coloring conclusion}. We have thus shown that the sets  $\{\Gamma^{k,z}(\vec{x},\vec{\gamma})\}_{(k,z)\in \N\times \Lambda}$ are pairwise disjoint.

Note that by the definition~\eqref{def of our parition}, for every $(\vec{x},\vec{\gamma})\in \ZZ$ we have
\begin{equation}\label{eq:union computation}
\bigcup_{k=1}^\infty\bigcup_{z\in \Lambda}\Gamma^{k,z}\big(\vec{x},\vec{\gamma}\big)=\bigcup_{\substack{(k,z)\in \N\times \Lambda\\   \chi(z)=\gamma_k}}(z+x_k+B_{\X})
\end{equation}
Indeed, it is immediate from~\eqref{def of our parition} that the left hand side of~\eqref{eq:union computation} is contained in the right hand side of~\eqref{eq:union computation}.  If $u$ belongs to the right hand side of~\eqref{eq:union computation}, then let $k$ be the minimum natural number  for which there is $z\in \Lambda$ with $\chi(z)=\gamma_k$ and $u\in z+x_k+B_{\X}$. So, for all $j\in \{1,\ldots,k-1\}$ and $w\in \Lambda$ with $\chi(w)=\gamma_j$ we have $u\notin w+x_j+B_{\X}$, and hence by~\eqref{def of our parition}  we have $v\in  \Gamma^{k,z}(\vec{x},\vec{\gamma})$, as required. By Lemma~\ref{lem:Omega properties}, if $(\vec{x},\vec{\gamma})\in \Omega$, then $\{z+x_k:\ (k,z)\in \N\times \Lambda\ \wedge\  \chi(z)=\gamma_k\}$ is dense in $\R^n$, and therefore the right hand side of~\eqref{eq:union computation} is equal to $\R^n$. Thus $\Part$ takes values in partitions of $\R^n$.
\end{proof}

Definition~\ref{def:hitting sets} introduces convenient notation that will be used several times in what follows.
\begin{definition}\label{def:hitting sets} If $\mathscr{M}\subset \R^n$ is Lebesgue measurable and $(k,z)\in \N\times \Lambda$, then define $\mathsf{H}_{\mathscr{M}}^{k,z}\subset \Omega$ by
\begin{equation}\label{eq:H event}
\mathsf{H}_{\mathscr{M}}^{k,z}\eqdef \big\{\big(\vec{x},\vec{\gamma}\big)\in \Omega:\ \chi(z)=\gamma_k\ \wedge\ z+x_k\in \mathscr{M}\big\}.
\end{equation}
If $\mathscr{S},\mathscr{T}\subset \R^n$ are Lebesgue measurable and $(k,z)\in \N\times \Lambda$, then define $\mathsf{K}_{\mathscr{S,T}}^{k,z}\subset \Omega$ by
\begin{equation}\label{eq:K event}
\mathsf{K}_{\mathscr{S,T}}^{k,z}\eqdef \mathsf{H}_{\mathscr{S}}^{k,z}\setminus \bigcup_{j=1}^{k-1}\bigcup_{w\in \Lambda} \mathsf{H}_{\mathscr{T}}^{j,w}.
\end{equation}
\end{definition}
%The sets $\mathsf{H}_{\mathscr{M}}^{k,z}, \mathsf{K}_{\mathscr{S,T}}^{k,z}$ of Definition~\ref{def:hitting sets} are $\Pr$-measurable.
The meaning of the set in~\eqref{eq:K event} is that it consists of all of those $(\vec{x},\vec{\gamma})\in \Omega$ such that the $k$'th coordinate of $\vec{\gamma}\in \{1,\ldots,5^{2n}\}^\N$ is the color of the lattice point $z\in \Lambda$, the $k$'th coordinate of  $\vec{x}\in \DD^\N$ satisfies $x_k\in \mathscr{S}-z$, and for no $j\in \{1,\ldots,k-1\}$ and no lattice point $w\in \Lambda$ do the same assertions hold with $\mathscr{S}$ replaced by $\mathscr{T}$.

\begin{lemma}\label{lem:volume ratio}
Suppose that $\mathscr{S},\mathscr{T}\subset \R^n$ are Lebesgue measurable sets of positive volume such that $\mathscr{S}\subset   \mathscr{T}$. Suppose also that $\diam_{\X}(\mathscr{T})\le 4$. Then the sets $$\big\{\mathsf{K}_{\mathscr{S,T}}^{k,z}\big\}_{(k,z)\in \N\times \Lambda}$$ are pairwise disjoint and
\begin{equation}\label{eq:hitting set prob}
\Pr\bigg[\bigcup_{k=1}^\infty\bigcup_{z\in \Lambda}\mathsf{K}_{\mathscr{S,T}}^{k,z}\bigg]=\frac{\vol_n(\mathscr{S})}{\vol_n(\mathscr{T})}.
\end{equation}
\end{lemma}

\begin{proof} The definition of the product measure $\Pr$ implies that for any Lebesgue measurable $\mathscr{M}\subset \R^n$,
\begin{equation}\label{eq:probability of H event}
\forall (j,w)\in \N\times \Lambda,\qquad \Pr\big[\mathsf{H}_{\mathscr{M}}^{j,w}\big]=\mu\big(\mathscr{M}-w\big)\nu\big(\chi(w)\big)=\frac{\vol_n\big(\DD\cap(\mathscr{M}-w)\big)}{5^{2n}\vol_n(\DD)}=\frac{\vol_n\big((\DD+w)\cap\mathscr{M}\big)}{5^{2n}\vol_n(\DD)},
\end{equation}
We claim if  $\diam_{\X}(\mathscr{M})\le 4$, then $\{\mathsf{H}_{\mathscr{M}}^{j,w}\}_{w\in \Lambda}$ are pairwise disjoint for every fixed $j\in \N$. Indeed, otherwise
$$
\exists \big(\vec{x},\vec{\gamma}\big)\in \mathsf{H}_{\mathscr{M}}^{j,w}\cap \mathsf{H}_{\mathscr{M}}^{j,z}
$$
for some distinct lattice points  $w,z\in \Lambda$. Then, $\chi(w)=\gamma_j=\chi(z)$ and $w+x_j,z+x_j\in \mathscr{M}$. Hence, $$\|w-z\|_{\X}=\|(w+x_j)-(z+x_j)\|_{\X}\le \diam_{\X}(\mathscr{M})\le 4.$$ Since $\DD\subset 3B_{\X}$, it follows that for every $\alpha,\beta\in \DD$ we have $$\|(w+\alpha)-(z+\beta)\|_{\X}\le \|w-z\|_{\X}+\|\alpha\|_{\X}+\|\beta\|_{\X}\le 4+3+3= 10,$$
which, by virtue of~\eqref{coloring conclusion}, contradicts the fact that $w\neq z$ and  $\chi(w)=\chi(z)$.

Since $\{\mathsf{H}_{\mathscr{M}}^{j,w}\}_{w\in \Lambda}$ are pairwise disjoint and $\{w+\DD\}_{w\in \Lambda}$ cover $\R^n$ and have pairwise disjoint interiors,
\begin{equation}\label{eq:sum of probs H events}
\Pr\bigg[\bigcup_{w\in \Lambda}\mathsf{H}_{\mathscr{M}}^{j,w} \bigg]=\sum_{w\in \Lambda}\Pr\big[\mathsf{H}_{\mathscr{M}}^{j,w}\big]\stackrel{\eqref{eq:probability of H event} }{=}\frac{1}{{5^{2n}\vol_n(\DD)}}\sum_{w\in \Lambda}\vol_n\big((\DD+w)\cap\mathscr{M}\big)=\frac{\vol_n(\mathscr{M})}{{5^{2n}\vol_n(\DD)}}.
\end{equation}
%where we recall that for~\eqref{eq:sum of probs H events} to hold we require that  $\mathscr{M}\subset X$ is Lebesgue measurable with $\diam_{\X}(\mathscr{M})\le 4$.

As $\mathscr{S}\subset \mathscr{T}$, we have $\diam_{\X}(\mathscr{S})\le\diam_{\X}(\mathscr{T})\le 4$. So, $\{\mathsf{H}_{\mathscr{S}}^{k,z}\}_{z\in \Lambda}$ are pairwise disjoint for every $k\in \N$ by the case $\mathscr{M}=\mathscr{S}$ of the above reasoning. Recalling~\eqref{eq:K event}, this implies that for every $k\in \N$ and distinct $w,z\in \Lambda$, $$ \mathsf{K}_{\mathscr{S,T}}^{k,w}\cap \mathsf{K}_{\mathscr{S,T}}^{k,z}=\emptyset.$$
To establish that $\{\mathsf{K}_{\mathscr{S,T}}^{k,z}\}_{(k,z)\in \N\times \Lambda}$ are pairwise disjoint it therefore remains to check that $$\mathsf{K}_{\mathscr{S,T}}^{k,z}\cap \mathsf{K}_{\mathscr{S,T}}^{j,w}=\emptyset$$ for every $j,k\in \N$ with $j<k$ and any $w,z\in \Lambda$. This is so because if $(\vec{x},\vec{\gamma})\in \mathsf{K}_{\mathscr{S,T}}^{k,z}$, then $(\vec{x},\vec{\gamma})\notin \mathsf{H}_{\mathscr{T}}^{j,w}$ by~\eqref{eq:K event}. Therefore either $\chi(w)\neq \gamma_j$ or $w+x_j\notin \mathscr{T}\supset \mathscr{S}$.  Consequently, $$\big(\vec{x},\vec{\gamma}\big)\notin \mathsf{H}_{\mathscr{S}}^{j,w}\supset \mathsf{K}_{\mathscr{S,T}}^{j,w}.$$
This concludes the verification of the disjointness of $\{\mathsf{K}_{\mathscr{S,T}}^{k,z}\}_{(k,z)\in \N\times \Lambda}$.

 Since for every $k\in \N$ and $z\in \Lambda$, the membership of $(\vec{x},\vec{\gamma})\in  \{1+,\ldots, 5^{2n}\}^\N\times  \DD^\N$ in $\mathsf{H}_{\mathscr{S}}^{k,z}$ and $\mathsf{H}_{\mathscr{T}}^{k,z}$ depends only on the $k$'th coordinates of $\vec{x}$ and $\vec{\gamma}$, it follows from the independence of the coordinates that
 \begin{align}\label{eq:use independence K event}
 \begin{split}
 \Pr\big[&\mathsf{K}_{\mathscr{S,T}}^{k,z}\big]\stackrel{\eqref{eq:K event}}{=}\Pr\bigg[\mathsf{H}_{\mathscr{S}}^{k,z} \bigcap \bigg(\bigcap_{j=1}^{k-1} \Big(\Omega \setminus \bigcup_{w\in \Lambda}\mathsf{H}_{\mathscr{T}}^{j,w}\Big)\bigg)\bigg]\\&=\Pr\big[\mathsf{H}_{\mathscr{S}}^{k,z} \big]\prod_{j=1}^{k-1} \bigg(1-\Pr\bigg[\bigcup_{w\in \Lambda}\mathsf{H}_{\mathscr{T}}^{j,w}\bigg]\bigg)
 \stackrel{\eqref{eq:probability of H event}\wedge \eqref{eq:sum of probs H events}}{=}\frac{\vol_n\big((\DD+z)\cap\mathscr{S}\big)}{5^{2n}\vol_n(\DD)}\bigg(1-\frac{\vol_n(\mathscr{T})}{{5^{2n}\vol_n(\DD)}}\bigg)^{k-1}.
 \end{split}
 \end{align}
Hence, since we already checked that $\{\mathsf{K}_{\mathscr{S,T}}^{k,z}\}_{(k,z)\in \N\times \Lambda}$ are pairwise disjoint,
\begin{align*}
\Pr\bigg[\bigcup_{k=1}^\infty\bigcup_{z\in \Lambda}\mathsf{K}_{\mathscr{S,T}}^{k,z}\bigg]&=\sum_{k=1}^\infty\sum_{z\in \Lambda}\Pr\big[\mathsf{K}_{\mathscr{S,T}}^{k,z}\big]
\\&\stackrel{\eqref{eq:use independence K event}}{=}\frac{1}{5^{2n}\vol_n(\DD)}\bigg(\sum_{z\in \Lambda} \vol_n\big((\DD+z)\cap\mathscr{S}\big)\bigg)\sum_{k=1}^\infty \bigg(1-\frac{\vol_n(\mathscr{T})}{{5^{2n}\vol_n(\DD)}}\bigg)^{k-1}=\frac{\vol_n(\mathscr{S})}{\vol_n(\mathscr{T})},
\end{align*}
where in the final step we used once more the fact that the sets $\{w+\DD\}_{w\in \Lambda}$ cover $\R^n$ and have pairwise disjoint interiors. This completes the verification of the desired identity~\eqref{eq:hitting set prob}.
\end{proof}

The following lemma is a computation of the probability of the ``padding event'' corresponding to the random partition $\Part$, as  a consequence of Lemma~\ref{lem:volume ratio}. In~\cite{MN07}  a similar argument was carried out for general finite metric spaces, but it relied on a different random partition in which the radius of the balls is also a random variable (namely, the partition of~\cite{CKR04}). This subtlety is circumvented here by using  properties of normed spaces that are not available in the full generality of~\cite{MN07}.

\begin{lemma}\label{lem:padded rho}
Let $\Part$ be the random partition of Lemma~\ref{lem:is a bounded partition}. For every $\rho\in (0,1)$ and $u\in \R^n$ we have
\begin{equation}\label{eq:padded rho}
 \Pr\Big[u+\rho B_{\X}\subset \Part(u)\Big]= \left(\frac{1-\rho}{1+\rho}\right)^n.
\end{equation}
\end{lemma}

\begin{proof} For every $k\in \N$, $z\in \Lambda$ and $r\in (0,\infty)$ define $\mathscr{E}_{u,r}^{k,z},\mathscr{F}_{u,r}^{k,z}\subset \Omega$ by
\begin{equation}\label{eq:rename to E event}
\mathscr{E}_{u,r}^{k,z}\eqdef \mathsf{H}^{k,z}_{u+rB_{\X}}\qquad\mathrm{and}\qquad \mathscr{F}_{u,r}^{k,z}\eqdef \mathsf{K}^{k,z}_{u+(1-r)B_{\X},u+(1+r)B_{\X}},
\end{equation}
i.e., we are using here the notations of Definition~\ref{def:hitting sets} for the sets  $\mathscr{M}=u+rB_{\X}$, $\mathscr{S}=u+(1-r)B_{\X}$ and $\mathscr{T}=u+(1+r)B_{\X}$. We claim that
\begin{equation}\label{eq:decompose padding event}
\forall (k,z)\in \N\times \Lambda,\qquad \big\{\big(\vec{x},\vec{\gamma}\big)\in \Omega:\  \Gamma^{k,z}\big(\vec{x},\vec{\gamma}\big)\supset u+\rho B_{\X}\big\}= \mathscr{F}_{u,\rho}^{k,z}.
\end{equation}
Note that, since $u+(1-\rho)B_{\X}\subset u+(1+\rho)B_{\X}$ and $\diam_{\X}(u+(1+\rho)B_{\X})=2(1+\rho)\le 4$, once~\eqref{eq:decompose padding event} is proven we could  apply Lemma~\ref{lem:volume ratio} to deduce the desired identity~\eqref{eq:padded rho} as follows.
\begin{multline*}
 \Pr\Big[u+\rho B_{\X}\subset \Part(u)\Big]\stackrel{\eqref{def of our parition}}{=}\Pr\Big[\big\{\big(\vec{x},\vec{\gamma}\big)\in \Omega:\ \exists(k,z)\in \N\times \Lambda,\quad \Gamma^{k,z}\big(\vec{x},\vec{\gamma}\big)\supset u+\rho B_{\X}\big\}\Big]\\
 \stackrel{\eqref{eq:decompose padding event}}{=}\Pr\bigg[\bigcup_{k=1}^\infty\bigcup_{z\in \Lambda}\mathscr{F}_{u,\rho}^{k,z}\bigg]\stackrel{\eqref{eq:hitting set prob}\wedge \eqref{eq:rename to E event}}{=} \frac{\vol_n(u+(1-\rho)B_{\X})}{\vol_n(u+(1+\rho)B_{\X})}
=\left(\frac{1-\rho}{1+\rho}\right)^n.
\end{multline*}

To establish~\eqref{eq:decompose padding event}, suppose first that $(\vec{x},\vec{\gamma})\in \mathscr{F}_{u,\rho}^{k,z}$. By the definition of  $\mathscr{F}_{u,\rho}^{k,z}$ we therefore know that $$\forall (j,w)\in \{1,\ldots,k-1\}\times \Lambda, \qquad \big(\vec{x},\vec{\gamma}\big)\in \mathscr{E}_{u,1-\rho}^{k,z}\qquad\mathrm{yet}\qquad  \big(\vec{x},\vec{\gamma}\big)\notin \mathscr{E}_{u,1+\rho}^{j,w}.$$
Hence, by the definition of $\mathscr{E}_{u,1-\rho}^{j,w}$ we know that $\chi(z)=\gamma_k$ and $z+x_k\in u+ (1-\rho)B_{\X}$, which (using the triangle inequality), implies that $z+x_k+B_{\X}\supset u+\rho B_{\X}$. At the same time, if $j\in \{1,\ldots,k-1\}$ and $w\in \Lambda$, then by the definition of $\mathscr{E}_{u,1+\rho}^{j,w}$, the fact that $(\vec{x},\vec{\gamma})\notin \mathscr{E}_{u,1+\rho}^{j,w}$ means that if $\chi(w)=\gamma_j$ then necessarily $\|w+x_j-u\|_{\X}>1+\rho$, which (using the triangle inequality) implies that $(w+x_j+B_{\X})\cap (u+\rho B_{\X})=\emptyset$. Hence, the ball $u+\rho B_{\X}$ does not intersect  the union of the balls $\{w+x_j+B_{\X}:\ (j,w)\in \{1,\ldots,k-1\}\times \Lambda\ \wedge\ \chi(w)=\gamma_j\}$. Since  $\chi(z)=\gamma_k$, due to~\eqref{def of our parition}, this implies that $$\Gamma^{k,z}\big(\vec{x},\vec{\gamma}\big)\cap (u+\rho B_{\X})=(z+x_k+B_{\X})\cap (u+\rho B_{\X})=u+\rho B_{\X},$$ i.e., $(\vec{x},\vec{\gamma})$ belongs to the left hand side of~\eqref{eq:decompose padding event}.

To establish the reverse inclusion, suppose that     $\Gamma^{k,z}(\vec{x},\vec{\gamma})\supset u+\rho B_{\X}$. The definition~\eqref{def of our parition}  implies in particular that $\Gamma^{k,z}(\vec{x},\vec{\gamma})\subset z+x_k+B_{\X}$ and that for $\Gamma^{k,z}(\vec{x},\vec{\gamma})$ to be nonempty we must  have $\chi(z)=\gamma_k$. So, we know that $\chi(z)=\gamma_k$ and $z+x_k+B_{\X}\supset u+\rho B_{\X}$. Assuming first that $z+x_k\neq u$,  consider the vector $$v=u+\frac{\rho}{\|u-z-x_k\|_{\X}}(u-z-x_k).$$ Then, $v\in u+\rho B_{\X}$ and hence also $v\in  z+x_k+B_{\X}$, i.e., $1\ge \|v-z-x_k\|_{\X}=\|u-z-x_k\|_{\X}+\rho$. This shows that $\|z+x_k-u\|_{\X}\le 1-\rho$, i.e., $z+x_k\in u+(1-\rho) B_{\X}$. We obtained this conclusion under the assumption that $z+x_k\neq u$, but it of course holds trivially also when $z+x_k=u$. We have thus shown that $(\vec{x},\vec{\gamma})\in \mathscr{E}_{u,1-\rho}^{k,z}$.

By the definition of $\mathscr{F}_{u,\rho}^{k,z}$, it remains to check that
\begin{equation}\label{eq:for j min}
\forall (j,w)\in \{1,\ldots,k-1\}\times \Lambda,\qquad \big(\vec{x},\vec{\gamma}\big)\notin \mathscr{E}_{u,1+\rho}^{j,w}.
\end{equation}
Indeed, if~\eqref{eq:for j min} does not hold, then let $j_{\min}$ be the minimum $j\in \{1,\ldots,k-1\}$ for which $(\vec{x},\vec{\gamma})\in\mathscr{E}_{u,1+\rho}^{j,w}$ for some $w\in \Lambda$. Hence, $\chi(w)=\gamma_{j_{\min}}$ and $w+x_{j_{\min}}\in u+(1+\rho)B_{\X}$.  If $w+x_{j_{\min}}\neq u$, then the vector $$u+\frac{\rho}{\|w+x_{j_{\min}}-u\|_{\X}}(w+x_{j_{\min}}-u)$$ is at $\X$-distance $\rho$ from $u$ and also at $\X$-distance $|\rho-\|w+x_{j_{\min}}-u\|_{\X}|\le 1$ from $w+x_{j_{\min}}$, where we used the fact that $\|w+x_{j_{\min}}-u\|_{\X}\le 1+\rho$. This shows that $(w+x_{j_{\min}}+B_{\X})\cap (u+\rho B_{\X})\neq \emptyset$ under the assumption $w+x_{j_{\min}}\neq u$, and this assertion trivially holds also if $w+x_{j_{\min}}= u$. The minimality of $j_{\min}$ implies that for every $j\in \{1,\ldots,j_{\min}-1\}$ and every $w'\in \Lambda$ with $\chi(w')=\gamma_j$ we have $w'+x_j\notin u +(1+\rho)B_{\X}$, i.e., $\|w'+x_j-u\|_{\X}>1+\rho$. Hence (by the triangle inequality) we have $(w'+ x_j+B_{\X})\cap (u+\rho B_{\X})=\emptyset$. The definition of $\Gamma^{j_{\min},w}(\vec{x},\vec{\gamma})$ now shows that $ (u+\rho B_\X)\cap \Gamma^{j_{\min},w}(\vec{x},\vec{\gamma})\neq \emptyset$, and since by Lemma~\ref{lem:is a bounded partition} we know that  $\Gamma^{j_{\min},w}(\vec{x},\vec{\gamma})$ and $\Gamma^{k,z}(\vec{x},\vec{\gamma})$ are disjoint (as $j_{\min}<k$), this contradicts the premise  $\Gamma^{k,z}(\vec{x},\vec{\gamma})\supset u+\rho B_{\X}$.
\end{proof}

The probability of the ``separation event'' corresponding to the random partition $\Part$ is estimated in the following lemma by using Lemma~\ref{lem:volume ratio}, together with input from Brunn--Minkowski theory.

\begin{lemma}\label{lem:separation psi} Let  $\Part$ be the random partition of Lemma~\ref{lem:is a bounded partition}. For every $u,v\in \R^n$  we have
\begin{equation}\label{eq:separating Delta=2} \Pr\big[\Part(u)\neq\Part(v)\big]\asymp\min\left\{1,\frac{\vol_{n-1}\big(\proj_{(u-v)^\perp}(B_{\X})\big)}{\vol_n(B_{\X})}\|u-v\|_{\ell_2^n}\right\}.
\end{equation}
More precisely, if we denote $\psi(0)=0$ and
\begin{equation}\label{eq:def psi}
\forall  w\in \R^n\setminus \{0\},\qquad \psi(w)\eqdef \frac{\vol_{n-1}\big(\proj_{w^\perp}(B_{\X})\big)}{\vol_n(B_{\X})}\|w\|_{\ell_2^n}=\frac{\|w\|_{\Pi^{\textbf *}\X}}{\vol_n(B_\X)},
\end{equation}
then for every $u,v\in \R^n$ we have
\begin{equation}\label{eq:separating Delta=2 psi}
\frac{2e^{\psi(u-v)}-2}{2e^{\psi(u-v)}-1}\le\Pr\big[\Part(u)\neq\Part(v)\big]\le \frac{2\psi(u-v)}{1+\psi(u-v)}.
\end{equation}
In particular, \eqref{eq:separating Delta=2 psi} implies the following more precise version of~\eqref{eq:separating Delta=2}.
$$
\frac{2e-2}{2e-1}\min\big\{1,\psi(u-v)\big\}\le \Pr\big[\Part(u)\neq\Part(v)\big]\le 2\min\big\{1,\psi(u-v)\big\}.$$
Moreover, \eqref{eq:separating Delta=2 psi} shows  that   $\Pr\big[\Part(u)\neq\Part(v)\big]=2\psi(u-v)+O\big(\psi(u-v)^2\big)$ as $u\to v $.
\end{lemma}

\begin{proof} If $\|u-v\|_\X>2$, then $\Pr[\Part(u)\neq \Part(v)]=1$ as $\Part$ is $2$-bounded. As $(2e^{\psi(u-v)}-2)/(2e^{\psi(u-v)}-1)<1$, the first inequality in~\eqref{eq:separating Delta=2 psi} holds. By~\eqref{eq:polar projection body bounds} we have $\psi(u-v)\ge \|u-v\|_\X/2> 1$, so $2\psi(u-v)/(\psi(u-v)+1)> 1$ and hence the second inequality in~\eqref{eq:separating Delta=2 psi}  holds. We will therefore assume from now on that $\|u-v\|_\X\le 2$.

Denote $\mathscr{I}(u,v)=(u+B_{\X})\cap (v+B_{\X})$ and $\mathscr{U}(u,v)=(u+B_{\X})\cup (v+B_{\X})$. We claim that
\begin{equation}\label{eq:non separating event}
\forall(k,z)\in \N\times \Lambda,\qquad \big\{(\vec{x},\vec{\gamma})\in \Omega:\ \{u,v\}\subset \Gamma^{k,z}(\vec{x},\vec{\gamma})\big\}= \mathsf{K}^{k,z}_{\mathscr{I}(u,v),\mathscr{U}(u,v)},
\end{equation}
where we recall the notation that was introduced in Definition~\ref{def:hitting sets}. Assuming~\eqref{eq:non separating event} for the moment, we will next explain how to conclude the proof of Lemma~\ref{lem:separation psi}.

Note that $\mathscr{I}(u,v)\subset \mathscr{U}(u,v)$ and $\diam_{\X}(\mathscr{U}(u,v))\le \|u-v\|_{\X}+2\diam_{\X}(B_{\X})\le 4$. Hence, by Lemma~\ref{lem:volume ratio},
\begin{multline*}
\Pr\big[\Part(u)=\Part(v)\big]\stackrel{\eqref{def of our parition}}{=}\Pr\Big[\big\{\big(\vec{x},\vec{\gamma}\big)\in \Omega:\ \exists (k,z)\in \N\times \Lambda,\quad \{u,v\}\subset \Gamma^{k,z}\big(\vec{x},\vec{\gamma}\big)\big\}\Big]\\\stackrel{\eqref{eq:non separating event}}{=}\Pr\bigg[\bigcup_{k=1}^\infty\bigcup_{z\in \Lambda}\mathsf{K}^{k,z}_{\mathscr{I}(u,v),\mathscr{U}(u,v)}\bigg]\stackrel{\eqref{eq:hitting set prob}}{=}\frac{\vol_n\big(\mathscr{I}(u,v)\big)}{\vol_n\big(\mathscr{U}(u,v)\big)}
=\frac{\vol_n\big((u+B_{\X})\cap (v+B_{\X})\big)}{2\vol_n(B_{\X})-\vol_n\big((u+B_{\X})\cap (v+B_{\X})\big)}.
\end{multline*}
Hence,
\begin{align}\label{eq:for schmu}
\Pr\big[\Part(u)\neq\Part(v)\big]=\frac{2-2\frac{\vol_n\big((u+B_{\X})\cap (v+B_{\X})\big)}{\vol_n(B_{\X})}}{2-\frac{\vol_n\big((u+B_{\X})\cap (v+B_{\X})\big)}{\vol_n(B_{\X})}}.
\end{align}
Now, by the work~\cite[Corollary~1]{Scm92} of Schmuckenschl\"ager  we have the following general estimates.
\begin{equation}\label{eq:quote schmu}
1-\psi(u-v)\le \frac{\vol_n\big((u+B_{\X})\cap (v+B_{\X})\big)}{\vol_n(B_{\X})}\le e^{-\psi(u-v)},
\end{equation}
where $\psi(\cdot)$ is defined in~\eqref{eq:def psi}. The mapping $t\mapsto (2-2t)/(2-t)$ is decreasing on $[0,1]$, so~\eqref{eq:separating Delta=2 psi} is consequence of~\eqref{eq:for schmu} and~\eqref{eq:quote schmu}. The remaining assertions of Lemma~\ref{lem:separation psi} (in particular the asymptotic evaluation~\eqref{eq:separating Delta=2} of the separation probability) follow from~\eqref{eq:separating Delta=2 psi} by elementary calculus.  Observe that for the purpose of bounding the separation modulus of $\X$ from above, we need only the first inequality in~\eqref{eq:quote schmu}; since it is stated in~\cite{Scm92} but not proved there, for completeness we will include its elementary proof  in Section~\ref{sec:schmu} below. The second inequality in~\eqref{eq:quote schmu} is used here only to show that our bounds are sharp; its proof in~\cite{Scm92} relies on a more substantial use of Brunn--Minkowski theory.

It remains to verify~\eqref{eq:non separating event}. Fix $(k,z)\in \N\times \Lambda$. Suppose first that $(\vec{x},\vec{\gamma})$ is an element of the right hand side of~\eqref{eq:non separating event}. Recalling the definitions~\eqref{eq:H event} and~\eqref{eq:K event}, this implies that $\chi(z)=\gamma_k$ and $z+x_k\in (u+B_{\X})\cap (v+B_{\X})$, while for every $j\in \{1,\ldots,k-1\}$ and $w\in \Lambda$ with $\chi(w)=\gamma_j$ we have $w+x_j\notin (u+B_{\X})\cup (v+B_{\X})$. By the triangle inequality these facts imply that $z+x_k+B_{\X}\supset \{u,v\}$ and the union of the balls $$\big\{w+x_j+B_{\X}:\ (j,w)\in \{1,\ldots,k-1\}\times \Lambda\ \wedge\ \chi(w)=\gamma_j\big\}$$ contains neither of the vectors $u,v$. The definition~\eqref{def of our parition} of $\Gamma^{k,z}(\vec{x},\vec{\gamma})$  now shows that $\{u,v\}\subset \Gamma^{k,z}(\vec{x},\vec{\gamma})$.

For the reverse inclusion, assume that $\{u,v\}\subset \Gamma^{k,z}(\vec{x},\vec{\gamma})$. Then  $\chi(z)=\gamma_k$ and $\{u,v\}\subset z+x_k+B_{\X}$ by~\eqref{def of our parition}, which implies that $z+x_k\in (u+B_{\X})\cap (v+B_{\X})= \mathscr{I}(u,v)$. If there were $j\in \{1,\ldots,k-1\}$ and $w\in \Lambda$ with $\chi(w)=\gamma_j$ such that $(w+x_j+B_{\X})\cap \{u,v\}\neq \emptyset $, then when one subtracts $w+x_j+B_{\X}$ from  $z+x_k+B_{\X}$ one removes at least one of the vectors $u,v$, which by~\eqref{def of our parition} would mean that one of these two vectors is not an element of $\Gamma^{k,z}(\vec{x},\vec{\gamma})$, in contradiction to our assumption. Hence for all $j\in \{1,\ldots,k-1\}$ and $w\in \Lambda$ with $\chi(w)=\gamma_j$ we have $u\notin w+x_j+B_{\X}$ and $v\notin w+x_j+B_{\X}$, i.e., $w+x_j\notin (u+B_{\X})\cup (v+B_{\X})=\mathscr{U}(u,v)$. This shows that $(\vec{x},\vec{\gamma})$ belongs to the the right hand side of~\eqref{eq:non separating event}, thus completing the proof of Lemma~\ref{lem:separation psi}.
\end{proof}

\begin{proof}[Proof of Theorem~\ref{thm:sep and pad at once}] By rescaling, namely considering the norm $(2/\Delta)\|\cdot\|_{\X}$, it suffices to treat the case $\Delta=2$. The desired random partition will then be the partition $\Part$ of Lemma~\ref{lem:is a bounded partition} and the conclusions of Theorem~\ref{thm:sep and pad at once}   follow from Lemma~\ref{lem:padded rho} and Lemma~\ref{lem:separation psi}.
\end{proof}

\subsubsection{Proof of the first inequality in~\eqref{eq:quote schmu}}\label{sec:schmu} The proof of the first inequality in~\eqref{eq:quote schmu} is a simple and elementary application of standard reasoning using Fubini's theorem. Denote
\begin{equation}\label{eq:def x t}
t\eqdef \|v-u\|_{\ell_2^n}\qquad \mathrm{and}\qquad  x\eqdef \frac{1}{t}(v-u)\in S^{n-1}.
\end{equation}
Then, $$
\vol_n\big((u+B_{\X})\cap (v+B_{\X})\big)=\vol_n\big(B_{\X}\cap (tx+B_{\X})\big),
$$
The desired estimate is therefore equivalent to the following assertion.
\begin{equation}\label{eq:desired shift xt}
\vol_n(B_{\X})\le \vol_n\big(B_{\X}\cap (tx+B_{\X})\big)+t\cdot \vol_{n-1}\big(\proj_{x^{\perp}} (B_{\X})\big).
\end{equation}

To prove~\eqref{eq:desired shift xt}, partition $B_{\X}$ into the following three sets.
\begin{align}
U&\eqdef B_{\X}\cap (tx+B_{\X}),\label{eq:def U}\\
V&\eqdef \Big\{y\in B_{\X}\setminus (tx+B_{\X}): \proj_{x^\perp}(y)\in \proj_{x^\perp}(U)\Big\},\label{eq:def V}\\
W&\eqdef B_{\X}\setminus (U\cup V)= \Big\{y\in B_{\X}:\ \proj_{x^\perp}(y)\notin \proj_{x^\perp}(U)\Big\}\label{eq:def W}.
\end{align}
A schematic depiction of this partition, as well as the notation of ensuing discussion, appears in Figure~\ref{fig:translation}. We recommend examining Figure~\ref{fig:translation} while reading the following reasoning because  it consists of a formal justification of a situation that is clear when one keeps the geometric picture in mind.

\begin{figure}[h]
\centering
\fbox{
\begin{minipage}{6.25in}
\centering
%\medskip
\includegraphics[width=5.25in]{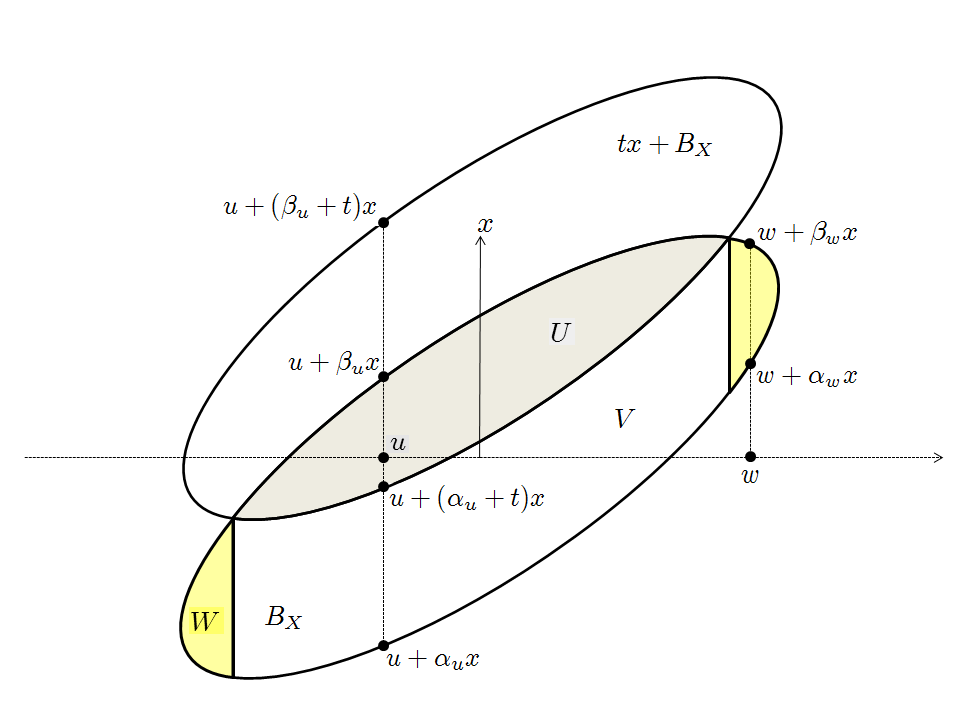}
\smallskip
\caption{A schematic depiction of the partition of $B_{\X}$ into the sets $U,V,W$ (with the sets $U,W$ shaded), as well as the line segments parallel  to $x$ that are used in the justification of the estimate~\eqref{eq:desired shift xt}.
}
\label{fig:translation}
\end{minipage}
}
\end{figure}

For every $z\in \proj_{x^{\perp}} (B_{\X})$ let $\alpha_z\in \R$ be the smallest real number such that $z+\alpha_zx\in B_{\X}$ and let $\beta_z\in \R$ be the largest real number such that $z+\beta_z x\in B_{\X}$. Thus the intersection of the line $z+\R x$ with $B_{\X}$ is the segment $w+[\alpha_z,\beta_z]x\subset \R^n$. Since $\|x\|_{\ell_2^n}=1$, by Fubini's theorem we have
\begin{equation}\label{eq:formula for ball volume}
 \vol_n(B_{\X})=\int_{\proj_{x^{\perp}} (B_{\X})} (\beta_z-\alpha_z)\ud z=
\int_{\proj_{x^{\perp}} (U)} (\beta_u-\alpha_u)\ud u+\int_{\proj_{x^{\perp}} (W)} (\beta_w-\alpha_w)\ud w.
\end{equation}
For the final step of~\eqref{eq:formula for ball volume}, note that by~\eqref{eq:def W} we have $\proj_{x^{\perp}} (B_{\X})=\proj_{x^{\perp}} (U) \cup \proj_{x^{\perp}} (W)$, and the the sets $\proj_{x^{\perp}} (U),\proj_{x^{\perp}} (W)$ have disjoint interiors (in the subspace $x^\perp$).

Since $U= B_{\X}\cap (tx+B_{\X})$ is convex, for every $u$ in the interior of $\proj_{x^\perp}(U)$ the line $u+\R x$ intersects $U$ in an interval, say $(u+ \R x)\cap U=u+[\gamma_u,\delta_u]x$ with $\gamma_u,\delta_u\in \R$ satisfying $\gamma_u< \delta_u$ such that $u+\gamma_ux,u+\delta_ux\in \partial U$ and $u+s x\in \mathbf{int}(U)$ for every $s\in (\gamma_u,\delta_u)$.  Also, $(u+ \R x)\cap B_{\X}=u+[\alpha_u,\beta_u]x$ with $u+\alpha_u x,u+\beta_ux\in \partial B_{\X}$. Thus $[\gamma_u,\delta_u]\subset [\alpha_u,\beta_u]$. Since $u+\gamma_ux\in U\subset tx+B_{\X}$, it follows that $\gamma_w-t\in [\alpha_w,\beta_w]$. But  $\gamma_u\in [\alpha_u,\beta_u]$, so $\beta_u-\alpha_u\ge t$ and therefore $\alpha_u+t,\beta_u-t\in [\alpha_u,\beta_u]$, or equivalently $u+(\alpha_u+t)x,u+(\beta_u-t)x\in B_{\X}$. Because $u+\alpha_u x,u+\beta_ux\in \partial B_{\X}$, we get that $u+(\alpha_u+t)x\in B_{\X}\cap  (tx+\partial B_{\X})\subset \partial U$ and $u+\beta_ux\in (\partial B_{\X})\cap (tx+B_{\X})\subset \partial U$. Hence $\gamma_u=\alpha_u+t$ and $\delta_u=\beta_u$, from which we conclude that
\begin{equation}\label{eq:interval in intersection}
u\in \proj_{x^\perp}(U)\implies (u+ \R x)\cap U=u+[\alpha_u+t,\beta_u]x,
\end{equation}
and therefore also
\begin{equation}\label{eq:interval in V}
u\in \proj_{x^\perp}(U)\implies (u+ \R x)\cap V\stackrel{\eqref{eq:def V}}{=}B_{\X}\setminus \big((u+ \R x)\cap U\big)\stackrel{\eqref{eq:interval in intersection}}{=}u+[\alpha_u,\alpha_u+t]x.
\end{equation}
Another application of Fubini's theorem now implies that
\begin{align}\label{eq:second fubini}
\begin{split}
\int_{\proj_{x^{\perp}} (U)} (\beta_u-\alpha_u)\ud u &=\int_{\proj_{x^{\perp}} (U)} \vol_1\big((u+ \R x)\cap U\big)\ud u+\int_{\proj_{x^{\perp}} (U)} t\ud u
\\&=\vol_n(U)+t\vol_{n-1}\big(\proj_{x^{\perp}} (U)\big) \\&=\vol_n(U)+t\Big(\vol_{n-1}\big(\proj_{x^{\perp}} (B_{\X})\big)-\vol_{n-1}\big(\proj_{x^{\perp}} (W)\big)\Big),
\end{split}
\end{align}
where the first step of~\eqref{eq:second fubini} uses~\eqref{eq:interval in intersection} and~\eqref{eq:interval in V} and for the last step of~\eqref{eq:second fubini} recall the definition~\eqref{eq:def W}.

Observe next that
\begin{equation}\label{eq:on W}
w\in \proj_{x^{\perp}}(W)\implies \beta_w-\alpha_w\le t.
\end{equation}
Indeed, if $w\in \proj_{x^{\perp}}(W)$ yet  $\beta_w-\alpha_w>t$ then $w+(\beta_w-t)x$ belongs to the interval joining $w+\alpha_wx$ and $w+\beta_wx$. By the convexity of $B_{\X}$ we therefore  have $w+(\beta_w-t)x\in B_{\X}$, or equivalently $w+\beta_wx\in tx+B_{\X}$. Recalling  that $w+\beta_w x\in B_{\X}$, this means that $w+\beta_w x\in B_{\X}\cap (tx+B_{\X})$. By the definition~\eqref{eq:def U} of $U$, it follows that $w\in \proj_{x^\perp}(U)$. By the definition~\eqref{eq:def W} of $W$, this means that $w\notin \proj_{x^\perp}(W)$, a contradiction.

Having established~\eqref{eq:on W} we see that
\begin{equation}\label{eq:use t bound}
\int_{\proj_{x^{\perp}} (W)} (\beta_w-\alpha_w)\ud w\stackrel{\eqref{eq:on W}}{\le} t\vol_{n-1}\big(\proj_{x^{\perp}} (W)\big).
\end{equation}
The desired estimate~\eqref{eq:desired shift xt} now follows from a substitution of~\eqref{eq:second fubini} and~\eqref{eq:use t bound} into~\eqref{eq:formula for ball volume}.\qed

\subsection{Proof of Theorem~\ref{eq:finite ellp}}\label{sec:finite Lp} For any $m\in \N$, because  $\evr(\ell_1^m)\asymp \sqrt{m}$, by the second part~\eqref{eq:sep lower dim reduction} of Theorem~\ref{thm:sep lower with cardinality size}  there exists $\sub\subset \R^m$ with $|\sub|\le e^{\beta m}$ for some universal constant $\beta>0$ such that $\sep(\sub_{\ell_1^m}) \gtrsim m$ (as we are considering here $\ell_1^m$ rather than more general normed spaces, this statement is due~\cite{CCGGP98}). Fix an integer $n\ge 2$ and $1\le p\le 2$. Let $m$ be the largest integer such that $e^{\beta m}\le n$. Thus $m\asymp \log n$ and
$$
\sep^n(\ell_p)\ge \sep\big(\sub_{\ell_p^m}\big)\ge \frac{\sep\big(\sub_{\ell_1^m}\big)}{d_{\BM}\big(\ell_1^m,\ell_p^m\big)}\gtrsim \frac{m}{d_{\BM}\big(\ell_1^m,\ell_p^m\big)}=m^{\frac{1}{p}}\asymp (\log n)^{\frac{1}{p}}.
$$
This proves the lower bound on $\sep^n(\ell_p)$ in Theorem~\ref{eq:finite ellp}.

It remains to prove the upper bound on $\sep^n(\ell_p)$ in Theorem~\ref{eq:finite ellp}, i.e., that for all $x_1,\ldots,x_n\in \ell_p$,
\begin{equation}\label{eq:separating finite Lp}
\sep\big(\{x_1,\ldots,x_n\},\|\cdot\|_{\ell_p}\big)\lesssim \frac{(\log n)^{\frac{1}{p}}}{p-1}.
\end{equation}
The proof of~\eqref{eq:separating finite Lp} will refer to the following technical probabilistic lemma.
\begin{lemma}\label{lem:laplace} Suppose that $p\in (1,\infty)$ and let $\XX$ be a nonnegative random variable, defined on some probability space $(\Omega,\Pr)$, that satisfies the following Laplace transform identity.
\begin{equation}\label{eq:laplace}
\forall  u\in [0,\infty),\qquad \E\Big[e^{-u\XX^2}\Big]=e^{-u^{\frac{p}{2}}}.
\end{equation}
Then
\begin{equation}\label{eq:expectation X}
\E[\XX]=\frac{\Gamma\big(1-\frac{1}{p}\big)}{\sqrt{\pi}}\asymp \frac{p}{p-1}.
\end{equation}
Moreover, we have
\begin{equation}\label{eq:small ball}
\forall  t\in (0,\infty),\qquad \Pr\big[\XX\le t\big]\le \exp\left(-\frac{\left(\frac{p}{2}\right)^{\frac{p}{2-p}}\left(1-\frac{p}{2}\right)}{t^{\frac{2p}{2-p}}}\right).
\end{equation}
\end{lemma}

\begin{proof}
Suppose that $\alpha\in (0,1)$. Then every $x\in (0,\infty)$ satisfies
\begin{equation}\label{eq:integration by parts}
\int_0^\infty \frac{1-e^{-ux}}{u^{1+\alpha}}\ud x =x^\alpha\int_0^\infty \frac{1-e^{-v}}{v^{1+\alpha}}\ud x= \frac{\Gamma(1-\alpha)}{\alpha}x^\alpha,
\end{equation}
where the first step of~\eqref{eq:integration by parts} is a straightforward change of variable and the last step of~\eqref{eq:integration by parts} follows by integration by parts.  The case $\alpha=1/2$ of~\eqref{eq:integration by parts} implies~\eqref{eq:expectation X} as follows.
\begin{multline*}
\E[\XX]=\E\bigg[\frac{1}{2\sqrt{\pi}}\int_0^\infty \frac{1-e^{-u\XX^2}}{u^{\frac{3}{2}}}\ud u\bigg]=\frac{1}{2\sqrt{\pi}}\int_0^\infty \frac{1-\E\big[e^{-u\XX^2}\big]}{u^{\frac{3}{2}}}\ud u\\\stackrel{\eqref{eq:laplace}}{=}\frac{1}{2\sqrt{\pi}}\int_0^\infty \frac{1-e^{-u^{\frac{p}{2}}}}{u^{\frac{3}{2}}}\ud u=\frac{1}{p\sqrt{\pi}}\int_0^\infty \frac{1-e^{-v}}{v^{1+\frac{1}{p}}}\ud v \stackrel{\eqref{eq:integration by parts}}{=}\frac{\Gamma\big(1-\frac{1}{p}\big)}{\sqrt{\pi}}.
\end{multline*}

The  small ball probability estimate~\eqref{eq:small ball} is a consequence of the following standard use of Markov's inequality. For every $u,t\in (0,\infty)$ we have
\begin{equation}\label{to optimize u}
\Pr\big[\XX\le t\big]=\Pr\Big[e^{-u\XX^2}\ge e^{-ut^2}\Big]\le e^{ut^2}\E\Big[e^{-u\XX^2}\Big]=e^{ut^2-u^{\frac{p}{2}}}.
\end{equation}
The value of $u\in (0,\infty)$ that minimizes the right-hand side of~\eqref{to optimize u} is
$$
u=u(p,t)\eqdef \left(\frac{p}{2t^2}\right)^{\frac{2}{2-p}}.
$$
A substitution of this value of $u$ into~\eqref{to optimize u} simplifies to give the desired estimate~\eqref{eq:small ball}.
\end{proof}

\begin{proof}[Proof of~\eqref{eq:separating finite Lp}] Fix distinct $x_1,\ldots,x_n\in \ell_p$. It suffices to prove~\eqref{eq:separating finite Lp} when  $p\in (1,2)$, since the quantity that appears in the right-hand side of~\eqref{eq:separating finite Lp} remains bounded as $p\to 2^{-}$, and every finite subset of $\ell_2$ embeds isometrically into $\ell_p$ for every $p\in [1,2]$ (see e.g.~\cite[Chapter~III.A]{Woj91}). We will therefore assume in the remainder of the proof of~\eqref{eq:separating finite Lp}  that $p\in (1,2)$.

Marcus and Pisier proved~\cite[Section~2]{MP84} the following  statement, relying on a structural result for $p$-stable processes; its deduction  from the formulation in~\cite{MP84} appears in~\cite[Lemma~2.1]{LMN05}). There is a probability space $(\Omega,\Pr)$ and a $\Pr$-to-Borel measurable mapping $(\omega\in \Omega)\mapsto T_\omega\in \mathcal{L}({\ell}_p,{\ell}_2)$ (here $\mathcal{L}({\ell}_p,{\ell}_2)$ is the space of bounded operators from ${\ell}_p$ to ${\ell}_2$, equipped with the strong operator topology) such that for every $\omega\in \Omega$ and $x\in {\ell}_p\setminus \{0\}$ the random variable
\begin{equation}\label{eq:X comes from normalized}
(\omega\in \Omega)\mapsto \frac{\|T_\omega(x)\|_{{\ell}_2}}{\|x\|_{{\ell}_p}}
\end{equation}
has the same distribution as the random variable $\XX$ of Lemma~\ref{lem:laplace} (in particular, its distribution is independent of the choice of $x\in {\ell}_p\setminus \{0\}$). Consequently,
\begin{equation}\label{eq:use expectation X}
\forall  i,j\in \n,\qquad  \int_\Omega \big\|T_\omega(x_i)-T_\omega(x_j)\|_{{\ell}_2} \ud\Pr(\omega) =\|x_i-x_j\|_{{\ell}_p} \E[\XX]\stackrel{\eqref{eq:expectation X}}{\asymp} \frac{\|x_i-x_j\|_{{\ell}_p}}{p-1}.
\end{equation}

It also follows from the above discussion and Lemma~\ref{lem:laplace} that for every $t\in (0,\infty)$ we have
\begin{align}\label{eq:to choose t depend on n}
\begin{split}
\Pr\bigg[&\bigcup_{i,j\in \n}\Big\{\omega\in \Omega:\ \|T_\omega(x_i)-T_\omega(x_j)\|_{{\ell}_2}\ge t\|x_i-x_j\|_{{\ell}_p}\Big\}\bigg]\\&\le \sum_{\substack{i,j\in \n\\i\neq j}}\Pr\Big[\Big\{\omega\in \Omega:\ \frac{\|T_\omega(x_i)-T_\omega(x_j)\|_{{\ell}_2}}{\|x_i-x_j\|_{{\ell}_p}}<t\Big\}\Big]\stackrel{\eqref{eq:small ball}}{\le} \binom{n}{2}\exp\left(-\frac{\left(\frac{p}{2}\right)^{\frac{p}{2-p}}\left(1-\frac{p}{2}\right)}{t^{\frac{2p}{2-p}}}\right).
\end{split}
\end{align}
If we choose
$$
t=t(n,p)\eqdef  \sqrt{\frac{p}{2}}\left(\frac{2-p}{4\log n}\right)^{\frac{1}{p}-\frac12},
$$
then the right hand side of~\eqref{eq:to choose t depend on n} becomes less than $1/2$. In other words, this shows that there exists a measurable subset $A\subset \Omega$ with $\Pr[A]\ge 1/2$ such that for every $\omega\in A$ and $i,j\in \n$,
\begin{equation}\label{eq:T non contracting}
 \|x_i-x_j\|_{{\ell}_p}\le \sqrt{\frac{2}{p}}\left(\frac{4\log n}{2-p}\right)^{\frac{1}{p}-\frac12} \|T_\omega(x_i)-T_\omega(x_j)\|_{{\ell}_2}\le 4(\log n)^{\frac{1}{p}-\frac12}\|T_\omega(x_i)-T_\omega(x_j)\|_{{\ell}_2},
\end{equation}
where the last step of~\eqref{eq:T non contracting} uses the elementary inequality $(2/(2-p))^{(2-p)/(2p)}\sqrt{2/p}\le 4$, which holds (with room to spare) for every $p\in [1,2)$.

$\{T_\omega(x_1),\ldots,T_\omega(x_n)\}\subset {\ell}_2$ is a subset of Hilbert space of size  at most $n$, so by the Johnson--Lindenstrauss dimension reduction lemma~\cite{JL84} there is $k\in \N$ with $k\lesssim \log n$ such that for every $\omega\in \Omega$ there is a linear operator $Q_\omega:{\ell}_2\to \R^k$ such that for all $i,j\in \n$,
\begin{equation}\label{eq:use JL}
 \|T_\omega(x_i)-T_\omega(x_j)\|_{{\ell}_2}\le \|Q_\omega T_\omega(x_i)-Q_\omega T_\omega(x_j)\|_{\ell_2^k}\le 2\|T_\omega(x_i)-T_\omega(x_j)\|_{{\ell}_2}.
\end{equation}
An examination of the proof in~\cite{JL84} reveals that the mapping $\omega\mapsto Q_\omega$ can be taken to be $\Pr$-to-Borel  measurable, but actually $Q_\omega$ can be chosen from a finite set of operators (see e.g.~\cite{Ach03}).

Fix $\Delta\in (0,\infty)$. Since by~\cite{CCGGP98} we have $\sep(\ell_2^k)\lesssim \sqrt{k}$, there exists a probability space $(\Theta,\mu)$ and a mapping $\theta\mapsto \mathscr{R}^\theta$ that is a random partition of $\R^k$ for which
\begin{equation}\label{eq:the diameter in random l2}
\forall(\omega,\theta,i)\in \Omega\times \Theta\times  \n,\qquad \diam_{\ell_2^k} \Big(\mathscr{R}^\theta \big(Q_\omega T_\omega(x_i)\big)\Big)\le \frac{\Delta}{4(\log n)^{\frac{1}{p}-\frac12}},
\end{equation}
and also every $\omega\in \Omega$ and $i,j\in \n$ satisfy
\begin{align}\label{eq:separation in random l2}
\begin{split}
\mu\Big(\Big\{\theta\in \Theta:\  \mathscr{R}^\theta \big(Q_\omega T_\omega(x_i)\big)\neq \mathscr{R}^\theta \big(Q_\omega T_\omega(x_j)\big)\Big\}\Big)&\lesssim \frac{\sqrt{k}}{\Delta/ \Big(4(\log n)^{\frac{1}{p}-\frac12}\Big)}\big\|Q_\omega T_\omega(x_i)-Q_\omega T_\omega(x_i)\big\|_{\ell_2^k}\\
 &\lesssim \frac{(\log n)^{\frac{1}{p}}}{\Delta}\big\|T_\omega(x_i)- T_\omega(x_i)\big\|_{{\ell}_2},
 \end{split}
\end{align}
where the last step of~\eqref{eq:separation in random l2} uses the right-hand inequality in~\eqref{eq:use JL}  and the fact that $k\lesssim \log n$.

Recalling the set $A\subset \Omega$ on which~\eqref{eq:T non contracting} holds for every $i,j\in \n$, let $\nu$ be the probability measure on $A$ defined by $\nu[E] =\Pr[E]/\Pr[A]$ for every $\Pr$-measurable $E\subset A$ (recall that $\Pr[A]\ge 1/2$). For every $(\omega,\theta)\in A\times \Theta$ define a partition $\mathscr{P}^{(\omega,\theta)}$ of $\{x_1,\ldots,x_n\}$ as follows.
\begin{equation}\label{eq:def pullback partition}
\forall  i\in \n,\qquad \mathscr{P}^{(\omega,\theta)}(x_i)\eqdef \Big\{x\in \{x_1,\ldots,x_n\}:\ Q_\omega T_\omega (x)\in \mathscr{R}^\theta \big(Q_\omega T_\omega(x_i)\big)\Big\}.
\end{equation}
Then, for every $(\omega,\theta)\in A\times \Theta$ and every $i\in \n$ we have
\begin{align}\label{eq:diameter pullback}
\begin{split}
 \diam_{{\ell}_p}\big(\mathscr{P}^{(\omega,\theta)}(x_i)\big)&=
 \max_{\substack{u,v\in \n\\ Q_\omega T_\omega (x_u),Q_\omega T_\omega (x_v)\in \mathscr{R}^\theta \big(Q_\omega T_\omega(x_i)\big)}}\|x_u-x_v\|_{{\ell}_p}\\&\le 4(\log n)^{\frac{1}{p}-\frac12}\max_{\substack{u,v\in \n\\ Q_\omega T_\omega (x_u),Q_\omega T_\omega (x_v)\in \mathscr{R}^\theta \big(Q_\omega T_\omega(x_i)\big)}}\|T_\omega(x_u)-T_\omega(x_v)\|_{{\ell}_2}
\\&\le 4(\log n)^{\frac{1}{p}-\frac12}\max_{\substack{u,v\in \n\\ Q_\omega T_\omega (x_u),Q_\omega T_\omega (x_v)\in \mathscr{R}^\theta \big(Q_\omega T_\omega(x_i)\big)}}\|Q_\omega T_\omega(x_u)-Q_{\omega}T_\omega(x_v)\|_{\ell_2^k}\\
&\le 4(\log n)^{\frac{1}{p}-\frac12}\diam_{\ell_2^k} \Big(\mathscr{R}^\theta \big(Q_\omega T_\omega(x_i)\big)\Big)\le \Delta,
\end{split}
\end{align}
where the first step of~\eqref{eq:diameter pullback} uses~\eqref{eq:def pullback partition}, the second step of~\eqref{eq:diameter pullback} uses~\eqref{eq:T non contracting}, the third step of~\eqref{eq:diameter pullback} uses~\eqref{eq:use JL}, and the final step of~\eqref{eq:diameter pullback} uses~\eqref{eq:the diameter in random l2}.
Also, every distinct $i,j\in \n$ satisfy
\begin{align}\label{eq:separation pullback}
\begin{split}
\nu\times \mu \Big(\Big\{(\omega,\theta)&\in A\times \Theta:\ \mathscr{P}^{(\omega,\theta)}(x_i)\neq \mathscr{P}^{(\omega,\theta)}(x_j)\Big\}\Big)\\ &=\int_A \mu\Big(\Big\{\theta\in \Theta:\  \mathscr{R}^\theta \big(Q_\omega T_\omega(x_i)\big)\neq \mathscr{R}^\theta \big(Q_\omega T_\omega(x_j)\big)\Big\}\Big) \ud \nu(\omega) \\
&\lesssim \frac{1}{\Pr[A]} \int_A \frac{(\log n)^{\frac{1}{p}}}{\Delta}\big\|T_\omega(x_i)- T_\omega(x_i)\big\|_{{\ell}_2} \dd\Pr(\omega)\\
&\le \frac{2(\log n)^{\frac{1}{p}}}{\Delta} \int_\Omega \big\|T_\omega(x_i)- T_\omega(x_i)\big\|_{{\ell}_2}\ud\Pr(\omega)\\
&\lesssim  \frac{(\log n)^{\frac{1}{p}}}{p-1}\cdot \frac{\|x_i-x_j\|_{{\ell}_p}}{\Delta},
\end{split}
\end{align}
where the first step of~\eqref{eq:separation pullback} uses~\eqref{eq:def pullback partition},  the second step of~\eqref{eq:separation pullback} uses~\eqref{eq:separation in random l2}, the third step of~\eqref{eq:separation pullback} uses  $\Pr[A]\ge \frac12$, and the last step of~\eqref{eq:separation pullback} uses~\eqref{eq:use expectation X}. By~\eqref{eq:diameter pullback} and~\eqref{eq:separation pullback}, the proof of~\eqref{eq:separating finite Lp} is complete.
\end{proof}

\section{Barycentric-valued Lipschitz extension}\label{sec:ext}

In this section, we will explain how separation profiles relate to Lipschitz extension. We cannot invoke~\cite{LN05} as a ``black box'' because we need a more general result and our  definition of random partitions differs from that of~\cite{LN05}. But, the modifications that are required in order to apply the ideas of~\cite{LN05} in the present setting are of a secondary nature, and the main geometric content of the phenomenon that is explained below is the same as in~\cite{LN05}.

In addition to making the present article self-contained, there are more advantages to including here complete  proofs of Theorem~\ref{thm:local compact ext} and Theorem~\ref{thm:polish extension}. Firstly, the reasoning of~\cite{LN05} was designed to deal with a more general setting (treating multiple notions of random partitions at once), and it is illuminating to present a proof for separating decompositions in isolation, which leads to simplifications. Secondly, since~\cite{LN05} appeared, alternative viewpoints have been developed that relate it to optimal transport, as carried  out by Kozdoba~\cite{Koz05}, Brudnyi and Brudnyi~\cite{BB07-2}, Ohta~\cite{Oht09},  and culminating more recently with a comprehensive treatment by Ambrosio and Puglisi~\cite{AP20}. Here we will   frame the construction using the optimal transport methodology, which has conceptual advantages that go beyond yielding a  clearer restructuring of the argument. The optimal transport viewpoint had an important role in  quantitative improvements that were obtained in~\cite{NR17,Nao20-almost}, as well as  results that will appear in forthcoming works. As a byproduct, we will use this viewpoint  to easily derive a stability statement for convex hull-valued Lipschitz extension under metric transforms.

\subsection{Notational preliminaries}\label{sec:transport notation} We will start by quickly setting notation and terminology for basic concepts in measure theory and optimal transport. Everything that we describe in this subsection is  standard and is included here only in order to avoid any ambiguities in the subsequent discussions.

Given a signed measure $\mu$ on a measurable space $(\Omega,\mathscr{F})$, its Hahn--Jordan decomposition is denoted $\mu=\mu^+-\mu^-$, i.e., $\mu^+,\mu^-$ are  disjointly supported nonnegative measures. The total variation measure of $\mu$ is $|\mu|=\mu^++\mu^-$. For $A\in \mathscr{F}$, the restriction of $\mu$ to $A$ is denoted $\mu\lfloor_A$, i.e., $\mu\lfloor_A(E)=\mu(A\cap E)$ for  $E\in \mathscr{F}$. If $(\Omega',\mathscr{F}')$ is another measurable space and $f:\Omega\to \Omega'$ is a measurable mapping, then  the push-forward of $\mu$ under $f$ is denoted $f_{\bf \#}\mu$. Thus $f_{\bf \#}\mu(E)=\mu(f^{-1}(E))$ for $E\in \mathscr{F}'$, or equivalently $$\forall h\in L_1(f_{\bf \#}\mu),\qquad \int_{\Omega'}h\big(\omega'\big) \ud f_{\bf \#}\mu\big(\omega'\big)=\int_{\Omega} h\big(f(\omega)\big)\ud\mu(\omega).$$

Suppose from now on that $(\MM,d_\MM)$ is a Polish metric space. A signed Borel measure $\mu$ on $\MM$ has finite first moment if $\int_\MM d_\MM(x,y)\ud|\mu|(y)<\infty$ for all $x\in \MM$. Note that this implies in particular that $|\mu|(\MM)<\infty$, because if $x,x'\in \MM$ are distinct points, then the mapping $(y\in \MM)\mapsto [d_\MM(x,y)+d_\MM(x',y)]/d_\MM(x,x')$ belongs to $L_1(|\mu|)$ and  takes values in $[1,\infty)$ by the triangle inequality.
%$$
%|\mu|(\MM)\le \int_{\MM} \frac{d_\MM(x,y)+d_\MM(x',y)}{d_\MM(x,x')}\ud|\mu|(y)<\infty
%$$

 The set of signed Borel measures on $\MM$ of finite first moment is denoted $\mathsf{M}_1(\MM,d_\MM)$ or simply $\mathsf{M}_1(\MM)$ if the metric is clear from the context. The set of all  nonnegative measures in $\mathsf{M}_1(\MM)$ is denoted $\mathsf{M}_1^+(\MM)$, the set of all $\mu\in \mathsf{M}_1(\MM)$ with total mass $0$, i.e., $\mu^+(\MM)=\mu^-(\MM)$, is denoted $\mathsf{M}_1^0(\MM)$,  and the set of all  probability measures in $\mathsf{M}_1(\MM)$ is denoted $\mathsf{P}_1(\MM)$.

Given $\mu,\nu\in \mathsf{M}^+_1(\MM)$ with $\mu(\MM)=\nu(\MM)$, a Borel measure $\pi$ on $\MM\times \MM$ is a coupling of $\mu$ and $\nu$ if $$\pi(E\times \MM)=\mu(A)\qquad \mathrm{and} \qquad \pi(\MM\times E)=\nu(A)$$ for every Borel subset $E\subset \MM$. The set of couplings of $\mu$ and $\nu$ is denoted $\Pi(\mu,\nu)\subset \mathsf{M}^+_1(\MM\times \MM)$. Note that $(\mu\times \nu)/\mu(\MM)=(\mu\times \nu)/\nu(\MM)\in \Pi(\mu,\nu)$, so $\Pi(\mu,\nu)\neq\emptyset $.  The Wasserstein-1 distance between $\mu$ and $\nu$ that is induced by the metric $d_\MM$, denoted $\W_1^{d_\MM}(\mu,\nu)$ or simply $\W_1(\mu,\nu)$ if the metric is clear from the context, is the infimum of $\int_{\MM\times \MM} d_\MM(x,y)\ud \pi(x,y)$ over all possible couplings $\pi\in \Pi(x,y)$. Since $(\MM,d_\MM)$ is Polish, the metric space $(\mathsf{P}_1(\MM),\W_1)$ is also Polish; see e.g.~\cite{Bol08} or~\cite[Proposition~7.1.5]{AGS08}. Throughout what follows, $\mathsf{P}_1(\MM)$ will be assumed to be equipped with the metric $\W_1$.  The Kantorovich--Rubinstein duality theorem (see e.g.~\cite[Theorem~5.10]{Vil09}) asserts that
 \begin{equation}\label{eq:KR}
 \W_1(\mu,\nu)=\sup_{\substack{\psi:\MM\to \R\\ \|\psi\|_{\Lip(\MM)= 1}}}\bigg(\int_\MM \psi \ud \mu-\int_\MM \psi \ud \nu\bigg).
 \end{equation}
Note that~\eqref{eq:KR} implies in particular  that $ \W_1(\mu+\tau,\nu+\tau)= \W_1(\mu,\nu)$ for every $\tau\in \mathsf{M}^+_1(\MM)$.

 For $\mu\in \mathsf{M}_1^0(\MM)$ we have $\mu^+(\MM)=\mu^-(\MM)$, so we can  define $\|\mu\|_{\W_1(\MM)}= \W_1(\mu^+,\mu^-)$.\footnote{Note for later use  that if $\mu,\nu \in \mathsf{M}^+_1(\MM)$ satisfy $\mu(\MM)=\nu(\MM)$, then $\mu-\nu\in \mathsf{M}_1^0(\MM)$ and  $\|\mu-\nu\|_{\W_1(\MM)}=\W_1(\mu,\nu)$. For a standard justification of the latter assertion, see e.g.~the simple deduction of equation (2.2) in~\cite{NS07}.}  This turns $\mathsf{M}_1^0(\MM)$ into a normed space whose completion is called the {\em free space} over $\MM$ (also known as the Arens--Eells space over $\MM$), and is denoted  $\mathfrak{F}(\MM)$; see~\cite{AE56,Wea99,God15} for more on this topic, and note that while $\mathfrak{F}(\MM)$ is commonly defined as the closure of the {\em finitely supported} measures in $\mathsf{M}_1^0(\MM)$ with respect to the Wasserstein-$1$ norm, since the finitely supported measures are dense in $\mathsf{M}_1^0(\MM)$ (see e.g.~\cite[Theorem~6.18]{Vil09}), the definitions coincide.  It follows from~\eqref{eq:KR}  that the dual of $\mathfrak{F}(\MM)$ is canonically isometric to the space of all the real-valued Lipschitz functions on $\MM$ that vanish at some (arbitrary but fixed) point $x_0\in \MM$, equipped with the norm $\|\cdot\|_{\Lip(\MM)}$.

Suppose that $(\bfZ,\|\cdot\|_\bfZ)$ is a separable Banach space and fix $\mu\in \mathsf{M}_1(\MM)$. By the Pettis measurability criterion~\cite{Pet38} (see also~\cite[Proposition~5.1]{BL00}), any $f\in \Lip(\MM;\bfZ)$ is $|\mu|$-measurable. Moreover, we have $\|f\|_\bfZ\in L_1(|\mu|)$ because if we fix $x\in \MM$, then for every $y\in \MM$, $$\|f(y)\|_\bfZ\le \|f(y)-f(x)\|_\bfZ+\|f(x)\|_{\X}\le \|f\|_{\Lip(\MM;\bfZ)}d_\MM(y,x)+\|f(x)\|_{\X}\in L_1(|\mu|),$$
where the last step holds by the definition of $\mathsf{M}_1(\MM)$ and the fact that it implies that $|\mu|(\MM)<\infty$.  By Bochner's integrability criterion~\cite{Boc33} (see also~\cite[Proposition~5.2]{BL00}), it follows  that the Bochner integrals $\int_\MM f\ud\mu^+$ and $\int_\MM f\ud\mu^-$ are well-defined elements of $\bfZ$, so we can consider  the vector
\begin{equation}\label{eq:def If}
\mathfrak{I}_{\!f}(\mu)\eqdef \int_\MM f\ud \mu=\int_\MM f\ud\mu^+-\int_\MM f\ud\mu^-\in \bfZ.
\end{equation}
 If $\mu\in \mathsf{M}_1^0(\MM)$, then $\mathfrak{I}_{\! f}(\mu)=\int_{\MM\times \MM} (f(x)-f(y))\ud \pi(x,y)$ for every coupling $\pi\in \Pi(\mu^+,\mu^-)$. Consequently, $\|\mathfrak{I}_{\! f}(\mu)\|_\bfZ\le \|f\|_{\Lip(\MM;\bfZ)}\int_{\MM\times \MM} d_\MM(x,y)\ud \pi(x,y)$, so by taking the infimum over all $\pi\in \Pi(\mu^+,\mu^-)$ we see that the norm of the linear operator $\mathfrak{I}_{\! f}$ from  $(\mathsf{M}_1^0(\MM),\|\cdot\|_{\W_1})$ to $\bfZ$ satisfies
 \begin{equation}\label{eq:norm of integral}
 \|\mathfrak{I}_{\! f}\|_{(\mathsf{M}_1^0(\MM),\|\cdot\|_{\W_1})\to \bfZ}\le \|f\|_{\Lip(\MM;\bfZ)}.
 \end{equation}
Since $\mathsf{M}_1^0(\MM)$ is dense in $\mathfrak{F}(\MM)$, it follows that $\mathfrak{I}_{\! f}$ extends uniquely to a linear operator $\mathfrak{I}_{\! f}:\mathfrak{F}(\MM)\to \bfZ$ of norm at most $\|f\|_{\Lip(\MM;\bfZ)}$. So, even though elements of $\mathfrak{F}(\MM)$ need not be measures, one can consider the ``integral'' $\mathfrak{I}_{\! f}(\phi)\in \bfZ$ of $f\in \Lip(\MM;\bfZ)$ with respect to $\phi\in \mathfrak{F}(\MM)$; see~\cite{GK03} for more on this topic.

\subsection{Refined extension moduli}\label{sec:ext refined} Continuing with the notation that was introduced by Matou\v{s}ek~\cite{Mat90}, we will consider the following   parameters related to Lipschitz extension. Suppose that $(\MM,d_\MM), (\NN,d_\NN)$ are metric spaces and that $\sub\subset \MM$. Denote by $\ee(\MM,\sub;\NN)$ the infimum over those $K\in [1,\infty]$ such that for every $f:\sub\to \NN$ with $\|f\|_{\Lip(\sub;\NN)}<\infty$ there is $F:\MM\to \NN$ that extends $f$ and satisfies $$\|F\|_{\Lip(\MM;\NN)}\le K\|f\|_{\Lip(\sub;\NN)}.$$
The supremum of $\ee(\MM,\sub;\NN)$ over all subsets $\sub\subset \MM$ will be denotes $\ee(\MM;\NN)$. Note that when $\NN$  is complete, $\NN$-valued Lipschitz functions on $\sub$ automatically extend to the closure of $\sub$ while preserving the Lipschitz constant, so we may assume here that $\sub$ is closed. The supremum of $\ee(\MM,\sub;\bfZ)$ over all Banach spaces $(\bfZ,\|\cdot\|_\bfZ)$ will be denoted below by $\ee(\MM,\sub)$. Thus, the notation $\ee(\MM)$ of the Introduction coincides with the supremum of $\ee(\MM,\sub)$ over all subsets $\sub\subset \MM$.

If $(\MM,d_\MM)$ is a metric space, $\sub\subset \MM$, and $(\bfZ,\|\cdot\|_\bfZ)$ is a Banach space, then it is natural  to consider variants of the above definitions with the additional restrictions that the extended mapping $F$ is required to take values in either the closure of the linear span of $f(\sub)$ or the closure of the convex hull of $f(\sub)$. Namely,  let $\ee_{\spn}(\MM,\sub;\bfZ)$  be the infimum over those $K\in [1,\infty]$ such that for every $f:\sub\to \bfZ$ there exists $$F:\MM\to \overline{\spn}\big(f(\sub)\big)$$ that extends $f$ and satisfies
\begin{equation}\label{eq:extension F f}
\|F\|_{\Lip(\MM;\bfZ)}\le K\|f\|_{\Lip(\sub;\bfZ)}.
\end{equation}

Analogously, let $\ee_{\conv}(\MM,\sub;\bfZ)$ be the infimum over  $K\in [1,\infty]$ such that for every $f:\sub\to \bfZ$ there exists $$F:\MM\to \overline{\conv}\big(f(\sub)\big)$$ that extends $f$ and satisfies~\eqref{eq:extension F f}. We then define $\ee_\conv(\MM,\sub)$ to be the supremum of $\ee_{\conv}(\MM,\sub;\bfZ)$ over all possible Banach spaces $(\bfZ,\|\cdot\|_\bfZ)$. Note that while one could attempt to define $\ee_\spn(\MM,\sub)$ similarly, there is no point to do so because it would result in the previously defined quantity   $\ee(\MM,\sub)$. By considering the supremum  of $\ee_\conv(\MM,\sub)$ over all subsets $\sub\subset \MM$, one defines the quantity $\ee_\conv(\MM)$.

\begin{remark}\label{rem:lindenstrauss span}
By~\cite{Lin64}  one can have $\ee(\MM,\sub;\bfZ)=\ee(\MM;\bfZ)=1$ yet $\ee_{\spn}(\MM,\sub,\bfZ)=\infty$ for some metric space $(\MM,d_\MM)$, some $\sub\subset \MM$ and some Banach space $(\bfZ,\|\cdot\|_\bfZ)$. Indeed, if $\X$ is a closed reflexive subspace of $\ell_\infty$ and $\mathbf{V}\subset \X$ is a closed uncomplemented subspace of $\X$, then by~\cite{Lin64} (see also~\cite[Corollary~7.3]{BL00}) there is no Lipschitz retraction from $\X$ onto $\mathbf{V}$. Equivalently, the identity mapping from $\mathbf{V}$ to $\mathbf{V}$ cannot be extended to a Lipschitz mapping from $\X$ to $\mathbf{V}$. Hence, since $\spn(\mathbf{V})=\mathbf{V}\subset \ell_\infty$, we have  $\ee_\spn(\X,\mathbf{V};\ell_\infty)=\infty$. In contrast, $\ee(\X;\ell_\infty)=1$ by the nonlinear Hahn--Banach theorem (see~\cite{McS34} or e.g.~\cite[Lemma~1.1]{BL00}). By combining~\cite{Sob41} with the discretization method of~\cite{JL84} (see also~\cite{MM16}), one can quantify the above example by showing that for arbitrarily large $n\in \N$ there are Banach spaces $(\X,\|\cdot\|_{\X}), (\bfZ,\|\cdot\|_\bfZ)$ and a subset $\sub\subset \X$ with $|\sub|=n$ for which we have
\begin{equation}\label{eq:lower conv in terms of cardinality}
\frac{\ee_\spn (\X,\sub;\bfZ)}{\ee (\X,\sub;\bfZ)}\gtrsim \sqrt{\frac{\log n}{\log\log n}}.
\end{equation}
(In fact, in~\eqref{eq:lower conv in terms of cardinality} one can have $\ee (\X,\sub;\bfZ)=\ee (\X;\bfZ)=1$.) At present, the right hand side of~\eqref{eq:lower conv in terms of cardinality} is the largest asymptotic dependence on $n$ that we are able to obtain for this question, and it remains an  interesting open problem to determine the best possible asymptotics here.
\end{remark}

Most, but not all, of the Lipschitz extension methods in the literature, including Kirszbraun's extension theorem~\cite{Kirsz34}, Ball's extension theorem~\cite{Bal92} and methods that rely on (variants of) partitions of unity such as in~\cite{JLS86,LN05,LS05,BB06}, yield convex hull-valued extensions, i.e., they actually provide bounds on the quantity $\ee_{\conv}(\MM,\sub;\bfZ)$. Nevertheless, it seems likely that there is no  $\varphi:[1,\infty)\to [1,\infty)$ such that $\ee_\conv(\MM)\le \varphi(\ee(\MM))$ for every Polish metric space $(\MM,d_\MM)$, though if such an estimate were available, then it would be valuable; see e.g.~Remark~\ref{rem:if vartheta}. In fact, we propose the following conjecture.

\begin{conjecture}\label{conj:no auto conv}
There exists a Polish metric space $(\MM,d_\MM)$ for which $\ee(\MM)<\infty$ yet $\ee_\conv(\MM)=\infty$.
\end{conjecture}

\begin{remark}\label{rem:conv} By definition, for every metric space $(\MM,d_\MM)$, every Banach space $(\bfZ,\|\cdot\|_\bfZ)$ and every $\sub\subset \MM$,
$$
\ee_{\conv}(\MM,\sub;\bfZ)\ge \ee_{\spn}(\MM,\sub;\bfZ)\ge \ee(\MM,\sub;\bfZ).
$$
We explained in Remark~\ref{rem:lindenstrauss span} that the second of these inequalities can be strict (in a strong sense). However, as a complement to Conjecture~\ref{conj:no auto conv}, we state that to the best of our knowledge it is unknown whether this is so for the first of these inequalities, i.e., if it could  happen that $\ee_{\spn}(\MM,\sub;\bfZ)<\infty$ yet $\ee_{\conv}(\MM,\sub;\bfZ)=\infty$. We suspect that this is possible, but if not, then it would be interesting to investigate how one could bound  $\ee_{\conv}(\MM,\sub;\bfZ)$ from above by a function of $\ee_{\spn}(\MM,\sub;\bfZ)$. We do know that there are a metric space $(\MM,d_\MM)$, a Banach space $(\bfZ,\|\cdot\|_\bfZ)$, a subset $\sub\subset \MM$ and a Lipschitz mapping $f:\sub\to \bfZ$ that can be extended to a Lipschitz mapping that takes values in $\overline{\spn}(f(\sub))$ but cannot be extended to a Lipschitz mapping that takes values in $\overline{\conv}(f(\sub))$. To see this, let $\{e_j\}_{j=1}^\infty$ be the standard basis of $\ell_\infty$. For $n\in \N$ set $m(n)=n(n-1)/2$ and let $\X_n$ be the span of $\{e_{m(n)+1},\ldots,e_{m(n+1)}\}$ in $\ell_\infty$. Thus, $\X_n$ is isometric to $\ell_\infty^n$ and $\ell_\infty= (\oplus_{n=1}^\infty \X_n)_\infty$. By~\cite{Sob41}, there is a linear subspace $\mathbf{V}_n$ of $\X_n$ such that every linear projection $\mathsf{Q}:\X_n\to \mathbf{V}_n$ satisfies $\|\mathsf{Q}\|_{\X_n\to \mathbf{V}_n}\gtrsim \sqrt{n}$. By the method of~\cite{JL84}, it follows that there exists\footnote{The subset $\mathscr{A}_n$ can be taken to be any $\e_n$-net of the unit sphere of $\mathbf{V}_n$, for any $\e_n\lesssim n^{-3/2}$. Note, however, that the bound that follows from~\cite{JL84} (and also~\cite[Appendix~C]{MM16}) is $\e_n\lesssim n^{-2}$, and this suffices for the present purposes; see~\cite[Theorem~23]{NR17} for the above stated weaker requirement from $\e_n$.} $\mathscr{A}_n\subset B_{V_n}=V_n\cap B_{\ell_\infty}$ with  $|\mathscr{A}_n|\le n^{O(n)}$ such that $\|F_n\|_{\Lip(X_n;V_n)}\gtrsim \sqrt{n}$ for any  $F_n:\X_n\to \mathbf{V}_n$ that extends the formal identity $\mathsf{Id}_{\mathscr{A}_n\to \mathbf{V}_n}:\mathscr{A}_n\to \mathbf{V}_n$. By compactness, there exists $\delta_n\in (0,1)$ such that if we define $$\sub_n=\mathscr{A}_n\cup \big\{\d_n e_{m(n)+1},\ldots,\d_n e_{m(n+1)}\big\}\cup\{0\},$$ then  also $\|\Phi_n\|_{\Lip(\X_n;\X_n)}\gtrsim \sqrt{n}$ for any mapping $\Phi_n$ from $\X_n$ to the polytope $\overline{\conv}(\sub_n)$ that extends the formal identity $\mathsf{Id}_{\sub_n\to \X_n}$. Consider the subset $$\sub=\bigcupdot_{n=1}^\infty \sub_n\subset \ell_\infty.$$ If $\Phi:\ell_\infty\to \overline{\conv}(\sub)$ extends $\mathsf{Id}_{\sub\to \ell_\infty}$, then for each $n\in \N$ the mapping $\mathsf{R}_n\circ (\Phi|_{\X_n}):\X_n\to \X_n$ extends  $\mathsf{Id}_{\sub_n\to \ell_\infty}$ and takes values in $\overline{\conv}(\sub_n)$, where $\mathsf{R}_n:\ell_\infty\to \X_n$ is the canonical restriction operator. Hence, $$\| \Phi\|_{\Lip(\ell_\infty;\X_n)}\ge \|\mathsf{R}_n\circ (\Phi|_{\X_n})\|_{\Lip(\X_n;\X_n)}\gtrsim \sqrt{n}.$$ Since this holds for every $n\in \N$, the mapping $\Phi$ is not Lipschitz. Consequently,   $\ee_\conv(\ell_\infty,\sub;\ell_\infty)=\infty$. At the same time, by construction we have $\overline{\spn}(\sub)=\overline{\spn}(\{e_j\}_{j=1}^\infty)= c_0$  (recall that $c_0$ commonly denotes the subspace of $\ell_\infty$ consisting of all those sequences that tend to $0$). So, any $2$-Lipschitz retraction $\rho$ of $\ell_\infty$ onto $c_0$ extends $\mathsf{Id}_{\sub\to \ell_\infty}$ and takes values in $\overline{\spn}(\sub)$; the existence of such a retraction $\rho$ is due to~\cite{Lin64} (see also~\cite[Example~1.5]{BL00}). If $\ee_\spn(\ell_\infty,\sub;\ell_\infty)$ were finite, then this example would answer the above question,\footnote{And, it would show that for arbitrarily large $k\in \N$  there exist a metric space $(\MM,d_\MM)$, a Banach space $(\bfZ,\|\cdot\|_\bfZ)$ and a  subset $\mathcal{S}\subset \MM$ with $|\mathcal{S}|=k$ such that $\ee_\conv(\MM,\mathcal{S};\bfZ)/\ee_\spn(\MM,\mathcal{S};\bfZ)\gtrsim \sqrt{(\log k)/\log\log k}$. It would then remain an interesting open question to determine the largest possible asymptotic dependence on $k$ here.} but we suspect that in fact $\ee_\spn(\ell_\infty,\sub;\ell_\infty)=\infty$.
\end{remark}

Proposition~\ref{lem:measure formulation} is a convenient characterization of the quantities $\ee(\MM,\sub)$ and $\ee_\conv(\MM,\sub)$; while it was not previously stated explicitly in this form, its proof is based on well-understood ideas.

\begin{proposition}\label{lem:measure formulation} Suppose that $(\MM,d_\MM)$ is a metric space, $\sub$ is a Polish subset of $\MM$ and $s_0\in \sub$. Fix two nonnegative functions  $\mathfrak{d}:\MM\times \MM\to [0,\infty)$ and $\e:\sub:\to [0,\infty)$. Then, the following two equivalences hold.
\begin{enumerate}
\item The following two statements are equivalent.
\begin{itemize}
\item For every Banach space $(\bfZ,\|\cdot\|_{\bfZ})$ and every mapping $f:\sub\to \bfZ$ that is $1$-Lipschitz with respect to the metric $d_\MM$ there exists $F:\MM\to \bfZ$ that satisfies the following two conditions.
    \begin{itemize}
    \item $\|F(s)-f(s)\|_\bfZ\le \e(s)$ for every $s\in \sub$.
    \item $\|F(x)-F(y)\|_\bfZ\le \mathfrak{d}(x,y)$ for every $x,y\in \MM$.
    \end{itemize}
    \item There exists a family $\{\phi_x\}_{x\in \MM}$ of elements of the free space $\mathfrak{F}(\sub)$ with the following properties.
    \begin{itemize}
     \item $\|\phi_s-\bd_s+\bd_{s_0}\|_{\mathfrak{F}(\sub)}\le \e(s)$ for every $s\in \sub$.
      \item $\|\phi_x-\phi_y\|_{\mathfrak{F}(\sub)}\le\mathfrak{d}(x,y)$ for every $x,y\in \MM$.
\end{itemize}
\end{itemize}

\item The following two statements are equivalent.
\begin{itemize}
\item For every Banach space $(\bfZ,\|\cdot\|_{\bfZ})$ and every mapping $f:\sub\to \bfZ$ that is $1$-Lipschitz with respect to the metric $d_\MM$ there exists $F:\MM\to \overline{\conv}\big(f(\sub)\big)$ that satisfies the following two conditions.
    \begin{itemize}
    \item $\|F(s)-f(s)\|_\bfZ\le \e(s)$ for every $s\in \sub$.
    \item $\|F(x)-F(y)\|_\bfZ\le \mathfrak{d}(x,y)$ for every $x,y\in \MM$.
    \end{itemize}
    \item  There exists a family $\{\mu_x\}_{x\in \MM}$ of probability measures in  $\mathsf{P}_1(\sub)$ with the following properties.
  \begin{itemize}
  \item $\W_1^{d_\MM}(\mu_s,\bd_s)\le \e(s)$ for every $s\in \sub$.
  \item $\W_1^{d_\MM}(\mu_x,\mu_y)\le \mathfrak{d}(x,y)$ for every $x,y\in \MM$.
  \end{itemize}
\end{itemize}
\end{enumerate}
\end{proposition}

In the setting of Proposition~\ref{lem:measure formulation}, if  $\e(s)= 0$ for every $s\in \sub$ and also $\mathfrak{d}=Kd_\MM$   for some $K\ge 1$, then in~\cite[Definition~2.7]{AP20} a family $\{\phi_x\}_{x\in \MM}\subset \mathfrak{F}(\sub)$ as in part {\em (1)} of Proposition~\ref{lem:measure formulation} is called a $K$-random projection of $\MM$ onto $\sub$, and in~\cite[Definition~3.1]{Oht09} a family $\{\mu_x\}_{x\in \MM}\subset \mathsf{P}_1(\sub)$ as in part {\em (2)} of Proposition~\ref{lem:measure formulation} is called a stochastic $K$-Lipschitz retraction of $\MM$ onto $\sub$ while in~\cite[Definition~2.7]{AP20} it is called a strong $K$-random projection of $\MM$ onto $\sub$.

\begin{proof}[Proof of Proposition~\ref{lem:measure formulation}] Suppose first that $\{\phi_x\}_{x\in \MM}\subset \mathfrak{F}(\sub)$ and $\{\mu_x\}_{x\in \MM}\subset \mathsf{P}_1(\sub)$ are as in the two parts of Proposition~\ref{lem:measure formulation}.  Let $(\bfZ,\|\cdot\|_\bfZ)$ be a Banach space and fix a $1$-Lipschitz function $f:\sub\to \bfZ$.  Since $\sub$ is Polish and hence separable, by replacing $\bfZ$ with the closure of the linear span of $f(\sub)$ we may assume that $\bfZ$ is separable. Recalling the notation~\eqref{eq:def If} and the discussion immediately following it for the (integration) operator $\mathfrak{I}_{\! f}:\mathsf{M}_1(\MM)\cup\mathfrak{F}(\MM)\to \bfZ$, define two (linear) mappings $$\mathsf{Ext}_\sub^\phi f,\mathsf{Ext}_\sub^\mu f:\MM\to \bfZ$$ by setting for every $x\in \MM$,
\begin{equation}\label{eq:def extensoion operators}
\mathsf{Ext}_\sub^\phi f(x)\eqdef f(s_0)+\mathfrak{I}_{\! f}(\phi_x)\qquad\mathrm{and}\qquad \mathsf{Ext}_\sub^\mu f(x)\eqdef \mathfrak{I}_{\! f}(\mu_x)\stackrel{\eqref{eq:def If}}{=}\int_\sub f\ud\mu_x.
\end{equation}
Observe that since $\mu_x$ is a probability measure, $\mathsf{Ext}_\sub^\mu f(x)$ belongs to the closure of the convex hull of $f(\sub)$.

For every $x,y\in \MM$ we have
$$
\big\| \mathsf{Ext}_\sub^\phi f(x)- \mathsf{Ext}_\sub^\phi f(y)\big\|_\bfZ=\big\|\mathfrak{I}_{\! f}(\phi_x-\phi_y)\big\|_\bfZ\stackrel{\eqref{eq:norm of integral}}{\le} \|\phi_x-\phi_y\|_{\mathfrak{F}(\sub)}\le \mathfrak{d}(x,y),
$$
and similarly (using Kantorovich--Rubinstein duality),
$$
\big\| \mathsf{Ext}_\sub^\mu f(x)- \mathsf{Ext}_\mu^\phi f(y)\big\|_\bfZ\le \W_1^{d_\MM}(\mu_x,\mu_y)\le \mathfrak{d}(x,y).
$$
Also, for every $s\in \sub$ we have
$$
\big\|\mathsf{Ext}_\sub^\phi f(s)-f(s)\big\|_{\bfZ}= \big\|\mathfrak{I}_{\! f}(\phi_s-\bd_s+\bd_{s_0})\big\|_{\bfZ}\le\|\phi_s-\bd_s+\bd_{s_0} \|_{\mathfrak{F}(\sub)}\le \e(s),
$$
and similarly,
$$
\big\|\mathsf{Ext}_\sub^\mu f(s)-f(s)\big\|_{\bfZ}= \big\|\mathfrak{I}_{\! f}(\phi_s-\bd_s)\big\|_{\bfZ}\le\W_1^{d_\MM}(\mu_s,\bd_s)\le \e(s).
$$

Conversely, define $f:\sub \to \mathfrak{F}(\sub)$ by setting $f(s)=\bd_s-\bd_{s_0}$ for each $s\in \sub$. Then $f$ is $1$-Lipschitz. Fix $F:\MM\to \mathfrak{F}(\sub)$. Writing $F(x)=\phi_x$ for each $x\in \MM$, the assumptions of the first half of part {\em (1)} of  Proposition~\ref{lem:measure formulation} coincide with the assertions of its second half. As $\sub$ is Polish, $\mathsf{P}_1(\sub)$ is closed in $\mathfrak{F}(\sub)$. Therefore,
$$
\overline{\conv}\big(f(\sub)\big)=\mathsf{P}_1(\sub)-\bd_{s_0},
$$
where the closure is in $\mathfrak{F}(\sub)$. Thus, if $F(\MM)\subset \overline{\conv}(f(\sub))$, then $\mu_x\eqdef F(x)+\bd_{s_0}\in \mathsf{P}_1(\sub)$ and the assumptions of the first half of part {\em (2)} of  Proposition~\ref{lem:measure formulation} coincide with the assertions of its second half.
\end{proof}

The proof of Proposition~\ref{lem:measure formulation} shows  that even though in the first parts of the two equivalences in Proposition~\ref{lem:measure formulation}  one assumes merely the existence of an $F$ with the desired properties, it follows that such an $F$ can in fact be chosen to depend {\em linearly} on the input $f$, per~\eqref{eq:def extensoion operators}.

Due to Proposition~\ref{lem:measure formulation}, the following question is closely related to Conjecture~\ref{conj:no auto conv}, though we think that it is also of independent interest.

\begin{question}\label{Q:retract to prob}
Characterize those Polish metric spaces $(\MM,d_\MM)$ for which there exists  a Lipschitz mapping $\rho:\mathfrak{F}(\MM)\to \mathsf{P}_1(\MM)$ (recall that by default $\mathsf{P}_1(\MM)$ is equipped with the Wasserstein-1 metric) and $x_0\in \MM$ such that $\rho(\d_y-\d_{x_0})=\d_y$ for every $y\in \MM$.
\end{question}

\subsection{Barycentric targets}\label{sec:bary} Following~\cite{MN13-bary}, say that a metric space $(\MM,d_\MM)$ is  {\em $\W_1$-barycentric}  with constant $\beta>0$ if there is a mapping $\mathfrak{B}:\mathsf{P}_1(\MM)\to \MM$ that satisfies $\mathfrak{B}(\bd_x)=x$ for every $x,\in \MM$, and also
$$
\forall \mu,\nu\in \mathsf{P}_1(\MM),\qquad d_\MM\big(\mathfrak{B}(\mu),\mathfrak{B}(\nu)\big)\le \beta\W_1^{d_\MM}(\mu,\nu).
$$
The infimal $\beta$ for which this holds is denoted $\beta_{1}(\MM)$. This notion (and variants thereof)  were studied in various contexts; see e.g.~\cite{EH99,LPS00,Gro03,Stu03,LN05,Oht09,Aus11,MN13-bary,Nav13,Lim18,Bas18}.  Any normed space $\X$ is $\W_1$-barycentric with constant $1$, as seen by considering $\mathfrak{B}(\mu)=\int_\X x\ud\mu(x)$. Other examples of spaces that are $\W_1$-barycentric with constant $1$ include Hadamard spaces and Busemann nonpositively curved spaces~\cite{BH99}, or more generally spaces with a conical geodesic bicombing~\cite{Des15}.

Thanks to Proposition~\ref{lem:measure formulation}, convex hull-valued (approximate) extension theorems automatically generalize to extension theorems for mappings that take value in $\W_1$-barycentric metric spaces.

\begin{proposition}\label{prop:pass to bary} Let $(\MM,d_\MM)$ be a metric space and suppose that $\sub\subset \MM$  is a Polish subset of $\MM$. Fix $\mathfrak{d}:\MM\times \MM\to [0,\infty)$ and $\e:\sub\to [0,\infty)$. Assume that for every Banach space $(\bfZ,\|\cdot\|_\bfZ)$ and every $f:\sub\to \bfZ$ that is $1$-Lipschitz with respect to $d_\MM$ there is $F:\MM\to \overline{\conv}(f(\sub))$ that satisfies $\|F(s)-f(s)\|_\bfZ\le \e(s)$ for every $s\in \sub$ and $\|F(x)-F(y)\|_\bfZ\le \mathfrak{d}(x,y)$ for every $x,y\in \MM$.
Fix $\eta:\sub \to (1,\infty)$ and $\tau:\MM\times \MM\to (1,\infty)$, as well as $\beta>0$ and a concave nondecreasing function $\omega:[0,\infty)\to [0,\infty)$ with $\omega(0)=0$. If $(\NN,d_\NN)$ is a $\W_1$-barycentric metric space with constant $\beta$ and $\phi:\sub\to \NN$ has modulus of uniform continuity $\omega$ with respect to $d_\MM$, namely $d_\NN(f(s),f(t))\le \omega(d_\MM(s,t))$ for every $s,t\in \sub$, then there is $\Phi:\MM\to \NN$ such that $d_\NN(\Phi(s),\phi(s))\le \omega(\eta(s)\e(s))$ for every $s\in \sub$ and $d_\NN(\Phi(x),\Phi(y))\le \omega (\tau(x,y)\mathfrak{d}(x,y))$ for every $x,y\in \MM$.
\end{proposition}

\begin{proof} By Proposition~\ref{lem:measure formulation}, there is a collection of measures $\{\mu_x\}_{x\in \MM}\subset  \mathsf{P}_1(\sub)$ such that
$$
\forall s\in \sub,\qquad \W_1^{d_\MM}(\mu_s,\bd_s)\le \e(s) \qquad \mathrm{and}\qquad   \forall x,y\in \MM\qquad \W_1^{d_\MM}(\mu_x,\mu_y)\le \mathfrak{d}(x,y).
$$
Hence, for every $s\in \sub$ and $x,y\in \MM$ there are couplings $\pi_s\in \Pi(\mu_s,\bd_s)$ and $\pi_{x,y}\in \Pi(\mu_x,\mu_y)$ such that
$$
\iint_{\sub\times \sub} d_\MM(u,v)\ud\pi_s(u,v)\le \eta(s)\e(s)\qquad\mathrm{and}\qquad \iint_{\sub\times \sub} d_\MM(u,v)\ud\pi_{x,y}(u,v)\le \tau(x,y)\mathfrak{d}(x,y),
$$
Since $(\phi\times \phi)_{\textbf \#} \pi_{s}\in \Pi(\phi_{\textbf \#} \mu_s,\phi_{\textbf \#} \bd_s)$ and  $(\phi\times \phi)_{\textbf \#} \pi_{x,y}\in \Pi(\phi_{\textbf \#} \mu_x,\phi_{\textbf \#} \mu_y)$, it follows that
\begin{align*}
\W_1^{d_\NN}(\phi_{\textbf \#} \mu_s,\phi_{\textbf \#} \bd_s)&\le \iint_{\NN\times \NN} d_{\NN}(a,b)\ud(\phi\times \phi)_{\textbf \#} \pi_{s}(a,b)\\&=\iint_{\NN\times \NN} d_\NN\big(\phi(u),\phi(v)\big)\ud \pi_s(u,v)\\&
\le \iint_{\NN\times \NN} \omega\big(d_\NN(u,v)\big)\ud \pi_s(u,v)\\&\le \omega\bigg(\iint_{\NN\times \NN} d_\NN(u,v)\ud \pi_s(u,v)\bigg)\\&\le \omega \big(\eta(s)\e(s)\big),
\end{align*}
where the penultimate step uses the concavity of $\omega$. For the same reason, also $$\W_1^{d_\NN}(\phi_{\textbf \#} \mu_x,\phi_{\textbf \#} \mu_y)\le \omega\big(\tau(x,y)\mathfrak{d}(x,y)\big).$$

Since $(\NN,d_\NN)$ is  $\beta$-barycentric there is  $\mathfrak{B}:\mathsf{P}_1(\NN)\to \NN$ satisfying $\mathfrak{B}(\bd_z)=z$ for every $z,\in \NN$, and
$$
\forall \nu_1,\nu_2\in \mathsf{P}_1(\NN),\qquad d_\NN\big(\mathfrak{B}(\nu_1),\mathfrak{B}(\nu_2)\big)\le \beta\W_1^{d_\NN}(\nu_1,\nu_2).
$$
Define $\Phi:\MM\to \NN$ by
$$
\forall x\in \MM,\qquad \Phi(x)\eqdef \mathfrak{B}(\phi_{\textbf \#} \mu_x).
$$
Then, for every $s\in \sub$   we have
$$
d_\NN\big(\Phi(s),\phi(s)\big)\le \beta \W_1^{d_\NN} \big(\phi_{\textbf \#} \mu_s,\phi_{\textbf \#} \bd_s\big)\le \omega \big(\eta(s)\e(s)\big),
$$
and for the same reason also $d_\NN\big(\Phi(x),\phi(y)\big)\le \omega\big(\tau(x,y)\mathfrak{d}(x,y)\big)$ for every $x,y\in \MM$.
\end{proof}

Because (as we will soon see) all of our new  Lipschitz extension theorems are in fact bounds on $\ee_\conv(\cdot)$, the following immediate corollary of Proposition~\ref{prop:pass to bary} (with $\mathfrak{d}$ a multiple of $d_\MM$ and $\omega$ linear) shows that they apply to barycentric targets and not only to Banach space targets.
\begin{corollary} Fix $\beta>0$. Suppose that $\MM$ is a Polish metric space and that $\NN$ is a complete $\W_1$-barycentric metric space with constant $\beta$. Then, $\ee_\conv(\MM,\NN)\le \beta \ee_\conv(\MM)$.
\end{corollary}

Another noteworthy special case of Proposition~\ref{prop:pass to bary} is when $\omega(s)=s^\theta$ for some $0<\theta\le 1$, i.e., in the setting of H\"older extension that we discussed in Remark~\ref{rem:holder}  and Section~\ref{sec:holder ext}. Analogously to~\eqref{eq:define holder ext}, we
denote the convex hull-valued $\theta$-H\"older extend modulus of a metric space $(\MM,d_\MM)$ by
$$
\ee_\conv^\theta(\MM)=\ee_\conv\big(\MM,d_\MM^\theta\big).
$$

\begin{corollary}\label{cor:holder ext} Suppose that $\MM$ is a Polish metric space. Then, for every $0<\theta\le 1$ we have
$$
\ee^\theta(\MM)\le \ee^\theta_\conv(\MM)\le \ee_\conv(\MM)^\theta.
$$
\end{corollary}
Because the upper bound on $\ee(\ell_\infty^n)$ that we obtain in Theorem~\ref{thm:ell infty} is actually an upper bound on $\ee_\conv(\ell_\infty^n)$, Corollary~\eqref{cor:holder ext} implies~\eqref{eq:theta holder l infty}. More generally, Proposition~\ref{prop:pass to bary} implies that
$$
\ee_\conv\big(\MM,\omega\circ d_\MM\big)\le \sup_{d>0}\frac{\omega\big(\ee_\conv(\MM)d\big)}{\omega(d)}\le \ee_\conv(\MM)
$$
for any concave nondecreasing function $\omega:[0,\infty)\to [0,\infty)$ with $\omega(0)=0$.
\begin{remark}\label{rem:if vartheta}
The question of how Lipschitz extension results imply extension results for other moduli of uniform continuity was studied in~\cite{Nao01} and treated definitively by Brudnyi and Shvartsman in~\cite{BS02} using an interesting connection to the Brudny\u{\i}--Krugljak $K$-divisibility theorem~\cite{BK81} (see also~\cite{Cwi84}) from the theory of real interpolation of Banach spaces. In particular, by~\cite{BS02} we have $\ee^\theta(\MM)\lesssim \ee(\MM)^2$, which remains the best-known bound on $\ee^\theta(\MM)$ in terms of $\ee(\MM)$ and it would be interesting to determine if it could be improved. As Corollary~\ref{cor:holder ext}  shows that a better bound is available in terms of $\ee_\conv^\theta(\MM)$, Conjecture~\ref{conj:no auto conv} and Question~\ref{Q:retract to prob} could be relevant for this purpose.
\end{remark}

\subsection{Gentle partitions of unity}

The following definition describes a numerical parameter that underlies the extension method of~\cite{LN05}.
\begin{definition}[modulus of gentle partition of unity]\label{def:gpu} Suppose that $(\MM,d_{\MM})$ is a metric space and that $\sub\subset \MM$ is nonempty and closed. Define  the {\bf \em modulus of gentle partition of unity} of $\MM$ relative to $\sub$, denoted $\gpu(\MM,d_\MM;\sub)$ or simply $\gpu(\MM;\sub)$ when the metric is clear from the context, to be the infimum over those $\mathcal{g}\in (0,\infty]$ such that for every $x\in \MM$ there is a Borel probability measure $\mu_x$ supported on $\sub$ with the requirements that if $s\in \sub$, then $\mu_s=\bd_s$, and also for every $x,y\in \MM$ we have
$$
\int_\sub  d_{\MM}(s,x)\ud|\mu_x-\mu_y|(s)\le \mathcal{g} d_{\MM}(x,y).
$$
The modulus of gentle partitions of unity of $\MM$, denoted $\gpu(\MM,d_{\MM})$ or simply $\gpu(\MM)$ when the metric is clear from the context, is the supremum of $\gpu(\MM,d_\MM;\sub)$ over all nonempty closed subsets $\sub\subset \MM$.
\end{definition}

The nomenclature of Definition~\ref{def:gpu} is derived from~\cite{LN05}, though we warn that Definition~\ref{def:gpu}  considers objects that  are not identical to those that were introduced in~\cite{LN05}. In~\cite{LN05} the measures $\{\mu_x\}_{x\in \MM\setminus \sub}$ were also required to have a Radon--Niko\'ym derivative  with respect to some reference measure $\mu$. This additional  requirement arises automatically from the constructions of~\cite{LN05} but it is not needed for any of the known applications of gentle partitions of unity, so  it is beneficial to remove it altogether. The formal connection between~\cite{LN05} and Definition~\ref{def:gpu}  was clarified  in~\cite{AP20}.

In anticipation of the proof of Theorem~\ref{thm:local compact ext}, one   can generalize Definition~\ref{def:gpu}  to the case of general profiles, analogously to what we did in Definition~\ref{def:separation profile}.

\begin{definition}[gentle partition of unity profile]\label{def:gpu profile} Suppose that $(\MM,d_{\MM})$ is a metric space and that $\sub\subset \MM$ is nonempty and closed. A metric $\mathfrak{d}:\MM\times \MM\to [0,\infty)$ is called a {\bf \em gentle partition of unity profile} for $(\MM,d_\MM)$ relative to $\sub$ if for every $x\in \MM$ there is a Borel probability measure $\mu_x$ supported on $\sub$  with the requirements that if $s\in \sub$, then $\mu_s=\bd_s$, and also for every $x,y\in \MM$ we have
$$
\int_\sub  d_{\MM}(s,x)\ud|\mu_x-\mu_y|(s)\le \mathfrak{d}(x,y).
$$
If $\mathfrak{d}$ is a gentle partition of unity profile for $(\MM,d_\MM)$ relative to every closed $\emptyset\neq \sub\subset \MM$, then we say that $\mathfrak{d}$ is a gentle partition of unity profile for $(\MM,d_\MM)$.
\end{definition}

Note in passing that if $\mathfrak{d}$ is a gentle partition of unity profile for $(\MM,d_\MM)$ relative to $\sub$, then  for every $x\in \MM$ the probability measure $\mu_x$ in Definition~\ref{def:gpu profile} has finite first moment. Indeed, for any $s_0\in \sub$,
\begin{equation}\label{eq:gpu finite moment}
\int_{\sub}d_{\MM}(s_0,s)\ud\mu_x(s)=\int_{\sub}d_{\MM}(s_0,s)\ud\big(\mu_x-\d_{s_0}\big)(s)\le \int_{\sub}d_{\MM}(s_0,s)\ud\big|\mu_x-\mu_{s_0}\big|(s)\le \mathfrak{d}(s_0,x)<\infty,
\end{equation}
where we used the fact that $\mu_{s_0}=\d_{s_0}$, since $s_0\in \sub$.

Suppose that $(\MM,d_{\MM})$ is a Polish metric space. The following estimate is implicit in~\cite{LN05}.
\begin{equation}\label{eq:e to gpu}
\ee_{\conv}(\MM)\le 2\gpu(\MM).
\end{equation}
In fact, the same reasoning as in~\cite{LN05} leads to the following more general lemma.
\begin{lemma}\label{lem:GPU to ext} Suppose that $(\MM,d_{\MM})$ is a Polish metric space and that $\sub\subset \MM$ is nonempty and closed. Assume that $\mathfrak{d}:\MM\times \MM\to [0,\infty)$ is a gentle partition of unity profile for $(\MM,d_\MM)$ relative to $\sub$. Then, for every Banach space $(\bfZ,\|\cdot\|_\bfZ)$ and every $1$-Lipschitz mapping $f:\sub \to \bfZ$ there exists
$$
F:\MM\to \overline{\conv}\big(f(\sub)\big)
$$
that extends $f$ and satisfies $\|F(x)-F(y)\|_\bfZ\le 2\mathfrak{d}(x,y)$ for every $x,y\in \MM$.
\end{lemma}

\begin{proof} Let $\{\mu_x\}_{x\in \MM}$ be probability measures as in Definition~\ref{def:gpu profile}. Then, $\{\mu_x\}_{x\in \MM}\subset \mathsf{P}_1(\sub)$ by~\eqref{eq:gpu finite moment}. So, by Proposition~\ref{lem:measure formulation}  (with $\e\equiv 0$) it suffices to check that $\W_1(\mu_x,\mu_y)\le 2\mathfrak{d}(x,y)$ for every $x,y\in \MM$. To this end, fix $\eta>0$ and $s_0\in \sub$ such that $d_\MM(x,s_0)\le d_\MM(x,\sub)+\eta$. Then,
$$
\forall s\in \sub, \qquad d_\MM(s,s_0)\le d_\MM(s,x)+d_\MM(x,s_0)\le d_\MM(s,x)+ d_\MM(x,\sub)+\eta\le 2d_\MM(s,x)+\eta.
$$
Consequently, every $1$-Lipschitz function $\psi:\sub \to \R$ satisfies
\begin{multline*}
\int_\sub \psi \ud\mu_x-\int_\sub\psi\ud\mu_y=\int_\sub \big(\psi(s)-\psi(s_0)\big)\ud (\mu_x-\mu_y)(s)\le \int_\sub |\psi(s)-\psi(s_0)|\ud |\mu_x-\mu_y|(s)\\\le \int_\sub d_\MM(s,s_0)\ud |\mu_x-\mu_y|(s)\le \int_\sub (2d_\MM(s,x)+\eta)\ud |\mu_x-\mu_y|(s)\le 2\mathfrak{d}(x,y)+2\eta.
\end{multline*}
The desired conclusion follows by letting $\eta\to 0$ and using the Kantorovich--Rubinstein duality~\eqref{eq:KR}.
\end{proof}

\subsection{The multi-scale construction} Suppose that $(\MM,d_\MM)$ is a Polish metric space and fix another metric $\mathfrak{d}$ on $\MM$. In this section we will show that there is a universal constant $\alpha\ge 1$ with the following property. Assume that either $(\MM,d_\MM)$ is locally compact and $\mathfrak{d}$ is a separation modulus for $(\MM,d_\MM)$ per Definition~\ref{def:separation profile}, or the assumptions of Theorem~\ref{thm:polish extension}  are satisfied. We will prove that either of these assumptions implies that $\alpha\mathfrak{d}$ is a gentle partition of unity profile  for $(\MM,d_\MM)$. By Lemma~\ref{lem:GPU to ext} this gives Theorem~\ref{thm:local compact ext} and  Theorem~\ref{thm:polish extension}, and will show that in fact these extension results are both  convex hull-valued and via a linear extension operator. This also implies that every locally compact metric space $\MM$ satisfies
\begin{equation}\label{eq:GPU less than sep}
\gpu(\MM)\lesssim\sep(\MM).
\end{equation}
\begin{remark}
The bound~\eqref{eq:GPU less than sep} need not be sharp. Indeed, it was proved in~\cite{LN05} that if $\MM$ is finite, then
\begin{equation}\label{eq:log over log log}
\gpu(\MM)\lesssim \frac{\log |\MM|}{\log\log |\MM|}.
\end{equation}
However, by~\cite{Bar96} sometimes $\sep(\MM)\gtrsim \log |\MM|$ (and always $\sep(\MM)\lesssim \log |\MM|$). A shorter presentation of the proof of~\eqref{eq:log over log log} can be found in~\cite{Nao15-class}, and a different proof of~\eqref{eq:log over log log} will appear in the forthcoming work~\cite{MN21}. Also, in the forthcoming work~\cite{MNR21}    it is proved that~\eqref{eq:log over log log} is optimal.
\end{remark}

The following theorem is a precise formulation of what we will prove in this section.

\begin{theorem}\label{thm:polish GPU} Let $(\MM,d_\MM)$ be a Polish metric space and fix another metric $\mathfrak{d}$ on $\MM$. Suppose that for every $\Delta>0$ there is a probability space $(\Omega_\Delta, \Pr_\Delta)$ and a sequence of set-valued mappings $\{\Gamma^k_\Delta:\Omega_\Delta\to 2^\MM\}_{k=1}^\infty$ such that one of the following two measurability assumptions hold.
\begin{itemize}
\item Either $(\MM,d_\MM)$  is locally compact and  $\Gamma^k_\Delta$ is strongly measurable for each fixed $k \in \N$ and $\Delta>0$,
\item or  $\Omega_\Delta$ is a Borel subset of some Polish metric space $\ZZ_\Delta$  and $\Pr_\Delta$ is a Borel probability measure supported on $\Omega_\Delta$,  and $\Gamma^k_\Delta$ is a standard set-valued mapping for each fixed $k \in \N$ and $\Delta>0$.
\end{itemize}
Suppose that the following three requirements hold.
\begin{enumerate}
\item $\Part_\Delta^\omega=\{\Gamma^k_\Delta(\omega)\}_{k=1}^\infty$ is a partition of $\MM$ for every $\omega \in \Omega_\Delta$,
\item $\diam_\MM(\Part_\Delta^\omega(x))< \Delta$ for every $x\in \MM$ and $\omega\in \Omega_\Delta$,
\item $\Delta \Pr_\Delta\big[\omega\in \Omega_\Delta:\ \Part_\Delta^\omega(x)\neq\Part_\Delta^\omega(y)\big]\le \mathfrak{d}(x,y)$ for every $x,y\in \MM$.
\end{enumerate}
Then, $\alpha\mathfrak{d}$ is a gentle partition of unity profile  for $(\MM,d_\MM)$ for some  universal constant $\alpha\in [1,\infty)$.
\end{theorem}

Suppose from now on that $\sub$ is a nonempty {\em closed} subset of $\MM$. We will first  set notation and record basic properties of  a sequence of bump functions that will be used in the proof of Theorem~\eqref{thm:polish GPU}; this part of the discussion is entirely standard and has nothing to do with random partitions.

Fix a $1$-Lipschitz function $\psi:[0,\infty)\to [0,\infty)$ such that $\supp(\psi)\subset [1,4]$ and $\psi(t)=1$ for every $t\in [2,3]$ (these requirements uniquely determine $\psi$, which is piecewise linear). Define for each $n\in \Z$,
$$
\forall x\in \MM,\qquad \phi_n(x)=\phi_n^\sub(x)\eqdef \psi\big(2^{-n}d_{\MM}(x,\sub)\big).
$$
Then $\|\phi_n\|_{\Lip(\MM)}\le 2^{-n}$ and if $\phi_n(x)\neq 0$ then necessarily $2^n\le d_{\MM}(x,\sub)\le 2^{n+2}$. We  also denote
$$
\forall x\in \MM,\qquad \Phi(x)=\Phi^\sub(x)\eqdef \sum_{m\in \Z} \phi_n(x).
$$

For each $x\in \MM$, at most two summands in the sum that defines $\Phi(x)$ do not vanish. If $x\in \MM\setminus \sub$, then since $\sub$ is closed we have $d_{\MM}(x,\sub)>0$, and therefore there is $n\in \Z$ for which $2^n\le d_{\MM}(x,\sub)< 2^{n+1}$. For this value of $n$ we have $\phi_n(x)=1$, so $\Phi(x)\ge 1$ for every $x\in \MM\setminus \sub$. Finally, for each $n\in \Z$ define
$$
\forall x\in \MM,\qquad  \lambda_n(x)=\lambda_n^\sub(x)\eqdef \left\{\begin{array}{ll}\frac{\phi_n(x)}{\Phi(x)} &\mathrm{if}\ x\in \MM\setminus \sub,
\\ 0 &\mathrm{if}\ x\in \sub.\end{array}\right.
$$
By design, $\sum_{n\in \Z}\lambda_n(x)=1$ for every $x\in \MM\setminus \sub$. Further properties of these bump functions are recorded in the following basic lemma, for ease of later reference.
\begin{lemma}\label{lem:lambdam lip}
Suppose that $x,y\in \MM$ satisfy $d_{\MM}(x,\sub)\ge d_{\MM}(y,\sub)>d_{\MM}(x,y)$. Then for every $n\in \Z$,
\begin{equation}\label{eq:phi support}
\frac{2^n}{d_{\MM}(y,\sub)}\notin \Big(\frac14, 2\Big)\implies \phi_n(x)=\phi_n(y)=\lambda_n(x)=\lambda_n(y)=0,
\end{equation}
and
\begin{equation}\label{eq:lambda m lip}
2^{n-1}< d_{\MM}(y,\sub)< 2^{n+2}\implies \big|\lambda_n(x)-\lambda_n(y)\big|\lesssim \frac{d_{\MM}(x,y)}{d_{\MM}(y,\sub)}.
\end{equation}
\end{lemma}

\begin{proof} Our assumption implies that $d_{\MM}(x,\sub),d_{\MM}(y,\sub)>0$, so $x,y\in \MM\setminus \sub$. To prove~\eqref{eq:phi support}, suppose first that $2^n\ge 2d_{\MM}(y,\sub)$. Then, $\phi_n(y)=\lambda_n(y)=0$ since $\supp(\psi)\subset [1,4]$ and $2^{-n}d_{\MM}(y,\sub)\le 1$. Also, $d_{\MM}(x,\sub)\le d_{\MM}(x,y)+d_{\MM}(y,\sub)< 2d_{\MM}(y,\sub)\le 2^n$, so $2^{-n}d_{\MM}(x,\sub)\le 1$ and hence $\phi_n(x)=\lambda_n(x)=0$. The remaining case of~\eqref{eq:phi support} is when $d_{\MM}(y,\sub)\ge 2^{n+2}$. Then,  $2^{-n}d_\MM(x,\sub)\ge 2^{-n}d_\MM(y,\sub)\ge 4$ and therefore $\{2^{-n}d_\MM(x,\sub), 2^{-n}d_\MM(y,\sub)\}\cap \supp(\psi)=\emptyset$. Consequently, $\phi_n(x)=\phi_n(y)=\lambda_n(x)=\lambda_n(y)=0$.

To prove~\eqref{eq:lambda m lip}, assume that $2^{n-1}< d_{\MM}(y,\sub)< 2^{n+2}$. Recalling that (point-wise) on $\MM\setminus \sub$ we have $\lambda_n=\phi_n/\Phi$ for all $n\in \Z$ and $\Phi\ge 1$, and moreover $\|\phi_n\|_{\Lip(\MM)}\le 2^{-n}$, we conclude as follows.
\begin{align*}
\big|\lambda_n(x)-\lambda_n(y)\big|&\le \frac{\big|\phi_n(x)-\phi_n(y)\big|}{\Phi(x)}+\frac{\phi_n(y)}{\Phi(x)\Phi(y)}\big|\Phi(y)-\Phi(x)\big|
\\&\le 2^{-n}d_{\MM}(x,y)+\sum_{n\in \Z} \big|\phi_n(x)-\phi_n(y)\big|\\&\!\!\!\!\stackrel{\eqref{eq:phi support}}{\le}  2^{-n}d_{\MM}(x,y)+\sum_{\substack{n\in \Z\\ 2^{n-1}< d_{\MM}(y,\sub)< 2^{n+2}}} 2^{-n}d_{\MM}(x,y)\asymp \frac{d_{\MM}(x,y)}{d_{\MM}(y,\sub)}.\tag*{\qedhere}
\end{align*}
\end{proof}

The interaction between $\{\lambda_n\}_{n\in \Z}$ and the random partitions of Theorem~\ref{thm:polish GPU}  is the content of the following lemma. Note that by reasoning as in~\eqref{eq:frak d lower},   the metric $\mathfrak{d}$ in Theorem~\ref{thm:polish GPU} must satisfy
$$
\forall x,y\in \MM, \qquad \mathfrak{d}(x,y)\ge d_\MM(x,y).
$$

\begin{lemma}\label{lem:gentle} In the setting  of Theorem~\ref{thm:polish GPU}, if $x\in \MM\setminus \sub$ and  $y\in \MM\setminus\{x\}$ satisfy $d_{\MM}(x,\sub)\ge d_{\MM}(y,\sub)$, then
\begin{equation}\label{eq:ratio gentle}
\sum_{n\in \Z}\sum_{k=1}^\infty\int_{\Omega_{2^n}} \big|\lambda_n(x)\1_{\Gamma_{2^n}^k(\omega)}(x)-\lambda_n(y)\1_{\Gamma_{2^n}^k(\omega)}(y)\big|\ud\Pr_{2^n}(\omega)\lesssim \frac{ \mathfrak{d}(x,y)}{d_{\MM}(y,\sub)+d_{\MM}(x,y)}.
\end{equation}
\end{lemma}

\begin{proof}
As $\sum_{n\in \Z}\lambda_n(x)=\sum_{n\in \Z}\lambda_n(y)=1$  and also $\sum_{k=1}^\infty \1_{\Gamma_{2^n}^k(\omega)}(x)=\sum_{k=1}^\infty \1_{\Gamma_{2^n}^k(\omega)}(y)=1$ for each $n\in \Z$ and  $\omega\in \Omega_{2^n}$, the left hand side of~\eqref{eq:ratio gentle} is at most $2$. Since $\mathfrak{d}(x,y)\ge d_\MM(x,y)$, it follows that~\eqref{eq:ratio gentle} holds if $d_{\MM}(y,\sub)\le d_{\MM}(x,y)$. So, we will assume in the rest of the proof of  Lemma~\ref{lem:gentle} that $d_{\MM}(x,y)<d_{\MM}(y,\sub)$ (thus, in particular, $y\in \MM\setminus \sub$), in which case the right-hand side of~\eqref{eq:ratio gentle} becomes at least a universal constant multiple of the quantity $\mathfrak{d}(x,y)/d_{\MM}(y,\sub)$.

We claim that for every $n\in \Z$ the following inequality holds for every $\omega\in \Omega_{2^n}$.
\begin{equation}\label{eq:indicator identity}
\sum_{k=1}^\infty \big|\lambda_n(x)\1_{\Gamma_{2^n}^k(\omega)}(x)-\lambda_n(y)\1_{\Gamma_{2^n}^k(\omega)}(y)\big|\lesssim \Big(2^{-n}d_{\MM}(x,y)+\1_{\big\{\mathscr{P}_{2^m}^\omega(x)\neq \mathscr{P}_{2^n}^\omega(y)\big\}}\Big)\1_{\big\{\frac14<\frac{2^n}{d_{\MM}(y,\sub)}< 2\big\}}.
\end{equation}
Assuming~\eqref{eq:indicator identity}, we conclude the proof of~\eqref{eq:ratio gentle} in the remaining case $d_{\MM}(x,y)<d_{\MM}(y,\sub)$ as follows.
\begin{align*}
\sum_{n\in \Z}\sum_{k=1}^\infty&\int_{\Omega_{2^n}} \big|\lambda_n(x)\1_{\Gamma_{2^n}^k(\omega)}(x)-\lambda_{n}(y)\1_{\Gamma_{2^n}^k(\omega)}(y)\big|\ud\Pr_{2^n}(\omega)\\
&\lesssim \sum_{\substack{n\in \Z\\ 2^{n-1}< d_{\MM}(y,\sub)< 2^{n+2}}} \Big(2^{-n}d_{\MM}(x,y)+\Pr_{2^n}\big[\{\omega\in \Omega_{2^n}:\ \mathscr{P}_{2^n}^\omega(x)\neq \mathscr{P}_{2^n}^\omega(y)\}\big]\Big)\\
&\lesssim  \sum_{\substack{n\in \Z\\ 2^{mn-1}< d_{\MM}(y,\sub)< 2^{n+2}}} 2^{-n}\big(d_{\MM}(x,y)+\mathfrak{d}(x,y)\big)\asymp \frac{\mathfrak{d}(x,y)}{d_{\MM}(y,\sub)}\asymp \frac{ \mathfrak{d}(x,y)}{d_{\MM}(y,\sub)+d_{\MM}(x,y)},
\end{align*}
where the first step uses~\eqref{eq:indicator identity}, the second step is where we used condition {\em (3)} of Theorem~\ref{thm:polish GPU}, the penultimate step uses $\mathfrak{d}(x,y)\ge d_\MM(x,y)$, and in the final step uses the assumption $d_{\MM}(x,y)<d_{\MM}(y,\sub)$.

It remains to establish~\eqref{eq:indicator identity}. By Lemma~\ref{lem:lambdam lip}, if it is not the case that $2^{n-1}< d_{\MM}(y,\sub)< 2^{n+2}$, then $\lambda_n(x)=\lambda_n(y)=0$, so both sides of~\eqref{eq:indicator identity} vanish.  We may therefore assume that $2^{n-1}< d_{\MM}(y,\sub)< 2^{n+2}$. Under this assumption, if $\mathscr{P}_{2^n}^\omega(x)\neq \mathscr{P}_{2^n}^\omega(y)$, then the right-hand side of~\eqref{eq:indicator identity} is at least $1$, while the left-hand side of~\eqref{eq:indicator identity} consists of a sum of two numbers, each of which is at most $1$. It therefore remains to establish~\eqref{eq:indicator identity}  when $\mathscr{P}_{2^n}^\omega(x)= \mathscr{P}_{2^n}^\omega(y)$ (and still $2^{n-1}< d_{\MM}(y,\sub)< 2^{n+2}$). In this case, \eqref{eq:indicator identity} becomes the inequality $|\lambda_{2^n}(x)-\lambda_{2^n}(y)|\le d_{\MM}(x,y)/d_{\MM}(y,\sub)$, which we proved in Lemma~\ref{lem:lambdam lip}.
\end{proof}

\begin{proof}[Proof of Theorem~\ref{thm:polish GPU}] By Lemma~\ref{lem:nearest point selection criterion} and Corollary~\ref{cor:proximal criterion standard}, for every $\Delta>0$ there exists a $\Pr_\Delta$-to-Borel measurable mapping $\gamma_\Delta^k:\Omega_{m}\to \sub$ such that
 \begin{equation}\label{eq:selector scale m}
 \forall  \omega\in \Omega_{\Delta},\qquad \Gamma_\Delta^k(\omega)\neq\emptyset \implies        d_{\MM}\big(\gamma_\Delta^k(\omega),\Gamma_\Delta^k(\omega)\big)\le d_{\MM}\big(\sub,\Gamma_\Delta^k(\omega)\big)+\Delta.
\end{equation}
(In fact, in the locally compact setting of Theorem~\ref{thm:polish GPU}, the use of Lemma~\ref{lem:nearest point selection criterion}  shows that the additive $\Delta$ term in the right hand side of~\eqref{eq:selector scale m} can be removed).

For every $x\in \MM\setminus \sub$ define a Borel measure $\mu_x$ supported on $\sub$ by
\begin{equation}\label{eq:def mu x}
\mu_x\eqdef \sum_{n\in \Z}\sum_{k=1}^\infty \lambda_n(x) \big(\gamma_{2^n}^k\big)_{\textbf \#}\Big(\Pr_{2^n}\bigl\lfloor_{\big\{\omega\in \Omega_{2^n}:\ x\in \Gamma_{2^n}^k(\omega)\big\}}\Big).
\end{equation}
In other words, for every Borel-measurable mapping $h:\sub\to [0,\infty)$ we have
\begin{equation}\label{eq:h version}
\int_\sub h(s)\ud\mu_x(s)= \sum_{n\in \Z}\sum_{k=1}^\infty \lambda_n(x)\int_{\big\{\omega\in \Omega_{2^n}:\ x\in \Gamma_{2^n}^k(\omega)\big\}}h\big(\gamma_{2^n}^k(\omega)\big)\ud\Pr_{2^n}(\omega).
\end{equation}
Since $\Part^\omega_{2^n}$ is a partition of $X$ for every $n\in \Z$ and  $\omega\in \Omega_{2^n}$, the special case $h=\1_\sub$ of~\eqref{eq:h version} implies  that
\begin{align*}
\mu_x(\sub)&= \sum_{n\in \Z}\sum_{k=1}^\infty \lambda_n(x)\Pr_{2^n}\Big[\big\{\omega\in \Omega_{2^n}:\ x\in \Gamma_{2^n}^k(\omega)\big\}\Big]
\\&=\sum_{n\in \Z} \lambda_n(x)\Pr_{2^n}\Big[\big\{\omega\in \Omega_{2^n}:\ x\in \bigcup_{k=1}^\infty\Gamma_{2^n}^k(\omega)\big\}\Big]=\sum_{n\in \Z}\lambda_n(x)=1.
\end{align*}
Thus $\mu_x$ is a probability measure.  Consequently,  if we also denote $\mu_s=\bd_s$ for every $s\in \sub$, then the proof of Theorem~\ref{thm:polish GPU} will be complete if we show that
\begin{equation}\label{eq:goal gpu}
\forall  x,y\in \MM,\qquad \int_\sub d_{\MM}(s,x)\ud|\mu_x-\mu_y|(s)\lesssim \mathfrak{d}(x,y).
\end{equation}

It suffices to prove~\eqref{eq:goal gpu} when $x,y\in \MM$ are distinct and  $\{x,y\}\not\subset \sub$. Indeed, if $\{x,y\}\subset \sub$ then $\mu_x=\bd_x$ and $\mu_y=\bd_y$, so the left hand side of~\eqref{eq:goal gpu} is equal to $d_{\MM}(x,y)$, which is at most $\mathfrak{d}(x,y)$. Hence, in the rest of the proof of Theorem~\ref{thm:polish GPU} we will assume without loss of generality that $x\in \MM\setminus \sub$ and $d_{\MM}(x,\sub)\ge d_{\MM}(y,\sub)$.

We claim that the left hand side of~\eqref{eq:goal gpu} can be bounded from above as follows.
\begin{align}\label{eq:gpu both cases at one}
\begin{split}
\int_\sub  &d_{\MM}(s,x)\ud|\mu_x-\mu_y|(s)\\&\le d_{\MM}(x,y)+\sum_{n\in \Z}\sum_{k=1}^\infty  \int_{\Omega_{2^n}} d_{\MM}\big(\gamma_{2^n}^k(\omega),x\big) \big|\lambda_n(x)\1_{\Gamma_{2^n}^k(\omega)}(x)-\lambda_{2^n}(y)\1_{\Gamma_{2^n}^k(\omega)}(y)\big|\ud\Pr_{2^n}(\omega).
\end{split}
\end{align}
Indeed, if $x,y\in \MM\setminus \sub$, then $\mu_x,\mu_y$ are defined according to~\eqref{eq:def mu x}, so that
\begin{align*}
 \int_\sub  &d_{\MM}(s,x)\ud|\mu_x-\mu_y|(s)\\&\le \sum_{n\in \Z}\sum_{k=1}^\infty  \int_\sub d_{\MM}(s,x)\ud  \bigg(\big(\gamma_{2^n}^k\big)_{\textbf \#} \Big|\lambda_n(x)\Pr_{2^n}\bigl\lfloor_{\left\{\omega\in \Omega_{2^n}:\ x\in \Gamma_{2^n}^k(\omega)\right\}}-\lambda_{n}(y)\Pr_{2^n}\bigl\lfloor_{\left\{\omega\in \Omega_{2^n}:\ y\in \Gamma_{2^n}^k(\omega)\right\}}\Big|\bigg)(s)\\
 &= \sum_{n\in \Z}\sum_{k=1}^\infty  \int_{\Omega_{2^n}} d_{\MM}\big(\gamma_{2^n}^k(\omega),x\big) \big|\lambda_n(x)\1_{\Gamma_{2^n}^k(\omega)}(x)-\lambda_{n}(y)\1_{\Gamma_{2^n}^k(\omega)}(y)\big|\ud\Pr_{2^n}(\omega),
\end{align*}
thus establishing~\eqref{eq:gpu both cases at one} in this case. The remaining case is when $x\in \MM\setminus \sub$ and $y\in \sub$, so that $\mu_x$ is given in~\eqref{eq:def mu x} and $\mu_y=\bd_y$. We can then use the following (crude) estimate.
\begin{align}\label{eq:y in S case}
\begin{split}
\int_\sub  d_{\MM}(s,x)\ud|\mu_x-\mu_y|(s)&\le \int_\sub  d_{\MM}(s,x)\ud\mu_y(s)+\int_\sub  d_{\MM}(s,x)\ud \mu_x(s)\\&=d_{\MM}(x,y)+\sum_{n\in \Z}\sum_{k=1}^\infty  \int_{\Omega_{2^n}} d_{\MM}\big(\gamma_{2^n}^k(\omega),x\big) \lambda_n(x)\1_{\Gamma_{2^n}^k(\omega)}(x)\ud\Pr_{2^n}(\omega).
\end{split}
\end{align}
It remains to observe that because $y\in \sub$ we have $\lambda_n(y)=0$ for all $n\in \Z$ and therefore the right hand side of~\eqref{eq:y in S case} coincides with the right hand side of~\eqref{eq:gpu both cases at one}.

Next, we claim that for every $(n,k)\in \Z\times \N$ and every $\omega\in \Omega_{2^n}$ we have
\begin{align}\label{eq:get rid of s}
\begin{split}
d_{\MM}\big(\gamma^k_{2^n}(\omega),x\big)\big|\lambda_n(x)\1_{\Gamma_{2^n}^k(\omega)}(x)&-\lambda_n(y)\1_{\Gamma_{2^n}^k(\omega)}(y)\big|\\&\lesssim \big(d_{\MM}(y,\sub)+d_{\MM}(x,y)\big)\big|\lambda_n(x)\1_{\Gamma_{2^n}^k(\omega)}(x)-\lambda_n(y)\1_{\Gamma_{2^n}^k(\omega)}(y)\big|.
\end{split}
\end{align}
By a substitution of the point-wise estimate~\eqref{eq:get rid of s} into~\eqref{eq:gpu both cases at one} and using $d_\MM(x,y)\le \mathfrak{d}(x,y)$ the desired estimate~\eqref{eq:goal gpu} follows from Lemma~\ref{lem:gentle}, thus completing the proof of Theorem~\ref{thm:polish GPU}.

To verify~\eqref{eq:get rid of s}, note first that both sides of~\eqref{eq:get rid of s} vanish unless $x\in \Gamma_{2^n}^k(\omega)$ or $y\in \Gamma_{2^n}^k(\omega)$ and also, due to Lemma~\ref{lem:lambdam lip}, $2^{n-1}< d_{\MM}(y,\sub)< 2^{n+2}$. So, assume from now on that
\begin{equation}\label{eq:sm asymp}
\{x,y\}\cap \Gamma_{2^n}^k(\omega)\neq \emptyset\qquad\mathrm{and}\qquad 2^{n-1}< d_{\MM}(y,\sub)< 2^{n+2}.
\end{equation}
Our goal~\eqref{eq:get rid of s} then becomes  to deduce that
\begin{equation}\label{eq:goal sum of distances}
d_{\MM}\big(\gamma^k_{2^n}(\omega),x\big)\lesssim d_{\MM}(y,\sub)+d_{\MM}(x,y).
\end{equation}
Choose a point $z\in \Gamma_m^k(\omega)$ such that
\begin{equation}\label{eq:z choice closest}
d_{\MM}\big(\gamma_{2^n}^k(\omega),z\big)\le d_{\MM}\big(\gamma_{2^n}^k(\omega),\Gamma_{2^n}^k(\omega)\big)+2^n\stackrel{\eqref{eq:selector scale m}}{=}d_{\MM}\big(\sub,\Gamma_{2^n}^k(\omega)\big)+2^{n+1}\stackrel{\eqref{eq:sm asymp}}{\asymp} d_{\MM}\big(\sub,\Gamma_{2^n}^k(\omega)\big)+d_\MM(y,\sub).
\end{equation}
If $x\in \Gamma_{2^n}^k(\omega)$, then
$$
d_{\MM}\big(\sub,\Gamma_{2^n}^k(\omega)\big)\le d_{\MM}(x,\sub)\le d_{\MM}(x,y)+d_{\MM}(y,\sub)\qquad\mathrm{and}\qquad
d_{\MM}(x,z)\le \diam_{\MM}\big(\Gamma_{2^n}^k(\omega)\big)\le 2^n\stackrel{\eqref{eq:sm asymp}}{\asymp}  d_\MM(y,\sub).
$$
By combining these two estimates with~\eqref{eq:z choice closest} and the triangle inequality, we see that
$$
d_{\MM}\big(\gamma^k_{2^n}(\omega),x\big)\le d_{\MM}\big(\gamma^k_{2^n}(\omega),z\big)+d_{\MM}(z,x)\lesssim d_{\MM}(x,y)+d_{\MM}(y,\sub).
$$
Hence, the desired estimate~\eqref{eq:goal sum of distances} holds when  $x\in \Gamma_{2^n}^k(\omega)$.

It remains to check~\eqref{eq:goal sum of distances} when $y\in \Gamma_{2^n}^k(\omega)$, in which case we proceed similarly by noting that now $$
d_{\MM}\big(\sub,\Gamma_{2^n}^k(\omega)\big)\le d_{\MM}(y,\sub)\qquad
\mathrm{and}\qquad  d_{\MM}(y,z)\le \diam_{\MM}\big(\Gamma_{2^n}^k(\omega)\big)\le 2^n\stackrel{\eqref{eq:sm asymp}}{\asymp}  d_\MM(y,\sub).
$$ By combining these two estimates with~\eqref{eq:z choice closest} and the triangle inequality, we conclude that
\begin{equation*}
d_{\MM}\big(\gamma_{2^n}^k(\omega),x\big)\le d_{\MM}\big(\gamma_{2^n}^k(\omega),z\big)+d_{\MM}(z,y)+d_{\MM}(y,x)\lesssim d_{\MM}(y,\sub)+d_{\MM}(x,y).\tag*{\qedhere}
\end{equation*}
\end{proof}

\section{Volume computations}\label{sec:volumes and cone measure}

In this section we will prove volume estimates that occur in our bounds on the separation modulus.

\subsection{Direct sums}\label{sec:direct} Fix $n\in \N$ and a normed space $\X=(\R^n,\|\cdot\|_{\X})$. Throughout what follows, the (normalized) {\em cone measure}~\cite{GM87} on $\partial B_{\X}$ will be denoted  $\kappa_{\X}$. Thus, for every measurable $A\subset \partial B_{\X}$,
\begin{equation}\label{eq:def cone measure}
\kappa_{\X}(A)\eqdef \frac{\vol_n([0,1]A)}{\vol_n(B_{\X})}=\frac{\vol_n(\{sv:\ (s,v)\in [0,1]\times A\})}{\vol_n(B_{\X})}.
\end{equation}

The probability measure $\kappa_{\X}$ is characterized by the following ``generalized polar coordinates'' identity, which holds for every $f\in L_1(\R^n)$; see e.g.~\cite[Proposition~1]{NR03}.
\begin{equation}\label{eq:polar cone}
\int_{\R^n} f(x)\ud x=n\vol_n(B_{\X})\int_0^\infty r^{n-1}\bigg(\int_{\partial B_{\X}} f(r\theta)\ud\kappa_{\X}(\theta)\bigg)\ud r.
\end{equation}
As a quick application of~\eqref{eq:polar cone}, we will next record for ease of later reference the following computation of the volume of the unit ball of an $\ell_p$ direct sum of normed spaces.
\begin{lemma}\label{lem:volume of lp direct sum}Fix  $n,m_1,\ldots,m_n\in \N$ and normed spaces $\big\{\X_j=\big(\R^{m_1},\|\cdot\|_{\X_{m_j}}\big)\big\}_{j=1}^n$. Then
\begin{equation}\label{eq:volume formula for direct sum}
\forall p\in [1,\infty],\qquad \vol_{m_1+\ldots+m_n} \big(B_{\X_1\oplus_p\ldots\oplus_p\X_n}\big)=\frac{\prod_{j=1}^n\Gamma\big(1+\frac{m_j}{p}\big)\vol_{m_j}\big(B_{\X_j}\big)}{\Gamma\big(1+\frac{m_1+\ldots+m_n}{p}\big)}.
\end{equation}
\end{lemma}
\begin{proof} This follows by induction on $n$ from the following identity (direct application of Fubini), which holds for every $a,b\in \N$ and any two normed spaces $\X=(\R^a,\|\cdot\|_{\X})$ and $\Y=(\R^b,\|\cdot\|_{\Y})$.
\begin{multline*}
\vol_{a+b}(B_{\X\oplus_p \Y})=\int_{B_{\X}} \vol_b\Big( \big(1-\|x\|_{\X}^p\big)^{\frac{1}{p}}B_{\Y}\Big)\ud x=\vol_b(B_{\Y})\int_{B_{\X}} \big(1-\|x\|_{\X}^p\big)^{\frac{b}{p}}\ud x\\\stackrel{\eqref{eq:polar cone}}{=}\vol_a(B_{\X})\vol_b(B_{\Y})\int_0^1 ar^{a-1}\big(1-r^p\big)^{\frac{b}{p}}\ud r
=\vol_a(B_{\X})\vol_b(B_{\Y})\frac{\Gamma\big(1+\frac{b}{p}\big)\Gamma\big(1+\frac{a}{p}\big)}{\Gamma\big(1+\frac{a+b}{p}\big)}.\tag*{\qedhere}
\end{multline*}
\end{proof}
By Lemma~\ref{lem:volume of lp direct sum}, for every $m\in \N$, every normed space $\X=(\R^m,\|\cdot\|_{\X})$ satisfies
\begin{equation}\label{eq:ell p X}
\vol_{nm}\big(B_{\ell_p^n(\X)}\big)=\frac{\Gamma\big(1+\frac{m}{p}\big)^n}{\Gamma\big(1+\frac{nm}{p}\big)}\vol_m(B_{\X})^n\quad\mathrm{and\ hence}\quad  \vol_{nm}\big(B_{\ell_p^n(\X)}\big)^{\frac{1}{nm}}\asymp \frac{\vol_{m}(B_{\X})^{\frac{1}{m}}}{n^{\frac{1}{p}}}.
\end{equation}
In particular, for every $m,n\in \N$ and $1\le p,q\le \infty$ we have
\begin{equation}\label{eq:volume of ellp ellq}
\vol_{nm}\big(B_{\ell_p^n(\ell_q^m)}\big)= \frac{2^{nm}\Gamma\big(1+\frac{1}{q}\big)^{nm}\Gamma\big(1+\frac{m}{p}\big)^n}{\Gamma\big(1+\frac{m}{q}\big)^n\Gamma\big(1+\frac{nm}{p}\big)}\quad\mathrm{and\ hence}\quad \vol_{nm}\big(B_{\ell_p^n(\ell_q^m)}\big)^{\frac{1}{nm}}\asymp \frac{1}{n^{\frac{1}{p}}m^{\frac{1}{q}}}.
\end{equation}

The following simple lemma records an extension of the second part of~\eqref{eq:ell p X} to $m$-fold iterations of the operation $\X\mapsto \ell_p^n(\X)$, i.e., to spaces of the form  $$\ell_{p_m}^{n_m}\Big(\ell_{p_{m-1}}^{n_{m-1}}\big(\cdots\ell_{p_1}^{n_1}(\X )\cdots \big)\Big);$$ the main point for us  here is that the implicit constants remain bounded as $m\to \infty$.

\begin{lemma}\label{lem:volume radius of nested} Fix $\{n_k\}_{k=0}^\infty\subset  \N$ and $\{p_k\}_{k=1}^\infty\subset [1,\infty]$. Let $\X=(\R^{n_0},\|\cdot\|_\X)$ be a normed space and define
$$
\forall k\in \N\cup\{0\}, \qquad \X_{k+1}=\ell_{p_k}^{n_k}(\X_k), \qquad\mathrm{where}\qquad \X_0=\X.
$$
Then, for every $m\in \N$ we have
$$
\vol_{n_0\cdots n_m}\big(B_{\X_m}\big)^{\frac{1}{n_0\cdots n_k}}\asymp \frac{\vol_{n_0}\big(B_\X\big)^{\frac{1}{n_0}}}{\prod_{k=1}^m n_k^{\frac{1}{p_k}}}.
$$
\end{lemma}

\begin{proof} With the convention that an empty product equals $1$, by applying~\eqref{eq:ell p X} inductively we see that
$$
 \vol_{n_0\cdots n_m}\big(B_{\X_m}\big)=\vol_{n_0}\big(B_\X\big)^{n_1\cdots n_m}\prod_{k=1}^m \frac{\Gamma\big(1+\frac{n_0\cdots n_{k-1}}{p_k}\big)^{n_k\cdots n_m}}{\Gamma\big(1+\frac{n_0\cdots n_{k}}{p_k}\big)^{n_{k+1}\cdots n_m}}.
$$
Hence,
\begin{equation}\label{eq:vo nested lp}
\frac{\vol_{n_0\cdots n_m}\big(B_{\X_m}\big)^{\frac{1}{n_0\cdots n_k}}\prod_{k=1}^m n_k^{\frac{1}{p_k}}}{\vol_{n_0}\big(B_\X\big)^{\frac{1}{n_0}}}=\prod_{k=1}^m \frac{\Gamma\big(1+\frac{n_0\cdots n_{k-1}}{p_k}\big)^{\frac{1}{n_0\cdots n_{k-1}}}}{\Gamma\big(1+\frac{n_0\cdots n_{k}}{p_k}\big)^{\frac{1}{n_{0}\cdots n_k}}}n_k^{\frac{1}{p_k}}=\prod_{k=1}^m f_{n_0\cdots n_{k-1},n_k}\left(\frac{1}{p_k}\right),
\end{equation}
where for  $u,v,t>0$ we denote
$$
f_{u,v}(t)\eqdef \frac{\Gamma(1+ut)^{\frac{1}{u}}}{\Gamma(1+uvt)^{\frac{1}{uv}}}v^t.
$$
Since $\big(\log \Gamma(z)\big)'=\int_0^\infty \frac{se^{-zs}}{1-e^{-s}}\ud s$ for $z>0$  (see e.g.~\cite[Chapter~XII]{WW62}), if $u,t>0$ and $v\ge1$, then
$$
\frac{\ud}{\ud t} \log f_{u,v}(t)=\log v+\int_0^\infty \left(e^{-uts}-e^{-uvts}\right)\frac{s e^{-s}}{1-e^{-s}}\ud s\ge 0.
$$
Thus, $f_{u,v}$ is increasing on $[0,\infty)$, and therefore we get from~\eqref{eq:vo nested lp} that
\begin{equation*}
1=\prod_{k=1}^m f_{n_0\cdots n_{k-1},n_k}(0)\le \frac{\vol_{n_0\cdots n_m}\big(B_{\X_m}\big)^{\frac{1}{n_0\cdots n_k}}\prod_{k=1}^m n_k^{\frac{1}{p_k}}}{\vol_{n_0}\big(B_\X\big)^{\frac{1}{n_0}}}\le \prod_{k=1}^m f_{n_0\cdots n_{k-1},n_k}(1)=\frac{(n_0!)^{\frac{1}{n_0}}n_1\cdots n_m}{\big((n_0\cdots n_m)!\big)^{\frac{1}{n_0\cdots n_m}}}\le e.\tag*{\qedhere}
\end{equation*}
\end{proof}

The first part of Lemma~\ref{lem:unconditional composiiton-later} below is a restatement of Lemma~\ref{lem:unconditional composiiton} from the Introduction. Qualitatively, it shows that the class of spaces for which Conjecture~\ref{weak isomorphic reverse conj1} holds is closed under unconditional composition, namely, norms of the form~\eqref{eq:def composed norm-later} below. The second part of Lemma~\ref{lem:unconditional composiiton-later} is further information that pertains to Conjecture~\ref{conj:weak reverse iso when canonical}, i.e., to the symmetric version of the weak reverse isoperimetric conjecture, for which we want the operator $S$ to be the identity mapping (i.e., weak reverse isoperimetry holds without the need to first change the ``position'' of the given normed space).

\begin{lemma}\label{lem:unconditional composiiton-later} Fix $n,m_1,\ldots,m_n\in \N$. Let $\X_1=(\R^{m_1},\|\cdot\|_{\X_1}),\ldots,\X_n=(\R^{m_n},\|\cdot\|_{\X_{n}})$ be normed spaces. Also, let $\bfE=(\R^n,\|\cdot\|_\bfE)$ be an unconditional normed space. Define a normed space $\X=(\R^{m_1}\times\ldots\times \R^{m_n},\|\cdot\|_\X)$ by
\begin{equation}\label{eq:def composed norm-later}
\forall x=(x_1,\ldots,x_n)\in \R^{m_1}\times\ldots\times \R^{m_n},\qquad \|x\|_\X\eqdef \big\|\big(\|x_1\|_{\X_1},\ldots,\|x_n\|_{\X_n}\big)\big\|_\bfE.
\end{equation}
Then,  Conjecture~\ref{weak isomorphic reverse conj1} (equivalently, Conjecture~\ref{conj:reverse FK}) holds for $\X$ if it holds for $\X_1,\ldots,\X_n$.

More precisely, suppose that there exist $\alpha>0$, linear transformations  $S_1\in \SL_{m_1}(\R),\ldots,S_n\in \SL_{m_n}(\R)$, and normed spaces $\Y_1=(\R^{m_1},\|\cdot\|_{\Y_1}),\ldots,\Y_n=(\R^{m_n},\|\cdot\|_{\Y_{n}})$ such that
\begin{equation}\label{eq:Yk assumption less alpha}
\forall k\in \n,\qquad B_{\Y_k}\subset S_kB_{\X_k}\qquad\mathrm{and}\qquad \frac{\iq\big(B_{\Y_k}\big)}{\sqrt{m_k}}\left(\frac{\vol_{m_k}\big(B_{\X_k}\big)}{\vol_{m_k}\big(B_{\Y_k}\big)}\right)^{\frac{1}{m_k}}\le \alpha.
\end{equation}
Then, there exist a normed space $\Y=(\R^{m_1}\times\ldots\times \R^{m_n},\|\cdot\|_\X)$ and $S\in \SL(\R^{m_1}\times\ldots\times \R^{m_n})$ such that
\begin{equation}\label{eq:Y conclusion composition}
B_\Y\subset S B_\X\qquad \mathrm{and}\qquad \frac{\iq(B_{\Y})}{\sqrt{m_1+\ldots+m_n}}\left(\frac{\vol_{m_1+\ldots+m_n}(B_{\X})}{\vol_{m_1+\ldots+m_n}(B_{\Y})}\right)^{\frac{1}{m_1+\ldots+m_n}}\lesssim \alpha.
\end{equation}

If furthermore  $S_1,\ldots, S_n$ are all identity mappings (of the respective dimensions), then $S$ can be taken to be the identity mapping provided the the following two conditions hold:
\begin{equation}\label{eq:Loz are uniform assumption}
\Big\|\sum_{i=1}^n e_i\Big\|_\bfE\Big\|\sum_{i=1}^n e_i\Big\|_{\bfE^{\textbf{*}}}\lesssim  n,
\end{equation}
and
\begin{equation}\label{eq:geometric mean condition}
\bigg(\prod_{k=1}^n m_k^{m_k}\vol_{m_k}\big(B_{\X_k}\big)\bigg)^{\frac{1}{m_1+\ldots+m_n}}\lesssim \frac{m_1+\ldots+m_n}{n}\min_{k\in \n} \vol_{m_k}\big(B_{\X_k}\big)^{\frac{1}{m_k}}.
\end{equation}
Note that~\eqref{eq:geometric mean condition} is satisfied in particular  if $m_i\asymp m_j$ and $\vol_{m_i}(B_{\X_i})^{\frac{1}{m_i}}\asymp \vol_{m_i}(B_{\X_j})^{\frac{1}{m_j}}$ for every $i,j\in \n$.
\end{lemma}

Prior to proving~\eqref{lem:unconditional composiiton-later}  we will make some  basic observations. Firstly, \eqref{eq:def composed norm-later} indeed defines a norm because it is well-known that the requirement that $\bfE=(\R^n,\|\cdot\|_\bfE)$ is an unconditional normed space is equivalent  to (see e.g.~\cite[Proposition~1.c.7]{LT77}) the following ``contraction property.''
\begin{equation}\label{eq:contraction unc}
\forall a,x\in \R^n,\qquad \|(a_1x_1,\ldots,a_nx_n)\|_\bfE\le \|a\|_{\ell_\infty^n} \|x\|_\bfE.
\end{equation}
Thus, $\|x\|_\bfE\le \|y\|_\bfE$ if $x,y\in \R^n$ satisfy $|x_i|\le |y_i|$ for every $i\in \n$, so the triangle inequality for~\eqref{eq:def composed norm-later} follows from applying the triangle inequalities entry-wise for each of the norms $\{\|\cdot\|_{\X_i}\}_{i=1}^n$, using this monotonicity property, and then applying the triangle inequality for $\|\cdot\|_{\bfE}$.

It is well-known that condition~\eqref{eq:Loz are uniform assumption} holds (as an equality) when $\bfE$ is a symmetric normed space (see e.g.~\cite[Proposition~3.a.6]{LT79}). More generally, condition~\eqref{eq:Loz are uniform assumption} holds (also as an equality)  in the setting of the following simple averaging lemma, which shows in particular that Lemma~\ref{lem:unconditional composiiton-later}  implies Lemma~\ref{lem:weak iso for enouhg permutations and unconditional}.

\begin{lemma}\label{lem:loz is uniform} Suppose that $\X=(\R^n,\|\cdot\|_\X)$ is a normed space such that for every $j,k\in \n$ there exists a permutation $\pi\in S_n$ with $\pi(j)=k$ such that $\|\sum_{i=1}^n a_{\pi(i)}e_i\|_\X= \|\sum_{i=1}^n a_{i}e_i\|_\X$ for every $a_1,\ldots,a_n\in \R$. Then, $$\Big\|\sum_{i=1}^n e_i\Big\|_\X\Big\|\sum_{i=1}^n e_i\Big\|_{\X^{\textbf{*}}}=n.$$
\end{lemma}

\begin{proof} Denote $\mathfrak{S}(\X)=\{\pi\in S_n:\ T_\pi\in \mathsf{Isom}(\X)\}$, where $T_\pi\in \GL_n(\R)$ was defined in Example~\ref{ex permutation and sign} for each $\pi \in S_n$. Then, $\mathfrak{S}(\X)$ is a subgroup of $S_n$ that we are assuming acts transitively on $\n$. Consequently,
\begin{equation}\label{eq:equal size orbits}
\forall i,j\in \n,\qquad |\{\pi\in \mathfrak{S}(\X):\ \pi(i)=j\}|=\frac{|\mathfrak{S}(\X)|}{n}.
\end{equation}
\begin{comment}
Indeed, for each fixed $i\in \n$ the $n$ sets $\{\{\pi\in \mathfrak{S}(\X):\ \pi(i)=j\}:\ j\in \n\}$ form a partition of $\mathfrak{S}(\X)$, so it suffices to check that $|\{\pi\in \mathfrak{S}(\X):\ \pi(i)=j\}|=|\{\pi\in \mathfrak{S}(\X):\ \pi(i)=k\}|$ for every $j,k\in \n$. By assumption, there exists $\sigma\in \mathfrak{S}(\X)$ for which $\sigma(j)=k$, so multiplication on the left by  $\sigma$ is a bijection between $\{\pi\in \mathfrak{S}(\X):\ \pi(i)=k\}$ and $\{\pi\in \mathfrak{S}(\X):\ \pi(i)=j\}$.
\end{comment}

For every $a_1,\ldots,a_n\in \R$ we have
$$
\frac{1}{|\mathfrak{S}(\X)|}\sum_{\pi\in \mathfrak{S}(\X)} \sum_{i=1}^n a_{\pi(i)}e_i=\sum_{i=1}^n \bigg(\sum_{j=1}^n \frac{|\{\pi\in \mathfrak{S}(\X):\ \pi(i)=j\}|}{|\mathfrak{S}(\X)|}a_j\bigg)e_i\stackrel{\eqref{eq:equal size orbits}}{=} \frac{\sum_{j=1}^na_j}{n}\sum_{i=1}^n e_i.
$$
Hence,
\begin{align*}
\Big|\Big\langle \sum_{j=1}^n e_j,\sum_{j=1}^n a_j e_j\Big\rangle\Big| =\Big|\sum_{j=1}^na_j\Big|&=\frac{n\big\|\frac{1}{|\mathfrak{S}(\X)|}\sum_{\pi\in \mathfrak{S}(\X)} \sum_{i=1}^n a_{\pi(i)}e_i\big\|_\X}{\big\|\sum_{i=1}^n e_i\big\|_\X}\\&\le \frac{\frac{n}{|\mathfrak{S}(\X)|}\sum_{\pi\in \mathfrak{S}(\X)}\big\|\sum_{i=1}^n a_{\pi(i)}e_i\big\|_\X}{\big\|\sum_{i=1}^n e_i\big\|_\X}=\frac{n\big\|\sum_{i=1}^n a_i e_i\big\|_\X}{\big\|\sum_{i=1}^n e_i\big\|_\X},
\end{align*}
where the penultimate step uses convexity and the final step uses the assumption that $T_\pi$ is an isometry of $\X$ for every $\pi\in \mathfrak{S}(\X)$. Since this holds for every $a_1,\ldots,a_n\in \R$, we have $\|\sum_{i=1}^n e_i\|_{\X^{\textbf{*}}}\le n/\|\sum_{i=1}^n e_i \|_\X$. The reverse inequality holds for any normed space $\X=(\R^n,\|\cdot\|_\X)$ because $\langle \sum_{i=1}^ne_i,\sum_{i=1}^ne_i\rangle=n$.
\end{proof}

By combining Lemma~\ref{lem:unconditional composiiton-later}  and Lemma~\ref{lem:loz is uniform} we obtain the following corollary that establishes Conjecture~\ref{conj:weak reverse iso when canonical} for the iteratively nested $\ell_p$ spaces of Lemma~\ref{lem:volume radius of nested}, provided it holds for the initial space $\X$.

\begin{corollary}\label{lem:nested tensorization} Fix $\{n_k\}_{k=0}^\infty\subset  \N$ and $\{p_k\}_{k=1}^\infty\subset [1,\infty]$. Let $\X=(\R^{n_0},\|\cdot\|_\X)$ be a normed space and define
\begin{equation}\label{eq:nested lp def in intro}
\forall k\in \N, \qquad \X_{k+1}=\ell_{p_k}^{n_k}(\X_k), \qquad\mathrm{where}\qquad \X_0=\X.
\end{equation}
Suppose that $\alpha>0$ and there exists a normed space $\Y=(\R^{n_0},\|\cdot\|_\Y)$ with $B_\Y\subset B_\X$ and that satisfies
\begin{equation}\label{eq:n0 assumption}
\frac{\iq(B_\Y)}{\sqrt{n_0}}\bigg(\frac{\vol_{n_0}(B_\X)}{\vol_{n_0}(B_\Y)}\bigg)^{\frac{1}{n_0}}\le \alpha.
\end{equation}
Then, for every $m\in \N$ there is a normed space $\Y_m=(\R^{n_0\cdots n_m},\|\cdot\|_{\Y_m})$ with $B_{\Y_m}\subset B_{\X_m}$ and
$$
\frac{\iq\big(B_{\Y_m}\big)}{\sqrt{n_0\cdots n_m}}\left(\frac{\vol_{n_0\cdots n_m}\big(B_{\X_m}\big)}{\vol_{n_0\cdots n_m}\big(B_{\Y_m}\big)}\right)^{\frac{1}{n_0\cdots n_m}}\lesssim \alpha,
$$
\end{corollary}

To see why Corollary~\ref{lem:nested tensorization} indeed follows from Lemma~\ref{lem:unconditional composiiton-later}  and Lemma~\ref{lem:loz is uniform}, observe that if we start with $\bfE_0=\R$ and define inductively $\bfE_{k+1}=\ell_{p_k}^{n_k}(\bfE_k)$, then for each $m\in \N$ the space $\bfE_m$ is unconditional and satisfies the assumptions of Lemma~\ref{lem:loz is uniform}. The space $\Y_m$ of Corollary~\ref{lem:nested tensorization} is the same space that is defined in Lemma~\ref{lem:unconditional composiiton-later} if we take $\bfE=\bfE_m$, and also $\X_1=\ldots=\X_m=\X$, which ensures that~\eqref{eq:geometric mean condition} holds.

\begin{proof}[Proof of Lemma~\ref{lem:unconditional composiiton-later}]   Denote
\begin{equation}\label{eq:def Mrho}
M\eqdef \sum_{k=1}^n m_k=\dim(\X)\qquad\mathrm{and}\qquad \forall k\in \n,\qquad \rho_k\eqdef \vol_{m_k}(B_{\X_k})^{\frac{1}{m_k}}.
\end{equation}
Fix positive numbers $c,C_1,\ldots,C_n, \gamma_1,\ldots,\gamma_n,w_1,\ldots w_n,w_1^*,\ldots,w_n^*,\beta_1,\ldots,\beta_n>0$ that satisfy the following conditions (their values will be specified later). Firstly, we require that
\begin{equation}\label{eq:loza theta}
\Big\|\sum_{i=1}^n w_i e_i\Big\|_\bfE=\Big\|\sum_{i=1}^n w^*_i e_i\Big\|_{\bfE^{\textbf{*}}}=1.
\end{equation}
Secondly, we require that
\begin{equation}\label{eq:ww* lower}
\forall k\in \n,\qquad w_kw_k^*\ge \frac{m_k}{\gamma_k M}.
\end{equation}
Finally, we require that
\begin{equation}\label{eq:final choice of ak}
\forall k\in \n,\qquad \frac{1}{c w_k\rho_k}\le \beta_k\le \frac{C_k}{ w_k\rho_k} ,
\end{equation}

Denote
\begin{equation}\label{eq:def D determinant}
D\eqdef \bigg(\prod_{k=1}^n \beta_k^{{m_k}}\bigg)^{\frac{1}{M}}.
\end{equation}
Consider the block diagonal linear operator $S:\R^{m_1}\times\ldots\times \R^{m_n}\to \R^{m_1}\times\ldots\times \R^{m_n}$ that is given by
\begin{equation}\label{eq:def S block}
\forall x=(x_1,\ldots,x_n)\in \R^{m_1}\times\ldots\times \R^{m_n},\qquad Sx\eqdef \frac{1}{D}\big(\beta_1S_1x_1,\ldots,\beta_nS_nx_n\big).
\end{equation}
The normalization by $D$ in~\eqref{eq:def S block} ensures that $S\in \SL(\R^{m_1}\times\ldots\times \R^{m_n})$.

Since $\sum_{k=1}^nw_k^*e_k$ is a unit functional in $\bfE^*$, for every $x=(x_1,\ldots,x_n)\in \R^{m_1}\times\ldots\times \R^{m_n}$ we have
$$
\big\|S^{-1}x\big\|_\X\stackrel{\eqref{eq:def composed norm-later}\wedge \eqref{eq:def S block}}{=}D\bigg\|\sum_{k=1}^n \frac{\|S_k^{-1}x_k\|_{\X_k}}{\beta_k}e_k\bigg\|_\bfE\stackrel{\eqref{eq:loza theta}}{\ge} D\left\langle\sum_{k=1}^n w_k^* e_k,\sum_{k=1}^n \frac{\|S_k^{-1}x_k\|_{\X_k}}{\beta_k}e_k\right\rangle  \stackrel{\eqref{eq:ww* lower}}{\ge} \frac{D}{M}\sum_{k=1}^n \frac{m_k\|S_k^{-1}x_k\|_{\X_k}}{\gamma_kw_k\beta_k}.
$$
This shows that
\begin{equation}\label{eq:SBX in l1 sum}
SB_\X\subset  \bigg\{x\in \R^{m_1}\times\ldots\times \R^{m_n}:\ \sum_{k=1}^n \frac{m_k\|S_k^{-1}x_k\|_{\X_k}}{\gamma_kw_k\beta_k}\le \frac{M}{D} \bigg\}=\frac{M}{D}  B_{\left(\frac{\gamma_1w_1\beta_1}{m_1}S_1 \X_1\right)\oplus_1\ldots\oplus_1 \left(\frac{\gamma_nw_n\beta_n}{m_n}S_n\X_n\right)}.
\end{equation}
Using Lemma~\ref{lem:volume of lp direct sum}, we therefore have
\begin{align}\label{eq:volume upper bound compised norm}
\begin{split}
\vol_{M}(B_\X)^{\frac{1}{M}}&\le \frac{M}{D} \vol_{M}\bigg(B_{\left(\frac{\gamma_1w_1\beta_1}{m_1}S_1 \X_1\right)\oplus_1\ldots\oplus_1 \left(\frac{\gamma_nw_n\beta_n}{m_n}S_n\X_n\right)}\bigg)^{\frac{1}{M}}\stackrel{\eqref{eq:volume formula for direct sum}}{=}\frac{1}{D}\left(
\frac{M^M}{M!}\prod_{k=1}^nm_k!\left(\frac{\gamma_kw_k\beta_k\rho_k}{m_k}\right)^{m_k}\right)^{\frac{1}{M}}\\&\stackrel{\eqref{eq:final choice of ak}}{\le} \frac{1}{D}\left(
\frac{M^M}{M!}\prod_{k=1}^n\frac{m_k!}{m_k^{m_k}}(\gamma_kC_k)^{m_k}\right)^{\frac{1}{M}}\le \frac{e}{D}\left(
\prod_{k=1}^n(\gamma_kC_k)^{m_k}\right)^{\frac{1}{M}}.
\end{split}
\end{align}

Next, for every $x=(x_1,\ldots,x_n)\in \R^{m_1}\times\ldots\times \R^{m_n}$ we have
\begin{align*}
\big\|S^{-1}x\big\|_\X\stackrel{\eqref{eq:def composed norm-later}\wedge \eqref{eq:def S block}}{=}D&\bigg\|\sum_{k=1}^n \frac{\|S_k^{-1}x_k\|_{\X_k}}{\beta_k}e_k\bigg\|_\bfE\\
 &\stackrel{\eqref{eq:contraction unc}}{\le} D\bigg(\max_{k\in \n} \frac{\|S_k^{-1}x_k\|_{\X_k}}{w_k\beta_k}\bigg) \bigg\|\sum_{k=1}^n w_k e_k\bigg\|_\bfE\stackrel{\eqref{eq:loza theta}}{=}D\max_{k\in \n} \frac{\|S_k^{-1}x_k\|_{\X_k}}{w_k\beta_k}.
\end{align*}
This establishes the following inclusion.
\begin{equation}\label{eq:product inclusion}
SB_\X\supseteq \frac{1}{D}\prod_{k=1}^n w_k\beta_k S_k B_{\X_k}\eqdef \Omega.
\end{equation}
Thanks to~\eqref{eq:spoectral reverse iso equiv}, the assumption~\eqref{eq:Yk assumption less alpha} of Lemma~\ref{lem:unconditional composiiton-later} implies that
\begin{equation}\label{eq:assumption on Sk Xl}
\forall k\in \n,\qquad \lambda\big(S_kB_{\X_k}\big)\rho_k^2\stackrel{\eqref{eq:def Mrho}}{=}\lambda\big(S_kB_{\X_k}\big)\vol_{m_k}\big(B_{\X_k}\big)^{\frac{2}{m_k}}\lesssim \alpha^2 m_k.
\end{equation}
For each $k\in \n$ take $f_k:S_kB_{\X_k}\to \R$ that is smooth on the interior of $S_kB_{\X_k}$, vanishes on $\partial S_kB_{\X_k}$, and satisfies $\Delta f_k=-\lambda(S_kB_{\X_k})f_k$ on the interior of $S_kB_{\X_k}$. Define $f:\Omega\to \R$ by
$$
\forall x=(x_1,\ldots,x_n)\in \Omega=\frac{1}{D}\prod_{k=1}^n w_k\beta_k S_k B_{\X_k},\qquad f(x)\eqdef  \prod_{k=1}^n f_k\Big(\frac{D}{w_k\beta_k}x_k\Big),
$$
Thus $f\equiv 0$ on the boundary of $\Omega$ and on the interior of  $\Omega$ it is smooth and satisfies
\begin{equation}\label{eq:eigenvalue on weighted product}
\Delta f=-D^2\bigg(\sum_{k=1}^n \frac{\lambda\big(S_kB_{\X_k}\big)}{(w_k\beta_k)^2}\bigg)f
\end{equation}
Hence,
\begin{equation}\label{eq:eigenvalue bound on composed norm}
\lambda(S\X)=\lambda(SB_\X)\stackrel{\eqref{eq:product inclusion}}{\le} \lambda(\Omega)\stackrel{\eqref{eq:eigenvalue on weighted product}}{\le} D^2\bigg(\sum_{k=1}^n \frac{\lambda\big(S_kB_{\X_k}\big)}{(w_k\beta_k)^2}\bigg)\stackrel{\eqref{eq:final choice of ak}}{\le} (c D)^2\bigg(\sum_{k=1}^n \lambda\big(S_kB_{\X_k}\big)\rho_k^2\bigg) \stackrel{\eqref{eq:assumption on Sk Xl}}{\lesssim} (c\alpha D)^2M.
\end{equation}

By combining~\eqref{eq:volume upper bound compised norm} and~\eqref{eq:eigenvalue bound on composed norm} we see that
$$
\lambda(S\X)\vol_{M}(B_\X)^{\frac{2}{M}}\lesssim c^2\bigg(\prod_{k=1}^n (\gamma_kC_k)^{m_k}\bigg)^{\frac{2}{M}}\alpha^2M.
$$
Another application of~\eqref{eq:spoectral reverse iso equiv} now shows that the desired conclusion~\eqref{eq:Y conclusion composition} holds with $\Y=\Ch S \X$ (recall the definition of Cheeger space in Section~\ref{sec:spectral}) provided
\begin{equation}\label{eq:robust condition for new S}
c\bigg(\prod_{k=1}^n (\gamma_kC_k)^{m_k}\bigg)^{\frac{1}{M}}\lesssim 1.
\end{equation}

To get~\eqref{eq:Y conclusion composition}, by the Lozanovski\u{\i} factorization theorem~\cite{Loz69} there exist $w_1,\ldots w_n,w_1^*,\ldots,w_n^*>0$ such that~\eqref{eq:loza theta} holds and also $w_kw_k^*=m_k/M$ for every $k\in \n$. Thus~\eqref{eq:ww* lower} holds (as an equality) if we choose $\gamma_1=\ldots=\gamma_n=1$. If we take $c=C_1=\ldots=C_n=1$ and $\beta_k=1/(w_k\rho_k)$ for each $k\in \n$, then  both~\eqref{eq:final choice of ak} and~\eqref{eq:robust condition for new S} also hold (as equalities). With these choices, \eqref{eq:Y conclusion composition} holds.

Suppose that the additional assumptions~\eqref{eq:Loz are uniform assumption} and~\eqref{eq:geometric mean condition} hold. Denote $\eta=\|\sum_{i=1}^n e_i\|_{\bfE}\|\sum_{i=1}^n e_i\|_{\bfE^{\textbf{*}}}/n$. So, $\eta=O(1)$ by~\eqref{eq:Loz are uniform assumption}.  Take
 $w_1=\ldots=w_n=1/\|\sum_{i=1}^n e_i\|_{\bfE}$ and $w_1^*=\ldots=w_n^*=1/\|\sum_{i=1}^n e_i\|_{\bfE^{\textbf{*}}}$, so that~\eqref{eq:loza theta} holds by design. This choice also ensures that if we take  $\gamma_k=m_k/(\eta M)$ for each $k\in \n$, then~\eqref{eq:ww* lower} holds (as an equality). Next, choose  $C_k=\rho_k$ for each $k\in \n$, as well as $\beta_1=\ldots=\beta_n=\|\sum_{i=1}^n e_i\|_\bfE$ and  $c=1/\min_{k\in \n} \rho_k$. This ensures that~\eqref{eq:final choice of ak} holds,  and also that~\eqref{eq:robust condition for new S} coincides with the assumption~\eqref{eq:geometric mean condition}, since $\eta=O(1)$. The desired conclusion~\eqref{eq:Y conclusion composition} therefore holds with $Sx=(S_1x_1,\ldots,S_nx_n)$ in~\eqref{eq:def S block}. In particular, if $S_k=\Id_{m_k}$ for every $k\in \n$, then we can take $S=\Id_{\R^{m_1}\times\ldots\times \R^{m_n}}$ in~\eqref{eq:Y conclusion composition}.
\end{proof}

The following lemma provides  a formula for the cone measure of Orlicz spaces. Fix a convex increasing function $\psi:[0,\infty)\to [0,\infty]$ that satisfies $\psi(0)=0$ and $\lim_{x\to \infty} \psi(x)=\infty$ (so, if $\lim_{x\to a^{-}} \psi(x)=\infty$ for some $a\in (0,\infty)$, then we require that $\psi(x)=\infty$ for every $x\ge a$). Henceforth, the associated Orlicz space (see e.g.~\cite{RaoRen91})  $\ell_\psi^n=(\R^n,\|\cdot\|_{\ell_{\psi}^n})$  will always be endowed with the Luxemburg norm that is given by
\begin{equation}\label{eq:define luxemburg}
\forall x\in \R^n,\qquad \|x\|_{\ell_\psi^n}=\inf\bigg\{s>0:\ \sum_{i=1}^n \psi\Big(\frac{|x_i|}{s}\Big)\le 1\bigg\}.
\end{equation}

\begin{lemma}\label{lem:orlicz} Suppose that $\psi:[0,\infty)\to [0,\infty]$ is convex, increasing, continuously differentiable on the set $\{x\in (0,\infty):\ \psi(x)<\infty\}$, and satisfies $\psi(0)=0$ and $\lim_{x\to \infty} \psi(x)=\infty$. Then, for every $g\in L_1(\kappa_{\ell_{\psi}^n})$ we have
\begin{align}\label{eq:orlicz change of variable-with const}
\begin{split}
\frac{n!}{2^n}&\vol_n\big(B_{\ell_{\psi}^n}\big)\int_{\partial B_{\ell_{\psi}^n}}  g(\theta) \ud \kappa_{\ell_{\psi}^n}(\theta)\\&=\int_{\partial B_{\ell_1^n}} g\big(\psi^{-1}(|\tau_i|)\sign(\tau_1),\ldots,\psi^{-1}(|\tau_n|)\sign(\tau_n)\big) \frac{\sum_{i=1}^n \psi^{-1}(|\tau_i|)\psi'\big(\psi^{-1}(|\tau_i|)\big)}{\prod_{i=1}^n\psi'\big(\psi^{-1}(|\tau_i|)\big)}\ud \kappa_{\ell_1^n}(\tau).
\end{split}
\end{align}
\end{lemma}

For example, when $\psi(t)=t^p$ for some $p\ge 1$ and every $t\ge 0$, in which case $\ell_\psi^n=\ell_p^n$, Lemma~\ref{lem:orlicz} gives
$$
\int_{\partial B_{\ell_{p}^n}} g\ud \kappa_{\ell_{\psi}^n}=\frac{\Gamma\big(1+\frac{n}{p}\big)}{n!\Gamma\big(1+\frac{1}{p}\big)^n}\int_{\partial B_{\ell_1^n}} \frac{g\circ M_{ 1\to p}^n (\tau)}{|\tau_1\cdots \tau_n|^{1-\frac{1}{p}}}\ud \kappa_{\ell_1^n}(\tau),
$$
where $M_{1\to p}:\R^n\to \R^n$ is the Mazur map~\cite{Maz29} from $\ell_1^n$ to $\ell_p^n$, i.e.,
$$
\forall x\in \R^n,\qquad M_{1\to p}^n(x_1,\ldots ,x_n)=\big(|x_1|^{\frac{1}{p}}\sign(x_1),\ldots,|x_n|^{\frac{1}{p}}\sign(x_n)\big).
$$
As another special case of Lemma~\ref{lem:orlicz}, consider the following family of Orlicz spaces $\Omega_\beta^n=(\R^n,\|\cdot\|_{\Omega_\beta^n})$:
\begin{equation}\label{eq:our orlicz notation}
\forall \beta>0,\qquad \Omega_\beta^n\eqdef \ell_{\psi_{\!\!\beta}}^n\qquad\mathrm{where}\qquad \forall t\ge 0,\qquad \psi_{\!\!\beta}(t)\eqdef \left\{\begin{array}{ll}\frac{1}{\beta}\log\left(\frac{1}{1-t}\right)&\mathrm{if}\ 0\le t<1,\\
\infty&\mathrm{if}\ t\ge 1.\end{array}\right.
\end{equation}
Observe that by considering the case $g\equiv 1$ of~\eqref{eq:orlicz change of variable-with const} we obtain the following identity.
\begin{equation}\label{eq:orlicz change of variable}
\int_{\partial B_{\ell_{\psi}^n}}  g\ud \kappa_{\ell_{\psi}^n}=\frac{\int_{\partial B_{\ell_1^n}} g\big(\psi^{-1}(|\tau_i|)\sign(\tau_1),\ldots,\psi^{-1}(|\tau_n|)\sign(\tau_n)\big) \frac{\sum_{i=1}^n \psi^{-1}(|\tau_i|)\psi'(\psi^{-1}(|\tau_i|))}{\prod_{i=1}^n\psi'(\psi^{-1}(|\tau_i|))}\ud \kappa_{\ell_1^n}(\tau)}{\int_{\partial B_{\ell_1^n}} \frac{\sum_{i=1}^n \psi^{-1}(|\tau_i|)\psi'(\psi^{-1}(|\tau_i|))}{\prod_{i=1}^n\psi'(\psi^{-1}(|\tau_i|))}\ud \kappa_{\ell_1^n}(\tau)}.
\end{equation}
When $\psi=\psi_{\!\!\beta}$ for some $\beta>0$ (we will eventually need to work with $\beta\asymp n$), for every $\tau\in \partial B_{\ell_1^n}$ we have
\begin{equation}\label{eq:orlicz density beta}
\frac{\sum_{i=1}^n \psi_{\!\!\beta}^{-1}(|\tau_i|)\psi_{\!\!\beta}'\big(\psi_{\!\!\beta}^{-1}(|\tau_i|)\big)}{\prod_{i=1}^n\psi_{\!\!\beta}'\big(\psi_{\!\!\beta}^{-1}(|\tau_i|)\big)}=\frac{\sum_{i=1}^n \big(1-e^{-\beta|\tau_i|}\big)\frac{e^{\beta|\tau_i|}}{\beta}}{\prod_{i=1}^n \frac{e^{\beta|\tau_i|}}{\beta}}=\frac{\beta^{n-1}\sum_{i=1}^n \big(e^{\beta|\tau_i|}-1\big)}{e^{\beta\|\tau\|_{\ell_1^n}}}= \frac{\beta^{n-1}}{e^\beta}\sum_{i=1}^n \big(e^{\beta|\tau_i|}-1\big).
\end{equation}
Consequently, \eqref{eq:orlicz change of variable}  gives the following identity, which we will need later.
\begin{equation}
\int_{\partial B_{\Omega_\beta^n}} g\ud\kappa_{\Omega_\beta^n}=\frac{ \int_{\partial B_{\ell_1^n}} g\big((e^{\beta|\tau_1|}-1)\sign(\tau_1),\ldots,(e^{\beta|\tau_n|}-1)\sign(\tau_n)\big)\sum_{i=1}^n \big(e^{\beta|\tau_i|}-1\big)\ud \kappa_{\ell_1^n}(\tau)}{\int_{\partial B_{\ell_1^n}} \sum_{i=1}^n \big(e^{\beta|\tau_i|}-1\big)\ud \kappa_{\ell_1^n}(\tau)}.
\end{equation}

\begin{proof}[Proof of Lemma~\ref{lem:orlicz}] For each $i\in \n$ define $f_i:\R^n\to \R$ by setting $f_i(0)=0$ and
$$
\forall y\in \R^n\setminus\{0\},\qquad f_i(y)=\|y\|_{\ell_1^n}\psi^{-1}\Big(\frac{|y_i|}{\|y\|_{\ell_1^n}}\Big)\sign(y_i).
$$
Consider $f=(f_1,\ldots,f_n):\R^n\to \R^n$.  Then, $\|f(y)\|_{\ell_\psi^n}=\|y\|_{\ell_1^n}$ for every $y\in \R^n$. Hence, $f(B_{\ell_1^n})=B_{\ell_\psi^n}$. Now,
\begin{multline*}
\int_{\partial B_{\ell_{\psi}^n}} g(\theta)\ud \kappa_{\ell_{\psi}^n}(\theta) \stackrel{\eqref{eq:polar cone}}{=}\frac{1}{\vol_n\big(B_{\ell_\psi^n}\big)}\int_{f(B_{\ell_1^n})} g\Big(\frac{1}{\|x\|_{\ell_\psi^n}}x\Big)\ud x\\= \frac{1}{\vol_n\big(B_{\ell_\psi^n}\big)}\int_{B_{\ell_1^n}} g\Big(\frac{1}{\|f(y)\|_{\ell_\psi^n}}f(y)\Big) |\det f'(y)|\ud y\stackrel{\eqref{eq:polar cone}}{=}\frac{\vol_n\big(B_{\ell_1^n}\big)}{\vol_n\big(B_{\ell_\psi^n}\big)}\int_{\partial B_{\ell_1^n}} g\big(f(\tau)\big) |\det f'(\tau)|\ud \kappa_{\ell_1^n}(\tau),
\end{multline*}
where in the final step we used the fact $f$ is positively homogeneous of order $1$, and hence its derivative is homogeneous of order $0$ almost everywhere ($f$ is   continuously differentiable on $\{y\in \R^n;\ y_1,\ldots,y_n\neq 0\}$). Since the volume of the unit ball of $\ell_1^n$ equals $2^n/n!$, it remains to check that the Jacobian of $f$ satisfies
\begin{equation*}
 \det f'(\tau)=\frac{\sum_{i=1}^n \psi^{-1}(|\tau_i|)\psi'\big(\psi^{-1}(|\tau_i|)\big)}{\prod_{i=1}^n\psi'\big(\psi^{-1}(|\tau_i|)\big)},
\end{equation*}
for every $\tau\in \partial B_{\ell_1^n}$ with $\tau_1,\ldots,\tau_n\neq 0$. This is so because for every such $\tau$ and  $i,j\in \n$ we have
$$
\partial_j f_i(\tau)=\frac{\delta_{ij}-\tau_i\sign(\tau_j)}{\psi'\big(\psi^{-1}(|\tau_i|)\big)} +\psi^{-1}(|\tau_i|)\sign(\tau_i)\sign(\tau_j).
$$
Hence, $f'(\tau)=A(\tau)+u(\tau)\otimes v(\tau)$, where $A(\tau)\in \M_n(\R)$ is the diagonal matrix $\mathrm{Diag}((1/\psi'(\psi^{-1}(|\tau_i|)))_{i=1}^n)$ and  $u(\tau)=(\psi^{-1}(|\tau_i|)\sign(\tau_i)-\tau_i/\psi'(\psi^{-1}(|\tau_i|)))_{i=1}^n, v(\tau)=(\sign(\tau_i))_{i=1}^n\in \R^n$. By the textbook formula for the determinant of a rank-$1$ perturbation of an invertible matrix (e.g.~\cite[Section~6.2]{Mey00}), it follows that
\begin{multline*}
\det f'(\tau)=\big(1+\langle A(\tau)^{-1}u(\tau),v(\tau)\rangle\big)  \det A(\tau)\\
=\frac{1+\sum_{i=1}^n \psi'\big(\psi^{-1}(|\tau_i|)\big)\left(\psi^{-1}(|\tau_i|)\sign(\tau_i)-\frac{\tau_i}{\psi'\big(\psi^{-1}(|\tau_i|)\big)}\right)\sign(\tau_i)}
{\prod_{i=1}^n\psi'\big(\psi^{-1}(|\tau_i|)\big)}= \frac{\sum_{i=1}^n \psi^{-1}(|\tau_i|)\psi'\big(\psi^{-1}(|\tau_i|)\big)}{\prod_{i=1}^n\psi'\big(\psi^{-1}(|\tau_i|)\big)}. \tag*{\qedhere}
\end{multline*}

\end{proof}

Another description of $\kappa_{\X}$  is the fact (see e.g.~\cite[Lemma~1]{NR03}) that   the Radon--Nikod\'ym derivative  of the $(n-1)$-dimensional Hausdorff (non-normalized surface area) measure on $\partial B_{\X}$ with respect to the (non-normalized cone) measure $\vol_n(B_{\X})\kappa_{\X}$ is equal at almost every $x\in \partial B_{\X}$ to $n$ times the Euclidean length of the gradient at $x$ of the function $u\mapsto\|u\|_{\X}$. In other words, for any $g\in L_1(\partial B_{\X})$,
\begin{equation}\label{eq:cone radon nikodym}
\int_{\partial B_{\X}} g(x)\ud x= n\vol_n(B_{\X})\int_{\partial B_{\X}}g(x)\big\|\nabla\|\cdot\|_{\X}(x)\big\|_{\ell_2^n}\ud\kappa_{\X}(x).
\end{equation}
The special case $g\equiv 1$ of~\eqref{eq:cone radon nikodym} gives the following identity.
\begin{equation}\label{eq:norm of gradient identity}
\frac{\vol_{n-1}(\partial B_{\X})}{\vol_n(B_{\X})}=n\int_{\partial B_{\X}} \big\|\nabla\|\cdot\|_{\X}(x)\big\|_{\ell_2^n}\ud\kappa_{\X}(x)=\fint_{B_{\X}}\frac{\|\nabla\|\cdot\|_{\X}(x)\|_{\ell_2^n}}{\|x\|_{\X}^{n-1}}\ud x,
\end{equation}
where the second equality in~\eqref{eq:norm of gradient identity} is an application of~\eqref{eq:polar cone} because it is straightforward to check that $\|\nabla\|\cdot\|_{\X}(rx)\|_{\ell_2^n}=\|\nabla\|\cdot\|_{\X}(x)\|_{\ell_2^n}$ for any $r>0$ and  $x\in \R^n$ at which the norm $\|\cdot\|_{\X}$ is smooth.

\begin{remark}
By applying Cauchy--Schwarz to the first equality in~\eqref{eq:norm of gradient identity}, we see that
\begin{align}\label{eq:use cs on volume ratio}
\begin{split}
\frac{\vol_{n-1}(\partial B_{\X})}{\vol_n(B_{\X})}\le n\bigg(\int_{\partial B_{\X}}\big\|\nabla\|\cdot\|_{\X}(x)\big\|_{\ell_2^n}^2\ud\kappa_{\X}(x)\bigg)^{\frac12}=
\bigg(\frac{n}{\vol_n(B_{\X})}\int_{\partial B_{\X}}\big\|\nabla\|\cdot\|_{\X}(x)\big\|_{\ell_2^n}\ud x\bigg)^{\frac12},
\end{split}
\end{align}
where the final step of~\eqref{eq:use cs on volume ratio} is an applications of~\eqref{eq:cone radon nikodym} with $g(x)= \|\nabla\|\cdot\|_{\X}(x)\|_{\ell_2^n}$.  If $\|\cdot\|_{\X}$ is twice continuously differentiable on $\R^n\setminus\{0\}$ and  $\f:\R\to [0,\infty)$ is twice continuously differentiable with $\f'(1)>0$ and $\f''(0)=0$, then because for every $x\in \partial B_{\X}$ the vector $\nabla\|\cdot\|_{\X}(x)/\|\nabla\|\cdot\|_{\X}(x)\|_{\ell_2^n}$ is the unit outer normal to $\partial B_{\X}$ at $x$, by the divergence theorem we have
\begin{multline*}
\int_{\partial B_{\X}}\Delta \big(\f\circ \|\cdot\|_{\X}\big) (x)\ud x=\int_{\partial B_{\X}} \mathrm{div} \nabla\big(\f\circ\|\cdot\|_{\X}\big)(x)\ud x=\int_{\partial B_{\X}} \frac{\langle \nabla(\f\circ\|\cdot\|_{\X})(x),\nabla\|\cdot\|_{\X}(x)\rangle}{\big\|\nabla\|\cdot\|_{\X}(x)\big\|_{\ell_2^n}}\ud x\\=
\int_{\partial B_{\X}} \f'\big(\|x\|_{\X}\big)\big\|\nabla\|\cdot\|_{\X}(x)\big\|_{\ell_2^n}\ud x=\f'(1)\int_{\partial B_{\X}}\big\|\nabla\|\cdot\|_{\X}(x)\big\|_{\ell_2^n}\ud x.
\end{multline*}
A substitution of this identity into~\eqref{eq:use cs on volume ratio} give the following bound.
\begin{equation}\label{eq:laplacian psi}
\frac{\vol_{n-1}(\partial B_{\X})}{\vol_n(B_{\X})}\le\frac{\sqrt{n}}{\sqrt{\f'(1)}}\bigg(\fint_{\partial B_{\X}}\Delta \big(\f\circ \|\cdot\|_{\X}\big) (x)\ud x\bigg)^{\frac12}.
\end{equation}
In particular, for every $p>2$ we have
\begin{equation}\label{eq:laplacial p}
\frac{\vol_{n-1}(\partial B_{\X})}{\vol_n(B_{\X})}\le\sqrt{\frac{n}{p}}\bigg(\fint_{ B_{\X}}\Delta \big(\|\cdot\|_{\X}^p\big) (x)\ud x\bigg)^{\frac12}.
\end{equation}
It is worthwhile to record~\eqref{eq:laplacian psi} separately because this estimate  is sometimes convenient for getting good bounds on  $\vol_{n-1}(\partial B_{\X})$. In particular, by using~\eqref{eq:laplacial p}  when $\X$ is an $\ell_p$ direct sum one can obtain an alternative derivation of some of the ensuing estimates. Another noteworthy consequence of~\eqref{eq:use cs on volume ratio} is when there is a transitive subgroup of permutations $G\le  S_n$ such that $\|(x_{\pi(1)},\ldots,x_{\pi(n)})\|_{\X}=\|x\|_{\X}$ for all $x\in \R^n$ and $\pi\in G$. Under this further symmetry assumption, the first inequality of~\eqref{eq:use cs on volume ratio} becomes
$$
\frac{\vol_{n-1}(\partial B_{\X})}{\vol_n(B_{\X})}\le n^{\frac32} \bigg(\int_{\partial B_{\X}} \bigg(\frac{\partial \|\cdot\|_{\X}}{\partial x_1}(x)\bigg)^2\ud \kappa_{\X}(x) \bigg)^{\frac12}.
$$
\end{remark}

The following lemma provides a probabilistic interpretation of the cone measure  which generalizes the treatment of the special case $\X=\ell_p^n$ by Schechtman--Zinn~\cite{SZ90} and Rachev--R\"uschendorf~\cite{RR91}.

\begin{lemma}[probabilistic representation of  cone measure]\label{lem:generalized cone representtaion} Fix $n\in \N$ and let $\X=(\R^n,\|\cdot \|_{\X})$ be a normed space. Suppose that $\f:[0,\infty)\to [0,\infty)$ is a continuous function such that $\f(0)=0$, $\f(t)>0$ when $t>0$ and $\int_0^\infty r^{n-1}\f(r)\ud r<\infty$. Let $\V$ be a random vector in $\R^n$ whose density at each $x\in \R^n$ is equal to
\begin{equation}\label{eq:phi density}
\frac{1}{n\vol_n(B_{\X})\int_0^\infty r^{n-1}\f(r)\ud r}\f\big(\|x\|_{\X}\big),
\end{equation}
where we note that~\eqref{eq:phi density} in indeed a probability density by~\eqref{eq:polar cone}.  Then, the density of $\|\V\|_{\X}$ at $s\in [0,\infty)$ is equal to $s^{n-1}\f(s)/\int_0^\infty r^{n-1}\f(r)\ud r$. Moreover, the following two assertions hold:
\begin{itemize}
\item $\V/\|\V\|_{\X}$ is distributed according to the cone measure $\kappa_{\X}$,
\item $\|\V\|_{\X}$ and $\V/\|\V\|_{\X}$  are (stochastically) independent.
\end{itemize}
\end{lemma}

\begin{proof} The density of $\|\V\|_{\X}$ at $s\in [0,\infty)$ is equal to
\begin{align*}
\frac{\ud}{\ud s}\Pr \big[\|\V\|_{\X}\le s\big]& \stackrel{\eqref{eq:phi density}}{=} \frac{\ud}{\ud s}\bigg(\frac{1}{n\vol_n(B_{\X})\int_0^\infty r^{n-1}\f(r)\ud r}\int_{sB_{\X}} \f\big(\|x\|_{\X}\big)\ud x\bigg)\\&\stackrel{\eqref{eq:polar cone}}{=} \frac{\ud}{\ud s}\bigg( \frac{\int_0^st r^{n-1}\f(r)\ud r}{\int_0^\infty r^{n-1}\f(r)\ud r}\bigg)= \frac{s^{n-1}\f(s)}{\int_0^\infty r^{n-1}\f(r)\ud r}.
\end{align*}
The rest of Lemma~\ref{lem:generalized cone representtaion} is equivalent to showing that for every measurable $A\subset \partial B_{\X}$  and $\rho>0$,
$$
\Pr\bigg[\frac{\V}{\|\V\|_{\X}}\in A \,\Big|\, \|\V\|_{\X}=\rho\bigg]=\kappa_{\X}(A).
$$
To prove this identity, observe first that for every $a,b\in \R$ with $a<b$ we have
$$
\vol_n([a,b]A)=\vol_n \bigg(b\Big(\big([0,1]A\big)\setminus \big(\frac{a}{b}[0,1]A\big)\Big)\bigg)=(b^n-a^n)\vol_n([0,1]A).
$$
Hence, it follows from the definition~\eqref{eq:def cone measure} that
\begin{equation}\label{eq:ab cone version}
\kappa_{\X}(A)=\frac{\vol_n([a,b]A)}{\vol_n([a,b]\partial B_{\X})}.
\end{equation}
Consequently,
\begin{multline*}
\Pr\bigg[\frac{\V}{\|\V\|_{\X}}\in A \,\Big|\, \|\V\|_{\X}=\rho\bigg]=\lim_{\e\to 0}\frac{\Pr[\V\in \|\V\|_{\X} A\ \mathrm{and}\ \rho-\e\le \|\V\|_{\X}\le \rho+\e]}{\Pr[\rho-\e\le \|\V\|_{\X}\le \rho+\e]}\\
 =\lim_{\e\to 0} \frac{\int_{([0,\infty)A)\cap ([\rho-\e,\rho+\e]\partial B_{\X})}\f(\|x\|_{\X})\ud x}{\int_{[\rho-\e,\rho+\e]\partial B_{\X}}\f(\|x\|_{\X})\ud x}=\lim_{\e\to 0} \frac{\vol_n([\rho-\e,\rho+\e]A)}{\vol_n([\rho-\e,\rho+\e]\partial B_{\X})}=\kappa_{\X}(A),
 \end{multline*}
where the penultimate step holds as $\f$ is continuous at $\rho$ and $\f(\rho)>0$, and the final step uses~\eqref{eq:ab cone version}.
\end{proof}

\begin{lemma}\label{lem:break surface of ellp(X)} Fix $m,n\in \N$ and $p\in (1,\infty)$. Suppose that $\X=(\R^{m},\|\cdot\|_{\X})$ is a normed space. Let $\RR_1,\ldots,\RR_n$ be i.i.d.~random variables taking values in $[0,\infty)$ whose density at each $t\in (0,\infty)$ is equal to
\begin{equation}\label{eq:R density}
\frac{p}{2(p-1)\Gamma\big(\frac{m}{p}\big)} t^{\frac{m}{2p-2}-1}e^{-t^{\frac{p}{2p-2}}}.
\end{equation}
Then,
\begin{equation}\label{eq:identity surface of ellpnX}
\frac{\vol_{nm-1}\big(\partial B_{\ell_p^n(\X)}\big)}{\vol_{nm}\big(B_{\ell_p^n(\X)}\big)}= \frac{p\Gamma\big(1+\frac{nm}{p}\big)}{\Gamma\big(1+\frac{nm-1}{p}\big)}\int_{(\partial B_{\X})^n}\E \bigg[\Big(\sum_{i=1}^n \mathsf{R}_i \big\|\nabla \|\cdot\|_{\X} (x_i)\big\|_{\ell_2^m}^2\Big)^{\frac12}\bigg]\ud\kappa_{\X}^{\otimes n}(x_1,\ldots,x_n).
\end{equation}
Furthermore,
\begin{equation}\label{eq:quadratic gradient identity}
\int_{\partial B_{\ell_p^n(\X)}} \big\|\nabla \|\cdot\|_{\ell_p^n(\X)}\big\|^2_{\ell_2^n(\ell_2^m)}\ud \kappa_{\ell_p^n(\X)}=\frac{n\Gamma\big(\frac{nm}{p}\big)\Gamma\big(\frac{m+2p-2}{p}\big)}{\Gamma\big(\frac{m}{p}\big)\Gamma\big(\frac{nm+2p-2}{p}\big)}\int_{\partial B_{\X}} \big\|\nabla\|\cdot\|_{\X}\big\|_{\ell_2^m}^2\ud\kappa_{\X}.
\end{equation}
\end{lemma}

\begin{proof} For almost every $x=(x_1,\ldots, x_n)\in \ell_p^n(\X)$ we have
$$
\nabla\|\cdot\|_{\ell_p^n(\X)}(x)=\frac{1}{\|x\|_{\ell_p^n(\X)}^{p-1}}\big(\|x_1\|_{\X}^{p-1}\nabla\|\cdot\|_{\X}(x_1),
\ldots,\|x_n\|_{\X}^{p-1}\nabla\|\cdot\|_{\X}(x_n)\big).
$$
Consequently,
\begin{align}\label{eq:normalized gradient identity}
\begin{split}
\|x\|_{\ell_p^n(\X)}^{p-1} \Big\|\nabla\|\cdot\|_{\ell_p^n(\X)}\Big(\frac{x}{\|x\|_{\ell_p^n(\X)}}\Big)\Big\|_{\ell_2^n(\ell_2^m)}&=\bigg(\sum_{i=1}^n \|x_i\|_{\X}^{2p-2}\Big\|\nabla\|\cdot\|_{\X}\Big(\frac{x_i}{\|x\|_{\ell_p^n(\X)}}\Big)\Big\|_{\ell_2^{m}}^2\bigg)^{\frac{1}{2}}\\
&= \bigg(\sum_{i=1}^n \|x_i\|_{\X}^{2p-2}\Big\|\nabla\|\cdot\|_{\X}\Big(\frac{x_i}{\|x_i\|_{\X}}\Big)\Big\|_{\ell_2^{m}}^2\bigg)^{\frac{1}{2}},
\end{split}
\end{align}
where we used the straightforward fact that the gradient of any (finite dimensional) norm is homogeneous of order $0$ (on its domain of definition, which is almost everywhere).

Let $\V=(\V_1,\ldots, \V_n)$ be a random vector on $\ell_p^n(\X)$ whose density at $x=(x_1,\ldots, x_n)\in \ell_p^n(\X)$ is
\begin{equation}\label{eq:the density on Y}
\frac{1}{\Gamma\big(1+\frac{nm}{p}\big)\vol_{nm}\big(B_{\ell_p^n(\X)}\big)}e^{-\|x\|_{\ell_p^n(\X)}^p}=\frac{1}{\Gamma\big(1+\frac{nm}{p}\big)
\vol_{nm}\big(B_{\ell_p^n(\X)}\big)}
\prod_{i=1}^n e^{-\|x_i\|_{\X}^p}.
\end{equation}
By combining Lemma~\ref{lem:generalized cone representtaion} with the first equality in~\eqref{eq:norm of gradient identity}, we see that
\begin{equation}\label{eq:sphere ratio By}
\frac{\vol_{nm-1}\big(\partial B_{\ell_p^n(\X)}\big)}{\vol_{nm}\big(B_{\ell_p^n(\X)}\big)} =nm \E \bigg[\Big\|\nabla\|\cdot\|_{\ell_p^n(\X)}\Big(\frac{\V}{\|\V\|_{\ell_p^n(\X)}}\Big)\Big\|_{\ell_2^n(\ell_2^m)}\bigg].
\end{equation}
Also, using the formula from Lemma~\ref{lem:generalized cone representtaion} for the density of $\|\V\|_{\ell_p^n(\X)}$, for every $q>-nm$ we have
\begin{equation}\label{eq:momebt p-1 By}
 \E \Big[\|\V\|_{\ell_p^n(\X)}^{q}\Big] =\frac{\int_0^\infty s^{nm+q-1} e^{-s^p}\ud s}{\int_0^\infty r^{nm-1} e^{-r^p}\ud r}=\frac{\Gamma\big(\frac{nm+q}{p}\big)}{\Gamma\big(\frac{nm}{p}\big)}.
\end{equation}
Consequently,
\begin{align}\label{first use of independence V}
\begin{split}
\E\bigg[\|\V\|_{\ell_p^n(\X)}^{p-1} \Big\|\nabla\|\cdot\|_{\ell_p^n(\X)}\Big(\frac{\V}{\|\V\|_{\ell_p^n(\X)}}\Big)\Big\|_{\ell_2^n(\ell_2^m)}\bigg]&= \E \Big[\|\V\|_{\ell_p^n(\X)}^{p-1}\Big] \E \bigg[\Big\|\nabla\|\cdot\|_{\ell_p^n(\X)}\Big(\frac{\V}{\|\V\|_{\ell_p^n(\X)}}\Big)\Big\|_{\ell_2^n(\ell_2^m)}\bigg]\\&=\frac{\Gamma
\big(\frac{nm+p-1}{p}\big)}
{nm\Gamma\big(\frac{nm}{p}\big)}\cdot \frac{\vol_{nm-1}\big(\partial B_{\ell_p^n(\X)}\big)}{\vol_{nm}\big(B_{\ell_p^n(\X)}\big)},
\end{split}
\end{align}
where the first step of~\eqref{first use of independence V} uses the independence of $\|\V\|_{\ell_p^n(\X)}$ and $\V/\|\V\|_{\ell_p^n(\X)}$, by Lemma~\ref{lem:generalized cone representtaion}, and the final step of~\eqref{first use of independence V} is a substitution of~\eqref{eq:sphere ratio By} and the case $q=p-1$ of~\eqref{eq:momebt p-1 By}. Hence,
\begin{align}\label{eq:idendidty before R_i}
\begin{split}
\frac{\vol_{nm-1}\big(\partial B_{\ell_p^n(\X)}\big)}{\vol_{nm}\big(B_{\ell_p^n(\X)}\big)}&= \frac{nm\Gamma\big(1+\frac{nm}{p}\big)}{\Gamma\big(1+\frac{nm-1}{p}\big)}\E\bigg[\|\V\|_{\ell_p^n(\X)}^{p-1} \Big\|\nabla\|\cdot\|_{\ell_p^n(\X)}\Big(\frac{\V}{\|\V\|_{\ell_p^n(\X)}}\Big)\Big\|_{\ell_2^n(\ell_2^m)}\bigg]\\&= \frac{p\Gamma\big(1+\frac{nm}{p}\big)}{\Gamma\big(1+\frac{nm-1}{p}\big)}\E\bigg[\Big(\sum_{i=1}^n \|\V_i\|_{\X}^{2p-2}\Big\|\nabla\|\cdot\|_{\X}\Big(\frac{\V_i}{\|\V_i\|_{\X}}\Big)\Big\|_{\ell_2^{m}}^2\Big)^{\frac{1}{2}}\bigg],
\end{split}
\end{align}
where in the last step we used the identity~\eqref{eq:normalized gradient identity}.

The product structure of the density of $\V$ in~\eqref{eq:the density on Y} means that $\V_1,\ldots,\V_n$ are (stochastically) independent. By Lemma~\ref{lem:generalized cone representtaion}, for each $i\in \n$ the random vector $\V_i/\|\V_i\|_{\X}$ is distributed on $\partial B_{\X}$ according to the cone measure $\kappa_{\X}$, and  it  is independent of the random variable
\begin{equation}\label{eq:def Ri}
\RR_i\eqdef \|\V_i\|_{\X}^{2p-2},
\end{equation}
whose density at $t\in (0,\infty)$ is equal (using Lemma~\ref{lem:generalized cone representtaion} once more) to
$$
\frac{\ud}{\ud t} \Pr\Big[ \|\V_i\|_{\X}\le t^{\frac{1}{2p-2}}\Big]=\frac{\ud}{\ud t} \int_0^{t^{\frac{1}{2p-2}}} \frac{s^{m-1}e^{-s^p}}{\int_0^\infty r^{m-1}e^{-r^p}\ud r}\ud s=\frac{p}{2(p-1)\Gamma\big(\frac{m}{p}\big)} t^{\frac{m}{2p-2}-1}e^{-t^{\frac{p}{2p-2}}}.
$$
Hence, the identity~\eqref{eq:idendidty before R_i} which we established above coincides with the desired identity~\eqref{eq:identity surface of ellpnX}.

To prove the identity~\eqref{eq:quadratic gradient identity}, let $\RR$ be a random variable whose density at each $t\in (0,\infty)$ is given by~\eqref{eq:R density}, i.e., $\RR_1,\ldots,\RR_n$ are independent copies of $\RR$. Then, for every $\alpha>-m/(2p-2)$ we have
\begin{equation}\label{eq:alpha moment of R}
\E\big[\RR^\alpha\big]=\frac{p}{2(p-1)\Gamma\big(\frac{m}{p}\big)}\int_0^\infty  t^{\frac{m}{2p-2}+\alpha-1}e^{-t^{\frac{p}{2p-2}}}\ud t =\frac{\Gamma\big(2\alpha +\frac{m-2\alpha}{p}\big)}{\Gamma\big(\frac{m}{p}\big)}.
\end{equation}
Using Lemma~\ref{lem:generalized cone representtaion} (including the independence of $\V_i/\|\V_i\|_{\X}$ and $\|\V_i\|_{\X}$), we have
\begin{align}\label{eq:sum of squares1}
\begin{split}
 \E\bigg[\sum_{i=1}^n \|\V_i\|_{\X}^{2p-2}\Big\|\nabla\|\cdot\|_{\X}\Big(\frac{\V_i}{\|\V_i\|_{\X}}\Big)\Big\|_{\ell_2^{m}}^2\bigg]&=n\E[\RR]\int_{\partial B_{\X}} \big\|\nabla\|\cdot\|_{\X}\big\|_{\ell_2^m}^2\ud\kappa_{\X}\\&
 =\frac{n\Gamma\big(\frac{m+2p-2}{p}\big)}{\Gamma\big(\frac{m}{p}\big)}\int_{\partial B_{\X}} \big\|\nabla\|\cdot\|_{\X}\big\|_{\ell_2^m}^2\ud\kappa_{\X},
\end{split}
\end{align}
where we recall~\eqref{eq:def Ri} and  the last step of~\eqref{eq:sum of squares1} is the case $\alpha=1$ of~\eqref{eq:alpha moment of R}. At the same time,
\begin{align}\label{eq:sum of squares 2}
\begin{split}
 \E\bigg[\sum_{i=1}^n \|\V_i\|_{\X}^{2p-2}\Big\|\nabla\|\cdot\|_{\X}\Big(\frac{\V_i}{\|\V_i\|_{\X}}\Big)\Big\|_{\ell_2^{m}}^2\bigg]&=\E\bigg[\|\V\|_{\ell_p^n(\X)}^{2p-2} \Big\|\nabla\|\cdot\|_{\ell_p^n(\X)}\Big(\frac{\V}{\|\V\|_{\ell_p^n(\X)}}\Big)\Big\|_{\ell_2^n(\ell_2^m)}^2\bigg]\\
 &=\E \Big[\|\V\|_{\ell_p^n(\X)}^{2p-2}\Big] \E \bigg[\Big\|\nabla\|\cdot\|_{\ell_p^n(\X)}\Big(\frac{\V}{\|\V\|_{\ell_p^n(\X)}}\Big)\Big\|_{\ell_2^n(\ell_2^m)}^2\bigg]\\&=
 \frac{\Gamma\big(\frac{nm+2p-2}{p}\big)}{\Gamma\big(\frac{nm}{p}\big)}\int_{\partial B_{\ell_p^n(\X)}} \big\|\nabla \|\cdot\|_{\ell_p^n(\X)}\big\|^2_{\ell_2^n(\ell_2^m)}\ud \kappa_{\ell_p^n(\X)},
 \end{split}
\end{align}
where the first step of~\eqref{eq:sum of squares 2} uses the identity~\eqref{eq:normalized gradient identity}, the second step of~\eqref{eq:sum of squares 2} uses the independence of $\|\V\|_{\ell_p^n(\X)}$ and $\V/\|\V\|_{\ell_p^n(\X)}$ per Lemma~\ref{lem:generalized cone representtaion}, and the final step of uses the case $q=2p-2$ of~\eqref{eq:momebt p-1 By} and Lemma~\ref{lem:generalized cone representtaion}. The desired identity~\eqref{eq:quadratic gradient identity} now follows by  substituting~\eqref{eq:sum of squares 2} into~\eqref{eq:sum of squares1}.
\end{proof}

The following lemma will have a central role in the proof of Theorem~\ref{prop:rounded cube}  and  Theorem~\ref{thm:strong conjecture for most lp}.

\begin{lemma}\label{lem:direct sum of orlicz} Suppose that $n,m\in \N$ and $\beta>0$ satisfy $\beta\le \frac{m-1}{2}$. Then, for every $1\le p\le m$ we have
$$
\iq\big(B_{\ell_p^n(\Omega^m_\beta)}\big)\asymp \sqrt{nm}=\sqrt{\dim\big(\ell_p^n(\Omega^m_\beta)\big)},
$$
where we recall that the normed space $\Omega_\beta^m=(\R^m,\|\cdot\|_{\Omega_\beta^m})$ was defined in~\eqref{eq:our orlicz notation}.
\end{lemma}

Prior to proving Lemma~\ref{lem:direct sum of orlicz}, we will  show how it implies Theorem~\ref{thm:strong conjecture for most lp}, and then deduce Theorem~\ref{prop:rounded cube}.

\begin{proof}[Proof of Theorem~\ref{thm:strong conjecture for most lp} assuming Lemma~\ref{lem:direct sum of orlicz}] By the assumption~\eqref{thm:strong conjecture for most lp} of Theorem~\ref{thm:strong conjecture for most lp},  write $n=km$ for some $k,m\in \N$ with $\max\{2,p\}\le m\le e^p$. Then $(m-1)/2>0$ and $m\ge p$, so we may apply  Lemma~\ref{lem:direct sum of orlicz} with $n$ replaced by $k$ and $\beta=(m-1)/2$. Denoting $\Y=\ell_p^k(\Omega_\beta^m)$, the conclusion of Lemma~\ref{lem:direct sum of orlicz} is that $\iq(B_\Y)\asymp \sqrt{n}$.

$\Y$ is canonically positioned (it is a space from Example~\ref{lem:cheeger uniqueness}). To prove Theorem~\ref{thm:strong conjecture for most lp}, it remains to check that $\|\cdot\|_\Y\asymp \|\cdot\|_{\ell_p^n}$, where, since $n=km$, we identify $\R^n$ with $\M_{k\times n}(\R)$, namely we identify $\ell_p^n$ with $\ell_p^k(\ell_p^m)$.

In fact,  for any $\beta>0$ (not only our choice $\beta=(m-1)/2$ above) we will check that
\begin{equation}\label{eq:Omega beta ell infty}
\forall x\in \R^m,\qquad  \left(1-e^{-\frac{\beta}{m}}\right)\|x\|_{\Omega_{\beta}^m}\le  \|x\|_{\ell_\infty^m}\le \|x\|_{\Omega_{\beta}^m}.
\end{equation}
It follows from~\eqref{eq:Omega beta ell infty} that  $\|\cdot\|_{\Omega_\beta^m}\asymp \|\cdot\|_{\ell_\infty^m}$ when $\beta\asymp m$. But, $\|\cdot\|_{\ell_p^m}\asymp \|\cdot\|_{\ell_\infty^m}$ by the assumption $e^p\ge m$. So,
$$
\beta\asymp n\implies  \|\cdot\|_\Y=\|\cdot\|_{\ell_p^k(\Omega_\beta^m)} \asymp \|\cdot\|_{\ell_p^k(\ell_\infty^m)}\asymp \|\cdot\|_{\ell_p^k(\ell_p^m)}=\|\cdot\|_{\ell_p^n}.
$$

Fix $x\in \R^m$. To verify the second inequality in~\eqref{eq:Omega beta ell infty},  the definition~\eqref{eq:our orlicz notation} gives $\sum_{i=1}^m \psi_{\!\!\beta}(|x_i|/s)=\infty$ when $0<s\le \|x\|_{\ell_\infty^m}$, so $\|x\|_{\Omega_\beta^m}\ge \|x\|_{\ell_\infty^m}$ by~\eqref{eq:define luxemburg}. For the first inequality in~\eqref{eq:Omega beta ell infty}, by direct differentiation it is elementary to  verify that   the function $u\mapsto \log (1/(1-u))/u$ is increasing on the interval $[0,1)$. Thus,
$$
0\le t\le \alpha< 1\implies \psi_{\!\!\beta}(t)=\frac{1}{\beta}\log\left(\frac{1}{1-t}\right)\le \frac{\log\left(\frac{1}{1-\alpha}\right)}{\alpha \beta}t.
$$
Hence, for every fixed $0<\alpha<1$,
\begin{equation}\label{eq:choose our s}
s\ge \frac{1}{\alpha}\|x\|_{\ell_\infty^m}\implies \sum_{i=1}^m \psi_{\!\!\beta}\Big(\frac{|x_i|}{s}\Big)\le \sum_{i=1}^m \frac{\log\left(\frac{1}{1-\alpha}\right)}{\alpha \beta s}|x_i|\le \frac{m\log\left(\frac{1}{1-\alpha}\right)}{\alpha \beta s}\|x\|_{\ell_\infty^m}.
\end{equation}
Provided  $\alpha\ge 1-e^{-\beta/m}$, the choice $s=m\log(1/(1-\alpha))\|x\|_{\ell_\infty^m}/(\alpha\beta)$ satisfies the requirement $s\ge \|x\|_{\ell_\infty^m}/\alpha$, so we get from~\eqref{eq:define luxemburg} and~\eqref{eq:choose our s} that
\begin{equation}\label{eq:to choose alpha Omega beta}
\|x\|_{\Omega_\beta^m}\le\frac{m\log\left(\frac{1}{1-\alpha}\right)}{\alpha \beta }\|x\|_{\ell_\infty^m}.
\end{equation}
The optimal choice of $\alpha$ in~\eqref{eq:to choose alpha Omega beta} is $\alpha= 1-e^{-\beta/m}$, giving the first inequality in~\eqref{eq:Omega beta ell infty}.
\end{proof}

Having proved Theorem~\ref{thm:strong conjecture for most lp} (assuming Lemma~\ref{lem:direct sum of orlicz}, which we will soon prove), we have also already established Theorem~\ref{prop:rounded cube} provided $n\in \N$ and $p\ge 1$ satisfy the divisor condition~\eqref{eq:divisor}. Indeed, the space $\Y$ that Theorem~\ref{thm:strong conjecture for most lp} provides is canonically positioned and hence by the discussion in Section~\ref{sec:positions} it is also  in its minimum surface area position, so by~\cite[Proposition~3.1]{GP99} we have
$$
\frac{\mathrm{MaxProj}(B_\Y)}{\vol_n(B_{\Y})}\asymp \frac{\vol_{n-1}(\partial B_{\Y})}{\vol_n(B_{\Y})\sqrt{n}}=\left(\frac{\iq(B_{\Y})}{\sqrt{n}}\right)\frac{1}{\vol_n(B_{\Y})^{\frac{1}{n}}}
\asymp \frac{1}{\vol_n(B_{\ell_p^n})^{\frac{1}{n}}}\stackrel{\eqref{eq:ell p X}}{\asymp} n^{\frac{1}{p}},
$$
where the penultimate step uses the fact that $\iq(B_\Y)\asymp\sqrt{n}$ by Theorem~\ref{thm:strong conjecture for most lp}, and also that by Theorem~\ref{thm:strong conjecture for most lp} we have  $\|\cdot\|_\Y\asymp \|\cdot\|_{\ell_p^n}$,  which implies that the $n$'th root of the volume of the unit ball of $\Y$ is proportional to the $n$'th root of the volume of the unit ball of $\ell_p^n$.

The deduction of Theorem~\ref{prop:rounded cube} for the remaining   values of $p\ge 1$ and $n\in \N$ uses the following identity, which we will also use in the proof of Proposition~\ref{prop:slightly less rounded cube} below.

\begin{lemma}\label{lem:product quadratic formula}Fix $n,m\in \N$. Suppose that $K\subset \R^n$ and $L\subset \R^m$ are convex bodies. Then,
$$
\frac{\mathrm{MaxProj}(K\times L)}{\vol_{n+m}(K\times L)}= \bigg(\frac{\mathrm{MaxProj}(K)^2}{\vol_{n}(K)^2}+\frac{\mathrm{MaxProj}(L)^2}{\vol_{m}(L)^2}\bigg)^{\frac12}.
$$
\end{lemma}

\begin{proof} Fix $z\in S^{n+m-1}$. By the Cauchy projection formula~\cite{Gar06} that we recalled in~\eqref{eq:use cauchy}, we have
$$
 \vol_{n+m-1}\big(\proj_{z^\perp}(K\times L)\big)=\frac12 \int_{\partial (K\times L)}\big|\langle z,N_{K\times L}(w)\rangle\big|\ud w,
$$
where $N_{K\times L}(w)$ is the (almost-everywhere defined)  unit outer normal to $\partial (K\times L)$ at $w\in \partial (K\times L)$. Now,
$$
\partial (K\times L)= (\partial K\times L)\cup (K\times \partial L)\qquad\mathrm{and}\qquad \vol_{n+m-1}\big((\partial K\times L)\cap (K\times \partial L)\big)=0.
$$
Consequently,
$$
 \vol_{n+m-1}\big(\proj_{z^\perp}(K\times L)\big)=\frac12 \int_{\partial K\times L}\big|\langle z,N_{K\times L}(w)\rangle\big|\ud w+\frac12 \int_{ K\times\partial L}\big|\langle z,N_{K\times L}(w)\rangle\big|\ud w.
$$
If we write each $w\in \R^n$ as $w=(w_1,w_2)$ where $w_1\in \R^{n}$  and $w_2\in \R^m$,  then for almost every (with respect to the $(n+m-1)$-dimensional Hausdorff measure)  $w\in \partial K\times L$ we have $N_{K\times L}(w)=(N_K(w_1),0)$. Also, for almost every $w\in K\times \partial L$ we have $N_{K\times L}(w)=(0,N_L(w_2))$. We therefore have
\begin{align*}
 \vol_{n+m-1}\big(\proj_{z^\perp}(K\times L)\big)&=\frac{\vol_m(L)}{2}\int_{\partial K} \big|\langle z_1,N_{K}(x)\rangle\big|\ud x+\frac{\vol_n(K)}{2}\int_{\partial L} \big|\langle z_2,N_{ L}(y)\rangle\big|\ud y\\
 &= \vol_m(L)\vol_{n-1}\big(\proj_{z_1^\perp} K\big)\|z_1\|_{\ell_2^n}+\vol_n(K)\vol_{m-1}\big(\proj_{z_2^\perp} L\big)\|z_2\|_{\ell_2^m},
\end{align*}
where the last step is two applications of the Cauchy projection formula (in $\R^n$ and $\R^m$). Hence,
\begin{align*}
\frac{ \vol_{n+m-1}\big(\proj_{z^\perp}(K\times L)\big)}{\vol_{n+m}(K\times L)}&=\frac{ \vol_{n+m-1}\big(\proj_{z^\perp}(K\times L)\big)}{\vol_{n}(K)\vol_m(L)}\\&=\frac{\vol_{n-1}\big(\proj_{z_1^\perp} K\big)}{\vol_{n}(K)}\|z_1\|_{\ell_2^n}+\frac{\vol_{m-1}\big(\proj_{z_2^\perp} L\big)}{\vol_{m}(L)}\|z_2\|_{\ell_2^m}.
\end{align*}
Consequently,
\begin{align*}
\frac{\mathrm{MaxProj}(K\times L)}{\vol_{n+m}(K\times L)}&=\max_{z\in S^{n+m-1}}\frac{ \vol_{n+m-1}\big(\proj_{z^\perp}(K\times L)\big)}{\vol_{n+m}(K\times L)}\\&=\max_{(u,v)\in S^1}\max_{x\in S^{n-1}}\max_{y\in S^{m-1}} \frac{ \vol_{n+m-1}\big(\proj_{(ux+vy)^\perp}(K\times L)\big)}{\vol_{n+m}(K\times L)}\\&= \max_{(u,v)\in S^1}\max_{x\in S^{n-1}}\max_{y\in S^{m-1}}\left(\frac{\vol_{n-1}\big(\proj_{x^\perp} K\big)}{\vol_{n}(K)}|u|+\frac{\vol_{m-1}\big(\proj_{y^\perp} L\big)}{\vol_{m}(L)}|v|\right)\\
&= \max_{(u,v)\in S^1} \left(\frac{\mathrm{MaxProj}(K)}{\vol_{n}(K)}|u|+\frac{\mathrm{MaxProj}(L)}{\vol_{m}(L)}|v|\right)\\
&=\bigg(\frac{\mathrm{MaxProj}(K)^2}{\vol_{n}(K)^2}+\frac{\mathrm{MaxProj}(L)^2}{\vol_{m}(L)^2}\bigg)^{\frac12}. \tag*{\qedhere}
\end{align*}
\end{proof}

We can now prove Theorem~\ref{prop:rounded cube}  in its full generality using the fact that we proved Theorem~\ref{thm:strong conjecture for most lp}.

\begin{proof}[Proof of Theorem~\ref{prop:rounded cube}] Let $m$ be any integer that satisfies $\max\{2,p\}\le m\le e^p$ (if $1\le p\le 2$, then take $m=2$, and if $p\ge 2$, then such an $m$ exists because $e^p-p\ge e^2-2>5$). Write $n=km+r$ for some $k\in \N\cup \{0\}$ and $r\in \{0,\ldots,m-1\}$. If $r=0$, then $m$ divides $n$  and we can conclude by applying Theorem~\ref{thm:strong conjecture for most lp}  as we did above (recall the paragraph immediately before Lemma~\ref{lem:product quadratic formula}). So, assume from now that $r\ge 1$.

By Theorem~\ref{thm:strong conjecture for most lp} there is a canonically positioned normed space $\Y=(\R^{km},\|\cdot\|_\Y)$ such that $\iq(B_\Y)\asymp \sqrt{km}$ and  $\|\cdot\|_\Y\asymp \|\cdot\|_{\ell_p^{km}}$. Define $\Y_p^n=\Y\oplus_\infty \Omega_{\beta}^r$, where $\beta\asymp r$ and $\iq(\Omega_\beta^r)\asymp \sqrt{r}$; such $\beta$ exists trivially if $r=1$, and if $r\ge 2$, then its existence follows from an application of Lemma~\ref{lem:direct sum of orlicz} (with the choices $n=1$ and $p=m=r$).

Since $\beta\asymp r$, by~\eqref{eq:Omega beta ell infty} we have $\|\cdot\|_{\Omega_\beta^r} \asymp \|\cdot\|_{\ell_\infty^r}$. Also, $\|\cdot\|_{\ell_\infty^r}\asymp \|\cdot\|_{\ell_p^r}$ since $e^p\ge m>r$. Consequently,
$$
\forall (x,y)\in \R^{km}\times \R^r,\qquad \max\big\{\|x\|_{\Y},\|y\|_{\Omega_\beta^r}\big\}\asymp \max\big\{\|x\|_{\ell_p^{km}},\|y\|_{\ell_p^r}\big\}\asymp \Big(\|x\|_{\ell_p^{km}}^p+\|y\|_{\ell_p^r}^p\Big)^{\frac{1}{p}}.
$$
Recalling the definition of $\Y_p^n$, this means that $\|\cdot\|_{\Y_p^n}\asymp \|\cdot\|_{\ell_p^n}$.

Since both $\Y$ and $\Omega_\beta^r$ are canonically positioned and hence in their minimum surface area positions,
$$
\frac{\mathrm{MaxProj}(B_\Y)}{\vol_{km}(B_\Y)}\asymp \left(\frac{\iq(B_\Y)}{\sqrt{km}}\right)\frac{1}{\vol_{km}(B_\Y)^{\frac{1}{km}}}\asymp \frac{1}{\vol_{km}\big(B_{\ell_p^{km}}\big)^{\frac{1}{km}}}\asymp (km)^{\frac{1}{p}},
$$
and
$$
\frac{\mathrm{MaxProj}\big(B_{\Omega_\beta^r}\big)}{\vol_{r}\big(B_{\Omega_\beta^r}\big)}\asymp \left(\frac{\iq\big(\Omega_\beta^r\big)}{\sqrt{r}}\right)\frac{1}{\vol\big(\Omega_\beta^r\big)^{\frac{1}{r}}}\asymp \frac{1}{\vol\big(\ell_\infty^r\big)^{\frac{1}{r}}}\asymp 1\asymp r^{\frac{1}{p}}.
$$
Consequently, since $B_{\Y_p^n}=B_\Y\times B_{\Omega_\beta^r}$, by Lemma~\ref{lem:product quadratic formula} we conclude that
\begin{equation*}
\frac{\mathrm{MaxProj}\big(B_{\Y_p^n}\big)}{\vol_{n}\big(B_{\Y_p^n}\big)}=\left(\frac{\mathrm{MaxProj}(B_\Y)^2}{\vol_{km}(B_\Y)^2}+\frac{\mathrm{MaxProj}
\big(B_{\Omega_\beta^r}\big)^2}{\vol_{r}\big(B_{\Omega_\beta^r}\big)^2}\right)^{\frac12}\asymp \Big((km)^{\frac{2}{p}}+r^{\frac{2}{p}}\Big)^{\frac12}\asymp (km+r)^{\frac{1}{p}}=n^{\frac{1}{p}}. \tag*{\qedhere}
\end{equation*}
\end{proof}

The following lemma will be used in the proof of Lemma~\ref{lem:direct sum of orlicz}.

\begin{lemma}\label{lem:exponential integrals on l1} Suppose that $m\in \N$, $r\in \N\cup \{0\}$  and $\beta>0$ satisfy  $\beta\le \frac{m+r-2}{2}$. Then
\begin{equation}\label{eq:exponential integrals on l1}
\int_{\partial B_{\ell_1^m}} \bigg(e^{\beta|\tau_1|}-\sum_{k={r-1}}^\infty \frac{\beta^k|\tau_1|^k}{k!}\bigg)\ud\kappa_{\ell_1^m}(\tau)=\int_{\partial B_{\ell_1^m}} \bigg(\sum_{k=r}^\infty \frac{\beta^k|\tau_1|^k}{k!}\bigg)\ud\kappa_{\ell_1^m}(\tau)\asymp \frac{\beta^r(m-1)!}{(m+r-1)!}.
\end{equation}
\end{lemma}

\begin{proof} Let $\H_1,\ldots,\H_m$ be independent random variables whose density at each $s\in \R$ is equal to $e^{-|s|}/2$. Then, $|\H_1|,\ldots,|H_m|$ are exponential random variables of rate $1$, and therefore if we denote $$\Gamma\eqdef \sum_{i=1}^m |\H_i|,$$ then $\Gamma$ has $\Gamma(m,1)$ distribution, i.e., its density at each $s\ge 0$ is equal to $s^{m-1}e^{-s}/(m-1)!$; the proof of this standard probabilistic fact can be found in e.g.~\cite{Dur19}. By~\cite{SZ90,RR91} (or Lemma~\ref{lem:generalized cone representtaion}), the random vector $(\H_1,\ldots,\H_m)/\Gamma$ is distributed according to $\kappa_{\ell_1^m}$ and is independent of $\Gamma$. Thus, for every $k\in \N$,
$$
\int_{\partial B_{\ell_1^m}} |\tau_1|^k \ud\kappa_{\ell_1^m}(\tau) =\E \bigg[\frac{|\H_1|^k}{\Gamma^k}\bigg]=\frac{\E\big[|\H_1|^k\big]}{\E\big[\Gamma^k\big]}=\frac{\int_0^\infty s^ke^{-s}\ud s}{\frac{1}{(m-1)!}\int_0^\infty s^{k+m-1}e^{-s}\ud s}=\frac{k!(m-1)!}{(k+m-1)!}.
$$
Consequently,
\begin{equation}\label{eq:integral form remainder}
\int_{\partial B_{\ell_1^m}} \bigg(\sum_{k=r}^\infty \frac{\beta^k|\tau_1|^k}{k!}\bigg)\ud\kappa_{\ell_1^m}(\tau)=\frac{(m-1)!}{\beta^{m-1}}\sum_{k=r}^\infty\frac{\beta^{k+m-1}}{(k+m-1)!}=
\frac{\beta^r(m-1)!}{(m+r-2)!}\int_0^1 e^{\beta t}(1-t)^{m+r-2}\ud t,
\end{equation}
where the last step is the integral form of the remainder of the Taylor series of the exponential function.

It is mechanical to check that~\eqref{eq:exponential integrals on l1} holds for $m\in \{1,2\}$, so assume for the rest of the proof of Lemma~\ref{lem:exponential integrals on l1} that $m\ge 3$.  We then see from~\eqref{eq:integral form remainder} that our goal~\eqref{eq:exponential integrals on l1} is equivalent to showing that
\begin{equation}\label{eq:integral taylor asymptotic}
\int_0^1 e^{\beta t}(1-t)^{m+r-2}\ud t\asymp \frac{1}{m+r}.
\end{equation}
For the upper bound in~\eqref{eq:integral taylor asymptotic}, estimate the integrand using $(1-t)^{m+r-2}\le e^{-(m+r-2)t}$ to get
$$
\int_0^1 e^{\beta t}(1-t)^{m+r-2}\ud t\le \int^1_0 e^{-(m+r-2-\beta)t}\ud t= \frac{1-e^{-(m+r-2-\beta)}}{m+r-2-\beta}\asymp \frac{1}{m+r},
$$
where we used $\beta<\frac{m+r-2}{2}$. For the lower bound in~\eqref{eq:integral taylor asymptotic}, since $(1-t)^{m+r-2}\gtrsim 1$ when $0\le t\le \frac{1}{m+r-2}$,
$$
\int_0^1 e^{\beta t}(1-t)^{m+r-2}\ud t\ge \int_0^{\frac{1}{m+r-2}} e^{\beta t}(1-t)^{m+r-2}\ud t\gtrsim  \int_0^{\frac{1}{m+r-2}} e^{\beta t}\ud t= \frac{e^{\frac{\beta}{m+r-2}}-1}{\beta}\asymp \frac{1}{m+r},
$$
where in the last step we used the assumption $\beta<\frac{m+r-2}{2}$ once more.
\end{proof}

\begin{proof}[Proof of Lemma~\ref{lem:direct sum of orlicz}] By combining the case $g\equiv 1$ of~\eqref{eq:orlicz change of variable-with const}  with~\eqref{eq:orlicz density beta}, we see that
\begin{equation}\label{eq:volume of orlicz}
\vol_m\big(B_{\Omega_\beta^m}\big)=\frac{\beta^{m-1}2^m}{e^\beta m!}m\int_{\partial B_{\ell_1^m}}  \big(e^{\beta|\tau_1|}-1\big) \ud \kappa_{\ell_1^m}(\tau)\stackrel{\eqref{eq:exponential integrals on l1}}{\asymp} \frac{(2\beta)^m}{e^\beta m!}.
\end{equation}
Since we are assuming in Lemma~\ref{lem:direct sum of orlicz} that $\beta\lesssim m$, in combination with~\eqref{eq:ell p X} we get from~\eqref{eq:volume of orlicz} that
\begin{equation}\label{eq:volume radius of orlicz}
\vol_{nm}\big(B_{\ell_p^n(\Omega_\beta^m)}\big)^{\frac{1}{nm}}\asymp \frac{\beta}{n^{\frac{1}{p}}m},
\end{equation}
At the same time, by applying  Cauchy–Schwarz to the identity~\eqref{eq:identity surface of ellpnX} of Lemma~\ref{lem:break surface of ellp(X)} we have
\begin{align}\label{eq:area by volume lp orlicz}
\begin{split}
\frac{\vol_{nm-1}\big(\partial B_{\ell_p^n(\Omega_\beta^m)}\big)}{\vol_{nm}\big(B_{\ell_p^n(\Omega_\beta^m)}\big)}&\le \frac{p\Gamma\big(1+\frac{nm}{p}\big)}{\Gamma\big(1+\frac{nm-1}{p}\big)}\bigg(n\Big(\E \big[\mathsf{R}_1\big]\Big) \int_{\partial B_{\Omega_\beta^m}}\big\|\nabla \|\cdot\|_{\Omega_\beta^m} (\theta)\big\|_{\ell_2^m}^2\ud\kappa_{\Omega_\beta^m}(\theta)\bigg)^{\frac12}\\&\asymp n^{\frac{1}{p}+\frac12}m\bigg(\int_{\partial B_{\Omega_\beta^m}}\big\|\nabla \|\cdot\|_{\Omega_\beta^m} (\theta)\big\|_{\ell_2^m}^2\ud\kappa_{\Omega_\beta^m}(\theta)\bigg)^{\frac12},
\end{split}
\end{align}
where the random variable $\mathsf{R}_1$ is as in Lemma~\ref{lem:break surface of ellp(X)}, i.e., its density is in~\eqref{eq:R density}, and the last step is an application the evaluation~\eqref{eq:alpha moment of R} of its  moments and Stirling's formula, using the assumption $1\le p\le m$.

Recalling~\eqref{eq:our orlicz notation}, even though  $\|\cdot\|_{\Omega_\beta^m}$ is defined implicitly by~\eqref{eq:define luxemburg}, we can compute $\nabla \|\cdot\|_{\Omega_\beta^m} (\theta)$ for almost every $\theta\in \partial B_{\Omega_\beta^m}$ as the unique vector $v\in \R^m$ that is normal to $\partial B_{\Omega_\beta^m}$ and satisfies $\langle v,\theta\rangle=1$. Indeed, since $\partial \Omega_\beta^m$ is parameterized as the zero set of the function $\Psi_{\!\!\beta}:\R^n\to \R^n$ that is given by
$$\forall x\in \R^n,\qquad \Psi_{\!\!\beta}(x)\eqdef 1-\sum_{i=1}^m \psi_{\!\!\beta}(|x_i|),
$$ the following vector is normal to $\partial B_{\Omega_\beta^m}$  for almost every $\theta\in \partial B_{\Omega_\beta^m}$.
$$v_\beta(\theta)\eqdef \nabla\Psi_{\!\!\beta}(\theta)=-\big(\psi_{\!\!\beta}'(|\theta_1|)\sign(\theta_1),\ldots,\psi_{\!\!\beta}'(|\theta_m|)\sign(\theta_m)\big).$$
  So, $\nabla \|\cdot\|_{\Omega_\beta^m} (\theta)=\lambda_\beta(\theta)v_\beta(\theta)$ for almost every $\theta\in \partial B_{\Omega_\beta^m}$, where $\lambda_\beta(\theta) \in \R$ is such that $\langle\lambda_\beta(\theta)v_\beta(\theta),\theta\rangle=1$, i.e., $\lambda_\beta(\theta)=-1/\langle v_\beta(\theta),\theta\rangle$. This shows that for almost every $\theta\in \partial B_{\Omega_\beta^m}$,
\begin{align}\label{eq:evaluate gradient}
\begin{split}
\nabla \|\cdot\|_{\Omega_\beta^m} (\theta)&=\frac{1}{\sum_{i=1}^m |\theta_i|\psi_{\!\!\beta}'(|\theta_i|)}\big(\psi_{\!\!\beta}'(|\theta_1|)\sign(\theta_1),\ldots,\psi_{\!\!\beta}'(|\theta_m|)\sign(\theta_m)\big)\\
&=\frac{1}{\sum_{i=1}^m \frac{|\theta_i|}{1-|\theta_i|}}\Big(\frac{\sign(\theta_1)}{1-|\theta_1|},\ldots,\frac{\sign(\theta_m)}{1-|\theta_m|}\Big),
\end{split}
\end{align}
where the first equality in~\eqref{eq:evaluate gradient} holds for any $\psi_{\!\!\beta}$ that satisfies the conditions of Lemma~\ref{lem:orlicz}, and   for the second equality in~\eqref{eq:evaluate gradient} recall the definition~\eqref{eq:our orlicz notation} of the specific  $\psi_{\!\!\beta}$ that we are using here. Therefore,
\begin{align}\label{eq:l2 norm of gradient orlicz}
\begin{split}
\int_{\partial B_{\Omega_\beta^m}}\big\|\nabla \|\cdot\|_{\Omega_\beta^m} (\theta)\big\|_{\ell_2^m}^2\ud\kappa_{\Omega_\beta^m}(\theta)&=\frac{\int_{\partial B_{\ell_1^m}} \frac{\sum_{i=1}^m e^{2\beta|\tau_i|}}{\sum_{i=1}^m (e^{\beta|\tau_i|}-1)}\ud \kappa_{\ell_1^m}(\tau)}{m\int_{\partial B_{\ell_1^m}}  \big(e^{\beta|\tau_1|}-1\big) \ud \kappa_{\ell_1^m}(\tau)}\\&\le \frac{\int_{\partial B_{\ell_1^m}} \frac{\sum_{i=1}^m e^{2\beta|\tau_i|}}{\beta \sum_{i=1}^m |\tau_i| }\ud \kappa_{\ell_1^m}(\tau)}{m\int_{\partial B_{\ell_1^m}}  \big(e^{\beta|\tau_1|}-1\big) \ud \kappa_{\ell_1^m}(\tau)}=\frac{\int_{\partial B_{\ell_1^m}} e^{2\beta|\tau_1|}\ud \kappa_{\ell_1^m}(\tau)}{\beta\int_{\partial B_{\ell_1^m}}  \big(e^{\beta|\tau_1|}-1\big) \ud \kappa_{\ell_1^m}(\tau)} \asymp \frac{m}{\beta^2},
\end{split}
\end{align}
where the first step of~\eqref{eq:l2 norm of gradient orlicz}  is a substitution of~\eqref{eq:evaluate gradient} into~\eqref{eq:orlicz change of variable} while using~\eqref{eq:orlicz density beta} and that $\psi_{\!\!\beta}^{-1}(t)=1-e^{-\beta t}$ for every $t\ge 0$, the second step of~\eqref{eq:l2 norm of gradient orlicz} uses the  inequality $e^t\ge t+1$ which holds for any $t\in \R$, and the final step of~\eqref{eq:l2 norm of gradient orlicz} is an application of Lemma~\ref{lem:exponential integrals on l1}. Now, a combination of~\eqref{eq:area by volume lp orlicz}  and~\eqref{eq:l2 norm of gradient orlicz} gives
\begin{equation}\label{eq:sharp bound for area by volume orlicz}
\frac{\vol_{nm-1}\big(\partial B_{\ell_p^n(\Omega_\beta^m)}\big)}{\vol_{nm}\big(B_{\ell_p^n(\Omega_\beta^m)}\big)}\lesssim \frac{n^{\frac{1}{p}+\frac12}m^{\frac32}}{\beta}.
\end{equation}
By combining~\eqref{eq:volume radius of orlicz} and~\eqref{eq:sharp bound for area by volume orlicz} we conclude that
$$
\iq\big(B_{\ell_p^n(\Omega^m_\beta)}\big)= \frac{\vol_{nm-1}\big(\partial B_{\ell_p^n(\Omega_\beta^m)}\big)}{\vol_{nm}\big(B_{\ell_p^n(\Omega_\beta^m)}\big)}\vol_{nm}\big(B_{\ell_p^n(\Omega_\beta^n)}\big)^{\frac{1}{nm}}\lesssim \sqrt{nm}.
$$
The reverse inequality $\iq(B_{\ell_p^n(\Omega^m_\beta)})\gtrsim \sqrt{nm}$ follows from the isoperimetric theorem~\eqref{eq:quote ispoperimetric theorem}, so the proof of Lemma~\ref{lem:direct sum of orlicz} is complete. Note that this also shows that all of the inequalities that we derived in the above proof of Lemma~\ref{lem:direct sum of orlicz} are in fact asymptotic equivalences. This holds in particular for~\eqref{eq:sharp bound for area by volume orlicz}, i.e.,
\begin{equation*}
\frac{\vol_{nm-1}\big(\partial B_{\ell_p^n(\Omega_\beta^m)}\big)}{\vol_{nm}\big(B_{\ell_p^n(\Omega_\beta^m)}\big)}\asymp \frac{n^{\frac{1}{p}+\frac12}m^{\frac32}}{\beta}.\tag*{\qedhere}
\end{equation*}
\end{proof}

The following asymptotic evaluation of the surface area of the sphere of $\ell_p^n(\ell_q^m)$ in the entire range of possible values of $p,q\ge 1$ and $m,n\in \N$ is an application of Lemma~\ref{lem:break surface of ellp(X)};  by~\eqref{eq:volume of ellp ellq} it is equivalent to~\eqref{eq:iq of lpn lqm}.

\begin{theorem}\label{cor:surface area on lplq} For every $n,m\in \N$ and $p,q\in [1,\infty]$ we have
\begin{comment}
$$
\frac{\vol_{nm-1}\big(\partial B_{\ell_p^n(\ell_q^m)}\big)}{\vol_{nm}\big(B_{\ell_p^n(\ell_q^m)}\big)}\asymp \left\{\begin{array}{ll}n^{1+\frac{1}{p}}m^{1+\frac{1}{q}}& m\le \min\left\{\frac{p}{n},q\right\},\\ \sqrt{q} n^{1+\frac{1}{p}} m^{\frac12+\frac{1}{q}} & q\le m\le \frac{p}{n},\\ \sqrt{p}n^{\frac12+\frac{1}{p}}m^{\frac12+\frac{1}{q}}&\frac{p}{n}\le m\le \min\{p,q\},\\\sqrt{pq} n^{\frac12+\frac{1}{p}}m^{\frac{1}{q}}
& \max\left\{\frac{p}{n},q\right\}\le m\le p,\\
n^{\frac12+\frac{1}{p}} m^{1+\frac{1}{q}}& p\le m\le q,\\
\sqrt{q} n^{\frac{1}{2}+\frac{1}{p}} m^{\frac{1}{2}+\frac{1}{q}} & m\ge\max\{p,q\}.
\end{array} \right.
$$
Hence,
$$
\iq\big(\ell_p^n(\ell_q^m)\big)\asymp
\left\{\begin{array}{ll}nm& m\le \min\left\{\frac{p}{n},q\right\},\\ n\sqrt{qm}  & q\le m\le \frac{p}{n},\\ \sqrt{pnm}&\frac{p}{n}\le m\le \min\{p,q\},\\\sqrt{pqn}
& \max\left\{\frac{p}{n},q\right\}\le m\le p,\\
m\sqrt{n}& p\le m\le q,\\
\sqrt{qnm}  & m\ge\max\{p,q\}.\end{array} \right.
$$
In particular,
$$
\iq\big(\ell_p^n(\ell_q^n)\big)\asymp
\left\{\begin{array}{ll}n^2& n\le \min\{\sqrt{p},q\},\\ \sqrt{q}n^{\frac32}  & q\le n\le \sqrt{p},\\ \sqrt{p}n&\sqrt{p}\le n\le \min\{p,q\},\\\sqrt{pqn}
& \max\{\sqrt{p},q\}\le n\le p,\\
n^{\frac32}& p\le n\le q,\\
\sqrt{q}n  & n\ge\max\{p,q\}.\end{array} \right.
$$
\end{comment}
\begin{equation}
\vol_{nm-1}\big(\partial B_{\ell_p^n(\ell_q^m)}\big)\asymp\frac{2^{nm}\Gamma\big(1+\frac{1}{q}\big)^{nm}
\Gamma\big(1+\frac{m}{p}\big)^n}{\Gamma\big(1+\frac{m}{q}\big)^n\Gamma\big(1+\frac{nm}{p}\big)}\cdot\left\{\begin{array}{ll}n^{1+\frac{1}{p}}m^{1+\frac{1}{q}}& m\le \min\left\{\frac{p}{n},q\right\},\\ \sqrt{q} n^{1+\frac{1}{p}} m^{\frac12+\frac{1}{q}} & q\le m\le \frac{p}{n},\\ \sqrt{p}n^{\frac12+\frac{1}{p}}m^{\frac12+\frac{1}{q}}&\frac{p}{n}\le m\le \min\{p,q\},\\\sqrt{pq} n^{\frac12+\frac{1}{p}}m^{\frac{1}{q}}
& \max\left\{\frac{p}{n},q\right\}\le m\le p,\\
n^{\frac12+\frac{1}{p}} m^{1+\frac{1}{q}}& p\le m\le q,\\
\sqrt{q} n^{\frac{1}{2}+\frac{1}{p}} m^{\frac{1}{2}+\frac{1}{q}} & m\ge\max\{p,q\}.
\end{array} \right.
\end{equation}

\end{theorem}

\begin{proof} By continuity we may assume that $p,q\in (1,\infty)$.
Suppose that $\sfG$ is a symmetric real-valued random variable whose density at each $s\in \R$ is equal to
\begin{equation}\label{eq:sz density}
\frac{1}{2\Gamma\big(1+\frac{1}{q}\big)}e^{-|s|^q}.
\end{equation}
Let $\sfG_1,\ldots,\sfG_m$ be independent copies of $\sfG$. Set $\sfU\eqdef(\sfG_1,\ldots,\sfG_m)\in \R^m$. By the probabilistic representation of the cone measure on $\partial B_{\ell_q^m}$ in~\cite{SZ90,RR91} (or Lemma~\ref{lem:generalized cone representtaion}), the random vector $\sfU/\|\sfU\|_{\ell_q^m}$ is distributed according to the cone measure on $\partial B_{\ell_q^m}$, and moreover it is independent of $\|\sfU\|_{\ell_q^m}$.

Consider the following  random variable.
\begin{equation}\label{eq:def sfN}
\sfN\eqdef \Big\|\nabla\|\cdot\|_{\ell_q^m}\Big(\frac{\sfU}{\|\sfU\|_{\ell_q^m}}\Big)\Big\|_{\ell_2^{m}}^2=\frac{1}{\|\sfU\|_{\ell_q^m}^{2q-2}}\sum_{j=1}^m |\sfG_j|^{2q-2}=\frac{\|\sfU\|_{\ell_{2q-2}^m}^{2q-2}}{\|\sfU\|_{\ell_q^m}^{2q-2}}.
\end{equation}
If we let $\sfN_1,\ldots,\sfN_n, \RR_1,\ldots,\RR_n$ be independent random variables such that $\sfN_1,\ldots,\sfN_n$ have the same distribution as $\sfN$, and $\RR_1,\ldots,\RR_n$ are as in Lemma~\ref{lem:break surface of ellp(X)}, then by Lemma~\ref{lem:break surface of ellp(X)} we have
\begin{equation}\label{eq:apply lemma introduce Z}
\frac{\vol_{nm-1}\big(\partial B_{\ell_p^n(\ell_q^m)}\big)}{\vol_{nm}\big(B_{\ell_p^n(\ell_q^m)}\big)}= \frac{p\Gamma\big(1+\frac{nm}{p}\big)}{\Gamma\big(1+\frac{nm-1}{p}\big)}\E[\sfZ]\asymp pn^{\frac{1}{p}}m^{\frac{1}{p}} \E[\sfZ],
\end{equation}
where for~\eqref{eq:apply lemma introduce Z} we introduce the following notation.
\begin{equation}\label{eq:def sfZ}
\sfZ\eqdef \bigg(\sum_{i=1}^n \RR_i\mathsf{N}_i\bigg)^{\frac12}.
\end{equation}

Let $\RR$ be a random variable that takes values in $[0,\infty)$ whose density at each $t\in (0,\infty)$ is given by~\eqref{eq:R density}, i.e., $\RR_1,\ldots,\RR_n$ are independent copies of $\RR$. We computed the moments of $\RR$ in~\eqref{eq:alpha moment of R} and by Stirling's formula this gives the following asymptotic evaluations.
\begin{align}
\E\big[\RR^\frac12\big]&\asymp \frac{m^{1-\frac{1}{p}}}{p},\label{eq:moment of R1/2}\\\E[\RR]&\asymp \max\Big\{\frac{m}{p},1\Big\}\frac{m^{1-\frac{2}{p}}}{p}, \label{eq:moment of R1}\\ \E\big[\RR^2\big] &\asymp \max\Big\{\frac{m^3}{p^3},1\Big\}\frac{m^{1-\frac{4}{p}}}{p}\label{eq:moment of R2}.
\end{align}

We  also need an analogous asymptotic evaluation of moments of the random variable $\sfN$ in~\eqref{eq:def sfN}.  Observe that the random variables $\sfN$ and $\|\sfU\|_{\ell_q^m}$ are independent, since $\sfU/\|\sfU\|_{\ell_q^m}$ and $\|\sfU\|_{\ell_q^m}$ are independent and $\sfN$ is a function $\sfU/\|\sfU\|_{\ell_q^m}$. Consequently, for every  $\beta>0$ we have
\begin{equation}\label{eq:use independencxe N}
\E\Big[\|\sfU\|_{\ell_q^m}^{(2q-2)\beta}\Big]\E \big[\sfN^\beta\big]=\E\Big[\|\sfU\|_{\ell_q^m}^{(2q-2)\beta}\sfN^\beta\Big]\stackrel{\eqref{eq:def sfN}}{=}\E\Big[\|\sfU\|_{\ell_{2q-2}^m}^{(2q-2)\beta}\Big].
\end{equation}
Since (e.g.~by Lemma~\ref{lem:generalized cone representtaion}) the density of $\|\sfU\|_{\ell_q^m}$ at $s\in (0,\infty)$ is proportional to $s^{m-1}e^{-s^q}$, we can compute analogously to~\eqref{eq:momebt p-1 By} that
$$
\E\Big[\|\sfU\|_{\ell_q^m}^{(2q-2)\beta}\Big]=\frac{\int_0^\infty s^{m-1+(2q-2)\beta}e^{-s^q}\ud s}{\int_0^\infty r^{m-1}e^{-r^q}\ud r}=\frac{\Gamma\big(2\beta+\frac{m-2\beta}{q}\big)}{\Gamma\big(\frac{m}{q}\big)}.
$$
Therefore~\eqref{eq:use independencxe N} implies that
\begin{equation*}
\E \big[\sfN^\beta\big]=\frac{\Gamma\big(\frac{m}{q}\big)}{\Gamma\big(2\beta+\frac{m-2\beta}{q}\big)}\E\Big[\|\sfU\|_{\ell_{2q-2}^m}^{(2q-2)\beta}\Big].
\end{equation*}
By considering each of the values $\beta\in \{\frac12,1,2\}$ in this identity and using Stirling's formula, we get the following asymptotic evaluations of moments of $\sfN$ in terms of moments of $\|\sfU\|_{\ell_{2q-2}^m}$.
\begin{align}
\E \big[\sfN^\frac12\big]&\asymp\frac{q}{m^{1-\frac{1}{q}}}\E\Big[\|\sfU\|_{\ell_{2q-2}^m}^{q-1}\Big],\label{eq:N1/2}\\
 \E[\sfN]&\asymp \min\left\{\frac{q}{m},1\right\}\frac{q}{m^{1-\frac{2}{q}}}\E\Big[\|\sfU\|_{\ell_{2q-2}^m}^{2q-2}\Big],
\label{eq:N1}\\
\E\big[\sfN^2\big]&\asymp \min\left\{\frac{q^3}{m^3},1\right\}\frac{q}{m^{1-\frac{4}{q}}}\E\Big[\|\sfU\|_{\ell_{2q-2}^m}^{4q-4}\Big].\label{eq:N2}
\end{align}

Due to~\eqref{eq:N1/2}, \eqref{eq:N1}, \eqref{eq:N2}, we will next evaluate the corresponding moments of $\|\sfU\|_{\ell_{2q-2}^m}$. Recalling the density~\eqref{eq:sz density} of $\sfG$, for every $\beta>-1/(2q-2)$ we have
$$
\E\big[|\sfG|^{(2q-2)\beta}\big]=\frac{1}{\Gamma\big(1+\frac{1}{q}\big)}\int_0^\infty s^{(2q-2)\beta}e^{-s^q}\ud s=\frac{\Gamma\big(\frac{2q-2}{q}\beta+\frac{1}{q}\big)}{q\Gamma\big(1+\frac{1}{q}\big)}.
$$
Hence,
\begin{equation}\label{eq:all moments are 1/q}
\E\big[|\sfG|^{q-1}\big]\asymp \E\big[|\sfG|^{2q-2}\big]\asymp \E\big[|\sfG|^{4q-4}\big]\asymp \frac{1}{q}.
\end{equation}
We therefore have
\begin{equation}\label{eq:second moment U}
\E \Big[\|\sfU\|_{\ell_{2q-2}^m}^{2q-2}\Big]=m\E\big[|\sfG|^{2q-2}\big]\stackrel{\eqref{eq:all moments are 1/q}}{\asymp} \frac{m}{q},
\end{equation}
and
\begin{equation}\label{eq:fourth moment U}
\E \Big[\|\sfU\|_{\ell_{2q-2}^m}^{4q-4}\Big]=\E \bigg[\Big(\sum_{j=1}^m |\sfG_j|^{2q-2} \Big)^2\bigg]= m\E\big[|\sfG|^{4q-4}\big]+m(m-1) \big(\E\big[|\sfG|^{2q-2}\big]\big)^2\stackrel{\eqref{eq:all moments are 1/q}}{\asymp} \max \Big\{\frac{m}{q},1\Big\}\frac{m}{q}.
\end{equation}
Consequently, using H\"older's inequality we get the following estimate.
\begin{align}\label{eq:use holder 3/2}
\begin{split}
\frac{m}{q}\stackrel{\eqref{eq:second moment U}}{\asymp} \E \Big[&\|\sfU\|_{\ell_{2q-2}^m}^{2q-2}\Big]=\E \Big[\|\sfU\|_{\ell_{2q-2}^m}^{\frac23(q-1)}\|\sfU\|_{\ell_{2q-2}^m}^{\frac13(4q-4)}\Big]\\&\le \Big(\E \Big[\|\sfU\|_{\ell_{2q-2}^m}^{q-1}\Big]\Big)^{\frac23}\Big(\E \Big[\|\sfU\|_{\ell_{2q-2}^m}^{4q-4}\Big]\Big)^{\frac13}\stackrel{\eqref{eq:fourth moment U}}{\asymp} \Big(\E \Big[\|\sfU\|_{\ell_{2q-2}^m}^{q-1}\Big]\Big)^{\frac23}\Big(\max \Big\{\frac{m}{q},1\Big\}\frac{m}{q}\Big)^{\frac13}.
\end{split}
\end{align}
This simplifies to give
\begin{equation}\label{eq:lower 1/2 U}
\E \Big[\|\sfU\|_{\ell_{2q-2}^m}^{q-1}\Big]\gtrsim \min\left\{\sqrt{\frac{m}{q}},\frac{m}{q}\right\}.
\end{equation}
At the same time, by Cauchy--Schwarz,
\begin{equation}\label{eq:CS U}
\E \Big[\|\sfU\|_{\ell_{2q-2}^m}^{q-1}\Big]\le \Big(\E \Big[\|\sfU\|_{\ell_{2q-2}^m}^{2q-2}\Big]\Big)^{\frac12}\stackrel{\eqref{eq:second moment U}}{\asymp} \sqrt{\frac{m}{q}}.
\end{equation}
Also, by the subadditivity of the square root on $[0,\infty)$,
\begin{equation}\label{eq:subadditivity U}
\E \Big[\|\sfU\|_{\ell_{2q-2}^m}^{q-1}\Big]=\E \bigg[\Big(\sum_{j=1}^m |\sfG_j|^{2q-2} \Big)^\frac12\bigg]\le \E \bigg[\sum_{j=1}^m |\sfG_j|^{q-1}\bigg]= m\E\big[|\sfG|^{q-1}\big]\stackrel{\eqref{eq:all moments are 1/q}}{\asymp}\frac{m}{q}.
\end{equation}
By combining~\eqref{eq:CS U} and~\eqref{eq:subadditivity U} we see that~\eqref{eq:lower 1/2 U} is in fact sharp, i.e.,
\begin{equation}\label{eq:1/2 U}
\E \Big[\|\sfU\|_{\ell_{2q-2}^m}^{q-1}\Big]\asymp \min\left\{\sqrt{\frac{m}{q}},\frac{m}{q}\right\}.
\end{equation}

By substituting~\eqref{eq:1/2 U} into~\eqref{eq:N1/2}, and correspondingly~\eqref{eq:second moment U} into~\eqref{eq:N1} and~\eqref{eq:fourth moment U} into~\eqref{eq:N2}, we get the following asymptotic identities.
\begin{align}
\E \big[\sfN^\frac12\big]&\asymp \min\left\{\sqrt{\frac{q}{m}},1\right\}m^{\frac{1}{q}},\label{eq:moment of N1/2}\\
 \E[\sfN]&\asymp \min\left\{\frac{q}{m},1\right\}m^{\frac{2}{q}},\label{eq:moment of N 1}\\
 \E\big[\sfN^2\big]&\asymp \min\Big\{\frac{q^2}{m^2},1\Big\}m^{\frac{4}{q}}\label{eq:moment of N2}.
\end{align}

By combining~\eqref{eq:moment of R1} and~\eqref{eq:moment of N 1} we see that
$$
\E \big[\sfZ^2\big]=n\big(\E[\RR]\big)\big( \E[\sfN]\big)\asymp \frac{\max\{m,p\}\min\{q,m\}}{p^2}nm^{\frac{2}{q}-\frac{2}{p}}.
$$
Using Cauchy--Schwarz, this implies the following upper bound on the final term in~\eqref{eq:apply lemma introduce Z}.
\begin{equation}\label{eq:second moment of Z with normalization}
pn^{\frac{1}{p}}m^{\frac{1}{p}}\E[\sfZ]\le pn^{\frac{1}{p}}m^{\frac{1}{p}}\big(\E\big[\sfZ^2\big]\big)^\frac12\asymp n^{\frac12+\frac{1}{p}}m^{\frac{1}{q}}\sqrt{\max\{m,p\}\min\{m,q\}}.
\end{equation}
Also, recalling~\eqref{eq:def sfZ} and using the subadditivity of the square root on $[0,\infty)$ in combination with~\eqref{eq:moment of R1/2} and~\eqref{eq:moment of N1/2}, we have the following additional upper bound on the final term in~\eqref{eq:apply lemma introduce Z}.
\begin{align}\label{eq:normalized Z expectation subadditivity}
\begin{split}
pn^{\frac{1}{p}}m^{\frac{1}{p}}\E[\sfZ]\le pn^{\frac{1}{p}}m^{\frac{1}{p}}\E \bigg[\sum_{i=1}^n \RR_i^{\frac12}\sfN_i^{\frac12}\bigg]= pn^{1+\frac{1}{p}}m^{\frac{1}{p}}\big(\E\big[\RR^\frac12\big] \big)\big(\E\big[\sfN^\frac12\big]\big)\asymp n^{1+\frac{1}{p}}m^{\frac12+\frac{1}{q}} \sqrt{\min\{m,q\}}.
\end{split}
\end{align}
It follows from~\eqref{eq:second moment of Z with normalization} and~\eqref{eq:normalized Z expectation subadditivity} that
\begin{align}\label{eq:six case upper}
\begin{split}
pn^{\frac{1}{p}}m^{\frac{1}{p}}\E[\sfZ]&\lesssim n^{\frac12+\frac{1}{p}} m^{\frac{1}{q}}\sqrt{\min\{m,q\}} \min\left\{\sqrt{nm},\sqrt{\max\{m,p\}}\right\}\\
&=\left\{\begin{array}{ll}n^{1+\frac{1}{p}}m^{1+\frac{1}{q}}& m\le \min\left\{\frac{p}{n},q\right\},\\ \sqrt{q} n^{1+\frac{1}{p}} m^{\frac12+\frac{1}{q}} & q\le m\le \frac{p}{n},\\ \sqrt{p}n^{\frac12+\frac{1}{p}}m^{\frac12+\frac{1}{q}}&\frac{p}{n}\le m\le \min\{p,q\},\\\sqrt{pq} n^{\frac12+\frac{1}{p}}m^{\frac{1}{q}}
& \max\left\{\frac{p}{n},q\right\}\le m\le p,\\
n^{\frac12+\frac{1}{p}} m^{1+\frac{1}{q}}& p\le m\le q,\\
\sqrt{q} n^{\frac{1}{2}+\frac{1}{p}} m^{\frac{1}{2}+\frac{1}{q}} & m\ge\max\{p,q\}.
\end{array} \right.
\end{split}
\end{align}

We will next prove that~\eqref{eq:six case upper} is optimal in all of the six ranges that appear in~\eqref{eq:six case upper}; by~\eqref{eq:apply lemma introduce Z} and~\eqref{eq:volume of ellp ellq}, this will complete the proof of Corollary~\ref{cor:surface area on lplq}. Recalling~\eqref{eq:def sfZ}  and using~\eqref{eq:moment of R1}, \eqref{eq:moment of R2}, \eqref{eq:moment of N 1}, \eqref{eq:moment of N2}, the fourth moment of $\sfZ$ can be evaluated (up to universal constant factors) as follows.
\begin{align}\label{eq:moment of z4}
\begin{split}
\E\big[\sfZ^4\big]&=\E \bigg[\sum_{i=1}^n\sum_{j=1}^n \RR_i\RR_j\sfN_i\sfN_j\bigg]\\
&=n\big(\E\big[\RR^2\big]\big)\big(\E\big[\sfN^2\big]\big)+n(n-1)\big(\E[\RR]\big)^2\big(\E[\sfN]\big)^2\\
&\asymp \frac{(\max\{m,p\})^3(\min\{m,q\})^2}{p^4} n m^{\frac{4}{q}-\frac{4}{p}-1}
+\frac{(\max\{m,p\}\min\{m,q\})^2}{p^4}n^2 m^{\frac{4}{q}-\frac{4}{p}}\\& \asymp \frac{(\max\{m,p\}\min\{m,q\})^2\max\{nm,p\}}{p^4}n m^{\frac{4}{q}-\frac{4}{p}-1}.
\end{split}
\end{align}
By using H\"older's inequality similarly to~\eqref{eq:use holder 3/2}, we conclude that
\begin{multline*}
pn^{\frac{1}{p}}m^{\frac{1}{p}}\E[\sfZ]\ge pn^{\frac{1}{p}}m^{\frac{1}{p}}\frac{\big(\E\big[\sfZ^2\big]\big)^{\frac{3}{2}}}{\big(\E\big[\sfZ^4\big]\big)^{\frac12}}\\\stackrel{\eqref{eq:second moment of Z with normalization} \wedge \eqref{eq:moment of z4}}{\asymp} n^{1+\frac{1}{p}}m^{\frac12+\frac1{q}}\frac{\sqrt{\max\{m,p\}\min\{m,q\}}}{\sqrt{\max\{nm,p\}}}
=\left\{\begin{array}{ll}n^{1+\frac{1}{p}}m^{1+\frac{1}{q}}& m\le \min\left\{\frac{p}{n},q\right\},\\ \sqrt{q} n^{1+\frac{1}{p}} m^{\frac12+\frac{1}{q}} & q\le m\le \frac{p}{n},\\ \sqrt{p}n^{\frac12+\frac{1}{p}}m^{\frac12+\frac{1}{q}}&\frac{p}{n}\le m\le \min\{p,q\},\\\sqrt{pq} n^{\frac12+\frac{1}{p}}m^{\frac{1}{q}}
& \max\left\{\frac{p}{n},q\right\}\le m\le p,\\
n^{\frac12+\frac{1}{p}} m^{1+\frac{1}{q}}& p\le m\le q,\\
\sqrt{q} n^{\frac{1}{2}+\frac{1}{p}} m^{\frac{1}{2}+\frac{1}{q}} & m\ge\max\{p,q\}.
\end{array} \right..\tag*{\qedhere}
\end{multline*}
\end{proof}

Lemma~\ref{lem:nested} below applies Theorem~\ref{cor:surface area on lplq} iteratively to obtain an upper bound on the surface area of the unit sphere of nested $\ell_p$ norms  on $k$-tensors  (the case $k=2$ corresponds to  $n$ by $m$ matrices equipped with the $\ell_p^n(\ell_q^m)$ norm). The second part of Lemma~\ref{lem:nested}, namely the conclusion~\eqref{eq:iq Yk} below, is an implementation of the approach towards Conjecture~\ref{isomorphic reverse conj1}   for the hypercube that we described in Remark~\ref{rem:nested lp}.
\begin{lemma}\label{lem:nested}Suppose that $k, n_1,\ldots,n_k\in \N$ and $p_1,\ldots,p_k\in [1,\infty]$ are such that $n_1\ge \max\{3,p_1-2\}$ and $n_1n_2\ldots n_{j-1}\ge p_j-2$ for every $j\in \{2,\ldots,k\}$.
Define normed spaces $\Y_0,\Y_1,\ldots,\Y_k$ by setting $\Y_0=\R$ and inductively  $\Y_{j}= \ell_{p_j}^{n_j}(\Y_{j-1})$ for $j\in \{1,\ldots,k\}$. Then,
\begin{equation}\label{eq:ratio for nested ellp}
\frac{\vol_{n_1\ldots n_k-1}\big(\partial B_{\Y_k}\big)}{\vol_{n_1\ldots n_k}\big(B_{\Y_k}\big)}\le e^{O(k)}\sqrt{p_1}\prod_{j=1}^k n_j^{\frac12+\frac{1}{p_j}}.
\end{equation}
Hence, using the natural identification of the vector space that underlies $\Y_k$ with $\R^{\dim(\Y_k)}=\R^{n_1n_2\ldots n_k}$, if in addition we have $n_1=O(1)$ and $p_j=\log n_j$ for every $j\in \k$, then
\begin{equation}\label{eq:iq Yk}
B_{\Y_k}\subset B_{\ell_\infty^{\dim(\Y_k)}}\subset e^{O(k)} B_{\Y_k}
\qquad \mathrm{and}\qquad \frac{\mathrm{MaxProj}\big(B_{\Y_k}\big)}{\vol_{\dim(\Y_k)}\big(B_{\Y_k}\big)}\le e^{O(k)},
\end{equation}
where we recall the notation~\eqref{eq:max projection display}.
\end{lemma}

\begin{proof} Suppose that $n,m\in \N$ and $p\in (1,\infty)$. By applying Cauchy--Schwarz to the right hand side of~\eqref{eq:identity surface of ellpnX} while using the case $\alpha=1$ of~\eqref{eq:alpha moment of R}, we see that for every normed space $\X=(\R^m,\|\cdot\|_\X)$ we have
\begin{equation}\label{eq:upper surface of ellpnX after Cauchy-Schwarz}
\frac{\vol_{nm-1}\big(\partial B_{\ell_p^n(\X)}\big)}{\vol_{nm}\big(B_{\ell_p^n(\X)}\big)}\le \frac{p\Gamma\big(1+\frac{nm}{p}\big)}{\Gamma\big(1+\frac{nm-1}{p}\big)}\Bigg(\frac{n\Gamma\big(\frac{m+2p-2}{p}\big)}{\Gamma\big(\frac{m}{p}\big)}\int_{\partial B_{\X}} \big\|\nabla\|\cdot\|_{\X}\big\|_{\ell_2^m}^2\ud\kappa_{\X}\Bigg)^{\frac12}.
\end{equation}
If also $m\ge \max\{3,p-2\}$, then by Stirling's formula~\eqref{eq:upper surface of ellpnX after Cauchy-Schwarz} gives the following estimate.
\begin{equation}\label{eq:upper surface of ellpnX after Cauchy-Schwarz2}
\frac{\vol_{nm-1}\big(\partial B_{\ell_p^n(\X)}\big)}{\vol_{nm}\big(B_{\ell_p^n(\X)}\big)}\lesssim n^{\frac12+\frac{1}{p}}m \bigg(\int_{\partial B_{\X}} \big\|\nabla\|\cdot\|_{\X}\big\|_{\ell_2^m}^2\ud\kappa_{\X}\bigg)^{\frac12}.
\end{equation}

By continuity we may assume that $p_1,\ldots,p_k\in (1,\infty)$. Denote $d_0=1$ and $d_j=\dim(\Y_j)=n_1n_2\ldots n_j$ for $j\in \k$. We will naturally identify $\Y_j$ with $(\R^{d_j},\|\cdot\|_{\Y_j})$. As $\Y_k=\ell_{p_k}^{n_k}(\Y_{k-1})$,  we deduce from~\eqref{eq:upper surface of ellpnX after Cauchy-Schwarz2} that
\begin{equation}\label{eq:upper surface of Yj}
\frac{\vol_{d_k-1}\big(\partial B_{\Y_k}\big)}{\vol_{d_k}\big(B_{\Y_k}\big)}\lesssim n_k^{\frac12+\frac{1}{p_k}}\bigg(\prod_{j=1}^{k-1}n_j\bigg)\bigg(\int_{\partial B_{\Y_{k-1}}} \big\|\nabla\|\cdot\|_{\Y_{k-1}}\big\|_{\ell_2^{d_{k-1}}}^2\ud\kappa_{\Y_{k-1}}\bigg)^{\frac12}.
\end{equation}
At the same time, by~\eqref{eq:quadratic gradient identity} for every $j\in \k$ we have
\begin{equation}\label{eq:recurse second moment Yj}
\int_{\partial B_{\Y_j}} \big\|\nabla \|\cdot\|_{\Y_j}\big\|^2_{\ell_2^{d_j}}\ud \kappa_{\Y_j}=\frac{n_j\Gamma\big(\frac{d_j}{p_j}\big)\Gamma\big(\frac{d_{j-1}+2p_j-2}{p_j}\big)}{\Gamma\big(\frac{d_{j-1}}{p_j}\big)
\Gamma\big(\frac{d_j+2p_j-2}{p_j}\big)}\int_{\partial B_{\Y_{j-1}}} \big\|\nabla\|\cdot\|_{\Y_{j-1}}\big\|_{\ell_2^{d_{j-1}}}^2\ud\kappa_{\Y_{j-1}}.
\end{equation}
If also $j\ge 2$, then $d_{j-1}\ge n_1\ge 3$ and by assumption $d_{j-1}\ge p_j-2$, so by Stirling's formula~\eqref{eq:recurse second moment Yj} gives
\begin{equation}\label{eq:recurse second moment Yj2}
\forall j\in \{2,\ldots,k\},\qquad \int_{\partial B_{\Y_j}} \big\|\nabla \|\cdot\|_{\Y_j}\big\|^2_{\ell_2^{d_j}}\ud \kappa_{\Y_j}\asymp n_j^{\frac{2}{p_j}-1}\int_{\partial B_{\Y_{j-1}}} \big\|\nabla\|\cdot\|_{\Y_{j-1}}\big\|_{\ell_2^{d_{j-1}}}^2\ud\kappa_{\Y_{j-1}}.
\end{equation}
When $j=1$ we have $d_0=1$ and  $n_1\ge \max\{3,p_1-2\}$, and therefore by Stirling's formula~\eqref{eq:recurse second moment Yj} gives
\begin{equation}\label{eq:recurse second moment Y1}
\int_{\partial B_{\Y_1}} \big\|\nabla \|\cdot\|_{\Y_1}\big\|^2_{\ell_2^{d_1}}\ud \kappa_{\Y_1}\asymp p_1n_1^{\frac{2}{p_1}-1}.
\end{equation}
Hence, by applying~\eqref{eq:recurse second moment Yj2} iteratively in combination with the base case~\eqref{eq:recurse second moment Y1}, we conclude that
\begin{equation}\label{eq:final integral of square}
\int_{\partial B_{\Y_{k-1}}} \big\|\nabla\|\cdot\|_{\Y_{k-1}}\big\|_{\ell_2^{d_{k-1}}}^2\ud\kappa_{\Y_{k-1}}\le e^{O(k)}p_1\prod_{j=1}^{k-1}  n_j^{\frac{2}{p_j}-1}.
\end{equation}
A substitution of~\eqref{eq:final integral of square} into~\eqref{eq:upper surface of Yj} yields the desired estimate~\eqref{eq:ratio for nested ellp}.

To deduce the conclusion~\eqref{eq:iq Yk}, note that for every $j\in \k$ we have the point-wise bounds
$$
\|\cdot\|_{\ell_{\infty}^{n_j}(\Y_{j-1})}\le \|\cdot\|_{\Y_j}=\|\cdot\|_{\ell_{p_j}^{n_j}(\Y_{j-1})}\le n_j^{\frac{1}{p_j}}\|\cdot\|_{\ell_{\infty}^{n_j}(\Y_{j-1})}.
$$
It follows by induction that
$$
\|\cdot\|_{\ell_\infty^{d_k}}\le \|\cdot\|_{\Y_k}\le \bigg(\prod_{j=1}^k n_j^{\frac{1}{p_j}}\bigg)\|\cdot\|_{\ell_\infty^{d_k}}=e^{O(k)}\|\cdot\|_{\ell_\infty^{d_k}},
$$
where the final step holds if $p_j=\log n_j$ for every $j\in \k$. This implies the inclusions in~\eqref{eq:iq Yk}. Furthermore,  $\Y_k$ belongs to the class of spaces from Example~\ref{ex permutation and sign}. Hence $\Y_k$ is canonically positioned and by the discussion in  Section~\ref{sec:positions} know that $B_{\Y'}$ is in its minimum surface area position. Therefore,
$$
\frac{\mathrm{MaxProj}\big(B_{\Y_k}\big)}{\vol_{d_k}\big(B_{\Y_k}\big)}\asymp \frac{\vol_{d_k-1}\big(\partial B_{\Y_k}\big)}{\vol_{d_k}\big(B_{\Y_k}\big)\sqrt{d_k}}\le e^{O(k)}\sqrt{p_1}\prod_{j=1}^k n_j^{\frac{1}{p_j}}\asymp e^{O(k)},
$$
where the first step uses~\cite[Proposition~3.1]{GP99}, the second step is~\eqref{eq:ratio for nested ellp}, and the final step holds because  $p_1=O(1)$ and $p_j=\log n_j$. This completes the proof of~\eqref{eq:iq Yk}.
\end{proof}

The following technical lemma replaces a more ad-hoc argument that we previously had to deduce Proposition~\ref{prop:slightly less rounded cube} below from Lemma~\ref{lem:nested}; it is due to Noga Alon and we thank him for allowing us to include it here. This lemma shows that the set of super-lacunary products $n_1n_2\ldots n_k$ that can serve as dimensions of the space $\Y_k$ in Lemma~\ref{lem:nested} for which~\eqref{eq:iq Yk} holds is quite dense in $\N$.

\begin{lemma}\label{lem log* approx} For every integer $n\ge 3$ there are $k,m\in \N\cup\{0\}$ and integers $n_1<n_2<\ldots <n_k$ that satisfy
\begin{itemize}
\item $n=n_1n_2\ldots n_k +m$,
\item $n_1\in \{6,7\}$ and $n_{i+1}\le 2^{n_i}\le n_{i+1}^3$ for every $i\in \{1,\ldots,k-1\}$,
\item $m\le (\log n)^{1+o(1)}$.
\end{itemize}
\end{lemma}
Prior to proving Lemma~\ref{lem log* approx}, we will make some preparatory (mechanical) observations for ease of later reference. Note first that the conclusion $n_{i+1}\le 2^{n_i}\le n_{i+1}^3$ of Lemma~\ref{lem log* approx} can be rewritten as
$$
\forall i\in \{1,\ldots, k-1\},\qquad \log_2 n_{i+1}\le n_i \le \log_{\sqrt[3]{2}} n_{i+1}.
$$
It follows by induction that
\begin{equation}\label{eq:iterated logarithms}
\forall i\in \{1,\ldots, k\},\qquad  \log_{2}^{[k-i]}n_k\le n_i\le \log_{\sqrt[3]{2}}^{[k-i]}n_k,
\end{equation}
where, as in~\eqref{eq:iterated logarithm def}, we denote the iterates of $\f:(0,\infty)\to \R$  by  $\f^{[j]}=\f\circ \f^{[j-1]}: (\f^{[j-1]})^{-1}(0,\infty)\to \R$ for each $j\in \N$, with the convention $\f^{[0]}(x)=x$ for every $x\in (0,\infty)$. Since $n_1\in \{6,7\}$, it follows from~\eqref{eq:iterated logarithms} that
\begin{equation}\label{eq:k is log *}
k\asymp \log^{*}\!\!n_k\lesssim \log^{*}\!\!n.
\end{equation}
Consequently,
\begin{multline*}
n_k\log n_k\asymp n_kn_{k-1}\le \prod_{i=1}^k n_k\le  n=m+\prod_{i=1}^k n_k\le (\log n)^{1+o(1)}+\prod_{i=1}^k \log_{\sqrt[3]{2}}^{[k-i]}n_k \\ \lesssim   (\log n)^2+ n_k(\log n_k) (\log\log n_k)^{O(\log^{*}\!\!n_k)}\lesssim  (\log n)^2+n_k (\log n_k)^2.
\end{multline*}
This implies the following (quite crude) bounds on $n_k$.
\begin{equation}\label{eq:nk value}
\frac{n}{(\log n)^2}\lesssim n_k\lesssim \frac{n}{\log n}.
\end{equation}
Note in particular that thanks to~\eqref{eq:nk value} we know that~\eqref{eq:k is log *} can be improved to $k\asymp \log^{*}\!\!n$.

\begin{proof}[Proof of Lemma~\ref{lem log* approx} ]  Let $\mathbb{M}\subset \N$ be the set of all those $x\in \N$ that can be written as $x=n_1n_2\ldots n_k$ for some $k,n_1,\ldots,n_k\in \N$ that satisfy $n_k>n_{k-1}>\ldots> n_1\in \{6,7\}$ and
\begin{equation}\label{eq:cubed condition}
\forall i\in \{1,\ldots,k-1\},\qquad n_{i+1}\le 2^{n_i}\le n_{i+1}^3.
\end{equation}
The goal of Lemma~\ref{lem log* approx} is to show that there exists $x\in \mathbb{M}$ such that
\begin{equation}\label{eq:log n approx from below}
n-(\log n)^{1+o(1)}\le x\le n.
\end{equation}
By adjusting the $o(1)$ term, we may assume that $n$ is sufficiently large, say, $n\ge n(0)$ for some fixed $n(0)\in \N$ that will be determined later. We will then find $x\in \mathbb{M}$ with a representation $x=n_1n_2\ldots n_k$ as above and
\begin{equation}\label{eq:product error term form}
n-n_1n_2\ldots n_{k-1}\le x\le n.
\end{equation}
This would imply the desired bound~\eqref{eq:log n approx from below} because
\begin{equation}\label{eq:product of first k-1}
\prod_{i=1}^{k-1}n_i\stackrel{\eqref{eq:iterated logarithms}}{\le} \prod_{i=1}^{k-1} \log_{\sqrt[3]{2}}^{[k-i]}n_k\stackrel{\eqref{eq:k is log *}}{\lesssim} (\log n_k)^{1+o(1)} \stackrel{\eqref{eq:nk value}}{\lesssim} (\log n)^{1+o(1)}.
\end{equation}

We will first construct $\{y_i\}_{i=1}^\infty\subset \mathbb{M}$ such that $y_1=7$ and $y_i<y_{i+1}<12y_i$ for every $i\in \N$. Furthermore, for each $i\in \N$ there are $k,n_1,\ldots,n_k\in \N$ with $y_i=n_1n_2\ldots n_k$ such that $n_k>n_{k-1}>\ldots> n_1\in \{6,7\}$ and
\begin{equation}\label{eq:squared condition}
\forall j\in \{1,\ldots,k-1\},\qquad n_{j+1}^2 \le 2^{n_j}\le 2n_{j+1}^2,
\end{equation}
which is a more stringent requirement than~\eqref{eq:cubed condition}. Note in passing that~\eqref{eq:squared condition} implies the (crude) bound
\begin{equation}\label{eq:product small}
\prod_{j=1}^k \Big(1+\frac{1}{n_j}\Big)\le 2.
\end{equation}
To verify~\eqref{eq:product small}, note that since $\{n_j\}_{j=1}^k$ is strictly increasing and the  second inequality in~\eqref{eq:squared condition} holds, it is mechanical to check that $n_1\ge 6$, $n_2\ge 7$, $n_3\ge 8$, $n_4\ge 12$ and $n_{j+1}\ge 3 n_j$ for every $j\in \{4,5,\ldots,k-1\}$. So,
$$
\prod_{j=1}^k \Big(1+\frac{1}{n_j}\Big)\le \Big(1+\frac{1}{6}\Big)\Big(1+\frac{1}{7}\Big)\Big(1+\frac{1}{8}\Big)e^{\sum_{s=0}^\infty \frac{1}{12\cdot 3^s}}=\Big(1+\frac{1}{6}\Big)\Big(1+\frac{1}{7}\Big)\Big(1+\frac{1}{8}\Big)e^{\frac18}\le 2.
$$

Suppose that $y_i$ has been defined with a representation $y_i=n_1n_2\ldots n_k$ that fulfils the above requirements. Define $m_0,m_1,\ldots,m_k\in \N$ with $m_0=6$, $m_k=n_k+1$ and $m_j\in \{n_j,n_j+1\}$ for all $j\in \{1,\ldots,k-1\}$ by  induction as follows. Assuming that $m_{j+1}$ has already been constructed for some $j\in \{1,\ldots,k-1\}$, let
\begin{equation}\label{eq:def mj}
m_j\eqdef \left\{\begin{array}{ll}n_j&\mathrm{if}\ \ m_{j+1}^2\le 2^{n_j},\\
n_j+1&\mathrm{if}\ \ m_{j+1}^2>2^{n_j}.\end{array}\right.
\end{equation}

Definition~\eqref{eq:def mj} implies that $m_j<m_{j+1}$. Indeed,  $n_j<n_{j+1}$ so if $m_j=n_j$, then $n_j<n_{j+1}\le m_{j+1}$ since $m_{j+1}\ge n_{j+1}$ by the induction hypothesis. On the other hand,  if $m_j=n_j+1$, then since the first inequality in~\eqref{eq:squared condition} holds, the definition~\eqref{eq:def mj} necessitates that $m_{j+1}=n_j+1$, so $m_j<m_{j+1}$ in this case as well.

Definition~\eqref{eq:def mj} also ensures that the requirement~\eqref{eq:squared condition} is inherited by $\{m_j\}_{j=1}^k$, i.e.,
\begin{equation}\label{eq:squared condition m}
\forall j\in \{1,\ldots,k-1\},\qquad m_{j+1}^2 \le 2^{m_j}\le 2m_{j+1}^2.
\end{equation}
Indeed, if $m_j=n_j$, then $m_{j+1}^2\le 2^{n_j}=2^{m_j}$ by~\eqref{eq:def mj}, i.e., the first inequality in~\eqref{eq:squared condition m} holds, and  the second inequality in~\eqref{eq:squared condition m} holds because $m_{j+1}\ge n_{j+1}$ and~\eqref{eq:squared condition} holds. On the other hand, if $m_j=n_j+1$, then by~\eqref{eq:def mj} we necessarily have $m_{j+1}=n_j+1$ and   $m_{j+1}^2>2^{n_j}$, which directly gives the second inequality in~\eqref{eq:squared condition m}, and in combination with~\eqref{eq:squared condition} we also get the first inequality in~\eqref{eq:squared condition m} because
$$
\frac{m_{j+1}}{2^{m_j}}=\frac{(n_j+1)^2}{2^{n_j+1}}\stackrel{\eqref{eq:squared condition}}{\le} \frac{(n_j+1)^2}{2n_j^2}\le 1,
$$
where the final step uses $n_j\ge 6$, though $n_j\ge 1/(\sqrt{2}-1)=2.414...$ is all that is  needed for this purpose.

If the above construction produces $m_1\in \{6,7\}$, then define $y_{i+1}=m_1m_2\ldots m_k$. Otherwise necessarily $m_1=n_1+1=8$, so~\eqref{eq:squared condition m} holds also when $j=0$ (recall that $m_0=6$, hence $m_1^2=2^6=2^{m_0}$), so  we can define $y_{i+1}=m_0m_1\ldots m_k$ and thanks to~\eqref{eq:squared condition m} in both cases $y_{i+1}$ has the desired form. Moreover,
$$
\frac{y_{i+1}}{y_i}\le 6\prod_{j=1}^k \Big(1+\frac{1}{n_j}\Big)\stackrel{\eqref{eq:product small}}{\le} 12.
$$
This completes the inductive construction of the desired sequence $\{y_i\}_{i=1}^\infty\subset \mathbb{M}$.

With the sequence $\{y_i\}_{i=1}^\infty\subset \mathbb{M}$ at hand, will next explain how to obtain for each integer $n\ge n(0)$, where $n(0)\in \N$ is a sufficiently large universal constant that is yet to be determined, an element $x\in \mathbb{M}$ that approximates $n$ as in~\eqref{eq:product error term form}. Let $i\in \N$ be such that $y_i\le n\le y_{i+1}$ and denote $y=y_i$. Thus, there are $k,n_1,\ldots,n_k\in \N$ for which $y=n_1n_2\ldots n_k$ such that $n_k>n_{k-1}>\ldots> n_1\in \{6,7\}$ and~\eqref{eq:squared condition} holds.

 If $y\ge n-n_1n_2\ldots n_{k-1}$, then $x=y$ has the desired approximation property, so suppose from now that $y<n-n_1n_2\ldots n_{k-1}$, or equivalently $n/(n_1n_2\ldots n_{k-1})> y/(n_1n_2\ldots n_{k-1})+1=n_k+1$. Hence, if we define
\begin{equation}\label{eq:def nk'}
n_k'\eqdef \left\lfloor \frac{n}{n_1n_2\ldots,n_{k-1}}\right\rfloor\qquad \mathrm{and}\qquad x=n_1n_2\ldots n_{k-1} n_k',
\end{equation}
Then $n_k'\ge n_k+1\gtrsim n/(\log n)^2$, where we used~\eqref{eq:nk value}. Consequently, recalling~\eqref{eq:iterated logarithms}, there is a universal constant $n(0)\in \N$ such that if $n\ge n(0)$, then $n_k'>\max\{144,n_{k-1}\}$. So, the sequence $n_1,n_2,\ldots,n_{k-1},n_k'$ is still increasing. Since by design $x$ satisfies~\eqref{eq:product error term form}, it remains to check that $x\in \mathbb{M}$, i.e., that~\eqref{eq:cubed condition} holds. Since $n_1,\ldots,n_k$ are assumed to satisfy the more stringent requirement~\eqref{eq:squared condition}, we only need to check that
\begin{equation}\label{eq:prime condition to check}
n_{k}'\le 2^{n_{k-1}}\le (n_{k}')^3.
\end{equation}
The second inequality in~\eqref{eq:prime condition to check} is valid since~\eqref{eq:squared condition} holds and $n_k'>n_k$. For the first inequality in~\eqref{eq:prime condition to check}, note that $y\le n\le 12 y$, as  $y_{i+1}\le 12 y_i$. Hence, $n_k'\le n/(n_1n_2\ldots n_{k-1})\le 12y/(n_1n_2\ldots n_{k-1})= 12n_k$. Therefore,
$$
2^{n_{k-1}}\stackrel{\eqref{eq:squared condition}}{\ge} n_k^2\ge \bigg(\frac{n_k'}{12}\bigg)^2>n_k',
$$
where the last step uses the fact that $n_k'>144$.
\end{proof}

We are now ready to extend the conclusion~\eqref{eq:iq Yk} of Lemma~\ref{lem:nested} to all dimensions $n\in \N$. Namely, we will prove the following proposition, which comes very close to proving Conjecture~\ref{isomorphic reverse conj1} for the hypercube $[-1,1]^n$ via a route that differs from the way by which we proved Theorem~\ref{prop:rounded cube}.

\begin{proposition}\label{prop:slightly less rounded cube} For any $n\in \N$ there is a normed space $\Y=(\R^n,\|\cdot\|_\Y)$ that satisfies
\begin{equation*}\label{eq:round cube'}
\forall x\in \R^n\setminus \{0\},\qquad  \|x\|_{\ell_\infty^n}\le \|x\|_\Y\le e^{O(\log^*\!\! n)}\|x\|_{\ell_\infty^n}\qquad \mathrm{and}\qquad \frac{\vol_{n-1}\big(\proj_{x^\perp}B_{\Y}\big)}{\vol_n(B_{\Y})}\le  e^{O(\log^*\!\! n)}.
\end{equation*}
Furthermore, $\Y$ can be taken to be an $\ell_\infty$ direct sum of nested $\ell_p$ spaces as in Lemma~\ref{lem:nested}.
\end{proposition}

\begin{proof} Let $\mathbb{M}\subset \N$ be the set of integers from the proof of Lemma~\ref{lem log* approx}, namely $m\in \mathbb{M}$ if and only if there are integers $n_k>n_{k-1}>\ldots >n_1\in \{6,7\}$ that satisfy~\eqref{eq:cubed condition} such that $m=n_1n_2\ldots n_k$. By Lemma~\ref{lem:nested}, there exists $C>1$ such that for every $m\in \mathbb{M}$ there is a normed space $\Y^m=(\R^m,\|\cdot\|_{\Y^m})$ that satisfies
$$
\|\cdot\|_{\ell_\infty^m}\le \|\cdot\|_{\Y^m}\le e^{C\log^*\!\! m}\|\cdot\|_{\ell_\infty^m} \qquad \mathrm{and}\qquad  \frac{\mathrm{MaxProj}\big(B_{\Y^m}\big)}{\vol_n\big(B_{\Y^m}\big)}\le e^{C\log^*\!\! m}.
$$

By applying Lemma~\ref{lem log* approx} iteratively  write $n=m_1+\ldots+m_{s+1}$ for  $m_1,\ldots,m_s\in \mathbb{M}$ and $m_{s+1}\in \{1,2\}$ that satisfy $m_{i+1}\le (\log m_i)^c$ for every $i\in \{1,\ldots,s\}$, where $c>1$ is a universal constant. Denote $\Y^{m_{s+1}}=\ell_\infty^{m_{s+1}}$ and consider the $\ell_\infty$ direct sum
$$
\Y\eqdef \Y^{m_1}\oplus_\infty \Y^{m_2}\oplus_\infty\ldots\oplus_\infty \Y^{m_{s+1}}=(\R^n,\|\cdot\|_{\Y}).
$$
Then $\|\cdot\|_{\ell_\infty^n}\le \|\cdot\|_\Y\le \max_{i\in \{1,\ldots,s+1\}} e^{C\log^*\!\! m_i}\|\cdot\|_{\ell_\infty^{m_i}}\le e^{C\log^*\!\! n}\|\cdot\|_{\ell_\infty^n}$. We claim that $\frac{\mathrm{MaxProj}(B_\Y)}{\vol_n(B_\Y)}\le e^{O(\log^*\!\! n)}$.

Since $B_\Y= B_{\Y^{m_1}}\times B_{\Y^{m_2}}\times \ldots\times B_{\Y^{m_{s+1}}}$, by an inductive application of Lemma~\ref{lem:product quadratic formula} we have
$$
\frac{\mathrm{MaxProj} \big(B_{\Y}\big)}{\vol_{n}\big(B_{\Y}\big)} \le \Bigg(\sum_{i=1}^{s+1} \frac{\mathrm{MaxProj} \big(B_{\Y^{m_{i}}}\big)^2}{\vol_{m_i}\big(B_{\Y^{m_i}}\big)^2}\Bigg)^{\frac12}\le \bigg(\sum_{i=1}^{s+1} e^{2C\log^*\!\! m_i}\bigg)^{\frac12}\lesssim e^{C\log^*\!\! n},
$$
where the first step uses Lemma~\ref{lem:product quadratic formula}, the penultimate step is our assumption on $\Y^{m_i}$, and the final step has the following elementary justification. Recall that for every $i\in \{1,\ldots,s\}$ we have  $m_{i+1}\le (\log m_i)^c$, where $c>1 $ is a universal constant. So, $m_{i+2}\le c^c(\log\log m_i)^c$ for every $i\in \{1,\ldots,s-1\}$. Fix $n_0\in \N$ such that $c^c(\log\log n)^c\le \log n$ for every $n\ge n_0$. Then, $m_{i+2}\le \log m_i$ if $m_i\ge n_0$, hence $\log^*\!\! m_{i+2}\le \log^*\!\! m_i-1$. Let $i_0$ be the largest $i\in \{1,\ldots,s+1\}$ for which $m_i<n_0$. Then,   $\log^*\!\! m_{2i}\le \log^*\!\! m_2-i\le \log^*\!\! n-i$ and $\log^*\!\! m_{2j+1}\le \log^*\!\! m_1-j\le \log^*\!\! n-j$ if $2i, 2j+1\in \{1,\ldots,i_0-1\}$. Also, $|\{i_0,\ldots,s+1\}|=O(1)$. Consequently,
\begin{equation*}
\sum_{i=1}^{s+1} e^{2C\log^*\!\! m_i}\le e^{2C\log^*\!\! n}\sum_{k=0}^\infty e^{-2Ck}+O(1)\lesssim e^{2C\log^*\!\! n}. \tag*{\qedhere}
\end{equation*}
\end{proof}

\begin{remark} A straightforward way to attempt to compute  the surface area of the unit sphere of a normed space $\X=(\R^n,\|\cdot\|_\X)$  is to fix a direction $z\in S^{n-1}$ and consider $\partial B_\X$ as the union of the two graphs of the functions $\Psi_z^\X, \psi^\X_z: \proj_{z^\perp}(B_\X)\to \R$ that are defined by setting $\Psi_z^\X(x)$ and $\psi^\X_z(x)$ for each $x\in \proj_{z^\perp}(B_\X)$ to be, respectively, the largest and smallest $s\in \R$ for which $x+sz\in \partial B_{\X}$. We then have
\begin{equation}\label{eq:two graphs}
\vol_{n-1}(\partial B_{\X})=\int_{\proj_{z^\perp}(B_{\X})} \sqrt{1+\|\nabla \Psi_z^\X(x)\|_{\ell_2^n}^2}\ud x+\int_{\proj_{z^\perp}(B_{\X})} \sqrt{1+\|\nabla \psi_z^\X(x)\|_{\ell_2^n}^2}\ud x.
\end{equation}
When $\X=\ell_p^n$ for some $p\in (1,\infty)$ and $z=e_{n}$,
$$
\forall x\in \proj_{e_n^\perp}\big(B_{\ell_p^n}\big)= B_{\ell_p^{n-1}},\qquad \Psi^{\ell_p^n}_{e_n}(x)=-\psi^{\ell_p^n}_{e_n}(x)=\big(1-\|x\|_{\ell_p^{n-1}}^p\big)^{\frac{1}{p}}.
$$
Therefore~\eqref{eq:two graphs} becomes
\begin{equation*}
\frac{\vol_{n-1}\big(\partial B_{\ell_p^n}\big)}{\vol_{n-1}\big(B_{\ell_p^{n-1}}\big)}=2\fint_{B_{\ell_p^{n-1}}} \Big(1+(1-\|x\|_{\ell_p^{n-1}}^p)^{-\frac{2(p-1)}{p}} \sum_{i=1}^{n-1} |x_i|^{2(p-1)}\Big)^{\frac12}\ud x.
\end{equation*}
By~\cite{BGMN05}, a point chosen from the normalized volume measure on $B_{\ell_p^{n-1}}$  is equidistributed with
\begin{equation*}
\big(|\sfG_1|^p+\ldots+|\sfG_{n-1}|^p+\sfZ\big)^{-\frac{1}{p}}(\sfG_1,\ldots,\sfG_{n-1})\in \R^{n-1},
\end{equation*}
where $\sfG_1,\ldots,\sfG_{n-1},\sfZ$ are independent random variables, the density of $\sfG_1,\ldots,\sfG_{n-1}$ at $s\in \R$ is equal to $2\Gamma(1+1/p)^{-1} e^{-|s|^p}$ and the density of $\sfZ$ at $t\in [0,\infty)$ is equal to $e^{-t}$. Consequently,
\begin{align}\label{eq:ZG}
\frac{\vol_{n-1}\big(\partial B_{\ell_p^n}\big)}{\vol_{n-1}\big(B_{\ell_p^{n-1}}\big)}=2\E \bigg[\Big(1+\sfZ^{-\frac{2(p-1)}{p}}\sum_{i=1}^{n-1}|\sfG_i|^{2(p-1)}\Big)^{\frac12}\bigg].
\end{align}
Optimal estimates on moments such as the right hand side of~\eqref{eq:ZG}  were derived (in greater generality) in~\cite{Nao07}, using which one can quickly get asymptotically sharp bounds on the left hand side of~\eqref{eq:ZG}. It is possible to implement this approach to get an alternative treatment of $\ell_p^n(\ell_q^m)$, though it is significantly more involved than the different way by which we proceeded above, and it becomes much more tedious and technically intricate when one aims to treat hierarchically nested $\ell_p$ norms as we did in Lemma~\ref{lem:nested}. Nevertheless, an advantage of~\eqref{eq:two graphs} is that it applies to normed spaces that do not have a product structure as in Lemma~\ref{lem:break surface of ellp(X)}, which is helpful in other settings that we will study elsewhere.
\end{remark}

\subsection{Negatively correlated normed spaces}\label{sec:negative correlation} Our goal here is to further  elucidate the role of symmetries in the context of the discussion in Section~\ref{sec:positions}. Fix $n\in \N$ and $\gamma\ge 1$. Say that a normed space $\X=(\R^n,\|\cdot\|_\X)$ is $\gamma$-{\em negatively correlated} if the standard scalar product $\langle \cdot,\cdot\rangle$ on $\R^n$ is invariant under its isometry group $\mathsf{Isom}(\X)$, i.e., $\mathsf{Isom}(\X)\le \O_n$, and there exists a Borel probability measure $\mu$ on $\mathsf{Isom}(\X)$ such that
\begin{equation}\label{eq:gamma neg}
\forall x,y\in \R^n,\qquad \int_{\mathsf{Isom}(\X)} |\langle Ux,y\rangle|\ud\mu(U)\le \frac{\gamma}{\sqrt{n}}\|x\|_{\ell_2^n}\|y\|_{\ell_2^n}.
\end{equation}
We were inspired to formulate this notion by the proof of Theorem~1.1 in~\cite{Sch89}. It is tailored for the purpose of bounding volumes of hyperplane projections of $B_\X$ from above in terms of the surface area of $\partial B_\X$, as exhibited by the following lemma which generalizes the reasoning in~\cite{Sch89}.

\begin{lemma}\label{lem:re schutt}  Fix $n\in \N$ and $\gamma\ge 1$. If $\X=(\R^n,\|\cdot\|_\X)$ is $\gamma$-negatively correlated, then
$$
\mathrm{MaxProj}(B_\X)\le \frac{\gamma}{2\sqrt{n}}\vol_{n-1}(\partial B_{\X}).
$$
\end{lemma}

\begin{proof} Recall that for every $y\in \partial B_\X$ at which $\partial B_\X$ is smooth we denote the unit outer normal to $\partial B_\X$ at $y$ by $N_\X(y)\in S^{n-1}$. By the Cauchy projection formula~\eqref{eq:use cauchy} for every $x\in S^{n-1}$ we have
$$
\vol_{n-1}\big(\proj_{x^\perp}(B_{\X})\big)=\frac12\int_{\partial B_{\X}} |\langle x,N_\X(y)\rangle| \ud y.
$$
Since every $U\in \mathsf{Isom}(\X)$ is an orthogonal transformation and $N_\X\circ U^{*}=U^{*}\circ N_\X$ almost surely on $\partial B_\X$,
$$
\vol_{n-1}\big(\proj_{x^\perp}(B_{\X})\big)=\frac12\int_{\partial B_{\X}} |\langle Ux,N_\X(y)\rangle| \ud y.
$$
By integrating this identity with respect to $\mu$, we therefore conclude that
$$
\vol_{n-1}\big(\proj_{x^\perp}(B_{\X})\big)= \frac12\int_{\partial B_{\X}} \bigg(\int_{\mathsf{Isom}(\X)}|\langle   Ux,N_\X(y)\rangle|\ud\mu(U)\bigg) \ud y \le \frac{\gamma}{2\sqrt{n}}\vol_{n-1}(\partial B_{\X}),
$$
where we used~\eqref{eq:gamma neg} and the fact that $\|x\|_{\ell_2^n}=1$ and $\|N_\X(y)\|_{\ell_2^n}=1$ for almost every $y\in \partial B_\X$.
\end{proof}

By substituting Lemma~\ref{lem:re schutt} into Theorem~\ref{thm:sep volume ratio-intro} and using~\eqref{eq:quote LN}, we get the following corollary.
\begin{corollary}\label{coro: surface version} Fix $n\in \N$ and $\gamma\ge 1$. If $\X=(\R^n,\|\cdot\|_\X)$ is $\gamma$-negatively correlated, then
$$
\ee(\X)\lesssim \sep(\X)\le 2\gamma \frac{\vol_{n-1}(\partial B_{\X})\diam_{\ell_2^n}(B_{\X})}{\vol_n(B_{\X})\sqrt{n}}.
$$
\end{corollary}

Corollary~\ref{coro: surface version} generalizes Corollary~\ref{cor:canonically positioned ext} since any canonically positioned normed space is $1$-negatively correlated. Indeed, suppose that $\X=(\R^n,\|\cdot\|_\X)$ is canonically positioned. Recall that in Section~\ref{sec:positions} we denoted the Haar probability measure on $\mathsf{Isom}(\X)$ by $\mathcal{h}_\X$. Fix $x,y\in \R^n$. The distribution of the random vector $Ux$ when $U$ is distributed according to $\mathcal{h}_\X$ is $\mathsf{Isom}(\X)$-invariant, and therefore it is isotropic. Hence,
\begin{align*}
\int_{\mathsf{Isom}(\X)} |\langle Ux,y\rangle |  \ud \mathcal{h}_\X(U)&\le \bigg(\int_{\mathsf{Isom}(\X)} \langle Ux,y\rangle^2 \ud \mathcal{h}_\X(U)\bigg)^{\frac12}\\ &\stackrel{\eqref{eq:isotropic}}{=} \frac{\|y\|_{\ell_2^n}}{\sqrt{n}}\bigg(\int_{\mathsf{Isom}(\X)} \|Ux\|_{\ell_2^n}^2 \ud \mathcal{h}_\X(U)\bigg)^{\frac12}= \frac{1}{\sqrt{n}}\|x\|_{\ell_2^n}\|y\|_{\ell_2^n},
\end{align*}
where the final step uses the fact that each $U\in \mathsf{Isom}(\X)$ is an orthogonal transformation.

One way to achieve~\eqref{eq:gamma neg}, which is close in spirit to the considerations in~\cite{Sch89}, is when there are $\Gamma\subset \{-1,1\}^n$ and $\Pi\subset S_n$ such that $U_{\e,\pi}\in \mathsf{Isom}(\X)$ for every $(\e,\pi)\in \Gamma\times \Pi$, where $U_{\e,\pi}\in \mathsf{GL}_n(\R)$ is given by
$$
\forall x=(x_1,\ldots,x_n)\in \R^n,\qquad U_{\e,\pi}x\eqdef \big(\e_1x_{\pi(1)},\ldots\e_nx_{\pi(n)}\big),
$$
and also there are $\alpha,\beta>0$ such that
\begin{equation}\label{eq:Gamma assumption}
\forall w\in \R^n,\qquad \frac{1}{|\Gamma|}\sum_{\e\in \Gamma} |\langle \e, w\rangle|\le \alpha \|w\|_{\ell_2^n},
\end{equation}
and
\begin{equation}\label{eq:Pi assumption}
\forall i,j\in \n,\qquad |\{\pi\in \Pi:\ \pi(i)=j\}|\le \beta\frac{|\Pi|}{n}.
\end{equation}
Under these assumptions, $\X$ is $\gamma$-negatively correlated with $\gamma=\alpha\sqrt{\beta}$. Indeed, we can take $\mu$ in~\eqref{eq:gamma neg} to be the uniform distribution over the finite set $\{U_{\e,\pi}:\ (\e,\pi)\in \Gamma\times \Pi\}\subset \mathsf{Isom}(\X)$, since every $x,y\in \R^n$ satisfy
\begin{multline*}
\frac{1}{|\Gamma\times \Pi|}\sum_{(\e,\pi)\in \Gamma\times \Pi} |\langle U_{\e,\pi}x,y\rangle |\stackrel{\eqref{eq:Gamma assumption}}{\le} \frac{1}{|\Pi|} \sum_{\pi\in \Pi}\alpha \bigg(\sum_{i=1}^n (x_{\pi(i)}y_i)^2\bigg)^{\frac12}\le \alpha\bigg( \sum_{i=1}^n \Big(\frac{1}{|\Pi|} \sum_{\pi\in \Pi} x_{\pi(i)}^2\Big) y_i^2  \bigg)^{\frac12}\\=\alpha\bigg( \sum_{i=1}^n \Big(\frac{1}{|\Pi|} \sum_{j=1}^n |\{\pi\in \Pi:\ \pi(i)=j\}| x_{j}^2\Big) y_i^2  \bigg)^{\frac12}\stackrel{\eqref{eq:Pi assumption}}{\le} \frac{\alpha\sqrt{\beta}}{\sqrt{n}} \|x\|_{\ell_2^n}\|y\|_{\ell_2^n}.
\end{multline*}

The condition~\eqref{eq:Gamma assumption} can be viewed as a negative correlation property of the coordinates of sign vectors that are chosen uniformly from $\Gamma$.  The condition~\eqref{eq:Pi assumption} roughly means that for each $i\in \n$ the sets $\{\pi\in \Pi:\ \pi(i)=1\},\ldots,\{\pi\in \Pi:\ \pi(i)=n\}$ form an approximately equitable partition of $\Pi$. This holds with $\beta=1$  if $\Pi$ is a transitive subgroup of $S_n$. One could formulate weaker conditions that ensure the validity of the conclusion of Lemma~\ref{lem:re schutt} (e.g.~considering bi-Lipschitz automorphisms of $\X$ rather than isometries of $\X$), and hence also the conclusion of Corollary~\ref{coro: surface version}, though we will not pursue this here as we expect that in concrete cases such issues should be easy to handle.

\subsection{Volume ratio computations}\label{sec:volume ratios} Here we will present asymptotic evaluations of volume ratios of some  normed spaces, for the purpose of plugging them into results that we stated in the Introdcution. Due to the large amount of knowledge on this topic that is available in the literature, we will only give a flavor of such applications. The main reference  for the contents of this section is the valuable work~\cite{Sch82}.

We will start by examining  the iteratively nested $\ell_p$ products $\{\X_k\}_{k=0}^\infty$ of Lemma~\ref{lem:nested tensorization}, in the special case when the initial space $\X=\X_0$  is a canonically positioned normed space for which Conjecture~\ref{conj:weak reverse iso when canonical} holds. Thus, we are fixing $\{n_k\}_{k=0}^\infty\subset  \N$ and $\{p_k\}_{k=1}^\infty\subset [1,\infty]$, and assuming that $\X=(\R^{n_0},\|\cdot\|_\X)$ is a canonically positioned normed space satisfying Conjecture~\ref{conj:weak reverse iso when canonical}, i.e., \eqref{eq:n0 assumption} holds with $\alpha=O(1)$; the case $\X=\R$ is sufficiently rich for our present illustrative purposes, but one can also take $\X=\bfE$ to be any symmetric space, per Lemma~\ref{lem:reverse iso for sym}. By Lemma~\ref{lem:nested tensorization} and Corollary~\ref{cor:sep sharp if conj}, if we define inductively
\begin{equation*}
\forall k\in \N, \qquad \X_{k+1}=\ell_{p_k}^{n_k}(\X_k), \qquad\mathrm{where}\qquad \X_0=\X,
\end{equation*}
then, because $\{\X_k\}_{k=1}^\infty$ are canonically positioned (they belong to the class of spaces in Remark~\ref{ex permutation and sign}),
\begin{equation}\label{eq:sep of nested}
\forall m\in \N,\qquad \sep(\X_m)\asymp \evr(\X_m)\sqrt{\dim(\X_m)}=\evr(\X_m)\sqrt{n_0\cdots n_m}.
\end{equation}

Let $\{\bfH_k\}_{k=0}^\infty$ be the sequence of Euclidean spaces that arise from the above construction with the same $\{n_k\}_{k=0}^\infty\subset  \N$ but with $p_k=2$ for all $k\in \N$ and $\X=\ell_2^{n_0}$. Thus, for each $m\in \N$ the Euclidean space $\bfH_m$ can be identified naturally with $\ell_2^{n_0\cdots n_m}$. Under this identification, by a straightforward inductive application of   H\"older's inequality and the fact that the $\ell_p$ norm deceases with $p$, the L\"owner ellipsoid of $\X_m$ satisfies\footnote{As $\X_m$ is canonically positioned, this holds as an equality, but for the present purposes we just need the stated inclusion.}
$$
\mathscr{L}_{\X_m}\subset  \left(\prod_{k=1}^m n_k^{\max\left\{\frac12-\frac{1}{p_k},0\right\}}\right)(\mathscr{L}_\X)^{n_1\cdots n_m}.
$$
 Also, by Lemma~\ref{lem:volume radius of nested} we have
$$
\vol_{n_0\cdots n_m}\big(B_{\X_m}\big)^{\frac{1}{n_0\cdots n_k}}\asymp \frac{\vol_{n_0}\big(B_\X\big)^{\frac{1}{n_0}}}{\prod_{k=1}^m n_k^{\frac{1}{p_k}}}.
$$
These facts combine to give the following consequence of~\eqref{eq:sep of nested}.
$$
\sep(\X_m)\asymp \evr(\X) \prod_{k=1}^m n_k^{\max\left\{\frac12,\frac{1}{p_k}\right\}}.
$$

In particular, when we take $\X=\R$  and consider only two steps of the above iteration, we get the following asymptotic evaluation of the separation modulus of the $\ell_p^n(\ell_q^m)$ norm the space of  $n$-by-$m$ matrices $M_{n\times m}(\R)$ for any $n,m\in \N$ and $p,q\ge 1$; the case of square matrices was stated in the Introduction as~\eqref{eq:ellpellq case overview}.
\begin{equation}\label{eq:sep  lpnlqm}
\sep\big(\ell_p^n(\ell_q^m)\big) \asymp n^{\max\left\{\frac1{p},\frac{1}{2}\right\}}m^{\max\left\{\frac1{q},\frac{1}{2}\right\}}=\max\Big\{\sqrt{nm},m^{\frac{1}{q}}\sqrt{n},
n^{\frac{1}{p}}\sqrt{m},n^{\frac{1}{p}}m^{\frac{1}{q}}\Big\}.
\end{equation}

Next, fix an integer $n\ge 2$ and let  $\bfE=(\R^n,\|\cdot\|_{\bfE})$ be an unconditional normed space. Given $q\in [2,\infty]$ and $\Lambda\ge 1$, one says (see e.g.~\cite[Definition1.f.4]{LT79}) that $\bfE$ satisfies a lower $q$-estimate with constant $\Lambda$ if for every $\{u_k\}_{k=1}^\infty\subset \R^n$ with pairwise disjoint supports we have
\begin{equation}\label{eq:def lower q}
\bigg(\sum_{k=1}^\infty \|u_k\|_\bfE^q\bigg)^{\frac{1}{q}}\le \Lambda \Big\|\sum_{k=1}^\infty u_k\Big\|_\bfE.
\end{equation}
Note that by~\eqref{eq:contraction unc} this always holds with $\Lambda=1$ if $q=\infty$.

In concrete cases it is often mechanical to evaluate up to universal constant factors the minimum radius of a Euclidean ball that circumscribes $B_\X$, but it is always within a $O(\sqrt{\log n})$  factor of the expression
\begin{equation}\label{eq:def rho E}
R_\bfE\eqdef \max_{\emptyset\neq S\subset \n}\bigg(\frac{\sqrt{|S|}}{\|\sum_{i\in S} e_i\|_\bfE}\bigg).
\end{equation}
More precisely, if $\bfE$ satisfies a lower $q$-estimate with constant $\Lambda$, then
\begin{equation}\label{eq:outrad symmetric}
R_\bfE\le \mathrm{outradius}_{\ell_2^n}(B_\X)\lesssim \Lambda (\log n)^{\frac{1}{2}-\frac{1}{q}}R_\bfE.
\end{equation}
The first inequality in~\eqref{eq:outrad symmetric} is immediate because $\|\sum_{i\in S} e_i\|_\bfE^{-1}\sum_{i\in S} e_i\in B_\bfE$  if $\emptyset\neq S\subset \n$. For a quick justification of the second inequality in~\eqref{eq:outrad symmetric}, note that by homogeneity we may assume without loss of generality that $\|e_i\|_\bfE\ge 1$ for every $i\in \N$. Therefore, using~\eqref{eq:contraction unc} we see that if $x=(x_1,\ldots,x_n)\in B_\bfE$, then $\max_{i\in \n} |x_i|\le 1$. Consequently, if we fix $x\in B_\bfE$ and denote for each $k\in \N$,
\begin{equation}\label{eq:break x xk}
S_k=S_k(x)\eqdef \Big\{i\in \n:\ \frac{1}{2^k}<|x_i|\le \frac{1}{2^{k-1}}\Big\},
\end{equation}
then the sets $\{S_k\}_{k=1}^\infty$ are a partition of $\n$ and in particular $\sum_{k=1}^\infty|S_k|=n$. Next,
\begin{equation}\label{eq:use lower q}
\Lambda R_\bfE\ge \Lambda R_\bfE\|x\|_\bfE \ge R_\bfE\bigg(\sum_{k=1}^\infty \Big\|\sum_{i\in S_k} x_i e_i\Big\|_\bfE^q\bigg)^{\frac{1}{q}}\ge \bigg(\sum_{k=1}^\infty R_\bfE^q\Big\|\sum_{i\in S_k} \frac{1}{2^k}e_i\Big\|_\bfE^q\bigg)^{\frac{1}{q}}\ge \bigg(\sum_{k=1}^\infty \frac{|S_k|^{\frac{q}{2}}}{2^{kq}}\bigg)^{\frac{1}{q}},
\end{equation}
where the second step of~\eqref{eq:use lower q} uses~\eqref{eq:def lower q}, the penultimate step of~\eqref{eq:use lower q} uses~\eqref{eq:contraction unc} and~\eqref{eq:break x xk}, and the final step of~\eqref{eq:use lower q} uses~\eqref{eq:def rho E}.  Now, for every $0<\theta<1$ we have
\begin{align}\label{eq:tri holder2}
\begin{split}
\|x\|_{\ell_2^n}&=\bigg(\sum_{k=1}^\infty \sum_{i\in S_k} x_i^2\bigg)^{\frac12}\le \bigg(\sum_{k=1}^\infty \frac{|S_k|}{2^{2(k-1)}}\bigg)^{\frac12}=2\bigg(\sum_{k=1}^\infty \frac{|S_k|^{1-\theta}}{2^{2k(1-\theta)}}|S_k|^\theta2^{-2k\theta}\bigg)^{\frac12}\\&\le 2\bigg(\sum_{k=1}^\infty \frac{|S_k|^{\frac{q}{2}}}{2^{2kq}}\bigg)^{\frac{1-\theta}{q}}\bigg(\sum_{k=1}^\infty |S_k|\bigg)^{\frac{\theta}{2}}\bigg(\sum_{k=1}^\infty 2^{-{\frac{2kq\theta}{(q-2)(1-\theta)}}}\bigg)^{\left(\frac12-\frac{1}{q}\right)(1-\theta)}\lesssim
 (\Lambda R_\bfE)^{1-\theta}n^{\frac{\theta}{2}}\theta^{-\left(\frac12-\frac{1}{q}\right)},
 \end{split}
\end{align}
where the second step of~\eqref{eq:tri holder2} uses~\eqref{eq:break x xk}, the penultimate step of~\eqref{eq:tri holder2} uses the trilinear H\"older inequality with exponents $1/\theta$, $q/(2(1-\theta))$ and $1/((1-2/q)(1-\theta))$, and the final step of~\eqref{eq:tri holder2} uses~\eqref{eq:use lower q}, the fact that $\sum_{k=1}^\infty |S_k|=n$, and elementary calculus. By choosing $\theta=1/\log n$ in~\eqref{eq:tri holder2}, we get~\eqref{eq:outrad symmetric}.

By the Lozanovski\u{\i} factorization theorem~\cite{Loz69}  there exist $w_1,\ldots, w_n>0$ such that
\begin{equation}\label{eq:loz}
\Big\|\sum_{i=1}^n w_i e_i\Big\|_{\bfE}=\Big\|\sum_{i=1}^n \frac{1}{w_i}e_i\Big\|_{\bfE^{\textbf{*}}}=\sqrt{n}.
\end{equation}
We will call any $w_1,\ldots, w_n>0$  that satisfy~\eqref{eq:loz} Lozanovski\u{\i}  weights for $\bfE$. They can be found by maximizing the concave function $w\mapsto \sum_{i=1}^n \log w_i$ over $w\in B_{\bfE}$ (see also e.g.~\cite[Chapter~3]{Pis89}), which can be done efficiently if $\bfE$ is given by an efficient oracle; their existence can also be established non-constructively using the Brouwer fixed point theorem~\cite{JR76}. By~\cite[Lemma~1.2]{Sch82} (note that we are using a different  normalization of the weights than in~\cite{Sch82}),
\begin{equation}\label{eq:lozanov to volume}
\vol_n(B_\bfE)^{\frac{1}{n}}\asymp \frac{(w_1\cdots w_n)^{\frac{1}{n}}}{\sqrt{n}}.
\end{equation}

By combining~\eqref{eq:outrad symmetric}  and~\eqref{eq:lozanov to volume}, we get the following lemma.

\begin{lemma}\label{lem:unconditional volume ratio} Fix an integer $n\ge 2$ and let  $\bfE=(\R^n,\|\cdot\|_{\bfE})$ be an unconditional normed space. Suppose that $\bfE$ satisfies a lower $q$-estimate with constant $\Lambda$ for some $q\ge 2$ and $\Lambda\ge 1$. Then,
$$
\evr(\bfE)\lesssim  \frac{\max_{\emptyset\neq S\subset \n}\Big(\frac{\sqrt{|S|}}{\|\sum_{i\in S} e_i\|_\bfE}\Big)}{\sqrt[n]{w_1\cdots w_n}}\Lambda (\log n)^{\frac12-\frac{1}{q}},
$$
for any  Lozanovski\u{\i}  weights $w_1,\ldots,w_n>0$  for $\bfE$.
If the L\"owner ellipsoid of $\bfE$ is a multiple of $B_{\ell_2^n}$, then
$$
 \frac{\max_{\emptyset\neq S\subset \n}\Big(\frac{\sqrt{|S|}}{\|\sum_{i\in S} e_i\|_\bfE}\Big)}{\sqrt[n]{w_1\cdots w_n}}\lesssim \evr(\bfE)\lesssim  \frac{\max_{\emptyset\neq S\subset \n}\Big(\frac{\sqrt{|S|}}{\|\sum_{i\in S} e_i\|_\bfE}\Big)}{\sqrt[n]{w_1\cdots w_n}}\Lambda (\log n)^{\frac12-\frac{1}{q}}.
$$
\end{lemma}

The following corollary is a consequence of Lemma~\ref{lem:unconditional volume ratio}  because if $\bfE=(\R^n,\|\cdot\|_{\bfE})$ is a normed space that satisfies the assumptions of Lemma~\ref{lem:weak iso for enouhg permutations and unconditional} (in particular, $\bfE$ is unconditional), then by Lemma~\ref{lem:loz is uniform} $$w_1=w_2=\ldots=w_n=\frac{\sqrt{n}}{\|e_1+\ldots + e_n\|_\bfE}$$ are Lozanovski\u{\i}  weights for $\bfE$.

\begin{corollary}\label{cor:symmetric} If $\bfE=(\R^n,\|\cdot\|_{\bfE})$ a normed space that satisfies the assumptions of Lemma~\ref{lem:weak iso for enouhg permutations and unconditional}, then
$$
\frac{\|e_1+\ldots +e_n\|_\EE}{\sqrt{n}}\bigg(\max_{k\in \n}\frac{\sqrt{k}}{\|e_1+\ldots+e_k\|_\EE}\bigg)\lesssim \evr(\bfE)\lesssim \frac{\|e_1+\ldots +e_n\|_\EE}{\sqrt{n}}\bigg(\max_{k\in \n}\frac{\sqrt{k}}{\|e_1+\ldots+e_k\|_\EE}\bigg) \sqrt{\log n}.
$$
Hence, by Corollary~\ref{cor:sep sharp if conj} we have
$$
\|e_1+\ldots +e_n\|_\EE\bigg(\max_{k\in \n}\frac{\sqrt{k}}{\|e_1+\ldots+e_k\|_\EE}\bigg)\lesssim \sep(\bfE)\lesssim \|e_1+\ldots +e_n\|_\EE\bigg(\max_{k\in \n}\frac{\sqrt{k}}{\|e_1+\ldots+e_k\|_\EE}\bigg)\sqrt{\log n},
$$
More succinctly, this can be written in the following form, which we already stated in Corollary~\ref{conj:symmetric volume ratio sep}.
\begin{equation}\label{sep lpn conj}
\sep(\bfE)= \|e_1+\ldots +e_n\|_\EE\bigg(\max_{k\in \n}\frac{\sqrt{k}}{\|e_1+\ldots+e_k\|_\EE}\bigg) n^{o(1)}.
\end{equation}
\end{corollary}

By~\cite[Proposition~2.2]{Sch82}, the unitary ideal of any  symmetric normed space $\bfE=(\R^n,\|\cdot\|_\bfE)$  satisfies
\begin{equation}\label{eq:vr of ideal}
\vr(\sfS_\bfE)\asymp \vr(\bfE).
\end{equation}
This implies that
\begin{equation}\label{eq:evr of ideal}
\evr(\sfS_\bfE)\asymp \evr(\bfE),
\end{equation}
by~\eqref{eq:state bourgain milman} combined with  $\sfS_\bfE^*=\sfS_{\bfE^{\textbf *}}$, though a straightforward adjustment of the proof of~\eqref{eq:vr of ideal} in~\cite{Sch82} yields~\eqref{eq:evr of ideal} directly,  without using the much deeper result~\eqref{eq:state bourgain milman}.   We therefore have the following corollary.

\begin{corollary}\label{cor:unitary} If $\bfE=(\R^n,\|\cdot\|_{\bfE})$ is a symmetric normed space, then
$$
\frac{\|e_1+\ldots +e_n\|_\EE}{\sqrt{n}}\bigg(\max_{k\in \n}\frac{\sqrt{k}}{\|e_1+\ldots+e_k\|_\EE}\bigg)\lesssim \evr(\sfS_\bfE)\lesssim \frac{\|e_1+\ldots +e_n\|_\EE}{\sqrt{n}}\bigg(\max_{k\in \n}\frac{\sqrt{k}}{\|e_1+\ldots+e_k\|_\EE}\bigg) \sqrt{\log n}.
$$
Hence, by Corollary~\ref{cor:sep sharp if conj} we have
$$
\|e_1+\ldots +e_n\|_\EE\bigg(\max_{k\in \n}\frac{\sqrt{k}}{\|e_1+\ldots+e_k\|_\EE}\bigg)\sqrt{n}\lesssim \sep(\sfS_\bfE)\lesssim \|e_1+\ldots +e_n\|_\EE\bigg(\max_{k\in \n}\frac{\sqrt{k}}{\|e_1+\ldots+e_k\|_\EE}\bigg)\sqrt{n}\log n,
$$
More succinctly, this can be written in the following form, which we already stated in Corollary~\ref{conj:symmetric volume ratio sep}.
\begin{equation}\label{sep lpn conj}
\sep(\sfS_\bfE)= \|e_1+\ldots +e_n\|_\EE\bigg(\max_{k\in \n}\frac{\sqrt{k}}{\|e_1+\ldots+e_k\|_\EE}\bigg) n^{\frac12+o(1)}.
\end{equation}
\end{corollary}

\begin{remark}\label{rem:rectangular schatten} In the above discussion, as well as in the ensuing treatment of tensor products, we prefer to consider square matrices rather than rectangular matrices because the setting of square matrices exhibits all of the key issues while being notationally simpler. Nevertheless, there are two places in which we do need to work with rectangular matrices, namely the above proof of Proposition~\ref{prop:slightly less rounded cube} and the proof of the first inequality in~\eqref{eq:sep schatten rank k}. For the latter, fix $p\ge 1$ and $n,m\in \N$. As in the proof of Theorem~\ref{thm:sparse and low rank}, denote the Schatten--von Neumann trace class on the $n$-by-$m$ real matrices $\M_{n\times m}(\R)$ by $\sfS_p^{n\times m}$; recall~\eqref{eq:def schatten rectangular}. The following asymptotic identity implies~\eqref{eq:evr rectangular k} (recall that in the setting of~\eqref{eq:evr rectangular k} we have $r\in \n$).
\begin{equation}\label{eq:evr rect}
\evr\big(\sfS^{n\times m}_p\big)\asymp \big(\min\{n,m\}\big)^{\max\left\{\frac{1}{p}-\frac12,0\right\}}.
\end{equation}
Volumes of unit balls of Schatten--von Neumann trace classes have been satisfactorily estimated in the literature, starting with~\cite{ST-J80} and the comprehensive work~\cite{Sch82}, through the more precise asymptotics in~\cite{S-R84,KPT20}. Unfortunately, all of these works dealt only with square matrices. Nevertheless, these references could be mechanically adjusted  to treat rectangular matrices as well. Since~\eqref{eq:evr rect} does not seem to have been stated in the literature, we will next  sketch its derivation by mimicking the reasoning of~\cite{Sch82}, though the more precise statements of~\cite{S-R84,KPT20} could  be derived as well via similarly straighforward modifications of the known proofs for square matrices. We claim that
\begin{equation}\label{eq:vol rect shatten p}
\vol_{nm}\big(B_{\sfS^{n\times m}_p}\big)^{\frac{1}{nm}}\asymp \frac{1}{\big(\min\{n,m\}\big)^{\frac{1}{p}}\sqrt{\max\{n,m\}}}.
\end{equation}
\eqref{eq:vol rect shatten p} gives~\eqref{eq:evr rect} since $\sfS_p^{n\times m}$ is canonically positioned, so by H\"older's inequality its  L\"owner ellipsoid is
$$
\mathscr{L}_{\sfS^{n\times m}_p}= \big(\min\{n,m\}\big)^{\max\big\{\frac12-\frac{1}{p},0\big\}}B_{\sfS^{n\times m}_2}.
$$
To prove~\eqref{eq:vol rect shatten p}, note first  that it follows from its special case $p=\infty$. Indeed, as $\sfS^{n\times m}_1=(\sfS^{n\times m}_\infty)^*$, by the Blaschke--Santal\'o inequality~\cite{Bla17,San49} and the Bourgain--Milman inequality~\cite{BM87} the case $p=1$ of~\eqref{eq:vol rect shatten p} follows from its case $p=\infty$. Now, \eqref{eq:vol rect shatten p}  follows in full generality since by H\"older's inequality.
$$
\frac{1}{\big(\min\{n,m\}\big)^{\frac{1}{p}}}B_{\sfS^{n\times m}_\infty}\subset B_{\sfS^{n\times m}_p}\subset \big(\min\{n,m\}\big)^{1-\frac{1}{p}} B_{\sfS^{n\times m}_1}.
$$
The upper bound $\vol_{nm}(B_{\sfS^{n\times m}_\infty})^{1/(nm)}\lesssim 1/\sqrt{\max\{n,m\}}$ follows from $B_{\sfS^{n\times m}_\infty}\subset \sqrt{\min\{n,m\}}B_{\sfS^{n\times m}_2}$. For the matching lower bound, if $\{\e_{ij}\}_{i,j\in \N}$ are i.i.d.~Bernoulli random variables, then by~\cite[Theorem~1]{BGN75},
$$
\E \bigg[\Big\|\sum_{i=1}^n\sum_{j=1}^m\e_{ij} e_{i}\otimes e_j\Big\|_{\sfS^{n\times m}_\infty}\bigg]\lesssim \sqrt{\max\{n,m\}},
$$
This implies the lower bound $\vol_{nm}(B_{\sfS^{n\times m}_\infty})^{1/(nm)}\gtrsim 1/\sqrt{\max\{n,m\}}$  by~\cite[Lemma~1.5]{Sch82}.
\end{remark}

\begin{proof}[Proof of Lemma~\ref{lem:reverse iso for sym}]   By equation~(2.2) in~\cite{Sch82} we have
\begin{equation}\label{eq:volume radius of ideal}
\vol_{n^2}\big(B_{\sfS_\bfE}\big)^{\frac{1}{n^2}}\asymp \frac{1}{\|e_1+\ldots+e_n\|_\bfE\sqrt{n}}.
\end{equation}
In particular,
\begin{equation}\label{eq:volume of schatten p}
\forall q\ge 1,\qquad \vol_{n^2}\big(B_{\sfS_q^n}\big)^{\frac{1}{n^2}}\asymp \frac{1}{n^{\frac12+\frac{1}{q}}}.
\end{equation}
Because $\sfS_q^n$ is canonically positioned (it belongs to the class of spaces in Example~\ref{ex permutation and sign}), and hence it is in its minimum surface area position, by combining~\cite[Proposition~3.1]{GP99} and~\eqref{eq:max proj lower} we see that
\begin{equation}\label{eq:cheeger ratio schatten q}
\frac{\vol_{n^2-1}\big(\partial B_{\sfS_q^n}\big)}{\vol_{n^2}\big(B_{\sfS_q^n}\big)}\asymp \frac{n\mathrm{MaxProj}\big(B_{\sfS_q^n}\big)}{{\vol_{n^2}\big(B_{\sfS_q^n}\big)}} \stackrel{\eqref{eq:quote with gid}}{\asymp} n^{\frac32+\frac{1}{q}}\sqrt{\min\{q,n\}}.
\end{equation}
Consequently,
\begin{equation}\label{eq:iq schatten p}
\iq\big(B_{\sfS_q^n}\big)= n\frac{\vol_{n^2-1}\big(\partial B_{\sfS_q^n}\big)}{\vol_{n^2}\big(B_{\sfS_q^n}\big)}\vol_{n^2}\big(B_{\sfS_q^n}\big)^{\frac{1}{n^2}}\stackrel{\eqref{eq:volume of schatten p} \wedge \eqref{eq:cheeger ratio schatten q}}{\asymp} \frac{n^{\frac32+\frac{1}{q}}\sqrt{\min\{q,n\}}}{n^{\frac12+\frac{1}{q}}} =n\sqrt{\min\{q,n\}}.
\end{equation}

Because  by~\eqref{eq:contraction unc} we have $\|x\|_\bfE\le \|e_1+\ldots+e_n\|_\bfE \|x\|_{\ell_\infty^n}$ for every $x\in \R^n$, every matrix $A\in \M_n(\R)$ satisfies $\|A\|_{\sfS_\bfE}\le \|e_1+\ldots+e_n\|_{\bfE} \|A\|_{\sfS_\infty^n}\le \|e_1+\ldots+e_n\|_{\bfE} \|A\|_{\sfS_q^n}$. Consequently,
\begin{equation}\label{eq:Sqn inclusion}
\frac{1}{\|e_1+\ldots+e_n\|_{\bfE} }B_{\sfS_q^n}\subset B_{\sfS_\bfE}.
\end{equation}
Moreover,
$$
\iq\Big(\frac{1}{\|e_1+\ldots+e_n\|_{\bfE} }B_{\sfS_q^n}\Big)=\iq\big(B_{\sfS_q^n}\big)\stackrel{\eqref{eq:iq schatten p}}{\asymp}  n\sqrt{\min\{q,n\}},
$$
and
$$
\vol_{n^2}\Big(\frac{1}{\|e_1+\ldots+e_n\|_{\bfE} }B_{\sfS_q^n}\Big)^{\frac{1}{n^2}}\stackrel{\eqref{eq:volume of schatten p}}{\asymp} \frac{1}{\|e_1+\ldots+e_n\|_{\bfE}n^{\frac12+\frac{1}{q}}}\stackrel{\eqref{eq:volume radius of ideal}}{\asymp}  \frac{\vol_{n^2}\big(B_{\sfS_\bfE}\big)^{\frac{1}{n^2}}}{n^{\frac{1}{q}}}.
$$
By choosing $q=\log n$ we get~\eqref{what we know about weak ideal} for the normed space $\Y$ whose unit ball is the left hand side of~\eqref{eq:Sqn inclusion}.\end{proof}

\begin{remark}\label{rem:schatten infty implies SE} An inspection of the proof of Lemma~\ref{lem:reverse iso for sym}  reveals that if Conjecture~\ref{conj:weak reverse iso when canonical} holds for $\sfS_\infty^n$, then also Conjecture~\ref{conj:weak reverse iso when canonical}  holds for $\sfS_\bfE$ for any symmetric normed space $\bfE=(\R^n,\|\cdot\|_\bfE)$. Indeed, we would then take $\Y'=(\M_n(\R),\|\cdot\|_{\Y'})$ to be the normed space whose unit ball is
$$
B_{\Y'}=\frac{1}{\|e_1+\ldots+e_n\|_\bfE}\Ch S_\infty^n= \frac{1}{\|e_1+\ldots+e_n\|_\bfE}\sfS_{\chi\ell_\infty^n},
$$
where we recall Corollary~\ref{cor chi space}. If Conjecture~\ref{conj:weak reverse iso when canonical} holds for $\sfS_\infty^n$, then
$
n\asymp \iq(\Ch S_\infty^n)=\iq(B_{\Y'}),
$
and also
$$
\vol_{n^2}\big(\Ch S_\infty^n\big)^{\frac{1}{n^2}}\asymp \vol_{n^2}\big(S_\infty^n\big)^{\frac{1}{n^2}}\stackrel{\eqref{eq:volume of schatten p}}{\asymp} \frac{1}{\sqrt{n}},
$$
from which we see that
$$
\vol_{n^2}\big(B_{\Y'}\big)^{\frac{1}{n^2}}=\frac{\vol_{n^2}\big(\Ch S_\infty^n\big)^{\frac{1}{n^2}}}{\|e_1+\ldots+e_n\|_\bfE}\asymp \frac{1}{\|e_1+\ldots+e_n\|_\bfE\sqrt{n}}\stackrel{\eqref{eq:volume radius of ideal}}{\asymp}
\vol_{n^2}\big(B_{\sfS_\bfE}\big)^{\frac{1}{n^2}}.
$$
This proves Conjecture~\ref{conj:weak reverse iso when canonical} for $\sfS_\bfE$. Note in passing that this also implies that
$$
\frac{1}{\sqrt{n}}\asymp \vol_{n^2}\big(S_{\chi\ell_\infty^n}\big)^{\frac{1}{n^2}}\stackrel{\eqref{eq:volume radius of ideal}}{\asymp}
\frac{1}{\|e_1+\ldots+e_n\|_{\chi\ell_\infty^n}\sqrt{n}}.
$$
Therefore,  if Conjecture~\ref{conj:weak reverse iso when canonical} holds for $S_\infty^n$, then $\|e_1+\ldots+e_n\|_{\chi\ell_\infty^n}\asymp 1$. More generally, by mimicking the above reasoning we deduce that if Conjecture~\ref{conj:weak reverse iso when canonical} holds for $\sfS_\bfE$, then  $\|e_1+\ldots+e_n\|_{\chi\bfE}\asymp \|e_1+\ldots+e_n\|_{\bfE}$, which would be a modest step towards Problem~\ref{prob: weyl}.
\end{remark}

Fix $n\in \N$ and $p,q\ge 1$. We claim that  the volume ratio of the projective tensor product $\ell_p^n\tp \ell_q^n$ satisfies
\begin{equation}\label{eq:vr of ti pq}
\vr\big(\ell_p^n\tp \ell_q^n\big)\asymp \Phi_{p,q}(n),
\end{equation}
where
\begin{equation}\label{eq: ti vr cases}
\Phi_{p,q}(n)\eqdef \left\{\begin{array}{ll}1 &\mathrm{if}\quad 1\le p,q\le 2,\\
n^{\frac12-\frac{1}{p}}&\mathrm{if}\quad q\le 2\le p\le \frac{q}{q-1},\\
n^{\frac1{q}-\frac{1}{2}}&\mathrm{if}\quad q\le 2\le \frac{q}{q-1}\le p,\\
n^{\frac12-\frac{1}{q}}&\mathrm{if}\quad p\le 2\le q\le \frac{p}{p-1},\\
n^{\frac1{p}-\frac{1}{2}}&\mathrm{if}\quad p\le 2\le \frac{p}{p-1}\le q,\\
1 &\mathrm{if}\quad p,q\ge 2\ \mathrm{and}\  \frac{1}{p}+\frac{1}{q}\ge \frac12,\\
n^{\frac12-\frac{1}{p}-\frac{1}{q}}&\mathrm{if}\quad \frac{1}{p}+\frac{1}{q}\le \frac12.
\end{array}\right.
\end{equation}
Assuming~\eqref{eq: ti vr cases} for the moment,  by substituting it into Theorem~\ref{thm:sep bounds in overview} we get that
$$
\sep\big(\ell_p^n\ti \ell_q^n\big)\gtrsim n\vr\big(\big(\ell_{p}^n\ti \ell_{q}^n\big)^*\big)=n\vr\big(\ell_{\!\! p^*}^n\tp \ell_{\!\! q^*}^n\big)\asymp n\Phi_{p^*,q^*}(n)= \left\{\begin{array}{ll}n &\mathrm{if}\quad p,q\ge 2,\\
n^{\frac{1}{p}+\frac12}&\mathrm{if}\quad \frac{q}{q-1}\le p\le 2\le q,\\
n^{\frac32-\frac1{q}}&\mathrm{if}\quad p\le \frac{q}{q-1}\le 2\le q,\\
n^{\frac{1}{q}+\frac12}&\mathrm{if}\quad \frac{p}{p-1}\le q\le 2\le p,\\
n^{\frac32-\frac1{p}}&\mathrm{if}\quad q\le \frac{p}{p-1}\le 2\le p,\\
n &\mathrm{if}\quad p,q\le 2\ \mathrm{and}\  \frac{1}{p}+\frac{1}{q}\le \frac32,\\
n^{\frac{1}{p}+\frac{1}{q}-\frac12}&\mathrm{if}\quad \frac{1}{p}+\frac{1}{q}\ge \frac32.
\end{array}\right.
$$
Since for any two normed spaces $\X=(\R^n,\|\cdot\|_\X)$ and $\Y=(\R^n,\|\cdot\|_\Y)$ the space of operators from $\X^*$ to $\Y$ is isometric to the injective tensor product $\X^*\ti \Y$ (see e.g.~\cite{DFS08}), we get from this that
\begin{equation}\label{eq:long version of 12}
\sep\big(\M_n(\R),\|\cdot\|_{\ell_p^n\to \ell_q^n} \big)=\sep\big(\ell_{\!\! p^*}^n\ti\ell_q^n\big)\gtrsim \left\{\begin{array}{ll}n &\mathrm{if}\quad p\le 2\le q,\\
n^{\frac32-\frac{1}{p}}&\mathrm{if}\quad 2\le p\le q,\\
n^{\frac32-\frac1{q}}&\mathrm{if}\quad  2\le q\le p,\\
n^{\frac{1}{q}+\frac12}&\mathrm{if}\quad p\le q\le 2,\\
n^{\frac1{p}+\frac12}&\mathrm{if}\quad q\le p\le 2,\\
n &\mathrm{if}\quad \frac{2p}{p+2}\le q\le 2\le p,\\
n^{\frac{1}{q}-\frac{1}{p}+\frac12}&\mathrm{if}\quad q\le \frac{2p}{p+2}.
\end{array}\right.
\end{equation}
Observe that the rightmost quantity in~\eqref{eq:long version of 12} coincides with the right hand side of~\eqref{eq:sep of operator from ellp to ellq}. Since $\ell_p^n\ti \ell_q^n$ belongs to the class of spaces in Remark~\ref{ex permutation and sign}, a positive answer to Conjecture~\ref{weak isomorphic reverse conj1-symmetric}  for $\ell_p^n\ti \ell_q^n$ would imply the following asymptotic evaluation of $\sep(\ell_p^n\ti \ell_q^n)$, which is equivalent to~\eqref{eq:sep of operator from ellp to ellq}.
$$
\sep\big(\ell_p^n\ti \ell_q^n\big)\asymp \left\{\begin{array}{ll}n &\mathrm{if}\quad p,q\ge 2,\\
n^{\frac12+\frac{1}{p}}&\mathrm{if}\quad \frac{q}{q-1}\le p\le 2\le q,\\
n^{\frac32-\frac1{q}}&\mathrm{if}\quad p\le \frac{q}{q-1}\le 2\le q,\\
n^{\frac12+\frac{1}{q}}&\mathrm{if}\quad \frac{p}{p-1}\le q\le 2\le p,\\
n^{\frac32-\frac1{p}}&\mathrm{if}\quad q\le \frac{p}{p-1}\le 2\le p,\\
n &\mathrm{if}\quad p,q\le 2\ \mathrm{and}\  \frac{1}{p}+\frac{1}{q}\le \frac32,\\
n^{\frac{1}{p}+\frac{1}{q}-\frac12}&\mathrm{if}\quad \frac{1}{p}+\frac{1}{q}\ge \frac32.
\end{array}\right.
$$
Furthermore, by Theorem~\ref{thm:sharp up to log n canonically} the leftmost quantity in~\eqref{eq:long version of 12}  is bounded from above by $O(\log n)$ times the rightmost quantity in~\eqref{eq:long version of 12}, thus implying the fourth bullet point of Corollary~\ref{coro:examples of apps}.

\begin{comment}
 \left\{ \begin{array}{ll}
n^{\frac32-\frac{1}{\min\{p,q\}}} &\mathrm{if}\quad  p,q\ge 2,\\
n^{\frac12+\frac{1}{\max\{p,q\}}} &\mathrm{if}\quad  p,q\le 2,\\
n& \mathrm{if}\quad  p\le 2\le q,\\
n^{\max\big\{1,\frac{1}{q}-\frac{1}{p}+\frac12\big\}} &\mathrm{if}\quad  q\le 2\le p.
  \end{array} \right.
\end{comment}

The asymptotic evaluation~\eqref{eq:vr of ti pq} of $\vr(\ell_p^n\tp \ell_q^n)$ was proved in~\cite{Sch82} up to constant factors that depend on $p,q$, namely~\cite[Theorem~3.1]{Sch82}  states that
\begin{equation}\label{eq:tp schutt with pq dependence}
\forall p,q>1,\qquad \vr\big(\ell_p^n\tp \ell_q^n\big)\asymp_{p,q} \Phi_{p,q}(n).
\end{equation}
If $2\in \{p,q\}$ and also $\min\{p,q\}\le 2$, then~\eqref{eq:tp schutt with pq dependence} is due to Szarek and Tomczak-Jaegermann~\cite{ST-J80}. More recently, Defant and Michels~\cite{DM05}  generalized~\eqref{eq:tp schutt with pq dependence} to projective tensor products of  symmetric normed  spaces that are either $2$-convex or $2$-concave. The proof of~\eqref{eq:tp schutt with pq dependence} in~\cite{Sch82} yields constants that degenerate as $\min\{p,q\}$ tends to $1$. We will therefore next improve the reasoning in~\cite{Sch82} to  get~\eqref{eq:vr of ti pq}.

\begin{lemma}\label{lem:bernoulli chevet} Fix $n\in \N$ and $p,q\ge 1$. Let $\{\e_{ij}\}_{i,j\in \n}$ be i.i.d.~ Bernoulli random variables (namely, they are independent and each of them is uniformly distributed over $\{-1,1\}$). Then,
\begin{equation}\label{eq:bernoulli chevet}
\E \bigg[\Big\|\sum_{i=1}^n\sum_{j=1}^n\e_{ij} e_i\otimes e_j\Big\|_{\ell_p^n\ti\ell_q^n}\bigg]\asymp n^{\beta(p,q)}\eqdef \left\{\begin{array}{ll}n^{\frac{1}{p}+\frac{1}{q}-\frac12} &\mathrm{if}\quad \max\{p,q\}\le 2,\\
n^{\frac{1}{\min\{p,q\}}}&\mathrm{if}\quad \max\{p,q\}\ge 2.\end{array}\right.
\end{equation}
\end{lemma}

Citing the work~\cite{Che78} of Chevet, a version of Lemma~\ref{lem:bernoulli chevet} appears as Lemma~2.3 in~\cite{Sch82}, except that in~\cite[Lemma~2.3]{Sch82} the implicit constants in~\eqref{eq:bernoulli chevet} depend on $p,q$. An inspection of the proof of~\eqref{eq:tp schutt with pq dependence} in~\cite{Sch82} reveals that this is the only source of the dependence of the constants on $p,q$ (in fact, for this purpose~\cite{Sch82} only needs half of~\eqref{eq:bernoulli chevet}, namely to bound from above its left hand side by its right hand side). Specifically, all of the steps within~\cite{Sch82} incur only a loss of a universal constant factor, and  the proof of~\eqref{eq:tp schutt with pq dependence} in~\cite{Sch82} also appeals to inequalities in the earlier work~\cite{Sch78} of Sch\"utt, as well a classical inequality of Hardy and Littlewood~\cite{HL34}; all of the constants in these cited inequalities are universal. Therefore, \eqref{eq:vr of ti pq} will be established after we prove Lemma~\ref{lem:bernoulli chevet}.

\begin{proof}[Proof of Lemma~\ref{lem:bernoulli chevet}] Denote the random matrix whose $(i,j)$ entry is $\e_{ij}$ by $\mathscr{E}\in \M_n(\R)$. Then,  the goal is
\begin{equation}\label{eq:operator norm goal}
\E \big[\|\mathscr{E}\|_{\ell_{\!\! p^*}^n\to \ell_q^n}\big]\asymp n^{\beta(p,q)}.
\end{equation}
In fact, the lower  bound on the expected norm in~\eqref{eq:operator norm goal} holds always, i.e., for a universal constant $c>0$,
\begin{equation}\label{eq:E always lower}
\forall {A}\in \M_n(\{-1,1\}),\qquad \|{A}\|_{\ell_{\!\! p^*}^n\to \ell_q^n}\ge cn^{\beta(p,q)}.
\end{equation}
A justification of~\eqref{eq:E always lower} appears in the {\em proof of} Proposition~3.2 of Bennett's work~\cite{Ben77} (specifically, see the reasoning immediately after inequality~(15) in~\cite{Ben77}), where it is explained that we can take $c=1$ if $\min\{p^*,q\}\ge 2$ or $\max\{p^*,q\}\le 2$, and  that we can take $c=1/\sqrt{2}$ otherwise.

Next, let $\{\g_{ij}\}_{i,j\in \n}$ be i.i.d.~ standard Gaussian random variables. By~\cite[Lemme~3.1]{Che78},
\begin{equation}\label{eq:quote chevet}
\E \bigg[\Big\|\sum_{i=1}^n\sum_{j=1}^n \g_{ij} e_i\otimes e_j\Big\|_{\ell_p^n\ti \ell_q^n}\bigg]\asymp n^{\max\left\{\frac{1}{p}+\frac{1}{q}-\frac12,\frac{1}{p}\right\}}\sqrt{p} +n^{\max\left\{\frac{1}{p}+\frac{1}{q}-\frac12,\frac{1}{q}\right\}}\sqrt{q}.
\end{equation}
Consequently,
\begin{equation}\label{eq:rademacher chevet}
\E \bigg[\Big\|\sum_{i=1}^n\sum_{j=1}^n \e_{ij} e_i\otimes e_j\Big\|_{\ell_p^n\ti \ell_q^n}\bigg]\le \sqrt{\frac{\pi}{2}}\E \bigg[\Big\|\sum_{i=1}^n\sum_{j=1}^n \g_{ij} e_i\otimes e_j\Big\|_{\ell_p^n\ti \ell_q^n}\bigg]\lesssim n^{\beta(p,q)}\sqrt{\max\{p,q\}},
\end{equation}
where the first step of~\eqref{eq:rademacher chevet} is a standard comparison between Rademacher and Gaussian averages (a quick consequence of Jensen's inequality; e.g.~\cite{MP76}) and final step of~\eqref{eq:rademacher chevet}  uses~\eqref{eq:quote chevet}. This proves the desired bound~\eqref{eq:bernoulli chevet} when $\max\{p,q\}\le 2$, so suppose from now on that $\max\{p,q\}\ge 2$.

It suffices to treat the case $p\ge 2$. Indeed,  if $p\le 2$, then  $q\ge 2$ since $\max\{p,q\}\ge 2$, so by the duality $$\|\mathscr{E}\|_{\ell_{\!\! p^*}^n\to \ell_q^n}=\|\mathscr{E}^*\|_{\ell_{\!\! q^*}^n\to \ell_p^n},$$
and the fact that the transpose $\mathscr{E}^*$ has the same distribution as $\mathscr{E}$, the case $p\le 2$ follows from the case $p\ge 2$. It also suffices to treat the case $q\le  p$ because if $q\ge p$, then $\|\cdot\|_{\ell_q^n}\le \|\cdot\|_{\ell_p^n}$ point-wise, and therefore
$$
\|\mathscr{E}\|_{\ell_{\!\! p^*}^n\to \ell_q^n}\le \|\mathscr{E}\|_{\ell_{\!\! p^*}^n\to \ell_p^n}.
$$
Consequently, since $\beta(p,q)=\beta(p,p)$ when $q\ge p$, the case $q\ge p$ follows from the case $q=p$.

So, suppose from now that $p\ge 2$ and $q\le p$. If we denote
$$
r\eqdef \frac{q(p-2)}{p-q},
$$
with the convention $r=\infty$ if $q=p$, then $r\ge 1$ and
\begin{equation}\label{eq:interpolation parameters}
\frac{1}{q}=\frac{1-\theta}{r}+\frac{\theta}{2},\qquad\mathrm{where}\qquad \theta\eqdef \frac{2}{p}\in [0,1].
\end{equation}
Hence, by the Riesz--Thorin interpolation theorem~\cite{Rie27,Tho48} we have
$$
\|\mathscr{E}\|_{\ell_{\!\!p^*}^n\to \ell_q^n}\le \|\mathscr{E}\|_{\ell_{1}^n\to \ell_r^n}^{1-\theta} \|\mathscr{E}\|_{\ell_{2}^n\to \ell_2^n}^{\theta}=\Big(\max_{i\in \n} \big\|\mathscr{E}e_i\big\|_{\ell_r^n}\Big)^{1-\theta}\|\mathscr{E}\|_{\ell_{2}^n\to \ell_2^n}^{\theta}=
n^{\frac{1-\theta}{r}}\|\mathscr{E}\|_{\ell_{2}^n\to \ell_2^n}^{\theta}.
$$
By taking expectations of this inequality, we get that
\begin{equation}\label{eq:use 2-2}
\E \big[\|\mathscr{E}\|_{\ell_{\!\! p^*}^n\to \ell_q^n}\big]\le n^{\frac{1-\theta}{r}}\E\big[\|\mathscr{E}\|_{\ell_{2}^n\to \ell_2^n}^{\theta}\big]\le
n^{\frac{1-\theta}{r}}\Big(\E\big[\|\mathscr{E}\|_{\ell_{2}^n\to \ell_2^n}\big]\Big)^\theta\lesssim  n^{\frac{1-\theta}{r}+\frac{\theta}{2}}=n^{\frac{1}{q}}=n^{\beta(p,q)},
\end{equation}
where the second step of~\eqref{eq:use 2-2} uses Jensen's inequality, the third step of~\eqref{eq:use 2-2} uses the classical fact that the expectation of the operator norm from $\ell_2^n$ to $\ell_2^n$ of an $n\times n$ matrix whose entries are i.i.d.~symmetric Bernoulli random variables is  $O(\sqrt{n})$ (this follows from~\eqref{eq:rademacher chevet}, though it is older; see e.g.~\cite{BGN75}), the penultimate step of~\eqref{eq:use 2-2} uses~\eqref{eq:interpolation parameters}, and the last step of~\eqref{eq:use 2-2}  uses the definition of $\beta(p,q)$ in~\eqref{eq:bernoulli chevet} while recalling that we are now treating the case $p\ge 2$ and $q\le p$.
\end{proof}

A substitution of Lemma~\ref{lem:bernoulli chevet}  into the proof of  Lemma~3.2 in~\cite{Sch82} yields the following asymptotic evaluations of the $n^2$-roots of volumes of the unit balls of injective and projective tensor products; the statement of~\cite[Lemma~3.2]{Sch82} is identical, except that the constant factors depend on $p,q$, but  that is due only to the dependence of the constants on $p,q$ in Lemma~2.3 in~\cite{Sch82}, which Lemma~\ref{lem:bernoulli chevet} removes.
\begin{equation}\label{eq:lower volume radius injective}
\vol_{n^2} \big(B_{\ell_p^n\ti \ell_q^n}\big)^{\frac{1}{n^2}}\asymp  n^{-\beta(p,q)}\qquad \mathrm{and}\qquad \vol_{n^2} \big(B_{\ell_p^n\tp \ell_q^n}\big)^{\frac{1}{n^2}}\asymp n^{\beta(p^*,q^*)-2}.
\end{equation}

Since $\ell_p^n\tp \ell_q^n$ belongs to the class of spaces in Remark~\ref{ex permutation and sign}, its L\"owner ellipsoid is the minimal multiple of the standard Euclidean ball $B_{\sfS_2^n}$ that superscribes the unit ball of $\ell_p^n\tp \ell_q^n$, namely
$$
\mathscr{L}_{\ell_p^n\tp \ell_q^n}=R(n,p,q)B_{\sfS_2^n},
$$
where, since $B_{\ell_p^n\tp \ell_q^n}$ is the convex hull of $B_{\ell_p^n}\otimes B_{\ell_q^n}$,
\begin{equation}\label{eq:R projective}
R(n,p,q) =\max_{\substack{x\in B_{\ell_p^n}\\ y\in B_{\ell_q^n}}} \|x\otimes y\|_{\sfS_2^n}= \Big(\max_{x\in B_{\ell_p^n}} \|x\|_{\ell_2^n}\Big)\Big(\max_{y\in B_{\ell_q^n}} \|y\|_{\ell_2^n}\Big)= n^{\max\big\{\frac12-\frac{1}{p},0\big\}+\max\big\{\frac12-\frac{1}{q},0\big\}}.
\end{equation}
By combining~\eqref{eq:lower volume radius injective} and~\eqref{eq:R projective} we  get that
\begin{align}\label{eq:vr of projective}
\begin{split}
\vr\big(\ell_{\!\! p^*}^n\ti \ell_{\!\! q^*}^n\big)\stackrel{\eqref{eq:state bourgain milman}}{\asymp}\evr&\big(\ell_p^n\tp \ell_q^n\big)= R(n,p,q) \bigg(\frac{\vol_{n^2}(B_{\sfS_2^n})}{\vol_{n^2} (B_{\ell_p^n\tp \ell_q^n})}\bigg)^{\frac{1}{n^2}}\\&\asymp n^{\max\big\{\frac12-\frac{1}{p},0\big\}+\max\big\{\frac12-\frac{1}{q},0\big\}-\beta(p^*,q^*)+1}\stackrel{\eqref{eq:bernoulli chevet}}{=} \left\{ \begin{array}{ll} \sqrt{n} &\mathrm{if}\quad  \max\{p,q\}\ge 2,\\ n^{\frac{1}{\max\{p,q\}}} & \mathrm{if}\quad  \max\{p,q\}\le 2.
  \end{array} \right.
\end{split}
\end{align}
A substitution of~\eqref{eq:vr of projective} into Theorem~\ref{thm:sep bounds in overview} gives
\begin{equation}\label{eq:sep lower projective}
\sep\big(\ell_p^n\tp \ell_q^n\big)\gtrsim \left\{ \begin{array}{ll} n^{\frac32} &\mathrm{if}\quad  \max\{p,q\}\ge 2,\\ n^{1+\frac{1}{\max\{p,q\}}} & \mathrm{if}\quad  \max\{p,q\}\le 2.\end{array} \right.
\end{equation}
Furthermore, if Conjecture~\ref{weak isomorphic reverse conj1-symmetric} holds for $\ell_p^n\tp \ell_q^n$, then~\eqref{eq:sep lower projective} is sharp, namely~\eqref{eq:projective in overview} holds. Also, by Theorem~\ref{thm:sharp up to log n canonically} the left hand side of~\eqref{eq:sep lower projective}  is bounded from above by $O(\log n)$ times the right  hand side of~\eqref{eq:sep lower projective}, thus implying the fifth bullet point of Corollary~\ref{coro:examples of apps}.

\begin{remark}\label{rem:cut and factor} The above results imply clustering statements (and impossibility thereof) for norms that have significance to algorithms and complexity theory. For example, the {\em cut norm}~\cite{FK99} on $\M_n(\R)$ is $O(1)$-equivalent~\cite{AN06} to the operator norm from $\ell_\infty^n$ to $\ell_1^n$. So, by~\eqref{eq:sharp lp sep assuming cojecture}  the separation modulus of the cut norm  on $\M_n(\R)$ is predicted to be bounded above and below by universal constant multiples of $n^{3/2}$, and by Theorem~\ref{thm:sharp up to log n canonically} we know that it is at least a universal constant multiple of $n^{3/2}$ and at most a universal constant multiple of $n^{3/2}\log n$. As another notable example,  we proved that $\sep(\ell_\infty^n\tp\ell_\infty^n)\gtrsim n^{3/2}$. Moreover, if Conjecture~\ref{weak isomorphic reverse conj1-symmetric}  holds for $\ell_\infty^n\tp \ell_\infty^n$, then $\sep(\ell_\infty^n\tp\ell_\infty^n)\asymp n^{3/2}$ and by Theorem~\ref{thm:sharp up to log n canonically} we have $\sep(\ell_\infty^n\tp\ell_\infty^n)\lesssim n^{3/2}\log n$.   Grothendieck's inequality~\cite{Gro53} implies that
\begin{equation}\label{eq:gamma2 duality gro}
\forall A\in \M_n(\R),\qquad \|A\|_{\ell_\infty^n\tp\ell_\infty^n}\asymp \gamma_2^{1\to \infty}(A),
\end{equation}
where $\gamma_2^{1\to \infty}(A)$ is the factorization-through-$\ell_2$ norm (see~\cite{Pis-factorization}) of $A$ as an operator from $\ell_1^n$ to $\ell_\infty^n$, i.e.,
$$
\gamma_2^{1\to \infty}(A)\eqdef \min_{\substack{X,Y\in \M_n(\R)\\ A=XY}} \|X\|_{\ell_2^n\to \ell_\infty^n}\|Y\|_{\ell_1^n\to \ell_2^n}=\min_{\substack{X,Y\in \M_n(\R)\\ A=XY}} \max_{i,j\in \n} \|\mathrm{row}_i(X)\|_{\ell_2^n}\|\mathrm{column}_j(Y)\|_{\ell_2^n}.
$$
Above, for each $i,j\in \n$ and $M\in \M_n(\R)$ we denote by $\mathrm{row}_i(M)$ and  $\mathrm{column}_j(M)$ the $i$'th row and $j$'th column of $M$, respectively. See~\cite{LMSS07} for the justification of~\eqref{eq:gamma2 duality gro}, as well as the importance of the factorization norm $\gamma_2^{1\to \infty}$ to  complexity theory (see~\cite{MNT20,BLN21} for further algorithmic significance of  factorization norms). Thanks to the above discussion, we know that
$$
n^{\frac32}\lesssim \sep\big(\M_n(\R),\gamma_2^{1\to \infty}\big)\lesssim n^{\frac32}\log n,
$$
and that $\sep(\M_n(\R),\gamma_2^{1\to \infty})\asymp n^{3/2}$ assuming Conjecture~\ref{weak isomorphic reverse conj1-symmetric}. To check that this does not follow from the previously known bounds~\eqref{eq:use CCGGP bounds}, we need to know the asymptotic growth rate of the Banach--Mazur distance between $\ell_\infty^n\tp\ell_\infty^n$ and each of the spaces $\ell_1^{n^2}, \ell_2^{n^2}$. However, these Banach--Mazur distances do not appear in the literature. In response to our inquiry, Carsten Sch\"utt answered this question, by showing that
\begin{equation}\label{eq:BM schutt}
d_{\BM}\big(\ell_2^{n^2},\ell_\infty^n\tp \ell_\infty^n\big)\asymp d_{\BM}\big(\ell_1^{n^2},\ell_\infty^n\tp \ell_\infty^n\big)\asymp n.
\end{equation}
More generally, Sch\"utt succeeded to evaluate the asymptotic growth rate of the Banach--Mazur distance between $\ell_p^n\tp\ell_q^n$ and $\ell_p^n\ti\ell_q^n$ to each of  $\ell_1^{n^2}, \ell_2^{n^2}$ for every $p,q\in [1,\infty]$ (this is a substantial matter that Sch\"utt communicated to us privately and he will publish it elsewhere).
Due to~\eqref{eq:BM schutt}, an application of~\eqref{eq:use CCGGP bounds} only gives the bounds $n\lesssim \sep(\ell_\infty^n\tp\ell_\infty^n)\lesssim n^2$, which hold for {\em every} $n^2$-dimensional normed space. More generally, Sch\"utt's result shows that~\eqref{eq:sharp lp sep assuming cojecture} and~\eqref{eq:projective in overview} do not follow from~\eqref{eq:use CCGGP bounds}.
\end{remark}

The volume computations of this section are only an indication of the available information. The literature contains many more volume estimates that could be substituted into Theorem~\ref{thm:sep bounds in overview}  and Conjecture~\ref{conj:symmetric volume ratio sep}  to yield new results (and conjectures)  on separation moduli of various spaces; examples of further pertinent results appear in~\cite{Sch82,Bal91-gl,GJ97,GJN97,GJ99,GP99,GPSTW17,DP09,DV20,KP20,KPT20}.

\section{Logarithmic weak isomorphic isoperimetry in minimum dual  mean width position}\label{sec:log weak} In this section we will prove the results that we stated in Section~\ref{sec:intersection}. We first claim that for every integer $n\ge 2$ and every $r>0$ we have
\begin{equation}\label{eq:for all r version}
\iq \big(B_{\ell_\infty^n}\cap (rB_{\ell_2^n})\big)=\iq \big([-1,1]^n\cap (rB_{\ell_2^n})\big)\gtrsim  \left(\min\big\{\sqrt{n},r\big\}\left(1-\frac{1}{\max\big\{1,r^2\big\}}\right)^{\frac{n-1}{2}}+1\right)\sqrt{n}
\end{equation}
Observe that~\eqref{eq:for all r version} implies~\eqref{eq:with gid s version}. Furthermore, \eqref{eq:for all r version} implies the direction $\gtrsim$ in~\eqref{eq:best radius cube} because
\begin{multline*}
\min_{r>0}\frac{\iq\big(B_{\ell_\infty^n}\cap (rB_{\ell_2^n})\big)}{\sqrt{n}}\left(\frac{\vol_n(B_{\ell_\infty^n})}{\vol_n\big(B_{\ell_\infty^n}\cap (rB_{\ell_2^n})\big)}\right)^{\frac{1}{n}}\ge \min_{r>0} \frac{\iq\big(B_{\ell_\infty^n}\cap (rB_{\ell_2^n})\big)}{\sqrt{n}}\left(\frac{2^n}{\vol_n\big( rB_{\ell_2^n}\big)}\right)^{\frac{1}{n}}\\\gtrsim \min_{r>0}\left(\min\Big\{\frac{\sqrt{n}}{r},1\Big\}\left(1-\frac{1}{\max\big\{1,r^2\big\}}\right)^{\frac{n-1}{2}}+\frac{1}{r}\right)\sqrt{n}\asymp \sqrt{\log n},
\end{multline*}
where the penultimate step uses~\eqref{eq:for all r version} and the final step is elementary calculus. Since the $K$-convexity constant of $\ell_\infty^n$ satisfies $K(\ell_\infty^n)\asymp \sqrt{\log n}$ (see~\cite[Chapter~2]{Pis89}), the matching upper bound in~\eqref{eq:best radius cube}  will follow after we will prove (below) Proposition~\ref{prop:K convexity}. This will also show that  Proposition~\ref{prop:K convexity} is sharp, though it would be worthwhile to find out  if it is sharp even for some normed space $\X=(\R^n,\|\cdot\|_\X)$ for which $K(\X)\asymp \log n$; such a space exists  by a remarkable (randomized) construction of Bourgain~\cite{Bou84}.

To prove~\eqref{eq:for all r version}, note first that if $0<r\le 1$, then $rB_{\ell_2^n}\subset [-1,1]^n$ and therefore
\begin{equation}\label{eq:intersection cube trivial1}
\forall 0<r\le 1,\qquad
\iq\big([-1,1]^n\cap (rB_{\ell_2^n})\big)=\iq(rB_{\ell_2^n})\asymp \sqrt{n}.
\end{equation}
Similarly, note that if $r\ge \sqrt{n}$, then $rB_{\ell_2^n}\subset [-1,1]^n$ and therefore
\begin{equation}\label{eq:intersection cube trivial2}
\forall r\ge \sqrt{n},\qquad \iq\big([-1,1]^n\cap (rB_{\ell_2^n})\big)=\iq\big([-1,1]^n\big)\asymp n.
\end{equation}
Both~\eqref{eq:intersection cube trivial1} and~\eqref{eq:intersection cube trivial2}  coincide with~\eqref{eq:for all r version} in the respective ranges. The less trivial range of~\eqref{eq:for all r version} is when $1< r< \sqrt{n}$, in which case the boundary of $[-1,1]^n\cap (rB_{\ell_2^n})$ contains the disjoint union of the intersection of $rB_{\ell_2^n}$ with the $2n$ faces of $[-1,1]^n$, each of which is isometric to the following set.
$$
[-1,1]^{n-1}\cap \big(\sqrt{r^2-1}B_{\ell_2^{n-1}}\big).
$$
Together with the straightforward inclusion
$$
[-1,1]^{n-1}\cap \big(\sqrt{r^2-1}B_{\ell_2^{n-1}}\big)\supseteq \sqrt{1-\frac{1}{r^2}} \big([-1,1]^{n-1}\cap (rB_{\ell_2^{n-1}})\big),
$$
the above observation implies that if $1< r< \sqrt{n}$, then
\begin{align}\label{eq:2n faces}
\begin{split}
\vol_{n-1} \Big(\partial \big([-1,1]^n\cap (rB_{\ell_2^n})\big)\Big)&\ge 2n\left(1-\frac{1}{r^2}\right)^{\frac{n-1}{2}} \vol_{n-1}\big( [-1,1]^{n-1}\cap (rB_{\ell_2^{n-1}})\big)\\&= n\left(1-\frac{1}{r^2}\right)^{\frac{n-1}{2}} \vol_{n}\Big( \big([-1,1]^{n-1}\cap (rB_{\ell_2^{n-1}})\big)\times [-1,1]\Big)\\&\ge n\left(1-\frac{1}{r^2}\right)^{\frac{n-1}{2}} \vol_n \big([-1,1]^n\cap (rB_{\ell_2^n})\big),
\end{split}
\end{align}
where the final step~\eqref{eq:2n faces} is a consequence of the straightforward inclusion
$$
\big([-1,1]^{n-1}\cap (rB_{\ell_2^{n-1}})\big)\times [-1,1]\supseteq [-1,1]^n\cap (rB_{\ell_2^n}).
$$
By combining~\eqref{eq:2n faces} with the definition~\eqref{eq:def iq} of the isoperimetric quotient, we see that
\begin{equation}\label{eq:cancel iq}
\iq\big([-1,1]^n\cap (rB_{\ell_2^n})\big)\ge\frac{n\left(1-\frac{1}{r^2}\right)^{\frac{n-1}{2}} \vol_n \big([-1,1]^n\cap (rB_{\ell_2^n})\big)}{\vol_n\big([-1,1]^n\cap (rB_{\ell_2^n})\big)^{\frac{n-1}{n}}}=n\left(1-\frac{1}{r^2}\right)^{\frac{n-1}{2}} \vol_n \big([-1,1]^n\cap (rB_{\ell_2^n})\big)^{\frac{1}{n}}.
\end{equation}
When $r\le \sqrt{n}$ we have $[-1,1]^n\cap (rB_{\ell_2^n})\supseteq [-r/\sqrt{n},r/\sqrt{n}]^n$. In combination with~\eqref{eq:cancel iq}, this implies that
$$
\forall 1<r<\sqrt{n},\qquad \iq\big([-1,1]^n\cap (rB_{\ell_2^n})\big)\ge 2r\sqrt{n}\left(1-\frac{1}{r^2}\right)^{\frac{n-1}{2}}.
$$
As also $\iq([-1,1]^n\cap (rB_{\ell_2^n}))\gtrsim \sqrt{n}$ by the isoperimetric theorem~\eqref{eq:quote ispoperimetric theorem}, this completes the proof of~\eqref{eq:for all r version}.\qed

\medskip

Passing to the proof of Proposition~\ref{prop:K convexity}, observe first that for every $r>0$ we have
\begin{equation}\label{eq:use Markov M}
\frac{\vol_n\big(B_\X\cap (rB_{\ell_2^n})\big)}{\vol_n(rB_{\ell_2^n})}= \frac{\vol_n \big(\{x\in rB_{\ell_2^n}:\ \|x\|_\X\le 1\}\big)}{\vol_n(rB_{\ell_2^n})}\ge 1-\fint_{rB_{\ell_2^n}} \|x\|_\X\ud x=1-\frac{nr}{n+1}M(\X),
\end{equation}
where the penultimate step in~\eqref{eq:use Markov M} is  Markov's inequality and the final step in~\eqref{eq:use Markov M} is integration in polar coordinates using the following standard notation for the mean of the norm on the Euclidean sphere:
\begin{equation}\label{eq:def M average}
M(\X)\eqdef \fint_{S^{n-1}} \|z\|_\X\ud z.
\end{equation}
We will also use the common notation $M^*(\X)\eqdef M(\X^*)$. By setting  $r=1/(2M(\X))$ in~\eqref{eq:use Markov M}  we get that
\begin{equation}\label{eq:markov used for volume lower bound}
\vol_n\bigg(B_\X\cap \Big(\frac{1}{2M(\X)}B_{\ell_2^n}\Big)\bigg)^{\frac{1}{n}}\ge \bigg(\frac12 \vol_n\Big(\frac{1}{2M(\X)}B_{\ell_2^n}\Big)\bigg)^{\frac{1}{n}}\asymp \frac{1}{M(\X)\sqrt{n}}.
\end{equation}
This simple consideration gives the following general elementary lemma.
\begin{lemma}\label{lem:use markov for intersection with ball} Let $\X=(\R^n,\|\cdot\|_\X)$ be a normed space. For $r=1/(2M(\X))$ and $L= B_\X\cap (r B_{\ell_2^n})$ we have
\begin{equation}\label{eq:maxprojL}
\vol_n(L)^{\frac{1}{n}}\gtrsim \frac{1}{M(\X)\sqrt{n}}\qquad\mathrm{and} \qquad \frac{\mathrm{MaxProj}(L) }{\vol_n(L)^{\frac{n-1}{n}}}\lesssim 1.
\end{equation}
\end{lemma}
\begin{proof} The first inequality in~\eqref{eq:maxprojL} follows from~\eqref{eq:markov used for volume lower bound}. For the second inequality in~\eqref{eq:maxprojL}, observe that $\proj_{z^\perp}(L)\subset \proj_{z^\perp} (rB_{\ell_2^{n}})$ for every $z\in S^{n-1}$, since $L\subset rB_{\ell_2^n}$. Consequently,
 \begin{equation}\label{eq:included in euclidean ball}\mathrm{MaxProj}(L)\le \mathrm{MaxProj}\big(rB_{\ell_2^n}\big)=r^{n-1} \vol_{n-1}\big(B_{\ell_2^{n-1}}\big)\asymp \vol_n\big(rB_{\ell_2^n}\big)^{\frac{n-1}{n}}\lesssim \vol_n(L)^{\frac{n-1}{n}},
\end{equation}
where the penultimate step of~\eqref{eq:included in euclidean ball} is a standard computation using Stirling's formula and the final step of~\eqref{eq:included in euclidean ball} uses the first inequality in~\eqref{eq:markov used for volume lower bound}.
\end{proof}

By~\eqref{eq:max proj lower}, the second inequality in~\eqref{eq:maxprojL} implies that $\iq(L)\lesssim \sqrt{n}$.  Hence, in order to use Lemma~\ref{lem:use markov for intersection with ball} in the context of Conjecture~\ref{weak isomorphic reverse conj1} it would  be beneficial  to choose $S\in \SL_n(\R)$ for which $M(S\X)$ is small. So, fix $\d>0$ and suppose  that $\d M(S\X)\le \min_{T\in \SL_n(\R)}M(T\X)$. By compactness, this holds for some $S\in \SL_n(\R)$ with $\d=1$, in which case the polar of $SB_\X$ is in {\em minimum mean width position} and we will say that $S\X$ is in {\em minimum dual mean width position} (the terminology that is used in~\cite{GMR00} is that $SB_\X$ {\em has minimal $M$}). By~\cite{GM00}, the matrix in $\SL_n(\R)$  at which $\min_{T\in \SL_n(\R)} M(T\X)$ is attained is unique up to orthogonal transformations. We  allow the flexibility of working with some universal constant $0<\d<1$ rather than considering only  the minimum dual mean width position since this will encompass other commonly used positions, such as the $\ell$-position (see~\cite[Section~1.11]{BGVV14}). By~\cite{GM00}, $\X$ is in minimum dual mean width position  if and only if the  measure $\ud \nu_\X(z)=\|z\|_\X\ud z$ on $S^{n-1}$  is isotropic. Since $\nu_\X$ is evidently $\mathsf{Isom}(\X)$-invariant, by~\eqref{eq:isotropic}  if $\X$ is canonically positioned, then it is in minimum dual mean width position.

Let $\gamma$ be the standard Gaussian measure on $\R$, i.e., its density is  $u\mapsto \exp(-u^2/2)/\sqrt{2\pi}$. The (Gaussian) $K$-convexity constant $K(\X)$ of  $\X$ is defined~\cite{MP76} to be the infimum over those $K>0$ that satisfy
$$
\bigg(\int_{\R^{\aleph_0}}\Big\|\sum_{i=1}^\infty \g_i' \int_{\R^{\aleph_0}}\g_if(\g)\ud \gamma^{\otimes\aleph_0}(\g)\Big\|_\X^2\ud \gamma^{\otimes\aleph_0}(\g')\bigg)^{\frac12} \le K\bigg(\int_{\R^{\aleph_0}}\|f(\g)\|_\X^2\ud \gamma^{\otimes \aleph_0}(\g)\bigg)^{\frac12},
$$
for every measurable $f:\R^{\aleph_0}\to \X$ with $\int_{\R^{\aleph_0}}\|f(\g)\|_\X^2\ud \gamma^{\otimes\aleph_0}(\g) <\infty$. By~\cite{FT79} there is $T\in \SL_n(\R)$ such that $M(T\X)M^*(T\X)\le K(\X).$
By the above assumption $\d M(S\X)\le M(T\X)$, so  $\d M(S\X)\le K(\X)/M^*(T\X)$. Next, $M(\X)\ge (\vol_n(B_{\ell_2^n})/\vol_n(B_\X))^{1/n}$; see e.g.~\cite[Section~2]{MR1008717} and~\cite[Lemma~30]{HN19} for two derivations of this well-known volumetric lower bound on $M(\X)$. Applying this lower bound to the dual of $T\X$, we get $M^*(T\X)\ge (\vol_n(B_{\ell_2^n})/\vol_n(B_{\X^{\textbf{*}}}))^{1/n}$. The Blaschke--Santal\'o inequality~\cite{Bla17,San49} states that $$\frac{\vol_n(B_{\ell_2^n})}{\vol_n(B_{\X^{\textbf{*}}})}\ge \frac{\vol_n(B_\X)}{\vol_n(B_{\ell_2^n})},$$ so we conclude that $\d M(S\X)\sqrt{n}\lesssim K(\X)/\sqrt[n]{\vol_n(B_\X)}$. A substitution of this into Lemma~\ref{lem:use markov for intersection with ball} gives:

\begin{proposition}\label{prop:K convexity'}Fix $0<\d\le 1$ and a normed space $\X=(\R^n,\|\cdot\|_\X)$. Suppose that $S\in \SL_n(\R)$ satisfies $\d M(S\X)\le \min_{T\in \SL_n(\R)}M(T\X)$. Then, denoting $r=1/(2M(S\X))$ we have
\begin{equation*}
\vol_n \big((SB_\X)\cap (r B_{\ell_2^n})\big)^{\frac{1}{n}}\gtrsim \frac{\d}{ K(\X)}\vol_n(B_\X)^{\frac{1}{n}} \qquad\mathrm{and}\qquad \mathrm{MaxProj} \big((SB_\X)\cap (r B_{\ell_2^n})\big)\lesssim \vol_n \big((SB_\X)\cap (r B_{\ell_2^n})\big)^{\frac{n-1}{n}}.
\end{equation*}
 Furthermore, if $\X$ is canonically positioned, then this holds when  $S$ is the identity matrix and $\d=1$.
\end{proposition}

By~\eqref{eq:max proj lower},  Proposition~\ref{prop:K convexity'} implies Proposition~\ref{prop:K convexity}, with the additional information that the conclusion of Proposition~\ref{prop:K convexity} holds with $S$ the identity matrix if $\X$ is in minimum dual mean width position, in which case we obtain an upper bound on $\mathrm{MaxProj}(L)$. Hence, by the reasoning in Section~\ref{sec:reverse iso}, if  $\X$ is in minimum dual mean width position, then $$\sep(\X)\lesssim K(\X) \frac{\diam_{\ell_2^n}(B_\X)}{\vol_n(B_\X)^{\frac{1}{n}}}.$$

\bibliographystyle{alphaabbrvprelim}
\bibliography{almost-ext}

\bigskip

\noindent{\bf Added in proof.} Since the initial posting of this work, the following progress has been made on some of the issues that are discussed herein. The forthcoming article~\cite{BBN24} proves Conjecture~\ref{weak isomorphic reverse conj1}  when $K\subset \R^n$ is any origin-symmetric convex polytope that has $O(n)$ faces. The forthcoming article~\cite{GN24} proves Conjecture~\ref{weak isomorphic reverse conj1} when $K\subset \M_n(\R)$ is the unit ball of the unitary ideal $\sfS_\bfE$ of any $1$-symmetric normed space $\bfE=(\R^n,\|\cdot\|_\bfE)$; consequently, Lemma~\ref{lem:reverse iso for sym} and Proposition~\ref{prop:weak rev iso sym} hold with all of the logarithmic factors that appear in them replaced by universal constants. The (quite major)  forthcoming article~\cite{BN24} builds on the results herein while adding multiple  innovations and ideas to obtain several new results. These include the estimate $
\ee(\X)\gtrsim \sqrt[4]{\dim(\X)}$ for every normed space $\X$, which is an improvement over the value of the universal constant $c$ that we obtained in the proof of Theorem~\ref{thm:power lower}. Conjecture~\ref{conj:no auto conv} is resolved (negatively) in~\cite{BN24}, where it is proved that $\ee_\conv(\MM)\lesssim \ee(\MM)^2$ for every Polish metric space $(\MM,d_\MM)$. It is also proved in~\cite{BN24} that $\ee(\ell_2^n,\NN)\asymp\sqrt[4]{n}$ for any $1$-net $\NN$ of $\ell_2^n$ and $\ee(\ell_2^n,\Z^n)\asymp\sqrt[6]{n}$; both of these asymptotic evaluations of Lipschitz extension moduli answer questions that were posed in the precursor~\cite{Nao17-SODA} of the present work. Finally, the lower order factor in the main result of~\cite{Nao20-almost} is removed in~\cite{BN24}, thus showing that an old Lipschitz almost-extension result of Bourgain~\cite{Bou87} is sharp up to universal constant factors.  Beyond the aforementioned examples of statements from~\cite{BN24}, multiple other new results on Lipschitz extension and separation moduli are obtained in~\cite{BN24}. The discussion in Remark~\ref{rem:octagon} (more generally, the role that canonically positioned norms play herein), evolved (very substantially) to the forthcoming work~\cite{BLNR24} which investigates the question of when is it possible to construct a norm with  prescribed group of isometries; as demonstrated in~\cite{BLNR24}, it turns out that the answer to this old inverse problem is quite subtle.
 \end{document}